\documentclass[11pt]{amsart}

\usepackage{amsthm,placeins}
\usepackage{amstext}
\usepackage{enumitem}
\usepackage{tikz-cd}
\usepackage{tikz}
\usepackage{makecell}
\usepackage{amssymb}
\usepackage{amsmath,calligra,mathrsfs}
\usepackage{stmaryrd}
\usepackage[all,cmtip]{xy}
\usepackage{dsfont}
\usepackage[backref=page]{hyperref}
\usepackage{adjustbox}
\tikzset{
  f/.style={every coordinate/.try}
}
\hypersetup{
	colorlinks   = true,          
	urlcolor     = blue,          
	linkcolor    = purple,          
	citecolor   = blue,             
        pdfauthor = {Name}
}
\usepackage{graphicx}
\usepackage{comment}
\newenvironment{detail}{\color{magenta}}{} 

\DeclareMathOperator{\Bl}{Bl}

\newcommand{\hr}[2]{\hyperref[#1]{#2}}
\newcommand{\tX}{{\widetilde{X}}}
\newcommand{\tY}{{\widetilde{Y}}}
\newcommand{\tR}{{\widetilde{R}}}
\newcommand{\hX}{{\widehat{X}}}
\newcommand{\RClass}{{T}}
\newcommand{\hY}{{\widehat{Y}}}
\newcommand{\hC}{{\widehat{C}}}
\newcommand{\oC}{{\overline{C}}}
\newcommand{\oY}{{\overline{Y}}}
\newcommand{\tF}{{\widetilde{F}}}
\newcommand{\tS}{{\widetilde{S}}}

\newcommand{\hS}{{\widehat{S}}}
\newcommand{\oS}{{\overline{S}}}
\newcommand{\oF}{{\overline{F}}}
\newcommand{\oX}{{\overline{X}}}
\newcommand{\oZ}{{\overline{Z}}}

\def\bA{{\mathbb{A}}}

\def\CC{{\mathbb C}}

\def\bP{{\mathbb{P}}}
\def\bF{{\mathbb{F}}}

\def\calX{{\mathcal{X}}}

\def\calO{{\mathcal O}}

\def\vol{{\mathrm{vol}}}

\def\hbar{{\overline{h}}}

\def\ord{{\mathrm{ord}}}

\def\id{{\mathrm{id}}}
\def\Pic{{\mathrm{Pic}}}

\def\red{{\mathrm{red}}}

\newcommand\cal{\mathcal}
\newcommand\bb{\mathbb}

\usepackage{xcolor}
\usepackage{mathtools}

\theoremstyle{plain}
\newtheorem{theorem}{Theorem}[section]
\newtheorem{lemma}[theorem]{Lemma}
\newtheorem{conjecture}[theorem]{Conjecture}

\newtheorem{proposition}[theorem]{Proposition}
\newtheorem{corollary}[theorem]{Corollary}
\theoremstyle{definition}
\newtheorem{definition}[theorem]{Definition}

\newtheorem{notation}[theorem]{Notation}
\theoremstyle{remark}
\newtheorem{remark}[theorem]{Remark}

\theoremstyle{plain}

\newtheorem{maintheorem}{Theorem}
\newtheorem{maincorollary}[maintheorem]{Corollary}

\makeatletter
\def\@tocline#1#2#3#4#5#6#7{\relax
	\ifnum #1>\c@tocdepth 
	\else
	\par \addpenalty\@secpenalty\addvspace{#2}%
	\begingroup \hyphenpenalty\@M
	\@ifempty{#4}{%
		\@tempdima\csname r@tocindent\number#1\endcsname\relax
	}{%
		\@tempdima#4\relax
	}%
	\parindent\z@ \leftskip#3\relax \advance\leftskip\@tempdima\relax
	\rightskip\@pnumwidth plus4em \parfillskip-\@pnumwidth
	#5\leavevmode\hskip-\@tempdima
	\ifcase #1
	\or\or \hskip 1em \or \hskip 2em \else \hskip 3em \fi%
	#6\nobreak\relax
	\dotfill\hbox to\@pnumwidth{\@tocpagenum{#7}}\par
	\nobreak
	\endgroup
	\fi}
\makeatother

\DeclareMathOperator{\diag}{diag}

\DeclareMathOperator{\Aut}{Aut}

\DeclareMathOperator{\sheafhom}{\mathscr{H}\text{\kern -3pt {\calligra\large om}}\,}

\DeclareMathOperator{\SL}{SL}
\DeclareMathOperator{\GL}{GL}

\DeclareMathOperator{\Sing}{Sing}
\DeclareMathOperator{\coeff}{coeff}

\newcommand{\Deltatilde}{\tilde{\Delta}}
\newcommand{\GITquartic}{\overline{M}^{\mathrm{GIT}}_4}
\newcommand{\GITR}{\overline{M}^{\mathrm{GIT}}_{(2,2)}}
\newcommand{\GITn}{\overline{M}^{\mathrm{GIT}}_{\mathfrak{n}}}
\newcommand{\GITnstack}{\overline{\mathcal{M}}^{\mathrm{GIT}}_{\mathfrak{n}}}
\newcommand{\Kmodulispace}[1]{\overline{M}^{\mathrm{K}}_{#1}}
\newcommand{\Kmodulitwoeighteen}{\overline{M_{2.18}}}
\newcommand{\Xthree}{X_{\mathfrak{n}}}
\newcommand{\cnot}{{c_1}}
\newcommand{\cone}{c_0}
\newcommand{\Rone}{R_0}
\newcommand{\Rtwo}{R_2}
\newcommand{\Rthree}{R_3}
\newcommand{\Rfour}{R_1}
\newcommand{\Cthree}{\mathbf{C}}
\newcommand{\piY}{\tilde{\pi}}
\newcommand{\piR}{\pi}
\newcommand{\pt}{p}
\newcommand{\exx}{X}
\newcommand{\dee}{D}
\newcommand{\regionA}{\mathfrak A}
\newcommand{\regionB}{\mathfrak B}
\newcommand{\regionC}{\mathfrak C}
\newcommand{\regionD}{\mathfrak D}
\newcommand{\regionE}{\mathfrak E}
\newcommand{\qhatpt}{{\hat{q}}}
\newcommand{\qbarpt}{{\overline{q}}}
\newcommand{\piX}{\varpi}
\newcommand{\piW}{\tilde{\pi}_1}
\newcommand{\Kmodulistack}[1]{\overline{\cal M}^{\mathrm{K}}_{#1}}
\newcommand{\piXthree}{\varpi'}
\newcommand{\ptE}{p_E}
\newcommand{\conept}{p_S}
\newcommand{\piP}{{\pi_+}}
\newcommand{\hsmall}{h}
\newcommand{\hsmallplus}{{h^+}}
\newcommand{\Rox}{R^{\mathfrak{n}}_3}
\newcommand{\Rwall}{R_{\cnot}}
\newcommand{\Gtwo}{\mathcal C_{\mathrm{non}\text{-}\mathrm{fin}}}
\newcommand{\Gthree}{\mathcal C_{\mathrm{non}\text{-}\mathrm{red}}}
\newcommand{\Gfour}{\mathcal C_{2\times A_1}}
\newcommand{\GITboundary}{\mathcal C^{\mathrm{PS}}}
\newcommand{\Disc}{\mathrm{Disc}}

\newcommand{\defi}[1]{\textsf{#1}} 

\def\Im{\text{Im}}

\let\phi\varphi
\title{The K-moduli space of a family of conic bundle threefolds}

\author{Kristin DeVleming}
\address{Department of Mathematics and Statistics, University of Massachusetts, Amherst, MA 01003-9305, USA}
\email{kdevleming@umass.edu}
\urladdr{https://people.math.umass.edu/\~{}devleming/}

\author{Lena Ji}
\address{Department of Mathematics, University of Michigan, 530 Church Street, Ann Arbor, MI 48109-1043, USA}
\email{lenaji.math@gmail.com}
\urladdr{https://www-personal.umich.edu/\~{}lenaji}

\author{Patrick Kennedy-Hunt}
\address{Centre for Mathematical Sciences, University of Cambridge, Wilberforce Road CB3 0WA, UK}
\email{pfk21@cam.ac.uk}
\urladdr{https://www.patrickkennedyhunt.com/}

\author{Ming Hao Quek}
\address{Department of Mathematics, 450 Jane Stanford Way, Building 380, Stanford, CA 94305-2125, USA}
\email{mhquek@stanford.edu}
\urladdr{https://www.minghaoquek.com/}


\thanks{During the preparation of this article, K.D. was partially supported by NSF grant DMS-2302163, L.J. was partially supported by NSF MSPRF grant DMS-2202444, and P.K.-H. was supported by an EPSRC Studentship, reference 2434344, and a London Mathematical Society early career fellowship.}
 \excludecomment{detail}

\begin{document}

\maketitle

\begin{abstract}
    We describe the 6-dimensional compact K-moduli space of Fano threefolds in deformation family \textnumero2.18. These Fano threefolds are double covers of $\mathbb P^1\times\mathbb P^2$ branched along smooth $(2,2)$-surfaces, and Cheltsov--Fujita--Kishimoto--Park proved that any smooth Fano threefold in this family is K-stable. A member of family \textnumero2.18 admits the structures of a conic bundle and a quadric surface bundle. We prove that K-polystable limits of these Fano threefolds admit conic bundle structures, but not necessarily del Pezzo fibration structures.
    
    We study this K-moduli space via the moduli space of log Fano pairs $(\mathbb P^1\times\mathbb P^2, c R)$ for $c=\frac{1}{2}$ and $R$ a $(2,2)$-divisor, which we construct using wall-crossings. In the case where the divisor is proportional to the anti-canonical divisor, the first author, together with Ascher and Liu, developed a framework for wall crossings in K-moduli and proved that there are only finitely many walls, which occur at rational values of the coefficient $c$. This paper constructs the first example of wall-crossing in K-moduli in the non-proportional setting, and we find a wall at an irrational value of $c$.

    In particular, we obtain explicit descriptions of the GIT and K-moduli spaces (for $c\leq\frac{1}{2}$) of these $(2,2)$-divisors. Furthermore, using the conic bundle structure, we study the relationship with the GIT moduli space of plane quartic curves.
\end{abstract}

\begin{detail}
    At present details are turned on and will be displayed.
\end{detail}

\section{Introduction}
Constructing moduli spaces is a fundamental question in algebraic geometry. Frequently, these moduli spaces are not compact, and it is an important question to construct geometrically meaningful compactifications. These compactifications allow one to study how algebro-geometric structures behave under degeneration, as it is natural to ask what geometric properties are preserved by the objects parametrized by the boundary of the compactification.

In this paper, we study the compactification of the moduli space of a family of Fano threefolds. For Fano varieties, the recent work of many authors has culminated in the K-moduli theorem, which shows the existence of a projective good moduli space for K-polystable Fano varieties of fixed dimension and anticanonical volume \cite{Jiang20,LWX21,CodogniPatakfalvi,BlumXu19,ABHLX20,BLX22,Xuquasimonomial,XZ20positivity,XZ21uniqueness,BHLLX21,LXZ-finite-generation}.
This compactification is constructed using K-stability, which originated in complex geometry and the study of the existence of K\"ahler--Einstein metrics. A natural question is what other geometric properties this compactification preserves.
Our paper studies the closure \(\Kmodulitwoeighteen\) of Fano family \textnumero2.18 in the moduli space of K-polystable Fano threefolds with anticanonical volume 24. Members of \textnumero2.18 are smooth Fano threefolds \(Y\) obtained as double covers of \(\bb P^1\times\bb P^2\) branched along \((2,2)\)-divisors, and they were recently shown to to be K-stable \cite{CFKP23}. These threefolds have the following geometric properties:
\begin{enumerate}[label=(\roman*)]
    \item\label{item:rational} They are rational varieties.
    \item\label{item:double-cover} They are double covers of \(\bb P^1\times\bb P^2\).
    \item\label{item:conic-bundle} They are standard conic bundles via the second projection \(Y\to\bb P^2\) (Definition~\ref{defn:conic-bundle}), and the discriminant curve is a plane quartic with at worst \(A_1\) singularities.
    \item\label{item:quadric-surface-bundle} They are quadric surface bundles via the first projection \(Y\to\bb P^1\).
\end{enumerate}
We construct the compact moduli space \(\Kmodulitwoeighteen\), which is a 6-dimensional projective variety, and we explicitly describe its members. Then, we consider whether the above geometric properties persist for all K-polystable Fano threefolds parametrized by this moduli space. Some of these properties follow from results in the literature. For example, \ref{item:rational} is answered by results of de Fernex--Fusi and Hacon--M\textsuperscript{c}Kernan showing that rationality specializes in families of klt threefolds \cite{deFernexFusi,HaconMcKernanRC}. Property~\ref{item:double-cover}, i.e. that K-polystable degenerations of \(\Kmodulitwoeighteen\) will still admit double cover structures (not necessarily of \(\bb P^1\times\bb P^2\)), is suggested by previous results on K-stability under finite covers, e.g. \cite{LiuZhu22,ADL-quartic-K3}. In order to to fully answer~\ref{item:double-cover} and to address~\ref{item:conic-bundle} and~\ref{item:quadric-surface-bundle}, we need to use our explicit description of the members of this moduli space.

Our approach to studying \(\Kmodulitwoeighteen\) is via wall-crossings in K-moduli. We study the K-moduli spaces $\Kmodulispace{c}$ of the log Fano pairs \((\bb P^1\times\bb P^2, cR)\) where \(R\) is a \((2,2)\)-surface and \(c\in(0,\tfrac{1}{2}]\cap\bb Q\). The first author, together with Ascher and Liu, developed the framework for wall crossings in K-moduli in the \emph{proportional} case, i.e. for log Fano pairs \((X, cD)\) with \(D\) a rational multiple of the anticanonical divisor \(-K_X\) \cite{ADL19}. They showed that, in the proportional case, there are finitely many walls as \(c\) varies, and these walls all occur at rational numbers. The \emph{non-proportional} case, however, is much more difficult, as many techniques from K-stability (such as interpolation) are not available in this setting.
The present paper is the first to study wall crossings in K-moduli in the non-proportional setting, and we find a unique wall crossing which occurs at an \emph{irrational} value of $c$.

In dimension 2, the K-moduli spaces of smooth(able) del Pezzo surfaces are understood \cite{MabuchiMukai,OdakaSpottiSun}. There are also some examples worked out in higher dimensions. In some cases, the K-moduli space coincides with the compactification from geometric invariant theory (GIT) \cite{SpottiSun17,Liucubic4fold,LiuXu-cubic}. However, in general, the GIT and K-moduli spaces differ. For log Fano pairs in the proportional case mentioned above, the technique of wall-crossings has been used in \cite{ADL19,ADL-quadric,ADL-quartic-K3,GMGS21,Zhao_2023,2022arXiv221206992Z}.
For Fano threefolds, there are recent results on explicit K-moduli \cite{2022arXiv221014781A}, and several low-dimensional components of the K-moduli spaces have been described \cite{2022arXiv221209332S,2023arXiv230912518A,2023arXiv230912524C,2023arXiv230912522C}.
However, much still remains unknown in higher dimensions.
\cite{ADL-quartic-K3} finds the K-moduli space of quartic double solids using VGIT wall-crossings in the proportional case. Aside from this example,
the present paper gives the next complete example of an explicit K-moduli space that differs from GIT and where \emph{both} the parametrized Fano varieties \emph{and} the moduli space are higher dimensional.
We also mention that there is forthcoming work \cite{LiuZhao2-15} on another higher-dimensional component using different methods from ours.

\subsection{Main results} A member of family \textnumero2.18 is a smooth Fano threefold obtained as a double cover \(Y \rightarrow \mathbb{P}^1 \times \mathbb{P}^2\) branched along a (2,2)-surface $R$. As mentioned above, Cheltsov--Fujita--Kishimoto--Park recently showed smooth Fano threefolds in family \textnumero2.18 are K-stable \cite{CFKP23}. More precisely, they showed that for a smooth \((2,2)\)-surface \(R\), the log Fano pair \((\bb P^1\times\bb P^2, cR)\) is K-stable for \(c\in (0,1)\cap\bb Q\).
To state our first result, we need to introduce some notation.
\begin{itemize}
    \item Let \(\Kmodulitwoeighteen \subset M^{\text{K-ps}}_{3,24}\) denote the irreducible component of the K-moduli space parametrizing Fano threefolds of volume 24 containing the members of family \textnumero2.18.
    \item For \(c\in (0,\tfrac{1}{2}]\cap\bb Q\), we let \(\Kmodulispace{c}\) denote the irreducible component in the moduli space of K-polystable pairs of dimension 3 and volume \(3(2-2c)(3-2c)^2\) containing the log Fano pairs $(\mathbb{P}^1 \times \mathbb{P}^2,cR)$ where \(R\) is a smooth \((2,2)\)-surface.
\end{itemize}

\begin{maintheorem}\label{thm:moduli-spaces-2.18}
    Let \(\cnot\approx 0.472\) be the irrational number defined to be the smallest root of the polynomial \(10c^3-34c^2+35c-10=0\).  The following hold for \(c\in(0,\tfrac{1}{2}]\cap\bb Q\):
    \begin{enumerate}
        \item\label{item:moduli-spaces-2.18-before-wall} \(\Kmodulispace{c}\cong\GITR \coloneqq |\cal O_{\bb P^1\times \bb P^2}(2,2)|^{ss}/\!\!/\SL_2\times \SL_3 \) if and only if \(c<\cnot\).
        \item\label{item:moduli-spaces-2.18-after-wall} \(\Kmodulispace{c}\cong\Kmodulispace{1/2}\) if and only if \(\cnot < c \leq \tfrac{1}{2}\), and there is a bijective morphism \(\Kmodulispace{1/2} \to \Kmodulitwoeighteen\).
        \item\label{item:moduli-spaces-2.18-wall-crossing} There is a wall-crossing morphism \(\Kmodulispace{\cnot+\epsilon}\to\Kmodulispace{\cnot-\epsilon}\) contracting a divisor $E_\mathfrak{n}$ to point $[\Rthree]$ corresponding to the non-normal surface \(\Rthree\) defined in coordinates by \[\Rthree \coloneqq (t_0t_1y_1^2 + (t_0y_2 + t_1y_0)^2 = 0) \subset \bb P^1_{[t_0:t_1]} \times \bb P^2_{[y_0:y_1:y_2]} .\]
        The exceptional divisor $E_\mathfrak{n}$ is isomorphic to the GIT moduli space \(\GITn\) of (pointed) nodal plane quartics (see Section~\ref{sec:nodalGIT} for the definition).
    \end{enumerate}
\end{maintheorem}

In particular, we give the first construction of K-moduli spaces of log Fano pairs \((X,cD)\) with a wall-crossing at an \emph{irrational} number \(c\). As mentioned above, \cite{ADL19} showed this cannot happen in the \(D\sim_{\bb Q}-rK_X\) setting. We note that Loginov has previously studied threefold log Fano pairs with reducible boundary divisors and found examples where the stability conditions changed at irrational numbers \cite{Loginov23}. Our work is the first to construct the associated K-moduli spaces; moreover, we give the first example where the divisor \(D\) is \emph{irreducible}.

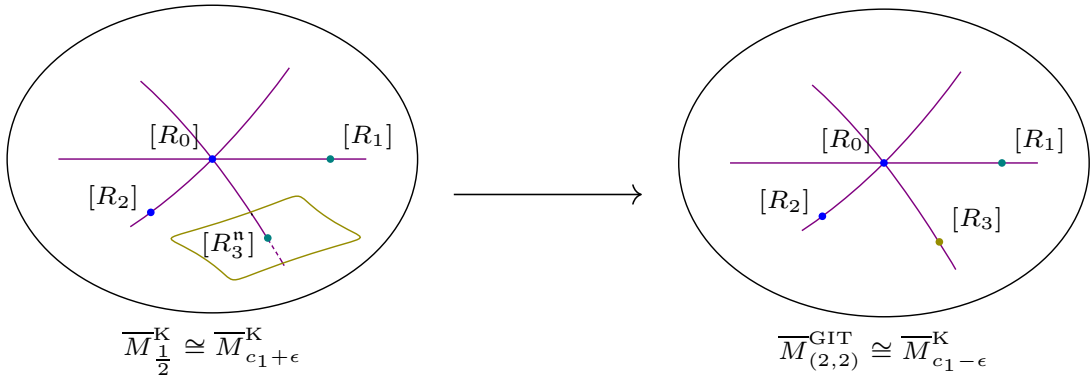
\begin{figure}[h]\label{fig:wall-crossing}
\caption{The wall-crossing morphism in Theorem~\ref{thm:moduli-spaces-2.18}. In each \(\Kmodulispace{c}\), the strictly polystable (non-stable) locus is a union of three rational curves meeting at the point \([\Rone]\), as described in Theorems~\ref{thm:GIT-polystable} and~\ref{thm:stability-on-X3}.}
\centering
\vspace{.2in}
    \adjustbox{width=.95\textwidth}{
    \begin{tikzcd}
        \begin{tabular}{c}
\begin{tikzpicture}
\draw (0,0) ellipse (2 and 1.5);
\draw[violet] (-1.5,0) to (1.5,0); 
\draw [violet] plot [smooth, tension=1] coordinates { (-.8,-.66)  (0,0) (0.75,0.89) };
\draw [violet] plot [smooth, tension=1] coordinates { (-.7,.76)  (0,0) (0.7,-1.04) };

\draw [color=olive] plot [smooth cycle, tension=2.5] coordinates { (0.52+.6,-.77+.2)  (0.52+.3,-.77-.2)  (0.52-.6,-.77-.2) (0.52-.3,-.77+.2) };
\draw[-] (0.38,-0.53) to (0.39,-0.55);
\draw[white,fill=white] (0.57,-0.83) circle (0.01);
\draw[white,fill=white] (0.61,-0.89) circle (0.01);
\draw[white,fill=white] (0.65,-0.95) circle (0.01);

\draw[blue,fill=blue] (0,0) circle (0.03);
\draw[blue,fill=blue] (-.6,-.52) circle (0.03);
\draw[teal,fill=teal] (.54,-.77) circle (0.03);
\draw[teal,fill=teal] (1.15,0) circle (0.03);
\node[above left, node font=\tiny] at (0,0) {$[\Rone]$};
\node[above left, node font=\tiny] at (-.6,-.6) {$[\Rtwo]$};
\node[left, node font=\tiny] at (.54,-.8) {$[\Rox]$};
\node[above right, node font=\tiny] at (1.15,0) {$[\Rfour]$};

\node[below,node font=\tiny] at (0,-1.55) {$\Kmodulispace{\tfrac{1}{2}} \cong \Kmodulispace{\cnot+\epsilon}$};

\end{tikzpicture}
\end{tabular} \arrow[rr] && \begin{tabular}{c}
\begin{tikzpicture}
\draw (0,0) ellipse (2 and 1.5);
\draw[violet] (-1.5,0) to (1.5,0); 
\draw [violet] plot [smooth, tension=1] coordinates { (-.8,-.66)  (0,0) (0.75,0.89) };
\draw [violet] plot [smooth, tension=1] coordinates { (-.7,.76)  (0,0) (0.7,-1.04) };

\draw[blue,fill=blue] (0,0) circle (0.03);
\draw[blue,fill=blue] (-.6,-.52) circle (0.03);
\draw[olive,fill=olive] (.54,-.77) circle (0.03);
\draw[teal,fill=teal] (1.15,0) circle (0.03);
\node[above left, node font=\tiny] at (0,0) {$[\Rone]$};
\node[above left, node font=\tiny] at (-.6,-.6) {$[\Rtwo]$};
\node[above right, node font=\tiny] at (0.54,-.77) {$[\Rthree]$};
\node[above right, node font=\tiny] at (1.15,0) {$[\Rfour]$};

\node[below,node font=\tiny] at (0,-1.55) {$\GITR \cong \Kmodulispace{\cnot-\epsilon}$};
\end{tikzpicture}
\end{tabular} \\
    \end{tikzcd}
    }
\end{figure}

\subsubsection{The GIT moduli space}
Before describing the wall-crossing in Theorem~\ref{thm:moduli-spaces-2.18}, we first give a full description of the GIT moduli space. The stable members are described as follows:
\begin{maintheorem}\label{thm:stable-members}
    For a \((2,2)\)-surface \(R\subset\bb P^1\times\bb P^2\), the following conditions are equivalent:
    \begin{enumerate}
        \item\label{item:stable-members-thm-GIT} \(R\) is GIT stable,
        \item\label{item:stable-members-thm-K} \((\bb P^1\times\bb P^2, cR)\) is K-stable for \(c\in(0,\frac{1}{2}] \cap \mathbb Q\),
        \item\label{item:stable-members-thm-A_n} Any singular point \(\pt\in R\) is an \(A_n\) singularity for some \(1\leq n\leq 7\) and satisfies the following: \begin{enumerate}\item If \(\pt\in R\) is an \(A_1\) singularity, then the fiber of \(\piR_2\colon R\to\bb P^2\) containing \(\pt\) does not contain any other singular points of \(R\). \item If \(\pt\in R\) is an \(A_n\) singularity and \(n\geq 2\), then the fiber of \(\piR_2\colon R\to\bb P^2\) containing \(\pt\) is finite and the fiber of \(\piR_1\colon R\to\bb P^1\) containing \(\pt\) is reduced. \end{enumerate}
    \end{enumerate}
\end{maintheorem}
To describe the strictly polystable members of \(\GITR\), we first define three more surfaces:
\[\Rone \coloneqq (t_0 t_1( y_0 y_2 + y_1^2) = 0), \; \Rfour \coloneqq ((t_0 + t_1)^2 y_1^2 + t_0 t_1 y_0 y_2 = 0), \; \Rtwo \coloneqq ((t_0 y_2 + t_1 y_1)(t_0 y_1 + t_1 y_0) = 0).\]
\begin{maintheorem}\label{thm:GIT-polystable}
The locus of GIT polystable, non-stable \((2,2)\)-surfaces in \(\GITR\) consists of three rational curves \(\Gfour\), \(\Gtwo\), and \(\Gthree\) that meet at a common point \([\Rone]\). For such a GIT polystable surface \(R\) and \(c\in(0,\frac{1}{2}]\cap\bb Q\), the following hold:
\begin{enumerate}
    \item\label{item:thm-GIT-polystable-not-R3} If \(R \not\cong \Rthree\), then the pair \((\bb P^1\times\bb P^2, cR)\) is K-polystable for all \(c\in (0,\frac{1}{2}] \cap \bb Q\).
    \item\label{item:thm-GIT-polystable-R3} For the surface \(\Rthree\), the pair \((\bb P^1\times\bb P^2, cR)\) is K-polystable if and only if \(0<c<\cnot\).
\end{enumerate}
The GIT strictly polystable surfaces have the following descriptions:
\begin{enumerate}\setcounter{enumi}{2}
    \item\label{item:thm-GIT-polystable-R1} \(\Rone\) is the union of a smooth \((0,2)\)-surface and two \((1,0)\)-surfaces.
    \item\label{item:thm-GIT-polystable-Gfour} If \(R\not\cong\Rone,\Rfour\) is in \(\Gfour\), then the projection \(\piR_2\colon R \to \bb P^2\) has two non-finite fibers that each have two \(A_1\) singularities. The surface \(\Rfour\in\Gfour\) has an additional fifth \(A_1\) singularity on a finite fiber of \(\piR_2\).
    \item If \(R\not\cong\Rone,\Rtwo\) is in \(\Gtwo\), then \(R\) has two \(A_2\) singularities, each of which is contained in a different non-finite fiber of the projection \(\piR_2\colon R \to \bb P^2\). The surface \(\Rtwo\in\Gtwo\) is reducible and is the union of two smooth \((1,1)\)-surfaces.
    \item\label{item:thm-GIT-polystable-Gthree} If \(R\not\cong\Rone,\Rthree\) is in \(\Gthree\), then \(R\) has two \(A_3\) singularities. Each of these singularities is contained in a finite fiber of \(\piR_2\colon R\to\bb P^2\) and in a non-reduced fiber of \(\piR_1\colon R \to \bb P^1\). The surface \(\Rthree\in\Gfour\) is the unique non-normal, irreducible GIT polystable \((2,2)\)-surface.
\end{enumerate}
Furthermore, every normal \((2,2)\)-surface with the above described singularities is necessarily GIT semistable.
\end{maintheorem}
Thus, from Theorems~\ref{thm:stable-members} and~\ref{thm:GIT-polystable}, we see that at the wall \(c = \cnot\), the birational transformation on the moduli spaces will be an isomorphism over \(\Kmodulispace{\cnot-\epsilon} \setminus [\Rthree]\). Next, to describe the exceptional locus \(E_{\mathfrak{n}}\) of the wall-crossing morphism mentioned in Theorem~\ref{thm:moduli-spaces-2.18}, we study the K-moduli spaces after this wall.

\subsubsection{After the wall at \(\cnot\approx 0.472\).}
The preimage \(E_{\mathfrak{n}}\) of $[\Rthree]$ parametrizes log Fano pairs \((\Xthree, cR)\) where \(R\) is a hypersurface in a singular toric threefold \(\Xthree\). This threefold \(\Xthree\), whose fan is described in Section \ref{sec:construction-of-X3}, can be characterized as follows:
\begin{enumerate}[label=(\alph*)]
    \item On \(\bb P(1,1,1,2)\), let \(\conept\) be the cone point, let \(\ptE\) be a smooth point, and let \(l_1\) be the line connecting \(\conept\) and \(\ptE\).
    \item Let \(W^+\to\bb P(1,1,1,2)\) be the blow-up of \(\conept\), followed by the blow-up of \(\ptE\).
    \item Let \(W^+\to\Xthree\) be the contraction of the strict transform of \(l_1\). This defines \(\Xthree\).
    \item Let \(W^+\dashrightarrow W\) be the flop of \(l_1\). Then \(W\) admits an \(\bb F_1\)-fibration structure \(\piW\colon W\to\bb P^1\), and \(W\to\Xthree\) is a toric small resolution. Moreover, \(W\) is the unique toric small resolution of \(\Xthree\) that admits a del Pezzo fibration structure.
\end{enumerate}
The threefold \(\Xthree\) has Picard rank 2, is non-\(\bb Q\)-factorial, has a single \(A_1\) singularity, and admits a smoothing to \(\bb P^1\times\bb P^2\) over \(\bb A^1\).
Let \(S\) (resp. \(E\)) denote the image in \(\Xthree\) of the exceptional divisor of the blow-up of \(\conept\) (resp. \(\ptE\)).
The singular locus of \(\Xthree\) is the intersection \(E \cap S\), and \(\Xthree\) admits a toric non-flat morphism to \(\bb P^2\) whose general fiber is \(\bb P^1\) and where one fiber is \(E\cong\bb P^2\).

Now we describe the pairs that appear on the exceptional locus \(E_{\mathfrak{n}}\) of the wall-crossing morphism in Theorem~\ref{thm:moduli-spaces-2.18}:

\begin{maintheorem}\label{thm:stability-on-X3}
Let \(\Xthree\) be the threefold defined above, and let \(R\) be a surface on \(\Xthree\) that is a degeneration of a \((2,2)\)-surface on \(\bb P^1\times\bb P^2\). If \(c < \cnot\), then \((\Xthree, cR)\) is K-unstable. Furthermore, the K-stable (resp. strictly K-polystable) members can be described as follows:
\begin{enumerate}
    \item\label{item:thm-stability-on-X3-stable} For \(c\in (\cnot,\tfrac{1}{2}]\cap\bb Q\), the pair \((\Xthree, cR)\) is K-stable if and only if the following conditions hold:
    \begin{enumerate}
    \item \(R \cap S=\emptyset\),
    \item \(R \cap E\) is a smooth conic, and
    \item Any singular point \(\pt\in R\) is an \(A_n\) singularity for some \(1 \leq n \leq 5\).
    In addition, if \(n\geq 2\), then the fiber of \(\piW|_{R_W}\colon R_W \to \bb P^1\) containing the strict transform of \(\pt\) is reduced. Here \(W \to \Xthree\) is the unique toric small resolution with a del Pezzo fibration structure \(\piW\), and \(R_W\) is the strict transform of \(R\) in \(W\).
    \end{enumerate}
    \item\label{item:thm-stability-on-X3-unique-polystable-ox} For \(c\in (\cnot,\tfrac{1}{2}]\cap\bb Q\), the pair \((\Xthree, cR)\) is K-polystable but not K-stable if and only if \(R=\Rox\) is the surface whose strict transform in \(\bb P(1,1,1,2)\) is (up to a coordinate change) defined by \(w^2=xy(z^2-xy)\). Here \([x:y:z:w]\) are coordinates on \(\bb P(1,1,1,2)\) where \(w\) has weight 2 and \(\ptE = [0:0:1:0]\).
\end{enumerate}
\end{maintheorem}

As a consequence of Theorem~\ref{thm:stability-on-X3}, we obtain the following description of all the double covers of \(\Xthree\) arising after the wall. In particular, by Theorem~\ref{thm:moduli-spaces-2.18}, this shows that there are K-polystable degenerations in \(\Kmodulitwoeighteen\) that do \emph{not} admit del Pezzo fibration structures.
That is, property~\ref{item:conic-bundle} is preserved in the K-moduli compactification of family \textnumero2.18, but property~\ref{item:quadric-surface-bundle} is not:
\begin{maincorollary}\label{cor:double-covers-of-X3}
Let \((\Xthree,\tfrac{1}{2}R)\) be a K-polystable pair in \(\Kmodulispace{1/2}\), and let \(Y\) be the double cover of \(\Xthree\) branched along \(R\). Then \(Y\) is a Fano threefold with Gorenstein terminal singularities and Picard rank 2. The two extremal contractions of \(Y\) are:
\begin{enumerate}
    \item A non-flat conic bundle \(Y\to\Xthree\to\bb P^2\), where one fiber is a smooth quadric surface, and the remaining fibers are all conics; and
    \item The contraction of the preimage of \(S\) (the image is the ample model of the ramification divisor \(R\)).
\end{enumerate}
\end{maincorollary}

\subsection{The conic bundle structure}
As mentioned above, members of family \textnumero2.18 are conic bundles by the second projection \(Y\to\bb P^2\). Explicitly, if the \((2,2)\)-surface \(R\) is defined by \[t_0^2 Q_1 + 2t_0 t_1 Q_2 + t_1^2 Q_3\] in \(\bb P^1_{[t_0:t_1]}\times\bb P^2\), then the \defi{discriminant} locus (parametrizing the singular fibers of \(Y\)) is the plane quartic curve defined by $Q_1Q_3 - Q_2^2$. This assignment induces a rational map
\begin{equation}\label{eqn:discriminant-map-GIT}\GITR \dashrightarrow \GITquartic \coloneqq |\cal O_{\bb P^2}(4)|^{ss}/\!\!/\SL_3 \end{equation}
to the GIT moduli space of quartic plane curves. 
This rational map is \emph{not} a morphism---if \(\Delta\) is GIT (semi)stable then so is \(R\) (Lemma~\ref{lem:GIT-stability-Delta-R}), but the converse does not hold in general. We show:

\begin{maintheorem}\label{thm:wall-crossing-resolves-delta} 
   The rational map $\GITR \dashrightarrow \GITquartic$ in~\eqref{eqn:discriminant-map-GIT} is defined away from a 2-dimensional locus. This locus consists of GIT stable \((2,2)\)-surfaces whose discriminant curve is the union of a line and a cubic meeting in an \(A_5\) singularity, together with the point \([\Rthree]\). After the wall-crossing, the indeterminacy locus of \(\Disc\colon \Kmodulispace{1/2} \dashrightarrow \GITquartic\) does not meet the strictly polystable locus. Furthermore, the exceptional divisor \(E_{\mathfrak{n}}\) of the wall-crossing morphism in Theorem~\ref{thm:moduli-spaces-2.18} meets the indeterminacy locus at a single point and maps birationally to the nodal locus inside \(\GITquartic\).
\end{maintheorem}
The morphism \(\Kmodulispace{1/2} \to \GITquartic\) is generically finite of degree 63, but it is not finite, see Remark~\ref{rem:discriminant-morphism-not-finite}. Furthermore, this morphism does not preserve the boundary: a K-stable member of \(\Kmodulispace{1/2}\) can have GIT strictly semistable discriminant curve.
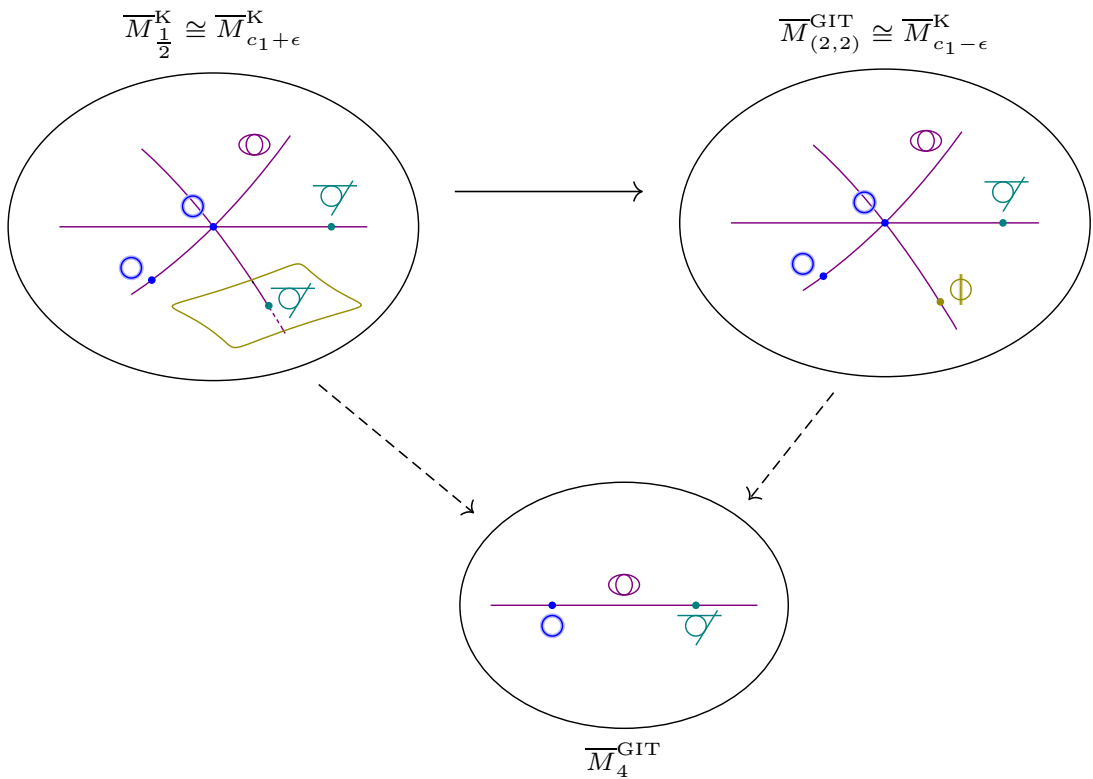
\begin{figure}[h]\label{fig:wall-crossing-discrininant}
\caption{The wall-crossing morphism and the rational map to \(\GITquartic\) (Theorem~\ref{thm:wall-crossing-resolves-delta}). In the diagram, the associated quartic curve \(\Delta\) is drawn for each strictly polystable surface.}
\centering
    \adjustbox{width=.95\textwidth}{
    \begin{tikzcd}
        \begin{tabular}{c}
\begin{tikzpicture}
\draw (0,0) ellipse (2 and 1.5);
\draw[violet] (-1.5,0) to (1.5,0); 
\draw [violet] plot [smooth, tension=1] coordinates { (-.8,-.66)  (0,0) (0.75,0.89) };
\draw [violet] plot [smooth, tension=1] coordinates { (-.7,.76)  (0,0) (0.7,-1.04) };

\draw [color=olive] plot [smooth cycle, tension=2.5] coordinates { (0.52+.6,-.77+.2)  (0.52+.3,-.77-.2)  (0.52-.6,-.77-.2) (0.52-.3,-.77+.2) };
\draw[-] (0.38,-0.53) to (0.39,-0.55);
\draw[white,fill=white] (0.57,-0.83) circle (0.01);
\draw[white,fill=white] (0.61,-0.89) circle (0.01);
\draw[white,fill=white] (0.65,-0.95) circle (0.01);

\draw[violet] (0.4, 0.8) ellipse (.08 and .1);
\draw[violet] (0.4, 0.8) ellipse (.15 and .1);

\draw[blue!30, very thick] (-0.2, 0.2) ellipse (.1 and .1);
\draw[blue] (-0.2, 0.2) ellipse (.1 and .1);

\draw[blue!30, very thick] (-0.8, -.4) ellipse (.1 and .1);
\draw[blue] (-0.8, -.4) ellipse (.1 and .1);

\draw[teal] (1.15,0.3) circle (0.1);
\draw[teal] (0.97,0.4) to (1.4,0.4);
\draw[teal] (1.15,0.11) to (1.36,0.45);

\draw[teal] (0.74,-0.7) circle (0.1);
\draw[teal] (0.56,-0.6) to (0.99,-0.6);
\draw[teal] (0.74,-.89) to (0.95,-0.55);


\draw[blue,fill=blue] (0,0) circle (0.03);
\draw[blue,fill=blue] (-.6,-.52) circle (0.03);
\draw[teal,fill=teal] (.54,-.77) circle (0.03);
\draw[teal,fill=teal] (1.15,0) circle (0.03);

\node[above,node font=\tiny] at (0,1.55) {$\Kmodulispace{\tfrac{1}{2}} \cong \Kmodulispace{\cnot+\epsilon}$};

\end{tikzpicture}
\end{tabular} \arrow[rr] \arrow[dashed,rrd, start anchor={[xshift=-8ex,yshift=-2ex]},end anchor={[xshift=-24ex,yshift=7ex]}] &  & \begin{tabular}{c}
\begin{tikzpicture}
\draw (0,0) ellipse (2 and 1.5);
\draw[violet] (-1.5,0) to (1.5,0); 
\draw [violet] plot [smooth, tension=1] coordinates { (-.8,-.66)  (0,0) (0.75,0.89) };
\draw [violet] plot [smooth, tension=1] coordinates { (-.7,.76)  (0,0) (0.7,-1.04) };

\draw[violet] (0.4, 0.8) ellipse (.08 and .1);
\draw[violet] (0.4, 0.8) ellipse (.15 and .1);

\draw[blue!30, very thick] (-0.2, 0.2) ellipse (.1 and .1);
\draw[blue] (-0.2, 0.2) ellipse (.1 and .1);

\draw[blue!30, very thick] (-0.8, -.4) ellipse (.1 and .1);
\draw[blue] (-0.8, -.4) ellipse (.1 and .1);

\draw[teal] (1.15,0.3) circle (0.1);
\draw[teal] (0.97,0.4) to (1.4,0.4);
\draw[teal] (1.15,0.11) to (1.36,0.45);

\draw[olive!30, very thick] (0.74, -.5) to (0.74, -.8);
\draw[olive] (0.74, -.5) to (0.74, -.8);
\draw[olive] (0.74, -.65) ellipse (.1 and .1);

\draw[blue,fill=blue] (0,0) circle (0.03);
\draw[blue,fill=blue] (-.6,-.52) circle (0.03);
\draw[olive,fill=olive] (.54,-.77) circle (0.03);
\draw[teal,fill=teal] (1.15,0) circle (0.03);

\node[above,node font=\tiny] at (0,1.55) {$\GITR \cong \Kmodulispace{\cnot-\epsilon}$};

\end{tikzpicture}
\end{tabular} \arrow[dashed,d, start anchor={[xshift=-3ex,yshift=0ex]},end anchor={[xshift=-8ex,yshift=-2ex]}] \\
        & & \hspace{-2in}\begin{tabular}{c}
\begin{tikzpicture}
\draw (0,0) ellipse (1.6 and 1.2);
\draw[violet] (-1.3,0) to (1.3,0); 

\draw[violet] (0, 0.2) ellipse (.08 and .1);
\draw[violet] (0, 0.2) ellipse (.15 and .1);

\draw[blue!30, very thick] (-0.7, -0.2) ellipse (.1 and .1);
\draw[blue] (-0.7, -0.2) ellipse (.1 and .1);

\draw[teal] (0.7,-0.2) circle (0.1);
\draw[teal] (0.52,-0.1) to (0.95,-0.1);
\draw[teal] (0.7,-0.39) to (0.91,-0.05);

\draw[blue,fill=blue] (-.7,0) circle (0.03);
\draw[teal,fill=teal] (0.7,0) circle (0.03);

\node[below,node font=\tiny] at (0,-1.25) {$\GITquartic$};

\end{tikzpicture}
\end{tabular} \\
    \end{tikzcd}
    }
\end{figure}

Thus, our above results show that if \(Y\) is a K-polystable degeneration of a smooth Fano threefold in family \textnumero2.18, then the following properties hold:
\begin{enumerate}[label=(\roman*)]
    \item \(Y\) is rational.
    \item \(Y\) is a double cover of \(\bb P^1\times\bb P^2\) or \(\Xthree\).
    \item \(Y\) admits a (not necessarily flat) conic bundle structure. The discriminant curve is a plane quartic with only \(A_n\) singularities, or a GIT polystable degeneration of such a quartic.
    \item If \(Y\) is a double cover of \(\Xthree\), then \(Y\) does not admit a del Pezzo fibration structure.
\end{enumerate}

\subsection{Family \textnumero2.18 and relation to previous results}
A \((2,2)\)-divisor in \(\bb P^1\times\bb P^2\) is defined by an equation of the form \(t_0^2 Q_1 + 2t_0 t_1 Q_2 + t_1^2 Q_3\) for \(Q_i\in H^0(\bb P^2,\cal O(2))\), and the discriminant curve \(\Delta\) of the associated conic bundle is defined by \[\det\begin{pmatrix} Q_1 & Q_2 \\ Q_2 & Q_3 \end{pmatrix}.\]
Furthermore, there is an associated \'etale double cover \(\Deltatilde\) of \(\Delta\) whose function field is obtained by adjoining \(\sqrt{Q_1}\) to the function field of \(\Delta\). The curves \(\Delta,\Deltatilde\) are unchanged by the action of \(\SL_2\times\SL_3\) \cite[Theorem 4.5]{FJSVV}. Conversely, a \(2\times 2\) matrix of quadratic forms as above defines a \((2,2)\)-divisor in \(\bb P^1\times\bb P^2\). Therefore, we see that Fano threefolds in family \textnumero2.18 are intimately related to the study of quadratic determinantal representations of plane quartics. A smooth plane quartic is classically known to admit 63 such determinantal representations, which are in bijection with the nontrivial 2-torsion divisor classes on \(\Delta\) \cite{Hesse55} (see also \cite[Section 6.2.2]{Dolgachev-CAG}); see \cite{Wall91,Scheiderer10} for extensions to singular quartics.

The above discussion shows that \(\GITR\), and hence \(\Kmodulispace{1/2}\), is birational to the genus 3 Prym moduli space parametrizing isomorphism classes of smooth connected genus 3 curves together with a nontrivial 2-torsion class. In particular, these moduli spaces are rational \cite{Dolgachev-Prym}.

Over \(\bb C\), any standard conic bundle over \(\bb P^2\) whose discriminant is a smooth quartic is birational (over \(\bb P^2\)) to a member of family \textnumero2.18, by purity for the Brauer group \cite[Theorems 3.7.3 and 6.1.3]{CTS-Brauer-book}. Therefore, family \textnumero2.18 arises naturally in the study of conic bundle threefolds, and by Theorems~\ref{thm:moduli-spaces-2.18} and~\ref{thm:wall-crossing-resolves-delta}, \(\Kmodulitwoeighteen\) can be viewed as a compactification of a moduli space for degree 4 conic bundles over \(\bb P^2\).
Family \textnumero2.18 is interesting from the point of view of rationality, especially over fields \(k\) that are not algebraically closed.
Namely, its members are all rational over \(\bb C\); however, over arithmetically interesting fields, these conic bundles exhibit surprising \(k\)-rationality behaviors \cite{FJSVV,JJ-deg4}, and their rationality is not yet fully understood, even over \(k=\bb R\).

\subsection*{Outline}
First, we collect preliminary results in Section~\ref{sec:preliminaries} on conic bundles, GIT stability, K-stability, the relationship between the \((2,2)\)-surface \(R\subset\bb P^1\times\bb P^2\) and the plane quartic \(\Delta\), and GIT of plane quartics.
Next, in Section~\ref{sec:GIT-stability-section}, we study GIT stability for \((2,2)\)-surfaces. In particular, we show that surfaces with worse than \(A_n\) singularities cannot be GIT stable, and we give a full list of the GIT strictly semistable surfaces.

Starting from Section~\ref{sec:stability-of-A_n}, we study K-stability. In Section~\ref{sec:stability-of-A_n}, we prove the K-stability of \((2,2)\)-surfaces with certain \(A_n\) singularities, by using the Abban--Zhuang method to compute lower bounds for \(\delta\)-invariants.
In Section~\ref{sec:GIT-strictly-polystable}, we find polystable representatives of the GIT strictly semistable surfaces described in Section~\ref{sec:GIT-stability-section}. We prove that they are GIT polystable by showing they are K-polystable, via computing lower bounds for the \(G\)-equivariant \(\delta\)-invariant.
Then, in Section~\ref{sec:wall-crossing-c_0}, we study the wall-crossing at the irrational value \(\cnot\approx 0.47\). We show that \((\bb P^1\times\bb P^2, c\Rthree)\) is K-unstable if and only if \(c>\cnot\), and we construct the K-moduli replacement \(\Xthree\) that appears after the wall. We describe the degeneration from \(\bb P^1\times\bb P^2\) to \(\Xthree\) and find some preliminary restrictions on the surfaces on \(\Xthree\) that arise as degenerations of \((2,2)\)-surfaces.
In Section~\ref{sec:K-stability-Xthree}, we study K-stability of surfaces in \(\Xthree\) that appear after the wall, by computing lower bounds for \(\delta\)-invariants using the Abban--Zhuang method.
Next, we study GIT of singular plane quartics in Section~\ref{sec:GIT-higher-An-quartics}: we first collect results on classical GIT of plane quartics with higher \(A_n\) singularities, and then in Section~\ref{sec:nodalGIT} we define the GIT of (pointed) nodal quartics that appears in Theorem~\ref{thm:moduli-spaces-2.18}.
Finally, in Section~\ref{sec:proofs-of-main-thms}, we put together the results proven in the preceding sections and prove the main theorems of the paper. We end the paper with Appendix~\ref{appendix:volume-lemmas}, where we collect some results on volumes of divisors on threefolds, which we use in the computations in Sections~\ref{sec:stability-of-A_n} and~\ref{sec:K-stability-Xthree}.

\subsection*{Acknowledgements}
This project began at the 2022 AGNES Summer School on Higher Dimensional Moduli, which was supported by the National Science Foundation, under Grant No. DMS-1937636, and Brown University. We thank the organizers of the summer school: Dan Abramovich, Melody Chan, Brendan Hassett, Eric Larson, and Isabel Vogt. We also thank Ivan Cheltsov, Paul Hacking, Brendan Hassett, Anne-Sophie Kaloghiros, Jesus Martinez Garcia, Calla Tschanz, Junyan Zhao, and Ziquan Zhuang for helpful conversations.
We are especially thankful to Igor Dolgachev for helpful comments on a preliminary draft and Yuchen Liu for suggesting the approach in Section~\ref{sec:nodalGIT}.

\tableofcontents

\section{Preliminaries}\label{sec:preliminaries}
We work over $\mathbb{C}$. 
We do not require that varieties be irreducible or reduced unless explicitly stated. In particular, a \defi{\((2,2)\)-surface} in \(\bb P^1\times\bb P^2\) will mean an effective divisor in \(|\cal O_{\bb P^1\times\bb P^2}(2,2)|\). We will frequently abbreviate \(|\cal O_{\bb P^1\times\bb P^2}(a,b)|=|(a,b)|\) when the context is clear.

\subsection{Notation and definitions}\label{sec:notation} 
A \defi{log Fano pair} \((X,D)\) consists of a normal irreducible variety \(X\) and an effective \(\bb Q\)-divisor \(D\) such that \(-(K_X+D)\) is \(\bb Q\)-Cartier and ample. A \defi{klt log Fano pair} is a log Fano pair \((X,D)\) that is also klt.

\begin{definition}\label{defn:conic-bundle}
    A \defi{conic bundle} is a triple \((Y,S,\piY)\) consisting of a smooth connected variety \(S\) and a proper morphism from a normal variety \(\piY\colon Y \to S\) whose generic fiber is a smooth irreducible genus 0 curve. A conic bundle is \defi{standard} if \(Y\) is smooth and \(\piY\) is flat with relative Picard rank \(1\).
\end{definition}
If \((Y,S,\piY)\) is a conic bundle such that \(Y\) admits an embedding into \(\bb P_S(\cal E)\) for some locally free sheaf \(\cal E\) of rank 3, then \(Y\) is locally defined by a ternary quadratic form on \(S\) \cite[1.7]{Sarkisov-conic-bundles}.

\begin{definition}\label{defn:discriminant}
    For a conic bundle that is locally given by a ternary quadratic form over the base \(S\), we define the \defi{discriminant} \(\Delta\subset S\) to be the divisor in \(S\) locally given by the vanishing of the determinant of the corresponding \(3\times 3\) symmetric matrix. If \((Y,S,\piY)\) is a conic bundle that is locally given by a ternary quadratic form over a subset \(U\subset S\) whose complement has codimension \(\geq 2\), we define the discriminant by taking the closure in \(S\) of the discriminant of \((Y_U,U,\piY|_{Y_U})\).
\end{definition}

If \(\piY\) is flat and \(Y\) is smooth, then \(\cal E \coloneqq \piY_* (\omega_Y^{-1})\) is locally free, and \(\omega_Y^{-1}\) defines an embedding \(Y\hookrightarrow\bb P_S(\cal E)\); in this case, \(\Delta\) is reduced \cite[\S1]{Sarkisov-conic-bundles}.

If \(Y\) is obtained as the double cover of \((X,\frac{1}{2}R)\) for some divisor \(R\) in \(X\) that maps generically finitely with degree 2 to \(S\), then we will also say that \(\Delta\) is the \defi{discriminant} of \((X,\frac{1}{2} R)\).

Fix once and for all coordinates \(([t_0:t_1],[y_0:y_1:y_2])\) on \(\bb P^1\times\bb P^2\). We will use these coordinates throughout the paper (except in Section~\ref{sec:GIT-higher-An-quartics}). A (2,2)-surface \(R\) on $\mathbb{P}^1 \times \mathbb{P}^2$ may be written \[R=(t_0^2 Q_1 + 2 t_0 t_1 Q_2 + t_1^2 Q_3=0)\] where $Q_i$ are quadrics in $y_0,y_1,y_2$. Let \(Y\coloneqq Y_R\to\bb P^1\times\bb P^2\) be the double cover branched along \(R\). For \(i=1,2\), we will use \(\piR_i\colon R\to\bb P^i\), \(\piY_i\colon Y\to\bb P^i\), and \(\piX_i\colon \bb P^1\times\bb P^2 \to \bb P^i\) to denote the projections.
 
Fibers of \(\piR_2\colon R\to\bb P^2\) are quadrics in $\mathbb{P}^1$ associated to the bilinear form with matrix  \[\begin{pmatrix} Q_1 & Q_2 \\ Q_2 & Q_3\end{pmatrix}.\] In particular, the {branch locus} of $\piR_2$ is the closed subscheme \(\Delta\coloneqq (Q_1 Q_3-Q_2^2=0)\) of $\mathbb{P}^2$. If the polynomial \(Q_1Q_3-Q_2^2\) is not uniformly zero, then \(\Delta\subset\bb P^2\) is a plane quartic and \(\piR_2\) is generically finite. The fiber of \(\piR_1\colon R\to\bb P^1\) over $[t_0:t_1]$ is defined by the quadratic form \(t_0^2 Q_1 + 2 t_0 t_1 Q_2 + t_1^2 Q_3\); thus, each fiber of \(\piR_1\) is either \(\bb P^2\) or a plane conic.
 
For a smooth \((2,2)\)-surface \(R\), the second projection \(\piY_2\colon Y\to\bb P^2\) is a standard conic bundle whose discriminant curve \(\Delta=  (Q_1 Q_3-Q_2^2=0)\) has at worst \(A_1\) singularities \cite[Proof of Proposition 1.8.5]{Sarkisov-conic-bundles}.

\subsection{GIT stability and the GIT moduli space} The reader is referred to \cite{mumford1977stability,MR1304906,thomas2006notes} for background and exposition on geometric invariant theory (GIT). There is a natural action of $G = \SL_2 \times \SL_3$ on the 17-dimensional linear system $|(2,2)|$ of $\mathbb{P}^1 \times \mathbb{P}^2$. We use the linearization from the line bundle \(\cal O_{\bb P^1\times\bb P^2}(1,1)\).
\begin{detail}
    Explicitly, the linear action of \(\SL_2\times\SL_3\) on \(H^0(\bb P^1\times\bb P^2,\cal O(2,2))\) is given as follows. For \[g = \begin{pmatrix} r_{00} & r_{01} \\ r_{01} & r_{11} \end{pmatrix}\times\begin{pmatrix} s_{00} & s_{01} & s_{02} \\ s_{01} & s_{11} & s_{12} \\ s_{02} & s_{12} & s_{22} \end{pmatrix}\in \SL_2\times\SL_3\] and a bidegree \((2,2)\) polynomial \(f = t_0^2 Q_1(y_0,y_1,y_2) + 2t_0t_1 Q_2(y_0,y_1,y_2) + t_1^2 Q_3(y_0,y_1,y_2)\), where the \(Q_i\in\bb C[y_0,y_1,y_2]_2\) are quadratic forms, then 
    \[ \resizebox{1\textwidth}{!}{ $ \begin{array}{rll}
    g^*f =& (r_{00} t_0 + r_{01} t_1)^2 Q_1(s_{00}y_0 + s_{01} y_1 + s_{02} y_2, s_{10}y_0 + s_{11} y_1 + s_{12} y_2, s_{20}y_0 + s_{21} y_1 + s_{22} y_2) \\ \hphantom{=}& + 2(r_{00} t_0 + r_{01} t_1)(r_{10} t_0 + r_{11} t_1) Q_2(s_{00}y_0 + s_{01} y_1 + s_{02} y_2, s_{10}y_0 + s_{11} y_1 + s_{12} y_2, s_{20}y_0 + s_{21} y_1 + s_{22} y_2) \\ \hphantom{=}& + (r_{10} t_0 + r_{11} t_1)^2 Q_3(s_{00}y_0 + s_{01} y_1 + s_{02} y_2, s_{10}y_0 + s_{11} y_1 + s_{12} y_2, s_{20}y_0 + s_{21} y_1 + s_{22} y_2).
\end{array} $} \]
\end{detail}
The GIT quotient is a projective variety of dimension $6$ and provides a first approximation to the component of the K-moduli space in which we are interested.

Orbits of points in $|(2,2)|$ under the action of $G$ fall into one of three types: stable, strictly semistable, or unstable. Points of the GIT moduli space biject with polystable orbits: whilst every stable orbit is polystable, a semistable orbit is polystable if it is closed in the semistable locus. 
Stability properties of an orbit can be established through the \defi{Hilbert--Mumford criterion}, which states that, for a point \([R]\in|(2,2)|\) corresponding to \(R=(f=0)\),
\begin{itemize}
    \item \([R]\in|(2,2)|\) is GIT semistable \(\iff\) \(\mu(f,\lambda)\leq 0\) for all 1-parameter subgroups \(\lambda\) of \(\SL_2\times\SL_3\) \(\iff\) $\lim_{t \to 0} \lambda_t \cdot f \neq 0$ for all 1-parameter subgroups \(\lambda\);
    \item \([R]\in|(2,2)|\) is GIT stable \(\iff\) \(\mu(f,\lambda) < 0\) for all 1-parameter subgroups \(\lambda\) of \(\SL_2\times\SL_3\) \(\iff\) $\lim_{t \to 0} \lambda_t \cdot f$ does not exist for all 1-parameter subgroups \(\lambda\).
\end{itemize}
A \defi{\(1\)-parameter subgroup} (1-PS) of \(\SL_2\times\SL_3\) is a homomorphism \(\lambda\colon \bb G_m\to\SL_2\times\SL_3\). A 1-PS is said to be \defi{normalized} if it has the form \(\lambda\colon \alpha\mapsto\diag(\alpha^{r_0},\alpha^{r_1})\times\diag(\alpha^{s_0},\alpha^{s_1},\alpha^{s_2})\) with \(r_0\leq r_1\)\footnote{Our convention on the inequalities of the weights is the opposite of that used by many authors, e.g. \cite{Mumford-Fogarty}. Hence, the inequalities in the Hilbert--Mumford criterion are also reversed.}, \(s_0\leq s_1\leq s_2\), and \(r_0+r_1=s_0+s_1+s_2=0\); we will abbreviate such a normalized 1-PS by \(\diag(r_0,r_1)\times\diag(s_0,s_1,s_2)\). For a normalized 1-PS \(\lambda=\diag(r_0,r_1)\times\diag(s_0,s_1,s_2)\) and a \((2,2)\)-surface \(R\) defined by \(f\),
the Hilbert--Mumford weight of \([R]\) at \(\lambda\) is \[\mu(f,\lambda) = \min\{ (2-i)r_0 + i r_1 + (2-j-k)s_0 + j s_1 + k s_2 \mid \coeff_f(t_0^{2-i} t_1^i y_0^{2-j-k}y_1^j y_2^k)\neq 0\}.\]
Every 1-PS is conjugate to a normalized one; this can be used to define \(\mu(f,\lambda)\) for a non-normalized 1-PS \(\lambda\) (after a change of coordinates on \(f\)).

\begin{remark}\label{rem:PR-GIT}
    The preprint \cite{PR-GIT} studies the GIT moduli space of \((2,2)\)-surfaces in \(\bb P^1\times\bb P^2\), but it contains several errors. In general, the arguments that certain \((2,2)\)-surfaces are GIT unstable are correct. However, the arguments asserting certain \((2,2)\)-surfaces are GIT (semi)stable are incorrect.
    The primary error is that this preprint works in a fixed coordinate system, often fixing the point \(([1:0],[1:0:0])\) in the surface, and then \emph{only} considers \emph{normalized} 1-parameter subgroups in this fixed coordinate system; this begins on \cite[page 4]{PR-GIT} and continues throughout.
    In particular, \cite[Section 8]{PR-GIT} claims to describe the strictly semistable locus in the GIT moduli space as the union of three rational curves \(\overline{\Gamma_2},\overline{\Gamma_3},\overline{\Gamma_4}\). Their description of \(\Gamma_4\) on \cite[page 22]{PR-GIT} contains the \((2,2)\)-surface with equation \(x_0^2(2y_1^2+y_0 y_2) + x_0 x_1(2y_1^2+2y_0 y_2) + x_1^2(y_1^2+y_0 y_2)\); after the coordinate change replacing \(x_1\) by \(x_1-x_0\), the equation of this surface becomes \(x_0^2 y_1^2 + x_1^2(y_1^2 + y_0y_2)\), and the 1-PS \(\diag(-2,2)\times\diag(-6,3,3)\) shows this surface is GIT unstable. We were unable to follow their assertion that the image of \(\Gamma_4\) in the GIT quotient is \(\overline{\Gamma_4}\) as described on \cite[page 23]{PR-GIT}, specifically for the surfaces with \(a_{02}=c_{02}=0\).  For \(\overline{\Gamma_2}\) and \(\overline{\Gamma_3}\), their descriptions turn out to be accurate although their arguments are not sufficient to show that the members of these sets are GIT semistable, as they only show that \(\mu(f,\lambda)\leq 0\) for normalized 1-PS's \(\lambda\) in their fixed coordinate system. Furthermore, this preprint does not mention GIT polystability. We use different arguments to prove GIT semistability (and moreover polystability) of these surfaces in Sections~\ref{sec:GIT-strictly-ss} and~\ref{sec:GIT-strictly-polystable}.

    In addition, \cite[Section 4]{PR-GIT} writes without justification that not all smooth \((2,2)\)-surfaces \(R\) are GIT stable; however, when \(R\) is smooth \cite[Theorem 2.1]{CFKP23} shows that the pair \((\bb P^1\times\bb P^2, cR)\) is K-stable for \(c\in(0,1)\cap\bb Q\), implying that \(R\) is GIT stable by Theorem~\ref{thm:K-implies-GIT}. Furthermore, \cite[Proposition 4.2]{PR-GIT} states that singular GIT stable surfaces can only have \(A_1,A_2,A_3\), or non-isolated singularities; however, the preprint does not give examples with non-isolated singularities, and in fact one can show that surfaces with non-isolated singularities cannot be GIT stable (Lemma~\ref{lem:non-A_n-not-GIT-stable}).

    In some parts of Section~\ref{sec:GIT-stability-section}, we do use arguments and results from \cite{PR-GIT}. We indicate each time we do so, and we have verified that the relevant parts that we cite are correct.
\end{remark}

\subsection{K-stability} While originally introduced to study the existence of K\"ahler--Einstein metrics on Fano varieties, the algebro-geometric interpretation of K-stability has been the main tool to construct moduli spaces of Fano varieties and log Fano pairs \cite[Chapters 7--8]{XuKmoduliBook}.  The valuative criterion of Fujita and Li \cite{Fujita-valuative,Li-valuative} is a powerful tool for establishing whether a given pair $(X,D)$ is K stable. To apply this criterion we compute the \defi{\(\delta\)-invariant}
\[\delta(X,D) \coloneqq \inf_{Y/X} \frac{A_{X,D}(Y)}{S_{X,D}(Y)}.\]
Here the infimum is taken over all prime divisors \(Y\) over \(X\). If \(f\colon\tX\to X\) is a projective birational morphism from a normal variety with \(Y\subset\tX\), then \(S_{X,D}(Y)\) is defined by \[S_{X,D}(Y) \coloneqq \frac{1}{\vol(-(K_X+D))}\int_0^\infty \vol( f^*(-(K_X+D) - uY) du,\] and \(A_{X,D}(Y) \coloneqq 1 + \ord_Y(K_\tX + D_\tX - f^*(K_X+D))\) is the \defi{log discrepancy}, where \(D_\tX\) is the strict transform of \(D\).
\begin{theorem}[{\cite{Fujita-valuative,Li-valuative,BlumJonsson,LXZ-finite-generation}}]\label{thm:valuative-criterion-K-stability}
    A log Fano pair \((X,D)\) is K-stable (resp. K-semistable) if and only if \(\delta(X,D)>1\) (resp. \(\delta(X,D)\geq 1\)).
\end{theorem}
For a point \(\pt\in X\), the local \(\delta\)-invariant at \(\pt\) is
\[\delta_{\pt}(X, D) \coloneqq \inf_{Y/Y, \, \pt\in C_X(Y)} \frac{A_{X,D}(Y)}{S_{X,D}(Y)}\] where the infimum is taken over all prime divisors \(Y\) over \(X\) whose centers on \(X\) contain \(\pt\).

We will also use the following criterion for K-polystability. For a log Fano pair \((X,D)\) and a reductive subgroup \(G\subset\Aut(X,D)\), define \[\delta_G(X,D)\coloneqq\inf_{Y/X}\frac{A_{X,D}(Y)}{S_{X,D}(Y)}\] where the infinimum is taken over all \(G\)-invariant prime divisors \(Y\) over \(X\).

\begin{theorem}[{\cite[Corollary 4.14]{Zhuang-equivariant}}]\label{lem:K-ps-G-invariant-delta}
    Let \((X,D)\) and \(G\) be as above. If \(\delta_G(X,D)>1\), then \((X,D)\) is K-polystable.
\end{theorem}

\subsection{Connection between K-stability and GIT stability}\label{sec:bg-K-vs-GIT}
In this section we connect GIT stability of a divisor $D$ in $\mathbb{P}^1 \times \mathbb{P}^2$ to K-stability of the pair $(\mathbb{P}^1 \times \mathbb{P}^2,D)$, generalizing \cite[Theorem 1.4]{ADL19} (see also \cite[Theorem 1.2]{GMGS21}) and \cite[Theorem 5.2]{ADL-quadric} to products of projective spaces. First, we record the following straightforward computation of intersection numbers on products of projective space.

\begin{lemma}\label{lem:IntTheoryProdProj}
    Let $X = \bP^{n_1} \times \dots \times \bP^{n_k}$ be a product of projective spaces, and let $D$ be a divisor of type $(d_1, \dots, d_k)$ on $X$. Write \(n=\dim X=\sum_{i=1}^k n_i\). Then the self-intersection
    \[ D^n = \binom{n}{n_1,\dots,n_k} \prod_{i=1}^k d_i^{n_i} \]
    where \(\binom{n}{n_1,\dots,n_k} = n!/(\prod_{i=1}^k n_i!)\) is the multinomial coefficient.
\end{lemma}

\begin{detail}
\begin{proof}
    Write \(H_i\) for the pullback of \(\cal O_{\bb P^{n_i}}(1)\) from the \(i\)-th factor. Then by the multinomial theorem,
    \begin{equation*}\begin{split}
        D^n &= (d_1 H_1 + \cdots + d_k H_k)^n \\
        &= \sum_{\substack{l_1 + \cdot + l_k = n, \\ l_1,\ldots,l_k \geq 0}} \binom{n}{l_1,\dots,l_k} \prod_{i=1}^k (d_i H_i)^{l_i} 
        = \binom{n}{n_1,\dots,n_k} \prod_{i=1}^k d_i^{n_i}.
    \end{split}\end{equation*}
\end{proof}
\end{detail}

Next, we show that K-stability implies GIT stability (with the linearization from \(\cal O_X(1,\ldots,1)\).)

\begin{theorem}\label{thm:K-implies-GIT}
    Let $X = \bP^{n_1} \times \dots \times \bP^{n_k}$ be a product of projective spaces, and let $D$ be a divisor of type $(d_1, \dots, d_k)$ on $X$.  If $(X,cD)$ is a K-(poly/semi)stable log Fano pair {for some \(c>0\)}, then $D$ is GIT (poly/semi)stable with respect to the natural action of \(\SL_{n_1+1}\times\cdots\times\SL_{n_k+1}\). 
\end{theorem}

\begin{proof}
    Let $\mathbf{d} = (d_1, \dots, d_k)$ and let $\mathbf{P}_{\mathbf{d}}$ be the linear system of degree $\mathbf{d}$ divisors in $X$. By \cite[Theorem~2.22]{ADL19}, it suffices to verify that the CM line bundle {(see \cite[Definition 2.20]{ADL19})} on the universal family $f\colon (X \times \mathbf{P}_{\mathbf{d}}, c \cal D) \to \mathbf{P}_\mathbf{d}$ is ample.

    Let $n =\dim X$. By \cite[Proposition 2.23]{ADL19}, the CM line bundle $\lambda \coloneqq \lambda_{\mathrm{CM},f,c}$ is given by the intersection formula 
    \[ c_1(\lambda) = -f_*((-K_{{X\times\mathbf{P}_\mathbf{d}}/{\mathbf{P}_\mathbf{d}}}-c\cal D)^{n+1}).\]
    If $p\colon X \times \mathbf{P}_\mathbf{d} \to X$ denotes the first projection, then $\cal O(-K_{X\times\mathbf{P}_\mathbf{d}/\mathbf{P}_\mathbf{d}}) = p^*(\cal O(n_1+1,\dots,n_k+1))$ and $\cal D = p^*(\cal O(d_1,\dots,d_k))\otimes f^*(\cal O(1))$.  Using Lemma \ref{lem:IntTheoryProdProj}, we may compute \begin{align*}
    -f_* (p^*(\cal O_X(n_1 +1 - cd_1, \dots, n_k+1 & - cd_k)) \otimes  f^*(\cal O_{\mathbf{P}_{\mathbf{d}}}(-c))^{n+1}) \\
        &= -( n +1) (\cal O_X(n_1 +1 - cd_1, \dots, n_k+1 - cd_k))^n (\cal O_{\mathbf{P}_{\mathbf{d}}}(-c)) \\
        &= \displaystyle (n +1) \binom{n}{n_1,\dots,n_k} \prod_{i=1}^k (n_i+1-cd_i)^{n_i} c \cdot \cal O_{\mathbf{P}_{\mathbf{d}}}(1) ,
    \end{align*}
    which is ample on $\mathbf{P}_\mathbf{d}$ because $(X,cD)$ was assumed to be log Fano.  This proves that K-stability implies GIT stability.
\end{proof}

Our proof that GIT (poly/semi)stability and K-(poly/semi)stability coincide for \(c \ll 1\) relies on a special case of the \defi{Gap Conjecture}, recalled below, which gives an upper bound on the normalized volume for non-smooth klt singularities (see \cite[Section 2.4]{Li-Liu} for more on normalized volumes).
\begin{conjecture}[{\cite[Conjecture 5.5]{SpottiSun17}}]\label{conj:gap-conjecture}
    Let \(y\in Y\) be an \(n\)-dimensional klt singularity that is not smooth. Then
    \[ \widehat{\mathrm{\vol}}(Y,cD_Y) \le 2(n-1)^n, \]
    and equality holds if and only if \(y\in Y\) is an ordinary double point. 
\end{conjecture}
When \(n\leq 3\), the Gap Conjecture holds by \cite[Proposition 4.10]{Li-Liu} and \cite[Theorem 1.3]{LiuXu-cubic}. 

\begin{proposition}\label{prop:small-c-no-deformations}
    Let $X = \bP^{n_1} \times \dots \times \bP^{n_k}$ be a product of projective spaces such that $k = 1$ or $\dim X \le 3$, and let $D$ be a divisor of type $(d_1, \dots, d_k)$ on $X$. If $c \ll 1$ and $(Y,cD_Y)$ is a K-semistable log Fano pair in the component of the K-moduli space parametrizing pairs $(X,cD)$, then $Y = X$.
\end{proposition}

\begin{proof}
    By Proposition \ref{thm:K-implies-GIT}, K-stability implies GIT stability. The \(k=1\) case follows from the proof of \cite[Theorem 9.24]{ADL19}, so we may assume \(n \coloneqq \dim X \leq 3\).
    If $(Y,cD_Y)$ is any K-semistable log Fano pair of dimension $n$, then by \cite{Li-Liu} we have the following inequality for any closed point $y\in Y$: 
    \[ (-K_Y - cD_Y)^n \le \left( 1 + \frac{1}{n} \right)^n \widehat{\mathrm{\vol}}(Y,cD_Y) .\]
    By choosing $c \ll 1$, then we may assume $(-K_Y-cD_Y)^n$ is arbitrarily close to $(-K_Y)^n$.
    
    If $Y$ is a $\bb Q$-Gorenstein degeneration of $X$, then
    \[ (-K_Y)^n = (-K_X)^n = \binom{n}{n_1,\dots,n_k} \prod_{i=1}^k(n_i+1)^{n_i}. \]
    If $(n_1, \dots, n_k) \in \{(n_1), (1,1), (1,2), (1,1,1) \}$, then we have
    \[ (-K_X)^n = \binom{n}{n_1,\dots,n_k} \prod_{i=1}^k(n_i+1)^{n_i} > 2\left(  1 + \frac{1}{n} \right)^n (n-1)^n. \]
    Since $n \le 3$, the Gap Conjecture (Conjecture~\ref{conj:gap-conjecture}) holds by \cite{Li-Liu,LiuXu-cubic}, so \(Y\) must be smooth.
    Therefore, if $\dim X \le 3$ and $(Y, cD_Y)$ is K-semistable, then $Y$ must be smooth. 

    Because $Y$ is smooth, it must be a $\mathbb{Q}$-Gorenstein degeneration of $X$.  
    If $X = \bP^n$, classical results \cite{Siu89} imply $Y = X$.  If $\dim Y \le 3$, because $Y$ is smooth, we have $(-K_X)^{\dim X} = (-K_Y)^{\dim Y}$, $\rho(X) = \rho(Y)$, and $i(X) = i(Y)$. Here $\rho$ is the Picard rank and $i$ is the Fano index, the largest integer $i$ such that the anticanonical divisor is divisible by $i$ in the Picard group. (C.f. \cite[Theorem 1.4]{JR11} for a more general statement in dimension 3).
    By the classification, these invariants prove $Y = X$ if $X$ is a product of projective spaces in dimension $\le 3$; see \cite{MoriMukai} for the classification in dimension 3.
\end{proof}

\begin{theorem}\label{thm:K-GIT-isomorphic-small-c}
    Under the hypotheses of the Proposition~\ref{prop:small-c-no-deformations}, K-(poly/semi)stability of $(X,cD)$ for rational numbers \(c \ll 1\) is equivalent to GIT (poly/semi)stability of $D$. In particular, the K-moduli space (resp. stack) for \(c \ll 1\) is isomorphic to the GIT moduli space (resp. stack).
\end{theorem}

\begin{proof}
    Our proof is identical to that of \cite[Theorem 9.24]{ADL19}. Indeed, by Theorem~\ref{thm:K-implies-GIT} there is an (open) immersion $\varphi\colon \Kmodulistack{c} \rightarrow {\mathcal M}^{\mathrm{GIT}}$ from the K-moduli stack to the GIT quotient stack. By Proposition \ref{prop:small-c-no-deformations} we know that the morphism $\varphi'\colon \Kmodulispace{c} \to \overline{M}^{\mathrm{GIT}}$ on good moduli spaces underlying $\varphi$ is proper, and $\varphi$ is injective. Thus, $\varphi$ is finite, and all finite open immersions are isomorphisms.
\end{proof}

\begin{remark}
    In the proof of Proposition~\ref{prop:small-c-no-deformations}, the argument using the normalized volume to show that $Y$ must be smooth does not extend to arbitrary products of projective space.  Even if the Gap Conjecture were known in generality, for many products of projective space, we have the inequality \[ (-K_X)^n \le 2\left( 1 + \frac{1}{n}\right)^n (n-1)^n,\] so there is no obstruction of this form to degenerations of $X$ with ordinary double points appearing in the K-moduli space for small coefficient. 
\end{remark}

\subsection{The Abban--Zhuang method}\label{sec:Abban-Zhuang}
In this paper, we will establish K-(semi/poly)stability of the log Fano pairs \((\bb P^1\times\bb P^2, cR)\) through the method of admissible flags developed in \cite{Abban-Zhuang-flags}. In this section we recall results of this method following \cite{Abban-Zhuang-flags,Fujita-3.11}.
We refer the reader to \cite{Abban-Zhuang-flags}, \cite[Section 2]{KristinNotes}, \cite[Sections 3 and 4]{Fujita-3.11}, and \cite[Section 1.7]{Fano-book} for details.

\begin{definition}[{\cite[Definition 3.19]{Fujita-3.11}}]\label{defn:plt-blowup}
    Let \((X,D)\) be a klt pair with \(D\) an effective \(\bb Q\)-divisor. A prime divisor \(Y\) over \(X\) is of \defi{plt-type} over \((X,D)\) if there is a projective birational morphism \(f\colon\tX\to X\) of normal varieties and a prime divisor \(Y\subset\tX\) such that \begin{enumerate}
        \item \(-Y\) is \(f\)-ample and \(\bb Q\)-Cartier, and
        \item If \(\tilde{D}\) is the \(\bb Q\)-divisor on \(\tX\) defined by \(K_{\tX}+\tilde{D}+(1-A_{X,D}(Y))Y = f^*(K_X+D)\), then the pair \((\tX,\tilde{D}+Y)\) is plt.
    \end{enumerate}
    The \defi{plt-blowup associated to \(Y\)} is the morphism \(f\), which is uniquely determined by \(Y\). If \(D_Y\) is defined by \(K_Y+D_Y = (K_{\tX}+\tilde{D}+Y)|_Y\), then \((Y,D_Y)\) is a klt pair.
\end{definition}

\begin{definition}
    Let \((X,D)\) be an \(n\)-dimensional klt log Fano pair with \(D\) an effective \(\bb Q\)-divisor, and assume \(L\coloneqq -(K_X+D)\) is Cartier. The \defi{complete linear series associated to \(L\)} is the graded linear series \(V_\bullet\coloneqq\{H^0(X,mL)\}_{m\in\bb N}\), and the volume of \(V_{\bullet}\) is \(\vol L\).
    
    For a plt-type prime divisor \(Y\) over \(X\), the \defi{refinement of \(V_{\bullet}\) by \(Y\)} is the multigraded linear series \[W^Y_{m,j}\coloneqq\Im\left( H^0(\tX, mL-jY) \to H^0(Y, mL|_Y-jY|_Y)\right).\] The volume of \(W^Y_{\bullet\bullet}\) is \(\vol(W^Y_{\bullet\bullet}) = \lim_{m\to\infty}\sum_{j\geq 0} (\dim W^Y_{m,j})/(m^n/n!)\).
    
    For a prime divisor \(C\) over \(Y\), define the \(\bb Z_{\geq 0}^2\)-graded linear series \(\cal F^t_C W^Y_{\bullet\bullet}\) by \[\cal F^{mt}_C W^Y_{m,j} \coloneqq \{s\in W^Y_{m,j} \mid \ord_C(s)\geq mt\}.\] The volume is defined to be \(\vol(\cal F^t_C W^Y_{\bullet\bullet}) = \lim_{m\to\infty}\sum_{j\geq 0} (\dim(\cal F^{mt}_C W^Y_{m,j}))/(m^n/n!)\). Set \[S_{Y, D_Y; W^Y_{\bullet \bullet}}(C) \coloneqq \frac{1}{\vol(W^Y_{\bullet\bullet})} \int_0^\infty\vol(\cal F^t_C W^Y_{\bullet\bullet}) dt.\]
\end{definition}

\begin{theorem}[{\cite[Theorem 3.3]{Abban-Zhuang-flags}, see also \cite[Theorem 3.20]{Fujita-3.11}}]
    Let \((X,D)\) be a klt log Fano pair, with \(X\) projective, let \(Z\subset X\) be an irreducible subvariety, and let \(Y\) be a plt-type prime divisor over \((X,D)\) with \(Z\subset C_X(Y)\). Let \(V_\bullet\) be the complete linear series associated to \(-(K_X+D)\), and let \(W^Y_{\bullet\bullet}\) be the refinement of \(V_\bullet\) by \(Y\). Let \(D_Y\) be as in Definition~\ref{defn:plt-blowup}. Then \[\delta_Z(X,D; V_\bullet) \geq \min\left\{ \frac{A_{X,D}(Y)}{S(V_\bullet,Y)}, \inf_{Z'}\delta_{Z'}(Y,D_Y; W^Y_{\bullet\bullet}) \right\}\] where the infimum is taken over all \(Z'\subset Y\subset\tX\) with \(f(Z')=Z\), and \[\delta_{Z'}(Y,D_Y; W^Y_{\bullet\bullet}) \coloneqq \inf_F \frac{A_{X, D_Y}(C)}{S_{Y, D_Y; W^Y_{\bullet \bullet}}(C)}\] where the infimum is taken over prime divisors \(C\) over \(Y\) with \(Z'\subset C_Y(C)\).
\end{theorem}
More generally, this framework is inductive: starting with $W^Y_{\bullet \bullet}$ in place of $V_\bullet$, one may further refine by plt divisors over $Y$ to obtain lower bounds for the terms $\delta_{Z'}(Y,D_Y; W^Y_{\bullet\bullet})$ and iterate this process.
In Sections~\ref{sec:stability-of-A_n} and~\ref{sec:K-stability-Xthree}, we will use a refinement of the Abban--Zhuang method for 3-dimensional Mori dream spaces due to Fujita \cite[Sections 4.2 and 4.3]{Fujita-3.11}, which gives formulas for \(S\) in terms of Zariski decompositions on \(Y\). The details of this are recalled in Section~\ref{sec:fujita-abban-zhuang}.

\subsection{The \texorpdfstring{$(2,2)$}{2,2}-surface \texorpdfstring{$R$}{R} and the quartic curve \texorpdfstring{$\Delta$}{d}}\label{sec:preliminaries-R-and-Delta}
 
Throughout this section, let
$$R=(t_0^2Q_1+2t_0t_1Q_2+t_1^2 Q_3=0)\subset\mathbb{P}^1\times\bb P^2 , \qquad \Delta=(Q_1Q_3-Q_2^2=0)\subset\bb P^2$$ be as defined in Section~\ref{sec:notation}. We first make some basic observations about the singularities of \(R\) and \(\Delta\).

\begin{lemma}\label{lem:sing-R-vs-Delta}
    The following hold:
    \begin{enumerate}
    \item If \(\Delta=\bb P^2\), then either \(R\in|(2,0)|+|(0,2)|\) is reducible, or \(R\) is non-reduced and has reduced subscheme \(R_{\red}\in|(1,1)|\).
    \item\label{item:smooth-Delta-pi2-finite} The map \(\piR_2\colon R\to\bb P^2\) is finite over the locus \(\bb P^2\setminus(Q_1=Q_2=Q_3=0)\).
    \item\label{item:sing-R-vs-Delta-1} \(\piR_2(\Sing R)\subset\Sing\Delta\).
\end{enumerate}
\end{lemma}
In particular, if \(\Delta\) is smooth then so is \(R\). If \(R\) is smooth, then \(\piR_2^{-1}(\Sing\Delta)\) is the locus \((Q_1=Q_2=Q_3=0)\) (i.e. the union of positive dimensional fibers of \(\piR_2\colon R\to\bb P^2\)).

\begin{proof}
    This follows from direct computation.
\begin{detail}    
    First assume \(\Delta=\bb P^2\), i.e. the polynomial \(Q_1Q_3-Q_2^2\) is uniformly zero.
\begin{enumerate}
    \item If \(Q_2\in\bb C[y_0,y_1,y_2]\) is irreducible, then unique factorization implies \(Q_1=\lambda Q_2\) and \(Q_3=\lambda^{-1}Q_2\) for some \(\lambda\in\bb C^*\). Then \(R\) has defining equation \(Q_2(\lambda t_0^2 + 2t_0t_1 + \lambda^{-1}t_1^2)\) and is in the linear system \(|(0,2)|+|(2,0)|\).
    \item If \(Q_2=L_0L_1\) is a product of two linear factors, then \(Q_1Q_3=L_0^2L_1^2\). We have two possibilities:
    \begin{enumerate}
        \item \(Q_1=\lambda L_0 L_1\) and \(Q_3=\lambda^{-1} L_0 L_1\) for some \(\lambda\in\bb C^*\), and \(R\in |(0,2)|+|(2,0)|\)
        \item After possibly exchanging \(Q_1\) and \(Q_3\), we have \(Q_1=\lambda L_0^2\) and \(Q_3=\lambda^{-1} L_1^2\) for some \(\lambda\in\bb C^*\). After changing coordinates on \(\bb P^1\) we may assume \(\lambda=1\). Then \(R\) is defined by \[(t_0L_0)^2+2 t_0 t_1 L_0L_1 + (t_1 L_1)^2 = (t_0 L_0 + t_1 L_1)^2.\]
    \end{enumerate}
\end{enumerate}
    
    Part~\eqref{item:smooth-Delta-pi2-finite} follows from the equations for \(R\), and part~\eqref{item:sing-R-vs-Delta-1} follows from direct computation. Indeed, let \(\pt\in\Sing R\). We may assume \(\pt\) is in the affine chart \((t_1\neq 0)\cap(y_3\neq 0)\), and work with affine coordinate \((t_0,y_0,y_1)\). The Jacobian matrices of \(R\) and \(\Delta\) on this chart are
\begin{equation}\label{eqn:jac-R}\begin{pmatrix} 2(t_0 Q_1+Q_2) & t_0^2\frac{\partial Q_1}{\partial y_0}+2t_0 \frac{\partial Q_2}{\partial y_0} + \frac{\partial Q_3}{\partial y_0} & t_0^2\frac{\partial Q_1}{\partial y_1}+2t_0 \frac{\partial Q_2}{\partial y_1} + \frac{\partial Q_3}{\partial y_1}
\end{pmatrix}\end{equation}
and
\begin{equation}\label{eqn:jac-Delta}\begin{pmatrix}
Q_1 \frac{\partial Q_3}{\partial y_0} + \frac{\partial Q_1}{\partial y_0} Q_3 - 2Q_2\frac{\partial Q_2}{\partial y_0} & Q_1 \frac{\partial Q_3}{\partial y_1} + \frac{\partial Q_1}{\partial y_1} Q_3 - 2Q_2\frac{\partial Q_2}{\partial y_1}
\end{pmatrix}. \end{equation}
Since \(\pt\in\Sing R\) we have \((t_0 Q_1 + Q_2)|_{\pt}=0\), and one can verify that this implies \(\piR_2(\pt)\in\Delta\).

Now we show that \(\piR_2(\pt)\in\Sing\Delta\). First, if \(Q_1|_{\piR_2(\pt)}=0\), then the equations for \(\Delta\) and \(R\) then imply that \(Q_2|_{\piR_2(\pt)}=Q_3|_{\piR_2(\pt)}=0\), so \(\Delta\) is singular at \(\piR_2(\pt)\) by~\eqref{eqn:jac-Delta}. If \(Q_1|_{\piR_2(\pt)}\neq 0\), then at \(\pt\) we may write \(t_0=- \frac{Q_2}{Q_1}|_{\pt}\). From the vanishing of the second and third entries of~\eqref{eqn:jac-R} and the equality \((Q_1 Q_3-Q_2^2)|_{\pt}\) we get
\[\left.\left(\frac{Q_2^2}{Q_1^2}\frac{\partial Q_1}{\partial y_i} - 2 \frac{Q_2}{Q_1}\frac{\partial Q_2}{\partial y_i} + \frac{\partial Q_3}{\partial y_i}\right)\right|_{\pt} = \left.\frac{1}{Q_1}\left(Q_3\frac{\partial Q_1}{\partial y_i} - 2 Q_2\frac{\partial Q_2}{\partial y_i} + Q_1\frac{\partial Q_3}{\partial y_i}\right)\right|_{\pt}. \]
Hence, the matrix~\eqref{eqn:jac-Delta} has rank \(0\) at \(\piR_2(\pt)\).
\end{detail}
\end{proof}

\begin{corollary}\label{cor:non-normal-sings}
    Assume \(R\) is non-normal, and let \(S\subset\Sing R\) be a \(1\)-dimensional irreducible component. Then either \(S\) is a fiber of \(\piR_2\), or \(\piR_2(S)\) is a non-reduced component of \(\Delta\).
\end{corollary}

\begin{lemma}
    If \(\piR_2\colon R\to\bb P^2\) is finite over \(\pt\in\Delta\), then there is a (formal) neighborhood \(U\) of \(\pt\) such that \(\Delta\cap U=(f=0)\) and \(\piR_2\colon R|_U\to U\) is defined by \(t^2=f\). In particular, if \(\pt\in\Delta\) is an \(A_n\) singularity, then the unique point of \(D\) over \(\pt\) is an \(A_n\) singularity.
\end{lemma}

\begin{proof}
    If \(\piR_2\) is finite over \(\pt\in\Delta\), then either \(Q_1(\pt)\neq 0\) or \(Q_3(\pt)\neq 0\) by Lemma~\ref{lem:sing-R-vs-Delta}\eqref{item:smooth-Delta-pi2-finite}. In the first case, away from the conic \((Q_1=0)\), the map \(\piR_2\colon D\to\bb P^2\) is given by \((t Q_1+Q_2)^2 = -(Q_1 Q_3-Q_2^2)\). Similarly, in the second case, away from \((Q_3=0)\), \(\piR_2\) is given by \((t' Q_3 + Q_2)^2=-(Q_1Q_3-Q_2)^2\).
\end{proof}

If \(R\) is normal, then it has at worst ADE singularities by the following lemma. Hence, a resolution of singularities \(\widetilde{R}\) of \(R\) is a weak del Pezzo surface of degree 2, and the singularities of these have been classified (see \cite{Wall-quartics} or \cite[Section 8.7.1]{Dolgachev-CAG}). In particular, the only possible ADE singularities on \(R\) are \(A_n\) for \(1 \leq n\leq 7\), \(D_n\) for \(4 \leq n \leq 6\), and \(E_n\) for \(n=6,7\).

\begin{lemma}\label{lem:2-2-surfaces-have-ADE-singularities}
    Let $R$ be a $(2,2)$-surface in $\bP^1 \times \bP^2$ with only isolated singularities. Then any singular point \(\pt\in R\) is an ADE singularity.
\end{lemma} 

\begin{proof}
    \(R\) is Gorenstein, and any isolated singularity \(\pt\in R\) is a double point since a general line on \(\bb P^1\times\bb P^2\) intersects \(R\) with multiplicity 2. It remains to show that \(R\) has rational singularities. For this, let \(f\colon \widetilde{R}\to R\) be a resolution of singularities. Then \(H^1(\widetilde{R},\cal O_{\widetilde{R}}) = H^2(\widetilde{R},\cal O_{\widetilde{R}}) = 0\) since \(\widetilde{R}\) is a smooth rational surface, so the Leray spectral sequence then shows that \(R^1 f_*\cal O_{\widetilde{R}}=0\). Thus \(\pt\in R\) is a rational double point, i.e. an ADE singularity.
\begin{detail}
    We give more details. First, \(R\) is rational because the conic bundle \(\piR_1\colon R\to\bb P^1\) has a section by Tsen's theorem, so \(\widetilde{R}\) is rational as well. Next, \(H^2(R, f_*\cal O_{\widetilde{R}})=0\) by Zariski's Main Theorem and direct computation on \(\bb P^1\times\bb P^2\), so the Leray spectral sequence implies \(H^0(R, R^1f_*\cal O_{\widetilde{R}}) =0 \). Since by assumption \(R^1f_*\cal O_{\widetilde{R}}\) is supported on a finite set of points, this implies \(R^1f_*\cal O_{\widetilde{R}}=0\).
\end{detail}
\end{proof}
In fact, even without assuming that \(R\) is normal, one can still show that every isolated singular point is an ADE singularity.
\begin{detail}
One can prove this as follows:
    \begin{proof}
    Consider the Stein factorization $R \to R' \to \bP^2$ where $R \to R'$ connected fibers and $R' \to \bP^2$ is finite.  If $\Delta \subset \bP^2$ is the discriminant as defined above, then $R'$ is the double cover of $\bP^2$ branched over $\Delta$ and hence is locally modeled by an equation $z^2 = f(x,y)$ where $f(x,y)$ is a quartic polynomial.  By assumption that $p \in R$ is an isolated singularity, we may assume $f(x,y)$ is reduced.  From the classification of quartic plane curves, the branched cover of any reduced curve other than four lines meeting at a point has ADE singularities and hence in the first case $R'$ has ADE singularities.  Because $R$ is Gorenstein and $R \to R'$ contracts only non-finite fibers of the map $\piR_2$ which are $-2$ curves, $R$ also has canonical (and hence ADE) singularities. 
    
    So, it suffices to prove that if this quartic consisting of four lines meeting at a point is given by the discriminant $Q_1 Q_3 - Q_2^2$ of a \((2,2)\)-surface $R$, then its preimage in \(R\) is a non-isolated singularity.  Suppose $C$ is the quartic given by four lines $l_1, l_2, l_3, l_4$ meeting at a point $q \in \bP^2$, and assume that $C = (Q_1Q_3 - Q_2^2 = 0)$ for some conics $Q_1, Q_2, Q_3$.  First, we claim that each $q \in Q_i$ for each $i$ (by abuse of notation, we use \(Q_i\) to denote both the curve and its equation).  Assume for contradiction that this is not the case.  By straightforward computation in coordinates, this implies that $q \notin Q_i$ for any $i$.  Because $C$ is contained in the pencil $\langle Q_1 Q_3, Q_2^2 \rangle$, the intersection of $C$ and $Q_2^2$ must coincide with the intersection of $C$ and $Q_1Q_3$.  Because $q \notin Q_2$, either $Q_2$ is a doubled line or the intersection $C \cap Q_2$ consists of at least 6 points of $Q_2$: at most two of the lines $l_i$ can meet $Q_2$ with multiplicity 2 and the others must meet transversally.  If $Q_2$ is a doubled line, then $C$ meets $Q_2$ in four distinct points, but these must be the only intersections of both $Q_1Q_3$ with $Q_2$ and $C$ with $Q_1Q_3$. This is impossible unless $Q_1$, $Q_3$ are multiples of $Q_2$ in which case $C$ is not contained in the pencil $\langle Q_1 Q_3, Q_2^2 \rangle$.  If $Q_2$ is not a doubled line, then the intersection of $C$ and $Q_2$ contains at least five points $\{ x_1, x_2, \dots, x_5\}$ with no three of them colinear and hence $Q_2$ is the unique conic passing through these points.   However, the intersection of $C$ and $Q_1Q_3$ must consist of the same points.  Because $Q_1 \cdot l_i = 2$ for each $i$, and of the five points $\{ x_1, x_2, \dots, x_5\}$, no three are colinear, these five points must in fact be contained in the intersection of $Q_1$ and $Q_3$.  Therefore, $Q_1$ and $Q_3$ pass through the same five points as $Q_2$ and hence are all equal up to a scalar multiple and again we contradict $C \in \langle Q_1 Q_3, Q_2^2 \rangle$.

    Now, we have shown that $q \in Q_i$ for each $i$.  We claim that each $Q_i$ must actually be singular at $q$.  Suppose $Q_j$ is smooth at $q$ for some $j \in \{ 1,2,3\}$.  If none of the lines $l_1, \dots, l_4$ are tangent to $Q_j$ at $q$, then the intersection of $Q_j$ and $C$ contains five non-colinear points and hence $Q_j$ must be the unique conic containing these five points.  However, this implies again that $C$ must intersect the remaining conics $Q_i$ at these five points and thus all $Q_i$ are scalar multiples of each other (and hence all smooth), contradicting that $C$ is contained in the pencil $\langle Q_1 Q_3, Q_2^2 \rangle$.  Finally, if one line is tangent to $Q_j$ at $q$, by a similar argument, we find that line meets the remaining $Q_i$ at $q$ to order 2, and hence each $Q_i$ is a scalar multiple of the unique conic containing the four intersection points of $Q_j$ and $C$ with a specified tangent line at $q$ and have the same contradiction.  

    Finally, we conclude that $q \in Q_i$ for each $i$ and each $Q_i$ is singular at $q$.  This implies that $R = (t_0^2Q_1 + 2t_0t_1 Q_2 + t_1^2 Q_3 = 0)$ contains and is singular along the fiber of $\piR_2$ above $q$.  In particular, $R$ does not have isolated singularities.
\end{proof}
\end{detail}
However, Lemma~\ref{lem:2-2-surfaces-have-ADE-singularities} will be sufficient for our purposes in this paper.

\begin{remark}
    In what follows, we will frequently need to distinguish between $A_n$ singularities and $D_n$, $E_n$ singularities.  We will refer to $D_n, E_n$, and non-canonical singularities as \defi{worse than} $A_n$ singularities.
\end{remark}

Next, we make the following observation relating the GIT stability of \(\Delta\) and \(R\).

\begin{lemma}\label{lem:GIT-stability-Delta-R} Let \(R\) be a \((2,2)\)-surface in \(\bb P^1\times\bb P^2\), let \(\Delta\) be defined as in Section~\ref{sec:notation}, and assume \(\Delta\neq\bb P^2\). If \(\Delta\) is GIT semistable (resp. stable), then so is \(R\).
\end{lemma}

\begin{proof}
    We will show the contrapositive. Let \(\lambda\) be a normalized 1-PS with \(\mu([R],\lambda) > 0\) (resp. \(\geq 0\)). After a coordinate change, we may assume \(\lambda=\diag(r_0,r_1)\times\diag(s_0,s_1,s_2)\) is normalized. Let \[f=\sum_{0\leq j,k, \; j+k\leq 2} a_{jk} t_0^2 y_0^{2-j-k} y_1^j y_2^k + b_{jk} t_0 t_1 y_0^{2-j-k} y_1^j y_2^k + c_{jk} t_1^2 y_0^{2-j-k} y_1^j y_2^k\] be a defining equation for \(R\) in these coordinates; then \(\Delta\) is defined by \[g = 4(\sum a_{jk}y_0^{2-j-k}y_1^j y_2^k)(\sum c_{j'k'}y_0^{2-j'-k'}y_1^{j'} y_2^{k'}) - (\sum b_{jk}y_0^{2-j-k}y_1^j y_2^k)^2.\]
    Let \(y_0^{4-m_1-m_2}y_1^{m_1}y_2^{m_2}\) be any monomial with nonzero coefficient in \(g\). Since \[\coeff_g(y_0^{4-m_1-m_2}y_1^{m_1}y_2^{m_2}) = \sum_{j+j'=m_1, k+k'=m_2} 4a_{j k}c_{j' k'}-b_{jk}b_{j'k'},\] we have that \(a_{jk}c_{j'k'}\neq 0\) or \(b_{jk}b_{j'k'}\neq 0\) for some non-negative integers \(j,j',k,k'\) with \(j+j'=m_1\) and \(k+k'=m_2\).
    If \(a_{jk}c_{j'k'}\neq 0\), then the assumption that \(\mu(f,\lambda)>0\) (resp. \(\geq 0\)) implies that the integers \(2r_0 + (2-j-k)s_0 + j s_1 + ks_2\) and \(2r_1 + (2-j'-k')s_0 + j' s_1 + k' s_2\) are positive (resp. non-negative), so their sum \[2(r_0+r_1) + (2-j-k + 2-j'-k')s_0 + (j+j') s_1 + (k+k') s_2 = (4-m_1 - m_2)s_0 + m_1 s_1 + m_2 s_2\] is positive (resp. non-negative). If \(b_{jk}b_{j'k'}\neq 0\) then the assumption that \(\mu(f,\lambda)>0\) (resp. \(\geq 0\)) implies that the integers \((2-j-k)s_0+j s_1 + ks_2\) and \((2-j'-k')s_0+j' s_1 + k's_2\) are positive (resp. non-negative), so again their sum \((4-m_1 - m_2)s_0 + m_1 s_1 + m_2 s_2\) is positive (resp. non-negative). Since this holds for all monomials with nonzero coefficient in \(g\), we conclude that \(\mu(g,\diag(s_0,s_1,s_2))>0\) (resp. \(\geq 0\)), so by the Hilbert--Mumford criterion, the quartic \(\Delta\) is GIT unstable (resp. not GIT stable).
\end{proof}

\begin{remark}
    The converse of Lemma~\ref{lem:GIT-stability-Delta-R}  is false. That is, GIT semistability of the \((2,2)\)-surface \(R\) does \emph{not} imply that of the quartic \(\Delta\): \(\Rthree\) (Theorem~\ref{thm:moduli-spaces-2.18}) is GIT polystable and has GIT unstable \(\Delta\), and there are GIT semistable surfaces with \(D_n\) singularities whose associated quartic curves are GIT unstable (Theorem~\ref{thm:GIT-polystable} and Lemma~\ref{lem:worse-than-A_n}). Furthermore, there are GIT stable \(R\) whose associated quartic is a union of a cubic and an inflectional tangent, which is GIT unstable (Corollary~\ref{cor:surfaces-R-for-epsilon-curves}).
\end{remark}

\subsection{GIT moduli space of plane quartics}\label{sec:GIT-quartics}
We recall results about the the GIT moduli space of plane quartics \(\GITquartic \coloneqq |\cal O_{\bb P^2}(4)|^{ss}/\!\!/\SL_3 \) from \cite[Chapter 4, \S2]{Mumford-Fogarty}. \(\GITquartic\) is a normal 6-dimensional projective variety (see \cite{Mumford-Fogarty} or \cite[Proposition 4.6]{ADL19}), and the strictly polystable locus is a \(\bb P^1\). The stable locus corresponds to quartics with only \(A_1\) or \(A_2\) singularities. A general member of the strictly semistable locus is a \defi{cat-eye}, which consists of two smooth rational curves \(C_1\) and \(C_2\) meeting in two tacnodes. There are two special points on the semistable locus corresponding to the \defi{ox} (a reducible curve consisting of three smooth rational curves \(C_1,C_2,C_3\) with \(C_1\) and \(C_2\) meeting each other in a node, and with \(C_1\) and \(C_2\) each meeting \(C_3\) in a tacnode) and the \defi{double conic} (the double of a smooth plane conic).

Quartics with non-\(A_n\) singularities are GIT unstable. Quartics with higher \(A_n\) singularities are either GIT semistable or GIT unstable; these are classified in \cite{Wall-quartics}, and we will recall this classification in more detail in Section~\ref{sec:GIT-higher-An-quartics}.

\section{GIT stability of \texorpdfstring{\((2,2)\)}{(2,2)}-surfaces in \texorpdfstring{\(\bb P^1\times\bb P^2\)}{P1xP2}}\label{sec:GIT-stability-section}
In this section, we study the GIT stability of \((2,2)\)-surfaces in \(\bb P^1\times\bb P^2\): we describe all the \((2,2)\)-surfaces that are not GIT stable, and we find all the GIT strictly semistable surfaces.

In more detail, first we show that \((2,2)\)-surfaces with worse than \(A_n\) singularities are not GIT stable in Section~\ref{sec:GIT-unstable}. Furthermore, we show that the irreducible ones with worse than \(A_n\) singularities are all GIT unstable or are S-equivalent to a certain surface, which we call \(\Rthree\). In the reducible case, we show these surfaces are all GIT unstable or S-equivalent to one of two specific surfaces, which we call \(\Rone\) and \(\Rtwo\). Next, in Section~\ref{sec:GIT-strictly-ss}, we describe the GIT strictly semistable \((2,2)\)-surfaces. In Proposition~\ref{prop:representatives-degen-Gamma_i} we describe the S-equivalence classes of all GIT strictly semistable surfaces; later, in Section~\ref{sec:GIT-strictly-polystable}, we will use K-stability to prove that these are in fact polystable representatives.
Finally, in Section~\ref{sec:A_n-on-bad-fibers-not-GIT-stable}, we show that \((2,2)\)-surfaces with \(A_n\) singularities other than those described in Theorem~\ref{thm:stable-members}\eqref{item:stable-members-thm-A_n} are not GIT stable. This will prove Theorem~\ref{thm:stable-members}\eqref{item:stable-members-thm-GIT}\(\Rightarrow\)Theorem~\ref{thm:stable-members}\eqref{item:stable-members-thm-A_n}.

We refer the reader to Remark~\ref{rem:PR-GIT} for a discussion concerning the preprint \cite{PR-GIT}.

\subsection{Worse than \texorpdfstring{\(A_n\)}{An} singularities}\label{sec:GIT-unstable}

In this section, we show that \((2,2)\)-surfaces with worse than \(A_n\) singularities (including \(D_n\) or \(E_n\)) are not GIT stable. First, we consider irreducible surfaces.

\begin{lemma}\label{lem:non-A_n-not-GIT-stable}
Let \(R\) be a \((2,2)\)-surface in \(\bb P^1\times\bb P^2\) that is irreducible and non-normal. Then \(R\) is not GIT stable. Furthermore, \(R\) is either GIT unstable or is (up to a change of coordinates  on \(\bb P^1\times\bb P^2\)) the surface \(\Rthree\) defined by \(t_0t_1y_1^2 + (t_0y_2 + t_1y_0)^2\).
\end{lemma}

\begin{proof}
$R$ is either non-reduced or $\dim \Sing R = 1$. If $R$ is non-reduced, then after a change of coordinates we may write $R = V((t_0y_0+t_1y_1)^2) \subset \bb P^1 \times \bb P^2$. Then $R$ is GIT unstable, as demonstrated by the one-parameter subgroup $\diag(1,-1) \times \diag(0,2,-2)$.

Henceforth, suppose \(R\) is integral and $C$ is a $1$-dimensional irreducible component of $\Sing R$. If $C$ is a fiber of $\piR_2$, say $\piR_2(C) = [1:0:0] \in \bb P^2$, then the defining equation $f$ of $R$ must satisfy $f \in (y_1,y_2)^2$. Then $R$ is GIT unstable, as demonstrated by the 1-PS $\diag(1,-1) \times \diag(-4,2,2)$. Henceforth, we assume $\piR_2(C)$ is $1$-dimensional. We consider the following cases:

\noindent\underline{Case (I):} First assume that $\piR_1(C)$ is a point in $\bb P^1$, say $[1:0] \in \bb P^1$. Since $R$ is irreducible,
\begin{detail}
we have $\piR_1^{-1}([1:0]) \neq \{[1:0]\} \times \bb P^2$. Then
\end{detail}
$\piR_1^{-1}([1:0])$ is a double line in $\{[1:0]\} \times \bb P^2$, which we may assume is given by $\{[1:0]\} \times V(y_0^2) \subset \{[1:0]\} \times \bb P^2$ after a change of coordinates. Then $R = V(t_0^2y_0^2+2t_0t_1Q_2+t_1^2Q_3) \subset \bb P^1 \times \bb P^2$ for some quadrics $Q_2$ and $Q_3$ in $y_0,y_1,y_2$. Furthermore, $y_0$ must divide $Q_2$ since $V(y_0) \subset \piR_2(C) \subset \Delta = V(y_0^2Q_3-Q_2^2)$. Then the 1-PS $\diag(-3,3)\times\diag(4,-2,-2)$ shows that $R$ is GIT unstable.
    
\noindent\underline{Case (II):} It remains to check the case $\piR_1(C) = \bb P^1$, where every fiber of $\piR_1$ is singular. We may assume the general fiber of $\piR_1$ has rank $2$, or else $R$ is non-reduced. By Lemma~\ref{lem:sing-R-vs-Delta}\eqref{item:sing-R-vs-Delta-1}, $\piR_2(C) \subset \Sing \Delta$ is a non-reduced component of $\Delta$, so $\piR_2(C)_\red$ is a line or a smooth conic in $\bb P^2$. \begin{enumerate}
        \item[(a)] If $\piR_2(C)_\red$ is smooth conic in $\bb P^2$, we may assume $\piR_2(C)_\red = V(y_0y_2 - y_1^2) \subset \bb P^2$ after a change of coordinates. Consider the isomorphism \begin{align*}
            \psi \colon \bb P^1_{[t_0:t_1]} \times V(y_0y_2-y_1^2) &\xrightarrow{\cong} \bb P^1_{[t_0:t_1]} \times \bb P^1_{[w_0:w_1]} \\
            ([t_0:t_1],[w_0^2:w_0w_1:w_1^2]) &\mapsfrom ([t_0:t_1],[w_0:w_1])
        \end{align*}
        Since the generic fiber of $\piR_1$ has rank $2$, $\psi(C)$ must be a $(1,1)$-divisor in $\bb P^1_{[t_0:t_1]} \times \bb P^1_{[w_0:w_1]}$. After a change of coordinates on $\bb P^1_{[t_0:t_1]}$ (and leaving $V(y_0y_2-y_1^2)$ and $\bb P^1_{[w_0:w_1]}$ invariant), we may assume $\psi(C) = V(t_0w_0+t_1w_1) \subset \bb P^1_{[t_0:t_1]} \times \bb P^1_{[w_0:w_1]}$. In other words, $C = V(y_0y_2-y_1^2,t_0y_0+t_1y_1,t_0y_1+t_1y_2) \subset \bb P^1_{[t_0:t_1]} \times \bb P^2_{[y_0:y_1:y_2]}$. Then $f \in (y_0y_2-y_1^2,t_0y_0+t_1y_1,t_0y_1+t_1y_2)^2$. By considering bidegrees, this can only mean $f = a(t_0y_0+t_1y_1)^2 + b(t_0y_0+t_1y_1)(t_0y_1+t_1y_2) + c(t_0y_1+t_1y_2)^2$ for some $a,b,c \in \bb C$, in which case $R$ is not integral.
        
        \item[(b)] If $\piR_2(C)_\red$ is a line in $\bb P^2$, we may assume $\bb P^1_{[y_0:y_2]} \cong V(y_1) \subset \bb P^2$ after a change of coordinates. Similar to (a), since the generic fiber of $\piR_1$ has rank $2$, we may write $C = V(y_1,t_0y_2+t_1y_0) \subset P^1_{[t_0:t_1]} \times \bb P^1_{[y_0:y_2]}$. Then $f \in (y_1,t_0y_2+t_1y_0)^2$, i.e. $f = q(t_0,t_1)y_1^2 + \ell(t_0,t_1)y_1(t_0y_2+t_1y_0) + c(t_0y_2+t_1y_0)^2$ for some $c \in \bb C$, linear form $\ell$, and quadratic form $q$. Then $c \neq 0$, or else $R$ is reducible. Replacing $f$ by $\frac{1}{c}f$, we assume $c=1$. 

        Next, we will use the following observation repeatedly: whenever we make any change of coordinates \[\qquad \qquad \varphi \colon \begin{pmatrix}
            t_0 \\ t_1
        \end{pmatrix} \mapsto \begin{pmatrix}
            a & b \\ c & d
        \end{pmatrix} \begin{pmatrix}
            t_0 \\ t_1
        \end{pmatrix} \qquad \textrm{on $\bb P^1$},\]
        we can and will always make the additional change of coordinates \[
            \qquad \qquad \widetilde{\varphi} = \begin{pmatrix}
                y_0 \\ y_2
            \end{pmatrix} \mapsto \begin{pmatrix}
                c & a \\
                d & b
            \end{pmatrix}^{-1} \begin{pmatrix}
                y_0 \\ y_2
            \end{pmatrix}  \qquad \textrm{on $\bb P^2$}
        \]
        so that $t_0y_2 + t_1y_0$ stays invariant. This can be verified via a direct computation. Using this observation, we claim that we can make a change of coordinates so that \[
            \qquad \qquad f = q''(t_0,t_1)y_1^2+(t_0y_2+t_1y_0)^2
        \]
        for some quadratic form $q''$. Indeed, if $\ell = 0$, we are done. If not, we make a coordinate change $\varphi$ that sends $\ell(t_0,t_1) \mapsto t_0$, followed by the corresponding $\widetilde{\varphi}$, so that $f$ is transformed to $q'(t_0,t_1)y_1^2+t_0y_1(t_0y_2+t_1y_0)+(t_0y_2+t_1y_0)^2$. Then we make the coordinate change $y_2 \mapsto -\frac{1}{2}y_1+y_2$ (with $y_0,y_1$ unchanged) to obtain \begin{align*}
            \qquad \qquad f &= q'(t_0,t_1)y_1^2+(\tfrac{1}{2}t_0y_1+t_0y_2+t_1y_0)(-\tfrac{1}{2}t_0y_1+t_0y_2+t_1y_0) \\
            &= q''(t_0,t_1)y_1^2+(t_0y_2+t_1y_0)^2 \qquad \textrm{where $q'' = q'-\tfrac{1}{4}t_0^2$, as desired.}
        \end{align*}
        
        Finally, if $f = q''(t_0,t_1)y_1^2+(t_0y_2+t_1y_0)^2$ as above, we may assume $q'' \neq 0$, or else $R$ is non-reduced. We consider two sub-cases: \begin{enumerate}
            \item[(bi)] If $q'' = (\ell'')^2$ for a non-zero linear factor $\ell''$, we make a change of coordinates $\varphi$ that sends $\ell'' \mapsto t_1$, followed by the corresponding $\widetilde{\varphi}$, so that $f = t_1^2y_1^2+(t_0y_2+t_1y_0)^2$. In this case, the 1-PS $\diag(-4,4) \times \diag(-3,-3,6)$ shows that $R$ is GIT unstable.

            \item[(bii)] If $q'' = \ell''_0 \ell_1''$ for distinct non-zero linear factors $\ell''_0$ and $\ell_1''$, we make a coordinate change $\varphi$ sending $\ell_0'' \mapsto t_0$ and $\ell_1'' \mapsto t_1$, followed by the corresponding $\widetilde{\varphi}$, so that $f = t_0t_1y_1^2+(t_0y_2+t_1y_0)^2$, which is the equation for $\Rthree$.
        \end{enumerate}
\end{enumerate}
This completes the proof.
\end{proof}

Next, we show a similar result for isolated non-\(A_n\) singularities: the GIT semistable ones are all S-equivalent to \(\Rthree\). We will later see that \(\Rthree\) is in fact GIT polystable (Theorem~\ref{thm:Kpolystable2,2div}). Recall that an isolated singularity \(\pt\in R\) is an \(A_n\) singularity if and only if the Hessian matrix has rank \(\geq 2\) at \(\pt\) \cite[Theorem 2.48]{GreuelLossenShustin}.

\begin{lemma}\label{lem:worse-than-A_n}
Let \(R\) be a $(2,2)$-surface \(R\) in $\bb P^1 \times \bb P^2$. If \(\pt\in R\) is a singular point where the Hessian matrix has rank \(\leq 1\), then \(R\) is not GIT stable. Furthermore, \(R\) is either GIT unstable or is \(S\)-equivalent to the non-normal surface \(\Rthree\) defined in Lemma~\ref{lem:non-A_n-not-GIT-stable}.
\end{lemma}

\begin{proof}
By a change of coordinates, we may assume $\pt = ([1:0],[1:0:0])$. Then the defining equation of $R$ in $\bb P^1_{[t_0:t_1]} \times \bb P^2_{[y_0:y_1:y_2]}$ is \(t_0^2Q_1+2t_0t_1Q_2+t_1^2 Q_3\), where \begin{gather*}
    Q_1 = a_{11}y_1^2+a_{12}y_1y_2+a_{22}y_2^2 , \qquad
    2Q_2 = b_{01}y_0y_1+b_{11}y_1^2+b_{12}y_1y_2+b_{22}y_2^2 +b_{02}y_0y_2 , \\
    Q_3 = c_{00}y_0^2+c_{01}y_0y_1+c_{11}y_1^2+c_{12}y_1y_2 + c_{22}y_2^2 + c_{02}y_0y_2
\end{gather*}
for some $a_{ij},b_{ij},c_{ij} \in \bb C$. The Hessian matrix of $R$ at $\pt$ is \begin{equation}\label{EQ:eqns-for-hessian}  H = \begin{pmatrix}
        2c_{00} & b_{01} & b_{02} \\
        b_{01} & 2a_{11} & a_{12} \\
        b_{02} & a_{12} & 2a_{22}
    \end{pmatrix}.\end{equation}
If $H$ has rank $0$, then $Q_1 = 0$ and $b_{01} = b_{02} = c_{00} = 0$. Then the 1-PS $\diag(-5,5) \times \diag(-2,1,1)$ shows $R$ is GIT unstable. Henceforth, we may assume $H$ has rank $1$. Then there exists a vector $(u,v,w) \in \bb \CC^3$ such that $H = (u,v,w)^t \cdot (u,v,w)$, whence we have the relations \(4a_{11}c_{00}=b_{01}^2\), \(4a_{22}c_{00}=b_{02}^2\), and \(4a_{11}a_{22}=a_{12}^2\).
The last relation implies $Q_1 = (\sqrt{a_{11}}y_1 \pm \sqrt{a_{22}}y_2)^2$. 

If $a_{11} = a_{22} = 0$, then $Q_1 = 0$ and $b_{01} = b_{02} = 0$, and the same 1-PS $\diag(-5,5) \times \diag(-2,1,1)$ from before shows $R$ is GIT unstable. On the other hand, if either $a_{11}$ or $a_{22}$ is non-zero, then without loss of generality we may assume $a_{11} \neq 0$. By the coordinate change $y_2 \mapsto \frac{1}{\sqrt{a_{22}}}(y_2 \mp \sqrt{a_{11}}y_1)$ (and replacing $b_{ij}$ and $c_{ij}$ correspondingly), we may assume $Q_1 = y_2^2$. Observe that this coordinate change fixes $\pt =([1:0],[1:0:0])$, and the relations $b_{01}^2 = 4a_{11}c_{00} = 0$ and $b_{02}^2 = 4a_{22}c_{00} = 4c_{00}$ still hold. To summarize, we may now assume the equation of \(R\) is
\[ t_0^2y_2^2 + t_0t_1(b_{11}y_1^2+b_{12}y_1y_2+b_{22}y_2^2+b_{02}y_0y_2) + t_1^2(\tfrac{1}{4}b_{02}^2y_0^2 + c_{01}y_0y_1 + c_{11}y_1^2 + c_{12}y_1y_2 + c_{22}y_2^2 + c_{02}y_0y_2) . \]
If $b_{02} = 0$, then $\diag(-2,2) \times \diag(-6,3,3)$ shows that $R$ is GIT unstable. If $b_{11} = 0$, then $\diag(-3,3) \times \diag(-2,-2,4)$ shows that $R$ is GIT unstable. If $b_{02}, b_{11} \neq 0$, then degenerating along \(\diag(-1,1)\times\diag(-1,0,1)\) shows \(R\) is S-equivalent to \(\Rthree\).
\end{proof}

Furthermore, \((2,2)\)-surfaces with \(E_n\) singularities are necessarily GIT unstable:
\begin{lemma}\label{lem:GITsemistable-D-n}
    Assume $R$ is a GIT semistable $(2,2)$-surface in $\bP^1 \times \bP^2$ with isolated singularities.  Then $R$ has $A_n$ $(1 \le n \le 7)$ or $D_n$ $(4 \le n \le 6)$ singularities.  In particular, \((2,2)\)-surfaces with $E_n$ singularities are GIT unstable. 
\end{lemma}

\begin{proof}
    First, a \((2,2)\)-surface in \(\bb P^1\times\bb P^2\) with only isolated singularities can only have ADE singularities by Lemma~\ref{lem:2-2-surfaces-have-ADE-singularities}. The claimed bounds on $n$ follow from the discussion before Lemma~\ref{lem:2-2-surfaces-have-ADE-singularities}.
    
    Now assume that $R$ is a GIT semistable surface with worse than $A_n$ isolated singularities at a point $p \in R$, so the Hessian of $R$ has $p$ has rank $\le 1$.  The proof of Lemma~\ref{lem:worse-than-A_n} shows that we may change coordinates so that \(\pt=([1:0],[1:0:0])\) and (after scaling \(y_0,y_1\)) $R$ is defined by
    \[ t_0^2y_2^2 + t_0t_1(y_1^2+b_{12}y_1y_2+b_{22}y_2^2+y_0y_2) + t_1^2(\tfrac{1}{4}y_0^2 + c_{01}y_0y_1 + c_{11}y_1^2 + c_{12}y_1y_2 + c_{22}y_2^2 + c_{02}y_0y_2) . \]
    To determine the singularity type of $R$ at $\pt$, we consider the first step of a resolution.

    We first recall relevant facts about surface singularities.
    In general, let $S \subset \mathbb{A}^3$ be a surface with type $D$ or $E$ singularities. If $\pi\colon S' \to S$ is the blow-up of a $D_4$ (respectively, $D_n$ with $n > 4$) singularity, then the exceptional divisor $C \subset S'$ is a rational curve such that $S'$ has three (respectively, two) singularities along $C$.  If instead $\pi$ is the blow-up of an $E_n$ singularity, then $S'$ has one singularity along the exceptional divisor $C$.
    
    We use this classification to determine the singularities of $R$.  By straightforward computation, the blow-up of $R$ given above $\pt$ has at least two singular points along the exceptional divisor and hence $R$ has $D_n$ singularities.  Precisely, $R$ has a $D_4$ singularity if $(2c_{01}-b_{12})^2-4(b_{22}-2c_{02}) \ne 0$ and a $D_5$ or $D_6$ singularity if $(2c_{01}-b_{12})^2-4(b_{22}-2c_{02}) = 0$.
\end{proof}

Next, we turn to reducible \((2,2)\)-surfaces. First, \cite{PR-GIT} shows that reducible surfaces cannot be GIT stable (see Remark~\ref{rem:PR-GIT} for a discussion regarding this preprint).

\begin{lemma}[{\cite[Section 6]{PR-GIT}}]\label{lem:reducible-not-GIT-stable}
    If \(R\subset\bb P^1\times\bb P^2\) is a reducible \((2,2)\)-surface, then \(R\) is not GIT stable.
\end{lemma}

We now show that the only GIT semistable reducible surfaces are essentially the following two:
\[\Rone :  t_0 t_1( y_0 y_2 + y_1^2), \qquad \Rtwo :  (t_0 y_2 + t_1 y_1)(t_0 y_1 + t_1 y_0).\]
\begin{proposition}\label{prop:reducible}
    For a reducible \((2,2)\)-surface \(R\subset\bb P^1\times\bb P^2\), the following conditions are equivalent.
    \begin{enumerate}
        \item\label{item:reducible-GIT-semistable} \(R\) is GIT semistable.
        \item\label{item:reducible-double-conic} \(\Delta\) (see Section~\ref{sec:preliminaries-R-and-Delta}) is a double conic (i.e. \(\Delta\)  is non-reduced and \(\Delta_{\red}\) is a smooth conic).
        \item\label{item:reducible-R_1-R_2} \(R\) is either \(S\)-equivalent to \(\Rone\), or \(R\) is (up to a coordinate change) \(\Rtwo\).
    \end{enumerate}
\end{proposition}

\begin{proof}
    First, since the double conic is GIT semistable (see Section~\ref{sec:GIT-quartics}), \eqref{item:reducible-double-conic}\(\Rightarrow\)\eqref{item:reducible-GIT-semistable} by Lemma~\ref{lem:GIT-stability-Delta-R}. Next, computing \(\Delta\) from the equations of \(\Rone,\Rtwo\) and openness of GIT semistability shows \eqref{item:reducible-R_1-R_2}\(\Rightarrow\)\eqref{item:reducible-GIT-semistable}. Thus, it remains to show that \eqref{item:reducible-GIT-semistable}\(\Rightarrow\)\eqref{item:reducible-double-conic},\eqref{item:reducible-R_1-R_2}. To prove this, we will consider several cases. Write \(R=S_0+S_1\) where each \(S_i=(f_i=0)\) is a (possibly reducible or non-reduced) divisor in \(\bb P^1\times\bb P^2\).
    \begin{enumerate}[label=(\roman*)]
        \item\label{item:case-2,0-plus-0,2} If \(R\in|(2,0)|+|(0,2)|\), then \(f_0=at_0^2+bt_0t_1+ct_1^2\), \(f_1\) is a quadratic form in \(\bb C[y_0,y_1,y_2]\), and \(\Delta\) is defined by \((4ac-b^2)f_1^2\). If \(\Delta_{\red}\) is not a smooth conic, then either \(f_0\) or \(f_1\) fails to have full rank. If \(f_0\) has rank 1, then after a coordinate change we may assume \(f_0=t_1^2\), showing \(R\) is GIT unstable. If \(f_0\) has full rank and \(f_1\) has rank \(\leq 2\), then after a coordinate change we may assume \(R\) is defined by \(t_0t_1y_2(b_{02}y_0+b_{12}y_1)\) or \(t_0t_1 y_2^2\), and \(\diag(0,0)\times\diag(-2,-2,4)\) shows \(R\) is GIT unstable. (See also \cite[Proof of Theorem 6.6, Case 2]{PR-GIT}.)
        \item If \(R\in|(2,1)|+|(0,1)|\), then \(R\) is GIT unstable by \cite[Proof of Theorem 6.6, Case 4]{PR-GIT}.
        \item\label{item:case-1,0-plus-1,2} If \(R\in|(1,0)|+|(1,2)|\), then after a coordinate change on \(\bb P^1\) we may write \(f_0=t_1\) and \(f_1=t_0q_0 + t_1q_1\) for quadratic forms \(q_i\in\bb C[y_0,y_1,y_2]\). Then \(\Delta\) is defined by \(q_0^2\), and the intersection \(S_0\cap S_1\) is defined by \((q_0=0)\subset[1:0]\times\bb P^2\). If \(q_0\) has rank \(\leq 2\), then \(R\) is GIT unstable by \cite[Proof of Proposition 6.5]{PR-GIT}. Indeed, after a coordinate change on \(\bb P^2\), we may assume \(R\) is defined by \(t_0t_1y_2(b_{02}y_0+b_{12}y_1) + t_1^2 q_1\) or \(t_0t_1 y_2^2 + t_1^2 q_1\), and \(\diag(-3,3)\times\diag(-2,-2,4)\) shows \(R\) is GIT unstable.
        \item\label{item:case-1,1-plus-1,1} If \(R\in|(1,1)|+|(1,1)|\), we may assume \(S_0\) and \(S_1\) are irreducible by the previous cases. We will show that \(R\) is GIT semistable if and only if \(R=\Rtwo\) (up to coordinate change) if and only if \(\Delta\) is a double conic.
        
        Write each \(f_i=t_0 l_i + t_1 h_i\) where \(l_i,h_i\in\bb C[y_0,y_1,y_2]\) are linear forms. Then \(\Delta\) is defined by \((h_0 l_1-h_1 l_0)^2\). If \(S_0\to\bb P^1\) and \(S_1\to\bb P^1\) have common fibers over a point in \(\bb P^1\), then \(R\) is GIT unstable by \cite[Lemma 6.3]{PR-GIT}, so we may assume that \(S_0\to\bb P^1\) and \(S_1\to\bb P^1\) do not have common fibers over any point in \(\bb P^1\). Then after a coordinate change, we may assume $h_0=y_1$ and $h_1=y_0$. Let $c_i$ denote the coefficient of $y_2$ in $l_i$. 
        \begin{enumerate}
            \item If $(c_0,c_1) = (0,0)$, then both $R$ and $\Delta$ do not involve the variable $y_2$, so $R$ is GIT unstable and $\Delta$ is $2 \times$(a conic of rank $\leq 2$).

            \item If $(c_0,c_1) \neq (0,0)$, then by swapping $y_0$ with $y_1$ if necessary, we may assume $c_0 \neq 0$. After a coordinate change, we may assume $l_0 = y_2$, so \(R\) is defined by \((t_0y_2+t_1 y_1)(t_0 l_1 + t_1y_0)\) and \(\Delta\) is defined by \(-(l_1 y_1 - y_0 y_2)^2\). Write \(l_1=ay_0+by_1+cy_2\)
        for some \(a,b,c\in\bb C\).
        After the coordinate change $t_1 \mapsto t_1 - at_0$ on $\bb P^1_{[t_0:t_1]}$, $R$ is defined by $(t_0(y_2-ay_1)+t_1y_1)(t_0(by_1+cy_2)+t_1y_0)$. We now consider two cases: \(b+ac = 0\) and \(b+ac \neq 0\).
        
        If $b+ac = 0$, then we make the coordinate change $y_2 \mapsto y_2+ay_1$ on $\bb P^2_{[y_0:y_1:y_2]}$ (keeping $y_0,y_1$ intact), so that $R$ is defined by $(t_0y_2+t_1y_1)(ct_0y_2+t_1y_0)$. Then $\diag(-2,2) \times \diag(-2,-1,3)$ shows $R$ is GIT unstable.
        
        If $b+ac \neq 0$, then by computation \(\Delta_{\red}=V((ay_0+by_1+cy_2)y_1-y_0y_2)\) is a smooth conic. Furthermore, in this case after the following coordinate changes:
        \[
            \qquad\qquad \begin{pmatrix}
                y_1 \\ y_2
            \end{pmatrix} \mapsto \begin{pmatrix}
                b & c \\
                -a & 1
            \end{pmatrix}^{-1} \begin{pmatrix}
                y_1 \\ y_2
            \end{pmatrix}, \quad
            \begin{pmatrix}
            t_0 \\ t_1
            \end{pmatrix} \mapsto \begin{pmatrix}
                1 & \frac{c}{b+ac} \\ 0 & b+ac
            \end{pmatrix} \begin{pmatrix}
                t_0 \\ t_1
            \end{pmatrix}, \quad y_0 \mapsto \frac{1}{b+ac}(y_0-cy_1),
        \]
\begin{detail}        
        the first coordinate change makes the defining equation of \(R\) \((t_0y_2 + \frac{1}{b+ac}(y_1-cy_2))(t_0y_1+t_1y_0)\). The second coordinate change makes it \((t_0y_2+t_1y_1)(t_0y_1+t_1((b+ac)y_0+cy_1)\). Finally, after the third coordinate change $y_0 \mapsto \frac{1}{b+ac}(y_0-cy_1)$ on $\bb P^2_{[y_0:y_1:y_2]}$ (with $y_1,y_2$ left intact),
\end{detail}
the defining equation of \(R\) becomes $(t_0y_2+t_1y_1)(t_0y_1+t_1y_0)$, i.e. $R = \Rtwo$.
        \end{enumerate}
    \end{enumerate}
    
    To summarize, from the above cases, we see that a reducible \((2,2)\)-surface \(R\) has \(\Delta\) a double conic if and only if \(R \in |(2,0)|+|(0,2)|\) with both \(f_i\) of full rank (case~\ref{item:case-2,0-plus-0,2}), \(R\in|(1,0)|+|(1,2)|\) with \(S_0\cap S_1\) a smooth conic (case~\ref{item:case-1,0-plus-1,2}), or \(R\in|(1,1)|+|(1,1)|\) is \(\Rtwo\) after a coordinate change (case~\ref{item:case-1,1-plus-1,1}). In the \(|(2,0)|+|(0,2)|\) case, \(R\) is \(\Rone\) after a coordinate change. In the \(|(1,0)|+|(1,2)|\) case, we may assume \(R\) is defined by \(t_0t_1 q_0+t_1^2 q_1\) where the quadratic form \(q_0\) has full rank; this degenerates to \(\Rone\) by \(\diag(1,-1)\times\diag(0,0,0)\).
\end{proof}
\subsection{GIT strictly semistable \texorpdfstring{\((2,2)\)}{(2,2)}-surfaces}\label{sec:GIT-strictly-ss}
In this section, we find all the \((2,2)\)-surfaces in \(\bb P^1\times\bb P^2\) that are GIT semistable but not GIT stable:

\begin{proposition}\label{prop:GIT-strictly-semistable}
Let \(R\subset\bb P^1\times\bb P^2\) be an irreducible \((2,2)\)-surface. Then \(R\) is GIT strictly semistable if and only if, up to a coordinate change, it is defined by one of the following:
\begin{enumerate}[label=(\alph*)]
    \item\label{case:Gamma2} \(t_0^2(a_{22} y_2^2 + a_{12} y_1 y_2) + t_0 t_1(b_{11} y_1^2 + b_{22} y_2^2 + b_{02} y_0 y_2 + b_{12} y_1 y_2) + t_1^2(c_{11}y_1^2 + c_{22} y_2^2 + c_{01} y_0 y_1 + c_{02} y_0 y_2 + c_{12} y_1 y_2)\) with \(a_{12}, b_{02}, c_{01} \neq 0\) or \(b_{11}, b_{02} \neq 0\),
    \item\label{case:Gamma3} \(t_0^2(a_{22} y_2^2) + t_0 t_1(b_{11} y_1^2 + b_{22} y_2^2 + b_{02} y_0 y_2 + b_{12} y_1 y_2) + t_1^2(c_{00} y_0^2 + c_{11}y_1^2 + c_{22} y_2^2 + c_{01} y_0 y_1 + c_{02} y_0 y_2 + c_{12} y_1 y_2)\) with \(a_{22}, b_{11}, c_{00} \neq 0\) or \(b_{11}, b_{02} \neq 0\),
    \item\label{case:Gamma4} \(t_0^2(a_{11} y_1^2 + a_{22} y_2^2 + a_{12} y_1 y_2 + a_{02} y_0 y_2) + t_0 t_1 (b_{11} y_1^2 + b_{22} y_2^2 + b_{12} y_1 y_2 + b_{02} y_0 y_2) + t_1^2 (c_{11} y_1^2 + c_{22} y_2^2 + c_{12} y_1 y_2 + c_{02} y_0 y_2)\) such that \(b_{02}^2-4 a_{02} c_{02}\neq 0\) and at least one of \(4 a_{11} c_{11} - b_{11}^2\) or \(2 a_{11} c_{02} + 2 a_{02} c_{11} - b_{02} b_{11}\) is nonzero.
\end{enumerate}
Furthermore, if the quartic curve \(\Delta\coloneqq(Q_1 Q_3 - Q_2^2=0)\) is GIT unstable, then \(R\) is in case~\ref{case:Gamma3} and is S-equivalent to \(\Rthree\).
\end{proposition}
The three cases~\ref{case:Gamma2},\ref{case:Gamma3},\ref{case:Gamma4} above are \emph{not} mutually exclusive. We also note that for one case in the proof (Case~\eqref{case:becomes-Gamma_3} in the proof), we rely on a result that we will prove later in Section~\ref{sec:G3-polystability}.
\begin{detail}
    In the Proposition~\ref{prop:GIT-strictly-semistable}, S-equivalence is needed, since \(D_n\) quartics are GIT unstable.
\end{detail}

\begin{proof}
    Let \(R=(f=0)\) be an irreducible \((2,2)\)-surface that is GIT semistable but not GIT stable. We may assume, after a coordinate change, that there is a normalized 1-parameter subgroup \(\lambda\) such that \(\mu(f,\lambda)\geq 0\). The proofs of the forward directions of \cite[Lemma 4.1, Theorem 1.2]{PR-GIT} show that the point \(\pt\coloneqq([1:0],[1:0:0])\) is contained in \(R\), and that
\begin{itemize}
\item If \(R\) has worse than an \(A_1\) singularity at \(\pt\), then \(f\) has the form
\begin{equation}\label{case:becomes-Gamma_2}\begin{array}{cc} t_0^2( a_{22} y_2^2 + a_{12} y_1 y_2) + t_0 t_1(b_{11} y_1^2 + b_{22} y_2^2 + b_{02} y_0 y_2 + b_{12} y_1 y_2) \\ + t_1^2(c_{11}y_1^2 + c_{22} y_2^2 + c_{01} y_0 y_1 + c_{02} y_0 y_2 + c_{12} y_1 y_2)\end{array} \quad\text{with } b_{02}\neq 0,\end{equation}
or
\begin{equation}\label{case:becomes-Gamma_3}
\begin{array}{cc}t_0^2(a_{22} y_2^2) + t_0 t_1(b_{11} y_1^2 + b_{22} y_2^2 + b_{02} y_0 y_2 + b_{12} y_1 y_2) \\ + t_1^2(c_{00} y_0^2 + c_{11}y_1^2 + c_{22} y_2^2 + c_{01} y_0 y_1 + c_{02} y_0 y_2 + c_{12} y_1 y_2)\end{array} \quad \text{with } (b_{02},c_{00})\neq(0,0).\end{equation}
Notice that in case~\eqref{case:becomes-Gamma_2}, the fiber of \(\piR_2\colon R\to\bb P^2\) containing \(\pt\) is non-finite. In case~\eqref{case:becomes-Gamma_3}, the fiber of \(\piR_1\colon R\to\bb P^1\) containing \(\pt\) is non-reduced or is \(\{[1:0]\}\times\bb P^2\).

\item If \(R\) is smooth at \(\pt\), then \(f\) has the form
\begin{equation}\label{case:smooth-becomes-Gamma_4}\begin{array}{cc} t_0^2 (a_{11} y_1^2 + a_{22} y_2^2 + a_{12} y_1 y_2 + a_{02} y_0 y_2) + t_0 t_1(b_{11} y_1^2 + b_{22} y_2^2 \\ + b_{02} y_0 y_2 + b_{12} y_1 y_2) + t_1^2(c_{11} y_1^2 + c_{22} y_2^2 + c_{02} y_0 y_2 + c_{12} y_1 y_2)\end{array} \quad \text{with } a_{02}\neq 0.\end{equation}
If \(R\) has an \(A_1\) singularity at \(\pt\), then \(f\) has the form
\begin{equation}\label{case:A_1-sing-becomes-Gamma_4}
\begin{array}{cc}t_0^2 (a_{11} y_1^2 + a_{22} y_2^2 + a_{12} y_1 y_2) + t_0 t_1(b_{11} y_1^2 + b_{22} y_2^2 + b_{02} y_0 y_2 + b_{12} y_1 y_2) \\ + t_1^2(c_{11} y_1^2 + c_{22} y_2^2 + c_{02} y_0 y_2 + c_{12} y_1 y_2)\end{array} \quad \text{with } a_{11}, b_{02}\neq 0.\end{equation}
Notice that the fiber of \(\piR_2\colon R\to\bb P^2\) containing \(\pt\) is non-finite in cases~\eqref{case:smooth-becomes-Gamma_4} and~\eqref{case:A_1-sing-becomes-Gamma_4}. In case~\eqref{case:smooth-becomes-Gamma_4}, this fiber is singular at the point(s) where \(2a_{02} t_0 = (-b_{02}\pm\sqrt{b_{02}^2-4a_{02}c_{02}})t_1\).
\end{itemize}
In each of the above cases, taking \(\lambda\) to be the following subgroups from the proof of \cite[Theorem 5.1]{PR-GIT}, we have \(\mu(f,\lambda)=0\):
\begin{itemize}
\item Case~\eqref{case:becomes-Gamma_2}: \(\diag(-1,1)\times\diag(-2,0,2)\), 
\item Case~\eqref{case:becomes-Gamma_3}: \(\diag(-1,1)\times\diag(-1,0,1)\), \item Cases~\eqref{case:smooth-becomes-Gamma_4} and~\eqref{case:A_1-sing-becomes-Gamma_4}: \(\diag(0,0)\times\diag(-1,0,1)\).\end{itemize}

The union of the four cases~\eqref{case:becomes-Gamma_2}, \eqref{case:becomes-Gamma_3}, \eqref{case:smooth-becomes-Gamma_4}, and \eqref{case:A_1-sing-becomes-Gamma_4} together gives a superset of the possibilities listed in Proposition~\ref{prop:GIT-strictly-semistable}. So it remains to show that
\begin{enumerate}[label=(\alph*)]
    \item In case~\eqref{case:becomes-Gamma_2}, \(R\) is GIT semistable if and only if \(a_{12},c_{01}\neq 0\) or \(b_{11}\neq 0\),
    \item In case~\eqref{case:becomes-Gamma_3}, \(R\) is GIT semistable if and only if \(a_{22},b_{11},c_{00}\neq 0\) or \(b_{11},b_{02}\neq 0\),
    \item In case~\eqref{case:smooth-becomes-Gamma_4} or case~\eqref{case:A_1-sing-becomes-Gamma_4}, \(R\) is GIT semistable if and only if \(b_{02}^2-4 a_{02} c_{02}\neq 0\) and \((4 a_{11} c_{11} - b_{11}^2 , 2 a_{11} c_{02} + 2 a_{02} c_{11} - b_{02} b_{11})\neq (0,0)\).
\end{enumerate}

\noindent\underline{Case~\eqref{case:becomes-Gamma_2}:} First, if \(a_{12}=b_{11}=0\) or if \(b_{11}=c_{01}=0\), then \(R\) is GIT unstable: use \(\diag(-2,2)\times(-2,-1,3)\) if \(a_{12}=b_{11}=0\), and use \(\diag(-2,2)\times(-5,-1,6)\) if \(b_{11}=c_{01}=0\).

Now assume that \(a_{12},c_{01}\neq 0\) or \(b_{11}\neq 0\). To show \(R\) is GIT semistable, define the \((2,2)\)-surface \[R' \coloneqq (t_0^2( a_{12} y_1 y_2) + t_0 t_1(b_{11} y_1^2 + b_{02} y_0 y_2) + t_1^2(c_{01} y_0 y_1)=0),\] with \(b_{02}\neq 0\), obtained by degeneration along \(\diag(-1,1)\times\diag(-2,0,2)\). If \(b_{11}=0\) and \(a_{12},c_{01}\neq 0\), then the associated quartic \(\Delta'\) is an ox \(y_0 y_2(y_1^2+y_0 y_2)\); if \(b_{11}\neq 0\), then \(\Delta'\) is a cat-eye or a double conic. In either case \(\Delta'\) is GIT semistable (Section~\ref{sec:GIT-quartics}), so \(R'\) and hence \(R\) is GIT semistable by Lemma~\ref{lem:GIT-stability-Delta-R}.

\noindent\underline{Case~\eqref{case:becomes-Gamma_3}:}
First, if \(b_{11}=0\) or if \(a_{22}=b_{02}=0\), then \(R\) is GIT unstable: use \(\diag(-5,5)\times\diag(-4,-3,7)\) if \(b_{11}=0\), and use \(\diag(-3,3)\times\diag(-2,1,1)\) if \(a_{22}=b_{02}=0\).

Next, assume that \(b_{11}\neq 0\) and that at least one of \(a_{22}\) or \(b_{02}\) is nonzero. Define
\[R'\coloneqq (t_0^2(a_{22} y_2^2) + t_0 t_1(b_{11} y_1^2 + b_{02} y_0 y_2) + t_1^2(c_{00} y_0^2)=0),\]
with at least one of \(b_{02}\) or \(c_{00}\) nonzero, obtained by degeneration along \(\diag(-1,1)\times\diag(-1,0,1)\). If \(4 a_{22} c_{00} - b_{02}^2 \neq 0\), then the quartic \(\Delta'\) associated to \(R'\) is either a cat-eye or a double conic, so \(R\) is GIT semistable by Lemma~\ref{lem:GIT-stability-Delta-R}. If \(4 a_{22} c_{00} - b_{02}^2 = 0\), then the assumptions imply that \(a_{22}\neq 0\), so after a coordinate change we may assume \(a_{22}=b_{11}=1\). The defining equation of \(R'\) is then \(t_0t_1y_1^2+(t_0y_2 + \frac{1}{2} b_{02} t_1 y_0)^2\), so \(R'\) is isomorphic to \(\Rthree\). We will show in Section~\ref{sec:G3-polystability} that the pair \((\bb P^1\times\bb P^2, c\Rthree)\) is K-polystable for \(c<0.25344\), so \(R'\) and hence \(R\) is GIT semistable by Theorems~\ref{lem:K-ps-G-invariant-delta} and~\ref{thm:K-implies-GIT}.

\noindent\underline{Cases~\eqref{case:smooth-becomes-Gamma_4} and~\eqref{case:A_1-sing-becomes-Gamma_4}:} In these two cases, the equation of \(R\) is given by \(t_0^2 (a_{11} y_1^2 + a_{22} y_2^2 + a_{12} y_1 y_2 + a_{02} y_0 y_2) + t_0 t_1(b_{11} y_1^2 + b_{22} y_2^2 + b_{02} y_0 y_2 + b_{12} y_1 y_2) + t_1^2(c_{11} y_1^2 + c_{22} y_2^2 + c_{02} y_0 y_2 + c_{12} y_1 y_2)\) with \(a_{02}\neq 0\) (case~\eqref{case:smooth-becomes-Gamma_4}), or with \(a_{02}=0\) and \(a_{11},b_{02}\neq 0\) (case~\eqref{case:A_1-sing-becomes-Gamma_4}). We will show that the following hold:
\begin{enumerate}[label=(\roman*)]
\item\label{item:b02^2=4a02c02} If \(b_{02}^2=4 a_{02} c_{02}\), then \(R\) is GIT unstable.
\item\label{item:double-reducible-conic} If \(4 a_{11} c_{11} - b_{11}^2 = 2 a_{11} c_{02} + 2 a_{02} c_{11} - b_{02} b_{11} = 0\), then \(R\) is GIT unstable.
\item\label{Gamma_4-all-others} If neither~\ref{item:b02^2=4a02c02} nor~\ref{item:double-reducible-conic} holds, then \(R\) is GIT semistable.
\end{enumerate}

For~\ref{item:b02^2=4a02c02}, in case~\eqref{case:A_1-sing-becomes-Gamma_4} we have \(a_{02}=b_{02}=0\), so \(R\) is GIT unstable using \(\diag(-2,2)\times\diag(-6,3,3)\). In case~\eqref{case:smooth-becomes-Gamma_4}, we may assume \(a_{02}=1\). If \(b_{02}=0\), then \(c_{02}=0\), so after \(t_0\leftrightarrow t_1\) we are in the previous case. For case~\eqref{case:smooth-becomes-Gamma_4} in general, the coefficient of \(y_0y_2\) in the equation of \(R\) is \(t_0^2 + b_{02} t_0 t_1 + c_{02} t_1^2 = (t_0 + \frac{b_{02}}{2} t_1)^2\), so after a coordinate change we may reduce to the \(b_{02}=0\) case. This shows~\ref{item:b02^2=4a02c02}.

For~\ref{item:double-reducible-conic}, first assume \(a_{11}=0\). The assumptions of~\ref{item:double-reducible-conic} then imply that \(b_{11}= a_{02}c_{11}=0\). Then \(R\) is GIT unstable: use \(\diag(-1,1)\times\diag(-3,-3,6)\) if \(a_{11}=b_{11}=c_{11}=0\), and use \(\diag(-2,2)\times\diag(-5,-1,6)\) if \(a_{11}=b_{11}=a_{02}=0\).

So we may assume \(a_{11} = 1\). The coefficient of \(y_1^2\) in the defining equation of \(R\) is \(t_0^2 + b_{11}t_0t_1 + c_{11}t_1^2 = (t_0 + \frac{b_{11}}{2} t_1)^2\), so after a coordinate change, the defining equation of \(R\) becomes
\[t_0^2(y_1^2 + a_{22}'y_2^2 + a_{12}'y_1y_2 + a_{02}'y_0y_2) + t_0 t_1(b_{22}'y_2^2 + b_{02}'y_0y_2 + b_{12}'y_1 y_2) + t_1^2(c_{22}'y_2^2 + c_{02}'y_0y_2 + c_{12}'y_1y_2)\]
where the new coefficient of \(t_1^2y_0y_2\) is \(c_{02}' = \frac{1}{4} a_{02}b_{11}^2 - \frac{1}{2} b_{02}b_{11} + c_{02} = a_{02}c_{11} - \frac{1}{2} b_{02}b_{11} + c_{02}=0\). So after \(t_0\leftrightarrow t_1\) we return to the \(a_{11}=b_{11}=a_{02}=0\) case. This shows~\ref{item:double-reducible-conic}.

For~\ref{Gamma_4-all-others}, let \(R'\) be the \((2,2)\)-surface defined by \[t_0^2(a_{11} y_1^2 + a_{02} y_0 y_2) + t_0 t_1 (b_{11} y_1^2 + b_{02} y_0 y_2) + t_1^2 (c_{11} y_1^2  + c_{02} y_0 y_2)\] with \(a_{02}\neq 0\) or \(a_{11},b_{02}\neq 0\), obtained by degeneration along \(\diag(0,0)\times\diag(-1,0,1)\). The associated quartic \(\Delta'\) is defined by \[y_1^4 (4 a_{11} c_{11}-b_{11}^2) + 2 y_0 y_2 y_1^2 (2 a_{11} c_{02} + 2 a_{02} c_{11} - b_{02} b_{11}) + y_0^2 y_2^2 (4 a_{02} c_{02}-b_{02}^2).\]
If \(4 a_{11} c_{11}-b_{11}^2\neq 0\), then \(\Delta'\) is a cat-eye if and only if \((2 a_{11} c_{02} + 2 a_{02} c_{11} - b_{02} b_{11})^2-(4 a_{11} c_{11}-b_{11}^2)(4 a_{02} c_{02}-b_{02}^2)\neq 0\); otherwise, \(\Delta'\) is the double of a smooth conic. If \(4 a_{11} c_{11}-b_{11}^2=0\), then the assumptions imply that \(\Delta'\) is an ox. So in any case, \(\Delta'\) is GIT semistable, so \(R'\) and \(R\) are as well by Lemma~\ref{lem:GIT-stability-Delta-R}.

\end{proof}

\begin{proposition}\label{prop:representatives-degen-Gamma_i}
    Let \(R\) be a GIT strictly semistable \((2,2)\)-surface in \(\bb P^1\times\bb P^2\). Then \(R\) is S-equivalent to a member of one of the following sets:
    \begin{equation*}\begin{split}
        \Gfour &\coloneqq \{(t_0^2 y_1^2+t_0t_1(b_{11}y_1^2+y_0 y_2)+t_1^2y_1^2=0) \mid b_{11}\in\bb C\}/\langle b_{11}\sim -b_{11}\rangle \cup [ \Rone ], \\
        \Gtwo &\coloneqq \{ (t_0^2 y_1 y_2 + t_0 t_1(b_{11} y_1^2 + y_0 y_2) + t_1^2y_0 y_1=0) \mid b_{11}\in\bb C\} \cup [ \Rone ], \\
        \Gthree &\coloneqq\{ (t_0^2 y_2^2 + t_0 t_1 (y_1^2 + b_{02} y_0 y_2) + t_1^2 y_0^2=0) \mid b_{02}\in\bb C\}/\langle b_{02}\sim -b_{02}\rangle \cup [ \Rone ].
    \end{split}\end{equation*}
\end{proposition}

\begin{proof}[Proof of Proposition~\ref{prop:representatives-degen-Gamma_i}]
    First, note that the GIT semistable reducible \((2,2)\)-surfaces from Proposition~\ref{prop:reducible} are \(\Rone\) and \(\Rtwo\in\Gtwo\). So it remains to deal with the irreducible divisors. For this, we address the cases in Proposition~\ref{prop:GIT-strictly-semistable} one by one. After first degenerating along the 1-parameter subgroups \(\lambda\) in the proof of Proposition~\ref{prop:GIT-strictly-semistable}, the proof of Case~\eqref{case:becomes-Gamma_2} shows that a surface in case~\ref{case:Gamma2} is S-equivalent to one in \(\Gtwo\), and the proof of Case~\eqref{case:becomes-Gamma_3} shows that a surface in case~\ref{case:Gamma3} is S-equivalent to one in \(\Gthree\). (See also \cite[Section 8.1]{PR-GIT}.) Lemma~\ref{lem:G_i-equivalent-in-GIT} below shows that \(b_{02}\) and \(-b_{02}\) define \(\SL_2\times\SL_3\)-equivalent surfaces in \(\Gthree\).

    It remains to show that a surface in case~\ref{case:Gamma4} is S-equivalent to one in \(\Gfour\). We first degenerate \(R\) along \(\diag(0,0)\times\diag(-1,0,1)\) to \(t_0^2(a_{11} y_1^2 + a_{02} y_0 y_2) + t_0 t_1 (b_{11} y_1^2 + b_{02} y_0 y_2) + t_1^2 (c_{11} y_1^2  + c_{02} y_0 y_2)\) with \(a_{02}\neq 0\) or \(a_{11},b_{02}\neq 0\), as in Cases~\eqref{case:smooth-becomes-Gamma_4} and~\eqref{case:A_1-sing-becomes-Gamma_4}. The assumption \(b_{02}^2\neq 4a_{02}c_{02}\) implies that, after a coordinate change, we may assume the equation of \(R'\) is \(y_1^2(a_{11}t_0^2 + b_{11}t_0t_1 + c_{11}t_1^2)+y_0y_2 t_0t_1\), with \(b_{11}^2-4 a_{11} c_{11}\neq 0\) or \(b_{11}\neq 0\) by the second assumption in~\ref{case:Gamma4}.
    
    First, if \(a_{11}=c_{11}=0\), then \(b_{11}\neq 0\) and \(R\cong \Rone \in \Gfour\). Thus, we may assume at least one of \(a_{11},c_{11}\) is nonzero. After a coordinate change, we may further assume \(a_{11}=1\). The equation becomes \(y_1^2(t_0^2 + b_{11}t_0t_1 + c_{11}t_1^2) + y_0y_2t_0t_1\) with \(b_{11}^2-4 c_{11}\neq 0\) or \(b_{11}\neq 0\).
    If \(c_{11}\neq 0\), then after a coordinate change we may assume \(c_{11}=1\). The defining equation then becomes \(y_1^2(t_0^2+b_{11}t_0t_1+t_1^2)+y_0y_2t_0t_1\), which is in \(\Gfour\). On the other hand, if \(c_{11}=0\), then \(b_{11}\neq 0\) so after a coordinate change, the defining equation is \(y_1^2(t_0^2+t_0t_1)+y_0y_2t_0t_1 = t_0t_1(y_1^2 +y_0y_2) + t_0^2y_1^2\). Using \(\diag(-1,1)\times\diag(0,0,0)\), we see this is S-equivalent to \(\Rone\). Thus, we've shown that every surface in case~\ref{case:Gamma4} is S-equivalent to one in \(\Gfour\). Lemma~\ref{lem:G_i-equivalent-in-GIT} below shows that \(b_{11}\) and \(-b_{11}\) define \(\SL_2\times\SL_3\)-equivalent surfaces in \(\Gfour\), completing the proof.
\end{proof}

\begin{lemma}\label{lem:G_i-equivalent-in-GIT}
    For \(b\in\bb C\), define the following bidegree \((2,2)\) polynomials:
    \begin{equation*}\begin{split}
        f_b^{(1)} &\coloneqq t_0^2 y_1^2+t_0t_1(b y_1^2+y_0 y_2)+t_1^2y_1^2, \\
        f_b^{(2)} &\coloneqq t_0^2 y_1 y_2 + t_0 t_1(b y_1^2 + y_0 y_2) + t_1^2y_0 y_1, \\
        f_b^{(3)} &\coloneqq t_0^2 y_2^2 + t_0 t_1 (y_1^2 + b y_0 y_2) + t_1^2 y_0^2.
    \end{split}\end{equation*}
    For each \(i=1,2,3\), let \(R_b^{(i)}\) be the \((2,2)\)-surface defined by \(f_b^{(i)}\); then \(R_b^{(1)}\in\Gfour\), \(R_b^{(2)}\in\Gtwo\), and \(R_b^{(3)}\in\Gthree\).
    Then \(R_b^{(2)}\) is not in the \(\SL_2\times\SL_3\)-orbit of \(R_{b'}^{(2)}\) for any \(b'\neq b\).
    For \(i=1,3\), \(R_b^{(i)}\) is in the same \(\SL_2\times\SL_3\)-orbit as \(R_{b'}^{(i)}\)if and only if \(b=\pm b'\).
\end{lemma}

\begin{proof}
    First, the surfaces \(R^{(i)}_b, R^{(i)}_{-b}\) are equivalent under \[\begin{pmatrix} 0 & 1 \\ -1 & 0 \end{pmatrix}\times\frac{1}{\sqrt[3]{-1}}\begin{pmatrix} 0 & 0 & 1 \\ 0 & -1 & 0 \\ -1 & 0 & 0 \end{pmatrix}
    \quad \text{ and }\quad
    \frac{1}{\sqrt{-1}}\begin{pmatrix} 0 & 1 \\ 1 & 0 \end{pmatrix}\times\frac{1}{\sqrt[3]{-1}}\begin{pmatrix} 0 & 0 & -1 \\ 0 & -1 & 0 \\ 1 & 0 & 0 \end{pmatrix}
    \] for \(i=1,3\) respectively. It remains to show that these are the only equivalences for \(i=1,3\), and that there are none for \(i=2\). Let \((A,C)\in\GL_2\times\GL_3\), and let \(\phi_{(A,C)}\) denote the corresponding coordinate change. The claims in the following paragraph can be verified by direct computation.

    We first consider \(i=1\). Solving for pairs of matrices \((A,C)\) such that \[\coeff_{f_b^{(1)}}(t_0^{2-i}t_1^iy_0^{2-j-k}y_1^j y_2^k) = \coeff_{\phi_{(A,C)}^* f_b^{(1)}}(t_0^{2-i}t_1^iy_0^{2-j-k}y_1^j y_2^k) \text{ for all }t_0^{2-i}t_1^iy_0^{2-j-k}y_1^j y_2^k\neq t_0t_1y_1^2\] yields that \(\phi_{(A,C)}^* f_b^{(1)}=f_b^{(1)}\) or \(\phi_{(A,C)}^* f_b^{(1)}=f_{-b}^{(1)}\). Performing the same checks for \(i=2,3\), we obtain the claimed the result.
\begin{detail}    
    That is, for \(R_b^{(2)}\), solving for pairs of matrices \((A,C)\) satisfying
    \[\coeff_{f_b^{(2)}}(t_0^{2-i}t_1^iy_0^{2-j-k}y_1^j y_2^k) = \coeff_{\phi_{(A,C)}^* f_b^{(2)}}(t_0^{2-i}t_1^iy_0^{2-j-k}y_1^j y_2^k) \text{ for all }t_0^{2-i}t_1^iy_0^{2-j-k}y_1^j y_2^k\neq t_0t_1y_1^2,\] gives \(\phi_{(A,C)}^* f_b^{(2)}=f_b^{(2)}\).
    Finally, for \(R_b^{(3)}\), finding matrices \((A,C)\) such that \[\coeff_{f_b^{(3)}}(t_0^{2-i}t_1^iy_0^{2-j-k}y_1^j y_2^k) = \coeff_{\phi_{(A,C)}^* f_b^{(3)}}(t_0^{2-i}t_1^iy_0^{2-j-k}y_1^j y_2^k)\] for all \(t_0^{2-i}t_1^iy_0^{2-j-k}y_1^j y_2^k\neq t_0t_1y_0y_2\) yields that \(\phi_{(A,C)}^* f_b^{(3)}=f_b^{(3)}\) or \(\phi_{(A,C)}^* f_b^{(3)}=f_{-b}^{(3)}\).
\end{detail}
\end{proof}

\begin{remark}\label{rem:frakG_i-surfaces-descriptions}
    One can check by direct computation that the \((2,2)\)-surfaces \(R\) contained in the loci \(\Gfour\), \(\Gtwo\), and \(\Gthree\) have the following descriptions.
    \begin{enumerate}
    \item If \(R\in\Gfour\setminus[ \Rone ]\) and \(b_{11}^2\neq 4\), then \(\Delta\) is a cat-eye. \(R\) has non-finite fibers over the two singular points \([1:0:0], [0:0:1]\) of \(\Delta\). The singular points of \(R\) are the four \(A_1\) singularities \(([1:0],[1:0:0])\), \(([0:1],[1:0:0])\), \(([1:0],[0:0:1])\), \(([0:1],[0:0:1])\); there are two singular points of \(R\) on each non-finite fiber of \(\piR_2\).
    \item If \(R\in\Gfour\setminus[ \Rone ]\) and \(b_{11}^2= 4\), then \(R\cong \Rfour\) and \(\Delta\) is an ox. \(\Rfour\) has the four \(A_1\) singularities described above, together with an additional \(A_1\) singularity on the (finite) fiber over the nodal point of the ox.
    \item If \(R\in\Gtwo\setminus[ \Rone ]\) and \(b_{11}\neq 1\), then \(\Delta\) is a cat-eye with two \(A_3\) singularities at \([1:0:0], [0:0:1]\). Over each of these two points, \(R\) has a non-finite fiber with exactly one singular point, which is an \(A_2\) singularity. These two \(A_2\) singularities \(([1:0],[1:0:0])\) and \(([0:1],[0:0:1])\) are the only singularities of \(R\). \item If \(R\in\Gtwo\setminus[ \Rone ]\) and \(b_{11}= 1\), then \(R \cong \Rtwo \in |\cal O(1,1)| + |\cal O(1,1)|\) and \(\Delta\) is a double conic.
    \item If \(R\in\Gthree\setminus[ \Rone ]\) and \(R\neq \Rthree\), then \(\Delta\) is a cat-eye with two \(A_3\) singularities at \([1:0:0]\) and \([0:0:1]\). The fibers of \(\piR_2\) over these two points are finite, and they are contained in non-reduced fibers of \(\piR_1\). The only singularities of \(R\) are the two \(A_3\) singularities \(([1:0],[1:0:0])\) and \(([0:1],[0:0:1])\).
    \item \(\Rthree\in\Gthree\) is non-normal, and the associated quartic is a conic + a double line.
\end{enumerate}
\end{remark}

We aim to show that the \((2,2)\)-surfaces in the sets \(\Gfour,\Gtwo,\Gthree\) defined in Proposition~\ref{prop:representatives-degen-Gamma_i} are GIT polystable, and hence that they fully describe the strictly polystable locus in \(\GITR\). We will do this in a later section (Section~\ref{sec:GIT-strictly-polystable}) by showing K-polystability for these surfaces.

\subsection{\texorpdfstring{\(A_n\)}{An} singularities on non-reduced fibers of \texorpdfstring{\(\piR_1\)}{pi1} or non-finite fibers of \texorpdfstring{\(\piR_2\)}{pi2}}\label{sec:A_n-on-bad-fibers-not-GIT-stable}
In this section, we prove that \eqref{item:stable-members-thm-GIT}\(\Rightarrow\)\eqref{item:stable-members-thm-A_n} in Theorem~\ref{thm:stable-members}. The full proof of Theorem~\ref{thm:stable-members} will appear later, in Section~\ref{sec:proofs-of-main-thms}.
From the earlier results in this section (Lemmas~\ref{lem:non-A_n-not-GIT-stable}, \ref{lem:worse-than-A_n}, and~\ref{lem:reducible-not-GIT-stable}), it will suffice to show that if \(R\) has \(A_n\) singularities not satisfying the conditions in Theorem~\ref{thm:stable-members}\eqref{item:stable-members-thm-A_n}, then \(R\) is not GIT stable.
\begin{proposition}\label{prop:A_n-sings-not-stable}
    Let \(R\) be a \((2,2)\)-surface in \(\bb P^1\times\bb P^2\), and let \(\pt\in R\) be an \(A_n\)-singularity.
    \begin{enumerate}[label=(\roman*)]
        \item\label{item:A_1-not-stable} Assume \(n=1\). If the fiber of \(\piR_2\colon R\to\bb P^2\) over \(\piR_2(\pt)\) is non-finite and contains another singular point $q$ of \(R\), then this fiber contains exactly two singular points \(\pt, q\) of \(R\), and \(R\) is not GIT stable. Furthermore, \(R\) is either GIT unstable or is S-equivalent to a member of the set \(\Gfour\). If $R$ is GIT unstable, \(q\) is not an \(A_1\) singularity.
        \item\label{item:higher-A_n-not-stable-non-finite} If \(n\geq 2\) and the fiber of \(\piR_2\colon R\to\bb P^2\) over \(\piR_2(\pt)\) is non-finite, then \(R\) is not GIT stable. Moreover, \(R\) is either GIT unstable or is S-equivalent to a member of the set \(\Gtwo\). If $R$ is GIT unstable, then either: \begin{enumerate}
            \item\label{item:higher-A_n-not-stable-non-finite-a} $n \geq 3$,
            \item\label{item:higher-A_n-not-stable-non-finite-b} The same fiber $\piR_2^{-1}(\piR_2(p))$ contains another singular point $q \neq p$ of $R$ that is not an $A_2$ singularity, or
            \item\label{item:higher-A_n-not-stable-non-finite-c} The fiber of \(\piR_1\colon R\to\bb P^1\) containing \(\pt\) is non-reduced, and this non-reduced fiber of \(\piR_1\) contains another singular point $q \neq p$ of $R$ that is not an $A_n$ singularity for $n \leq 4$.
        \end{enumerate}
        \item\label{item:higher-A_n-not-stable-non-reduced} If \(n\geq 2\) and the fiber of \(\piR_1\colon R\to\bb P^1\) over \(\piR_1(\pt)\) is non-reduced, then \(n\geq 3\) and \(R\) is not GIT stable. Moreover, \(R\) is either GIT unstable or is S-equivalent to a member of the set \(\Gthree\). If $R$ is GIT unstable, then either: \begin{enumerate}
            \item\label{item:higher-A_n-not-stable-non-reduced-a} $n \geq 5$, or
            \item\label{item:higher-A_n-not-stable-non-reduced-b} The non-reduced fiber of \(\piR_1\) containing \(\pt\)
            contains another singular point $q \neq \pt$ of $R$ that is not an $A_n$ singularity for $n \leq 4$.
        \end{enumerate}
    \end{enumerate}
\end{proposition}

\begin{proof}
     After a coordinate change, we may assume $\pt = ([1:0],[1:0:0])$. In each of the above cases, we will show (after potentially ruling out certain GIT unstable scenarios) that \(R\) falls into one of the cases~\eqref{case:becomes-Gamma_2}, \eqref{case:becomes-Gamma_3}, or~\eqref{case:A_1-sing-becomes-Gamma_4} defined at the beginning of the proof of Proposition~\ref{prop:GIT-strictly-semistable}. Then the explicit 1-PS's in the proof of Proposition~\ref{prop:GIT-strictly-semistable} show that \(R\) is not GIT stable.

     \(R\) is defined in $\bb P^1_{[t_0:t_1]} \times \bb P^2_{[y_0:y_1:y_2]}$ by $f(t_0,t_1,y_0,y_1,y_2) = t_0^2Q_1 + 2 t_0t_1Q_2 + t_1^2Q_3$, where
     \begin{gather*}
    Q_1 = a_{11}y_1^2+a_{12}y_1y_2+a_{22}y_2^2 , \qquad
    2 Q_2 = b_{01}y_0y_1+b_{11}y_1^2+b_{12}y_1y_2+b_{22}y_2^2 +b_{02}y_0y_2 , \\
    Q_3 = c_{00}y_0^2+c_{01}y_0y_1+c_{11}y_1^2+c_{12}y_1y_2 + c_{22}y_2^2 + c_{02}y_0y_2
    \end{gather*}
    for some $a_{ij},b_{ij},c_{ij} \in \bb C$.
    
\begin{enumerate}[label=(\roman*)]
\item
For~\ref{item:A_1-not-stable}, after a coordinate change, we may assume the fiber contains an additional singular point \(q = ([0:1],[1:0:0])\). Then \(R\) has defining equation \[t_0^2(a_{11}y_1^2 + a_{12}y_1y_2 + a_{22} y_2^2) + t_0t_1 (b_{01}y_0y_1 + b_{11}y_1^2 + b_{12}y_1y_2 + b_{22}y_2^2 + b_{02}y_0y_2) + t_1^2(c_{11} y_1^2 + c_{12}y_1y_2 + c_{22} y_2^2).\] Since the Hessian at \(\pt\) has determinant \(-2a_{11}b_{02}^2\) (c.f. \eqref{EQ:eqns-for-hessian}), so the assumption that \(\pt\) is an \(A_1\) singularity ensures \(a_{11},b_{02}\neq 0\). By making the coordinate change \(y_2 \mapsto y_2 - \frac{b_{01}}{b_{02}}y_1\), we may assume \(b_{01}=0\), so $R$ has an equation of the form~\eqref{case:A_1-sing-becomes-Gamma_4}. A computation shows that \(q\) is the only other singular point on this fiber. 
    \begin{detail}Indeed, on the $t_1 \neq 0$ and $y_0 \neq 0$ chart, we have \[
        \left.\left(\frac{\partial f(t_0,1,1,y_1,y_2)}{\partial t_0},\frac{\partial f(t_0,1,1,y_1,y_2)}{\partial y_1},\frac{\partial f(t_0,1,1,y_1,y_2)}{\partial y_2}\right)\right|_{(y_1,y_2)=(0,0)} = (0,0,t_0b_{02})
    \]
    which equals $(0,0,0)$ if and only if $t_0 = 0$.
    \end{detail}
Proposition~\ref{prop:GIT-strictly-semistable}\ref{case:Gamma4} shows that \(R\) is either GIT unstable or S-equivalent to a member of \(\Gfour\), and that \(R\) is GIT unstable if and only if \(4 a_{11} c_{11} - b_{11}^2 = b_{02}b_{11} = 0\), i.e. if and only if $b_{11} = c_{11} = 0$. The Hessian at \(q\) has determinant \(-2b_{02}^2 c_{11}\), so in the GIT unstable case, $q$ cannot be an $A_1$ singularity.

    \item For~\ref{item:higher-A_n-not-stable-non-finite}, since the fiber of $\piR_2$ over $\piR_2(\pt)$ is non-finite, $c_{00} = 0$. If \(b_{01}=b_{02}=0\), then \(\diag(-3,3) \times \diag(-10,5,5)\) shows \(R\) is GIT unstable. So we may assume \(b_{01}\) or \(b_{02}\) is nonzero; making the coordinate change \(y_1\leftrightarrow y_2\) if necessary, we may assume \(b_{02}\neq 0\). Then, making the further
    coordinate change $y_2 \mapsto y_2 - \frac{b_{01}}{b_{02}}y_1$, we assume $b_{01} = 0$. The Hessian~\eqref{EQ:eqns-for-hessian} at \(\pt\) has determinant \(-2a_{11}b_{02}^2\), so the assumption that $\pt$ is an $A_n$ singularity of $R$ with $n \geq 2$ implies \(a_{11}=0\). Thus, $R$ has equation of the form~\eqref{case:becomes-Gamma_2}. In the latter case, Proposition~\ref{prop:GIT-strictly-semistable}\ref{case:Gamma2} shows that \(R\) is not GIT stable, and is either GIT unstable (if \(a_{12}=b_{11}=0\) or \(b_{11}=c_{01}=0\)) or is $S$-equivalent to a member of the set of $\Gtwo$. 

In what follows, we investigate the singularities of $R$ in the GIT unstable cases.

    \underline{Case~\ref{item:higher-A_n-not-stable-non-finite-a}:}
    If $b_{02} = 0$, a computation shows that blowing up $R$ at $\pt$ does not resolve the singularity at $\pt$, so $\pt$ cannot be an $A_2$ singularity of $R$.
\begin{detail}
        Indeed, the equation of $R$ in the $t_0 \neq 0 \neq y_0$ chart is \[
            (a_{11}y_1^2+a_{12}y_1y_2+a_{22}y_2^2)+t_1(b_{11}y_1^2+b_{12}y_1y_2+b_{22}y_2^2) + t_1^2(c_{01}y_1+c_{02}y_2+c_{11}y_1^2+c_{12}y_1y_2+c_{22}y_2^2).
        \]
        On the $t_1'$-chart of the blow-up of $\pt$, we have $y_1 \mapsto y_1't_1$, $y_2 \mapsto y_2't_1$, and $t_1=0$ is the exceptional divisor. The strict transform of $R$ on the $t_1'$-chart of the blow-up has an equation of the form \[
            (a_{11}y_1'^2+a_{12}y_1'y_2'+a_{22}y_2'^2) + t_1g(t_1,y_1',y_2')
        \]
        where $g(0,0,0) = 0$, so it is singular at $(t_1,y_1',y_2') = (0,0,0)$.
\end{detail}
    
    \underline{Case~\ref{item:higher-A_n-not-stable-non-finite-b}:}  If $b_{02} \neq 0$ and \(a_{11}=b_{11}=c_{01}=0\), a computation shows that $R$ has another singular point at $([-c_{02}:1],[1:0:0])$ in the non-finite fiber $\piR_2^{-1}(\piR_2(\pt))$, and blowing that singular point up resolves the singularity if and only if $c_{11} \neq 0$, i.e. if $c_{11} = 0$, it cannot be an $A_1$ or \(A_2\) singularity of $R$. On the other hand, if $c_{11} \neq 0$, the singularity is an $A_1$ singularity of $R$.
\begin{detail}
        Indeed, after a coordinate change, we may assume $b_{02} = 1$. Then the equation of $R$ in the $y_0 \neq 0$ chart is \[
        t_0^2(a_{12}y_1y_2+a_{22}y_2^2) + t_0t_1(b_{22}y_2^2+y_2+b_{12}y_1y_2)+t_1^2(c_{11}y_1^2+c_{22}y_2^2+c_{02}y_2+c_{12}y_1y_2).
    \]
    By making the coordinate change $t_0 \mapsto t_0-c_{02}t_1$, we may assume $c_{02} = 0$. Then (replacing \(b_{12}\), \(b_{22}\), \(c_{11}\), \(c_{12}\), \(c_{22}\) suitably) the equation of $R$ in the $t_1 \neq 0 \neq y_1$ chart is \[
        t_0^2(a_{12}y_1y_2+a_{22}y_2^2) + t_0(b_{22}y_2^2+y_2+b_{12}y_1y_2)+(c_{11}y_1^2+c_{22}y_2^2+c_{12}y_1y_2)
    \] 
    which by inspection is singular at $(t_0,y_1,y_2) = (0,0,0)$. Call this point \(q\). The Hessian at \(q\) has determinant \(-2c_{11}\), so
    \(q\) is an $A_1$ singularity if $c_{11} \neq 0$. Henceforth, assume $c_{11} = 0$. On the $y_1'$-chart of the blow-up of \(q\), we have $t_0 \mapsto t_0'y_1$, $y_2 \mapsto y_2'y_1$, and $y_1 = 0$ is the exceptional divisor. The strict transform of $R$ on the $y_1'$-chart of the blow-up has the equation \[
        t_0'y_2'+c_{22}y_2'^2+c_{12}y_2'+ y_1t_0'(b_{22}y_2'^2+b_{12}y_2') + y_1^2 t_0'^2(a_{12}y_2'+a_{22}y_2'^2)
    \]
    and its singular locus along the exceptional divisor $y_1 = 0$ is given by $V(y_1,y_2',t_0'+2c_{22}y_2'+c_{12},t_0'(b_{22}y_2'^2+b_{12}y_2'))$, so it is singular at the point $(t_0',y_1,y_2')=(-c_{12},1,0)$, as desired.
\end{detail}

    \underline{Case~\ref{item:higher-A_n-not-stable-non-finite-c}:} If \(a_{11}=a_{12}=b_{11}=0\), then the fiber \(\piR_1^{-1}(\piR_1(\pt)) = \piR_1^{-1}([1:0])\) is non-reduced. A computation\footnote{The computation works regardless of whether $b_{02} = 0$ or not.} shows that $R$ has another singular point $q\coloneqq ([1:0],[0:1:0])$ (the only point in $\piR_2^{-1}([0:1:0])$), and two point blow-ups cannot resolve the singularity, so it cannot be an $A_n$ singularity of $R$ for $n \leq 4$.
\begin{detail}
    Indeed, the equation of $R$ in the $t_0 \neq 0 \neq y_1$ chart is \[
        a_{22}y_2^2 + t_1(b_{22}y_2^2+b_{02}y_0y_2+b_{12}y_2)+t_1^2(c_{11}+c_{22}y_2^2+c_{01}y_0+c_{02}y_0y_2+c_{12}y_2)
    \]
    which by inspection is singular at $(t_0,y_0,y_2) = (0,0,0)$. On the $y_0'$-chart of the blow-up of that singular point, we have $t_0 \mapsto t_0'y_0$, $y_2 \mapsto y_2'y_0$, and $y_0=0$ is the exceptional divisor. The strict transform of $R$ on the $y_0'$-chart of the blow-up has an equation of the form \[
        (a_{22}y_2'^2+b_{12}t_1'y_2'+c_{11}t_1'^2) + y_0g(t_0',y_0,y_2')
    \] 
    where $g(0,0,0) = 0$, so it is singular at $(t_0',y_0,y_2') = (0,0,0)$. Next, on the $y_0'$-chart of the blow-up of that second singular point, we have $t_1 \mapsto t_1'y_0$ and $y_2 \mapsto y_2'y_0$, and $y_0=0$ is the exceptional divisor. The strict transform of $R'$ on the $y_0'$-chart of the second blow-up has an equation of the form\footnote{Use the computation that $g = t_1(b_{22}y_2^2 + b_{02}y_2)+t_1^2(c_{01}+c_{12}y_2)+y_0t_1^2(c_{22}y_2^2+c_{02}y_2)$.} \[
        (a_{22}y_2''^2+b_{12}t_1''y_2''+c_{11}t_1''^2) + y_0h(t_0'',y_0,y_2'')
    \]
    where $h(0,0,0) = 0$, so it is singular at $(t_0'',y_0,y_2'') = (0,0,0)$, as desired.
\end{detail}

    \item
    Finally, for~\ref{item:higher-A_n-not-stable-non-reduced}, the hypothesis implies $Q_1 = (\sqrt{a_{11}}y_1\pm\sqrt{a_{22}}y_2)^2$. By making a coordinate change and replacing the $b_{ij}$ and $c_{ij}$ accordingly, we may assume $Q_1 = y_2^2$, i.e. assume $a_{22} = 1$ and $a_{11} = a_{12} = 0$ in the equation $R=(f=0)$. Then Hessian of $R$ at $\pt$~\eqref{EQ:eqns-for-hessian} has determinant $-2b_{01}^2$, so the assumption that \(\pt\) is not an \(A_1\) singularity implies \(b_{01}=0\).
    
    First, notice that a single blow-up of \(R\) at \(\pt\) does not resolve the singularity, so we must have $n \geq 3$.
\begin{detail}
    Indeed, on the $t_0 \neq 0 \neq y_0$ chart, $R$ has the equation \[ y_2^2+t_1(b_{11}y_1^2+b_{22}y_2^2+b_{02}y_2+b_{12}y_1y_2)+t_1^2(c_{00}+c_{11}y_1^2+c_{22}y_2^2+c_{01}y_1+c_{02}y_2+c_{12}y_1y_2).
    \] 
    On the $y_1'$-chart of the blow-up of the point $\pt = V(t_0,y_1,y_2)$, we have $t_0 \mapsto t_0'y_1$, $y_2 \mapsto y_2'y_1$, and $y_1 = 0$ is the exceptional divisor. The strict transform of $R$ on the $y_1'$-chart of the blow-up has an equation of the form \[
        (y_2'^2+b_{02}t_1'y_2'+c_{00}t_1'^2) + y_1g(t_0',y_1,y_2')
    \]
    where $g(0,0,0) = 0$, so it is singular at $(t_0',y_1,y_2') = (0,0,0)$, as desired.
\end{detail}
Now we consider GIT stability.

    Now the equation of \(R\) has the form in~Proposition~\ref{prop:GIT-strictly-semistable}\ref{case:Gamma3}, so \(R\) is either S-equivalent to a member of \(\Gthree\) or is GIT unstable (which happens if and only if \(b_{11}=0\) or \(b_{02}=c_{00}=0\)).
    In what follows, we investigate the singularities of $R$ in the GIT unstable cases.

    \underline{Case~\ref{item:higher-A_n-not-stable-non-reduced-a}:} If $b_{02}=c_{00}=0$, then a computation shows the following statements: \begin{enumerate}
        \item\label{item:higher-A_n-not-stable-non-reduced-a-at-least-one-b11-c01-zero} At least one of $b_{11}$ or $c_{01}$ must be \(0\). Otherwise, the blow-up of $R$ at $\pt$ will have more than one singular points along the exceptional divisor, contradicting the assumption that $\pt$ is $A_n$ singularity of $R$.
        \item\label{item:higher-A_n-not-stable-non-reduced-a-c01-zero} If $c_{01} = 0$, then two point blow-ups cannot resolve the singularity at $\pt$, so $\pt$ cannot be an $A_n$ singularity of $R$ for $n \leq 4$.
        \item\label{item:higher-A_n-not-stable-non-reduced-a-b11-zero} If $b_{11} = 0$, then we are in Case~\ref{item:higher-A_n-not-stable-non-finite-c}, so the non-reduced fiber of \(\piR_1\) containing \(\pt\) contains another singular point \(q\neq\pt\)
        that is not an $A_n$ singularity for $n \leq 4$.
    \end{enumerate}
\begin{detail}
    Indeed, the equation of $R$ on the $t_0 \neq 0 \neq y_0$ chart is \[
        y_2^2+t_1(b_{11}y_1^2+b_{22}y_2^2+b_{12}y_1y_2)+t_0t_1^2(c_{11}y_1^2+c_{22}y_2^2+c_{01}y_1+c_{02}y_2+c_{12}y_1y_2).
    \]
    By~\ref{item:higher-A_n-not-stable-non-reduced-a-b11-zero}, we may assume $b_{11} \neq 0$. On the $t_1'$-chart of the blow-up of $p$, we have $y_1 \mapsto y_1't_1$, $y_2 \mapsto y_2't_1$, and $t_1 = 0$ is the exceptional divisor. The strict transform $R'$ of $R$ on the $t_1'$-chart of the blow-up has the equation \[
        y_2'^2 + t_1(b_{11}y_1'^2+b_{12}y_1'y_2'+b_{22}y_2'^2+c_{01}y_1'+c_{02}y_2') + t_1^2(c_{11}y_1'^2+c_{12}y_1'y_2'+c_{22}y_2'^2)
    \]
    and its singular locus along the exceptional divisor $t_1 = 0$ is given by $V(t_1,y_2',b_{11}y_1'^2+b_{12}y_1'y_2'+b_{22}y_2'^2+c_{01}y_1'+c_{02}y_2') = V(t_1,y_2',y_1'(b_{11}y_1'+c_{01}))$. Therefore, $R'$ has the singular points $(t_1,y_1',y_2') = (0,0,0)$ and $(0,-\tfrac{c_{01}}{b_{11}},0)$, which are distinct if and only if $c_{01} \neq 0$. This shows~\ref{item:higher-A_n-not-stable-non-reduced-a-at-least-one-b11-c01-zero}. Finally, consider the $y_1''$-chart of the blow-up of the singular point $(t_1,y_1',y_2') = (0,0,0)$, on which we have $t_1 \mapsto t_1''y_1'$, $y_2' \mapsto y_2''y_1'$, and $y_1'=0$ is the exceptional divisor. Then the strict transform $R''$ of $R'$ on the $y_1''$-chart of this second blow-up has an equation of the form \[
        (c_{01}y_1'' + y_2''^2+c_{02}t_1''y_2'')+y_1'g(t_1'',y_1',y_2'') 
    \]
    where $g(0,0,0) = 0$. If \(c_{01}=0\), then $R'$ is singular at $(t_1'',y_1',y_2'') = (0,0,0)$, showing ~\ref{item:higher-A_n-not-stable-non-reduced-a-c01-zero}.
\end{detail}

    \underline{Case~\ref{item:higher-A_n-not-stable-non-reduced-b}:} If \(b_{11}=0\), a computation shows that $R$ has another singular point $([1:0],[0:1:0])$ in the finite fiber $\piR_2^{-1}([0:1:0])$, and two point blow-ups cannot resolve the singularity, so it cannot be an $A_n$ singularity of $R$ for $n \leq 4$.
\begin{detail}
    Indeed, on the $t_0 \neq 0 \neq y_1$ chart, $R$ has the equation \[
        y_2^2 + t_1(b_{22}y_2^2+b_{02}y_0y_2+b_{12}y_2) + t_1^2(c_{00}y_0^2+c_{11}+c_{22}y_2^2+c_{01}y_0+c_{02}y_0y_2+c_{12}y_2)
    \]
    which by inspection is singular at $(t_1,y_0,y_2)=(0,0,0)$. On the $y_0'$-chart of the blow-up of that singular point, we have $t_1 \mapsto t_1'y_0$, $y_2 = y_2'y_0$, and $y_0=0$ is the exceptional divisor. The strict transform $R'$ of $R$ on the $y_0'$-chart of the blow-up has an equation of the form \[
        (y_2'^2+b_{12}t_1'y_2'+c_{11}t_1'^2) + y_0g(t_1',y_0,y_2')
    \]
    where $g(0,0,0) = 0$, so it is singular at $(t_1',y_0,y_2') = (0,0,0)$. Next, on the $y_0'$-chart of the blow-up of that second singular point, we have $t_1' \mapsto t_1''y_0$, $y_2' \mapsto y_2''y_0$, and $y_0 = 0$ is the exceptional divisor. The strict transform $R''$ of $R$ on the $y_0'$-chart of the second blow-up has an equation of the form\footnote{Use the computation that $g = t_1'(b_{22}y_2'^2+b_{02}y_2')+t_1'^2(c_{01}+c_{12}y_2')) + y_0t_1'^2(c_{00}+c_{22}y_2'^2+c_{02}y_2')$.} \[
        (y_2''^2+b_{12}t_1''y_2''+c_{11}t_1''^2) + y_0h(t_1'',y_0,y_2'')
    \]
    where $h(0,0,0) = 0$, so it is singular at $(t_1'',y_0,y_2'') = (0,0,0)$, as desired.
\end{detail}
\end{enumerate}
\end{proof}

\section{K-stability of \texorpdfstring{\((2,2)\)}{(2,2)}-surfaces with \texorpdfstring{\(A_n\)}{An} singularities}\label{sec:stability-of-A_n}
In this section, we establish the K-stability of \((2,2)\)-surfaces with certain (isolated) \(A_n\) singularities. More precisely, we prove that \eqref{item:stable-members-thm-A_n}\(\Rightarrow\)\eqref{item:stable-members-thm-K} in Theorem~\ref{thm:stable-members}.

The main result in this section is:
\begin{theorem}\label{thm:local-delta-A_n-singularities}
    Let \(R\) be a \((2,2)\)-surface in \(\bb P^1\times\bb P^2\), and let \(\pt\in R\) be an isolated singularity. Let \(\cone\approx 0.3293\) be the smallest root of the cubic polynomial \(7 c^3 - 23 c^2 + 22 c - 5\). Consider the log Fano pair \((\exx,\dee) \coloneqq (\bb P^1\times\bb P^2, cR)\) for \(c\in(0,1)\).
    \begin{enumerate}
        \item\label{part:local-delta-A_n-singularities:A_1-finite-reduced} If \(\pt\) is an \(A_1\) singularity, the fiber of \(\piR_2\colon R\to\bb P^2\) over \(\piR_2(\pt)\) is finite, and the fiber of \(\piR_1\colon R\to\bb P^1\) over \(\piR_1(\pt)\) is reduced, then \(\delta_{\pt}(\exx,\dee)>1\) for all \(c\in(0,1)\).
        \item\label{part:local-delta-A_n-singularities:A_1-finite-nonreduced} If \(\pt\) is an \(A_1\) singularity, the fiber of \(\piR_2\) over \(\piR_2(\pt)\) is finite, and the fiber of \(\piR_1\) over \(\piR_1(\pt)\) is nonreduced, then \(\delta_{\pt}(\exx,\dee)>1\) for all \(c\in(0,\tfrac{1}{2}]\).
        \item\label{part:local-delta-A_n-singularities:A_1-infinite} If \(\pt\) is an \(A_1\) singularity and is the only singular point contained in the fiber of \(\piR_2\) over \(\piR_2(\pt)\), then \(\delta_{\pt}(\exx,\dee)>1\) for all \(c\in(\cone,\tfrac{1}{2}]\) and \(\delta_{\pt}(\exx,\dee)\geq 1\) for all \(c\in(0,\tfrac{1}{2}]\).
        
        Moreover, if \(\delta_q(\exx, \dee)\geq 1\) for all \(q\in R\) and \(c\in (0,\tfrac{1}{2}]\cap\bb Q\) and if there exists \(c'\in(\cone,\tfrac{1}{2}]\cap\bb Q\) such that \(\delta_q(\exx,c'R) > 1\) for all \(q\in R\), then \((\exx,\dee)\) is K-stable for all \(c\in(0,\tfrac{1}{2}]\cap\bb Q\).
        \item\label{part:local-delta-A_n-singularities:A_2} If \(\pt\) is an \(A_2\) singularity and the fiber of \(\piR_2\) over \(\piR_2(\pt)\) is finite, then \(\delta_{\pt}(\exx,\dee)>1\) for all \(c\in (0,\tfrac{1}{2}]\).
        \item\label{part:local-delta-A_n-singularities:higher-A_n} If \(\pt\) is an \(A_n\) singularity with \(n\geq 3\), the fiber of \(\piR_2\) over \(\piR_2(\pt)\) is finite, and the fiber of \(\piR_1\) over \(\piR_1(\pt)\) is reduced, then \(\delta_{\pt}(\exx,\dee)>1\) for all \(c\in (0, 0.7055]\).
        \item\label{part:local-delta-A_n-singularities:A_1-infinite-twosingularities} If \(\piR_2\) has a non-finite fiber containing exactly two \(A_1\) singularities, then  \(\delta_{\pt'}(\exx,\dee)=1\) for all \(c\in(0,\tfrac{3}{4})\) and for any point \(\pt'\) in this fiber.
    \end{enumerate}
\end{theorem}

\begin{remark}
    The assumptions about the finite (resp. reduced) fibers of \(\piR_2\) (resp. \(\piR_1\)) in Theorem~\ref{thm:local-delta-A_n-singularities} are necessary: as we saw in Proposition~\ref{prop:A_n-sings-not-stable}, surfaces with \(A_n\) singularities not satisfying these conditions are not GIT stable. In fact, the surfaces in part~\eqref{part:local-delta-A_n-singularities:A_1-infinite-twosingularities} with two \(A_1\) singularities on a non-finite fiber of \(\piR_2\) are always GIT strictly semistable by Proposition~\ref{prop:A_n-sings-not-stable}\ref{item:higher-A_n-not-stable-non-finite} (ultimately, by Theorem~\ref{thm:Kpolystable2,2div}, we will see that \((\bb P^1\times\bb P^2, cR)\) is in fact strictly K-semistable for \(c<\tfrac{3}{4}\)). We also note that if \(\pt\in R\) is an \(A_2\) singularity as in part~\eqref{part:local-delta-A_n-singularities:A_2}, then the fiber of \(\piR_1\) containing \(\pt\) is necessarily reduced by Proposition~\ref{prop:A_n-sings-not-stable}\ref{item:higher-A_n-not-stable-non-reduced}.
\end{remark}

For smooth points, lower bounds on \(\delta\) were shown in \cite{CFKP23}:
\begin{theorem}[{\cite[Lemmas 2.2, 2.3, and 2.4]{CFKP23}}]
    Let \(R\) be a \((2,2)\)-surface in \(\bb P^1\times\bb P^2\). If \(\pt\not\in R\), or if \(\pt\in R\) is a smooth point in a finite fiber or a smooth fiber of \(\piR_2\colon R \to \bb P^2\), then \(\delta_{\pt}(\bb P^1\times\bb P^2, cR)>1\) for all \(c\in(0,1)\).
\end{theorem}

The main technique we use in this section is the Abban--Zhuang method \cite{Abban-Zhuang-flags} (see Section~\ref{sec:Abban-Zhuang}), which uses admissible flags to give lower bounds on local \(\delta\)-invariants. We will use refinements of their theory by Fujita \cite{Fujita-3.11} for 3-dimensional Mori fiber spaces (see Section~\ref{sec:fujita-abban-zhuang}). Later, in Section~\ref{sec:K-stability-Xthree}, we will again do similar computations for divisors on a different threefold.

\begin{proof}[Proof of Theorem~\ref{thm:local-delta-A_n-singularities}]
    We show parts~\eqref{part:local-delta-A_n-singularities:A_1-finite-reduced} and~\eqref{part:local-delta-A_n-singularities:A_1-finite-nonreduced} in Section~\ref{sec:local-delta-A_n:A_1-finite}, and parts~\eqref{part:local-delta-A_n-singularities:A_1-infinite} and~\eqref{part:local-delta-A_n-singularities:A_1-infinite-twosingularities} in Section~\ref{sec:local-delta-A_n:A_1-infinite}. Part~\eqref{part:local-delta-A_n-singularities:A_2} is shown for \(0 < c < \cnot \approx 0.47\) in Section~\ref{sec:local-delta-A_n:higher-part-1} and for \(0.33 \approx \cone < c \leq \tfrac{1}{2}\) in Section~\ref{sec:local-delta-A_n:A_2}. Finally, we show part~\eqref{part:local-delta-A_n-singularities:higher-A_n} for \(0<c<\cnot\approx 0.47\) in Section~\ref{sec:local-delta-A_n:higher-part-1} and for \(0.39 < c \leq \tfrac{1}{2}\) in Section~\ref{sec:local-delta-A_n:higher-part-2}.
\end{proof}

Throughout Section~\ref{sec:stability-of-A_n}, we will use the following notation.
\begin{notation}\label{notation:delta-invariant-computations}
    We write \(\exx=\bb P^1\times\bb P^2\) and \(\dee=cR\). We will let \(L\) denote the pullback of \(-K_\exx-\dee\) on a birational model of \(\exx\); this has \(\vol L = \vol(-K_\exx-\dee) = 3(2-2c)(3-2c)^2\). The point for which we compute the local \(\delta\)-invariant will be denoted \(\pt\in R\). If \(Y\) is a divisor on a birational model \(\tX\to \exx\), then \(R|_Y\) will denote the intersection of \(Y\) with the strict transform of \(R\) on \(\tX\).

    \(S\cong\bb P^1\times\bb P^1\) will be a \((0,1)\)-surface containing \(\pt\); sometimes \(S\) will be chosen generally, and sometimes it will be chosen with additional conditions. \(l_1 \subset S\) will denote the fiber of the second projection \(\piX_2\colon X\to\bb P^2\) through \(\pt\) (or a strict transform of this line). That is, \(\pt \in l_1 \in |\cal O_S(0,1)|\),
    
    \(F\cong\bb P^2\) will be a fiber of the first projection \(\piX_1\colon \exx\to\bb P^1\) containing \(\pt\), and \(l_2\subset F\) will denote a line in \(F\) through \(\pt\) (or a strict transform of such a line). That is, \(\pt \in l_2 \in |\cal O_S(1,0)|\).

    To compute \(S_{\exx,\dee}(Y)\), we will find Nakayama--Zariski decompositions \(L-uY=P(u)+N(u)\) (see \cite[Section 4.1]{Fujita-3.11}) and use that \(\vol(L-uY)=\vol(P(u))=P(u)^3\). We will use \(P\) (resp. \(N\)) to denote the positive (resp. negative) part of a (Nakayama--)Zariski decomposition, and we will use subscripts \(\regionA,\regionB,\regionC,\regionD,\regionE\) on \(P(u)\) and \(N(u)\) to denote the decomposition for values of \(u\) in certain regions.

\subsection{Fujita's refinements of the Abban--Zhuang method}\label{sec:fujita-abban-zhuang}
    Throughout, we will use Abban--Zhuang's method of admissible flags \cite{Abban-Zhuang-flags} and refinements by Fujita \cite{Fujita-3.11} to give lower bounds for local \(\delta\)-invariants. To introduce this, we first fix:
    
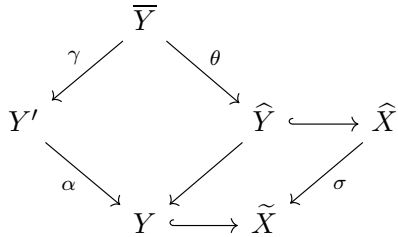
\[\begin{tikzcd}
    & \oY \ar[ld, "\gamma" {swap}] \ar[rd, "\theta"] \\
    Y' \ar[rd, "\alpha" {swap}] & & \hY \ar[r, hook] \ar[ld] & \hX \ar[ld, "\sigma"] \\
    & Y \ar[r, hook] & \tX \end{tikzcd}\]
    with
    \begin{itemize}
        \item \(Y\) a plt-type divisor over \(\exx\) whose center on \(\exx\) contains \(\pt\), \(\phi\colon \tX \to \exx\) a plt blow-up extracting \(Y\), and \(L \coloneqq \phi^*(-K_\exx - \dee)\);
        \item \(\sigma\colon \hX \to \tX\) a common resolution of the models of \(\tX\) computing the Nakayama--Zariski decomposition of \(L-uY\), and \(\hY\) the strict transform of \(Y\) on \(\oX\);
        \item \(Y' \to Y\) a plt blow-up from a normal surface \(Y'\) extracting a plt-type divisor \(C\subset Y'\) such that \(C\) is a smooth projective curve; and
        \item \(\oY\) a common resolution of \(Y'\) and \(\hY\), \(\oC \subset \oY\) the strict transform of \(C\) in \(\oY\), and \(\gamma^* C = \oC + \Sigma\).
    \end{itemize}
    Following \cite[Notation 4.11]{Fujita-3.11}, let \(\sigma^*(L - uY) = P(u)+N(u)\) be the Nakayama--Zariski decomposition, let \(P(u)|_\hY, N(u)|_\hY\) be the restrictions to \( \hY \subset \hX\), and let \(P(u)|_\oY, N(u)|_\oY\) be the pullbacks under \(\theta\colon \oY \to \hY\). For any \(0\leq u \leq \tau^+\), where \(\tau^+\) is the pseudoeffective threshold of \(L-uY\), define \[N'(u)|_\oY \coloneqq N(u)|_\oY - d(u)\oC\] where \(d(u) \coloneqq \ord_\oC (N(u)|_\oY)\). Let \(t(u)\) be the pseudoeffective threshold of \(P(u)|_\oY - v\oC\), and for \(0\leq v\leq t(u)\), let \(P(u)|_\oY - v\oC = P(u,v) + N(u,v)\) be the Zariski decomposition.
    
    Let \(V_{\bullet \bullet}\) be the refinement of the filtration of \(L\) by \(Y\), let \(W_{\bullet \bullet \bullet}\) be the further refinement of \(V_{\bullet \bullet}\) by \(C\), define \(\dee_Y\) by \(K_Y + \dee_Y = (K_\tX + \dee_\tX + Y)|_Y\) (where \(\dee_\tX\) is the strict transform of \(\dee\) on \(\tX\)), and let \(\dee_C\) be the restriction to \(C\) of the strict transform \(\dee_{Y'}\) of \(\dee_Y\).

    By the Abban--Zhuang method \cite{Abban-Zhuang-flags}, \[\delta_{\pt}(\exx,\dee) \geq \min \left\{  \frac{A_{\exx,\dee}(Y)}{S_{\exx,\dee}(Y)}, \inf_{q \in Y} \delta_q(Y, \dee_Y; V_{\bullet \bullet}) \right\},\]
    and, for each closed point \(q\in Y\),
     \[ \delta_q(Y, \dee_Y; V_{\bullet \bullet}) \ge \min \left\{ \frac{A_{Y, \dee_Y}(C)}{S_{Y,\dee_Y; V_{\bullet \bullet}}(C)}, \inf_{
     q' \in C :
     \alpha(q') = q } \frac{A_{C, \dee_C}(q')}{S_{C,\dee_C; W_{\bullet \bullet \bullet})}(q')} \right\}. \]
     Fujita gives formulas for these quantities in terms of Zariski decompositions. Namely,
     \[ S_{Y,\dee_Y; V_{\bullet \bullet}}(C) = \frac{3}{\vol L} \int_0^{\tau^+} \left( (P(u)|_{\oY})^2\cdot\ord_\oC(N(u)|_\oY) + \int_0^\infty \vol( P(u)|_\oY - v\oC) dv \right) du \]
     by \cite[Theorem 4.8]{Fujita-3.11}. For any closed point \(q'\in C\), by \cite[Theorem 4.17]{Fujita-3.11}
    \[ S_{C,\dee_C; W_{\bullet \bullet \bullet}}(q') = \frac{3}{\vol L} \int_0^{\tau^+} \left( \int_{0}^{t(u)} (P(u,v)\cdot \oC)^2 dv \right) du + F_{q'}(W_{\bullet \bullet \bullet}) \]
    where
    \[ F_{q'}(W_{\bullet \bullet \bullet})  = \frac{6}{\vol L} \int_0^{\tau^+} \left( \int_{0}^{t(u)} (P(u,v)\cdot \oC) \cdot \ord_{q'}((N'(u)|_\oY + N(u,v) - (v + d(u))\Sigma)|_\oC) dv \right) du. \]
    
\end{notation}

We will repeatedly make use of the above formulas throughout this section and Section~\ref{sec:K-stability-Xthree}. For the benefit of the reader, we will include the full details of the computation for the first case in Section~\ref{sec:local-delta-A_n:A_1-finite} below. In later computations, we will give the Zariski decompositions \(P(u,v)+N(u,v)\) on the surfaces in the flags and the final values for \(S_{Y,\dee_Y; V_{\bullet \bullet}}(C)\) and \(S_{C,\dee_C; W_{\bullet \bullet \bullet}}(q')\); the interested reader may use the Zariski decompositions and the above formulas to verify these values.

\subsection{\texorpdfstring{\(A_1\)}{A1} singularities on finite fibers of \texorpdfstring{\(\piR_2\)}{piR2}}\label{sec:local-delta-A_n:A_1-finite}
In this section, we show Theorem~\ref{thm:local-delta-A_n-singularities}\eqref{part:local-delta-A_n-singularities:A_1-finite-reduced} and~\eqref{part:local-delta-A_n-singularities:A_1-finite-nonreduced}. We first note that for part~\eqref{part:local-delta-A_n-singularities:A_1-finite-nonreduced}, the assumption that \(c  \leq \tfrac{1}{2}\) is used in Case~\ref{item:A_1-case-2-in-lF-nonreduced} below. For all other cases, the flags we use show that the local \(\delta\)-invariant is \(>1\) for all \(c\in (0,1)\). In particular, if the fiber of \(\piR_1\) at \(\pt\) is reduced, then \(\delta_{\pt}(\exx, \dee)>1\) for all \(c\in(0,1)\).

Let \(F\cong\bb P^2\) be a fiber of the first projection \(\piR_1\colon R\to\bb P^1\) containing \(\pt\), and let \(S\cong\bb P^1\times\bb P^1\) be a general \((0,1)\)-surface containing \(\pt\). Let \(\phi\colon\tX \to \exx\) be the blow-up of \(\pt\) with exceptional divisor \(Y\cong\bb P^2\), let \(\dee_{\tX}\) be the strict transform of \(\dee\), let \(L=\phi^*(-K_{\exx}-\dee)\), and let \(\dee_Y\) be defined by \[K_Y+\dee_Y = (K_{\tX}+\dee_{\tX}+Y)|_Y.\] By the Abban--Zhuang method, if \(V_{\bullet\bullet}\) is the refinement of the filtration of \(L\) by \(Y\), then \[\delta_{\pt}(\exx,\dee) \geq \min \left\{  \frac{A_{\exx,\dee}(Y)}{S_{\exx,\dee}(Y)}, \inf_{q \in Y} \delta_q(Y, \dee_Y; V_{\bullet \bullet}) \right\}.\]

\subsubsection{Computation of \(A_{\exx,\dee}(Y)/S_{\exx,\dee}(Y)\)}\label{sec:local-delta-A_n:A_1-finite-part-1}

First, \(A_{\exx,\dee}(Y) = 3 - 2c\). To compute \(S_{\exx,\dee} (Y)\), we need a Nakayama--Zariski decomposition of \[L-uY = \phi^*(-K_X - D) - uY = -K_\tX - D_\tX +(2-2c-u)Y \] and its pseudoeffective threshold.  Because $L = \phi^*( (2-2c)F + (3-2c)S) = (2-2c)F_\tX + (3-2c) S_\tX + (5-4c) Y$, the pseudoeffective threshold is $u = 5-4c$.

By Corollary~\ref{cor:generatorsofmoricone}, the Mori cone of $\tX$ is generated by $l_1, l_2, l_3$ and possibly additional curves in $F_\tX$ and $S_\tX$, where $l_1$ is the strict transform of the ruling of $S$ through $p$ (the fiber of the second projection through $p$), $l_2$ is the strict transform of a line in $F$ through $p$,  and $l_3$ is a line in the exceptional divisor $Y$.  Every curve in $F_\tX \cong \bF_1$ and $S_\tX \cong \Bl_p (\bP^1 \times \bP^1)$ is a convex combination of $l_1, l_2, l_3$, so we conclude that the Mori cone is generated by $l_1, l_2, l_3$. Then, we compute: 
\[(L - uY) \cdot l_1 = 2-2c-u, \qquad
(L-uY) \cdot l_2 = 3-2c-u, \qquad
(L-uY) \cdot l_3 = u.\]
and therefore $L-uY$ is ample for $u \le 2 - 2c$.  

To compute the Nakayama--Zariski decomposition of $L-uY$ for $ 0 \le u \le 5-4c$, we compute the ample models of the divisor $L-uY$ on $\tX$, i.e. the associated birational modifications as $L-uY$ stops being ample.  In this case, we will first perform the Atiyah flop of $l_1$ at $u = 2-2c$ and then perform the contraction of $l_2$ at $u = 3-2c$.  We verify these details below. 

Because $\cal N_{l_1/\tX} = \cal O(-1) \oplus O(-1)$, we may perform the Atiyah flop of $l_1$.  Let $\psi\colon \hX \to \tX$ be the blow-up of $l_1$ with exceptional divisor $G \cong \bb P^1 \times \bb P^1$, and let $\psi^+\colon \hX \to \tX^+$ be the contraction of the ruling of $G$ that maps to $l_1$.  

We may then write the Nakayama--Zariski decomposition of $\psi^*(L-uY)$ on $\hX$ as \[  \psi^*(L -uY) = -K_{\hX} - \dee_\hX - (u+2c-2) Y_\hX + G = P(u) + N(u), \]
where $N(u) = 0$ for $u \le 2 - 2c$.  For $u \ge 2-2c$, we study the birational geometry of $\tX^+$.  On $\tX^+$, the Mori cone certainly has as extremal rays the images of $l_1^+ \coloneqq \psi^+(G)$, the image of $l_2$, and the image of the line $l_3$ that met $l_1$ in $\tX$.  Furthermore, for these curves, we have 
    \begin{align*}
    (\psi^*(L-uY) - (u+2c-2)G) \cdot l_1^+ &= u+2c-2, \\
    (\psi^*(L-uY) - (u+2c-2)G) \cdot l_2 &= 3-2c-u, \\
    (\psi^*(L-uY) - (u+2c-2)G) \cdot l_3 &= 2-2c,
    \end{align*}
so $\psi^*(L-uY) - (u+2c-2)G$ is ample on these extremal rays.  Then, it is straightforward to compute that every other curve in $\tX^+$ not contained in the strict transform of $Y, \tF$, or $\tS$ can be written as a nonnegative combination of $l_1^+, l_2, l_3$ (c.f. Lemma~\ref{lem:moricone} for a similar computation).
Therefore, we conclude that the Nakayama--Zariski decomposition for $2 - 2c \le u \le 3-2c$ is given by
 \[ P(u) = -K_{\hX} - \dee_\hX - (u+2c-2) Y_\hX - (u+2c-3) G, \quad N(u) =  (u+2c - 2) G .\]
Note that $P(u) = (\psi^+)^*(L^+ - uY^+) $ where $L^+ = - K_{\tX^+} - D^+ + (2-2c)Y^+$ and $D^+$, $Y^+$ are the strict transforms of $D$, $Y$ on $\tX^+$.  We may compute the remaining pieces of the Nakayama-Zariski decomposition on $\tX^+$.  At $u = 3-2c$, we may perform the contraction of $l_2$, which contracts the surface $F^+ \cong \bF_1$ (the strict transform of $F_{\tX}$) to its negative section.  Let $\pi\colon \tX^+ \to X'$ be the contraction of $l_2$.  As $X'$ has Picard rank 2 and the images of $l_3$ and $l_1^+$ are extremal, they must be the only extremal rays in $\overline{NE(X')}$.  If we denote by $P'(u) = L' - uY'$, where $L' = - K_{X'} - D'+(2-2c)Y'$, and by $D', Y'$ the strict transforms of $D,Y$ on $X'$, we may then compute that $P'(u)$ is positive on $\overline{NE(X')}$ for $u$ up to the pseudoeffective threshold $u = 5 - 4c$.  Therefore, if $N'(u) = (u +3 -2c)F'$, the Nakayama--Zariski decomposition of $L^+ - uY^+$ is given by $\pi^*(P'(u)) + N'(u)$, and therefore we may write the full Nakayama--Zariski decomposition of $\psi^*(L-uY)$ on $\hX$ by pullback.  
 
In summary, we consider the following diagram of varieties: 
\[\begin{tikzcd}
& & &  \hX \arrow[ld, swap, "\psi"] \arrow[rd] & & & \\
\exx = \bb P^1\times\bb P^2 & & \tX \arrow[ll, "\phi" {swap}, "\text{blow up }\pt"] \arrow[rr, dashed, "\text{Atiyah flop of }l_1" {swap}] & & \tX^+ \arrow[rr, "\text{contract }l_2" {swap}] & & X'
\end{tikzcd}\]

\noindent and we write the Nakayama--Zariski decomposition of $\psi^*(L-uY)$ on $\hX$ as \[  \psi^*(L -uY) = -K_{\hX} - \dee_\hX - (u+2c-2) Y_\hX + G = P(u) + N(u). \]  For $0 \le u \le 2-2c$ we have $P(u) = \psi^*(L-uY)$; for $2-2c \le u \le 3-2c$, the positive part is the pullback from $\tX^+$; and finally for $3 - 2c \le u \le 5-4c$, the positive part is the pullback from $X'$. We will use subscripts \(\regionA,\regionB,\regionC\) on \(P(u)\) and \(N(u)\) for values of \(u\) in these three respective regions. We find that:
\[ \resizebox{1\textwidth}{!}{ $ \begin{array}{rll}
    0 \le u \le 2-2c: & P_{\regionA}(u) = -K_{\hX} - \dee_\hX - (u+2c-2) Y_\hX + G & N_{\regionA}(u) = 0 \\
    2-2c \le u \le 3-2c: & P_{\regionB}(u) = -K_{\hX} - \dee_\hX - (u+2c-2) Y_\hX - (u+2c-3) G & N_{\regionB}(u) =  (u+2c - 2) G \\
    3-2c \le u \le 5-4c: & P_{\regionC}(u) = -K_{\hX} - \dee_\hX - (u+2c-2) Y_\hX - (u+2c-3) (G+F_\hX) & N_{\regionC}(u) =(u+2c - 2) G + (u+2c-3)F_\hX.
\end{array} $} \]

To find \(S_{\exx, \dee}(Y)\), we must next compute the volume of \(P(u)\) on each region \(\regionA,\regionB,\regionC\). First, \(P_{\regionA}(u)^3 = (L-uY)^3 = 3(2-2c)(3-2c)^2 - u^3\).
\begin{detail}
Here we use that $L^2 \cdot Y = L \cdot Y^2 = 0$ and $Y^3 = 1$.
\end{detail}
Next, $P_{\regionB}(u) = \psi^*(L - uY) - (u+2c - 2)G$, where $G \cong \bb P^1 \times \bb P^1$, so Lemma~\ref{lem:volume-lemma-4} implies
\(P_{\regionB}(u)^3 = 3 (2-2c)(3-2c)^2 - u^3 + (u+2c-2)^3\).
Lastly, on region \(\regionC\), we have
$P_{\regionC}(u) = P_{\regionB}(u) - (u+2c-3)F_\hX$. Let \(f_1\) and \(\sigma\) denote the fiber and negative section, respectively, of $F_\hX \cong \bb F_1$. We compute that $F_\hX|_{F_\hX} = - \sigma$ and $P_{\regionB}(u)|_{F_\hX} = (3-2c-u)\sigma +(3-2c) f_1$. By Lemma~\ref{lem:volume-lemma-3}, we then find 
 \(P_{\regionC}(u)^3 = 3(2-2c)(3-2c)^2 - u^3 + (u+2c-2)^3 + (3-2c-u)^2(3u+2(3-2c-u))\).
Finally, we integrate the previous expressions to obtain 
\[ S_{\exx, \dee}(Y) = \frac{1}{\vol L} \left( \int_0^{2-2c} P_{\regionA}(u)^3 du + \int_{2-2c}^{3-2c} P_{\regionB}(u)^3 du + \int_{3-2c}^{5-4c} P_{\regionC}(u)^3 du \right) = \frac{9-7c}{3}.\]
Therefore, for all \(c\in(0,1)\), we have
\[ \frac{A_{\exx,\dee} (Y)}{S_{\exx,\dee}(Y)} = \frac{9 - 6 c}{9 - 7 c} > 1.\]

\subsubsection{Computation of \(\delta_q(Y, \dee_Y; V_{\bullet \bullet})\) for \(q\in Y\)}\label{sec:local-delta-A_n:A_1-finite-part-2} As $R$ has an $A_1$ singularity at $\pt$, the restriction $R_Y$ to $Y$ of the strict transform of $R$ is a smooth conic in $Y$. Let \(\tS\) and \(\tF\) denote the strict transforms of \(S\) and \(T\) in \(\tX\), respectively. Let $l_{\tF}$ be the intersection of $Y$ and $\tF$, a line in $Y$ that is tangent to \(R_Y\) if and only if the fiber of \(\piR_1\colon R\to\bb P^1\) over \(\piR_1(\pt)\) is non-reduced.  Let $l_{\tS}$ be the intersection of $Y$ and $\tS$, a line in $Y$. Let $y_1 \in Y$ be the intersection point of $l_1$ and $l_{\tS}$.

We will consider several cases for $q \in Y$, and in each case, we will choose a plt type divisor $C$ over $Y$ with center on $Y$ containing $q$. If $\alpha\colon \oY \to Y$  is the plt blow-up extracting $C$, then
\[ \delta_q(Y, \dee_Y; V_{\bullet \bullet}) \ge \min \left\{ \frac{A_{Y, \dee_Y}(C)}{S_{Y,\dee_Y; V_{\bullet \bullet}}(C)}, \inf_{
     \qhatpt \in C :
     \alpha(\qhatpt) = q } \delta_{\qhatpt}(C,\dee_C; W_{\bullet \bullet \bullet}) \right\} \]
where $W_{\bullet \bullet \bullet}$ is the further refinement of $V_{\bullet \bullet}$ by $C$, and $\dee_C$ is the restriction of $\dee_Y$ to $C$.  
\begin{enumerate}
    \item\label{item:A_1-case-1} Case 1: $q \notin R_Y$.
    In this case, if \(q\neq y_1\), then \(C\) will be the line connecting \(q\) and \(y_1\). If \(q=y_1\), then \(C\) will be the line connecting \(y_1\) to a general point on \(Y\).
    \item\label{item:A_1-case-2} Case 2: $q \in R_Y$, $q \ne y_1$, and the line connecting $q$ to $y_1$ is not tangent to $R_Y$ at $q$.
    \begin{enumerate}
        \item\label{item:A_1-case-2-notin-lF} If \(q \notin l_{\tF}\), then \(C\) will be the line connecting \(q\) and \(y_1\).
        \item\label{item:A_1-case-2-in-lF} If \(q \in l_{\tF}\), then we have two sub-cases:
        \begin{enumerate}
            \item\label{item:A_1-case-2-in-lF-reduced} If the fiber of \(\piR_1\colon R\to\bb P^1\) at \(\pt\) is reduced, then we use the flag coming from the blow-up of \(q\). We take \(C\) to be the exceptional divisor of this blow-up.
            \item\label{item:A_1-case-2-in-lF-nonreduced} If the fiber of \(\piR_1\) at \(\pt\) is non-reduced, then we use the flag coming from the \((2,1)\) weighted blow-up of \(q\) separating \(R_Y\) and \(l_{\tF}\). We take \(C\) to be the exceptional divisor of this weighted blow-up.
        \end{enumerate}
    \end{enumerate}
    \item\label{item:A_1-case-3} Case 3: $q \in R_Y$, $q \ne y_1$, and the line connecting $q$ to $y_1$ is tangent to $R_Y$ at $q$.

    We use the flag coming from the \((2,1)\) weighted blow-up of \(q\), which separates \(R_Y\) and its tangent line at \(q\), and we take \(C\) to be the exceptional divisor.
\end{enumerate}
See Figure~\ref{fig:A_1-finite-fiber-cases} for an illustration of the flags used in each case. Note that if the fiber of \(\piR_1\colon R \to \bb P^1\) over \(\piR_1(\pt)\) is non-reduced and if \(q \in l_{\tF}\), then \(l_{\tF}\) is the tangent line to \(R_Y\) at \(q\), so since \(y_1 \notin l_{\tF}\) we must be in Case~\ref{item:A_1-case-2}. In other words, Case~\ref{item:A_1-case-3} only occurs if the fiber of \(\piR_1\) at \(\pt\) is reduced.

\begin{remark}\label{rem:A_1-finite-fiber-y_1}
    The assumption that the fiber of \(\piR_2\colon R\to \bb P^2\) is finite at \(\pt\) implies that \(y_1 \notin R_Y\).
\end{remark}

\begin{figure}[h]
\caption{Flags used in each of the different cases in Section~\ref{sec:local-delta-A_n:A_1-finite-part-2}.}
\centering
    \begin{tabular}{c}
\begin{tikzpicture}

\draw [-] plot [smooth cycle, tension=0.1] coordinates { (0,0) (2,0) (2,2) (0,2)}; 
\draw[thick, olive] (1, 1) ellipse (0.7 and 0.5); 
\draw[thick] (0.2,1.8) -- (1.8, 0.2); 
\node[node font=\tiny] at (0.5,1.75) {$C$};
\draw[blue,fill=blue] (1.7,0.3) circle (0.03); 
\node[left, node font=\tiny] at (1.7,.3) {$y_1$}; 
\draw[blue,fill=blue] (0.8,1.2) circle (0.03); 
\node[right, node font=\tiny] at (0.8,1.2) {$q$}; 

\node[node font=\small] at (1,2.5) {Case 1: $ q \notin R$};

\end{tikzpicture}
\end{tabular}
    \\
    \begin{tabular}{c}
\begin{tikzpicture}

\node[node font=\small] at (4,2.5) {Case 2: $ q \in R$ and the line connecting $q$ to $y_1$ is not tangent to $R$};

\node[node font=\tiny] at (1,-0.5) {2(a): $q \notin l_F$};
\draw [-] plot [smooth cycle, tension=0.1] coordinates { (0,0) (2,0) (2,2) (0,2)}; 
\draw[-] (0.35,0.2) -- (0.35,1.8); 
\node[right, node font=\tiny] at (0.35,0.3) {$l_F$};
\draw[thick, olive] (1, 1) ellipse (0.7 and 0.5); 
\draw[thick,] (0.2,1.8) -- (1.8, 0.2); 
\node[node font=\tiny] at (0.5,1.75) {$C$};
\draw[blue,fill=blue] (1.7,0.3) circle (0.03); 
\node[left, node font=\tiny] at (1.7,.3) {$y_1$}; 
\draw[blue,fill=blue] (0.6,1.4) circle (0.03); 
\node[below, node font=\tiny] at (0.6,1.4) {$q$}; 

\node[node font=\tiny] at (4,-0.5) {2(b)(i): $q \in l_F$,};
\node[node font=\tiny] at (4,-0.8) {$\pi_1$ reduced};
\draw [-] plot [smooth cycle, tension=0.1] coordinates { (3,0) (5,0) (5,2) (3,2)}; 
\draw[-] (3.6,0.2) -- (3.6,1.8); 
\node[node font=\tiny] at (3.35,0.3) {$l_F$};
\draw[thick, olive] (4, 1) ellipse (.7 and 0.5); 
\draw[thick,] (3.2,1.8) -- (4.8, 0.2); 
\node[node font=\tiny] at (3.25,1.5) {$C$};
\draw[blue,fill=blue] (4.7,0.3) circle (0.03); 
\node[left, node font=\tiny] at (4.7,.3) {$y_1$}; 
\draw[blue,fill=blue] (3.6,1.4) circle (0.03); 
\node[right, node font=\tiny] at (3.65,1.35) {$q$}; 

\node[node font=\tiny] at (7,-0.5) {2(b)(ii): $q \in l_F$,};
\node[node font=\tiny] at (7,-0.8) {$\pi_1$ non reduced};
\draw [-] plot [smooth cycle, tension=0.1] coordinates { (6,0) (8,0) (8,2) (6,2)}; 
\draw[-] (6.3,0.2) -- (6.3,1.8); 
\node[right, node font=\tiny] at (6.3,0.3) {$l_F$};
\draw[thick, olive] (7, 1) ellipse (.7 and 0.5); 
\draw[thick,] (6.1,1.1) -- (7.9, 0.2); 
\node[node font=\tiny] at (7.2,0.8) {$C$};
\draw[blue,fill=blue] (7.7,0.3) circle (0.03); 
\node[left, node font=\tiny] at (7.7,.3) {$y_1$}; 
\draw[blue,fill=blue] (6.3,1) circle (0.03); 
\node[right, node font=\tiny] at (6.3,1.05) {$q$}; 

\end{tikzpicture}
\end{tabular}
    \\
    \begin{tabular}{c}
\begin{tikzpicture}

\node[node font=\small] at (2.5,2.5) {Case 3: $ q \in R$ and the line connecting $q$ to $y_1$ is tangent to $R$};

\node[node font=\tiny] at (1,-0.5) {3(a): $q \notin l_F$};
\draw [-] plot [smooth cycle, tension=0.1] coordinates { (0,0) (2,0) (2,2) (0,2)}; 
\draw[-] (0.35,0.2) -- (1.8,1.8); 
\node[node font=\tiny] at (0.35,0.4) {$l_F$};
\draw[thick, olive] (1, 1) ellipse (0.7 and 0.5); 
\draw[thick] (1.7,1.8) -- (1.7, 0.2); 
\draw[blue,fill=blue] (1.7,0.3) circle (0.03); 
\node[left, node font=\tiny] at (1.7,.3) {$y_1$}; 
\draw[blue,fill=blue] (1.7,1) circle (0.03); 
\node[left, node font=\tiny] at (1.7,1) {$q$}; 

\node[node font=\tiny] at (4,-0.5) {3(b): $q \in l_F$};
\draw [-] plot [smooth cycle, tension=0.1] coordinates { (3,0) (5,0) (5,2) (3,2)}; 
\draw[-] (3.7,0.2) -- (3.7,1.8); 
\node[right, node font=\tiny] at (3.7,1.7) {$l_F$};
\draw[thick, olive] (4, 1) ellipse (.7 and 0.5); 
\draw[thick,] (3.2,0.73) -- (4.8, 0.15); 
\draw[blue,fill=blue] (4.7,0.19) circle (0.03); 
\node[above, node font=\tiny] at (4.7,.19) {$y_1$}; 
\draw[blue,fill=blue] (3.7,0.55) circle (0.03); 
\node[above right, node font=\tiny] at (3.7,0.55) {$q$}; 

\node[node font=\tiny] at (2.5,-0.8) {$C$ is the exceptional divisor of};
\node[node font=\tiny] at (2.5,-1.1) {the weighted blow-up of $q$};

\end{tikzpicture}
\end{tabular}
    \label{fig:A_1-finite-fiber-cases}
\end{figure}

\noindent\underline{Cases~\ref{item:A_1-case-1} and~\ref{item:A_1-case-2-notin-lF}:} We will apply the Abban--Zhuang method to the line $C$ connecting $q$ to $y_1$.  We will compute these quantities on $\hY$, the strict transform of $Y$ in $\hX$.  Using the computations of $P(u)$ and $N(u)$ in Section~\ref{sec:local-delta-A_n:A_1-finite-part-1} above, we first observe that \(\ord_C(N_{\regionA}(u)|_\hY) = \ord_C(N_{\regionB}(u)|_\hY) = \ord_C(N_{\regionC}(u)|_\hY) = 0 \)
because the support of these restrictions do not contain $C$. In particular \(N(u)|_\hY = N'(u)|_\hY\) on all regions. Therefore, the formula in \cite[Theorem 4.8]{Fujita-3.11} becomes
\begin{equation}\label{eqn:Fujita-double-bullet-S-A_1-case-1} S_{Y,\dee_Y; V_{\bullet \bullet}}(C) = \frac{3}{\vol L} 
\int_0^{5-4c} \left( \int_0^\infty \vol( P(u)|_\hY - vC) dv \right) du. \end{equation}

We will compute this using a Zariski decomposition of the divisor $P(u)|_\hY - vC$ on $\hY$.  By abuse of notation, we will call $\psi\coloneqq \psi|_{\hY} \colon \hY \to Y$.  This is the blow-up of $y_1 \in Y \cong \bb P^2$, so $\hY \cong \bb F_1$.  Let $f$ denote a ruling of of $\bb F_1$ and let $e$ denote the exceptional divisor.  Because $C \subset Y$ was a line containing $y_1$, $\psi^*C = \hC + e$, where $\hC$ is the strict transform of $C$, and $\hC \sim f$.

Consider the divisor $P(u)|_\hY - v\hC$, where $P(u)$ is as computed in Section~\ref{sec:local-delta-A_n:A_1-finite-part-1}, and let $t(u)$ be the pseudoeffective threshold.  Let $P(u,v)$ be its positive part and $N(u,v)$ its negative part. The formula~\eqref{eqn:Fujita-double-bullet-S-A_1-case-1} then becomes:
\[ S_{Y,\dee_Y; V_{\bullet \bullet}}(C) = \frac{3}{\vol L} 
\int_0^{5-4c} \left( \int_{0}^{t(u)} P(u,v)^2 dv \right) du . \]
Furthermore, by \cite[Theorem 4.17]{Fujita-3.11}, we also have the formulas for $W_{\bullet \bullet \bullet}$:
\[ S_{C,\dee_C; W_{\bullet \bullet \bullet}}(q) = \frac{3}{\vol L} 
\int_0^{5-4c} \left( \int_{0}^{t(u)} (P(u,v)\cdot \hC)^2 dv \right) du + F_q(W_{\bullet \bullet \bullet}) \]
where
\[ F_q(W_{\bullet \bullet \bullet})  = \frac{6}{\vol L} 
\int_0^{5-4c} \left( \int_{0}^{t(u)} (P(u,v)\cdot \hC)\ord_q((N'(u)|_\hY + N(u,v) - ve)|_\hC) dv \right) du. \]

So, we compute the Zariski decomposition of $P(u)|_\hY - v\hC$ and integrate.
By abuse of notation, we again use \(l_{\tF}\) to refer to its strict transform under \(\hY\to Y\); note that \(l_{\tF} \sim f+e\).
On each region $\regionA, \regionB, \regionC$, we use the formulas for $P(u)$ from Section~\ref{sec:local-delta-A_n:A_1-finite-part-1} and restrict to $\hY$: 
\[ \resizebox{1\textwidth}{!}{ $ \begin{array}{rll}
    0 \le u \le 2-2c: & P_{\regionA}(u)|_\hY = ue + uf & N_{\regionA}(u)|_\hY  =  0 \\
    2-2c \le u \le 3-2c: & P_{\regionB}(u)|_\hY = (2-2c)e + uf & N_{\regionB}(u)|_\hY = (u+2c - 2) e \\
    3-2c \le u \le 5-4c: & P_{\regionC}(u)|_\hY = (5-4c-u)e + (3-2c)f & N_{\regionC}(u)|_\hY  =(u+2c-2)e + (u+2c-3) l_\tF.
\end{array} $} \]

Starting in region $\regionA$ for $0 \le u \le 2-2c$, $P_{\regionA}(u)|_\hY - v\hC = ue + (u-v)f$.  Here, $t(u) = u$, and we find the Zariski decomposition for \(0\leq v\leq u\):
\[ P_{\regionA}(u,v) = (u-v)(e+f) , \qquad N_{\regionA}(u,v) = ve,\]
so the relevant quantities for $0 \le u \le 2-2c$ are: 
\[\begin{array}{r|l}
     & 0 \le v \le u \\
     \hline
    P_{\regionA}(u,v)^2 & (u-v)^2 \\
    P_{\regionA}(u,v) \cdot \hC & u-v \\
    (P_{\regionA}(u,v) \cdot \hC)^2 & (u-v)^2 \\
    \ord_{\qhatpt}((N_{\regionA}'(u)|_\hY + N_{\regionA}(u,v) - ve)|_\hC) &  0 
\end{array}\]

For region $\regionB$ for $2 -2c \le u \le 3-2c$, $P_{\regionB}(u)|_\hY - v\hC = (2-2c)e + (u-v)f$ and again $t(u) = u$.  Here, the Zariski decomposition is
\begin{equation*}\begin{split}
P_{\regionB}(u,v) &= \begin{cases}
        (2-2c)e + (u-v) f & 0 \le v \le u-2+2c  \\
        (u-v)(e+f) & u-2+2c \le v \le u ,
\end{cases}\\
N_{\regionB}(u,v) &= \begin{cases}
        0 \hphantom{(2-c)e + (u-v) f} & 0 \le v \le u-2+2c  \\
        (v-(u-2+2c))e & u-2+2c \le v \le u ,
\end{cases}\end{split}\end{equation*}
so the relevant quantities for $2-2c \le u \le 3-2c$ are:
\[\begin{array}{r|ll}
   & 0 \le v \le u-2+2c & u-2+2c \le v \le u \\
       \hline
   (P_{\regionB}(u,v))^2   & (2-2c)(2u-2v-(2-2c)) & (u-v)^2 \\
   P_{\regionB}(u,v)\cdot \hC  & 2-2c & u-v \\
   (P_{\regionB}(u,v)\cdot \hC)^2  & (2-2c)^2 & (u-v)^2 \\
   \ord_{\qhatpt}((N_{\regionB}'(u)|_\hY + N_{\regionB}(u,v) - ve)|_\hC) & 
   \begin{cases}
       0  & \qhatpt \ne e\cap f  \\
     u+2c-2-v   & \qhatpt = e\cap f
   \end{cases} & 0 \\
\end{array}\]

For region $\regionC$ with $3-2c \le u \le 5-4c$, $P_{\regionC}(u)|_\hY - v\hC = (5-4c-u)e + (3-2c-v)f$ and $t(u) = 3-2c$.  
Here, the Zariski decomposition is 
\begin{equation*}\begin{split}
P_{\regionC}(u,v) &= \begin{cases}
        (5-4c-u)e + (3-2c-v)f  & 0 \le v \le u-2+2c  \\
        (3-2c-v)(e+f) & u-2+2c \le v \le 3-2c ,
\end{cases}\\
N_{\regionC}(u,v) &= \begin{cases}
        0 \hphantom{(5-4c-u)e + (3-c-v)f}& 0 \le v \le u-2+2c  \\
        (v-(u-2+2c))e & u-2+2c \le v \le 3-2c ,
\end{cases}\end{split}\end{equation*}
so the relevant quantities for $3-2c \le u \le 5-4c$ are: 
\[ \resizebox{1\textwidth}{!}{ $\begin{array}{r|ll}
   & 0 \le v \le u-2+2c & u-2+2c \le v \le 3-2c \\
       \hline
   (P_{\regionC}(u,v))^2   & (5-4c-u)(6 - 4c-2v-(5-4c-u)) & (3-2c-v)^2 \\
   P_{\regionC}(u,v)\cdot \hC  & 5-4c-u & 3-2c-v \\
   (P_{\regionC}(u,v)\cdot \hC)^2  & (5-4c-u)^2 & (3-2c-v)^2 \\
   \ord_{\qhatpt}((N_{\regionC}'(u)|_\hY + N_{\regionC}(u,v) - ve)|_\hC) & 
   \left\{ \begin{array}{ll}
      0  & \qhatpt \ne e \cap f, l_\tF \cap f  \\
      u+2c-3   & \qhatpt = l_\tF \cap f \\
     u+2c-2-v   & \qhatpt = e \cap f
   \end{array} \right. & \left\{ \begin{array}{ll}
      0  & \qhatpt \ne l_\tF \cap f  \\
      u+2c-3   & \qhatpt = l_\tF \cap f \\
   \end{array} \right. 
\end{array}$ } \]

Therefore, using \cite[Theorem 4.8]{Fujita-3.11}, we may compute \(S_{Y,\dee_Y; V_{\bullet \bullet}}(C)\) by integrating $P_{\regionA}(u,v)^2$, $P_{\regionB}(u,v)^2$, and $P_{\regionC}(u,v)^2$. We have $A_{Y, \dee_Y}(C) = 1$ because $C$ is not contained in $R_{\hY}$, so we obtain 
\[S_{Y,\dee_Y; V_{\bullet \bullet}}(C) = \frac{3-2c}{3}, \qquad \frac{A_{Y,\dee_Y} (C)}{S_{Y,\dee_Y; V_{\bullet \bullet}}(C)} = \frac{3}{3-2c} > 1.\]

Next, we use the above formulas and \cite[Theorem 4.17]{Fujita-3.11} to find $S_{C,\dee_C; W_{\bullet \bullet \bullet}}(\qhatpt)$ for \(\qhatpt\) over \(q\):
\[ S_{C,\dee_C; W_{\bullet \bullet \bullet}}(\qhatpt) = \begin{cases}
    \displaystyle \frac{(1-c)(6c^2-20c+17)}{3(3-2c)^2} & \qhatpt = q \notin l_\tF \text{ and } q \ne y_1 \\
    \displaystyle 1-c & \qhatpt = q \in l_{\tF} \\
    \displaystyle \frac{3-2c}{3} & q = y_1.
\end{cases}\]
Since $A_{C, \dee_C}(q) = 1$ (in Case~\ref{item:A_1-case-1}) or $1-c$ (in Case~\ref{item:A_1-case-2-notin-lF}), we have that
\[ \frac{A_{C, \dee_C}(\qhatpt)}{S_{C,\dee_C; W_{\bullet \bullet \bullet}}(\qhatpt)} \geq \begin{cases}
    \displaystyle \frac{3(3-2c)^2}{6c^2-20c+17} & \qhatpt = q \notin l_\tF \text{ and } q \ne y_1 \\
    \displaystyle \frac{1}{1-c} & \text{Case~\ref{item:A_1-case-1} and }\qhatpt = q \in l_{\tF} \\
    \displaystyle \frac{3}{3-2c} & q = y_1 \text{ (necessarily Case~\ref{item:A_1-case-1})},
\end{cases}\]
and this is \(>1\) for \(c\in(0,1)\). This completes the analysis in Cases~\ref{item:A_1-case-1} and~\ref{item:A_1-case-2-notin-lF}.

\noindent\underline{Case~\ref{item:A_1-case-2-in-lF-reduced}:} In this case, we consider \(\qhatpt \in R_{\hY}\cap l_{\tF}\) and use the flag coming from the exceptional divisor \(w\) of the blow-up \(\pi'\colon \hY' \to \hY\) at \(\qhatpt\). By abuse of notation, we use \(e\) and \(l_\tF\) to denote the strict transforms of these curves in \(\hY'\), and we use \(f\) to denote the strict transform of a fiber of \(\hY \cong \bb F_1\) passing through \(\qhatpt\).
\begin{detail}
We have \(e^2 = f^2 = w^2 = -1\), \(e\cdot f = f \cdot w = 1\), and \(e\cdot w = 0\).
\end{detail}

First, using the formulas for \(P(u)|_{\hY}\) computed in Cases~\ref{item:A_1-case-1} and~\ref{item:A_1-case-2-notin-lF} above, we find the Zariski decompositions of \(P(u)|_{\hY'} - v w = \pi'^* P(u)|_{\hY} - v w\) on the relevant regions. Note that on regions \(\regionA\) and \(\regionB\) we have \(d(u)=0\), and on region \(\regionC\) we have \(d(u) = u+2c-3\).

On region \(\regionA\) for \(0 \le u \le 2-2c\), \(P_{\regionA}(u)|_{\hY'} - v w = ue + uf + (u-v)w \). Here \(t(u) = u\), and the Zariski decomposition for \(0\leq v\leq u\) is 
\[ P_{\regionA}(u,v) = ue + uf + (u-v)w , \qquad N_{\regionA}(u,v) = 0.\]
\begin{detail}
So the relevant quantities for $0 \le u \le 2-2c$ and $q' \in w$ are: 
\[\begin{array}{r|l}
     & 0 \le v \le u \\
     \hline
    P_{\regionA}(u,v)^2 & u^2 - v^2 \\
    P_{\regionA}(u,v) \cdot w & v \\
    (P_{\regionA}(u,v) \cdot w)^2 & v^2 \\
    \ord_{q'}((N_{\regionA}'(u)|_{\hY'} + N_{\regionA}(u,v))|_w) &  0 
\end{array}\]
\end{detail}
On region \(\regionB\) for \(2-2c \le u \le 3-2c\), \(P_{\regionB}(u)|_{\hY'} - v w = (2-2c)e + uf + (u-v)w\). Here \(t(u) = u\), and the Zariski decomposition is 
\begin{equation*}\begin{split}
P_{\regionB}(u,v) &= \begin{cases}
        (2-2c)e + uf + (u-v)w  & 0 \le v \le 2-2c  \\
        (2-2c)(e+f) + (u-v)(f+w) & 2-2c \le v \le u ,
\end{cases}\\
N_{\regionB}(u,v) &= \begin{cases}
        0 \hphantom{(2-2)(e+f) + (u-v)(f+w)}& 0 \le v \le 2-2c  \\
        (v-2+2c)f & 2-2c \le v \le u .
\end{cases}\end{split}\end{equation*}
\begin{detail}
So the relevant quantities for $3-2c \le u \le 5-4c$ and $q' \in w$ are:
\[\begin{array}{r|ll}
   & 0 \le v \le 2-2c & 2-2c \le v \le u \\
       \hline
   (P_{\regionB}(u,v))^2   & -(2-2c) (2 - 2c - 2u) - v^2 & 2(2-2c)(u-v) \\
   P_{\regionB}(u,v)\cdot w  & v & 2-2c \\
   (P_{\regionB}(u,v)\cdot w)^2  & v^2 & (2-2c)^2 \\
   \ord_{q'}((N_{\regionB}'(u)|_{\hY'} + N_{\regionB}(u,v))|_w) & 
   0 & \begin{cases}
       0  & q' \notin f  \\
       v-2+2c  & q' = w \cap f
   \end{cases} \\
\end{array}\]
\end{detail}
On region \(\regionC\) for \(3-2c \le u \le 5-4c\), \(P_{\regionC}(u)|_{\hY'} - v w = (5-4c-u)e + (3-2c)f + (3-2c-v)w \). Here \(t(u) = 3-2c\), and the Zariski decomposition is 
\begin{equation*}\begin{split}
P_{\regionC}(u,v) &= \begin{cases}
        (5-4c-u)e + (3-2c)f + (3-2c-v)w  & 0 \le v \le 5-4c-u  \\
        (5-4c-u)(e+f) + (3-2c-v)(f+w) & 5-4c-u \le v \le 3-2c ,
\end{cases}\\
N_{\regionC}(u,v) &= \begin{cases}
        0 \hphantom{(5-c-u)e + (3-2c)f + (3-2c-v)w}& 0 \le v \le 5-4c-u  \\
        (u+v-5+4c)f & 5-4c-u \le v \le 3-2c ,
\end{cases}\end{split}\end{equation*}
and we have \(N_{\regionC}(u)|_{\hY'} = (u+2c-2)e + (u+2c-3) l_\tF + (u+2c-3)w\) and \(N_{\regionC}'(u) = (u+2c-2)e + (u+2c-3) l_\tF\).
\begin{detail}
So the relevant quantities for $3-2c \le u \le 5-4c$ and $q' \in w$ are:
\[ \resizebox{1\textwidth}{!}{ $\begin{array}{r|ll}
   & 0 \le v \le 5-4c-u & 5-4c-u \le v \le 3-2c \\
       \hline
   (P_{\regionC}(u,v))^2   & (1 + u) (5 - 4 c - u) - v^2 & 2(5-4c-u)(3-2c-v) \\
   P_{\regionC}(u,v)\cdot w  & v & 5-4c-u \\
   (P_{\regionC}(u,v)\cdot w)^2  & v^2 & (5-4c-u)^2 \\
   \ord_{q'}((N_{\regionC}'(u)|_{\hY'} + N_{\regionC}(u,v))|_w) & 
   \left\{ \begin{array}{ll}
      0  & q' \notin l_\tF  \\
      u+2c-3 & q' = l_\tF \cap w
   \end{array} \right. & \left\{ \begin{array}{ll}
      0  & q' \notin l_\tF, f  \\
      u+2c-3 & q' = l_\tF \cap w \\
      u+v-5+4c & q' = f \cap w
   \end{array} \right. 
\end{array}$ } \]
\end{detail}
Therefore, by \cite[Theorem 4.8]{Fujita-3.11} we may compute \(S_{Y,\dee_Y; V_{\bullet \bullet}}(w) = (6 - 5 c)/3\) by integrating the above formulas.
\begin{detail}
In more detail,
\begin{equation*}\begin{split}
    S_{Y,\dee_Y; V_{\bullet \bullet}}(w) &= \frac{3}{\vol L} \left( \int_{3-2c}^{5-4c} (P_{\regionC}(u)|_{\tY'})^2(u+2c-3) du + \int_0^{5-4c} \int_0^{t(u)} P(u,v)^2 dv \; du \right) \\
    &= \frac{2 (c - 1) (2 c - 3)^2 (5 c - 6)}{3(2-2c)(3-2c)^2} = \frac{6 - 5 c}{3}
\end{split}\end{equation*}
\end{detail}
We have $A_{Y, \dee_Y}(w) = 2-c$ because \(\qhatpt\in R_{\hY}\), so we obtain 
\(A_{Y,\dee_Y} (w) / S_{Y,\dee_Y; V_{\bullet \bullet}}(w) = (6-3 c)/(6-5 c) > 1\) for \(c\in (0,1)\).

Next, we use \cite[Theorem 4.17]{Fujita-3.11} to compute $S_{w,\dee_w; W_{\bullet \bullet \bullet}}(q')$ for \(q'\in w\):
\[ S_{w,\dee_w; W_{\bullet \bullet \bullet}}(q') = \begin{cases}
    \displaystyle \frac{(1 - c) (6 c^2 - 20 c + 17)}{3 (2 c - 3)^2} & q' \notin l_\tF, f \\
    \displaystyle 1 - c & q' = l_\tF \cap w \\
    \displaystyle \frac{3 - 2 c}{3} & q' = f \cap w.
\end{cases}\]
Since $A_{w, \dee_w}(q') = 1-c$ (if \(q'\in R_w\), which can only happen in the first case below) or $1$, we have
\[ \frac{A_{w, \dee_w}(q')}{S_{w,\dee_w; W_{\bullet \bullet \bullet}}(q')} = \begin{cases}
    \displaystyle \geq \frac{3 (2 c - 3)^2}{6 c^2 - 20 c + 17} & q' \notin l_\tF, f \\
    \displaystyle \frac{1}{1-c} & q' = l_\tF \cap w \\
    \displaystyle \frac{3}{3 - 2 c} & q' = f \cap w ,
\end{cases}\]
and in all cases this is \(>1\) for \(c\in(0,1)\).

\noindent\underline{Case~\ref{item:A_1-case-2-in-lF-nonreduced} for \(c \leq \tfrac{1}{2}\)}: In this case, we consider \(\qhatpt = R_{\hY}\cap l_{\tF}\) and use the flag coming from the exceptional divisor \(x\) of the \((2,1)\) weighted blow-up \(\pi''\colon \hY'' \to \hY\) separating \(R_y\) and \(l_{\tF}\). By abuse of notation, we use \(e\) and \(l_\tF\) to denote their strict transforms in \(\hY''\), and we use \(f\) to denote the strict transform of a fiber of \(\hY \cong \bb F_1\) passing through \(q\).

On \(\hY''\), we have \(x^2=f^2=-\tfrac{1}{2}\), \(l_\tF^2 = e^2 = -1\), \(x\cdot f = \tfrac{1}{2}\), \(x\cdot l_\tF = 1\), and \(x\cdot e = 0\).
We also have \(\pi''^* e = e\), \(\pi''^* f = f + x\), and \(\pi''^* l_\tF = l_\tF + 2x\). Note that \(\hY''\) is singular at \(f\cap x\).

We first use the formulas for \(P(u)|_{\hY}\) computed in Cases~\ref{item:A_1-case-1} and~\ref{item:A_1-case-2-notin-lF} above to find the Zariski decompositions of \(P(u)|_{\hY''} - v x = \pi''^* P(u)|_{\hY} - v x\) on the relevant regions. Note that \(l_\tF \sim e+f\) on \(\hY\) (since \(\psi_*f \subset Y\) is a line passing through the center \(y_1\) of the blow-up \(\psi\colon \hY \to Y\), whereas \(\psi_*l_\tF \subset Y\) does not pass through \(y_1\)). Also note that \(d(u)=0\) on regions \(\regionA\) and \(\regionB\), and that \(d(u) = 2(u+2c-3)\) on region \(\regionC\).

On region \(\regionA\) for \(0 \le u \le 2-2c\), \(P_{\regionA}(u)|_{\hY''} - v x = u l_\tF + (2u-v)x \). Here \(t(u) = 2u\), and the Zariski decomposition for \(0\leq v\leq u\) is 
\begin{equation*}\begin{split}
P_{\regionA}(u,v) &= \begin{cases}
        u l_\tF + (2u-v)x  & 0 \le v \le u  \\
        (2u-v)(l_\tF + x) & u \le v \le 2u ,
\end{cases}\\
N_{\regionA}(u,v) &= \begin{cases}
        0 \hphantom{u l_\tF + (2u-v)x} & 0 \le v \le u  \\
        (v-u)l_\tF & u \le v \le 2u.
\end{cases}\end{split}\end{equation*}
\begin{detail}
So the relevant quantities for $0 \le u \le 2-2c$ and $q'' \in x$ are: 
\[\begin{array}{r|ll}
   & 0 \le v \le u & u \le v \le 2u \\
       \hline
   (P_{\regionA}(u,v))^2   & u^2 - v^2/2 & (2u-v)^2/2 \\
   P_{\regionA}(u,v)\cdot x  & v/2 & (2u-v)/2 \\
   (P_{\regionA}(u,v)\cdot x)^2  & v^2/4 & (2u-v)^2/4 \\
   \ord_{q''}((N_{\regionA}'(u)|_{\hY''} + N_{\regionA}(u,v))|_x) & 
   0 & \begin{cases}
       0  & q' \notin l_\tF  \\
       v-u  & q' = x \cap l_\tF
   \end{cases} \\
\end{array}\]
\end{detail}
On region \(\regionB\) for \(2-2c \le u \le 3-2c\), \(P_{\regionB}(u)|_{\hY''} - v x = (2-2c)l_\tF + (u-2+2c)f + (u+2-2c-v)x\). Here \(t(u) = u+2-2c\). When computing the Zariski decomposition, \(l_\tF\) is contracted at \(v=u\), and \(f\) is contracted at \(v=4-4c\). The assumption that \(c\leq\tfrac{1}{2}\) implies that \(u\leq 3-2c \leq 4-4c\). Thus, the Zariski decomposition is 
\begin{equation*}\begin{split}
P_{\regionB}(u,v) &= \begin{cases}
        (2-2c)l_\tF + (u-2+2c)f + (u+2-2c-v)x  & 0 \le v \le u  \\
        (u+2-2c-v)(l_\tF + x) + (u-2+2c)f & u \le v \le 4-4c \\
        (u+2-2c-v)(l_\tF + f + x) & 4-4c \le v \le u+2-2c ,
\end{cases}\\
N_{\regionB}(u,v) &= \begin{cases}
        0 \hphantom{(2-c)l_\tF + (u-2+2c)f + (u+2-2c-v)x} & 0 \le v \le u  \\
        (v-u)l_\tF & u \le v \le 4-4c \\
        (v-u)l_\tF + (v-4+4c)f & 4-4c \le v \le u+2-2c .
\end{cases}\end{split}\end{equation*}
\begin{detail}
So the relevant quantities for $2-2c \le u \le 3-2c$ and $q'' \in x$ are:
\[ \resizebox{1\textwidth}{!}{ $ \begin{array}{r|lll}
   & 0 \le v \le u & u \le v \le 4-4c & 4-4c \le v \le u+2-2c \\
       \hline
   (P_{\regionB}(u,v))^2  & (2-2c)(2u-2+2c)-v^2/2 & (u+2-2c-v)^2/2 + (u+2-2c-v)(u-2+2c) - (u-2+2c)^2/2 &  (u+2-2c-v)^2 \\
   P_{\regionB}(u,v)\cdot x & v/2 & u-v/2 & u+2-2c-v \\
   (P_{\regionB}(u,v)\cdot x)^2 & v^2/4 & (u-v/2)^2 & (u+2-2c-v)^2 \\
   \ord_{q''}((N_{\regionB}'(u)|_{\hY''} + N_{\regionB}(u,v))|_x) & 0 & \left\{ \begin{array}{ll}
      0  & q'' \notin l_\tF  \\
     v-u   & q'' = x \cap l_\tF
   \end{array} \right. &
   \left\{ \begin{array}{ll}
      0  & q'' \notin l_\tF, f  \\
     v-u   & q'' = x \cap l_\tF \\
     (v-4+4c)/2   & q'' = x \cap f
   \end{array} \right. \\
\end{array} $} \]
\end{detail}
On region \(\regionC\) for \(3-2c \le u \le 5-4c\), we have \(P_{\regionC}(u)|_{\hY''} - v x = (5-4c-u)l_\tF + (2c-2+u)f + (8-6c-u-v)x \) and \(N_{\regionC}'(u)=(u+2c-2)e + (u+2c-3)l_\tF\). Here \(t(u) = 8-6c-u\). Since \((P_{\regionC}(u)|_{\hY''} - v x) \cdot l_\tF = 3-2c-v\) and \((P_{\regionC}(u)|_{\hY''} - v x)\cdot f = \frac{1}{2} (10 -8 c - 2 u - v)\), we have two subregions \(\regionC_1,\regionC_2\) to consider for the Zariski decomposition.
Note that since \(c\leq\frac{1}{2}\), we have \(3-2c \leq \frac{7}{2}-3c\). Also note that, since \(2-2c \leq u\), we have \(10-8c-2u \leq 8-6c-u\).

If \(3-2c \le u \le \tfrac{7}{2} - 3c\), then \(3-2c \leq 10-8c-2u\), so the Zariski decomposition is
\begin{equation*}\begin{split}
P_{\regionC_1}(u,v) &= \begin{cases}
        (5-4c-u)l_\tF + (2c-2+u)f + (8-6c-u-v)x  & 0 \le v \le 3-2c  \\
        (8-6c-u-v)(l_\tF + x) + (2c-2+u)f & 3-2c \le v \le 10-8c-2u \\
        (8-6c-u-v)(l_\tF + f + x) & 10-8c-2u \le v \le 8-6c-u ,
\end{cases}\\
N_{\regionC_1}(u,v) &= \begin{cases}
        0 \hphantom{(5-c-u)l_\tF + (2c-2+u)f + (8-6c-u-v)x} & 0 \le v \le 3-2c  \\
        (v - 3 + 2 c)l_\tF & 3-2c \le v \le 10-8c-2u \\
        (v - 3 + 2 c)l_\tF + (v - 10 + 8 c + 2 u)f & 10-8c-2u \le v \le 8-6c-u .
\end{cases}\end{split}\end{equation*}
\begin{detail}
So the relevant quantities for $q'' \in x$ are:
\[ \resizebox{1\textwidth}{!}{ $ \begin{array}{r|lll}
   & 0 \le v \le 3-2c & 3-2c \le v \le 10-8c-2u & 10-8c-2u \le v \le 8-6c-u \\
       \hline
   (P_{\regionC_1}(u,v))^2  & (5-4c-u)(1+u)-v^2/2 & (8-6c-u-v)^2/2 + (8-6c-u-v)(2c-2+u) - (2c-2+u)^2/2 &  (8-6c-u-v)^2 \\
   P_{\regionC_1}(u,v)\cdot x & v/2 & 3 -2 c - v/2 & 8-6c-u-v \\
   (P_{\regionC_1}(u,v)\cdot x)^2 & v^2/4 & (3 -2 c - v/2)^2 & (8-6c-u-v)^2 \\
   \ord_{q''}((N_{\regionC_1}'(u)|_{\hY''} + N_{\regionC_1}(u,v))|_x) & \left\{ \begin{array}{ll}
      0  & q'' \notin l_\tF  \\
      u+2c-3   & q'' = x \cap l_\tF
   \end{array} \right. & \left\{ \begin{array}{ll}
      0  & q'' \notin l_\tF  \\
      u + v + 4 c - 6   & q'' = x \cap l_\tF
   \end{array} \right. &
   \left\{ \begin{array}{ll}
      0  & q'' \notin l_\tF,f  \\
      u + v + 4 c - 6   & q'' = x \cap l_\tF \\
      (v - 10 + 8 c + 2 u)/2 & q'' = x \cap f
   \end{array} \right. \\
\end{array} $} \]
\end{detail}
If \(\tfrac{7}{2} - 3c \le u \le 5-4c\), then \(3-2c \geq 10-8c-2u\), and the Zariski decomposition is 
\begin{equation*}\begin{split}
P_{\regionC_2}(u,v) &= \begin{cases}
        (5-4c-u)l_\tF + (2c-2+u)f + (8-6c-u-v)x  & 0 \le v \le 10-8c-2u  \\
        (5-4c-u)l_\tF + (8-6c-u-v)(f + x) & 10-8c-2u \le v \le 3-2c \\
        (8-6c-u-v)(l_\tF + f + x) & 3-2c \le v \le 8-6c-u ,
\end{cases}\\
N_{\regionC_2}(u,v) &= \begin{cases}
        0 \hphantom{(5-c-u)l_\tF + (2c-2+u)f + (8-6c-u-v)x} & 0 \le v \le 10-8c-2u  \\
        (v - 10 + 8 c + 2 u)f & 10-8c-2u \le v \le 3-2c \\
        (v - 3 + 2 c)l_\tF + (v - 10 + 8 c + 2 u)f & 3-2c \le v \le 8-6c-u .
\end{cases}\end{split}\end{equation*}
\begin{detail}
So the relevant quantities for $q'' \in x$ are:
\[ \resizebox{1\textwidth}{!}{ $ \begin{array}{r|lll}
   & 0 \le v \le 10-8c-2u & 10-8c-2u \le v \le 3-2c & 3-2c \le v \le 8-6c-u \\
       \hline
   (P_{\regionC_2}(u,v))^2  & (5-4c-u)(1+u)-v^2/2 & (5 -4 c - u) (-8 c - u - 2 v + 11) &  (8-6c-u-v)^2 \\
   P_{\regionC_2}(u,v)\cdot x & v/2 & 5-4c-u & 8-6c-u-v \\
   (P_{\regionC_2}(u,v)\cdot x)^2 & v^2/4 & (5-4c-u)^2 & (8-6c-u-v)^2 \\
   \ord_{q''}((N_{\regionC_2}'(u)|_{\hY''} + N_{\regionC_2}(u,v))|_x) & \left\{ \begin{array}{ll}
      0  & q'' \notin l_\tF  \\
      u+2c-3   & q'' = x \cap l_\tF
   \end{array} \right. & \left\{ \begin{array}{ll}
      0  & q'' \notin l_\tF,f  \\
      u+2c-3   & q'' = x \cap l_\tF \\
      (v - 10 + 8 c + 2 u)/2 & q'' = x \cap f
   \end{array} \right. &
   \left\{ \begin{array}{ll}
      0  & q'' \notin l_\tF,f  \\
      u + v + 4 c - 6   & q'' = x \cap l_\tF \\
      (v - 10 + 8 c + 2 u)/2 & q'' = x \cap f
   \end{array} \right. \\
\end{array} $} \]
\end{detail}
Next, compute \(S_{Y,\dee_Y; V_{\bullet \bullet}}(x) = (9 - 8 c)/3\) using \cite[Theorem 4.8]{Fujita-3.11}.
\begin{detail}
In more detail,
\begin{equation*}\begin{split}
    S_{Y,\dee_Y; V_{\bullet \bullet}}(x) &= \frac{3}{\vol L} \left( \int_{3-2c}^{5-4c} (P_{\regionC}(u)|_{\tY''})^2 \cdot 2(u+2c-3) du + \int_0^{5-4c} \int_0^{t(u)} P(u,v)^2 dv \; du \right) \\
    &= \frac{(8/3) (-1 + c)^3 (-5 + 3 c) + (2/3) (c - 1) (20 c^3 - 88 c^2 + 128 c - 61)}{3(2-2c)(3-2c)^2} = \frac{9-8c}{3}.
\end{split}\end{equation*}
\end{detail}
Since \(A_{Y,\dee_Y}(x) = 3-2c\), we have that \(A_{Y,\dee_Y}(x)/S_{Y,\dee_Y; V_{\bullet \bullet}}(x) = (9 - 6 c)/(9 - 8 c) > 1\) for \(c\in (0,1)\).

Next, we compute $S_{x,\dee_x; W_{\bullet \bullet \bullet}}(q'')$ for \(q''\in x\) and \(c\in(0,\tfrac{1}{2}]\): 
\[ S_{x,\dee_x; W_{\bullet \bullet \bullet}}(q'') = \begin{cases}
    \displaystyle \frac{(14 c - 13) (2 c - 3)}{96 (1 - c)} & q'' \notin l_\tF, f \\
    \displaystyle 1-c & q'' = x \cap l_\tF \\
    \displaystyle \frac{3 - 2 c}{6} & q'' = x \cap f.
\end{cases}\]
The log discrepancy \(A_{x, \dee_x}(q'')\) is \(1\) if \(q''=x \cap l_\tF\), and is \(\tfrac{1}{2}\) if \(q''=x \cap f\). If \(q''\notin l_\tF,f\), then $A_{x, \dee_x}(q'')$ is either \(1\) or $1-c$. Therefore,
\[ \frac{A_{x, \dee_x}(q'')}{S_{x,\dee_x; W_{\bullet \bullet \bullet}}(q')} = \begin{cases}
    \displaystyle \geq \frac{96 (1 - c)^2}{(3 - 2 c) (13 - 14 c)} & q'' \notin l_\tF, f \\
    \displaystyle \frac{1}{1-c} & q'' = x \cap l_\tF  \\
    \displaystyle \frac{3}{3 - 2 c} & q'' = x \cap f.
\end{cases}\]
and in all cases this is \(>1\) for \(c\in(0,\tfrac{1}{2}]\).

\noindent\underline{Case~\ref{item:A_1-case-3}:} Suppose that $q \in R_Y$, $q \ne y_1$, and the line connecting $q$ to $y_1$ is tangent to $R_Y$ at $q$. (Recall that in this case, the fiber of \(\piR_1\) at \(\pt\) is necessarily reduced, since \(y_1\notin l_{\tF}\).)  We use the flag coming from the $(2,1)$ weighted blow-up of the point $q$ separating $R_Y$ and its tangent line $f$ at $q$.  As before, we work on $\hY$. By abuse of notation, we will also use \(q\) to denote its preimage in \(\hY\).

Let $\beta\colon \oY \to \hY $ be the $(2,1)$ weighted blow-up of $q$ with exceptional divisor $z$.  Let $q_0 \in z$ be the singular point.  By abuse of notation, let $e$, $f$, $l_\tF$ denote their strict transforms on $\oY$.  As $q$ was not on $e$, we may compute the pullback of $P(u)|_\hY$ to $\oY$.  Note that if $q \in l_\tF$ on $Y$, then $l_\tF$ intersects $z$ at $q_0$.   By construction, the pullback of the tangent line to $R_Y$ at $q$ is $f + 2z$.  Let $N(u)|_\oY$ denote the pullback of $N(u)|_\hY$, and let \(N'(u)|_\oY = N(u)|_\oY - \ord_z(N(u)|_\hY) z\).
\begin{detail}
If $P(u)|_\oY - vz = P(u,v) + N(u,v)$ is a Zariski decomposition, then by \cite[Theorems 4.8 and 4.17]{Fujita-3.11}, we have 
\[ \delta_q(Y, \dee_Y; V_{\bullet \bullet}) \ge \min \left\{ \frac{A_{Y, \dee_Y}(z)}{S_{Y,\dee_Y; V_{\bullet \bullet}}(z)}, \inf_{
     \qbarpt \in z :
     \beta(\qbarpt) = q } \delta_{\qbarpt}(z,\dee_z; W_{\bullet \bullet \bullet}) \right\} \]
where 
\[ S_{Y,\dee_Y; V_{\bullet \bullet}}(z) = \frac{3}{\vol L} 
\int_0^{5-4c} \left( \int_0^\infty \vol( P(u)|_\oY - vz) dv \right) du \]
and
\[ S_{z,\dee_z; W_{\bullet \bullet \bullet}}(\qbarpt) = \frac{3}{\vol L} 
\int_0^{5-4c} \left( \int_{0}^{t(u)} (P(u,v)\cdot z)^2 dv \right) du + F_{\qbarpt}(W_{\bullet \bullet \bullet}) \]
where 
\[ F_{\qbarpt}(W_{\bullet \bullet \bullet})  = \frac{6}{\vol L} 
\int_0^{5-4c} \left( \int_{0}^{t(u)} (P(u,v)\cdot \hC)\ord_{\qbarpt}((N'(u)|_\oY + N(u,v))|_z dv) \right) du. \]
\end{detail}
From above, we have:
\[ \resizebox{1\textwidth}{!}{ $ \begin{array}{rll}
    0 \le u \le 2-2c: & P_{\regionA}(u)|_\oY = ue + uf + 2uz & N_{\regionA}(u)|_\oY  =  0 \\
    2-2c \le u \le 3-2c: & P_{\regionB}(u)|_\oY = (2-2c)e + uf + 2uz & N_{\regionB}(u)|_\oY = (u+2c - 2) e \\
    3-2c \le u \le 5-4c: & P_{\regionC}(u)|_\oY = (5-4c-u)e + (3-2c)f + (6-4c)z  & N_{\regionC}'(u)|_\oY  =(u+2c-2)e + (u+2c-3) l_\tF.
\end{array} $} \]
We have \(d(u)=0\) on regions \(\regionA\) and \(\regionB\), and on \(\regionC\) if \(q\notin l_\tF\). If \(q\in l_\tF\), then \(d(u)=u+2c-3\) on region \(\regionC\).
Note that $e^2 = -1$, $f^2 = -2$, $z^2 = -\frac{1}{2}$, $e\cdot f = 1$, $e \cdot z = 0$, and $f \cdot z = 1$.

For each $P(u)|_\oY$, we compute the Zariski decomposition of $P(u)|_\oY - vz = P(u,v) + N(u,v)$: 

For region $\regionA$ for $0 \le u \le 2-2c$, $P_{\regionA}(u)|_\oY - vz = ue + uf + (2u-v)z$ and $t(u) = 2u$.  Here, the Zariski decomposition is
\[P_{\regionA}(u,v) = \begin{cases} ue + uf + (2u-v)z & 0 \le v \le u \\ (2u-v)(e+f+z) & u \le v \le 2u , \end{cases} \qquad N_{\regionA}(u,v) = \begin{cases} 0 & 0 \le v \le u   \\
(v-u)(e+f) & u \le v \le 2u .
\end{cases}\]
\begin{detail}
So the relevant quantities for $0 \le u \le 2-2c$ are: 
\[\begin{array}{r|ll}
   & 0 \le v \le u & u \le v \le 2u \\
       \hline
   (P_{\regionA}(u,v))^2   & u^2 - \frac{v^2}{2} & \frac{(2u-v)^2}{2} \\
   P_{\regionA}(u,v)\cdot z  & \frac{v}{2} & \frac{2u-v}{2} \\
   (P_{\regionA}(u,v)\cdot z)^2  & \frac{v^2}{4} & \frac{(2u-v)^2}{4} \\
   \ord_{\qbarpt}((N_{\regionA}'(u)|_\oY + N_{\regionA}(u,v))|_z) & 0 &
   \left\{ \begin{array}{ll}
      0  & \qbarpt \ne z \cap f  \\
     v-u   & \qbarpt = z \cap f
   \end{array} \right. \\
\end{array}\]
\end{detail}
For region $\regionB$ for $2-2c \le u \le 3-2c$, $P_{\regionB}(u)|_\oY - vz =(2-2c)e + uf + (2u-v)z$ and $t(u) = 2u$.  Here, the Zariski decomposition is
\begin{equation*}\begin{split}P_{\regionB}(u,v) &= \begin{cases}
        (2-2c)e + uf + (2u-v)z  & 0 \le v \le 2-2c  \\
        (2-2c)e + \frac{1}{2}(2u-v+2-2c)f + (2u-v)z & 2-2c \le v \le 2u-2+2c \\
        (2u-v)(e+f+z) & 2u - 2 + 2c \le 2u ,
\end{cases}\\
N_{\regionB}(u,v) &= \begin{cases}
        0\hphantom{(2-2c)e + \frac{1}{2}(2u-v+2-2c)f + (2u-v)}& 0 \le v \le u   \\
        \frac{1}{2}(v-2+2c)f  & 2-2c \le v \le 2u-2+2c \\
        (v-2u+2-2c)e + (v-u)f & 2u - 2 + 2c \le 2u .
\end{cases}\end{split}\end{equation*}
\begin{detail}
So the relevant quantities for $2-2c \le u \le 3-2c$ are:
\[ \resizebox{1\textwidth}{!}{ $ \begin{array}{r|lll}
   & 0 \le v \le 2-2c & 2-2c \le v \le 2u-2 + 2c & 2u-2+2c \le v \le 2u \\
       \hline
   (P_{\regionB}(u,v))^2  & (2-2c)(2u-2+2c)-\frac{v^2}{2} & (2-2c)(2u-v-1+c) &  \frac{(2u-v)^2}{2} \\
   P_{\regionB}(u,v)\cdot z & \frac{v}{2} & 1-c & \frac{2u-v}{2} \\
   (P_{\regionB}(u,v)\cdot z)^2 & \frac{v^2}{4} & (1-c)^2 & \frac{(2u-v)^2}{4} \\
   \ord_{\qbarpt}((N_{\regionB}'(u)|_\oY + N_{\regionB}(u,v))|_z) & 0 & \left\{ \begin{array}{ll}
      0  & \qbarpt \ne z \cap f  \\
     \frac{1}{2}(v-2+2c)   & \qbarpt = z \cap f
   \end{array} \right. &
   \left\{ \begin{array}{ll}
      0  & \qbarpt \ne z \cap f  \\
     v-u   & \qbarpt = z \cap f
   \end{array} \right. \\
\end{array} $} \]
\end{detail}
Finally, for region $\regionC$ on the interval $3-2c \le u \le 5-4c$, $P_{\regionC}(u)|_\oY - vz = (5-4c-u)e + (3-2c)f + (6-4c-v)z $ and $t(u) = 6-4c$.  Here, the Zariski decomposition is 
\begin{equation*}
\begin{split}
    P_{\regionC}(u,v) &= \begin{cases}
        (5-4c-u)e + (3-2c)f + (6-4c-v)z & 0 \le v \le 5-4c-u  \\
        (5-4c-u)e + \frac{1}{2}(11-8c-u-v)f + (6-4c-v)z & 5-4c-u \le v \le u+1 \\
        (6-4c-v)(e+f+z)  & u+1 \le v \le 6-4c ,
    \end{cases} \\
    N_{\regionC}(u,v) &= \begin{cases}
        0 \hphantom{(5-4c-u)e + \frac{1}{2}(11-8c-u-v)f + (6-4c-v)} & 0 \le v \le 5-4c-u   \\
        \frac{1}{2}(v-5+4c+u)f & 5-4c-u \le v \le u+1 \\
        (v-u+1)e + (v-3+2c)f & u+1 \le v \le 6-4c .
    \end{cases}
\end{split}
\end{equation*}
\begin{detail}
So the relevant quantities for $3-2c \le u \le 5-4c$ are:
\[ \resizebox{1\textwidth}{!}{ $ \begin{array}{r|lll}
   & 0 \le v \le 5-4c-u  & 5-4c-u \le v \le u+1  & u+1 \le v \le 6-4c \\
       \hline
   (P_{\regionC}(u,v))^2  & (5-4c-u)(u+1)-\frac{v^2}{2} & (5-4c-u)(\frac{7-4c-2v+u}{2}) &  \frac{(6-4c-v)^2}{2} \\
   P_{\regionC}(u,v)\cdot z & \frac{v}{2} & \frac{5-4c-u}{2} & \frac{6-4c-v}{2} \\
   (P_{\regionC}(u,v)\cdot z)^2 & \frac{v^2}{4} & \frac{(5-4c-u)^2}{4} & \frac{(6-4c-v)^2}{4} \\
   \ord_{\qbarpt}((N_{\regionC}'(u)|_\oY + N_{\regionC}(u,v))|_z) & 0 & \left\{ \begin{array}{ll}
      0  & \qbarpt \ne z \cap f  \\
     \frac{1}{2}(v-5+4c+u)   & \qbarpt = z \cap f
   \end{array} \right. &
   \left\{ \begin{array}{ll}
      0  & \qbarpt \ne z \cap f \text{ or } \qbarpt =q_0 \text{ and } \beta(\qbarpt) \notin l_\tF  \\
     v-3+2c   & \qbarpt = z \cap f \\
     \frac{u-3+2c}{2} & \qbarpt = q_0 \text{ and } \beta(\qbarpt) \in l_\tF 
   \end{array} \right. \\
\end{array} $} \]
\end{detail}
Using \cite[Theorem 4.8]{Fujita-3.11}, we then compute
\[ S_{Y,\dee_Y; V_{\bullet \bullet}}(z) =  \begin{cases} \displaystyle\frac{-(22c^3 - 98c^2 + 145c - 71)}{3(3-2c)^2} & q \notin l_\tF \\
\displaystyle \frac{9 - 7 c}{3} & q \in l_\tF . \end{cases}\]
Because $A_{Y,\dee_Y}(z) = 3-2c$, we have \(A_{Y,\dee_Y}(z)/S_{Y,\dee_Y; V_{\bullet \bullet}}(z) > 1\) for \(c \in (0,1)\).

Finally, we compute $S_{z,\dee_z; W_{\bullet \bullet \bullet}}(\qbarpt)$ for $\qbarpt \in z$, using \cite[Theorem 4.17]{Fujita-3.11}.  If $\qbarpt \ne f \cap z$ and $\qbarpt$ is not the intersection of the strict transform of $l_\tF$ and $z$, then
\[ S_{z,\dee_z; W_{\bullet \bullet \bullet}}(\qbarpt) =  \frac{(1-c)(6c^2-20c+17)}{6(3-2c)^2} \]
and $A_{z,\dee_z}(\qbarpt) = 1$, $1-c$, or $\frac{1}{2}$: it is $1$ if $\qbarpt \notin \dee_z$, $1-c$ if $\qbarpt$ is the point of intersection of the strict transform of $R_Y$ and $z$, and it is $\frac{1}{2}$ if $\qbarpt$ is the singular point of $\oY$ on $z$.  In each case, we find 
\[ \frac{A_{z,\dee_z}(q)}{S_{z,\dee_z; W_{\bullet \bullet \bullet}}(\qbarpt)} > 1. \]

Finally, suppose $\qbarpt$ is either the intersection point of $f$ and $z$ or the strict transform of $l_\tF$ and $z$, which is the singular point \(q_0\) of $z$.
We compute that \(A_{z,\dee_z}(z \cap f) = 1\) and \(A_{z,\dee_z}(q_0) = \frac{1}{2}\), and
\[ S_{z,\dee_z; W_{\bullet \bullet \bullet}}(\qbarpt) = \begin{cases}
    \displaystyle\frac{3-2c}{3} & \qbarpt = z \cap f \\
    \displaystyle\frac{(1 - c) (7 c^2 - 22 c + 18)}{6 (3 - 2 c)^2} & \qbarpt = q_0 \text{ and } \beta(\qbarpt)\in l_\tF. \end{cases}\]
Therefore, in either case, 
\[ \frac{A_{z,\dee_z}(\qbarpt)}{S_{z,\dee_z; W_{\bullet \bullet \bullet}}(\qbarpt)}  > 1. \]
This completes the analysis in Case~\ref{item:A_1-case-3} that $\delta_q(Y,\dee_Y; V_{\bullet \bullet}) > 1$, thus showing Theorem~\ref{thm:local-delta-A_n-singularities}\eqref{part:local-delta-A_n-singularities:A_1-finite-reduced}.

\subsection{Higher \texorpdfstring{\(A_n\)}{An} singularities for \texorpdfstring{\(c < \cnot \approx 0.47\)}{c < c1}}\label{sec:local-delta-A_n:higher-part-1}
We may use the previous computations in Section~\ref{sec:local-delta-A_n:A_1-finite-part-2} to show \((2,2)\)-surfaces with higher $A_n$ singularities satisfying the conditions of Theorem~\ref{thm:local-delta-A_n-singularities}\eqref{part:local-delta-A_n-singularities:higher-A_n} are K-stable for \(c\in(0, \cnot)\), where \(\cnot \approx 0.47233\) is the irrational number defined in Theorem~\ref{thm:moduli-spaces-2.18}. Later, in Section~\ref{sec:local-delta-A_n:higher-part-2} below, we will use a different flag to show K-stability for the remaining values of \(c \leq 0.7055\).

Let $R$ be a \((2,2)\)-surface with an $A_n$ singularity at $\pt$ with \(n\geq 2\), and assume that the fiber of \(\piR_2\colon R\to\bb P^2\) over \(\piR_2(\pt)\) is finite and that the fiber of \(\piR_1\colon R\to\bb P^1\) over \(\piR_1(\pt)\) is reduced. Consider the same set-up as in Cases~\ref{item:A_1-case-1} and~\ref{item:A_1-case-2-notin-lF} of Section~\ref{sec:local-delta-A_n:A_1-finite} above, and note that the \(n\geq 2\) assumption implies $R_Y$ is a rank 2 conic, i.e. two lines in $Y$ meeting at a point $q_1 \ne y_1$. The computations of Section~\ref{sec:local-delta-A_n:A_1-finite-part-1} and Cases~\ref{item:A_1-case-1}, \ref{item:A_1-case-2-notin-lF}, and~\ref{item:A_1-case-2-in-lF-reduced} of Section~\ref{sec:local-delta-A_n:A_1-finite-part-2} apply for any \(A_n\) singularity, so it suffices to address the point \(q_1\).  We use the exceptional divisor from the blow-up of $q_1$ to show K-stability in this range. Recall that in the previous section we used \(e\) to denote the exceptional divisor of \(\hY\to Y\) and \(f\) to denote a ruling of \(\bb F_1\cong\hY\); we now choose \(f\) so that its image on \(Y\) is the line through \(y_1\) and \(q_1\). Let $\oY$ be the blow-up of $\hY$ at $q_1$, and let $\oC$ be the exceptional divisor. By abuse of notation, we will again use \(e\) and \(f\) to denote their strict transforms on \(\oY\).

Then, we compute as above.  We use the formulas for $P(u)$ from the earlier computations and restrict to $\oY$.
\begin{detail}
Note that in this case we have 
\[ F_q(W_{\bullet \bullet \bullet})  = \frac{6}{\vol L} 
\int_0^{5-4c} \left( \int_{0}^{t(u)} (P(u,v)\cdot \oC)\ord_q((N/(u)|_\oY + N(u,v)|_\oC) dv \right) du. \]
\end{detail}
From Section~\ref{sec:local-delta-A_n:A_1-finite-part-2}, we have:
\[ \resizebox{1\textwidth}{!}{ $ \begin{array}{rll}
    0 \le u \le 2-2c: & P_{\regionA}(u)|_\hY = ue + uf & N_{\regionA}(u)|_\hY  =  0 \\
    2-2c \le u \le 3-2c: & P_{\regionB}(u)|_\hY = (2-2c)e + uf & N_{\regionB}(u)|_\hY = (u+2c - 2) e \\
    3-2c \le u \le 5-4c: & P_{\regionC}(u)|_\hY = (5-4c-u)e + (3-2c)f & N_{\regionC}(u)|_\hY  =(u+2c-2)e + (u+2c-3) l_\tF.
\end{array} $} \]
Note that \(q_1 \notin e\). The assumption that the fiber of \(\piR_1\) is reduced at \(\pt\) implies \(q_1 \notin l_\tF\), so \(N(u)=N'(u)\) on all regions (see Remark~\ref{rem:higher-A_n-non-reduced-fiber-first-region} below). As mentioned in Remark~\ref{rem:A_1-finite-fiber-y_1} above, the assumption that the fiber of \(\piR_2\) is finite at \(\pt\) implies \(y_1 \notin R_Y\).

Starting in region $\regionA$ for $0 \le u \le 2-2c$, $P_{\regionA}(u)|_\oY - v\oC = ue + u f + (u-v) \oC$.  Here, $t(u) = u$, and we find the Zariski decomposition for \(0 \le v \le u\) is
\[ P_{\regionA}(u,v) = u(e+f) + (u - v) \oC , \qquad N_{\regionA}(u,v) = 0.\]
\begin{detail}
On \(\oY\) we have \(e^2 = -1\), \(f^2 = -1\), \(\oC^2 = -1\), \(e\cdot f = 1\), \(e\cdot\oC = 0\), \(f\cdot\oC = 1\).
So the relevant quantities for $0 \le u \le 2-2c$ and \(q\in \oC\) are: 
\[\begin{array}{r|l}
     & 0 \le v \le u \\
     \hline
    P_{\regionA}(u,v)^2 & u^2 - v^2 \\
    P_{\regionA}(u,v) \cdot \oC & v \\
    (P_{\regionA}(u,v) \cdot \oC)^2 & v^2 \\
    \ord_q((N_{\regionA}'(u)|_\oY + N_{\regionA}(u,v) )|_\oC) &  0 
\end{array}\]
\end{detail}
For region $\regionB$ for $2 -2c \le u \le 3-2c$, $P_{\regionB}(u)|_\oY - v\oC = (2-2c)e + uf +(u-v)\oC$ and again $t(u) = u$.  Here, the Zariski decomposition is
\begin{equation*}\begin{split}
    P_{\regionB}(u,v) &= \begin{cases}
        (2-2c)e + u f + (u-v)\oC & 0 \le v \le 2-2c  \\
        (2-2c)(e+f) + (u-v)(f+\oC) & 2-2c \le v \le u ,
    \end{cases} \\
    N_{\regionB}(u,v) &= \begin{cases}
        0 \qquad \hphantom{((2-2c)e + u f + (u-v)\oC} & 0 \le v \le 2-2c  \\
        (v-(2-2c))f & 2-2c \le v \le u .
    \end{cases}
\end{split} \end{equation*}
\begin{detail}
So the relevant quantities for $2-2c \le u \le 3-2c$ are: 
\[\begin{array}{r|ll}
   & 0 \le v \le 2-2c & 2-2c \le v \le u \\
       \hline
   (P_{\regionB}(u,v))^2   & -(2-2c-u)^2+(u-v)(u+v) & 2(2-2c)(u-v) \\
   P_{\regionB}(u,v)\cdot \oC  & v & 2-2c \\
   (P_{\regionB}(u,v)\cdot \oC)^2  & v^2 & (2-2c)^2 \\
   \ord_q((N_{\regionB}'(u)|_\oY + N_{\regionB}(u,v))|_\oC) & 0 &
   \left\{ \begin{array}{ll}
      0  & q \notin f  \\
     v-2+2c   & q = \oC \cap f
   \end{array} \right.  \\
\end{array}\]
\end{detail}
For region $\regionC$ with $3-2c \le u \le 5-4c$, $P_{\regionC}(u)|_\oY - v\oC = (5-4c-u)e + (3-2c)f + (3-2c-v) \oC$ and $t(u) = 3-2c$.  
Here, the Zariski decomposition is 
\begin{equation*}\begin{split}
    P_{\regionC}(u,v) &= \begin{cases}
        (5-4c-u)e + (3-2c)f + (3-2c-v)\oC & 0 \le v \le 5-4c-u  \\
        (5-4c-u)(e+f) + (3-2c-v)(f+\oC) & 5-4c-u \le v \le 3-2c ,
    \end{cases} \\
    N_{\regionC}(u,v) &= \begin{cases}
        0 & 0 \le v \le 5-4c-u \\
        (v-(5-4c-u))f \hphantom{(e+f) + (3-2c-2\oC} & 5- 4c- u \le v \le 3-2c .
    \end{cases}
\end{split} \end{equation*}
\begin{detail}
So the relevant quantities for $3-2c \le u \le 5-4c$ are: 
\[ \resizebox{1\textwidth}{!}{ $\begin{array}{r|ll}
   & 0 \le v \le 5-4c-u & 5-4c-u \le v \le 3-2c \\
       \hline
   (P_{\regionC}(u,v))^2   & -(2-2c-u)^2 +(3-2c-v)(3-2c+v) & 2(5-4c-u)(3-2c-v) \\
   P_{\regionC}(u,v)\cdot \oC  & v & 5-4c-u \\
   (P_{\regionC}(u,v)\cdot \oC)^2  & v^2 & (5-4c-u)^2 \\
   \ord_q((N_{\regionC}'(u)|_\oY + N_{\regionC}(u,v))|_\oC) & 0 & \left\{ \begin{array}{ll}
      0  & q \notin f  \\
      v-(5-4c-u)   & q = \oC \cap f \\
   \end{array} \right. 
\end{array}$ } \]
\end{detail}
As $\oC$ was the exceptional divisor of the blow-up of the intersection point of the components of $R_Y$, we have $A_{Y, \dee_Y}(\oC) = 2-2c$. Using the above and \cite[Theorem 4.8]{Fujita-3.11}, we compute \(S_{Y,\dee_Y; V_{\bullet \bullet}}(\oC)\):
\[ S_{Y,\dee_Y; V_{\bullet \bullet}}(\oC) = \frac{-(14c^3-62c^2+91c-44)}{3(3-2c)^2}, \qquad \frac{A_{Y,\dee_Y} (\oC)}{S_{Y,\dee_Y; V_{\bullet \bullet}}(\oC)} >1 \] for all $c \in (0,\cnot)$.

Next, we compute $S_{\oC,\dee_\oC; W_{\bullet \bullet \bullet}}(q)$ for $q \in \oC$. We have two cases to consider: \(q\notin f\) (noting that this applies to the points in $\dee_\oC$), and \(q\in f\) (in this case \(q\not\in\dee_\oC\)). Using \cite[Theorem 4.17]{Fujita-3.11},
\[ A_{\oC,\dee_\oC}(q) = \begin{cases}
    1 \text{ or }1-c & q\notin f \\
    1 & q = \oC \cap f ,
\end{cases} \qquad
S_{\oC,\dee_\oC; W_{\bullet \bullet \bullet}}(q) = \begin{cases}
    \displaystyle \frac{(1-c)(6c^2-20c+17)}{3(3-2c)^2} & q\notin f \\
    \displaystyle \frac{3-2c}{3} & q = \oC \cap f ,
\end{cases}\]
so in either case, we get
\(A_{\oC, \dee_\oC}(q)/S_{\oC,\dee_\oC; W_{\bullet \bullet \bullet}}(q) >1 \)
for all \(c\in(0,1)\). This proves that $\delta_{\pt}(\exx,\dee) > 1$ for $c \in (0,\cnot)$.

\begin{remark}\label{rem:higher-A_n-non-reduced-fiber-first-region}
    The above computation does not work if the fiber of \(\piR_1\colon R\to\bb P^1\) is non-reduced at \(\pt\). In this case \(q_1 \in l_\tF\), so \(d(u)\neq 0\) on region \(\regionC\) and the formula for \(S_{Y,\dee_Y; V_{\bullet \bullet}}(\oC)\) has an additional term. Then \(S_{Y,\dee_Y; V_{\bullet \bullet}}(\oC) = (6 - 5 c)/3\), and \(A_{Y, \dee_Y}(\oC) = 2-2c\) still, so \(A_{Y, \dee_Y}(\oC) / S_{Y,\dee_Y; V_{\bullet \bullet}}(\oC) < 1\) for \(c \in (0,1)\).
\end{remark}

\subsection{\texorpdfstring{\(A_1\)}{A1} singularities on non-finite fibers of \texorpdfstring{\(\piR_2\)}{pi2}}\label{sec:local-delta-A_n:A_1-infinite}

First, note that \cite[Lemma 2.4]{CFKP23} proves that, as long as $R$ is smooth along the non-finite fiber, then $\delta_\pt(\exx,\dee) > 1$ for all points in the fiber.  Their computation does not extend to the singular case in general.  In what follows, we prove that surfaces with one $A_1$ (resp. two) singularity along a non-finite fiber are K-stable (resp. strictly K-semistable) in Section~\ref{sec:local-delta-A_n:A_1-infinite-onesingularity} (resp. Lemma~\ref{lem:A_1-infinite-twosingularities-Ksemistable}). In particular, we prove Theorem~\ref{thm:local-delta-A_n-singularities}\eqref{part:local-delta-A_n-singularities:A_1-infinite} and~\eqref{part:local-delta-A_n-singularities:A_1-infinite-twosingularities}.

The idea of the argument is as follows. First, in Lemma~\ref{A_1-infinite-local-delta-part1}, we show that if \(\pt\) is an \(A_1\) singularity on a non-finite fiber \(f\) of \(\piR\), then \(\delta_\pt(\exx,\dee)\geq 1\) for all \(c\in (0,\frac{1}{2}]\). Next, in Lemma~\ref{lem:A_1-infinite-onesingularity-Kstable} we show that if \(\pt\) is the only singularity of \(R\) in \(f\), then in fact \(\delta_\pt(\exx,\dee)> 1\) for \(c\in(\cone\approx 0.33,\frac{1}{2}]\). Therefore, if \(R\) only has singularities of the form described in Theorem~\ref{thm:local-delta-A_n-singularities}, we know that \(\delta_{\pt'}(\exx,cR)>1\) for all \(\pt'\in \exx\) and \(c\in(\cone\approx 0.33,\frac{1}{2}]\). Hence, \((\exx,cR)\) is K-stable for \(c\in(\cone\approx 0.33,\frac{1}{2}]\). This will then imply that the pair is K-stable for $c \in (0, \frac{1}{2}$.

\begin{lemma}\label{A_1-infinite-local-delta-part1}
    Suppose $R$ has an $A_1$ singularity at a point $\pt$ along a non-finite fiber of the second projection.  Then, $\delta_\pt (\exx, \dee) \ge 1$ for all $c \in (0,\frac{1}{2}]$.  
\end{lemma}

\begin{proof}
We modify the computation from Section~\ref{sec:local-delta-A_n:A_1-finite-part-2} to show that surfaces with $A_1$ singularities at points $\pt$ along non-finite fibers have $\delta_\pt(\exx,\dee) \ge 1$ for all $c \in (0,\frac{1}{2}]$.  This computation is similar to that in Section~\ref{sec:local-delta-A_n:A_1-finite-part-2} and can be thought of as a continuation of Cases \ref{item:A_1-case-1}, \ref{item:A_1-case-2}, and~\ref{item:A_1-case-3}.  In that notation, these are:

\begin{enumerate}
    \setcounter{enumi}{3}
    \item \label{item:A_1-case-4} Case 4: $q = y_1 \in R_Y$.  Geometrically, these correspond to surfaces $R$ containing the fiber of the second projection.  
\end{enumerate}

Here, we consider the $(2,1)$ weighted blow-up $\nu\colon Y' \to Y$ of $y_1$ separating $R_Y$ and its tangent line $f$ at $y_1$, with $z'$ the exceptional divisor. Recall from Section~\ref{sec:local-delta-A_n:A_1-finite-part-2} that \(\hY \to Y\) is the blow-up of \(y_1\) with exceptional divisor \(e\). We construct a common resolution of $Y'$ and $\hY$ by $\gamma\colon \oY \to Y'$ (resp. $\theta\colon \oY \to \hY$) by resolving the singularity on $Y'$ (resp. blowing up the intersection point of $R_Y$ and $e$ on $\hY$, which is also the intersection point \(e \cap f\)).  Let $z \subset \oY$ be the strict transform of $z'$. Note that \(z\) is also the exceptional divisor of \(\oY \to \hY\), and that \(\gamma^* z' = z + \frac{1}{2}e\).

\begin{detail}
We again use subscripts \(\regionA,\regionB,\regionC\) to denote the Zariski decompositions in the regions from Section~\ref{sec:local-delta-A_n:A_1-finite}. On each region, let \(P(u)|_\oY = \theta^* (P(u)|_\hY)\) and \(N(u)|_\oY = \theta^* (N(u)|_\hY)\) be the pullbacks, let \(d(u)=\ord_z (\theta^*(N(u)|_\hY))\), and let \(N'(u)|_\oY = N'(u)|_\oY - d(u) z\).
For each $u$, let $t(u)$ be the pseudoeffective threshold of $P(u)|_\oY - vz$, and let $P(u,v) + N(u,v)$ be its Zariski decomposition.  Then, by \cite[Theorem 4.8, Notation 4.11, and Theorem 4.17]{Fujita-3.11}, we have 
\[ \delta_q(Y, \dee_Y; V_{\bullet \bullet}) \ge \min \left\{ \frac{A_{Y, \dee_Y}(z)}{S_{Y, \dee_Y; V_{\bullet \bullet}}(z)}, \inf_{
     v \in z :
     \beta(c) = q } \delta_{v}(z,\dee_z; W_{\bullet \bullet \bullet}) \right\} \]
where 
\[ S_{Y, \dee_Y; V_{\bullet \bullet}}(z) = \frac{3}{\vol L} 
\int_0^{5-4c} \left( (P(u)|_{\oY})^2 \ord_z N(u)|_{\oY} + \int_0^\infty \vol( P(u)|_\oY - vz) dv \right) du \]
where $\ord_z N(u)|_{\oY} = 0$ in region $\regionA$ and $u+2c-2$ in regions $\regionB$ and $\regionC$,
and
\[ S_{z,\dee_z; W_{\bullet \bullet \bullet}}(q) = \frac{3}{\vol L} 
\int_0^{5-4c} \left( \int_{0}^{t(u)} (P(u,v)\cdot z)^2 dv \right) du + F_v(W_{\bullet \bullet \bullet}) \]
where 
\[ F_v(W_{\bullet \bullet \bullet})  = \frac{6}{\vol L} 
\int_0^{5-4c} \left( \int_{0}^{t(u)} (P(u,v)\cdot \hC)\ord_q((N'(u)|_\oY + N(u,v) - (v+d(u))\frac{e}{2})|_z) dv \right) du. \]
\end{detail}
By abuse of notation, we continue to write $e$ for its strict transforms on $\oY$, and we write $f$, $l_\oF$ for the strict transforms of $f$ (the fiber of \(\hY \cong \bb F_1\) passing through \(R_\hY \cap e\)), $l_\tF$. First, using the formulas computed in Section~\ref{sec:local-delta-A_n:A_1-finite-part-2} for Case~\ref{item:A_1-case-1}, we have:
\[ \resizebox{1\textwidth}{!}{ $ \begin{array}{rlll}
    0 \le u \le 2-2c: & P_{\regionA}(u)|_\oY = ue + uf + 2uz & N_{\regionA}'(u)|_\oY  =  0  & d(u)= 0\\
    2-2c \le u \le 3-2c: & P_{\regionB}(u)|_\oY = (2-2c)e + uf + (2-2c+u)z & N_{\regionB}'(u)|_\oY = (u+2c - 2) e & d(u) = u+2c-2\\
    3-2c \le u \le 5-4c: & P_{\regionC}(u)|_\oY = (5-4c-u)e + (3-2c)f + (8-6c-u)z  & N_{\regionC}'(u)|_\oY  =(u+2c-2)e + (u+2c-3) l_\oF & d(u) = u+2c-2 .
\end{array} $} \]
Note that on \(\oY\), the intersection theory is given by $e^2 = -2$, $f^2 = z^2 = -1$, $e\cdot f = 0$, and $e \cdot z = f \cdot z = 1$.
On each region for $u$, we compute the Zariski decomposition of $P(u)|_\oY - vz = P(u,v) + N(u,v)$ as follows: 

For region $\regionA$ for $0 \le u \le 2-2c$, $P_{\regionA}(u)|_\oY - vz = ue + uf + (2u-v)z$ and $t(u) = 2u$.  Here, the Zariski decomposition is 
\begin{equation*}\begin{split}
    P_{\regionA}(u,v) &= \begin{cases}
        (2u-v)(\frac{1}{2}e +z) + u f & 0 \le v \le u  \\
        (2u-v)(\frac{1}{2}e + f + z) & u \le v \le 2u ,
    \end{cases} \\
    N_{\regionA}(u,v) &= \begin{cases}
        \frac{1}{2}ve \hphantom{(2u-v)(+z) + u f} & 0 \le v \le u   \\
        \frac{1}{2}ve + (v-u)f & u \le v \le 2u .
    \end{cases}
\end{split}\end{equation*}
\begin{detail}
So the relevant quantities for $0 \le u \le 2-2c$ and \(\qbarpt\in z\) are:
\[\begin{array}{r|ll}
  & 0 \le v \le u & u \le v \le 2u \\
       \hline
   (P_{\regionA}(u,v))^2   & u^2 - \frac{v^2}{2} & \frac{(2u-v)^2}{2} \\
   P_{\regionA}(u,v)\cdot z  & \frac{v}{2} & \frac{2u-v}{2} \\
   (P_{\regionA}(u,v)\cdot z)^2  & \frac{v^2}{4} & \frac{(2u-v)^2}{4} \\
   \ord_{\qbarpt}((N_{\regionA}'(u)|_\oY + N_{\regionA}(u,v) - (v+d(u))\frac{e}{2})|_z) & 0 &
   \left\{ \begin{array}{ll}
      0  & \qbarpt \ne z \cap f  \\
     v-u   & \qbarpt = z \cap f
   \end{array} \right. \\
\end{array}\]
\end{detail}
For region $\regionB$ for $2-2c \le u \le 3-2c$, $P_{\regionB}(u)|_\oY - vz =(2-2c)e + uf + (2-2c+u-v)z$ and $t(u) = 2-2c+u$.  We have \(0 \leq u-2+2c \leq 2-2c\) since we've assumed $c \le \frac{1}{2}$, so the Zariski decomposition here is 
\begin{equation*}\begin{split}
    P_{\regionB}(u,v) &= \begin{cases}
        (2-2c)e + uf + (2-2c+u-v)z  & 0 \le v \le u-2+2c  \\
        (2-2c+u-v)(\frac{1}{2}e + z) + uf & u-2+2c \le v \le 2-2c \\
        (2-2c+u-v)(\frac{1}{2}e+f+ z) & 2-2c \le v \le u+2-2c ,
    \end{cases} \\
    N_{\regionB}(u,v) &= \begin{cases}
        0 & 0 \le v \le u-2+2c   \\
        \frac{1}{2}(v-u+2-2c)e  & u-2+2c \le v \le 2-2c \\
        \frac{1}{2}(v-u+2-2c)e + (v-2+2c)f & 2-2c \le v \le u+2-2c .
    \end{cases}
\end{split}\end{equation*}
\begin{detail}
So the relevant quantities for $2-2c \le u \le 3-2c$ and \(\qbarpt\in z\) are:
\[ \resizebox{1\textwidth}{!}{ $ \begin{array}{r|lll}
   & 0 \le v \le u-2+2c & u-2+2c \le v \le 2-2c  & 2-2c \le v \le u+2-2c \\
       \hline
   (P_{\regionB}(u,v))^2  & 2(2-2c)(u-v) - (v-2+2c)^2 & u(2-2c-v)+\frac{1}{2}(u^2-(2-2c-v)^2)  &  \frac{(2-2c+u-v)^2}{2} \\
   P_{\regionB}(u,v)\cdot z & v & \frac{v+u-2+2c}{2} & \frac{2-2c+u-v}{2} \\
   (P_{\regionB}(u,v)\cdot z)^2 & v^2 & \frac{(v+u-2+2c)^2}{4} & \frac{(2-2c+u-v)^2}{4} \\
   \ord_{\qbarpt}((N_{\regionB}'(u)|_\oY + N_{\regionB}(u,v) - (v+d(u))\frac{e}{2})|_z) & \left\{ \begin{array}{ll}
      0  & \qbarpt \ne z \cap e  \\
     \frac{1}{2}(u-v-2+2c)   & \qbarpt = z \cap e
   \end{array} \right. & 0 &
   \left\{ \begin{array}{ll}
      0  & \qbarpt \ne z \cap f  \\
     v-2+2c   & \qbarpt = z \cap f
   \end{array} \right. \\
\end{array} $} \]
\end{detail}
Finally, for region $\regionC$ on the interval $3-2c \le u \le 5-4c$, $P_{\regionC}(u)|_\oY - vz = (5-4c-u)e + (3-2c)f + (8-6c-u-v)z $ and $t(u) = 8-6c-u$.  Since \(c\leq \frac{1}{2}\), we have \(3-2c \leq 7/2 - 3c\). We have two subregions \(\regionC_1,\regionC_2\) to consider for the Zariski decomposition. First, for $3-2c \le u \le 7/2 -3c$ we have \(u-2+2c \leq 5-4c-u\), so the Zariski decomposition is
\begin{equation*}\begin{split}
    P_{\regionC_1}(u,v) &= \begin{cases}
        (5-4c-u)e + (3-2c)f + (8-6c-u-v)z & 0 \le v \le u-2+2c  \\
        (8-6c-u-v)(\frac{1}{2}e+z) + (3-2c)f & u-2+2c \le v \le 5-4c-u \\
        (8-6c-u-v)(\frac{1}{2}e+f+z)  & 5-4c-u \le v \le 8-6c-u ,
    \end{cases} \\
    N_{\regionC_1}(u,v) &= \begin{cases}
        0 \hphantom{(5-c-u)e + (3-2c)f + (8-6c-u-v)z} & 0 \le v \le u-2+2c    \\
        \frac{1}{2}(v-u+2-2c)e &  u-2+2c \le v \le 5-4c-u  \\
        \frac{1}{2}(v-u+2-2c)e + (v-5+4c+u)f & 5-4c-u \le v \le 8-6c-u .
    \end{cases}
\end{split}\end{equation*}
\begin{detail}
So the relevant quantities for $3-2c \le u \le 7/2-3c$ and \(\qbarpt\in z\) are:
\[ \resizebox{1\textwidth}{!}{ $ \begin{array}{r|lll}
   & 0 \le v \le u-2+2c   & u-2+2c \le v \le 5-4c-u  &  5-4c-u \le v \le 8-6c-u \\
       \hline
   (P_{\regionC_1}(u,v))^2  & 2(5-4c-u)(3-2c-v)-(v-5+4c+u)^2 & (3-2c)(5-4c-u-v)+\frac{1}{2}((3-2c)^2-(5-4c-u-v)^2) &  \frac{(8-6c-u-v)^2}{2} \\
   P_{\regionC_1}(u,v)\cdot z & v & \frac{u+v-2+2c}{2} & \frac{8-6c-u-v}{2} \\
   (P_{\regionC_1}(u,v)\cdot z)^2 & v^2 & \frac{(u+v-2+2c)^2}{4} & \frac{(8-6c-u-v)^2}{4} \\
   \ord_{\qbarpt}((N_{\regionC_1}'(u)|_\oY + N_{\regionC_1}(u,v) - (v+d(u))\frac{e}{2})|_z)  & \left\{ \begin{array}{ll}
      0  & \qbarpt \ne z \cap e  \\
     \frac{1}{2}(u-v+2c-2)   & \qbarpt = z \cap e
   \end{array} \right. & 0 &
   \left\{ \begin{array}{ll}
      0  & \qbarpt \ne z \cap f  \\
     v-5+4c+u & \qbarpt = z \cap f 
   \end{array} \right. \\
\end{array} $} \]
\end{detail}
The Zariski decomposition on the region $ 7/2 -3c \le u \le 5-4c$ is
\begin{equation*}\begin{split}
    P_{\regionC_2}(u,v) &= \begin{cases}
        (5-4c-u)e + (3-2c)f + (8-6c-u-v)z & 0 \le v \le 5-4c-u \\
        (5-4c-u)e + (8-6c-u-v)(f+ z) & 5-4c-u \le v \le u-2+2c  \\
        (8-6c-u-v)(\frac{1}{2}e+f+z)  & u-2+2c \le v \le 8-6c-u ,
    \end{cases} \\
    N_{\regionC_2}(u,v) &= \begin{cases}
        0 \hphantom{(5-4c-u)e + (3-c)f + (8-6c-u-v)z} & 0 \le v \le 5-4c-u    \\
        (v-5+4c+u)f &   5-4c-u \le v \le u-2+2c  \\
        \frac{1}{2}(v-u+2-2c)e + (v-5+4c+u)f & u-2+2c \le v \le 8-6c-u .
    \end{cases}
\end{split}\end{equation*}
\begin{detail}
So the relevant quantities for $7/2-3c \le u \le 5-4c$ are:
\[ \resizebox{1\textwidth}{!}{ $ \begin{array}{r|lll}
   & 0 \le v \le 5-4c-u   &  5-4c-u \le v \le u-2+2c  & u-2+2c \le v \le 8-6c-u \\
       \hline
   (P_{\regionC_2}(u,v))^2  & 2(5-4c-u)(3-2c-v)-(v-5+4c+u)^2 & 2(5-4c-u)(3-2c-v) &  \frac{(8-6c-u-v)^2}{2} \\
   P_{\regionC_2}(u,v)\cdot z & v & 5-4c-u & \frac{8-6c-u-v}{2} \\
   (P_{\regionC_2}(u,v)\cdot z)^2 & v^2 & (5-4c-u)^2 & \frac{(8-6c-u-v)^2}{4} \\
   \ord_q((N_{\regionC_2}'(u)|_\oY + N_{\regionC_2}(u,v) - (v+d(u))\frac{e}{2})|_z)  & \left\{ \begin{array}{ll}
      0  & q \ne z \cap e  \\
     \frac{1}{2}(u-v+2c-2)   & q = z \cap e
   \end{array} \right. & \left\{ \begin{array}{ll}
      0  & q \ne z \cap f, z \cap e  \\
     v-5+4c+u & q = z \cap f \\
     \frac{1}{2}(u-v+2c-2)   & q = z \cap e 
   \end{array} \right.  &
   \left\{ \begin{array}{ll}
      0  & q \ne z \cap f  \\
     v-5+4c+u & q = z \cap f 
   \end{array} \right. \\
\end{array} $} \]
\end{detail}
Using \cite[Theorem 4.8]{Fujita-3.11}, we then compute that \(S_{Y, \dee_Y; V_{\bullet \bullet}}(z) =  3-2c \).
\begin{detail}
Indeed,
\begin{equation*}\begin{split}
    S_{Y, \dee_Y; V_{\bullet \bullet}}(z) &= \frac{3}{\vol L} \left( \int_{2-2c}^{5-4c} (P(u)|_{\oY})^2 \cdot (u+2c-2) du + \int_0^{5-4c} \int_0^{t(u)} P(u,v)^2 dv \; du \right) \\
    &= \frac{(2-2c)(3-2c)^3}{(2-2c)(3-2c)^2} = 3-2c.
\end{split}\end{equation*}
\end{detail}
Since $A_{Y, \dee_Y}(z) = 3-2c$, then for \(c\in (0,\frac{1}{2}]\) we have
\[\frac{A_{Y, \dee_Y}(z)}{S_{Y, \dee_Y; V_{\bullet \bullet}}(z)} =  \frac{3-2c}{3-2c} = 1. \]

Finally, we compute $\delta_{z,\dee_z; W_{\bullet \bullet \bullet}}(\qbarpt)$ for $\qbarpt \in z$. First, for the log discrepancies, if \(q\notin f, e\) we have that $A_{z,\dee_z}(\qbarpt)$ is either $1$ or $1-c$, depending on whether or not $\qbarpt \in R_{\oY}$; if \(\qbarpt = z\cap f\) we have \(A_{z,\dee_z}(\qbarpt) = 1\) as \(R_\oY\) and \(f\) were separated by the weighted blow-up; and if \(\qbarpt = z \cap e\), this is the preimage of the singular point of \(Y'\) and does not meet \(R_\oY\), so \(A_{z,\dee_z}(\qbarpt) = \frac{1}{2}\). Thus, using the above and \cite[Theorem 4.17]{Fujita-3.11} to compute \(S_{z,\dee_z; W_{\bullet \bullet \bullet}}(\qbarpt)\), we find that
\[ A_{z,\dee_z}(\qbarpt) = \begin{cases}
    1 \text{ or }1-c & \qbarpt\notin f, e \\
    1 & \qbarpt = z \cap f \\
    1/2 & \qbarpt = z \cap e ,
\end{cases} \quad
S_{z,\dee_z; W_{\bullet \bullet \bullet}}(\qbarpt) = \begin{cases}
    \displaystyle \frac{(13 - 14 c) (3 - 2 c)}{48 (2 - 2c)} & \qbarpt\notin f, e \\
    \displaystyle \frac{3-2c}{3} & \qbarpt = z \cap f \\
    \displaystyle \frac{3-2c}{6} & \qbarpt = z \cap e
\end{cases}\]
so for all \(\qbarpt \in z\) we have \(A_{z,\dee_z}(\qbarpt)/S_{z,\dee_z; W_{\bullet \bullet \bullet}}(\qbarpt) > 1\) for \(c \in (0, \frac{1}{2}]\).
This completes the proof of Lemma~\ref{A_1-infinite-local-delta-part1}.
\end{proof}

\subsubsection{Only one \(A_1\) singularity along a non-finite fiber of \(\piR\)}\label{sec:local-delta-A_n:A_1-infinite-onesingularity}

If there is only one $A_1$ singularity along a non-finite fiber, and $\pt$ is any smooth point of the fiber, then the computation in \cite[Lemma 2.4]{CFKP23} applies directly to show that $\delta_\pt (\exx,\dee) > 1$, so by Lemma \ref{A_1-infinite-local-delta-part1} we conclude the following. Note that their argument does not apply if there are two \(A_1\) singularities on the non-finite fiber because in this case (in their notation), $R_E$ contains $\mathbf{s}$ as a component.  In fact, in the case of two $A_1$ singularities along a non-finite fiber, we show  \(\delta_\pt (\exx, \dee) = 1\) for all $\pt$ in the fiber in Lemma~\ref{lem:A_1-infinite-twosingularities-Ksemistable}.

\begin{corollary}\label{cor:A_1-infinite-onesingularity-Ksemistable}
    If $R$ has at most one $A_1$ singularity along any non-finite fiber of the second projection, then for all $\pt$ in this fiber, $\delta_\pt (\exx, \dee) \ge 1$.  In particular, if we know that $\delta_{\pt'} (\exx, \dee) \ge 1$ for all other \(\pt'\in R\) (e.g. if $R$ is smooth away from this fiber), then $(\exx, \dee)$ is K-semistable. 
\end{corollary}

In fact, even for $\pt$ the singular point of the non-finite fiber, we can show that $\delta_\pt (\exx, \dee) > 1$ and conclude that these pairs are K-stable:

\begin{proposition}\label{prop:A_1-infinite-onesingularity-Kstable}
    Let \(\cone\approx 0.3293\) be the irrational number defined in Theorem~\ref{thm:local-delta-A_n-singularities}. Suppose that $R$ has at most one $A_1$ singularity \(\pt\) along any non-finite fiber of the second projection, and that \(\delta_q(\exx, cR) \geq 1\) for all other points \(q\neq p\in R\). If there exists some \(c'\in(\cone,\frac{1}{2}] \cap \bb Q\) such that all other \(q\neq\pt\in R\) satisfy $\delta_q(\exx, c' R) > 1$, then $(\exx, cR)$ is K-stable for all \(c\in (0,\frac{1}{2}] \cap \bb Q\).
\end{proposition}

\begin{proof}
    Corollary \ref{cor:A_1-infinite-onesingularity-Ksemistable} and the assumptions imply that $\delta(\exx, cR) \ge 1$ for all $c \in (0, \frac{1}{2}]$, so \((X,cR)\) is K-semistable for all \(c\in(0,\frac{1}{2}]\cap\bb Q\) by Theorem~\ref{thm:valuative-criterion-K-stability}. Furthermore, the following Lemma~\ref{lem:A_1-infinite-onesingularity-Kstable} shows that $\delta_\pt(\exx, cR) > 1$ for all $c \in (\cone, \frac{1}{2}]$, so we have that $(\exx, c' R)$ is K-stable by Theorem~\ref{thm:valuative-criterion-K-stability}.   If the pair was strictly semistable for $c \le \cone$, there must exist a strictly polystable degeneration of $(\exx,\dee)$ with non-finite fiber of the second projection with K-semistable threshold less that $\cone$ as $(\exx, \dee)$ admits no strictly polystable degeneration for $c > c'$.  However, by Lemmas~\ref{lem:frakG1polystable} and~\ref{lem:frakG2polystable}, all such strictly polystable pairs have K-semistable threshold greater than $\frac{1}{2}$, so no such polystable degeneration exists. 
\end{proof}

\begin{lemma}\label{lem:A_1-infinite-onesingularity-Kstable}
    Suppose $R$ has at most one $A_1$ singularity at $\pt$ along any non-finite fiber of the second projection \(\piR_2\colon R \to \bb P^2\).  Then, $\delta_\pt(\exx, \dee) > 1$ for all $c \in (\cone, \frac{1}{2}]$.
\end{lemma}

\begin{proof}
We will use the Abban--Zhuang method. First, we construct a plt blow-up of \(\exx\). Let \(X_1 \to \exx\) be the blow-up of the fiber of \(\piR\) containing \(\pt\), and let \(E \cong \bb P^1\times\bb P^1\) be the exceptional divisor. Let \(R_1\) denote the intersection of \(E\) with the strict transform of \(R\) on \(X_1\). Then \(R_1\) has two irreducible components \(f_1 \in |\cal O_E(1,0)|\) and \(\sigma_{1,1}\in|\cal O_E(1,1)|\). Next, let \(X_2 \to X_1\) be the blow-up of \(\sigma_{1,1}\) with exceptional divisor \(Y_2 \cong \bb F_3\). Since \(K_{X_2}\) intersects the strict transform of \(f_1\) trivially, we may contract the strict transform of \(E\) in \(X_2\) to a threefold \(\phi\colon \tX \to X\) with \(\tX\) singular along the image of \(E\). Let \(Y\) be the strict transform of \(Y_2\) in \(\tX\). Let \(\Sigma_1 \subset X_1\) be the strict transform of the unique \((1,1)\)-surface in \(\exx\) such that \(\Sigma_1 \cap E = \sigma_{1,1}\), and let \(\Sigma\cong\bb F_1\) denote the strict transform of \(\Sigma_1\) on \(\tX\). Let \(\sigma_Y, f_Y\) denote the negative section and a fiber of \(Y \cong \bb F_3\), respectively, and let \(\sigma_\Sigma, f_\Sigma\) denote the negative section and ruling of \(\Sigma\cong \bb F_1\), respectively. Note that \(\phi_*\sigma_Y \subset \exx\) is the non-finite fiber of \(\piR\) containing \(\pt\).

We will use the plt type divisor \(Y \subset \tX\) to bound the \(\delta\)-invariant. Let \(L = \phi^*(-K_\exx - \dee)\). To find \(S_{\exx,\dee}(Y)\), we need to compute \(\vol(L - u Y)\), so we will first compute the Nakayama--Zariski decomposition of \(L - u Y\). Using that $\tX$ is a (singular) del Pezzo fibration over $\bP^1$, one can show that the Mori cone of \(\tX\) is generated by \(\sigma_Y, f_Y, f_\Sigma\), and one can verify that the pseudoeffective threshold of \(L - u Y\) is \(u = 5-4c\). By computation, the Nakayama--Zariski decomposition will correspond to the following birational transformations (see Figure~\ref{fig:A_1-non-finite-Zariski}): before $u = 3-2c$, the divisor is ample; at $u = 3-2c$, we contract the ruling \(f_\Sigma\) of \(\Sigma\) in \(\pi\colon \tX \to \oX\); at $u = 4-4c$, we flip $\sigma_Y$ in \(\oX \dashrightarrow \oX^+\); and we stop at $u = 5-4c$. We will use subscripts \(\regionA,\regionB,\regionC\) to denote values of \(u\) in these three respective regions. On \(\regionA,\regionB\) we have
\[ \resizebox{1\textwidth}{!}{ $\begin{array}{rll}
    0 \le u \le 3-2c: &\quad P_{\regionA}(u) = \phi^*(-K_\exx - \dee) - u Y &\quad N_{\regionA}(u) = 0 \\
    3-2c \le u \le 4-4c: &\quad P_{\regionB}(u) = \phi^*(-K_\exx - \dee) - u Y - (u-3+2c)\Sigma &\quad N_{\regionB}(u) = (u-3+2c)\Sigma. 
\end{array} $}\]
\noindent We discuss region $\regionC$ below. 

We summarize the construction and birational transformations in the following diagram: 
\[\begin{tikzcd}
X_1 \ar[d, "\text{blow up fiber containing }\pt" {swap}] & & X_2 \ar[ll, "\text{blow up }\sigma_{1,1}" {swap}] \ar[d, "\text{contract } E" {swap}] &  & & \hX \arrow[ld] \arrow[rd] & & & & \\
\exx = \bb P^1\times\bb P^2 & & \tX \arrow[ll, "\phi" {swap}, "\text{extract }Y"] \arrow[rr,  "\text{contract }f_\Sigma" {swap}] & & \oX \arrow[rr, dashed, "\text{flip }\sigma_Y" {swap}] & & \oX^+
\end{tikzcd}\]

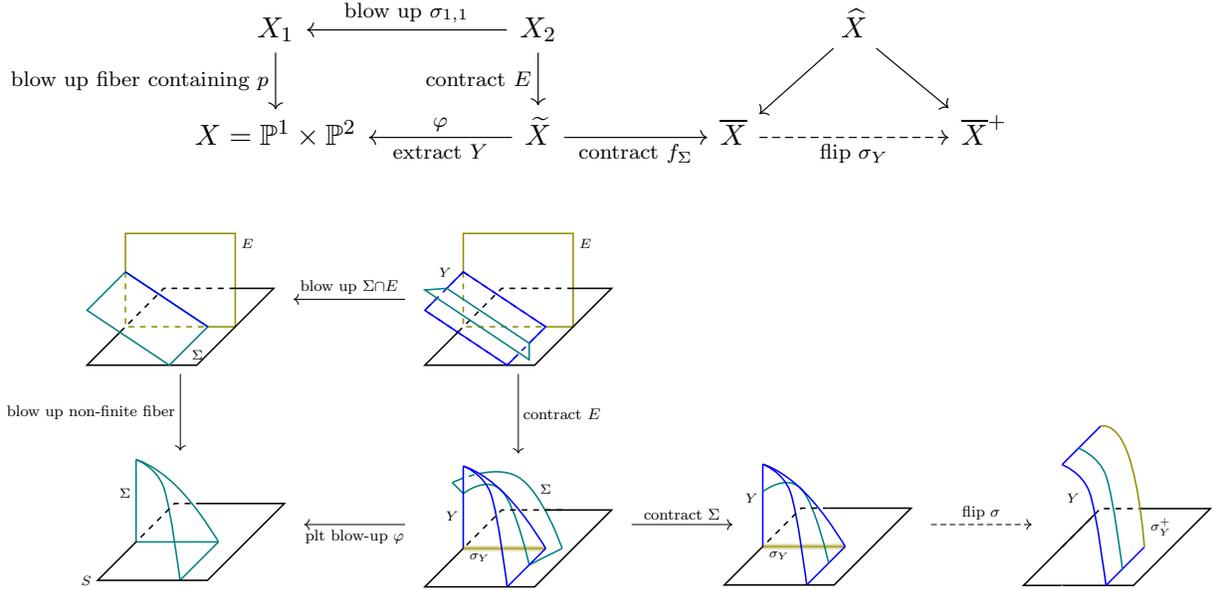
\begin{figure}[h]
\centering
    \adjustbox{width=\textwidth}{
    \begin{tikzcd}
    \begin{tabular}{c}
\begin{tikzpicture}

\draw[olive,thick] (1.5,0) -- (2,0) -- (2,1.7) -- (0,1.7) -- (0,1); 
\draw[dashed, olive,thick] (0,1) -- (0,0) -- (1.5,0); 
\draw[thick] (-.1,-.1) -- (-.7,-.7) -- (1.3,-.7) -- (2,0) -- (2.7,0.7) -- (2,0.7);
\draw[dashed,thick] (2,0.7) -- (0.7,0.7) -- (-.1,-.1);
\draw[teal,thick] (0,1) -- (-.7,0.3) -- (0.8,-.7) -- (1.5,0) -- (0,1); 
\draw[blue,thick] (0,1) -- (1.5,0); 

\node[right, node font=\tiny] at (1.1,-0.5) {$\Sigma$};
\node[below right, node font=\tiny] at (2,1.7) {$E$};

\end{tikzpicture}
\end{tabular} \arrow[d, swap, "\text{blow up non-finite fiber}"] & &\begin{tabular}{c}
\begin{tikzpicture}

\draw[olive,thick] (1.5,0) -- (2,0) -- (2,1.7) -- (0,1.7) -- (0,1); 
\draw[dashed, olive,thick] (0,1) -- (0,0) -- (1.5,0); 
\draw[thick] (-.1,-.1) -- (-.7,-.7) -- (1.3,-.7) -- (2,0) -- (2.7,0.7) -- (2,0.7);
\draw[dashed,thick] (2,0.7) -- (0.7,0.7) -- (-.1,-.1);
\draw[blue, dashed,thick] (0,1) -- (-.7,0.3) -- (0.8,-.7) -- (1.5,0) -- (0,1);
\draw[blue,thick] (-.3,.7) --(0,1) -- (1.5,0) -- (1.2,-.3);
\draw[blue,thick] (-.48,.52) -- (-.7,0.3) -- (0.8,-.7) -- (1,-.5);
\draw[teal,thick] (-.3,0.7) -- (1.2,-.3) -- (1.2,-.6) -- (-.7,0.67) -- (-.3,0.7); 

\node[below right, node font=\tiny] at (2,1.7) {$E$};
\node[left, node font=\tiny] at (-.1,1) {$Y$};

\end{tikzpicture}
\end{tabular} \arrow[ll,swap, "\text{blow up }\Sigma \cap E"] \arrow[d, "\text{contract } E"] & & & & \\
    \begin{tabular}{c}
\begin{tikzpicture}

\draw[teal,thick] (0,0) -- (1.5,0);
\draw[teal,thick] (0,0)--(0,1.5);
\draw[teal,thick] (1.5,0) -- (0.8,-0.7);
\draw[teal,thick] plot [smooth, tension=.6] coordinates { (0,1.5) (.5,1) (0.8,-.7)};
\draw[teal,thick] plot [smooth, tension=1] coordinates { (0,1.5) (.75,1) (1.5,0)};
\draw[thick] (0,0) -- (-.7,-.7) -- (1.3,-.7) -- (2,0) -- (2.7,0.7) -- (1.2,0.7);
\draw[dashed,thick] (1.2,0.7) -- (0.7,0.7) -- (0,0);
\node[left, node font=\tiny] at (-.7,-.7) {$S$};
\node[node font=\tiny] at (-.2,.8) {$\Sigma$};

\end{tikzpicture}
\end{tabular} & & \begin{tabular}{c}
\begin{tikzpicture}

\draw[line width=0.1cm, olive!30] (0,0) -- (1.5,0);
\draw[olive,semithick] (0,0) -- (1.5,0);
\draw[teal,thick] plot [smooth, tension=.8] coordinates { (0,1) (.7,1.05) (1.2,-.3)}; 
\draw[teal,thick] (0,1) -- (-.2,1.2); 
\draw[teal,thick] (1.2,-.3) -- (1.8,0); 
\draw[teal,thick] plot [smooth, tension=.8] coordinates { (-.2,1.2) (1,1.3) (1.8,0)}; 
\draw[white,fill=white] (0.05,1.3) circle (0.04);
\draw[white,fill=white] (0.25,1.35) circle (0.04);
\draw[blue,thick] (0,0)--(0,1.5);
\draw[blue,thick] (1.5,0) -- (0.8,-0.7);
\draw[blue,thick] plot [smooth, tension=.6] coordinates { (0,1.5) (.5,1) (0.8,-.7)};
\draw[blue,thick] plot [smooth, tension=1] coordinates { (0,1.5) (.75,1) (1.5,0)};
\draw[thick] (0,0) -- (-.7,-.7) -- (1.3,-.7) -- (2,0) -- (2.7,0.7) -- (1.5,0.7);
\draw[dashed,thick] (1.8,0.7) -- (0.7,0.7) -- (0,0);
\node[below right, node font=\tiny] at (0,0) {$\sigma_Y$};
\node[node font=\tiny] at (1.5,1) {$\Sigma$};
\node[node font=\tiny] at (-.2,.5) {$Y$};

\end{tikzpicture}
\end{tabular} \arrow[ll, "\text{plt blow-up } \phi"] \arrow[rr, "\text{contract } \Sigma"] && \hspace{-.2in}\begin{tabular}{c}
\begin{tikzpicture}

\draw[line width=0.1cm, olive!30] (0,0) -- (1.5,0);
\draw[blue,thick] (0,0)--(0,1.5);
\draw[olive,semithick] (0,0) -- (1.5,0);
\draw[blue,thick] (1.5,0) -- (0.8,-0.7);
\draw[teal,thick] plot [smooth, tension=.8] coordinates { (0,1) (.7,1.05) (1.2,-.3)}; 
\draw[blue,thick] plot [smooth, tension=.6] coordinates { (0,1.5) (.5,1) (0.8,-.7)};
\draw[blue,thick] plot [smooth, tension=1] coordinates { (0,1.5) (.75,1) (1.5,0)};
\draw[thick] (0,0) -- (-.7,-.7) -- (1.3,-.7) -- (2,0) -- (2.7,0.7) -- (1.2,0.7);
\draw[dashed,thick] (1.2,0.7) -- (0.7,0.7) -- (0,0);

\node[below right, node font=\tiny] at (0,0) {$\sigma_Y$};
\node[node font=\tiny] at (-.2,.8) {$Y$};

\end{tikzpicture}
\end{tabular} \arrow[dashed,rr, "\text{flip } \sigma"] && \hspace{-.2in}\begin{tabular}{c}
\begin{tikzpicture}

\draw[blue,thick] (1.5,0) -- (0.8,-0.7);
\draw[blue,thick] (0.7,2.2) -- (0,1.5);
\draw[teal,thick] plot [smooth, tension=.6] coordinates { (.3,1.8) (.8,1.3) (1.1,-.4)};
\draw[thick] (.6,.6) -- (-.7,-.7) -- (1.3,-.7) -- (2,0) -- (2.7,0.7) -- (1.5,0.7);
\draw[dashed,thick] (1.5,0.7) -- (0.7,0.7) -- (0,0);
\draw[blue,thick] plot [smooth, tension=.6] coordinates { (0,1.5) (.5,1) (0.8,-.7)};
\draw[olive,thick] plot [smooth, tension=1] coordinates { (.7,2.2) (1.2,1.7) (1.5,0)};
\draw[white] (0,-.7) -- (0,-1.5); 


\node[node font=\tiny] at (.2,.8) {$Y$};
\node[ node font=\tiny] at (1.8,.3) {$\sigma_Y^+$};

\end{tikzpicture}
\end{tabular} 
    \end{tikzcd}
    }
\caption{Computing the Nakayama--Zariski decomposition for \(L - u Y\) in Section~\ref{sec:local-delta-A_n:A_1-infinite-onesingularity}.}\label{fig:A_1-non-finite-Zariski}
\end{figure}

Using Lemma \ref{lem:volume-lemma-3}, the volume can be computed on \(\regionA\) using that \(Y|_Y = -f_Y - \sigma_Y/2\) and \(L|_Y = (2-2c)f_Y\), and on \(\regionB\) using that \(L-u Y = \pi^*(L_\oX - u \oY) + (u-3+2c)\Sigma\), \(Y|_\Sigma = \sigma_\Sigma\), \(L|_\Sigma = (3-2c)\sigma_\Sigma + (5-4c) f_\Sigma\), and \(\Sigma|_\Sigma = -\sigma_\Sigma + 2f_\Sigma\).
\begin{detail}
    In more detail, for \(\regionA\), we have \(Y\cdot f = \frac{1}{2} = -Y\cdot\sigma\) (using that the pullback of \(Y\) under \(\psi\colon X_2 \to \tX\) is \(\psi^*Y = Y_2 + \frac{1}{2} E_2\)), so \(Y|_Y = -f - \sigma/2\). Then
    \begin{equation*}\begin{split}
        \vol(P_{\regionA}(u)) &= (L-uY)^3 = L^3 - 3u L^2 \cdot Y + 3 u^2 L \cdot Y^2 - u^3 Y^3 \\
        &= (-K_\exx-\dee)^3 - 3u ((2-2c)f)^2 + 3u^2 (2-2c)f(-f - \sigma/2) - u^3 (-f - \sigma/2)^2 \\
        &= 3(2-2c)(3-2c)^2-(3/2)u^2(2-2c)-u^3/4
    \end{split}\end{equation*}
    so we get
    \[\int_0^{3-2c} P_{\regionA}(u)^3 du = 1/16 (-3 + 2 c)^3 (-77 + 78 c).\]
    For \(\regionB\), we have \(Y|_\Sigma = \sigma_\Sigma\) by construction. Using that \(\Sigma\) is isomorphic to its image in \(\exx\), which is a \((1,1)\)-surface, we compute \(\phi^*(-K_\exx-\dee)\cdot\sigma_\Sigma = 2-2c\) and \(\phi^*(-K_\exx-\dee)\cdot f_\Sigma = 3-2c\), so \(L|_\Sigma = (3-2c)\sigma_\Sigma + (5-4c) f_\Sigma\). Since \(\Sigma\cdot f_\Sigma=-1\) and \(\Sigma\cdot\sigma_\Sigma = -\sigma^2=3\), we have \(\Sigma|_\Sigma = -\sigma_\Sigma + 2f_\Sigma\). Applying Lemma~\ref{lem:volume-lemma-3} to \(p=3-2c-u\), \(q=1\), \(a=2-2c+u\), \(b=3\), then we find
    \[\vol(P_{\regionB}(u)) = \vol(P_{\regionA}(u)) + (-2 c - u + 3)^2 (-14 c - u + 18)\]
    so we get
    \[\int_{3-2c}^{4-4c} P_{\regionB}(u)^3 du = 173/16 - (81 c)/2 + (99 c^2)/2 - 26 c^3 + 5 c^4. \]
\end{detail}

For region \(\regionC\) with \(4-4c \le u \le 5-4c\), we will compute the volume on a different birational model. Write \(L_\oX = -K_\oX - \dee_\oX+(2-2c)\oY\) and \(L_{\oX^+} = -K_{\oX^+} - \dee_{\oX^+}+(2-2c)\oY^+\), where \(\dee_{\oX^+},\oY^+\) are the strict transforms. Let \(g\colon \hX \to \oX\) and \(g^+\colon \hX \to \oX^+\) be the resolution of the rational map $\oX \dashrightarrow \oX^+$ and let \(\Gamma\subset\hX\) be the exceptional divisor. Then, \(\Gamma\cong\bb P^1\times\bb P^1\), and \(g\) (resp. \(g^+\)) is the contraction of a ruling \(f_1\) (resp. \(f_2\)) of \(\Gamma\). Since \(P_{\regionC}(u)\) is the pullback of \(L_{\oX^+}-u\oY^+\) from \(\oX^+\), we will compute its volume on \(\hX\). Using that \((g^+)^*(L_{\oX^+}-u \oY^+) = g^*(L_\oX - u \oY) + (2-2c-u/2)\Gamma\), \(g^*(L_\oX - u \oY)|_\Gamma = (2-2c-u/2)f_2\), and \(\Gamma|_\Gamma = -f_1 - 2 f_2\), one can compute this volume using \(\vol(P_{\regionB}(u))\).  This is analogous to the computation in Lemma~\ref{lem:volume-lemma-4} and only differs slightly as $\Gamma \cdot f_1 = -2$ instead of $-1$.
\begin{detail}
    In more detail, if \(\hY\) is the strict transform of \(Y_2\) on \(\hX\), then we have \(g^*(L_\oX - u \oY) = -K_\hX - D_{\hX} + (2-2c-u)\hY + (1-c-u/2)\Gamma\), and \((g^+)^*(L_{\oX^+} - u \oY^+) = -K_{\hX^+} - D_{\hX} + (2-2c-u)\hY + (3-3c-u)\Gamma\) since \(\oY^+\) and \(R_{\oX^+}\) both contain and are smooth along the image of \(\Gamma\). Thus, \(L_{\oX^+}-u \oY^+\) has the same volume as \(g^*(L_\oX - u \oY) + (2-2c-u/2)\Gamma\). To compute this, \((g^*(L_\oX - u \oY))\cdot f_1 = 0\) and \((g^*(L_\oX - u \oY))\cdot f_2 = (L_\oX - u \oY)\cdot\sigma=2-2c-u\), so \(g^*(L_\oX - u \oY)|_\Gamma = (2-2c-u/2)f_2\).  Next, we have \(\Gamma\cdot f_1 = -2\) by using that \(\hY\cdot f_1 = 1\), \(g^*\oY = \hY + \frac{1}{2}\Gamma\), and intersecting with \(f_1\), and \(\Gamma\cdot f_2 = -1\), so \(\Gamma|_\Gamma = -f_1-2f_2\). So 
    \[ \resizebox{1\textwidth}{!}{ $\begin{array}{rl}
    \vol(P_{\regionC}(u)) &= \vol(P_{\regionB}(u)) + 3(2-2c-u/2)g^*(L_\oX - u \oY)^2\Gamma + 3(2-2c-u/2)^2g^*(L_\oX - u \oY)\Gamma^2 + (2-2c-u/2)^3\Gamma^3 \\
        &= \vol(P_{\regionB}(u))
        + 3(2-2c-u/2)(0)
        + 3(2-2c-u/2)^2(-1)(2-2c-u/2)
        + (2-2c-u/2)^3(4) \\
        &= 3(2-2c)(3-2c)^2-(3/2)u^2(2-2c)-u^3/4 + (2 c + u - 3) (-28 c^2 - 16 c u + 78 c - u^2 + 21 u - 54) + (1/8) (-4 c - u + 4)^3
\end{array} $}\]
    and we get
    \[\int_{4-4c}^{5-4c} P_{\regionC}(u)^3 du = 37/32. \]
\end{detail}

Computing these volumes as described above, we find that \[S_{\exx,\dee}(Y) = \frac{1408 c^4 - 7680 c^3 + 15552 c^2 - 13824 c + 4541}{96 (2 - 2c) (3 - 2 c)^2}, \]
\begin{detail}
    as
    \begin{equation*}\begin{split}
    S_{\exx,\dee}(Y) &= \frac{1}{3(2-2c)(3-2c)^2} \left( \int_0^{3-2c} P_{\regionA}(u)^3 du + \int_{3-2c}^{4-4c} P_{\regionB}(u)^3 du + \int_{4-4c}^{5-4c} P_{\regionC}(u)^3 du\right) \\
    &= \frac{(1/32)(1408 c^4 - 7680 c^3 + 15552 c^2 - 13824 c + 4541))}{3(2-2c)(3-2c)^2},
    \end{split}\end{equation*}
\end{detail}
so since \(A_{\exx,\dee}(Y) = 3-2c\), we have \(A_{\exx,\dee}(Y)/S_{\exx,\dee}(Y)>1\) for \(c\in (0,0.97976]\).

To complete the proof, by the Abban--Zhuang method it suffices to show that \(\delta_q(Y, \dee_Y; V_{\bullet \bullet})\) for all \(q\in Y\), where \(V_{\bullet \bullet}\) is the refinement of the filtration of \(L\) by \(Y\).  For simplicity, denote $f_Y$ be $f$ and $\sigma_Y$ by $\sigma$.  We will use the flag coming from a fiber \(f\) of \(Y\cong\bb F_3\) containing \(q\).

To compute this, we will use the formulas from \cite[Theorems 4.8 and 4.17]{Fujita-3.11}, so we first compute the Zariski decomposition of the divisor \(P(u)|_Y - vf\) on \(Y\). From the above descriptions and using that \(\Sigma|_Y \sim \sigma + 3f\) we can compute the restrictions \(P(u)|_Y\) and \(N(u)|_Y\) on regions \(\regionA\) and \(\regionB\). For \(\regionC\), we compute on the strict transform of \(Y\) in \(\hX\) (which is isomorphic to \(Y\)), using that \((g^+)^*(L_{\oX^+}-u \oY^+) = g^*(L_\oX - u \oY) + (2-2c-u/2)\Gamma\) and \(\Gamma|_\oY = \sigma\). We find
\[ \resizebox{1\textwidth}{!}{ $\begin{array}{rll}
    0 \le u \le 3-2c: & P_{\regionA}(u)|_Y = (2-2c+u)f + (u/2)\sigma & N_{\regionA}(u)|_Y  =  0 \\
    3-2c \le u \le 4-4c: & P_{\regionB}(u)|_Y \sim (11-8c-2u)f + (3-2c-u/2)\sigma & N_{\regionB}(u)|_Y = (u-3+2c)\Sigma|_Y \\
    4-4c \le u \le 5-4c: & P_{\regionC}(u)|_Y \sim (11-8c-2u)f + (5-4c-u)\sigma & N_{\regionC}(u)|_Y  = (u-3+2c)\Sigma|_Y + (u/2 - 2 + 2c)\sigma.
\end{array}$ } \]
Since no fiber of \(Y\cong\bb F_3\) is a component of \(\Sigma|_Y\), we have \(\ord_f N(u)|_Y = 0\) on all regions.
Next, we find the Zariski decompositions of \(P(u)|_Y - vf\).

In region $\regionA$ for $0 \le u \le 3-2c$, $P_{\regionA}(u)|_Y - vf = (2-2c+u-v)f + (u/2)\sigma$.  Here, $t(u) = 2-2c+u$, and we find the Zariski decomposition for \(0 \le v \le 2-2c+u\):
\begin{equation*}\begin{split}
    P_{\regionA}(u,v) &= \begin{cases}
        (2-2c+u-v)f + \frac{u}{2}\sigma & 0 \le v \le 2 -2 c - \frac{u}{2} \\
        (2-2c+u-v)(f + \frac{1}{3}\sigma) & 2 -2 c - \frac{u}{2} \le v \le 2-2c+u ,
    \end{cases} \\
    N_{\regionA}(u,v) &= \begin{cases}
        0 \hphantom{(2-c+u-v)f + l\frac{u}{2}v\sigma} &  0 \le v \le 2 -2 c - \frac{u}{2} \\
        \tfrac{1}{3}(\frac{u}{2} + v - 2 + 2 c) \sigma & 2 -2 c - \frac{u}{2} \le v \le 2-2c+u .
    \end{cases}
\end{split}\end{equation*}
\begin{detail}
So the relevant quantities for $0 \le u \le 3-2c$ are: 
\[ \begin{array}{r|ll}
   &  0 \le v \le 2 -2 c - u/2 & 2 -2 c - u/2 \le v \le 2-2c+u \\
       \hline
   (P_{\regionA}(u,v))^2 & (2-2c+u-v)u - 3 u^2/4 & (2-2c+u-v)^2/3 \\
   P_{\regionA}(u,v)\cdot f & u/2 & (2-2c+u-v)/3 \\
   (P_{\regionA}(u,v)\cdot f)^2 & u^2/4  & (2-2c+u-v)^2/9 \\
   \ord_q((N_{\regionA}'(u)|_f + N_{\regionA}(u,v))|_f) & 0 & \left\{ \begin{array}{ll}
        0 & q \notin \sigma \\
      (u/2 + v - 2 + 2 c)/3  & q \in \sigma 
   \end{array} \right. \\
\end{array} \]
\end{detail}
In region $\regionB$ for $3-2c \le u \le 4-4c$, $P_{\regionB}(u)|_Y - vf = (11-8c-2u-v)f + (3-2c-u/2)\sigma$.  Here, $t(u) =11-8c-2u$, and we find the Zariski decomposition for \(0 \le v \le 2-2c+u\):
\begin{equation*}\begin{split}
    P_{\regionB}(u,v) &= \begin{cases}
        (11-8c-2u-v)f + (3-2c-\frac{u}{2})\sigma & 0 \le v \le 2 -2 c - \frac{u}{2} \\
        (11-8c-2u-v)(f + \frac{1}{3}\sigma) & 2 -2 c - \frac{u}{2} \le v \le 11-8c-2u ,
    \end{cases} \\
    N_{\regionB}(u,v) &= \begin{cases}
        0 \hphantom{(11-8c-u-v)f + (3-2c-\frac{u}{2})\sigma} &  0 \le v \le 2 -2 c - \frac{u}{2} \\
        \frac{1}{3}(\frac{u}{2} + v - 2 + 2 c)\sigma & 2 -2 c - \frac{u}{2} \le v \le 11-8c-2u .
    \end{cases}
\end{split}\end{equation*}
\begin{detail}
So the relevant quantities for $3-2c \le u \le 4-4c$ are: 
\[ \resizebox{1\textwidth}{!}{ $ \begin{array}{r|ll}
   &  0 \le v \le 2 -2 c - u/2 & 2 -2 c - u/2 \le v \le 11-8c-2u \\
       \hline
   (P_{\regionB}(u,v))^2 & (3-2c-u/2)(-20 c - 5 u - 4 v + 26)/2  & (11-8c-2u-v)^2/3 \\
   P_{\regionB}(u,v)\cdot f & 3-2c-u/2 & (11-8c-2u-v)/3 \\
   (P_{\regionB}(u,v)\cdot f)^2 & (3-2c-u/2)^2  & (11-8c-2u-v)^2/9 \\
   \ord_q((N_{\regionB}'(u)|_f + N_{\regionB}(u,v))|_f) & \left\{ \begin{array}{ll}
        0 & q \notin \Sigma|_Y \\
      u-3+2c  & q \in \Sigma|_Y
   \end{array} \right.  & \left\{ \begin{array}{ll}
        0 & q \notin \Sigma|_Y, \sigma \\
      u-3+2c  & q \in \Sigma|_Y \\
      (u/2 + v - 2 + 2 c)/3 & q \in \sigma
   \end{array} \right. \\
\end{array} $} \]
\end{detail}
In region $\regionC$ for $4-4c \le u \le 5-4c$, $P_{\regionC}(u)|_Y - vf = (11-8c-2u-v)f + (5-4c-u)\sigma$.  Here, $t(u) = 11-8c-2u$, and we find the Zariski decomposition for \(0 \le v \le 11-8c-2u\):
\begin{equation*}\begin{split}
    P_{\regionC}(u,v) &= \begin{cases}
        (11-8c-2u-v)f + (5-4c-u)\sigma & 0 \le v \le u - 4 + 4 c \\
        (11-8c-2u-v)(f + \frac{1}{3}\sigma) & u - 4 + 4 c \le v \le 11-8c-2u ,
    \end{cases} \\
    N_{\regionC}(u,v) &= \begin{cases}
        0 \hphantom{(11-8c-u-v)f + (5-4c-u)\sigma} &  0 \le v \le u - 4 + 4 c \\
        \frac{1}{3}(- u + v + 4 -4 c) \sigma & u - 4 + 4 c \le v \le 11-8c-2u .
    \end{cases}
\end{split}\end{equation*}
\begin{detail}
So the relevant quantities for $4-4c \le u \le 5-4c$ are: 
\[ \resizebox{1\textwidth}{!}{ $ \begin{array}{r|ll}
   &  0 \le v \le u - 4 + 4 c & u - 4 + 4 c \le v \le 11-8c-2u \\
       \hline
   (P_{\regionC}(u,v))^2 & (5-4c-u)(-4 c - u - 2 v + 7)  & (11-8c-2u-v)^2/3 \\
   P_{\regionC}(u,v)\cdot f & 5-4c-u & (11-8c-2u-v)/3 \\
   (P_{\regionC}(u,v)\cdot f)^2 & (5-4c-u)^2  & (11-8c-2u-v)^2/9 \\
   \ord_q((N_{\regionC}'(u)|_f + N_{\regionC}(u,v))|_f) & \left\{ \begin{array}{ll}
        0 & q \notin \Sigma|_Y, \sigma \\
      u-3+2c  & q \in \Sigma|_Y \\
      u/2 - 2 + 2c & q \in \sigma
   \end{array} \right.  & \left\{ \begin{array}{ll}
        0 & q \notin \Sigma|_Y, \sigma \\
      u-3+2c  & q \in \Sigma|_Y \\
      (u/2 + v - 2 + 2 c)/3 & q \in \sigma
   \end{array} \right. \\
\end{array} $} \]
\end{detail}
The log discrepancy is \(A_{Y,\dee_Y}(f)=1\). Using \cite[Theorem 4.8]{Fujita-3.11}, we compute that \[S_{Y,\dee_Y; V_{\bullet \bullet}}(f) = \frac{-14 c^3 + 58 c^2 - 80 c + 37}{3 (3 - 2 c)^2}, \qquad \frac{A_{Y,\dee_Y}(f)}{S_{Y,\dee_Y; V_{\bullet \bullet}}(f)} > 1 \text{ for }c\in (\cone,1).
\]

Finally, we turn to computing \(A_{f,\dee_f}(q)/S_{f,\dee_f; W_{\bullet \bullet \bullet}}(q)\) for \(q\in f\subset Y\), where \(\dee_f\) is the restriction of \(\dee_Y\) to \(f\), and \( W_{\bullet \bullet \bullet}\) is the further refinement of \(V_{\bullet \bullet}\) by \(f\). Using \cite[Theorem 4.17]{Fujita-3.11}, we find that
\[S_{f,\dee_f; W_{\bullet \bullet \bullet}}(q) = \begin{cases}
    \displaystyle \frac{(14 c - 13) (2 c - 3)}{96 (1 - c)} & q\notin \Sigma|_Y, \sigma \\
    \displaystyle \frac{(1 - c) (6 c^2 - 20 c + 17)}{3 (3 - 2 c)^2} & q \in \Sigma|_Y \\
    \displaystyle \frac{-10 c^3 + 46 c^2 - 71 c + 37}{6 (3 - 2 c)^2} & q \in \sigma.
\end{cases}\]
Since \(A_{f,\dee_f}(q)\) is \(1\) or \(1-c\), we find that
\[\frac{A_{f,\dee_f}(q)}{S_{f,\dee_f; W_{\bullet \bullet \bullet}}(q)} >1 \text{ for all }c \in \begin{cases}
    (0, 1) & q\notin \sigma \\
    (0, 0.5166] & q \in \sigma.
\end{cases}\]
\end{proof}

\subsubsection{Two \(A_1\) singularities along a non-finite fiber of \(\piR_2\)} In this section, we show that \(\delta=1\) in the case of two \(A_1\) singularities \(\pt_1,\pt_2\) on a non-finite fiber of \(\piR_2\). (Recall that we showed in Proposition~\ref{prop:A_n-sings-not-stable}\ref{item:higher-A_n-not-stable-non-finite} that \(R\) is in fact GIT strictly semistable in this case.)

\begin{lemma}\label{lem:A_1-infinite-twosingularities-Ksemistable}
    Suppose $R$ has two $A_1$ singularities at $\pt$ along any non-finite fiber \(l_1\) of the second projection \(\piR_2\colon R \to \bb P^2\). Then, for any \(\pt \in l_1\) we have $\delta_\pt(\exx, \dee) = 1$ for all $c \in (0, \frac{3}{4})$.
\end{lemma}
 
Let $X_1 \to \exx$ be the blow-up of the non-finite fiber \(l_1\) of $\piR_2$ containing $\pt_1,\pt_2$, and let $E_1 \cong \bP^1 \times \bP^1$ be the exceptional divisor.  Because $R$ has two $A_1$ singularities along \(l_1\), if $R_{E_1}$ denotes the strict transform of $R$ restricted to $E_1$, we have $R_{E_1} = f_1 + f_2 + s_1$, where $f_1, f_2$ are the fibers above the singular points $\pt_1, \pt_2$, and $s_1$ is a horizontal section of $E_1$.  Note also that this blow-up separates the pencil of surfaces $S \in |\calO_\exx(0,1)|$ containing $f$, so there is a unique surface $S_0$ whose strict transform on \(X_1\) also intersects $E_1$ precisely in the curve $s_1$.  Next, let $X_2 \to X_1$ be the blow-up of $s_1$ with exceptional divisor $Y_2 \cong \bP^1 \times \bP^1$.  Because the strict transform of $R$ in $X_1$ is smooth, in this case $R_{Y_2} \coloneqq R_{X_2} \cap Y_2$ is a smooth $(2,1)$-curve in $Y_2$.  Furthermore, we may contract the strict transform of $E_1$ in $X_2$ over $\exx$ to a new threefold $\phi\colon \tX \to \exx$ with a $\frac{1}{2}(1,1)$ singularity along the image of $E_1$, which we will denote $s_0$.  Let $Y \cong \bP^1 \times \bP^1$ be the image of $Y_2$ in $\tX$.  Let $f$ denote a fiber of $Y$ contracted by $\phi$, i.e. a $(1,0)$-curve, and let $s$ denote a section of type $(0,1)$.  By construction, $\dee_Y = \frac{1}{2} s_0 + cR_Y$ where $R_Y$ is a smooth $(2,1)$-curve. Observe that the strict transform $\tilde{S}_0$ of $S_0$ on $\tX$ intersects $Y$ in a $(0,1)$-curve $s_\infty$ missing the singular locus of $\tX$, and that $\tX$ is $ \bP(1,1,2)$-bundle over $\bP^1$. 

Let \(L = \phi^*(-K_\exx - \dee)\). Using $S_0$, we compute the pseudoeffective threshold of $L - uY$ to be $u = 6-4c$.  To compute $S_{\exx,\dee}(Y)$, we note that at $u = 3-2c$, $\tS_0 \cong \bb P^1\times\bb P^1$ is contracted to $s_\infty$, and its image is a smooth curve.  Therefore, the Nakayama--Zariski decomposition of $L - uY $ is 
\[\resizebox{1\textwidth}{!}{ $\begin{array}{rll}
     0 \le u \le 3-2c: &\quad  P_{\regionA}(u) = \phi^*(-K_\exx - \dee) -uY &\quad N_{\regionA}(u) =0 \\
    3-2c \le u \le 6-4c: &\quad  P_{\regionB}(u) = \phi^*(-K_\exx - \dee) -uY + (3-2c-u)\tS_0 &\quad N_{\regionB}(u) =  (u-3+2c) \tS_0. \\
\end{array} $}\]

  If \(f_S\) is a \((1,0)\)-curve on \(\tS_0\), then we have \(L|_Y = (2-2c)f\), \(Y|_Y = -s/2\), \(L|_{\tS_0} = (2-2c) f_S + (3-2c)s_\infty\), \(Y|_{\tS_0} = s_\infty\), and \(S|_S = -s_\infty\). Therefore, we compute that \(S_{\exx,\dee}(Y) = 3-2c\) by applying Lemma~\ref{lem:volume-lemma-3} twice.
\begin{detail}
    For more detail on the volume computations, we compute \(L|_Y = (2-2c)f\) by intersecting with \(f,s\). We have \(Y|_Y = -s/2\) since \(Y\cdot f = -\frac{1}{2}\) and \(Y\cdot s = 0\). On region \(\regionA\), we apply Lemma~\ref{lem:volume-lemma-3} to \(P_2 = -K_{\exx}-\dee\), \(P_1 = P_{\regionA}(u)\), \(a=2-2c\), \(b=0\), \(p=-u/2\), and \(q=\frac{1}{2}\) to obtain
    \(\vol (P_{\regionA}(u)) = 3(2-2c)(3-2c)^2 -3u^2(1-c)\).
    For region \(\regionB\), applying Lemma~\ref{lem:volume-lemma-3} to the contraction of \(\tS_0\) yields \(\vol(P_{\regionB}(u)) = 3(2-2c)(3-2c)^2 -3u^2(1-c) + 3(3-2c-u)^2(2-2c)\).
\end{detail}
Using that $\phi\colon \tX \to X$ is the $(2,1)$ weighted blow-up of the generic point of \(l_1\), we also compute $A_{\exx,\dee} = 3-2c$, so for all \(c\in (0,1)\) we obtain that 
\[ \frac{A_{\exx,\dee}(Y)}{S_{\exx,\dee}(Y)} = 1.\]

Next, we compute lower bounds $\delta_q(Y, \dee_Y; V_{\bullet \bullet})$ for $q\in Y$, where \(V_{\bullet \bullet}\) is the refinement of the filtration of \(L\) by \(Y\). Recall that $s_0 \subset Y$ is the singular locus of $\tX$, and $s_\infty \subset Y$ is the intersection of $\tS_0$ and $Y$.
For points $q \in Y$ not contained in the intersection $R_Y \cap s_0$ or $R_Y \cap s_\infty$, we use the flag where $C = f$ is the fiber of $Y$ containing $q$. If \(q\) is contained in the intersection of \(R_Y\) with \(s_0\) (resp. \(s_\infty\)) we will use \(C=s_0\) (resp. \(C=s_\infty\)). Let $W_{\bullet \bullet \bullet}$ be the further refinement of $V_{\bullet \bullet}$ by $C$, and let $\dee_C$ be the restriction of $\dee_Y$ to $C$. By the Abban--Zhuang method, it suffices to compute \(A_{Y, \dee_Y}(C) / S_{Y,\dee_Y; V_{\bullet \bullet}}(C)\) and \(\delta_q(C,\dee_C; W_{\bullet \bullet \bullet})\).

By computing Zariski decompositions of \(P(u)|_Y - vC\) on \(Y\) and using \cite[Theorem 4.8]{Fujita-3.11}, it is straightforward to find that
\[A_{Y,\dee_Y}(C) = \begin{cases}
    1 & C = f \text{ or }s_\infty \\
    1/2 & C = s_0
\end{cases}, \qquad
S_{Y,\dee_Y; V_{\bullet \bullet}}(C) = \begin{cases}
    1-c & C = f \\
    (3 - 2 c)/6 & C = s_0 \\
    (3 - 2 c)/3 & C = s_\infty
\end{cases}\]
and, by using \cite[Theorem 4.17]{Fujita-3.11} (and using \cite[Corollary 4.18]{Fujita-3.11} with \(\nu=\id\) and \(C=s_0\) for \(A_{C, \dee_C}(q)\) in the \(q\notin R_Y\) cases):
\[A_{C, \dee_C}(q) = \begin{cases}
    1 & q \notin R_Y, s_0 \\
    1-c & q \in R_Y \\
    1/2 & q\in s_0, q \notin R_Y
\end{cases} \qquad
S_{C,\dee_C; W_{\bullet \bullet \bullet}}(q) = \begin{cases}
    (3 - 2 c)/6 & q\notin s_\infty, R_Y \cap s_0; C = f \\
    (3 - 2 c)/3 & q \in s_\infty, q \notin R_Y; C = f \\
    1-c & q \in R_Y \cap s_0 \text{ or }R_Y \cap s_\infty .
\end{cases}\]
Thus, we have \(A_{Y,\dee_Y}(C)/S_{Y,\dee_Y; V_{\bullet \bullet}}(C) > 1\) for all \(c\in (0,1)\). For \(q\notin R_Y \cap s_0\) or \(R_Y\cap s_\infty\), because \(s_0\cap s_\infty=\emptyset\) we have \(\delta_q(f,\dee_f; W_{\bullet \bullet \bullet}) \geq \min\{3/(3 - 2 c), (6 (1 - c))/(3 - 2 c) \} >1\) for all \(c\in(0,3/4)\). For \(q\) contained in the intersection \(R_Y \cap C\) for \(C\in\{s_0, s_\infty\}\), we have \(\delta_q(C,\dee_C; W_{\bullet \bullet \bullet}) = 1\).

This proves that, for all $\pt \in l_1$ a non-finite fiber of the second projection such that $R$ has two $A_1$ singularities along $l_1$, $\delta_{\pt}(\exx,\dee) = 1$ for all $c \in (0, \frac{3}{4})$.

\begin{detail}
    In more detail, from the formulas for \(P(u)\) and \(N(u)\) from above, we have
\[ \begin{array}{rll}
    0 \le u \le 3-2c: & P_{\regionA}(u)|_Y = (2-2c)f + (u/2)s & N_{\regionA}(u)|_Y  =  0 \\
    3-2c \le u \le 6-4c: & P_{\regionB}(u)|_Y = (2-2c)f + (3 -2 c - u/2)s & N_{\regionB}(u)|_Y = (u-3+2c) s_\infty .
\end{array}\]

    For \(q\) not contained in \(s_0 \cap R_Y\) or \(s_\infty \cap R_Y\), we compute the Zariski decompositions for \(P(u)|_Y - vf\):

    In region $\regionA$ for $0 \le u \le 3-2c$, $P_{\regionA}(u)|_Y - vf = (2-2c-v)f + (u/2)s$.  Here, $t(u) = 2-2c$, and the Zariski decomposition for \(0 \le v \le 2-2c\) is
    \[ P_{\regionA}(u,v) = (2-2c-v)f + (u/2)s, \qquad N_{\regionA}(u,v) = 0.\]
    In region \(\regionB\) for \(3-2c \le u \le 6-4c\), \(P_{\regionB}(u)|_Y - vf = (2-2c-v)f + (3 -2 c - u/2)s\). Here, \(t(u) = 2-2c\), and the Zariski decomposition for \( 0 \le v \le 2-2c\) is
    \[ P_{\regionB}(u,v) = (2-2c-v)f + (3 -2 c - u/2)s, \qquad N_{\regionB}(u,v) = 0.\]
    So the relevant quantities are
    \[ \begin{array}{r|ll}
   &  \regionA,\; 0 \le v \le 2-2c & \regionB,\; 0 \le v \le 2-2c \\
       \hline
   (P(u,v))^2 & (2-2c-v)u & 2(2-2c-v)(3 -2 c - u/2) \\
   P(u,v)\cdot f & u/2 & 3 -2 c - u/2 \\
   (P(u,v)\cdot f)^2 & u^2/4  & (3 -2 c - u/2)^2 \\
   \ord_q((N'(u)|_f + N(u,v))|_f) & 0 & \left\{ \begin{array}{ll}
        0 & q \notin s_\infty \\
      u-3+2c  & q \in s_\infty 
   \end{array} \right. \\
\end{array} \]
    and we compute that
\[ \resizebox{1\textwidth}{!}{ $ \begin{array}{rl}
     S_{Y,\dee_Y; V_{\bullet \bullet}}(f) &= \displaystyle\frac{3}{3(2-2c)(3-2c)^2} \left(\int_0^{3-2c} \left( \int_0^{2-2c} P_{\regionA}(u,v)^2 dv \right) du
        + \int_{3-2c}^{6-4c} \left( \int_0^{2-2c} P_{\regionB}(u,v)^2 dv \right) du \right) \\
        &= \displaystyle\frac{2 (c - 1)^2 (2 c - 3)^2}{(2-2c)(3-2c)^2} = 1 - c.
\end{array} $} \]

    The log discrepancy is \(A_{Y,\dee_Y}(f)=1\).

    Next, we compute \(S_{f,\dee_f; W_{\bullet \bullet \bullet}}(q)\). First, note that \(A_{f, \dee_f}(q)\) is 1 if \(q \notin R_Y, s_0\), is \(1-c\) if \(q\) is in \(R_Y\) but not in \(s_0\), and is \(\frac{1}{2}\) if \(q\in s_0\) but not in \(R_Y\). Then we can use the above formulas to compute
    \[S_{f,\dee_f; W_{\bullet \bullet \bullet}}(q) = \begin{cases}
    \displaystyle (3 - 2 c)/6 & q\notin s_\infty \\
    \displaystyle (3 - 2 c)/3 & q \in s_\infty.
\end{cases}\]

    For the case when \(q\) is contained in \(C \cap R_Y\) for \(C\in\{s_0, s_\infty\}\), we compute the Zariski decompositions for \(P(u)|_Y - v s\):

    In region $\regionA$ for $0 \le u \le 3-2c$, $P_{\regionA}(u)|_Y - vf = (2-2c)f + (u/2-v)s$. Here, \(t(u) = u/2\), and the Zariski decomposition for \(0 \le v \le 3 -2 c - u/2\) is
    \[ P_{\regionA}(u,v) = (2-2c)f + (u/2-v)s, \qquad N_{\regionA}(u,v) = 0.\]
    In region \(\regionB\) for \(3-2c \le u \le 6-4c\), \(P_{\regionB}(u)|_Y - vf = (2-2c)f + (3 -2 c - u/2-v)s\). Here, \(t(u) = 3 -2 c - u/2\), and the Zariski decomposition for \( 0 \le v \le 3 -2 c - u/2\) is
    \[ P_{\regionB}(u,v) = (2-2c)f + (3 -2 c - u/2-v)s, \qquad N_{\regionB}(u,v) = 0.\]
    So, recalling from \cite[Notation 4.11]{Fujita-3.11} that \(N'(u)|_Y \coloneqq N(u)|_Y - \ord_C(N(u)|_Y) C\), the relevant quantities are
    \[ \begin{array}{r|ll}
   &  \regionA, \; 0 \le v \le u/2 & \regionB, \; 0 \le v \le 3 -2 c - u/2 \\
       \hline
   (P(u,v))^2 & 2(2-2c)(u/2-v) & 2(2-2c)(3 -2 c - u/2-v) \\
   P(u,v)\cdot C & 2-2c & 2-2c \\
   (P(u,v)\cdot C)^2 & (2-2c)^2  & (2-2c)^2 \\
   \ord_q((N'(u)|_C + N(u,v))|_C) & 0 & 0 \\
\end{array} \]
    For \(C=s_0\) and \(q \in s_0 \cap R_Y\), the log discrepancy is \(A_{Y,\dee_Y}(s_0)=\frac{1}{2}\), and we compute by the above formulas that \(S_{Y,\dee_Y; V_{\bullet \bullet}}(s_0) = (3 - 2 c)/6\). Next, for \(\delta_{s_0,\dee_{s_0}; W_{\bullet \bullet \bullet}}(q)\), first note that \(A_{s_0, \dee_{s_0}}(q) = 1-c\). Using the above formulas, we compute \(S_{s_0,\dee_{s_0}; W_{\bullet \bullet \bullet}}(q) = 1 - c\).

    For \(C=s_\infty\) and \(q \in s_\infty \cap R_Y\), the log discrepancy is \(A_{Y,\dee_Y}(s_\infty)=1\), and we compute by the above formulas and \cite[Theorem 3.8]{Fujita-3.11} that
    \begin{equation*}\begin{split}
        S_{Y,\dee_Y; V_{\bullet \bullet}}(s_\infty) &= \frac{3 - 2 c}{6} + \frac{3}{3(2-2c)(3-2c)^2}\int_{3-2c}^{6-4c} (P_{\regionB}(u)|_Y )^2 \cdot (u-3+2c) du \\
        &= \frac{3 - 2 c}{6} + \frac{3}{3(2-2c)(3-2c)^2} \int_{3-2c}^{6-4c} 2(2-2c)(3 -2 c - u/2)(u-3+2c) du \\
        &= \frac{3 - 2 c}{6} + \frac{3 - 2 c}{6} = \frac{3 - 2 c}{3}.
    \end{split}\end{equation*}
    Next, we have \(A_{s_\infty, \dee_{s_\infty}}(q) = 1-c\), and by the same computation as for \(C=s_0\) we have \(S_{s_\infty,\dee_{s_\infty}; W_{\bullet \bullet \bullet}}(q) = 1 - c\).
\end{detail}

\subsection{\texorpdfstring{\(A_2\)}{A2} singularities}\label{sec:local-delta-A_n:A_2}
In this section, we finish the proof of Theorem~\ref{thm:local-delta-A_n-singularities}\eqref{part:local-delta-A_n-singularities:A_2}.

Let $R$ be a $(2,2)$-surface with an $A_2$ singularity at $\pt \in R$, and assume that the fiber of the second projection \(\piR_2 \colon R \to \bb P^2\) at $\pt$ is finite. Note that, by Proposition~\ref{prop:A_n-sings-not-stable}\ref{item:higher-A_n-not-stable-non-reduced}, the \(A_2\) assumption implies that the fiber of \(\piR_1 \colon R \to \bb P^1\) at \(\pt\) is reduced. Recall that, in Section~\ref{sec:local-delta-A_n:higher-part-1}, we showed that \(\delta_{\pt}(\exx,\dee)>1\) for \(c\in(0,\cnot)\). We aim to show $\delta_{\pt}(\exx,\dee) > 1$ for all $c \in (0, \frac{1}{2}]$.  Assume $c \le \frac{1}{2}$ in what follows. First, we construct a divisor $Z$ of plt type over $\exx$.

\subsubsection{Set-up and computation of $\delta(Z)$ for a natural divisor $Z$ over the singular point}\label{sec:local-delta-A_n:A_2-part-1} Let $F \cong \bb P^2$ be the fiber of the first projection to $\bb P^1$ containing $\pt$, and let $S \cong \bb P^1 \times \bb P^1$ be a general $(0,1)$-surface through $\pt$.

Let $\tX_1 \to X$ be the blow-up of $\pt$ with exceptional divisor $Y \cong \bb P^2$.  Then \(R|_Y\) is a rank 2 conic with singular point $q_1$.  Because $\piR_1$ is reduced, $q_1 \notin Y$.  Let \(r_1\) be one of the components of \(R|_Y\), let $\tX_2 \to \tX_1$ be the blow-up of $r_1$ with exceptional divisor $Z_2 \cong \bb F_2$, and let \(Y_2\cong\bb P^2\) denote the strict transform of \(Y\) in \(\tX_2\).  The negative  section $\sigma$ of $Z_2 \cong \bb F_2$ is the intersection \(Z_2\cap Y_2\).
The strict transform of $R$ in $\tX_1$ is smooth since \(\pt\) is an \(A_2\) singularity, so $R|_{Z_2}$ is a smooth section of $Z_2 \cong \bb F_2$ meeting $\sigma$ at one point (the preimage of $q_1$ on $\sigma$).
The divisor \(Y_2\) in \(\tX_2\) may be contracted in a morphism $\tX_2 \to \tX$.  By computation, a line $l \subset Y_2$ satisfies $K_{\tX_2} \cdot l = -1$ and $Y_2 \cdot l = -2$, so $\tX$ has a $\frac{1}{2}(1,1,1)$ singularity at the image of $Y_2$.  Let $Z \cong \bb P(1,1,2)$ be the image of $Z_2$.  By construction, there is a map $\phi\colon \tX \to X$ contracting $Z$.  Let $\dee_{\tX}$ be the strict transform of $\dee$.  We compute that
\[ L \coloneqq \phi^*(-K_{\exx} - \dee) = -K_\tX - \dee_\tX + (3-3c) Z \]
and that $Z$ is a prime divisor of plt type over $\exx$.

We next compute a Nakayama--Zariski decomposition of $L - uZ$.  By Corollary~\ref{cor:generatorsofmoricone}, the Mori cone is generated by the strict transform $l_S$ of the fiber of the second projection $\exx \to \bb P^2$ through $\pt$ (not contained in $\tR$ by assumption); the strict transform $l_F$ of the line through $\pt \in F$ that is contained in $R$ and that meets $r_1$ in $\tX_1$; and a ruling $l_Z$ of the cone $Z = \bb P(1,1,2)$.  Note that $l_S$ meets $Z$ in the singular point.  It is straightforward to show the pseudoeffective threshold of $L - uZ$ is $u = 5-4c$.  By computation, the Nakayama--Zariski decomposition will correspond to the following birational transformations: before $u = 3-2c$, the divisor is ample; at $u = 3-2c$, we perform the Atiyah flop $l_F$ from $F$ into $Z$; at $u = 4 - 4c$, we flip $l_S$ as constructed in Section \ref{construction:flip}; and we stop at $u = 5-4c$, corresponding to a contraction of $\tX$ onto a surface.  We may construct a common resolution $\psi\colon \oX \to \tX$ of these birational modifications by blowing up the singular point, blowing up the strict transform of $l_F$, and then blowing up the strict transform of $l_S$.

\[\begin{tikzcd}
\tX_1 \ar[d, "\text{blow up }\pt" {swap}] & & \tX_2 \ar[ll, "\text{blow up }r_1" {swap}] \ar[d, "\text{contract } Y_2" {swap}] & \hX_1 \arrow[ld, "\mu" {swap}] \arrow[rd] & \oX \ar[l, "\nu" {swap}] \ar[r] & \hX_3 \arrow[ld] \arrow[rd] & &  \\
\exx = \bb P^1\times\bb P^2 & & \tX \arrow[ll, "\phi" {swap}, "\text{extract }Z"] \arrow[rr, dashed, "\text{Atiyah flop of }l_F" {swap}] & & \tX^+ \arrow[rr, dashed, "\text{flip }l_S" {swap}] & & X_3
\end{tikzcd}\]

Let $E_F \cong \bb P^1 \times \bb P^1$ denote the exceptional divisor of the blow-up of $l_F$ and $E_S \cong \bb P^1 \times \bb P^1$ the exceptional divisor of the blow-up of $l_S$, and let $\oY$, $\oZ$, $\dee_{\oX}$ denote the strict transforms of $Y_2$, $Z$, $\dee_{\tX}$, respectively. We find that
\begin{equation*}
\begin{split}
    \psi^*(L - uZ) &= \psi^*(-K_\tX - \dee_{\tX} + (3-3c-u) Z) \\
    &= -K_\oX - \dee_{\oX} +\tfrac{(4-4c-u)}{2} \oY +(3-3c-u) \oZ + (1-c) E_F + E_S .
\end{split}
\end{equation*}

We compute the Nakayama--Zariski decomposition $\psi^*(L - uZ)  = P(u) + N(u)$:
\[ \resizebox{1\textwidth}{!}{ $ \begin{array}{rl}
     0 \le u \le 3-2c: & P_{\regionA}(u) = -K_\oX - \dee_{\oX} + \tfrac{(4-4c-u)}{2} \oY +(3-3c-u) \oZ + (1-c) E_F + E_S \\
    3-2c \le u \le 4-4c: & P_{\regionB}(u) = -K_\oX - \dee_{\oX} + \tfrac{(4-4c-u)}{2} \oY +(3-3c-u) \oZ + (4-3c-u) E_F + E_S \\
     4-4c \le u \le 5-4c: & P_{\regionC}(u) =-K_\oX - \dee_{\oX} +(4-4c-u) \oY +(3-3c-u) \oZ + (4-3c-u) E_F + (5-4c-u)E_S 
\end{array} $} \]
and 
\[ \begin{array}{rl}
    0 \le u \le 3-2c: & N_{\regionA}(u) =0  \\
    3-2c \le u \le 4-4c: & N_{\regionB}(u) =  (u-3+2c) E_F \\
    4-4c \le u \le 5-4c: & N_{\regionC}(u) = (u-3+2c) E_F + \tfrac{(u-4+4c)}{2}(\oY + 2 E_S).
\end{array} \]

We first compute $\frac{A_{\exx,\dee}(Z)}{S_{\exx,\dee}(Z)}$.  From above, $A_{\exx,\dee}(Z) = 4-3c$.  To compute $S_{\exx,\dee}(Z)$, we find the volumes of $P_{\regionA}(u), P_{\regionB}(u), P_{\regionC}(u)$ above. First, on region \(\regionA\), since $\phi$ is a blow-up of a point, $L|_Z = 0$, and $Z|_Z = - f$ where $f$ is a ruling of $\bb P(1,1,2)$.  Thus, $Z^3 = (Z|_Z)^2 = \frac{1}{2}$ so
\[ P_{\regionA}(u)^3 = (L - uZ)^3 = 3(2-2c)(3-2c)^2 - \tfrac{u^3}{2}.\]

Now, $P_{\regionB}(u) = \nu^*(\mu^*(L - uZ) -(u-3+2c) E_F)$, where $\mu\colon \hX_1 \to \tX$ is the blow-up of $l_F$, and $\psi\colon \oX \to \tX$ factors as $\nu \circ \mu$. Then, for $P_{\regionC}(u)$ we note that $P_{\regionC}(u) = P_{\regionB}(u) -  \tfrac{(u-4+4c)}{2}(\oY + 2 E_S) $. So 
by applying Lemma~\ref{lem:volume-lemma-4} with \(p=-(u-3+2c)\) on region \(\regionB\), and Lemma~\ref{lem:volume-lemma-5} with \(p=\tfrac{-(u-4+4c)}{2}\) on region \(\regionC\), we get \(P_{\regionB}(u)^3 = 3 (2-2c)(3-2c)^2 - \tfrac{u^3}{2} + (u-3+2c)^3\) and \(P_{\regionC}(u)^3 = 3 (2-2c)(3-2c)^2 - \tfrac{u^3}{2} + (u-3+2c)^3  + \tfrac{(u-4+4c)^3}{2}\).
Thus, we find that 
\[ S_{\exx,\dee}(Z) = -\frac{2 (17 c^3 - 73 c^2 + 104 c - 49)}{3 (3 - 2 c)^2} . \]
Using that $A_{\exx,\dee}(Z) = 4-3c$, we obtain that $\frac{A_{\exx,\dee}(Z)}{S_{\exx,\dee}(Z)} >1$ for all $c \in(0,1)$. 

\subsubsection{Computation of \(\delta_q(Z, \dee_Z; V_{\bullet \bullet})\) for \(q\in Z\)}\label{sec:local-delta-A_n:A_2-part-2}
Next, we move to computing a lower bound on $\delta$ using flags coming from curves on $Z$.
\begin{detail}
For each point \(q\in Z\), we will choose a plt type divisor $C$ over $Z$ with center on $Z$ containing $q$.  If $\alpha\colon Z' \to Z$  is the plt blow-up extracting $C$, then
\[ \delta_q(Z, \dee_Z; V_{\bullet \bullet}) \ge \min \left\{ \frac{A_{Z, \dee_Z}(C)}{S_{Z,\dee_Z; V_{\bullet \bullet}}(C)}, \inf_{
     c \in C :
     \alpha(c) = q } \delta_c(C,\dee_C; W_{\bullet \bullet \bullet}) \right\} \]
where $W_{\bullet \bullet \bullet}$ is the further refinement of $V_{\bullet \bullet}$ by $C$ and $\dee_C$ is the restriction of $\dee_Z$ to $C$.
\end{detail}
In what follows, we continue to assume that $c \le \frac{1}{2}$ so the formulas for $P(u)$ and $N(u)$ derived above still hold.  For any $q \in Z$, we will use $C$ a ruling of the cone $Z = \bb P(1,1,2)$ through $q$.  There are two cases: $C$ is a general ruling, and $C$ is the unique ruling that meets $l_F$ in $\tX$.  Let \(z_0\) denote this intersection point from the second case.

By \cite[Theorem 4.8]{Fujita-3.11}, we compute these quantities on $\oZ$, the strict transform of $Z$ in the common resolution $\oX$ defined above.  In this case, $\oZ \cong \Bl_{z_0} \bb F_2$, and the Picard group is generated by $\sigma$ the strict transform of the negative section of $\bb F_2$, $f$ the strict transform of the fiber through $z_0$, and $e$ the exceptional divisor of the blow-up of $z_0$. By abuse of notation, we will denote by $\psi\coloneqq \psi|_{\oZ} \colon \oZ \to Z$. Using the notations and computations of $P(u)$ and $N(u)$ from Section~\ref{sec:local-delta-A_n:A_2-part-1}, as $E_F|_\oZ = e$, $\oY|_\oZ = \sigma$, and $E_S|_\oZ = 0$, we first observe that for the strict transform \(\oC\) of any ruling \(C\) of $Z$, we have \(\ord_\oC(N(u)|_\oZ) = 0 \) on all regions \(\regionA,\regionB,\regionC\) defined in Section~\ref{sec:local-delta-A_n:A_2-part-1}.
\begin{detail}
Therefore, by \cite[Theorem 4.8]{Fujita-3.11}, we have 
\[ S_{Z,\dee_Z; V_{\bullet \bullet}}(\oC) = \frac{3}{\vol L} 
\int_0^{5-4c} \left( \int_0^\infty \vol( P(u)|_\oZ - v\oC) dv \right) du. \]

We will compute this using a Zariski decomposition of the divisor $P(u)|_\oZ - v\oC$ on $\oZ$, where $P(u)$ is as computed above in Section~\ref{sec:local-delta-A_n:A_2-part-1}. Let $t(u)$ be the pseudoeffective threshold.  Let $P(u,v)$ be its positive part and $N(u,v)$ its negative part.  Then, the formula for $S$ becomes:
\[ S_{Z,\dee_Z; V_{\bullet \bullet}}(C) = \frac{3}{\vol L} 
\int_0^{5-4c} \left( \int_{0}^{t(u)} P(u,v)^2 dv \right) du . \]

Furthermore, by \cite[Theorem 4.17]{Fujita-3.11}, we also have the formulas for $W_{\bullet \bullet \bullet}$.
Since
\[\Sigma \coloneqq \psi^* C - \oC = \begin{cases}
    \sigma/2 & C \text{ is a general ruling} \\
    \sigma/2 + e & C = \psi_*f \text{ is the ruling containing }z_0
\end{cases}\]
we have
\[ S_{C,\dee_C; W_{\bullet \bullet \bullet}}(q) = \frac{3}{\vol L} 
\int_0^{5-4c} \left( \int_{0}^{t(u)} (P(u,v)\cdot \oC)^2 dv \right) du + F_q(W_{\bullet \bullet \bullet}) \]
where 
\[ F_q(W_{\bullet \bullet \bullet})  = \frac{6}{\vol L} 
\int_0^{5-4c} \left( \int_{0}^{t(u)} (P(u,v)\cdot \oC)\ord_q((N(u)|_\oZ + N(u,v) - v\Sigma)|_\oC) dv \right) du. \]
\end{detail}
Using the formulas for $P$ and $N$ from Section~\ref{sec:local-delta-A_n:A_2-part-1}, on \(\oZ \cong \Bl_{z_0} \bb F_2\) we have
\[ \resizebox{1\textwidth}{!}{ $ \begin{array}{rll}
    0 \le u \le 3-2c: & P_{\regionA}(u)|_\oZ = u(\sigma/2 + f + e) & N_{\regionA}(u)|_\oZ  =  0 \\
    3-2c \le u \le 4-4c: & P_{\regionB}(u)|_\oZ = u(\sigma/2 + f) + (3-2c)e & N_{\regionB}(u)|_\oZ = (u-3+2c) e \\
    4-4c \le u \le 5-4c: & P_{\regionC}(u)|_\oZ = (2-2c)\sigma + uf+ (3-2c)e & N_{\regionC}(u)|_\oZ  =(u/2-2+2c)\sigma + (u-3+2c) e.
\end{array} $} \]

\noindent\emph{Flag from a general ruling of \(Z\).}
Let us first assume that $C$ is a general ruling of $Z$, so $\oC = e + f$.  Then we find the Zariski decompositions for $P(u)|_\oZ - v\oC$: 

Starting in region $\regionA$ for $0 \le u \le 3-2c$, $P_{\regionA}(u)|_\oZ - v\oC = u(\sigma/2) + (u-v)(f+e)$.  Here, $t(u) = u$, and we find the Zariski decomposition for \(0 \le v \le u \):
\[ P_{\regionA}(u,v) = (u-v)(\tfrac{1}{2}\sigma + f + e) , \qquad N_{\regionA}(u,v) = \tfrac{v}{2}\sigma .\]
\begin{detail}
So the relevant quantities for $0 \le u \le 3-2c$ are: 
\[\begin{array}{r|l}
     & 0 \le v \le u \\
     \hline
    P_{\regionA}(u,v)^2 & (u-v)^2/2 \\
    P_{\regionA}(u,v) \cdot \oC & (u-v)/2 \\
    (P_{\regionA}(u,v) \cdot \oC)^2 & (u-v)^2/4 \\
    \ord_q((N_{\regionA}'(u)|_\oZ + N_{\regionA}(u,v) - v(\sigma/2))|_\oC) &  0 
\end{array}\]
\end{detail}
For region $\regionB$ for $3 -2c \le u \le 4-4c$, $P_{\regionB}(u)|_\oZ - v\oC = u(\sigma/2) + (u-v)f + (3-2c-v)e$ and $t(u) = 3-2c$.  Here, the Zariski decomposition is 
\begin{equation*}\begin{split}
    P_{\regionB}(u,v) &= \begin{cases}
        (u-v)(\frac{1}{2}\sigma + f) + (3-2c-v)e \hphantom{t+\sigma e} & 0 \le v \le 6-4c-u  \\
        (3-2c-v)(\sigma+2f + e) & 6-4c-u \le v \le 3-2c ,
    \end{cases} \\
    N_{\regionB}(u,v) &= \begin{cases}
        \frac{v}{2}\sigma & 0 \le v \le  6-4c-u   \\
        (\frac{u}{2}+v-3+2c)\sigma + (u-6+4c-v)f &  6-4c-u  \le v \le 3-2c.
    \end{cases}
\end{split}\end{equation*}
\begin{detail}
So the relevant quantities for $3-2c \le u \le 4-4c$ are: 
\[\resizebox{1\textwidth}{!}{ $\begin{array}{r|ll}
   & 0 \le v \le 6-4c-u & 6-4c-u \le v \le 3-2c \\
       \hline
   (P_{\regionB}(u,v))^2   & -(u-v)^2/2+2(u-v)(3-2c-v) - (3-2c-v)^2 & (3-2c-v)^2 \\
   P_{\regionB}(u,v)\cdot \oC  & (u-v)/2 & 3-2c-v \\
   (P_{\regionB}(u,v)\cdot \oC)^2  & (u-v)^2/4 & (3-2c-v)^2 \\
   \ord_q((N_{\regionB}'(u)|_\oZ + N_{\regionB}(u,v) - v(\sigma/2))|_\oC) & 0 &
   \left\{ \begin{array}{ll}
      0  & q \notin \sigma  \\
     u/2+v/2-3+2c   & q \in \sigma
   \end{array} \right. \\
\end{array} $} \]
\end{detail}
For region $\regionC$ with $4-4c \le u \le 5-4c$, $P_{\regionC}(u)|_\oZ - v\oC = (2-2c)\sigma + (u-v)f + (3-2c-v) e$ and $t(u) = 3-2c$.  
Here, the Zariski decomposition is 
\begin{equation*}\begin{split}
    P_{\regionC}(u,v) &= \begin{cases}
        (2-2c)\sigma + (u-v)f + (3-2c-v) e & 0 \le v \le u-4+4c  \\
        (u-v)(\frac{1}{2}\sigma+f) + (3-2c-v) e &  u-4+4c \le v \le 6-4c-u \\
        (3-2c-v)(\sigma+2f + e) & 6-4c- u \le v \le 3-2c ,
    \end{cases} \\
    N_{\regionC}(u,v) &= \begin{cases}
        0 \hphantom{(2-c)\sigma + (u-v)f + (3-2c-v) e} & 0 \le v \le u-4+4c   \\
        \frac{1}{2}(v-u+4-4c)\sigma &  u-4+4c \le v \le 6-4c-u \\
        (v-1)\sigma + (u-6+4c+v)f & 6-4c- u \le v \le 3-2c .
    \end{cases}
\end{split}\end{equation*}
\begin{detail}
So the relevant quantities for $4-4c \le u \le 5-4c$ are: 
\[ \resizebox{1\textwidth}{!}{ $\begin{array}{r|lll}
   & 0 \le v \le u-4+4c &  u-4+4c \le v \le 6-4c-u & 6-4c- u \le v \le 3-2c \\
       \hline
   (P_{\regionC}(u,v))^2 & 2(2-2c)(u-v-2+2c)-(u-3+2c)^2   & -(u-v)^2/2+2(u-v)(3-2c-v) - (3-2c-v)^2 & (3-2c-v)^2 \\
   P_{\regionC}(u,v)\cdot \oC & 2-2c & (u-v)/2 & 3-2c-v \\
   (P_{\regionC}(u,v)\cdot \oC)^2 & (2-2c)^2 & (u-v)^2/4 & (3-2c-v)^2 \\
   \ord_q((N_{\regionC}'(u)|_\oZ + N_{\regionC}(u,v) - v(\sigma/2))|_\oC) & 
   \left\{ \begin{array}{ll}
        0 & q \notin \sigma \\
      u/2 - 2 + 2c - v/2  & q \in \sigma  
   \end{array} \right. & 0  & \left\{ \begin{array}{ll}
      0  & q \notin \sigma  \\
      u/2+v/2-3+2c   & q \in \sigma \\
   \end{array} \right. 
\end{array}$ } \]
\end{detail}
We have \(A_{Z,\dee_Z}(C) = 1\) because $R$ had an $A_2$ singularity, and hence the strict transform of $R$ in $\tX_1$ is smooth.  In particular, when constructing the blow-up $\tX$, we blow-up a Cartier divisor in the strict transform of $R$, so the curve $R|_{Z_2} \subset Z_2 \cong \bF_2$ is an irreducible section.  This implies that $R|_Z$ is also irreducible and contains no rulings of $Z$.  
Therefore, using \cite[Theorem 4.8]{Fujita-3.11}, we compute
\[S_{Z,\dee_Z; V_{\bullet \bullet}}(C) = \frac{3-2c}{3}, \qquad  \frac{A_{Z,\dee_Z}(C)}{S_{Z,\dee_Z; V_{\bullet \bullet}}(C)} = \frac{3}{3-2c} > 1  \]
for \(c\in(0,\frac{1}{2}]\).
\begin{remark}
     This is where we need that \(\pt\) is an $A_2$ singularity. If it were $A_3$, then the curve $\dee_Z$ would be one ruling plus a section of the cone; this computation for the fiber would then have $A_{Z,\dee_Z}(C)= 1-c$, so we would have $A_{Z,\dee_Z}(C)/S_{Z,\dee_Z; V_{\bullet \bullet}}(C) < 1$ for \(c>0\).
\end{remark}

Next we compute $S_{C,\dee_C; W_{\bullet \bullet \bullet}}(q)$ for points \(q\in\oC\) that are not contained in \(\sigma\):  If $q \notin \sigma$ then 
\[ S_{C,\dee_C; W_{\bullet \bullet \bullet}}(q) =  \frac{(1-c)(6c^2-20c+17)}{3(3-2c)^2}, \]
using \cite[Theorem 4.17]{Fujita-3.11}. The log discrepancy $A_{C,\dee_C}(q)$ is either $1$ or $1-c$, depending on whether or not $q \in \dee_C$, and we find
\[ \frac{A_{C,\dee_C}(q)}{S_{C,\dee_C; W_{\bullet \bullet \bullet}}(q)} \ge  \frac{3(3-2c)^2}{6c^2-20c+17} > 1\] for \(c\in (0,\frac{1}{2}]\). This gives a lower bound for points \(q\) not contained in \(\sigma\) or \(f\).

\noindent\emph{Flag from the ruling of \(Z\) containing \(z_0\).}
If $q \in \sigma$
or for all other points on $f$, we will use a different flag.  In this case, we use $\oC = f$. Then we find the Zariski decompositions for $P(u)|_\oZ - v\oC$: 

Starting in region $\regionA$ for $0 \le u \le 3-2c$, $P_{\regionA}(u)|_\oZ - v\oC = u(\sigma/2+e) + (u-v)f$.  Here, $t(u) = u$, and we find the Zariski decomposition for \( 0 \le v \le u\):
\[ P_{\regionA}(u,v) = (u-v)(\tfrac{\sigma}{2} + f + e) , \qquad N_{\regionA}(u,v) = v(\tfrac{\sigma}{2}+e) . \]
\begin{detail}
So the relevant quantities for $0 \le u \le 3-2c$ are: 
\[\begin{array}{r|l}
     & 0 \le v \le u \\
     \hline
    P_{\regionA}(u,v)^2 & (u-v)^2/2 \\
    P_{\regionA}(u,v) \cdot \oC & (u-v)/2 \\
    (P_{\regionA}(u,v) \cdot \oC)^2 & (u-v)^2/4 \\
    \ord_q((N_{\regionA}'(u)|_\oZ + N_{\regionA}(u,v) - v(\sigma/2+e))|_\oC) &  0 
\end{array}\]
\end{detail}
For region $\regionB$ for $3 -2c \le u \le 4-4c$, $P_{\regionB}(u)|_\oZ - v\oC = u(\sigma/2) + (u-v)f + (3-2c)e$ and $t(u) = u$.  Here, the Zariski decomposition is 
\begin{equation*}\begin{split}
    P_{\regionB}(u,v) &= \begin{cases}
        (u-v)(\frac{1}{2}\sigma + f) + (3-2c)e & 0 \le v \le u-3+2c  \\
        (u-v)(\frac{1}{2}\sigma + f + e) & u-3+2c \le v \le u ,
    \end{cases} \\
    N_{\regionB}(u,v) &= \begin{cases}
        \frac{v}{2}\sigma &  0 \le v \le u-3+2c  \\
        \frac{v}{2}\sigma + (v-u+3-2c)e \hphantom{(2c)v} &  u-3+2c \le v \le u .
    \end{cases}
\end{split}\end{equation*}
\begin{detail}
So the relevant quantities for $3-2c \le u \le 4-4c$ are: 
\[ \resizebox{1\textwidth}{!}{ $ \begin{array}{r|ll}
   &  0 \le v \le u-3+2c & u-3+2c \le v \le u \\
       \hline
   (P_{\regionB}(u,v))^2   & -(u-v)^2/2+2(u-v)(3-2c) - (3-2c)^2 & (u-v)^2/2 \\
   P_{\regionB}(u,v)\cdot \oC  & v/2-u/2+3-2c & (u-v)/2 \\
   (P_{\regionB}(u,v)\cdot \oC)^2  & (v/2-u/2+3-2c)^2 & (u-v)^2/4 \\
   \ord_q((N_{\regionB}'(u)|_\oC + N_{\regionB}(u,v) - v(\sigma/2+e))|_\oC) &  
   \left\{ \begin{array}{ll}
      0  & q \notin e \\
     u-v-3+2c  & q \in e
   \end{array} \right. & 0 \\
\end{array} $} \]
\end{detail}
For region $\regionC$ with $4-4c \le u \le 5-4c$, $P_{\regionC}(u)|_\oZ - v\oC = (2-2c)\sigma + (u-v)f + (3-2c) e$ and $t(u) = u$.  
Here, the Zariski decomposition is 
\begin{equation*}\begin{split}
    P_{\regionC}(u,v) &= \begin{cases}
        (2-2c)\sigma + (u-v)f + (3-2c) e \hphantom{\sigma '+\sigma} & 0 \le v \le u-4+4c  \\
        (u-v)(\frac{1}{2}\sigma+f) + (3-2c) e &  u-4+4c \le v \le u-3+2c \\
        (u-v)(\frac{1}{2}\sigma+f + e) &  u - 3+2c \le v \le u ,
    \end{cases} \\
    N_{\regionC}(u,v) &= \begin{cases}
        0 & 0 \le v \le u-4+4c   \\
        \frac{1}{2}(v-u+4-4c)\sigma &  u-4+4c \le v \le  u-3+2c \\
        \frac{1}{2}(v-u+4-4c)\sigma + (v-u+3-2c)e &  u-3+2c \le v \le u .
    \end{cases}
\end{split}\end{equation*}
\begin{detail}
So the relevant quantities for $4-4c \le u \le 5-4c$ are: 
\[ \resizebox{1\textwidth}{!}{ $\begin{array}{r|lll}
   & 0 \le v \le u-4+4c &  u-4+4c \le v \le u-3+2c & u-3+2c \le v \le u \\
       \hline
   (P_{\regionC}(u,v))^2 & 2(2-2c)(u-v-2+2c)-(u-v-3+2c)^2   & -(u-v)^2/2+2(u-v)(3-2c) - (3-2c)^2 & (u-v)^2/2 \\
   P_{\regionC}(u,v)\cdot \oC & 5-4c-u+v & v/2-u/2+3-2c & (u-v)/2 \\
   (P_{\regionC}(u,v)\cdot \oC)^2 & (5-4c-u+v)^2 & (v/2-u/2+3-2c)^2 & (u-v)^2/4 \\
   \ord_q((N_{\regionC}'(u)|_\oZ + N_{\regionC}(u,v) - v(\sigma/2))|_\oC) & 
   \left\{ \begin{array}{ll}
        0 & q \notin \sigma, e \\
      u/2 - 2 + 2c - v/2  & q \in \sigma  \\
      u-v -3+2c  & q \in e
   \end{array} \right.   & \left\{ \begin{array}{ll}
      0  & q \notin e  \\
      u-v -3+2c   & q \in e \\
   \end{array} \right. & 0
\end{array}$ } \]
\end{detail}
Therefore, we may compute using \cite[Theorem 4.8]{Fujita-3.11} to obtain: 
\[ S_{Z,\dee_Z; V_{\bullet \bullet}}(C) = \frac{-14 c^3 + 58 c^2 - 80 c + 37}{3 (3 - 2 c)^2}. \]
As $A_{Z,\dee_Z}(C) = 1$, we compute 
\(A_{Z,\dee_Z}(C)/S_{Z,\dee_Z; V_{\bullet \bullet}}(Z)  > 1\) for $c \in (\cone,\frac{1}{2}]$, where \(\cone \approx 0.3293\) is the smallest root of \(7 c^3 - 23 c^2 + 22 c - 5 = 0\).

Finally, we find $S_{C,\dee_C; W_{\bullet \bullet \bullet}}(q)$ for $q \in \oC$, using \cite[Theorem 4.17]{Fujita-3.11} and the above Zariski decompositions. We have three cases, and we compute that:
\[ S_{C,\dee_C; W_{\bullet \bullet \bullet}}(q) = \begin{cases}
    \displaystyle\frac{32 c^4 - 176 c^3 + 360 c^2 - 324 c + 107}{24 (1 - c) (3 - 2 c)^2} & q \notin \sigma, e \\
    \displaystyle\frac{(1 - c) (6 c^2 - 20 c + 17)}{3 (2 c - 3)^2} & q \in e \\
    \displaystyle\frac{3-2c}{6} & q \in \sigma .
\end{cases}\]
The log discrepancies are $A_{C,\dee_C}(q) = 1$ if $q \notin \sigma, e$ (as the intersection points with $\dee_C$ are on $\sigma$ and $e$), $A_{C,\dee_C}(q) = 1-c$ if \(q\in e\), and $A_{C,\dee_C}(q)$ is $1-\frac{c}{2}$ if \(q\in\sigma\). So for all cases of \(q\in\oC\), we have \[ \frac{ A_{C,\dee_C}(q) }{ S_{C,\dee_C; W_{\bullet \bullet \bullet}}(q) } > 1 \text{ for all }c \in (0,\tfrac{1}{2}]. \]

This proves that $\delta_{\pt}(\exx,\dee) > 1$ for $c \in (\cone,\frac{1}{2}]$. Combining this with Section~\ref{sec:local-delta-A_n:higher-part-1}, where we proved that $\delta_{\pt}(\exx,\dee) > 1$ for $c \in (0,\cnot)$, \(\cone \approx 0.33 < \cnot \approx 0.47\) we find that, for $\pt$ an $A_2$ singular point on $R$ such that $R$ does not contain the fiber of the second projection at $\pt$, $\delta_{\pt}(\exx,\dee) > 1$ for all $c \in (0,\frac{1}{2}]$. This completes the proof of Theorem~\ref{thm:local-delta-A_n-singularities}\eqref{part:local-delta-A_n-singularities:A_2}.

\subsection{Higher \texorpdfstring{\(A_n\)}{An} singularities for \texorpdfstring{\(0.39 < c \leq \frac{1}{2}\)}{0.39 < c leq 1/2}}\label{sec:local-delta-A_n:higher-part-2}

Finally, in this section we show Theorem~\ref{thm:local-delta-A_n-singularities}\eqref{part:local-delta-A_n-singularities:higher-A_n}. Let $R$ be a $(2,2)$-surface with an $A_n$ singularity at $\pt \in R$, \(n\geq 3\).  Assume that the fiber of \(\piR_2 \colon R\to\bb P^2\) at $\pt$ is finite and that the fiber of \(\piR_1 \colon R\to\bb P^1\) at $\pt$ is reduced.  Let $\dee = cR$.  We aim to show $\delta_p(X,\dee) > 1$ for all $c \in (0, \tfrac{1}{2}]$.  From the earlier Section~\ref{sec:local-delta-A_n:higher-part-1}, we know  $\delta_{\pt}(\exx,\dee) > 1$ for all \(c\in(0,\cnot)\), where \(\cnot \approx 0.47\), so it suffices to show a lower bound for \(c\in (0.39, \tfrac{1}{2}]\).

\subsubsection{Set-up and computation of $\delta(Z)$ for a natural divisor $Z$ over the singular point}\label{sec:local-delta-A_n:higher-A_n-part-1}
First, we construct a divisor $Z$ of plt type over $\exx = \bb P^1 \times \bb P^2$ (see Figure~\ref{fig:A_n-plt-blow-up}).

Let $F \cong \bb P^2$ be the fiber of the first projection to $\bb P^1$ containing $\pt$.
Let $\tX_1 \to X$ be the blow-up of $\pt$ with exceptional divisor $Y \cong \bb P^2$.  The intersection $R_Y$ of $Y$ with the strict transform of $R$ is a rank 2 conic with singular point $q_1$.  Let $l_1'$ be the strict transform of the fiber of the second projection of $X \to \bb P^2$ at $\pt$, and let $y_1 \in Y$ be the intersection point of $l_1'$ and $Y$.  Let $S_0 \cong \bb P^1 \times \bb P^1$ be the unique $(0,1)$-surface on $\exx$ through $\pt$ such that its strict transform on $\tX_1$ contains $l_1'$ and $q_1$.

Let $\tX_2 \to \tX_1$ be the blow-up of $q_1$ with exceptional divisor $Z_2$.  Let $R_{Z_2}$ denote the intersection of \(Z_2\) with the strict transform of $R$, and let $Y_2$ denote the strict transform of $Y$ in $\tX_2$. Because $\pt$ is an $A_n$ singularity of $R$ for $n \ge 3$, $R_{Z_2}$ is a smooth conic (if \(n=3\)) or a union of two distinct lines in $Z_2$ meeting the strict transform of the two lines $R_Y$ at two distinct points along $Y_2 \cap Z_2$ (if \(n\geq 4\)).  Next, from $\tX_2$ we may contract $Y_2$ in a morphism $\phi_2\colon \tX_2 \to \tX$, and $\tX$ has canonical singularities along the image of $Y_2$.
\begin{detail}
Indeed, $Y_2 \cong \bb F_1$ with negative section $Y_2 \cap Z_2$.  The rulings of $Y_2$ are contractible, and $\phi_2$ contracts $Y_2$ onto its negative section. By computation, a ruling of $Y_2$ has trivial intersection with $K_{\tX_2}$, showing the claim about canonical singularities.
\end{detail}
Let $Z \cong \bb P^2$ be the image of $Z_2$ in $\tX$.  By construction, $\tX$ admits a morphism $\phi \colon \tX \to \exx$ contracting $Z$ to a point.  Let $\dee_\tX$ be the strict transform of $\dee$ on $\tX$.  One sees that $-Z$ is $\phi$-ample, and $(\tX, Z+\dee_\tX)$ is plt for $c < 1$ as $\dee_\tX$ is not contained in the singular locus of $\tX$ and we have constructed $\phi \colon \tX \to \exx$ with a divisor of plt type over $(\exx, \dee)$.

\begin{figure}[h]
\centering
    \adjustbox{width=.5\textwidth}{
    \begin{tikzcd}
    \begin{tabular}{c}
\begin{tikzpicture}

\draw[teal,thick] (0.5,0) -- (0.7,-.5);
\draw[dashed, thick] (1.5,0.7) -- (0.7,0.7) -- (0,0);
\draw[thick] (-.7,-.7) -- (-.7,1.3) -- (0.7,2.7)--(0.7,0.7);
\draw[thick] (0,0) -- (-.7,-.7) -- (1.3,-.7) -- (2,0) -- (2.7,0.7) -- (0.7,0.7);
\draw[olive,thick] (0,0) -- (1.5,0) -- (0,1.5) -- (0,0);
\draw[blue,fill=blue] (1,0) circle (0.05);

\node[below, node font=\tiny] at (.3,-.7) {$S$};
\node[left, node font=\tiny] at (-.7,.5) {$F$};
\node[right, node font=\tiny] at (0.7,-.5) {$l_1$};
\node[above, node font=\tiny] at (1,0) {$q_1$};
\node[above right, node font=\tiny] at (0.7,0.7) {$Y$};

\end{tikzpicture}
\end{tabular} \arrow[d, swap, "\text{blow up }\pt"] & &\begin{tabular}{c}
\begin{tikzpicture}

\draw[teal,thick] (0.5,0) -- (0.7,-.5);
\draw[thick] (-.7,-.7) -- (-.7,1.3) -- (0.7,2.7)--(0.7,1.5);
\draw[thick,dashed] (0.7,1.5)--(0.7,0.7);
\draw[thick] (0,0) -- (-.7,-.7) -- (1.3,-.7) -- (2,0) -- (2.7,0.7) -- (1.5,0.7); 
\draw[dashed, thick] (1.5,0.7) -- (0.7,0.7) -- (0,0);
\draw[olive,thick] (1.4,0) -- (1.5,0) -- (1.5,1.5) -- (0,1.5) -- (0,0)--(1,0);
\draw[olive,thick,dashed] (1,0) -- (1.4,0);
\draw[blue, thick] (1,0) -- (1.4,-.5) -- (1.4,1) -- (1,1.5) -- (1,0);

\node[below, node font=\tiny] at (.3,-.7) {$S_0$};
\node[left, node font=\tiny] at (-.7,.5) {$F$};
\node[right, node font=\tiny] at (0.7,-.5) {$l_1$};

\end{tikzpicture}
\end{tabular} \arrow[ll,swap, "\text{blow up }q_1"] \arrow[d, "\text{contract } Y"]  \\
    \begin{tabular}{c}
\begin{tikzpicture}

\draw[teal,thick] (0,0) -- (1.5,0);
\draw[thick] (-.7,-.7) -- (-.7,1.3) -- (0.7,2.7)--(0.7,0.7);
\draw[thick] (0,0) -- (-.7,-.7) -- (1.3,-.7) -- (2,0) -- (2.7,0.7) -- (0.7,0.7) -- (0,0);
\draw[olive,fill=olive] (0,0) circle (0.05);

\node[below, node font=\tiny] at (.3,-.7) {$S$};
\node[left, node font=\tiny] at (-.7,.5) {$F$};
\node[right, node font=\tiny] at (1.5,0) {$l_1$};

\end{tikzpicture}
\end{tabular} & & \begin{tabular}{c}
\begin{tikzpicture}

\draw[teal,thick] (0,0) -- (0.7,-.5);
\draw[thick] (-.7,-.7) -- (-.7,1.3) -- (0.7,2.7)--(0.7,1.5);
\draw[thick,dashed] (0.7,1.5)--(0.7,0.7);
\draw[thick] (0,0) -- (-.7,-.7) -- (1.3,-.7) -- (2,0) -- (2.7,0.7) -- (1.5,0.7); 
\draw[dashed, thick] (1.5,0.7) -- (0.7,0.7) -- (0,0);
\draw[line width=0.1cm, olive!30] (0,0) -- (0,1.5);
\draw[olive,semithick] (0,0) -- (0,1.5);
\draw[blue,thick] (0,0) -- (1.5,0) -- (1.5,1.5) -- (0,1.5);

\node[below, node font=\tiny] at (.3,-.7) {$S_0$};
\node[left, node font=\tiny] at (-.7,.5) {$F$};
\node[right, node font=\tiny] at (0.7,-.5) {$l_1$};
\node[above, node font=\tiny] at (1.5,1.5) {$Z$};

\end{tikzpicture}
\end{tabular} \arrow[ll, "\text{plt blow-up } \phi"] 
    \end{tikzcd}
    }
\caption{Construction of the plt type divisor \(Z\) on \(\tX\) in Section~\ref{sec:local-delta-A_n:higher-A_n-part-1}.}\label{fig:A_n-plt-blow-up}
\end{figure}
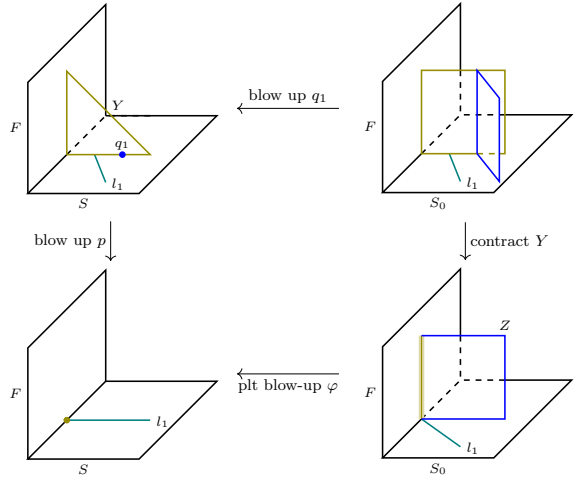

\begin{remark}
    The assumption that the fiber of \(\piR_1\colon R\to\bb P^1\) at \(\pt\) is reduced implies that \(q_1\) is not in the strict transform of \(F\) in \(\tX_1\).
    The assumption that the fiber of \(\piR_2\colon R\to\bb P^2\) at \(\pt\) is finite implies that \(y_1 \notin R_Y\).
\end{remark}

We will use the divisor $Z$ and the Abban--Zhuang method to compute the stability of $(\exx, \dee)$ at \(\pt\). We first compute \(A_{\exx,\dee}(Z) / S_{\exx,\dee} (Z)\). The log discrepancy is $A_{\exx,\dee}(Z) = 5 - 4c$.

To find $S_{\exx,\dee} (Z)$, let \(L \coloneqq \phi^*(-K_{\exx} - \dee) = - K_\tX - \dee_\tX + (4-4c) Z \). We will compute the volume \(\vol(L-uZ)\) on a birational model. Note that $\phi_2^*L = -K_{\tX_2} - \dee_{\tX_2} + (2-2c)Y_2 + (4-4c)Z_2$ and $\phi_2^*Z = Z_2 + \frac{1}{2} Y_2$. Using $-K_{\exx} - \dee \sim (2-2c)F + (3-2c)S_0$, it is straightforward to compute that the pseudoeffective threshold of $L - uZ$ is \(u = 8-6c\).

To compute the volume, we find a Nakayama--Zariski decomposition of $\phi_2^*(L-uZ)$ on $\tX_2$.  By Lemma~\ref{lem:moricone} and similar computations, the Mori cone of $\tX_2$ is generated by $l_1, l_2, l_3, l_4$ where $l_1$ is the strict transform of the fiber of $X \to \bb P^2$ at $\pt$, $l_2$ is the strict transform of a line in $F$ through $\pt$, $l_3$ is the strict transform of a line in $Y$ through $q_1$, and $l_4$ is a line in $Z$.  The pieces of the Nakayama--Zariski decomposition will correspond to the following birational transformations: before $u = 4-4c$, the divisor is ample; at $u = 4-4c$, we flop $l_1$ from $S_0$ into $Z$; at $u = 5 - 4c$, we contract the strict transform of $S_0$ onto a curve in the strict transform of $Z$; at $u = 6-4c$, we contract the strict transform of $F$ onto a curve not contained in the strict transform of $Z$ and we stop at $u = 8-6c$, corresponding to a contraction of the strict transform of $\tX$ to a point.

\begin{enumerate}
    \item At \(u=4-4c\), we flop \(l_1\) from \(S_0\) into \(Z\). This flop is constructed in Section~\ref{construction:complicatedflop}.  
    This gives \(\tX \dashrightarrow \tX^+\).  Let $\oX$ be the resolution of this rational map, where $\Delta_1, \Delta_2$, and $\Delta_3$ the exceptional divisors as in Lemma \ref{lem:volume-lemma-6}.
    \item At \(u=5-4c\), we have the contraction \(g\colon\tX^+ \to \tX^{(2)}\) of the strict transform of \(S_0\) onto a curve in the strict transform of \(Z\).
    \item At \(u=6-4c\), we have the contraction $h\colon \tX^{(2)} \to \bb P^3$ of the strict transform of $F$. (The final model is \(\bb P^3\) because it is smooth and has Picard rank 1, and one can directly compute that it has Fano index 4.)
    The image of $Z$ is a hyperplane $H_Z$.
\end{enumerate}

We may more simply describe the birational models after the flop as the composition $h\colon \tX^{(2)} \to \bb P^3$ of the blow-up of a line in $H_Z$ with exceptional divisor $E \cong \bb P^1 \times \bb P^1$, followed by the blow-up of a $(1,1)$-surface in $E$, followed by the contraction of the strict transform of $E$ (contracting the same fibers originally extracted).  Then, $g$ is the blow-up of a line in the strict transform of $H_Z$ meeting the exceptional locus but not the singular locus of $\tX^{(2)}$. 

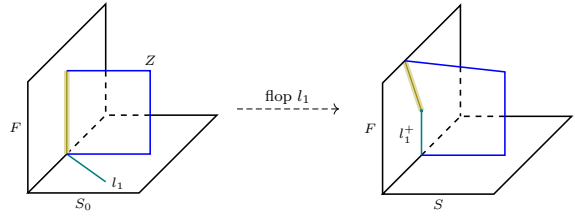
\[\begin{tikzcd}
\tX_1 \ar[d, "\text{blow up }\pt" {swap}] & & \tX_2 \ar[ll, "\text{blow up }q_1" {swap}] \ar[d, "\text{contract } Y_2" {swap}] & \oX \arrow[rd] \arrow[ld, swap, "\psi"] &  &  & & &  \\
\exx = \bb P^1\times\bb P^2 & & \tX \arrow[ll, "\phi" {swap}, "\text{extract }Z"] \arrow[rr, dashed, "\text{flop }l_1" {swap}] & & \tX^+ \arrow[rr, "\text{contract }S_0" {swap}] & & \tX^{(2)} \arrow[rr, "\text{contract }F" {swap}] & &  \bb P^3
\end{tikzcd}\]

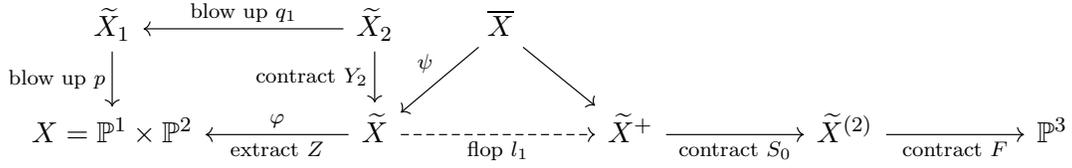
\begin{figure}[h]\label{fig:A_n-flop}
\caption{The flop \(\tX \dashrightarrow \tX^+\) at \(u=4-4c\).}
\centering
    \adjustbox{width=.5\textwidth}{
    \begin{tikzcd}
    \begin{tabular}{c}
\begin{tikzpicture}

\draw[teal,thick] (0,0) -- (0.7,-.5);
\draw[thick] (-.7,-.7) -- (-.7,1.3) -- (0.7,2.7)--(0.7,1.5);
\draw[thick,dashed] (0.7,1.5)--(0.7,0.7);
\draw[thick] (0,0) -- (-.7,-.7) -- (1.3,-.7) -- (2,0) -- (2.7,0.7) -- (1.5,0.7); 
\draw[dashed, thick] (1.5,0.7) -- (0.7,0.7) -- (0,0);
\draw[line width=0.1cm, olive!30] (0,0) -- (0,1.5);
\draw[olive,semithick] (0,0) -- (0,1.5);
\draw[blue,thick] (0,0) -- (1.5,0) -- (1.5,1.5) -- (0,1.5);

\node[below, node font=\tiny] at (.3,-.7) {$S_0$};
\node[left, node font=\tiny] at (-.7,.5) {$F$};
\node[right, node font=\tiny] at (0.7,-.5) {$l_1$};
\node[above, node font=\tiny] at (1.5,1.5) {$Z$};

\end{tikzpicture}
\end{tabular} \arrow[rr, dashed, "\text{flop }l_1"] && \begin{tabular}{c}
\begin{tikzpicture}

\draw[thick,dashed] (0.7,1.6)--(0.7,0.7);
\draw[thick] (0,0) -- (-.7,-.7) -- (1.3,-.7) -- (2,0) -- (2.7,0.7) -- (1.5,0.7); 
\draw[dashed, thick] (1.5,0.7) -- (0.7,0.7) -- (0,0);
\draw[line width=0.1cm, olive!30] (0,0.8) -- (-.3,1.7);
\draw[olive,semithick] (0,.8) -- (-.3,1.7);
\draw[thick] (-.7,-.7) -- (-.7,1.3) -- (0.7,2.7)--(0.7,1.6);
\draw[teal,thick] (0,0) -- (0,.8);
\draw[teal,fill=teal] (0,0.8) circle (0.02);
\draw[blue,thick] (0,0) -- (1.5,0) -- (1.5,1.5) -- (-.3,1.7);

\node[below, node font=\tiny] at (.3,-.7) {$S$};
\node[left, node font=\tiny] at (-.7,.5) {$F$};
\node[left, node font=\tiny] at (0,0.4) {$l_1^+$};

\end{tikzpicture}
\end{tabular} 
    \end{tikzcd}
    }
\end{figure}

Now we can compute the Nakayama--Zariski decomposition of $\psi^*(L - uZ)$ on $\oX$, where $\Delta_1, \Delta_2, \Delta_3$ are as in Lemma~\ref{lem:volume-lemma-6}.  Denote by $\oZ, \oS_0, \oF$ the strict transforms of these surfaces on $\oX$, and denote by $\oY$ the exceptional divisor of $\psi$ from the blow-up of the singular locus of $\tX$.  We find:
\[ \resizebox{1\textwidth}{!}{ $ \begin{array}{rlll}
    0 \le u \le 4-4c: & P_{\regionA}(u) =& \hspace{-.75em} -K_\oX - \dee_{\oX} + (4-4c-u) (\oZ + \tfrac{1}{2}\oY)  +  \Delta_1 + 2\Delta_3 + (3-2c-\tfrac{u}{2})\Delta_2 \\
    4-4c \le u \le 5-4c: &  P_{\regionB}(u) =& \hspace{-.75em} -K_\oX - \dee_{\oX} + (2-2c-\tfrac{u}{2}) (2 \oZ + \oY) +  (3-2c-\tfrac{u}{2})(\Delta_1+2\Delta_3) + (5-4c-u)\Delta_2 \\
    5-4c \le u \le 6-4c: & P_{\regionC}(u) =& \hspace{-.75em} -K_\oX - \dee_{\oX} + (2-2c-\tfrac{u}{2}) (2 \oZ + \oY) +  (3-2c-\tfrac{u}{2})(\Delta_1+2\Delta_3) + (5-4c-u)(\Delta_2 + \oS_0) \\    6-4c \le u \le 8-6c: & P_{\regionD}(u) =& \hspace{-.75em} -K_\oX - \dee_{\oX} +(4-4c-u) \oZ + (5-4c-u) \oY  +  (6-4c-u)(\Delta_1+2\Delta_3) + (11 -8 c - 2u)\Delta_2 \\ & & \hspace{-.75em} + (5-4c-u) \oS_0 + (6-4c-u)\oF
\end{array} $ } \]
and
\[ \resizebox{1\textwidth}{!}{ $ \begin{array}{rl}
    0 \le u \le 4-4c: & N_{\regionA}(u) = 0 \\
    4-4c \le u \le 5-4c: & N_{\regionB}(u) = (u-4+4c)(\tfrac{1}{2}\Delta_1+\tfrac{1}{2}\Delta_2+\Delta_3) \\
    5-4c \le u \le 6-4c: & N_{\regionC}(u) = (u-4+4c)(\tfrac{1}{2}\Delta_1+\tfrac{1}{2}\Delta_2+\Delta_3) + (u-5+4c) \oS_0 \\
    6-4c \le u \le 8-6c: & N_{\regionD}(u) = (u-4+4c)(\tfrac{1}{2}\Delta_1+\tfrac{1}{2}\Delta_2+\Delta_3) + (u-5+4c) \oS_0 + (\tfrac{u}{2}-3+2c)(\oY + \Delta_1 + 2\oF + 2\Delta_2 + 2\Delta_3 )
\end{array} $} \]

\begin{detail}
    For the details of the computation of the Nakayama--Zariski decomposition:
    
    For \(\regionA\), \(P_{\regionA}(u) = \phi^*(L-uZ)\) and \(N_{\regionA}(u) = 0\). Since \(\psi^*K_\tX = K_\oX - \Delta_1 - \Delta_2 - 2\Delta_3\), \(\psi^*D_\tX = D_\oX\), and \(\psi^*Z = \oZ + \tfrac{1}{2}Y + \Delta_2\), we have
    \begin{equation*}\begin{split}
        \phi^*(L-uZ) &= \psi^*(-K_\tX - \dee_{\tX} + (4-4c-u)Z) \\
    &= -K_\oX - \dee_{\oX} +(4-4c-u) \oZ + (2-2c-\tfrac{u}{2}) \oY  +  \Delta_1 + (3-2c-\tfrac{u}{2})\Delta_2 + 2\Delta_3 .
    \end{split}\end{equation*}

    For \(\regionB\), write \(L^+ = -K_{\tX^+}-\dee_{\tX^+} + (4-4c) Z^+\). Since
    \[
        {\psi^+}^* K_{\tX^+} = K_{\oX} - \Delta_1 - \Delta_2 - 2\Delta_3 , \quad
        {\psi^+}^* (\dee_{\tX^+}) = \dee_{\oX} , \quad
        {\psi^+}^* Z^+ = \oZ + \tfrac{1}{2} \oY + \tfrac{1}{2} \Delta_1 + \Delta_2 + \Delta_3
    \] we have that
    \[\resizebox{1\textwidth}{!}{ $P_{\regionB}(u) = {\psi^+}^*(L^+ - uZ^+) = -K_{\oX} - \dee_{\oX} + (4-4c-u)\oZ + (2-2c+\tfrac{u}{2})\oY + (3-2c-\tfrac{u}{2}) \Delta_1 +(5-4c-u) \Delta_2 +(6-4c-u)\Delta_3, $}\] and \(N_{\regionB}(u) = \psi^*(L-uZ) - P_{\regionB}(u) = (\tfrac{u}{2}-2+2c)(\Delta_1 + \Delta_2 + 2\Delta_3)\).

    For \(\regionC\), define \(L^{(2)} = -K_{\tX^{(2)}}-\dee_{\tX^{(2)}} + (4-4c) Z^{(2)}\). Since \(g^*K_{\tX^{(2)}} = K_{\tX^+}+S_0\), \(g^*\dee^{(2)} = \dee_{\tX^+}\), and \(g^* Z^{(2)} = Z^+ + S_0\), we have that
    \(L^+ - uZ^+ = g^*(L^{(2)}-u Z^{(2)}) + (u-5-4c)S_0\). Therefore, since \({\psi^+}^* S_0 = \oS_0\), we have \[P_{\regionC}(u) = {\psi^+}^*g^*(L^{(2)}-u Z^{(2)}), \qquad N_{\regionC}(u) = N_{\regionB}(u) + (u-5-4c) \oS_0. \]

    For \(\regionD\), on the final model \(W \cong \bb P^3\), define \(L_W = -K_W - \dee_W + (4-4c) Z_W\). Since \(\dee_W = cR_W\), where \(R_W\) is a quadric in \(\bb P^3\), and \(Z_W\) is a plane, we have
    \(L_W - u Z_W = \cal O_{\bb P^3}(8 - 6c - u)\).
    By computing intersection numbers with the fibers of \(F\) contracted by \(h\colon \tX^{(2)}\to W\), we find that \(L^{(2)}-u Z^{(2)} = h^*(L_W - u Z_W) + (u-6+4c) F\). Therefore, \(P_{\regionD}(u) = {\phi^+}^*g^*h^* (L_W - u Z_W)\) and \[N_{\regionD}(u) = N_{\regionC}(u) + {\phi^+}^* g^* (u-6+4c)F.\]
    Since \({\phi^+}^* g^* F = Y/2 + \Delta_1/2 + \oF + \Delta_2 + \Delta_3\), we obtain the claimed expressions.
\end{detail}

We can now compute $S_{\exx,\dee}(Z)$, by computing the volumes of $P_{\regionA}(u), P_{\regionB}(u), P_{\regionC}(u), P_{\regionD}(u)$ above:
\[ \resizebox{1\textwidth}{!}{ $ \begin{array}{rll}
    0 \le u \le 4-4c: & P_{\regionA}(u)^3 = 3(2-2c)(3-2c)^2 - \tfrac{1}{4} u^3 \\
    4-4c \le u \le 5-4c: & P_{\regionB}(u)^3 =  3(2-2c)(3-2c)^2 - \tfrac{1}{4} u^3 - 2(2-2c-\tfrac{u}{2})^3 \\ & \text{ by Lemma~\ref{lem:volume-lemma-6} applied to }l_1\text{ with }p = 2-2c-\tfrac{u}{2}.\\
    5-4c \le u \le 6-4c: & P_{\regionC}(u)^3 = 3(2-2c)(3-2c)^2 - \tfrac{1}{4} u^3 - 2(2-2c-\tfrac{u}{2})^3 +3(2-2c)(5-4c-u)^2 \\    
    & \text{ by Lemma~\ref{lem:volume-lemma-3} applied to }S_0\text{ (}n=0, p = 5-4c-u, q = 1, a = 2-2c, b =0\text{)}.
    \\ 6-4c \le u \le 8-6c: & P_{\regionD}(u)^3 = (8-6c-u)^3 \text{ using that the final birational model is }\bP^3.
\end{array} $ }\]
\begin{detail}
    In more detail for \(\regionA\), if \(\ell\) is a line on \(Z\), we have \(Z|_Z = -\ell / 2\) and \(L|_Z = \phi^*(-K_\exx-\dee)|_Z = 0\), so
    \begin{equation*}\begin{split}
        P_{\regionA}(u)^3 &= (L-uZ)^3 = L^3 - 3t L^2 Z + 3u^2 L Z^2 - u^3 Z^3 \\
        &= (-K_\exx-\dee)^3 - \tfrac{1}{4}u^3 = 3(2-2c)(3-2c)^2 - \tfrac{1}{4}u^3.
    \end{split}
    \end{equation*}
\end{detail}
So
\[ S_{\exx,\dee}(Z) =\frac{- (42 c^3 - 182 c^2 + 262 c - 125)}{3(3-2c)^2},  \]
and since $A_{\exx,\dee}(Z) = 5-4c$, we obtain that $A_{\exx,\dee}(Z)/S_{\exx,\dee}(Z) >1$ for all $c \in (0,1)$. 

\subsubsection{Computation of lower bound on $\delta$ using curves on $Z$}\label{sec:local-delta-A_n:higher-A_n-part-2}

Next, we compute $\delta_q(Z, \dee_Z; V_{\bullet \bullet})$ for $q\in Z$.  For each point, we will choose a plt type divisor $C$ over $Z$ whose center on $Z$ contains $q$.  
\begin{detail}
If $\alpha\colon Z' \to Z$  is the plt blow-up extracting $C$, then
\[ \delta_q(Z, \dee_Z; V_{\bullet \bullet}) \ge \min \left\{ \frac{A_{Z, \dee_Z}(C)}{S_{Z,\dee_Z; V_{\bullet \bullet}}(C)}, \inf_{
     c \in C :
     \alpha(c) = q } \delta_c(C,\dee_C; W_{\bullet \bullet \bullet}) \right\} \]
where $W_{\bullet \bullet \bullet}$ is the further refinement of $V_{\bullet \bullet}$ by $C$ and $\dee_C$ is the restriction of $\dee_Z$ to $C$. 
\end{detail}

Define the following two lines on \(Z\): let \(l_Y'\) be the intersection of intersection of $Z$ with the strict transform of $Y$, and let \(l_S'\) be the intersection of \(Z\) with the strict transform of \(S_0\).
If \(R_Z\) is the intersection of \(Z\) with the strict transform of \(R\), then \(R_Z\) is a conic of rank \(\geq 2\) and meets $l_Y'$ and $l_S'$ transversally, and furthermore we have $\dee_Z = cR_Z + \frac{1}{2} l_Y'$.

By \cite[Theorem 4.8]{Fujita-3.11}, we compute the above quantities on $\oZ$, the strict transform of $Z$ in the common resolution $\oX$ above.
We have that $\oZ \cong \bb F_1$ is the blow-up of \(l_Y' \cap l_S'\), and the Picard group is generated by the negative section $\sigma$ of $\bb F_1$, and the fiber $f$.  
Let $l_Y$ and $l_S$ be the strict transforms of $l_Y', l_S'$ on $\oZ$, both numerically equivalent to $f$.  Using the notations from Section~\ref{sec:local-delta-A_n:higher-A_n-part-1} above, we know $\Delta_1|_\oZ = \Delta_3|_\oZ = \oF|_\oZ = 0$, $\Delta_2|_\oZ = \sigma$, $\oY|_\oZ = l_Y$, and $\oS_0|_\oZ = l_S$. Using the formulas for $P(u)$ and $N(u)$ from Section~\ref{sec:local-delta-A_n:higher-A_n-part-1} above, we have:
\[ \resizebox{1\textwidth}{!}{ $ \begin{array}{rll}
    0 \le u \le 4-4c: & P_{\regionA}(u)|_\oZ = (u/2)(\sigma+f) & N_{\regionA}(u)|_\oZ  =  0 \\
    4-4c \le u \le 5-4c: & P_{\regionB}(u)|_\oZ = (2-2c)\sigma + (u/2) f & N_{\regionB}(u)|_\oZ = (u/2+2c-2)\sigma \\
    5-4c \le u \le 6-4c: & P_{\regionC}(u)|_\oZ = (2-2c)\sigma + (5-4c-u/2)f & N_{\regionC}(u)|_\oZ  =(u/2+2c-2)\sigma + (u-5+4c)l_S \\
    6-4c \le u \le 8-6c: & P_{\regionD}(u)|_\oZ = (8-6c-u)(\sigma+f) & N_{\regionD}(u)|_\oZ  =(3u/2-8+6c)\sigma + (u-5+4c)l_S + (u/2-3+2c)l_Y.
\end{array} $} \]

\noindent\underline{Case: \(q\neq R_Z \cap l_Y', R_Z \cap l_S', l_Y' \cap l_S'\):}
First, we bound $\delta_q(Z, \dee_Z; V_{\bullet \bullet})$ for points $q\in Z$ such that $q$ is not one of the intersection points of $R_Z \cap l_Y'$, $R_Z \cap l_S'$, or $l_Y' \cap l_S'$. In these cases, we let \(C\) be the following line in \(Z\):
\begin{itemize}
    \item If \(q\) is not a singular point of \(R_Z\), let \(C\) be a general line containing \(q\) that does not contain the center of the blow-up. 
    \item If \(n\geq 4\) and \(q\) is the singular point of the rank 2 conic \(R_Z\), let \(C\) be one of the two components of \(R_Z\).
\end{itemize}
Then the strict transform of \(C\) in \(\oZ\) is $\oC \sim \sigma + f$ and is equal to its pullback because $C$ did not contain the center of the blow-up. Note that, in this case, \(N(u)|_\oZ = N'(u)|_\oZ\) on all regions.
\begin{detail}
Indeed, we have \[\ord_C(N_{\regionA}(u)|_\oZ) = \ord_C(N_{\regionB}(u)|_\oZ) = \ord_C(N_{\regionC}(u)|_\oZ) = 0 \]
because the support of these restrictions do not contain $\oC$.  Therefore, by \cite[Theorem 4.8]{Fujita-3.11}, we have 
\[ S_{Z,\dee_Z; V_{\bullet \bullet}}(C) = \frac{3}{\vol L} 
\int_0^{5-4c} \left( \int_0^\infty \vol( P(u)|_\oZ - vC) dv \right) du. \]

Furthermore, by \cite[Theorem 4.17]{Fujita-3.11}, we also have the formulas for $W_{\bullet \bullet \bullet}$:
\[ S_{C,\dee_C; W_{\bullet \bullet \bullet}}(q) = \frac{3}{\vol L} 
\int_0^{5-4c} \left( \int_{0}^{t(u)} (P(u,v)\cdot \oC)^2 dv \right) du + F_q(W_{\bullet \bullet \bullet}) \]
where 
\[ F_q(W_{\bullet \bullet \bullet})  = \frac{6}{\vol L} 
\int_0^{5-4c} \left( \int_{0}^{t(u)} (P(u,v)\cdot \oC)\ord_q((N(u)|_\oZ + N(u,v))|_\oC) dv \right) du \]
where $P(u,v), N(u,v)$ is a Zarisiki decomposition of $P(u)|_\oZ - v \oC$.  
\end{detail}
We next give the Zariski decompositions for each region. 

In region $\regionA$ for $0 \le u \le 4-4c$, $P_{\regionA}(u)|_\oZ - v\oC = (u/2-v)(\sigma+f)$.  Here, $t(u) = u/2$, and we find the Zariski decomposition for \(0 \le v \le u/2\):
\[ P_{\regionA}(u,v) = (u/2-v)(\sigma + f) , \qquad N_{\regionA}(u,v) = 0 . \]
\begin{detail}
So the relevant quantities for $0 \le u \le 4-4c$ are: 
\[\begin{array}{r|l}
     & 0 \le v \le u/2 \\
     \hline
    P_{\regionA}(u,v)^2 & (u/2-v)^2 \\
    P_{\regionA}(u,v) \cdot \oC & (u/2-v) \\
    (P_{\regionA}(u,v) \cdot \oC)^2 & (u/2-v)^2 \\
    \ord_{\qbarpt}((N_{\regionA}'(u)|_\oZ + N_{\regionA}(u,v))|_\oC) &  0 
\end{array}\]
\end{detail}
For region $\regionB$ for $4 -4c \le u \le 5-4c$, $P_{\regionB}(u)|_\oZ - v\oC = (2-2c-v)\sigma + (u/2-v)f$ and $t(u) = 2-2c$.  Here, the Zariski decomposition for \(0 \le v \le 2-2c \) is 
\[ P_{\regionB}(u,v) = (2-2c-v)\sigma + (u/2-v)f , \qquad N_{\regionB}(u,v) = 0 .\]
\begin{detail}
So the relevant quantities for $4-4c \le u \le 5-4c$ are: 
\[  \begin{array}{r|l}
   &  0 \le v \le 2-2c \\
       \hline
   (P_{\regionB}(u,v))^2   & (2-2c-v)(u-2+2c-v) \\
   P_{\regionB}(u,v)\cdot \oC  & (u/2-v) \\
   (P_{\regionB}(u,v)\cdot \oC)^2  & (u/2-v)^2 \\
   \ord_{\qbarpt}((N_{\regionB}'(u)|_\oC + N_{\regionB}(u,v))|_\oC) & 0 \\
\end{array}  \]
\end{detail}
For region $\regionC$ with $5-4c \le u \le 6-4c$, $P_{\regionC}(u)|_\oZ - v\oC = (2-2c-v)\sigma + (5-4c-u/2-v)f$ and $t(u) = 2-2c$.  Here, the Zariski decomposition for \(0 \le v \le 2-2c\) is
\[ P_{\regionC}(u,v) = (2-2c-v)\sigma + (5-4c-u/2-v)f , \qquad N_{\regionC}(u,v) = 0 .\]
\begin{detail}
So the relevant quantities for $5-4c \le u \le 6-4c$ are: 
\[ \begin{array}{r|l}
   & 0 \le v \le 2-2c \\
       \hline
   (P_{\regionC}(u,v))^2 & (2-2c-v)(8-6c-u-v) \\
   P_{\regionC}(u,v)\cdot \oC & (5-4c-u/2-v) \\
   (P_{\regionC}(u,v)\cdot \oC)^2 & (5-4c-u/2-v)^2  \\
   \ord_{\qbarpt}((N_{\regionC}'(u)|_\oZ + N_{\regionC}(u,v))|_\oC) & 
   \left\{ \begin{array}{ll}
        0 & \qbarpt \notin l_S \\
      u-5+4c  & \qbarpt \in l_S 
   \end{array} \right. 
\end{array} \]
\end{detail}
For region $\regionD$ with $6-4c \le u \le 8-6c$, $P_{\regionD}(u)|_\oZ - v\oC = (8-6c-u-v)(\sigma+f)$ and $t(u) = 8-6c-u$.
Here, the Zariski decomposition for \(0 \le v \le 8-6c-u\) is 
\[ P_{\regionD}(u,v) = (8-6c-u-v)(\sigma+f), \qquad N_{\regionD}(u,v) = 0 . \]
\begin{detail}
so the relevant quantities for $6-4c \le u \le 8-6c$ are: 
\[ \begin{array}{r|l}
   & 0 \le v \le 8-6c-u \\
       \hline
   (P_{\regionD}(u,v))^2 & (8-6c-u-v)^2 \\
   P_{\regionD}(u,v)\cdot \oC & (8-6c-u-v) \\
   (P_{\regionD}(u,v)\cdot \oC)^2 & (8-6c-u-v)^2  \\
   \ord_{\qbarpt}((N_{\regionD}'(u)|_\oZ + N_{\regionD}(u,v))|_\oC) & 
   \left\{ \begin{array}{ll}
        0 & \qbarpt \notin l_S, l_Y \\
      u-5+4c  & \qbarpt \in l_S \\
      u/2-3+2c & \qbarpt \in l_Y
   \end{array} \right. 
\end{array} \]
\end{detail}
Therefore, we may compute using \cite[Theorem 4.8]{Fujita-3.11} to obtain: 
\[S_{Z,\dee_Z; V_{\bullet \bullet}}(C) = \frac{(1-c)(6c^2-20c+17)}{3(3-2c)^2} .  \]
Since \(R_Z\) is a smooth conic (if \(n=3\)) or a union of two distinct lines (if \(n\geq 4\)), we have $A_{Z,\dee_Z}(C) = 1$ or $1-c$ (the latter case occurs when $C$ is one of the two components of $R_Z$). So we compute
\[ \frac{A_{Z,\dee_Z}(C)}{S_{Z,\dee_Z; W_{\bullet \bullet \bullet}}(Z)} \geq \frac{3 (3 - 2 c)^2}{(6 c^2 - 20 c + 17)} > 1\] for $c \in (0,1)$, using \cite[Theorem 4.17]{Fujita-3.11}.

Next, we compute $S_{C,\dee_C; W_{\bullet \bullet \bullet}}(\qbarpt)$ for $\qbarpt \in \oC$, again using \cite[Theorem 4.17]{Fujita-3.11}.

\noindent\underline{Case: \(q\) is not \(R_Z \cap l_Y', R_Z \cap l_S'\), or \(l_Y' \cap l_S'\):}
Using the above formulas, we find
\[ S_{C,\dee_C; W_{\bullet \bullet \bullet}}(\qbarpt) = \begin{cases}
    \displaystyle \frac{-(12 c^3 - 52 c^2 + 74 c - 35)}{6 (3 - 2 c)^2} & \qbarpt\notin l_S \cup l_Y \\
    \displaystyle \frac{3 - 2 c}{3} & \qbarpt \in l_S, q\notin R_Z \cup l'_Y \\
    \displaystyle \frac{-(14 c^3 - 58 c^2 + 80 c - 37)}{6 (3 - 2 c)^2} & \qbarpt \in l_Y, q\notin R_Z \cup l'_S.
\end{cases}\]

Since we are assuming that \(q\) is not $R_Z \cap l_Y'$, $R_Z \cap l_S'$, or $l_Y' \cap l_S'$, we have
\[ A_{C,\dee_C}(\qbarpt) = \begin{cases}
    1 \text{ or } 1-c & \qbarpt \notin l_S \cup l_Y \\
    1 & \qbarpt \in l_S, q\notin R_Z \cup l'_Y \\
    1/2 & \qbarpt \in l_Y, q \notin R_Z \cup l'_S.
\end{cases}\]
Note that, for the \(q\notin l'_S, l'_Y\) case, we can always assume $C$ and $R_Z$ are transverse at $q$ by picking an appropriate line.
(If \(n \geq 4\) and \(q\) is the singular point of \(R_Z\), then we again have \(A_{C,\dee_C}(q)=1\) since we took \(C\) to be a component of \(R_Z\).)
Thus, in each of these cases, we find that \[ \frac{A_{C,\dee_C}(q)}{S_{C,\dee_C; W_{\bullet \bullet \bullet}}(q)}  > 1 \quad\text{for }c \in (\cone, 0.7055].\]
(More precisely, the upper bound is the real root of \(12 c^3 - 44 c^2 + 52 c - 19\).)

\noindent\underline{Case: \(q\) is \(R_Z \cap l_Y', R_Z \cap l_S'\), or \(l_Y' \cap l_S'\):}
To finish the proof, it remains to show that $\delta_q(Z, \dee_Z; V_{\bullet \bullet})>1$ for $q \in \{ R_Z \cap l_Y', R_Z \cap l_S', l_Y' \cap l_S'\}$.
Let $f$ denote a fiber of $\oZ \cong \bb F_1$, and let $f'$ denote its image on $Z$.  We will use the flag coming from $f = l_Y$ (if \(q = R_Z \cap l_Y'\)) or $l_S$ (if \(q = R_Z \cap l_S'\) or \(l_Y' \cap l_S'\)).
Since \(f'\) contains the center \(l_Y' \cap l_S'\) of the blow-up \(\oZ \to Z\), we have \(\oC \sim f\).

Note that, unlike the previous cases, \(N(u)|_\oZ \neq N'(u)|_\oZ\) on some regions here. By \cite[Theorem 4.8]{Fujita-3.11}, we have 
\[ S_{Z,\dee_Z; V_{\bullet \bullet}}(f) = \frac{3}{\vol L} 
\int_0^{5-4c} \left(G(f) + \ord \int_0^\infty \vol( P(u)|_\oZ - vf) dv \right) du \]
where $G(f) = 0$ if $f$ is a general fiber, and $G(f) = (P(u)|_\oZ)^2 \ord_\oC (N(u)|_\oZ) $ if $f = l_S'$ or $l_Y'$.

\begin{detail}
Let $P(u,v)$ be the positive part and $N(u,v)$ the negative part of $P(u)|_\oZ - vf$.  Then, the formula for $S$ becomes:
\[ S_{Z,\dee_Z; V_{\bullet \bullet}}(f) = \frac{3}{\vol L} 
\int_0^{5-4c} \left( G(f)+ \int_{0}^{t(u)} P(u,v)^2 dv \right) du.  \]
We recall for formulas for $W_{\bullet \bullet \bullet}$ given by \cite[Theorem 4.17]{Fujita-3.11} $l_Y'$,
and define $d(u) = \ord_\oC (N(u)|_\oZ)$ and $N(u)|_\oZ = d(u) \oC + N'(u)|_\oZ$.  Then,
\[ S_{C,\dee_C; W_{\bullet \bullet \bullet}}(q) = \frac{3}{\vol L} 
\int_0^{5-4c} \left( \int_{0}^{t(u)} (P(u,v)\cdot \oC)^2 dv \right) du + F_q(W_{\bullet \bullet \bullet}) \]
where 
\[ F_q(W_{\bullet \bullet \bullet})  = \frac{6}{\vol L} 
\int_0^{5-4c} \left( \int_{0}^{t(u)} (P(u,v)\cdot \oC)\ord_q((N'(u) + N(u,v) - (d(u)+v)\sigma)|_\oC) dv \right) du. \]
\end{detail}
Below are the Zariski decompositions of $P(u)|_\oZ-v\oC$ for each region. 

In region $\regionA$ for $0 \le u \le 4-4c$, $P_{\regionA}(u)|_\oZ - v\oC = u/2(\sigma+f) - vf$.  Here, $t(u) = u/2$, and we find the Zariski decomposition for \(0 \le v \le u/2\):
\[ P_{\regionA}(u,v) = (\tfrac{u}{2}-v)(\sigma + f) , \qquad N_{\regionA}(u,v) = v\sigma . \]
\begin{detail}
So the relevant quantities for $0 \le u \le 4-4c$ are: 
\[\begin{array}{r|l}
     & 0 \le v \le u/2 \\
     \hline
    P_{\regionA}(u,v)^2 & (u/2-v)^2 \\
    P_{\regionA}(u,v) \cdot \oC & (u/2-v) \\
    (P_{\regionA}(u,v) \cdot \oC)^2 & (u/2-v)^2 \\
    \ord_{\qbarpt}((N_{\regionA}'(u)|_\oZ + N_{\regionA}(u,v)-(d(u)+v)\sigma)|_\oC) &  0 
\end{array}\]
\end{detail}
For region $\regionB$ for $4 -4c \le u \le 5-4c$, $P_{\regionB}(u)|_\oZ - v\oC = (2-2c)\sigma + (u/2-v)f$ and $t(u) = u/2$.  Here, the Zariski decomposition is 
\begin{equation*}\begin{split}
    P_{\regionB}(u,v) &= \begin{cases}
        (2-2c)\sigma + (\frac{u}{2}-v)f & 0 \le v \le \frac{u}{2}-2+2c \\
        (\frac{u}{2}-v)(\sigma+f) & u/2-2+2c \le v \le u/2 ,
    \end{cases} \\
    N_{\regionB}(u,v) &= \begin{cases}
        0 \hphantom{(2-c)\sigma + (\frac{u}{2}-v)f} &  0 \le v \le \frac{u}{2}-2+2c \\
        (v-\frac{u}{2}+2-2c)\sigma & \frac{u}{2}-2+2c \le v \le u/2 .
    \end{cases}
\end{split}\end{equation*}
\begin{detail}
So the relevant quantities for $4-4c \le u \le 5-4c$ are: 
\[ \resizebox{1\textwidth}{!}{ $ \begin{array}{r|ll}
   &  0 \le v \le u/2-2+2c & u/2-2+2c \le v \le u/2 \\
       \hline
   (P_{\regionB}(u,v))^2 & - (2-2c)^2 + 2(2-2c)(u/2-v)  & (u/2-v)^2 \\
   P_{\regionB}(u,v)\cdot \oC & 2-2c & (u/2-v) \\
   (P_{\regionB}(u,v)\cdot \oC)^2 & (2-2c)^2  & (u/2-v)^2 \\
   \ord_{\qbarpt}((N_{\regionB}'(u)|_\oC + N_{\regionB}(u,v)-(d(u)+v)\sigma)|_\oC) & \left\{ \begin{array}{ll}
        0 & \qbarpt \notin \sigma \\
      u/2+2c-2-v  & \qbarpt \in \sigma 
   \end{array} \right.  & 0 \\
\end{array} $} \]
\end{detail}
For region $\regionC$ with $5-4c \le u \le 6-4c$, $P_{\regionC}(u)|_\oZ - v\oC = (2-2c)\sigma + (5-4c-u/2-v)f$ and $t(u) = 5-4c-u/2$.  

Here, the Zariski decomposition is 
\begin{equation*}\begin{split}
    P_{\regionC}(u,v) &= \begin{cases}
        (2-2c)\sigma + (5-4c-\frac{u}{2}-v)f & 0 \le v \le 3-2c-\frac{u}{2} \\
        (5-4c-\frac{u}{2}-v)(\sigma+f) & 3-2c-\frac{u}{2} \le v \le 5 - 4c -\frac{u}{2} ,
    \end{cases} \\
    N_{\regionC}(u,v) &= \begin{cases}
        0 \hphantom{(2-2)\sigma + (5-4c-\frac{u}{2}-v)f} & 0 \le v \le  3-2c-\frac{u}{2}   \\
        (v-3+2c+\frac{u}{2}) \sigma & 3-2c-\frac{u}{2} \le v \le 5 - 4c -\frac{u}{2} .
    \end{cases}
\end{split}\end{equation*}
\begin{detail}
So the relevant quantities for $5-4c \le u \le 6-4c$ are: 
\[ \resizebox{1\textwidth}{!}{ $\begin{array}{r|ll}
   & 0 \le v \le 3-2c-u/2  & 3-2c-u/2 \le v \le 5 - 4c -u/2\\
       \hline
   (P_{\regionC}(u,v))^2 & - (2-2c)^2 + 2(2-2c)(5-4c-u/2-v)  & (5-4c-u/2-v)^2 \\
   P_{\regionC}(u,v)\cdot \oC & 2-2c & (5-4c-u/2-v) \\
   (P_{\regionC}(u,v)\cdot \oC)^2 & (2-2c)^2 & (5-4c-u/2-v)^2  \\
   \ord_{\qbarpt}((N_{\regionC}'(u)|_\oZ + N_{\regionC}(u,v)-(d(u)+v)\sigma)|_\oC) & \left\{ \begin{array}{ll}
        0 & \qbarpt \notin \sigma \\
      u/2+2c-2 -v & \qbarpt \in \sigma , \oC \neq l_S \\
      3 - 2 c - u/2 - v & \qbarpt \in \sigma , \oC = l_S \\
   \end{array} \right. &  \left\{ \begin{array}{ll}
        0 & \qbarpt \notin \sigma \\
      u+4c-5  & \qbarpt \in \sigma, \oC \ne l_S \\
      0 & \qbarpt \in \sigma, \oC = l_S
   \end{array} \right. 
\end{array}$ } \]
\end{detail}

For region $\regionD$ with $6-4c \le u \le 8-6c$, $P_{\regionD}(u)|_\oZ - v\oC = (8-6c-u)(\sigma+f) - vf$ and $t(u) = 8-6c-u$.  
Here, the Zariski decomposition for \(0 \le v \le 8-6c-u\) is 
\[ P_{\regionD}(u,v) = (8-6c-u-v)(\sigma+f) , \qquad N_{\regionD}(u,v) = v\sigma . \]
\begin{detail}
So the relevant quantities for $6-4c \le u \le 8-6c$ are: 
\[ \begin{array}{r|l}
   & 0 \le v \le 8-6c-u \\
       \hline
   (P_{\regionD}(u,v))^2 & (8-6c-u-v)^2 \\
   P_{\regionD}(u,v)\cdot \oC & (8-6c-u-v) \\
   (P_{\regionD}(u,v)\cdot \oC)^2 & (8-6c-u-v)^2  \\
   \ord_{\qbarpt}((N_{\regionD}'(u)|_\oZ + N_{\regionD}(u,v)-(d(u)+v)\sigma)|_\oC) & \left\{ \begin{array}{ll}
        0 & \qbarpt \notin \sigma \\
      3u/2-8+6c  & \qbarpt \in \sigma, \oC \ne l_S, l_Y \\
       u/2 -3 + 2c & \qbarpt \in \sigma, \oC = l_S \\
        u -5 + 4c & \qbarpt \in \sigma, \oC = l_Y \\
   \end{array} \right. 
\end{array} \]
\end{detail}
Therefore, we may compute using \cite[Theorem 4.8]{Fujita-3.11} to obtain:
\[ S_{Z,\dee_Z; V_{\bullet \bullet}}(C) = \begin{cases}
    \displaystyle \frac{-12 c^3 + 52 c^2 - 74 c + 35}{6 (3 - 2 c)^2} & C = f' \text{ is the image of a general fiber} \\
    \displaystyle \frac{3-2c}{3} & C = l_S' \\
    \displaystyle \frac{-14 c^3 + 58 c^2 - 80 c + 37}{6 (3 - 2 c)^2} & C = l_Y' .
\end{cases}\]
As \(A_{Z,\dee_Z}(f') = A_{Z,\dee_Z}(l_S') = 1\) and \(A_{Z,\dee_Z}(l_Y') = \tfrac{1}{2}\), we compute that in each of these cases,
\[ \frac{A_{Z,\dee_Z}(C)}{S_{Z,\dee_Z; V_{\bullet \bullet}}(C)} > 1 \text{ for }c\in (\cone, 1). \]

Finally, we compute $S_{C,\dee_C; W_{\bullet \bullet \bullet}}(q)$ for $q \in C$ using \cite[Theorem 4.17]{Fujita-3.11}:
\[ S_{\oC,\dee_\oC; W_{\bullet \bullet \bullet}}(\qbarpt) = \begin{cases}
    \displaystyle \frac{(1 - c) (6 c^2 - 20 c + 17)}{3 (2 c - 3)^2} & \qbarpt \notin \sigma, \oC \in \{ f, l_S, l_Y\} \\
    \displaystyle \frac{-14 c^3 + 58 c^2 - 80 c + 37}{6 (3 - 2 c)^2} & q=l_S'\cap l_Y', \oC = l_S
\end{cases}\]
We have
\[ A_{\oC,\dee_\oC}(\qbarpt) = \begin{cases}
    1 & \qbarpt \notin \sigma \text{ and } q \neq R_Z\cap l_Y', R_Z \cap l_S' \\
    1-c & q = R_Z\cap l_Y' \text{ or } R_Z \cap l_S' \\
    1/2 & q = l_S'\cap l_Y' \text{ and } C = l_S' ,
\end{cases}\]
so if \(q\) is one of the intersection points \(R_Z \cap l_Y'\), \(R_Z \cap l_S'\), or \(l_S' \cap l_Y'\), then we have
\[ \frac{A_{C,\dee_C}(q)}{S_{C,\dee_C; W_{\bullet \bullet \bullet}}(q)} > 1 \text{ for }c \in ( \cone , 1).\]
Therefore, we have shown that $\delta_{\pt}(\exx,\dee) > 1$ for all $c \in (\cone, 0.7055]$. Combining this with the earlier result that $\delta_{\pt}(\exx,\dee) > 1$ for all $c \in(0,\cnot)$ we have $\delta_{\pt}(\exx,\dee) >1$ for all $c \in (0,0.7055]$. This completes the proof of Theorem~\ref{thm:local-delta-A_n-singularities}\eqref{part:local-delta-A_n-singularities:higher-A_n}.

\section{Polystable \texorpdfstring{\((2,2)\)}{(2,2)}-surfaces in \texorpdfstring{\(\bb P^1\times\bb P^2\)}{P1xP2}}\label{sec:GIT-strictly-polystable}
In Section~\ref{sec:GIT-strictly-ss}, we described the GIT strictly semistable \((2,2)\)-surfaces in \(\bb P^1\times\bb P^2\), and we defined loci \(\Gtwo,\Gthree,\Gfour\) in Proposition~\ref{prop:representatives-degen-Gamma_i} that gave representatives for their S-equivalence classes. We now show that these in fact give a list of K-polystable (and hence GIT-polystable) representatives. Our main result is recorded as Theorem~\ref{thm:Kpolystable2,2div}.
First, we recall the loci defined in Proposition~\ref{prop:representatives-degen-Gamma_i} and define distinguished \((2,2)\)-surfaces in these loci.

\subsection{Important loci in the GIT moduli space}\label{sec:important-loci}
By Lemma~\ref{lem:non-A_n-not-GIT-stable} and Proposition~\ref{prop:reducible}, there are three S-equivalence classes of \((2,2)\)-surfaces in \(\bb P^1\times\bb P^2\) that are GIT semi-stable and have non-isolated singularities:
    \begin{equation*}\begin{split} \Rone :  t_0 t_1( y_0 y_2 + y_1^2), \qquad \Rtwo :  (t_0 y_2 + t_1 y_1)(t_0 y_1 + t_1 y_0), \qquad \Rthree:  t_0t_1y_1^2 + (t_0y_2 + t_1y_0)^2.\end{split}\end{equation*}
We also define the following \((2,2)\)-surface, whose associated discriminant curve is an ox:
    \[\Rfour : (t_0 + t_1)^2 y_1^2 + t_0 t_1 y_0 y_2.\]

In Proposition~\ref{prop:representatives-degen-Gamma_i}, we defined the following sets:
\begin{equation*}\begin{split}
        \Gfour &\coloneqq \{(t_0^2 y_1^2+t_0t_1(b_{11}y_1^2+y_0 y_2)+t_1^2y_1^2=0) \mid b_{11}\in\bb C\}/\langle b_{11}\sim -b_{11}\rangle \cup [ \Rone ], \\
        \Gtwo &\coloneqq \{ (t_0^2 y_1 y_2 + t_0 t_1(b_{11} y_1^2 + y_0 y_2) + t_1^2y_0 y_1=0) \mid b_{11}\in\bb C\} \cup [ \Rone ], \\
        \Gthree &\coloneqq\{ (t_0^2 y_2^2 + t_0 t_1 (y_1^2 + b_{02} y_0 y_2) + t_1^2 y_0^2=0) \mid b_{02}\in\bb C\}/\langle b_{02}\sim -b_{02}\rangle \cup [ \Rone ]..
    \end{split}\end{equation*}
We have \(\Rfour\in\Gfour\), \(\Rtwo\in\Gtwo\), and \(\Rthree\in\Gthree\). For every other \((2,2)\)-surface \(R\neq\Rone,\Rfour,\Rtwo,\Rthree\) in the above loci, the associated discriminant curve is a cat-eye (see Remark~\ref{rem:frakG_i-surfaces-descriptions}). 

\begin{theorem}\label{thm:Kpolystable2,2div}
    Define \(\GITboundary \coloneqq \Gtwo \cup \Gthree \cup \Gfour\).
    \begin{enumerate}
        \item For \(R\in \GITboundary \setminus[\Rthree]\), the pair \((\bb P^1\times\bb P^2, cR)\) is K-polystable for \(c\in (0,\tfrac{3}{4}) \cap \bb Q\).
        \item The pair \((\bb P^1\times\bb P^2, c\Rthree)\) is K-polystable for $c \in (0,0.25344] \cap \bb Q.$
    \end{enumerate}
    In particular, by Theorem~\ref{thm:K-implies-GIT} and Proposition~\ref{prop:representatives-degen-Gamma_i}, the strictly polystable locus in the GIT moduli space \(\GITR\) is the union of three rational curves \(\Gtwo\cup\Gthree\cup\Gfour\), which meet at the point \([ \Rone ]\).
\end{theorem}

In fact, in Section~\ref{sec:wall-crossing-c_0} we will show that \((\bb P^1\times\bb P^2, c\Rthree)\) is K-polystable for $c \in (0,\cnot) \cap \bb Q$, where \(\cnot \approx 0.472\) is the smallest root of the polynomial \(10c^3-34c^2+35c-10\).

\begin{remark}
    The set \([\Rone]\) (resp. \(\Gtwo\)) in Proposition~\ref{prop:representatives-degen-Gamma_i} is the same as the set \(\overline{\Gamma_1}\) (resp. \(\overline{\Gamma_2}\)) in \cite[Section 8.1]{PR-GIT}. Imposing the equivalence \(b_{02}\sim-b_{02}\) on the set \(\overline{\Gamma_3}\) from \cite[Section 8.1]{PR-GIT} gives \(\Gthree\). However, \cite{PR-GIT} does not prove that the surfaces in these sets are GIT semistable (because this preprint work in a fixed coordinate system and only tests normalized 1-parameter subgroups in that coordinate system), and, furthermore, does not comment on the polystability of these surfaces. See Remark~\ref{rem:PR-GIT} for a more detailed discussion regarding this preprint.
\end{remark}

Throughout this section, we will use Notation~\ref{notation:delta-invariant-computations} when computing lower bounds for \(\delta\)-invariants. In particular, we will write \(\exx=\bb P^1\times\bb P^2\), \(\dee = cR\), and \(L = -K_\exx-\dee\). Then \(L \sim \cal O_{\exx}(2-2c, 3-2c)\) has \(\vol L=3(2-2c)(3-2c)^2\). For \(i=1,2\), let \(\piX_i\colon X\to\bb P^i\) denote the projections.

To prove Theorem~\ref{thm:Kpolystable2,2div}, we appeal to Theorem~\ref{lem:K-ps-G-invariant-delta}. For each \(R\) in each component \(\Gfour\), \(\Gtwo\), \(\Gthree\) of \(\GITboundary\), we find an appropriate \(G \subset \Aut(\exx,\dee)\) and show that \(\delta_C(\exx,\dee)>1\) for every \(G\)-invariant irreducible subvariety \(C\subset\exx\). We do this using Abban--Zhuang's method of admissible flags (see Section~\ref{sec:Abban-Zhuang}), which implies that if \(C\) is a curve and \(Y\subset\exx\) is a prime divisor containing \(C\), then \[\delta_C(\exx,\dee) \geq \min\left\{ \frac{A_{\exx,\dee}(Y)}{S_{\exx,\dee}(Y)}, \delta_C(Y, \dee_Y; W^Y_{\bullet \bullet}) \right\}\]
where \(\dee_Y = \dee|_Y\) and \(W^Y_{\bullet \bullet}\) is the refinement of the filtration of \(L\) by \(Y\). When the divisor \(Y\) is clear, we will often drop it from the notation and write \(W_{\bullet \bullet}\). Because \(C \subset Y \subset \exx = \bb P^1 \times \bb P^2\), we will be able to directly compute \(\delta_C(Y, \dee_Y; W_{\bullet \bullet}) = A_{Y, \dee_Y}(C)/S_{Y, \dee_Y; W_{\bullet \bullet}}(C)\) from the definitions.

\begin{proof}[Proof of Theorem~\ref{thm:Kpolystable2,2div}]
    The statements for $R \in \GITboundary$ are established in Lemmas~\ref{lem:frakG1polystable}, \ref{lem:frakG2polystable}, \ref{lem:frakG3polystable}, and~\ref{lem:frakG4polystable} for \(\Rone\), \(\Gtwo\), \(\Gthree\), and \(\Gfour\), respectively.
\end{proof}

\subsection{Divisorial stability for \texorpdfstring{\((2,2)\)}{(2,2)}-surfaces} To exploit the Abban--Zhuang method, it is first necessary to establish lower bounds for $A_{\exx, \dee}(Y)/S_{\exx, \dee}(Y)$ when $Y$ is a divisor. The computation holds for \((2,2)\)-surfaces in more generality, and we record this bound in Lemma \ref{lem:divis-stable}.

\begin{lemma}\label{lem:divis-stable}
    Let \(R\subset\exx\) be a \((2,2)\)-surface. Let \(Y\subset\exx\) be a prime divisor, and assume that either
    \begin{enumerate}
        \item \(Y \not\subset R\), or
        \item If \(Y \subset R\), then \(R\) has multiplicity \(1\) along \(Y\) and the bidegree of \(Y\) is not \((1,0)\) or \((0,2)\).
    \end{enumerate}
    Then \(A_{\exx, \dee}(Y)/S_{\exx, \dee}(Y) > 1\) for all \(c\in (0,1)\). If \(Y\) is a \((0,2)\)-divisor and \(Y\subset R\) has multiplicity \(1\), then \(A_{\exx, \dee}(Y)/S_{\exx, \dee}(Y) > 1\) for all \(c\in (0 , \tfrac{3}{4})\).
\end{lemma}

\begin{proof}
    We first deal with the case when \(Y\in|\cal O_{\exx}(d,d)|\). We compute
    \[S_{\exx, \dee}(Y) = \frac{1}{\vol L}\int_0^{(2-2c)/d} 3(2-2c-du)(3-2c-du)^2 \; du = \frac{(1 - c) (6 c^2 - 20 c + 17)}{3 (2 c - 3)^2 d}.\]
    If \(Y\not\subset R\), then \(A_{\exx,\dee}(Y)=1\) and \[ \frac{A_{\exx,\dee}(Y)}{S_{\exx,\dee}(Y)} = \frac{3 d (3 - 2 c)^2}{(1 - c) (6 c^2 - 20 c + 17)}> 1\] for \(c\in(0,1)\), since \(d \geq 1\). If \(Y\subset R\), then \(d\in \{1, 2\}\) and \(A_{\exx,\dee}(Y) = 1-c\), so
    \[ \frac{A_{\exx,\dee}(R)}{S_{\exx,\dee}(R)} = \frac{3 d (3 - 2 c)^2}{6 c^2 - 20 c + 17}> 1\]
    for \(c\in (0,1)\).
    
    Next, assume \(Y\) has bidegree \((a,b)\) for \(a \neq b\). There are two cases for the pseudoeffective threshold, leading to different computations of \(S_{\exx,\dee}(Y)\):
    \begin{enumerate}
        \item If \(a=0\) or \(\frac{2-2c}{a} \geq \frac{3-2c}{b}\), then the pseudoeffective threshold is \(u=\frac{3-2c}{b}\), and
\begin{detail}
        \[ S_{\exx,\dee}(Y) = \frac{1}{3(2-2c)(3-2c)^2} \int_0^{(3-2c)/b} 3(2-2c-ua)(3-2c-ub)^2 \; du \]
        so
\end{detail}
        \[\frac{1}{S_{\exx,\dee}(Y)} = \frac{12 b^2 (2 - 2c)}{(2 c - 3) (a (3 - 2 c) + 8 b (c - 1))}.\]
        If \(a=0\), this is \( 3 b/(3 - 2 c) > 1\) for \(c\in (0,1)\). If \(a \neq 0\), then \((2-2c)b \geq (3-2c)a\) implies
        \begin{equation*}\begin{split}
        \frac{1}{S_{\exx,\dee}(Y)} & \geq \frac{12 a b}{-a(3 - 2 c) + 8b(1-c)} > \frac{12 a b}{8b(1-c)} = \frac{3 a}{2 - 2 c} > 1 \text{ for } c \in (0,1).
        \end{split}\end{equation*}
        \item If \(b=0\) or \(\frac{2-2c}{a} \leq \frac{3-2c}{b}\), then the pseudoeffective threshold of \(L-uY\) is \(u=\frac{2-2c}{a}\), and
\begin{detail}
        \[  S_{\exx,\dee}(Y) = \frac{1}{3(2-2c)(3-2c)^2} \int_0^{(2-2c)/a} 3(2-2c-ua)(3-2c-ub)^2 \; du \]
        so
\end{detail}
        \[\frac{1}{S_{\exx,\dee}(Y)} = \frac{3 a^3 (3 - 2 c)^2}{(1 - c) (3 a^2 (3 - 2 c)^2 - 4 a b (c - 1) (2 c - 3) + 2 b^2 (c - 1)^2)}. \]
        If \(b=0\), this is \(a/(1 - c) > 0\) for \(c \in (0,1)\). If \(b = 1\), then one can verify directly that
        \[ \frac{1}{S_{\exx,\dee}(Y)} = \frac{3 a^3 (3 - 2 c)^2}{(1 - c) (3 a^2 (3 - 2 c)^2 - 4 a (1 - c) (3 - 2 c) + 2 (1 - c)^2)} > 1 \text{ for }c \in (0,1).\]
        Finally, if \(b \geq 2\), then the assumption that \((3-2c)a \geq (2-2c) b\) implies
        \begin{equation*}\begin{split}
        \frac{1}{S_{\exx,\dee}(Y)} & \geq \frac{6 a^2 b (3 - 2 c)}{3 a^2 (3 - 2 c)^2 - 4 a b (1 - c) (3 - 2 c) + 2 b^2 (1 - c)^2} \\
        & = \frac{6 a^2 b}{3 a^2 (3 - 2 c) - 3 a b (1 - c) - a b (1 - c) + 2 b^2 \frac{(1 - c)^2}{3 - 2 c}} \\
        & \geq \frac{6 a^2 b}{3 a^2 (3 - 2 c) - 3 a b (1 - c)} \geq \frac{2 b}{3 - 2 c} > 1 \text{ for }c \in (0,1).
        \end{split}\end{equation*}
    \end{enumerate}
    Thus, if \(Y \not\subset R\), then we see that \(A_{\exx, \dee}(Y)/S_{\exx,\dee}(Y) > 1\) for \(c\in (0,1)\).

    If \(Y\subset R\), then our assumption implies \(A_{\exx,\dee}(Y) = 1-c\), then \((a,b)\) is one of \((1,0)\), \((0,1)\), \((1,2)\), \((2,1)\), \((2,0)\), or \((0,2)\). Using the above computations, if \(Y\) has bidegree \((1,0)\), \((1,2)\), \((2,1)\), or \((2,0)\) then \(A_{\exx,\dee}(Y)/S_{\exx,\dee}(Y) > 1\) for \(c\in (0,1)\). If \(Y\) has bidegree \((0,2)\), then \(A_{\exx,\dee}(Y)/S_{\exx,\dee}(Y) > 1\) for \(c\in (0,\tfrac{3}{4})\). If \(Y\) has bidegree \((1,0)\), then \(A_{\exx,\dee}(Y)/S_{\exx,\dee}(Y) = 1\).
\end{proof}

\subsection{Polystability for \texorpdfstring{$[\Rone]=\Gtwo \cap \Gthree \cap \Gfour$}{R0}}\label{sec:G1-polystability} In this section we establish Lemma \ref{lem:frakG1polystable}.

\begin{lemma}\label{lem:frakG1polystable}
    Let \(\Rone\) be the \((2,2)\)-surface defined by \(t_0t_1(y_1^2 + y_0y_2)=0\). Then the pair \((\exx, c\Rone)\) is K-polystable for \(c\in(0,\frac{3}{4})\cap\bb Q\).
\end{lemma} 
\subsubsection{The group $G_0$ and its fixed varieties}\label{sec:gpG1}
Let \(C\subset\bb P^2\) be the smooth conic defined by \(y_1^2+y_0y_2\). Then \(\Aut(C)\) embeds into \(\Aut(\bb P^2)\),
and we define \(G_0\) to be the subgroup of \(\Aut(\exx, \Rone)\) generated by \(\Aut(C)\), \(\iota_1 \colon [t_0:t_1] \mapsto [t_1: t_0]\), and  \(\theta_\lambda  \colon [t_0:t_1] \mapsto [\lambda t_0 : \lambda^{-1} t_1]\) for \(\lambda\in\bb C^*\).
\begin{lemma}\label{lem:GinvGamma1}
    There are no \(G_0\)-invariant curves or points in \(\exx\), and the only \(G_0\)-invariant divisor is \((y_1^2 + y_0y_2 = 0)\).
\end{lemma}
\begin{proof}
      Since $G_0$ acts transitively on points of $\mathbb{P}^1_{[t_0:t_1]}$, there are no fixed points. If \(C\subset\bb P^2\) is the conic defined by \(y_1^2+y_0y_2\), then by intersection theory any $G_0$-invariant subvariety of \(\exx\) of dimension at least one must intersect \(\piX_2^{-1}(C)\). Since $G_0$ acts transitively on points of \(\piX_2^{-1}(C)\), every such invariant subvariety contains, and is thus equal to, the divisor \(\piX_2^{-1}(C)\). 
\end{proof}

\begin{proof}[Proof of Lemma~\ref{lem:frakG1polystable}] The Lemma follows from Theorem~\ref{lem:K-ps-G-invariant-delta} applied to the group \(G_0\), together with Lemmas~\ref{lem:divis-stable} and~\ref{lem:GinvGamma1}.
\end{proof}

\subsection{Polystability for \texorpdfstring{\(R\in \Gtwo\)}{R in non-finite locus}}\label{sec:G2-polystability}
Recall \(\Gtwo\) is the set
\[\{ (t_0^2 y_1 y_2 + t_0 t_1(b_{11} y_1^2 + y_0 y_2) + t_1^2y_0 y_1=0) \mid b_{11}\in\bb C\} \cup \{ \Rone \}.\]
In this section we prove Lemma \ref{lem:frakG2polystable}.

\begin{lemma}\label{lem:frakG2polystable}
    For \(R \in\Gtwo \setminus \{[\Rone], [\Rtwo]\}\), the pair \((\exx, cR)\) is K-polystable for \(c\in(0,1)\cap\bb Q\). The pair \((\exx,c\Rtwo)\) is K-polystable for \(c\in(0, \frac{13 - \sqrt{13}}{12}) \cap \bb Q\).
\end{lemma}
\begin{proof}
    Apply Theorem~\ref{lem:K-ps-G-invariant-delta} to the group $G_2$ defined in Section \ref{sec:gpG2}. The required lower bounds on \(\delta\)-invariants are established in Lemmas~\ref{lem:divis-stable} and~\ref{lem:G2deltabounds}.
\end{proof}
\subsubsection{The group $G_2$ and its fixed varieties}\label{sec:gpG2} Let $G_2$ be the subgroup of \(\mathrm{Aut}(\exx, R)\) generated by
\begin{equation*}
    \begin{split}
        \sigma_\lambda\colon ([t_0:t_1],[y_0:y_1:y_2]) & \mapsto ([t_0: \lambda t_1],[y_0 : \lambda y_1 : \lambda^2 y_2]) \\ \iota\colon ([t_0:t_1],[y_0:y_1:y_2]) & \mapsto ([t_1: t_0],[y_2:y_1: y_0])
    \end{split}
\end{equation*}
for $\lambda \in \mathbb{C}^\star$ and \(R \in \Gtwo \setminus [\Rone]\).

\begin{lemma}\label{lem:GinvGamma2}
    The \(G_2\)-invariant curves in \(\exx\) are
    \[C_1 = (y_0=y_2=0), \quad C_{2,\pm} = (t_0^2 y_2 \pm t_1^2 y_0 = y_1 = 0), \quad C_{3,\pm}^a = (t_0y_1 \pm \sqrt{a}t_1y_0 = \sqrt{a}t_0y_2 \pm t_1y_1 = 0)\]
    for \(a\in\bb C^*\), and there are no \(G_2\)-invariant points.
\end{lemma}

\begin{proof}    
    The image of a \(G_2\)-invariant subvariety \(Z\) of \(\exx\) under $\piX_2$ is irreducible and fixed by the actions of $\sigma_\lambda$ and $\tau$ on the second factor. Thus, \(\piX_2(Z)\) is one of \(\bb P^2\), \((y_1=0)\), \((y_1^2-ay_0y_2 = 0)\) for some \(a \in \bb C^\ast\), or \([0:1:0]\).

\begin{detail}
Indeed, if $\piX_2(Z)$ is a curve, then let $f = \sum_{0 \leq i+j \leq d}{a_{i,j}y_0^iy_1^{d-i-j}y_2^j}$ be its defining equation in $\bb P^2$ of degree $d$. Since \((f=0)\) is \(G_2\)-invariant, then for each \(\lambda\in\bb C^*\) we have \(\sigma_\lambda^* f = b_\lambda f\) for some \(b_\lambda\in\bb C^*\). Under \(\sigma_\lambda^*\) we have \(a_{ij} y_0^i y_1^{d-i-j}y_2^j \mapsto a_{ij}\lambda^{d-i+j} y_0^i y_1^{d-i-j}y_2^j\). So we have \(b_\lambda a_{ij} = a_{ij}\lambda^{d-i+j}\) for all \(i,j\). So the difference \(i-j\) must be constant for all coefficients \(a_{ij}\neq 0\). If either \(i-j\neq 0\) or if \(i-j=a_{00}=0\), then \((f=0)\) is reducible. So we may assume \(i=j\).
If $d > 1 $ is odd, then $f$ is divisible by $y_1$, i.e. $(f=0)$ is reducible. If $d$ is even, $(f=0)$ is irreducible if and only if $d=2$ and $a_{00} \neq 0 \neq a_{11}$.
\end{detail}
We deduce that any \(G_2\)-invariant point is of the form \(Z=([t_0:t_1],[0:1:0])\), and observe no such $Z$ is $G_2$ fixed.

\begin{detail}
Indeed, if \(Z\) is a \(G_2\)-invariant point, then \(Z=(([t_0:t_1],[0:1:0])\). Then
        \begin{gather*}
            \sigma_\lambda(Z) = ([t_0:\lambda t_1],[0:1:0]) \\
            \iota(Z) = ([t_1:t_0],[0:1:0])
        \end{gather*}
        and it is impossible for \(Z=\iota(Z)=\sigma_\lambda(Z)\) for all \(\lambda\in\bb C^*\). Thus, there are no \(G_2\)-invariant points. 
\end{detail}

For a $G_2$-invariant curve \(Z\), there are three possibilities for \(\piX_2(Z)\). If \(\piX_2(Z)=[0:1:0]\), then $Z = C_1$ is the fiber. If \(\piX_2(Z)=(y_1=0) \subset\bb P^2\), then $Z$ is a curve inside $(y_1=0)=\bb P^1_{[t_0:t_1]} \times \bb P^1_{[y_0:y_2]}$. One can deduce from the irreducibility of $Z$ and invariance of $Z$ under $\sigma_\lambda$ that $Z = (y_2t_0^2+cy_0t_1^2 = 0) \subset \bb P^1_{[t_0:t_1]} \times \bb P^1_{[y_0:y_2]}$ for some $c \in \bb C^\ast$. Invariance of $Z$ under $\iota$ further implies $Z = C_{2,+}$ or \(C_{2,-}\).
\begin{detail}
If $\piX_2(Z) = (y_1=0)$, then $Z \subset \bb P^1_{[t_0:t_1]} \times \bb P^1_{[y_0:y_2]}$, and let $g = \sum_{0 \leq i \leq k,0 \leq j \leq d}{c_{i,j}t_0^{k-i}t_1^iy_0^{d-j}y_2^j}$ be its defining equation of bidegree $(k,d)$. We have $\sigma_\lambda^*(c_{i,j}t_0^{k-i}t_1^iy_0^{d-j}y_2^j) =  c_{i,j}\lambda^{i+2j}t_0^{k-i}t_1^iy_0^{d-j}y_2^j$ for every $i,j$. Since $g=0$ is irreducible, $g$ is not divisible by $t_0$ or $y_2$, so $c_{k,j'} \neq 0 \neq c_{i',0}$ for some $i',j'$, which means that $\lambda^{k+2j'}=\lambda^{i'}$ for all $\lambda \in \bb C^\ast$, i.e. $k+2j' = i'$. But $i' \leq k$, so this forces $j' = 0$ and $i' = k$. In other words, $c_{k,0} \neq 0$ and $\sigma_\lambda^*$ must send $f$ to $\lambda^k f$. Similarly, by the fact that $g$ is not divisible by $t_1$ or $y_0$, one can show $c_{0,d} \neq 0$ and $\sigma_\lambda^*$ must also send $f$ to $\lambda^{2d} f$. Hence, $k=2d$, and for every $i,j$ with $c_{i,j} \neq 0$, we must have $\lambda^{i+2j} = \lambda^{2d}$ for every $\lambda \in \bb C^\ast$, i.e. $i = 2d-2j$. Thus, we can rewrite $g$ as $\sum_{0 \leq j \leq d}{c_{2d-2j,j}t_0^{2j}t_1^{2d-2j}y_0^{d-j}y_2^j} = \sum_{0 \leq j \leq d}{c_{2d-2j,j}(t_0^2y_2)^j(t_1^2y_0)^{d-j}}$, which is irreducible if and only if $d=1$, in which case $g = c_{1,0}t_1^2y_0+c_1t_0^2y_2$ for $c_{1,0} \neq 0 \neq c_{0,1}$. By the invariance of $(g=0)$ under $\iota$, we further conclude that $c_{1,0} = \pm c_{0,1}$. Thus, $Z = C_{2,\pm}$.
\end{detail}

It remains to handle the third possibility, namely that $\piX_2(Z)$ is \((y_1^2 - ay_0 y_2=0)\) for some \(a\in\bb C^*\). Then $Z$ is a curve inside $\bb P^1_{[t_0:t_1]} \times (y_1^2-ay_0y_2 = 0)$. We consider the image of $Z$ under the isomorphism \begin{align*}
        \tau \colon \bb P^1_{[t_0:t_1]} \times (y_1^2-ay_0y_2 = 0) &\xrightarrow{\cong} \bb P^1_{[t_0:t_1]} \times \bb P^1_{[w_0:w_1]} \\
        ([t_0:t_1],[w_0^2:\sqrt{a}w_0w_1:w_1^2]) &\mapsfrom ([t_0:t_1],[w_0:w_1]).
\end{align*}
Since $Z$ is $G_2$-invariant, $\tau(Z)$ is invariant under the automorphisms \begin{align*}
    \widetilde{\sigma}_\lambda \colon ([t_0:t_1],[w_0:w_1]) &\mapsto ([t_0:\lambda t_1],[w_0:\lambda w_1]) \qquad \textrm{for all } \lambda \in \CC^\ast \\
    \widetilde{\iota} \colon ([t_0:t_1],[w_0:w_1]) &\mapsto ([t_1:t_0],[w_1:w_0]).
\end{align*}
Similarly to the preceding paragraph, one can deduce from the irreducibility of $\tau(Z)$ and invariance of $\tau(Z)$ under $\widetilde{\sigma}_\lambda$ that $\tau(Z) = (t_0w_1 + ct_1w_0 = 0) \subset \bb P^1_{[t_0:t_1]} \times \bb P^1_{[w_0:w_1]}$ for some $c \in \CC^\ast$. The invariance of $\tau(Z)$ under $\widetilde{\iota}$ further implies that $\tau(Z) = (t_0w_1 \pm t_1w_0 = 0) \subset \bb P^1_{[t_0:t_1]} \times \bb P^1_{[w_0:w_1]}$. Consequently, $Z = C_{3,+}^a$ or $C_{3,-}^a$. 
\end{proof}

First, we address the singular locus of the reducible \((2,2)\)-surface \(\Rtwo\) (i.e. \(b_{11}=1\)).

\begin{lemma}\label{lem:R_2-sing-deltabounds}
    Let \(\Rtwo\in\Gtwo\) be the reducible surface defined by \((t_0 y_2 + t_1 y_1)(t_0 y_1 + t_1 y_0)\), and let \(C=(t_0 y_2 + t_1 y_1 = t_0 y_1 + t_1 y_0 = 0)\) be the singular locus of \(\Rtwo\). Then \(\delta_C(\exx, c\Rtwo)>1\) for \(c\in(0, \frac{13 - \sqrt{13}}{12} \approx 0.7829)\).
\end{lemma}

\begin{proof}
Write \(\dee = c\Rtwo\). Let \(\phi\colon \tX \to \exx\) be the blow-up of \(C\) with exceptional divisor \(Y \cong C\times\bb P^1\). Since \(C\) is invariant under \(G_2\), the action of \(G_2\) lifts to \(\tX\). By computing the lifts of \(\sigma_\lambda\) and \(\iota\), one sees that \(Y\) has no proper \(G_2\)-invariant subvarieties.
\begin{detail}
Indeed, \(\tX\) is given in coordinates by
\[ \tX = ( v(t_0 y_2 + t_1 y_1) - u(t_0 y_1 + t_1 y_0) = 0 ) \subset \exx\times\bb P^1_{[u:v]},\]
and the automorphisms \(\sigma_\lambda\) and \(\iota\) lift to \([u:v] \mapsto [\lambda u:v]\) and \([u:v] \mapsto [v:u]\), respectively.
\end{detail}
So by \cite{Zhuang-equivariant,Abban-Zhuang-flags} it suffices to show that \(A_{\exx,\dee}(Y)/S_{\exx,\dee}(Y)>1\).

Let \(L = \phi^*(-K_\exx-\dee) = -K_\tX - \dee_\tX + (1-2c)Y\). Let \(B\subset\exx\) be the \((0,2)\)-divisor defined by \(y_1^2-y_0y_2\), and let \(\widetilde{B}\) be its strict transform in \(\tX\). We compute that the pseudoeffective threshold of \(L-uY\) is \(u=(5-4c)/2\), and at \(u=2-2c\) we contract \(\widetilde{B}\) to a curve in a morphism \(\psi\colon \tX \to \oX\). We find that the positive part of the Nakayama--Zariski decomposition is
\begin{align*}
    0 \le u \le 2-2c: &\quad P_{\regionA}(u) = L-uY \\
    2-2c \le u \le \tfrac{5-4c}{2}: &\quad P_{\regionB}(u) = L-uY - (u-2+2c) \widetilde{B}.
\end{align*}
\begin{detail}
Indeed, if \(T_1\) is one of the irreducible components of \(\Rtwo\), then we have \(-K_\exx-\dee \sim (2-2c) T_1 + \frac{1}{2} B\), and \(\ord_Y (\phi^*(F_1 + \frac{1}{2} B)) = 2-2c+\frac{1}{2}\), so this shows the pseudoeffective threshold is \(\geq (5-4c)/2\).
Next, the Mori cone of \(\tX\) is generated by the strict transform \(l_1\) of a fiber of \(\piX_2\colon\exx\to\bb P^2\) meeting \(C\), the strict transform \(l_2\) of a line in a fiber of \(\piX_1\colon\exx\to\bb P^2\) meeting \(C\), and a fiber \(f_Y\) of \(Y\to C\). Since \((L-uY)\cdot l_1 = 2-2c-u\), then at \(u=2-2c\) we contract \(\widetilde{B}\) to a curve in a morphism \(\psi\colon \tX \to \oX\).
To describe this morphism in more detail, note that \(\tX\) is the total space of the pencil of \((1,1)\)-surfaces spanned by the two irreducible components of \(\Rtwo\). Each surface \(T\) in the pencil is isomorphic to \(\bb F_1\), and contracting \(l_1\) contracts the negative section of each \(\bb F_1\). If \(f\sim l_2\) is a fiber of \(T\), then we compute that \(P_{\regionB}(u)\cdot f = (3-2c) - u - (u-2+2c) = 5 -4 c - 2 u\), confirming that the pseudoeffective threshold is \(u=(5-4c)/2\).
\end{detail}
Let \(f_Y\) and \(l_Y\), respectively, denote vertical and horizontal rulings of \(Y\to C\), and let \(l_B\) denote the ruling of \(\widetilde{B}\) not contracted by \(\psi\).
Applying Lemma~\ref{lem:volume-lemma-3} twice, and using that \(L|_Y = (8-6c) f_Y\), \(Y|_Y = 3 f_Y - l_Y\), \((L-uY)|_{\widetilde{B}} = (6-4c-u)l_1 + (2-2c-u)l_B\), and \(\widetilde{B}|_{\widetilde{B}} = 3 l_1 - l_B\),
\begin{detail}
(indeed, we have $\phi^*B = \widetilde{B} + Y$ since $C$ is contained in $B$, and intersecting with a fiber $l_1$ of the second projection, we have $0 = \widetilde{B} \cdot l_1 + Y \cdot l_1$, and we have $Y \cdot l_1 = 1$),
\end{detail}
we compute that
\[ S_{\exx,\dee}(Y) = \frac{12 c^2 - 34 c + 23}{6 (3 -2 c)}.\]
\begin{detail}
Indeed, the volumes on the respective regions are
\begin{equation*}\begin{split}P_{\regionA}(u)^3 &= 3(2-2c)(3-2c)^2 - 3 u^2 (8-6c) + 6 u^3, \\
P_{\regionB}(u)^3 &= P_{\regionA}(u)^3+6(u-2+2c)^2(6-4c-u)+6 (2 c + u - 2)^2 (-c - u)+6(u-2+2c)^3. \end{split}\end{equation*}
\end{detail}
Since \(A_{\exx,\dee}(Y) = 2-2c\), we find that \(A_{\exx,\dee}(Y)/S_{\exx,\dee}(Y) > 1\) for \(c\in(0, \frac{13 - \sqrt{13}}{12})\).
\end{proof}
 
\begin{lemma}\label{lem:G2deltabounds} Let \(R\in\Gtwo \setminus[\Rone]\), and let \(C\) be any of the \(G_2\)-invariant curves in Lemma~\ref{lem:GinvGamma2}. Then \(\delta_{\exx, cR}(C)>1 \) for \(c \in (0,1)\) if \(R\neq\Rtwo\), and \(\delta_{\exx, c\Rtwo}(C)>1\) for \(c\in(0, \frac{13 - \sqrt{13}}{12} \approx 0.7829)\).
\end{lemma}

\begin{proof}We establish lower bounds for $C = C_1, C_{2,\pm}, C_{3,\pm}^a$ in turn. 

\noindent\underline{Case $C = C_1$:} First, if \(b_{11}=0\), then one can check that \(R\) is smooth along \(C_1\subset R\). Then \cite[Lemma 2.4]{CFKP23} shows that \(\delta_{\pt}(\exx, cR)>1\) for all \(\pt\in C_1\) and all \(c\in(0,1)\). We may thus assume $b_{11}\ne 0$.

Consider the flag \(C_1\subset Y_1\coloneqq (y_0=0)\subset\exx\). Observe \(Y_1=\bb P^1\times\bb P^1\), \(\cal O_{\exx}(a,b)|_{Y_1}=\cal O_{Y_1}(a,b)\), and \(C_1\in|\cal O_{Y_1}(0,1)|\).
Note also that if \(a\geq -1\) and \(b\geq -2\), then \(H^i(\exx, \cal O_{\exx}(a,b)) = 0\) for \(i \geq 1\).
Thus, the refinement of the filtration of \(L = -K_\exx-\dee\) by \(Y_1\) is given by
\[W_{m,j} =\begin{cases}H^0(\exx, \cal O_{\exx}(m(2-2c),m(3-2c)-j)) & \text{if }m(2-2c)\geq 0 \text{ and } m(3-2c)-j\geq 0, \\
0 & \text{otherwise}.\end{cases}\]
By the Abban--Zhuang method and Lemma~\ref{lem:divis-stable}, it suffices to establish the inequality $\delta_{C_1}(Y_1, \dee_{Y_1}; W_{\bullet \bullet}) >1$ for $c\in (0,1)$.
Computing using the K\"unneth formula, the volume of \(W_{m,j}\) is
\begin{equation*}\begin{split}
\vol(W_{m,j}) &= \lim_{m\to\infty} \frac{\sum_{j=0}^{m(3-2c)} (m(2-2c)+1)(m(3-2c)-j+1)}{m^3/3!} \\
&= 6 \lim_{m\to\infty}\left(2-2c+\frac{1}{m}\right) \lim_{m\to\infty} \sum_{j=0}^{m(3-2c)} \left((3-2c) - \frac{j-1}{m}\right)\frac{1}{m} \\
&= 6(2-2c) \int_{0}^{3-2c} 3-2c-x \; dx = 3(2-2c)(3-2c)^2.
\end{split}\end{equation*}

Next, the \(\bb Z_{\geq 0}^2\)-graded linear series \(\cal F^{mt}_{C_1} W_{m,j} = \{s \in W_{m,j} \mid \ord_{C_1} (s) \geq mt\}\) on \(Y_1\) is given by
\[\resizebox{1\textwidth}{!}{ $\cal F^{mt}_{C_1} W_{m,j} = \begin{cases} \lceil mt\rceil C_1 + H^0(Y_1, \cal O_{Y_1}(m(2-2c), m(3-2c)-j-\lceil mt\rceil) & \text{if }2m(1-c)\geq 0 \text{ and }m(3-2c)-\lceil mt\rceil\geq j \\
0 & \text{otherwise}
\end{cases} $} \]
so we may compute
\[ \resizebox{1\textwidth}{!}{ $ \displaystyle \vol(\cal F^{mt}_{C_1} W_{m,j}) = \lim_{m\to\infty} \frac{\sum_{j=0}^{m(3-2c-t)} (m(2-2c)+1)(m(3-2c)-j-\lceil mt\rceil+1)}{m^3/3!} =  3(2-2c)(3-2c-t)^2. $} \]
\begin{detail}
Indeed,
\[\resizebox{1\textwidth}{!}{ $ \dim\cal F^{mt}_{C_1} W_{m,j} = \begin{cases} (m(2-2c)+1)(m(3-2c)-j-\lceil mt\rceil+1) & \text{if }2m(1-c)\geq 0 \text{ and }m(3-2c-t)\geq j, \\ 0 & \text{otherwise} \end{cases} $ }\] so we have
\begin{equation*}\begin{split}
\vol(\cal F^{mt}_{C_1} W_{m,j}) &= \lim_{m\to\infty} \frac{\sum_{j=0}^{m(3-2c-t)} (m(2-2c)+1)(m(3-2c)-j-\lceil mt\rceil+1)}{m^3/3!} \\
&= 6(2-2c) \lim_{m\to\infty} \sum_{j=0}^{m(3-2c-t)} \left(3-2c-\frac{j-1}{m}- t \right)\frac{1}{m} \\
&= 6(2-2c) \int_{0}^{3-2c-t} 3-2c-x - t \; dx \\
&= 3(2-2c)(3-2c-t)^2.
\end{split}\end{equation*}
\end{detail}
Integrating, we find that
\begin{equation*}
\begin{split}
S_{Y_1, \dee_{Y_1}; W_{\bullet \bullet}}(C_1) &= \frac{1}{\vol(W_{\bullet \bullet})} \int_0^\infty \vol(\cal F^{mt}_{C_1} W_{\bullet \bullet}) \; dt \\
&= \frac{1}{3(2-2c)(3-2c)^2} \int_0^{3-2c} 3(2-2c)(3-2c-t)^2 \; dt
= \frac{3-2c}{3}.
\end{split}
\end{equation*}
\begin{detail}
The restriction \(R|_{Y_1}\) is defined by \(t_0^2 y_1 y_2 + t_0 t_1 (b_{11}y_1^2)\).
\end{detail}
Since we've assumed \(b_{11}\neq 0\), we have \(A_{Y_1, \dee_{Y_1}}(C_1) = 1\).
Thus, whenever \(b_{11}\neq 0\), we have \(\delta_{C_1}(Y_1, \dee_{Y_1}; W^Y_{\bullet \bullet}) = 3/(3-2c)>1\) for \(c\in (0,1)\).

\noindent\underline{Case $C = C_{2,\pm}$:} Let \(C_2\) be \(C_{2,+}\) or \(C_{2,-}\). Using the flag \(C_2\subset Y_2\coloneqq(y_1=0)\subset\exx\), by the Abban--Zhuang method and Lemma~\ref{lem:divis-stable}, it suffices show that \(\delta_{C_2}(Y_2, \dee_{Y_2}; W_{\bullet \bullet}) >1\) for \(c\in (0,1)\). We compute that \[A_{Y_2, \dee_{Y_2}}(C_2)=1, \qquad S_{Y_2, \dee_{Y_2}; W_{\bullet \bullet}}(C_2) = \frac{(1 - c) (17 c^2 - 54 c + 43)}{12(3-2c)^2},\]
which establishes the desired inequality.

\begin{detail}
Indeed, we have \(Y_2=\bb P^1_{[t_0:t_1]}\times\bb P^1_{[y_0:y_2]}\), \(\cal O_{\exx}(a,b)|_{Y_2} = \cal O_{Y_2}(a,b)\), and \(C_2\in|\cal O_{Y_2}(2,1)|\).
We have \(A_{Y_2, cR|_{Y_2}}(C_2)=1\), since the restriction \(R|_{Y_2}\) is defined by \(t_0 t_1 (y_0y_2)\), and the refinement $W^{Y_2}_{\bullet \bullet}$ by \(Y_2\) again has volume $3(2-2c)(3-2c)^2$.

We have that \(\cal F^{mt}_{C_2} W^{Y_2}_{m,j}\) is \(\lceil mt\rceil C_2 + H^0(Y_2, \cal O_{Y_2}(m(2-2c)-2\lceil mt\rceil, m(3-2c)-j-\lceil mt\rceil)\) if \(m(1-c)\geq \lceil mt\rceil\) and \(m(3-2c)-\lceil mt\rceil \geq j\), and is 0 otherwise. So, computing as above, we find
$\vol(\cal F^{mt}_{C_2} W^{Y_2}_{m,j}) = 6(1-c-t)(3-2c-t)^2$
because
\[\resizebox{1\textwidth}{!}{ $ \dim (\cal F^{mt}_{C_2} W^{Y_2}_{m,j}) = \begin{cases} (m(2-2c)-2\lceil mt\rceil+1)(m(3-2c)-j-\lceil mt\rceil+1) & \text{if }2m(1-c-t)\geq 0 \text{ and }m(3-2c-t)\geq j, \\ 0 & \text{otherwise} \end{cases} $ }\] and
\begin{equation*}\begin{split}
\vol(\cal F^{mt}_{C_2} W^{Y_2}_{m,j}) &= \lim_{m\to\infty} \frac{\sum_{j=0}^{m(3-2c-t)} (m(2-2c)-2\lceil mt\rceil+1)(m(3-2c)-j-\lceil mt\rceil+1)}{m^3/3!} \\
&= 6 \lim_{m\to\infty}\left(2-2c-2 t +\frac{1}{m}\right) \lim_{m\to\infty} \sum_{j=0}^{m(3-2c-t)} \left(3-2c-\frac{j-1}{m}- t \right)\frac{1}{m} \\
&= 6(2-2c-2 t) \int_{0}^{3-2c-t} 3-2c-x - t \; dx \\
&= 6(1-c-t)(3-2c-t)^2.
\end{split}\end{equation*}
    So
\begin{equation*}
\begin{split}
S_{Y_2, \dee_Y; W^{Y_2}_{\bullet \bullet}}(C_2) &= \frac{1}{\vol(W^{Y_2}_{\bullet \bullet})} \int_0^\infty \vol(\cal F^{mt}) \; dt \\
&= \frac{1}{3(2-2c)(3-2c)^2} \int_0^{1-c} 6(1-c-t)(3-2c-t)^2\; dt \\
&= \frac{(1 - c) (17 c^2 - 54 c + 43)}{12(3-2c)^2}
\end{split}
\end{equation*}
Thus after integrating we find \[\delta_{C_2}(Y_2, \dee_Y; W_{\bullet \bullet}) = \frac{1}{S_{Y_2, \dee_Y; W^{Y_2}_{\bullet \bullet}}(C_2)} = \frac{12(3-2c)^2}{(1 - c) (17 c^2 - 54 c + 43)} > 1 \text{ for }c\in (0,1). \]
\end{detail}

\noindent\underline{Case $C = C_{3,\pm}^a$ for \(R\neq\Rtwo\):}
First assume \(R \neq \Rtwo\), i.e. \(b_{11}\neq 1\).
Let \(C\) be a curve of the form \(C_{3,+}^a\) or \(C_{3,-}^a\) for \(a\in\bb C^*\), so that \(C\) has the form \((t_0y_1 \pm \sqrt{a}t_1y_0 = \sqrt{a}t_0y_2 \pm t_1y_1 = 0)\). Let \(Y\) be the \((1,1)\)-divisor defined by \(t_0y_1 \pm \sqrt{a}t_1y_0\). By the Abban--Zhuang method and Lemma~\ref{lem:divis-stable}, it suffices to show $\delta_{C}(Y, \dee_{Y}; W_{\bullet \bullet}) >1$ for $c\in (0,1)$.

First, note that the log discrepancy \(A_{Y_3, cR|_{Y_3}}(C)\) is either \(1\) or \(1-c\).
\begin{detail}
Indeed, the restriction of \(R\) to \(Y_3\) is defined by \(t_1 y_0 ((1 \mp \sqrt{a}) t_0 y_2 + (1 \mp \sqrt{a} b_{11}) t_1 y_1) \).
\end{detail}
Note also that the \((1,1)\)-surface \(Y\) is isomorphic to \(\bb F_1\).

\begin{detail}
Recall that, if \(a\geq -1\) and \(b\geq -2\), then \(H^i(\exx, \cal O_{\exx}(a,b)) = 0\) for \(i \geq 1\).
\end{detail}
The refinement of the filtration of \(L\) by \(Y\) is given by
\[ W_{m,j} = \begin{cases}H^0(Y, \cal O_{\exx}(m(2-2c)-j, m(3-2c)-j)|_Y) & \text{if }m(2-2c) \geq j, \\
0 & \text{otherwise}.\end{cases}\]
Using Riemann--Roch for surfaces and Kodaira vanishing, for \(m(2-2c) \geq j\) we have
\begin{equation*}\begin{split}
    \dim(W_{m,j}) &= \chi(Y, \cal O_X(mL-jY)|_Y) \\
    &= 1 + \tfrac{1}{2}(mL-jY)\cdot(mL-jY-K_\exx -Y)\cdot Y \\
    &= 1 + \tfrac{1}{2}(12c^2m^2 + 12cjm - 32cm^2 - 10cm + 3j^2 - 16jm - 5j + 21m^2 + 13m)
\end{split}\end{equation*}
so we compute that the volume is \(\vol(W_{\bullet \bullet}) = 3(2-2c)(3-2c)^2\).
\begin{detail}
    In more detail,
    one computes that the above intersection product on \(X\) is \((mL-jY)\cdot(mL-jY-K_\exx -Y)\cdot Y = (m(2-2c)-j)(m(3-2c)-j+2) + (m(3-2c)-j)(m(2-2c)-j+1) + (m(3-2c)-j)(m(3-2c)-j+2)
    = 12 c^2 m^2 + 12 c j m - 32 c m^2 - 10 c m + 3 j^2 - 16 j m - 5 j + 21 m^2 + 13 m\), so
    \[ \resizebox{1\textwidth}{!}{ $ \begin{array}{ll}
    \vol(W_{m,j}) &= \displaystyle 3 \lim_{m\to\infty} \sum_{j=0}^{m(2-2c)} \left( \tfrac{2}{m^2} + 12c^2 + 12c\tfrac{j}{m} - 32c - 10c\tfrac{1}{m} + 3(\tfrac{j}{m})^2 - 16\tfrac{j}{m} - 5\tfrac{j}{m}\tfrac{1}{m} + 21 + 13\tfrac{1}{m} \right) \tfrac{1}{m} \\
    &= \displaystyle 3 \int_0^{2-2c} (12c^2 + 12c x - 32c + 3 x^2 - 16x + 21) dx = 3(2-2c)(3-2c)^2. \end{array} $} \]
\end{detail}

Next, \(\cal F^{mt}_C W_{m,j}\) is equal to
\[\resizebox{1\textwidth}{!}{ $ \begin{cases} \lceil mt \rceil C + H^0(Y,\cal O_X(m(2-2c)-(j+\lceil mt\rceil), m(3-2c)-(j+\lceil mt\rceil) )|_Y) & \text{if }m(2-2c)-\lceil mt\rceil \geq j, \\ 0 & \text{otherwise}, \end{cases}
$} \]
so using the above computation, we find that
\begin{detail}
\begin{equation*}\begin{split}
    \vol(\cal F^{mt}_C W_{m,j}) &= \displaystyle 3 \int_0^{2-2c-t} (12c^2 + 12c (x+t) - 32c + 3 (x+t)^2 - 16(x+t) + 21) dx
\end{split}\end{equation*}
and so
\end{detail}
\(\vol(\cal F^{mt}_C W_{m,j}) = 3(3 - 2 c - t)^2 (2 - 2 c - t)\).
Thus
\[S_{Y, \dee_Y; W_{\bullet \bullet}}(C) = \frac{(1 - c) (6 c^2 - 20 c + 17)}{3 (2 c - 3)^2}, \qquad  \delta_{C}(Y, \dee_Y; W_{\bullet \bullet}) \geq \frac{3 (2 c - 3)^2}{6 c^2 - 20 c + 17} > 1 \]
for \(c\in (0,1)\). This completes the proof for \(R\neq \Rtwo\).

\noindent\underline{Case $C = C_{3,\pm}^a$ for \(\Rtwo\):}
Finally, we address the reducible surface \(\Rtwo\) defined by \((t_0 y_2 + t_1 y_1)(t_0 y_1 + t_1 y_0)\). Let \(C_3\) be a curve of the form \(C_{3,+}^a\) or \(C_{3,-}^a\) for \(a\in\bb C^*\). The curve \(C_{3,+}^1\) is addressed by Lemma~\ref{lem:R_2-sing-deltabounds}. Thus, we may assume \(C_3 \neq C_{3,+}^1\). We will use the flag \(C_3 \subset Y_3\coloneqq(y_1^2-a y_0y_2=0)\subset\exx\) and show $\delta_{C_3}(Y_3, \dee_{Y_3}; W_{\bullet \bullet}^{Y_3}) >1$ for $c\in (0,1)$.

Indeed, \(Y_3\cong\bb P^1\times\bb P^1\) via the isomorphism \(\tau\) from the proof of Lemma~\ref{lem:GinvGamma2}, \(\cal O_{\exx}(b,c)|_{Y_3}\cong\cal O_{Y_3}(b,2c)\), and \(C_3\in|\cal O_{Y_3}(1,1)|\). Since \(C \neq C_{3,+}^1\), we compute that \(A_{Y_3, \dee_{Y_3}}(C_3)\) is either \(1\) or \(1-c\), \(\vol(W_{m,j}) = 3(2-2c)(3-2c)^2\), and \(\vol(\cal F^{mt}_{C_3} W_{m,j}) = \tfrac{3}{4} (2-2c- t) (-6 + 4 c + t)^2\), so \[S_{Y_3, \dee_{Y_3}; W_{\bullet \bullet}}(C_3) = \frac{(1 - c) (17 c^2 - 54 c + 43)}{6 (3 - 2 c)^2}, \qquad \delta_{C_3}(Y_3, \dee_{Y_3}; W_{\bullet \bullet}) \geq \frac{6 (3 - 2 c)^2}{17 c^2 - 54 c + 43} > 1\]
for \(c \in (0,1)\). (For the curve \(C_{3,+}^1\), we have \(A_{Y_3, \dee_{Y_3}}(C_{3,+}^1) = 1-2c\), so using the same flag shows that \(\delta_{C_{3,+}^1}(Y_3, \dee_{Y_3}; W_{\bullet \bullet}) > 1\) for \(c\in (0,0.1614]\).)

\begin{detail}
For more details on this computation with this flag, the refinement of \(H^0(\exx,mL)\) by \(Y_3\) is \begin{equation*}\begin{split}
W^{Y_3}_{m,j} &= \begin{cases}H^0(Y_3, \cal O_{Y_3}(m(2-2c),2m(3-2c)-4j)) & \text{if }m(2-2c)\geq 0 \text{ and } m(3-2c)-2j\geq 0, \\
0 & \text{otherwise},\end{cases}
\end{split}\end{equation*} so
\[\dim W^{Y_3}_{m,j} = \begin{cases}(m(2-2c)+1)(2m(3-2c)-4j+1) & \text{if }m(2-2c)\geq 0 \text{ and } m(3-2c)-2j\geq 0, \\
0 & \text{otherwise}\end{cases}\]
and the volume of \(W^{Y_3}_{m,j}\) is
\begin{equation*}\begin{split}
\vol(W^{Y_3}_{m,j}) &= \lim_{m\to\infty} \frac{\sum_{j=0}^{m(3-2c)/2} (m(2-2c)+1)(2m(3-2c)-4j+1)}{m^3/3!} \\
&= 6 \lim_{m\to\infty}\left(2-2c+\frac{1}{m}\right) \lim_{m\to\infty} \sum_{j=0}^{m(3-2c)/2} \left(2(3-2c) - \frac{4j-1}{m}\right)\frac{1}{m} \\
&= 6(2-2c) \int_{0}^{\frac{3-2c}{2}} 2(3-2c)-4x \; dx \\
&= 3(2-2c)(3-2c)^2.
\end{split}\end{equation*}

Next, 
\[ \resizebox{1\textwidth}{!}{ $ \dim(\cal F^{mt}_{C_3} W^{Y_3}_{m,j}) = \begin{cases} (m(2-2c) -\lceil mt\rceil+1)(2m(3-2c)-4j -\lceil mt\rceil+1) & \text{if }m(2-2c-t)\geq 0 \text{ and }m(3-2c-t/2)\geq 2j, \\ 0 & \text{otherwise} \end{cases} $ }\] so
\begin{equation*}\begin{split}
\vol(\cal F^{mt}_{C_3} W^{Y_3}_{m,j}) &= \lim_{m\to\infty} \frac{\sum_{j=0}^{\frac{m(3-2c-t/2)}{2}} (m(2-2c) -\lceil mt\rceil+1)(2m(3-2c)-4j -\lceil mt\rceil+1)}{m^3/3!} \\
&= 6 (2-2c -t) \lim_{m\to\infty} \sum_{j=0}^{\frac{m(3-2c-t/2)}{2}} \left(2(3-2c)-\frac{4j-1}{m} -t \right)\frac{1}{m} \\
&= 6(2-2c- t) \int_{0}^{(3-2c-t/2)/2} 2(3-2c)-4x - t \; dx \\
&= \tfrac{3}{4} (2-2c- t) (-6 + 4 c + t)^2
\end{split}\end{equation*}
So
\begin{equation*}
\begin{split}
S_{Y_3, \dee_{Y_3}; W^{Y_3}_{\bullet \bullet}}(C) &= \frac{1}{\vol(W^{Y_3}_{\bullet \bullet})} \int_0^\infty \vol(\cal F^{mt}) \; dt \\
&= \frac{1}{3(2-2c)(3-2c)^2} \int_0^{2-2c} 3/4 (2-2c- t) (-6 + 4 c + t)^2 \; dt\\
&= \frac{(1 - c) (17 c^2 - 54 c + 43)}{6 (2 c - 3)^2}.
\end{split}
\end{equation*}

For the log discrepancy, \(R|_{Y_3}\) is defined by \(y_1(t_0^2 y_2 + t_1^2 y_0) + (a b_{11} + 1) t_0 t_1 y_0y_2\). Under the isomorphism \(\tau\) in the proof of Lemma~\ref{lem:GinvGamma2}, we see that \(\tau(R|_{Y_3})\) is defined by \[t_0^2 w_1^2 + \tfrac{ab_{11}+1}{\sqrt{a}} t_0 t_1 w_0 w_1 + t_1^2 w_0^2\]
and \(\tau(C_3) = (t_0w_1 \pm t_1 w_0 = t_0 w_1 \pm t_1 w_0 = 0)\). If \(\tfrac{ab_{11}+1}{\sqrt{a}} \neq 0, \pm 2\) then \(A_{Y_3, \dee_{Y_3}}(C_3) = 1\). If \(\tfrac{ab_{11}+1}{\sqrt{a}} = 0\) then \(A_{Y_3, \dee_{Y_3}}(C_3) = 1-c\). Finally, if \(\tfrac{ab_{11}+1}{\sqrt{a}} = \pm 2\) then \(A_{Y_3, \dee_{Y_3}}(C_3) = 1-2c\). Since we've assumed \(b_{11}=1\), the equality \(a+1 = 2 \sqrt{a}\) holds only for \(a=1\), and \(a+1 = - 2 \sqrt{a}\) has no solutions.
\end{detail}
\end{proof}

\subsection{Polystability for \texorpdfstring{\(R\in\Gthree\)}{R in non-reduced locus}}\label{sec:G3-polystability}
Recall that \(\Gthree\) is the set
\[ \{ (t_0^2 y_2^2 + t_0 t_1 (y_1^2 + b_{02} y_0 y_2) + t_1^2 y_0^2=0) \mid b_{02}\in\bb C\}/\langle b_{02}\sim -b_{02}\rangle \cup \{\Rone\}.\]

In this section we establish Lemma \ref{lem:frakG3polystable}.
\begin{lemma}\label{lem:frakG3polystable}
    For $R \in \Gthree \setminus [\Rthree]$, the pair $(\exx,cR)$ is K-polystable for $c \in (0,\tfrac{3}{4}) \cap \mathbb{Q}.$ The pair \((\exx, c\Rthree)\) is K-polystable for $c \in (0,0.25344] \cap \mathbb{Q}.$
\end{lemma}
\begin{proof}
    The statement for \(\Rone\) was shown in Lemma~\ref{lem:frakG1polystable}. If \(R \in \Gthree\setminus\{\Rone\}\), apply Theorem~\ref{lem:K-ps-G-invariant-delta} to the group $G_3$ defined in Section \ref{sec:TheGroupG3}. Required bounds on \(\delta\)-invariants are established in Lemmas~\ref{lem:divis-stable} and~\ref{lem:ineqdeltaG3}.
\end{proof}
\subsubsection{The group $G_3$ and its fixed varieties}\label{sec:TheGroupG3} Let \(R\in\Gthree\setminus \{\Rone\}\), and define automorphisms 
\begin{equation*}
    \begin{split}
        \phi_\lambda\colon ([t_0:t_1],[y_0:y_1:y_2]) & \mapsto ([\lambda t_0: \lambda^{-1} t_1],[\lambda y_0:y_1: \lambda^{-1} y_2]) \\ \iota\colon ([t_0:t_1],[y_0:y_1:y_2]) & \mapsto ([t_1: t_0],[y_2:y_1: y_0])
    \end{split}
\end{equation*}
and let $G_3$ be the subgroup of $\mathrm{Aut}(\exx, R)$ generated by \(\{\phi_\lambda\mid\lambda\in\bb C^*\}\) and \(\iota\).

\begin{lemma}\label{lem:G_3-invariant-curves}
    The \(G_3\)-invariant curves are
    \[ C_1 = (y_0=y_2=0), \quad C_{2,\pm}' = (t_0y_2 \pm t_1y_0 = y_1 = 0), \quad C^{a'}_{3,\pm} = (t_0y_2 \pm t_1 y_0 = y_1^2 - a y_0 y_2 = 0)\]
    for \(a\in\bb C^*\), and there are no \(G_3\)-invariant points.
\end{lemma}

\begin{proof}    
    Just as in Lemma~\ref{lem:GinvGamma2}, for \(Z\) a \(G_3\)-invariant subvariety of \(\exx\) the image \(\piX_2(Z)\subset\bb P^2\) is one of
    \(\bb P^2\), \((y_1=0)\), \((y_1^2-ay_0y_2 = 0)\) for some \(a \in \bb C^\ast\), or \([0:1:0]\).
    As in the proof of Lemma~\ref{lem:GinvGamma2}, there are no \(G_3\)-invariant points.
\begin{detail}
Indeed, any \(G_3\)-invariant point must be of the form $Z=([t_0:t_1],[0:1:0])$, and
     \begin{gather*}
            \phi_\lambda(Z) = ([\lambda t_0:\lambda^{-1} t_1],[0:1:0]) \\
            \iota(Z) = ([t_1:t_0],[0:1:0])
        \end{gather*}
        and it is impossible for \(Z=\iota(Z)=\phi_\lambda(Z)\) for all \(\lambda\in\bb C^*\). So there are no \(G_3\)-invariant points.
\end{detail}
    If \(Z\subset\exx\) is a \(G_3\)-invariant curve, we have three possibilities for its image in $\mathbb{P}^2$.
    If $\piX_2(Z) = [0:1:0]$, then $Z = C_1$. If $\piX_2(Z) = (y_1=0)$, then invariance under \(G_3\) and irreducibility imply that \(Z\) is \(C_{2,+}'\) or \(C_{2,-}'\).
\begin{detail}
    Indeed, let $g = \sum_{0 \leq i \leq k,0 \leq j \leq d}{b_{i,j}t_0^{k-i}t_1^iy_0^{d-j}y_2^j}$ be the bidegree $(k,d)$ defining equation of $Z \subset \bb P^1_{[t_0:t_1]} \times \bb P^1_{[y_0:y_2]}$. We have $\sigma_\lambda(b_{i,j}t_0^{k-i}t_1^iy_0^{d-j}y_2^j) = b_{i,j}\lambda^{k-2i+d-2j}t_0^{k-i}t_1^iy_0^{d-j}y_2^j$ for every $i,j$. Since $g=0$ is irreducible, $g$ is not divisible by $t_0$ or $y_2$, so $b_{k,j'} \neq 0 \neq b_{i',0}$ for some $i',j'$, which means that $\lambda^{-k+d-2j'}=\lambda^{k-2i'+d}$ for all $\lambda \in \bb C^\ast$, i.e. $k+j' = i'$. But $i' \leq k$, so this forces $j' = 0$ and $i' = k$. In other words, $b_{k,0} \neq 0$ and $\sigma_\lambda$ must send $f$ to $\lambda^{-k+d} f$. Similarly, by the fact that $g$ is not divisible by $t_1$ or $y_0$, one can show $b_{0,d} \neq 0$ and $\sigma_\lambda$ must also send $f$ to $\lambda^{k-d} f$. Hence, $k-d=-k+d$, i.e. $k=d$, and for every $i,j$ with $b_{i,j} \neq 0$, we must have $\lambda^{k-2i+d-2j} = \lambda^0$ for every $\lambda \in \bb C^\ast$, i.e. $i+j=d$. Thus, we can rewrite $g$ as $\sum_{0 \leq j \leq d}{b_jt_0^jt_1^{d-j}y_0^{d-j}y_2^j} = \sum_{0 \leq j \leq d}{b_j(t_0y_2)^j(t_1y_0)^{d-j}}$, which is irreducible if and only if $d=1$, in which case $g = b_0t_1y_0+b_1t_0y_2$ for $b_0 \neq 0 \neq b_1$. By the invariance of $(g=0)$ under $\iota$, we further conclude that $b_0 = \pm b_1$. Thus, $Z = C_{2,\pm}'$.
\end{detail}
    If $\piX_2(Z) = (y_1^2-ay_0y_2=0)$, then arguing as in the proof of Lemma~\ref{lem:GinvGamma2}, \(Z\) is \(C^{a'}_{3,+}\) or \(C^{a'}_{3,-}\).
\begin{detail}
    Indeed, $Z$ is a curve inside $\bb P^1_{[t_0:t_1]} \times (y_1^2-ay_0y_2 = 0)$. We consider the image of $Z$ under the isomorphism \begin{align*}
    \tau \colon \bb P^1_{[t_0:t_1]} \times (y_1^2-ay_0y_2 = 0) &\xrightarrow{\cong} \bb P^1_{[t_0:t_1]} \times \bb P^1_{[w_0:w_1]} \\
    ([t_0:t_1],[w_0^2:\sqrt{a}w_0w_1:w_1^2]) &\mapsfrom ([t_0:t_1],[w_0:w_1]).
    \end{align*}
    Since $Z$ is $G_3$-invariant, $\psi(Z)$ is invariant under the automorphisms \begin{align*}
    \widetilde{\phi}_\lambda \colon ([t_0:t_1],[w_0:w_1]) &\mapsto ([\lambda^2 t_0:\lambda^{-2} t_1],[\lambda w_0:\lambda^{-1} w_1]) \qquad \textrm{for all } \lambda \in \CC^\ast \\
    \widetilde{\iota} \colon ([t_0:t_1],[w_0:w_1]) &\mapsto ([t_1:t_0],[w_1:w_0]).
    \end{align*}
    Similar to the preceding paragraph, one can deduce from the irreducibility of $\tau(Z)$ and invariance of $\tau(Z)$ under $\widetilde{\sigma}_\lambda$ that $\tau(Z) = (t_0w_1^2 + ct_1w_0^2 = 0) \subset \bb P^1_{[t_0:t_1]} \times \bb P^1_{[w_0:w_1]}$ for some $c \in \CC^\ast$. The invariance of $\tau(Z)$ under $\widetilde{\iota}$ further implies that $\tau(Z) = (t_0w_1^2 \pm t_1w_0^2 = 0) \subset \bb P^1_{[t_0:t_1]} \times \bb P^1_{[w_0:w_1]}$. Consequently, $Z = C^{a'}_{3,\pm}$.
\end{detail}
\end{proof}

\subsubsection{Lower bounds on \(\delta\)-invariants of $G_3$-invariant subvarieties}
\begin{lemma}\label{lem:ineqdeltaG3}
Let \(R \in \Gthree\setminus \{\Rone\}\), and let \(C\) be one of the \(G_3\)-invariant curves in Lemma~\ref{lem:G_3-invariant-curves}.
\begin{enumerate}
    \item If \(b_{02}\neq \pm 2\) or \(C \neq C_{2,\pm}'\), then \(\delta_{\exx,\dee}(C)>1\) for all \(c\in (0,\tfrac{3}{4})\).
    \item If \(b_{02} = \pm 2\), then \(\delta_{\exx,\dee}(C_{2,\pm}')>1\) for all \(c\in (0,0.2534]\).
\end{enumerate}
\end{lemma}

If \(b_{02} = \pm 2\), then \(R\) is isomorphic to \(\Rthree\) under \(\Aut(X)\). This surface is non-normal and \(C_{2,\pm}'\) is its singular locus. We will address \(\Rthree\) for \(c > 0.2534\) separately in Section~\ref{sec:wall-crossing-c_0}.

\begin{proof}The computations proceed as in the proof of Lemma~\ref{lem:G2deltabounds}.

\noindent\underline{Case $C = C_1'$:} Since \(Y_1 \coloneqq (y_0=0)\) is not a component of any \(R \in \Gthree\), the inequality \(\delta_{\exx,\dee}(C_1)>1\) was established in Lemma~\ref{lem:G2deltabounds}.

\noindent\underline{Case $C = C_{2,\pm}'$:} We use the flag \(C'_2\subset Y_2\coloneqq(y_1=0)\subset\exx\). By Abban--Zhuang theory and Lemma~\ref{lem:divis-stable}, it suffices to show $\delta_{C_{2,\pm}'}(Y_2, \dee_{Y_2}; W_{\bullet \bullet}) >1$ for $c\in (0,1)$. First, note that the log discrepancy \(A_{Y_2, cR|_{Y_2}}(C_{2,\pm}')\) is \(1-2c\) if \(b_{02} = \pm 2\), and is \(1\) otherwise. We compute that \(S_{Y_2, \dee_{Y_2}; W_{\bullet \bullet}}(C_{2,\pm}') = ((1 - c) (6 c^2 - 20 c + 17))/(3 (3 - 2 c)^2)\), so
\[ \delta_{C_{2,\pm}'}(Y_2, \dee_{Y_2}; W_{\bullet \bullet}) = \begin{cases} \displaystyle \frac{3 (2 c - 3)^2 (1 - 2 c)}{(1 - c) (6 c^2 - 20 c + 17)} & \text{if }b_{02} = \pm 2 \\ \displaystyle \frac{3 (2 c - 3)^2}{(1 - c) (6 c^2 - 20 c + 17)} & \text{otherwise}\end{cases} \]
so if \(b_{02} \neq \pm 2\), then we have \(\delta_{C_{2,\pm}'}(Y_2, \dee_{Y_2}; W_{\bullet \bullet}) > 1\) for \(c \in (0,1)\). If \(b_{02} = \pm 2\), then \(\delta_{C_{2,\pm}'}(Y_2, \dee_{Y_2}; W_{\bullet \bullet}) > 1\) for \(c \in (0,0.2534]\) (more precisely, for \(c\) less than the irrational number that is the smallest root of the cubic polynomial \(18 c^3 - 58 c^2 + 53 c - 10\)).

\begin{detail}
In more detail, \(Y_2=\bb P^1_{[t_0:t_1]}\times\bb P^1_{[y_0:y_2]}\), \(\cal O_{\exx}(a,b)|_{Y_2} = \cal O_{Y_2}(a,b)\), and \(C_{2,\pm}'\in|\cal O_{Y_2}(1,1)|\).
The restriction \(R|_{Y_2}\) is defined by \(t_0^2 y_2^2 + t_0 t_1 (b_{02} y_0y_2) + t_1^2 y_0^2\), so \[A_{Y_2, cR|_{Y_2}}(C_{2,\pm}')= \begin{cases} 1-2c & \text{if }b_{02} = \pm 2, \\ 1 & \text{otherwise}.\end{cases}\]
We already computed in Lemma~\ref{lem:G2deltabounds} that
\(\vol(W^{Y_2}_{m,j}) = 3(2-2c)(3-2c)^2.\)
Next, we have
\[\resizebox{1\textwidth}{!}{ $ \dim(\cal F^{mt}_{C_{2,\pm}'} W_{m,j}) = \begin{cases} (m(2-2c)-\lceil mt\rceil+1)(m(3-2c)-j-\lceil mt\rceil+1) & \text{if }m(2-2c)\geq \lceil mt \rceil \text{ and }m(3-2c)-\lceil mt\rceil \geq j, \\ 0 & \text{otherwise} \end{cases}$ }\] and we compute
\begin{equation*}\begin{split}
\vol(\cal F^{mt}_{C_{2,\pm}'} W_{m,j}) &= \lim_{m\to\infty} \frac{\sum_{j=0}^{m(3-2c-t)} (m(2-2c)-\lceil mt\rceil+1)(m(3-2c)-j-\lceil mt\rceil+1)}{m^3/3!} \\
&= 6 \lim_{m\to\infty}\left(2-2c-t +\frac{1}{m}\right) \lim_{m\to\infty} \sum_{j=0}^{m(3-2c-t)} \left(3-2c-\frac{j-1}{m}- t \right)\frac{1}{m} \\
&= 6(2-2c-t) \int_{0}^{3-2c-t} 3-2c-x - t \; dx = 3(2-2c-t)(3-2c-t)^2.
\end{split}\end{equation*}
Integrating, we find
\begin{equation*}
\begin{split}
S_{Y_2, \dee_{Y_2}; W^{Y_2}_{\bullet \bullet}}(C'_2) &= \frac{1}{3(2-2c)(3-2c)^2} \int_0^{2-2c} 3(2-2c-t)(3-2c-t)^2 \; dt \\
&= \frac{(1 - c) (6 c^2 - 20 c + 17)}{3 (3 - 2 c)^2}.
\end{split}
\end{equation*}
\end{detail}

\noindent\underline{Case $C = C^{a'}_{3,\pm}$:}
For the curve \(C^{a'}_{3,\pm}\) defined by \( (t_0y_2 \pm t_1 y_0 = y_1^2 - a y_0 y_2 = 0)\), we use the flag \(C^{a'}_{3,\pm}\subset Y'_{3,\pm} \coloneqq (t_0 y_2 \pm t_1 y_0=0)\subset\exx\). By Abban--Zhuang theory and Lemma~\ref{lem:divis-stable}, it suffices to bound $\delta_{C^{a'}_{3,\pm}}(Y'_{3,\pm}, \dee_{Y'_{3,\pm}}; W_{\bullet \bullet})$. First, note that the restriction of \(R\) to \({Y'_{3,\pm}}\) is defined by \(t_0t_1(y_1^2 + (b_{02} \mp 2) y_0 y_2)\), so the log discrepancy \(A_{Y'_{3,\pm}, cR|_{Y'_{3,\pm}}}(C^{a'}_{3,\pm})\) is \(1-c\) if \(-a = b_{02} \mp 2\), and is \(1\) otherwise. In what follows, we will denote \(C = C^{a'}_{3,\pm}\) and \(Y = Y'_{3,\pm}\). We note that the \((1,1)\)-surface \(Y\) is isomorphic to \(\bb F_1\).

\begin{detail}
Recall that, if \(a\geq -1\) and \(b\geq -2\), then \(H^i(\exx, \cal O_{\exx}(a,b)) = 0\) for \(i \geq 1\).
\end{detail}
The refinement of the filtration of \(L\) by \(Y\) is given by
\[ W_{m,j} = \begin{cases}H^0(Y, \cal O_{\exx}(m(2-2c)-j, m(3-2c)-j)|_Y) & \text{if }m(2-2c) \geq j, \\
0 & \text{otherwise}.\end{cases}\]
Using Riemann--Roch for surfaces and Kodaira vanishing, for \(m(2-2c) \geq j\) we have
\begin{equation*}\begin{split}
    \dim(W_{m,j}) &= \chi(Y, \cal O_X(mL-jY)|_Y) \\
    &= 1 + \tfrac{1}{2}(mL-jY)\cdot(mL-jY-K_\exx -Y)\cdot Y \\
    &= 1 + \tfrac{1}{2}(12c^2m^2 + 12cjm - 32cm^2 - 10cm + 3j^2 - 16jm - 5j + 21m^2 + 13m)
\end{split}\end{equation*}
so we compute that the volume is \(\vol(W_{\bullet \bullet}) = 3(2-2c)(3-2c)^2\).
\begin{detail}
    In more detail,
    one computes that the above intersection product on \(X\) is \((mL-jY)\cdot(mL-jY-K_\exx -Y)\cdot Y = (m(2-2c)-j)(m(3-2c)-j+2) + (m(3-2c)-j)(m(2-2c)-j+1) + (m(3-2c)-j)(m(3-2c)-j+2)
    = 12 c^2 m^2 + 12 c j m - 32 c m^2 - 10 c m + 3 j^2 - 16 j m - 5 j + 21 m^2 + 13 m\), so
    \[ \resizebox{1\textwidth}{!}{ $ \begin{array}{ll}
    \vol(W_{m,j}) &= \displaystyle 3 \lim_{m\to\infty} \sum_{j=0}^{m(2-2c)} \left( \tfrac{2}{m^2} + 12c^2 + 12c\tfrac{j}{m} - 32c - 10c\tfrac{1}{m} + 3(\tfrac{j}{m})^2 - 16\tfrac{j}{m} - 5\tfrac{j}{m}\tfrac{1}{m} + 21 + 13\tfrac{1}{m} \right) \tfrac{1}{m} \\
    &= \displaystyle 3 \int_0^{2-2c} (12c^2 + 12c x - 32c + 3 x^2 - 16x + 21) dx = 3(2-2c)(3-2c)^2. \end{array} $} \]
\end{detail}
Next, \(\cal F^{mt}_C W_{m,j}\) is equal to
\[\resizebox{1\textwidth}{!}{ $ \begin{cases} \lceil mt \rceil C + H^0(Y,\cal O_X(m(2-2c)-j, m(3-2c)-j-2\lceil mt\rceil )|_Y) & \text{if }m(2-2c) \geq j \text{ and }m(3-2c)-2\lceil mt\rceil \geq j, \\ 0 & \text{otherwise}. \end{cases}
$} \]
Let
\(S\in|\cal O_X(0,1)|\). Then, as above, we compute that
\begin{equation*}\begin{split}
    \dim(\cal F^{mt}_C W_{m,j}) &= \chi(Y, \cal O_X(m(2-2c)-j, m(3-2c)-j-2\lceil mt\rceil )|_Y) \\
    &= 1 + \tfrac{1}{2}(mL-jY-2 \lceil mt\rceil S )\cdot(mL-jY-2 \lceil mt\rceil S -K_\exx - Y )\cdot Y \\
\end{split}\end{equation*}
for \(j \leq \min\{m(2-2c), m(3-2c)-2\lceil mt\rceil\}\), and \(\dim(\cal F^{mt}_C W_{m,j})=0\) otherwise. Expanding out the intersection product, we compute that
\begin{equation*}\begin{split}
    \vol(\cal F^{mt}_C W_{m,j}) &\leq \displaystyle\lim_{m\to\infty} \tfrac{\sum_{j=0}^{m(3-2c)-2\lceil mt\rceil} 2 + (mL-jY-2 \lceil mt\rceil S )\cdot(mL-jY-2 \lceil mt\rceil S -K_\exx - Y )\cdot Y}{2 m^3/3!} \\
    &= 3(2-2c) (3 - 2 c - 2 t)^2.
\end{split}\end{equation*}
\begin{detail}
In more detail, the intersection product \((mL-jY-2 \lceil mt\rceil S )\cdot(mL-jY-2 \lceil mt\rceil S -K_\exx - Y )\cdot Y\) is equal to \(12c^2m^2 + 12cjm + 16cm\lceil mt\rceil - 32cm^2 - 10cm + 3j^2 + 8j\lceil mt\rceil - 16jm - 5j + 4\lceil mt\rceil^2 - 20m\lceil mt\rceil - 6\lceil mt\rceil + 21m^2 + 13m\), so
\[ \resizebox{1\textwidth}{!}{ $ \begin{array}{ll}
    \vol(\cal F^{mt}_C W_{m,j}) &\leq \displaystyle\lim_{m\to\infty} \tfrac{\sum_{j=0}^{m(3-2c)-2\lceil mt\rceil} 2 + 12c^2m^2 + 12cjm + 16ctm^2 - 32cm^2 - 10cm + 3j^2 + 8jtm - 16jm - 5j + 4t^2m^2 - 20tm^2 - 6tm + 21m^2 + 13m}{2 m^3/3!} \\
    &= 3 \lim_{m\to\infty} \sum_{j=0}^{?} \left( 12c^2m^2 + 12cjm + 16ctm^2 - 32cm^2 - 10cm + 3j^2 + 8jtm - 16jm - 5j + 4t^2m^2 - 20tm^2 - 6tm + 21m^2 + 13m \right) \frac{1}{m} \\
    &= \displaystyle 3 \int_0^{3-2c-2t} (12c^2 + 12c x + 16ct - 32c + 3x^2 + 8t x - 16 x + 4t^2 - 20t + 21) dx \\
    &= 3(2-2c) (3 - 2 c - 2 t)^2.
\end{array} $} \]
\end{detail}
So
\begin{equation*}\begin{split}
    S_{Y, \dee_Y; W_{\bullet \bullet}}(C) &= \frac{1}{\vol(W_{\bullet \bullet})} \int_0^\infty \vol(\cal F^{mt}_C W_{m,j}) dt \\
    & \leq \frac{1}{3(2-2c)(3-2c)^2} \int_0^{(3-2c)/2} 3(2-2c) (3 - 2 c - 2 t)^2 \; dt = \frac{3 - 2 c}{6}
\end{split}\end{equation*}
and we find that
\[ \delta_{C}(Y, \dee_Y; W_{\bullet \bullet}) \geq \begin{cases} \displaystyle \frac{6 (1 - c)}{3 - 2 c} & > 1 \text{ for }c\in(0,\tfrac{3}{4}) \text{ if }-a = b_{02} \mp 2 \\ \displaystyle \frac{6}{3 - 2 c} & > 1 \text{ for }c\in(0,1) \text{ otherwise}.\end{cases} \]
\end{proof}

\subsection{Polystability for \texorpdfstring{\(R\in\Gfour\)}{R in 2 A1 locus}}\label{sec:G4-polystability} Recall that \(\Gfour\) is the set
\[\{(t_0^2 y_1^2+t_0t_1(b_{11}y_1^2+y_0 y_2)+t_1^2y_1^2=0) \mid b_{11}\in\bb C\}/\langle b_{11}\sim -b_{11}\rangle \cup \{\Rone\}.\]

\begin{lemma}\label{lem:frakG4polystable}
    For $R \in \Gfour$, the pair $(\exx,cR)$ is K-polystable for $c \in (0,\tfrac{3}{4}) \cap \mathbb{Q}$.
\end{lemma}

\begin{proof}
    For \(R=\Rone\), this was established in Lemma~\ref{lem:frakG1polystable}. For \(R\in\Gfour\setminus\{\Rone\}\), apply Theorem~\ref{lem:K-ps-G-invariant-delta} to the group $G_1$. The required lower bounds on \(\delta\)-invariants are established in Lemma~\ref{lem:G4deltabounds}.
\end{proof}

\subsubsection{The group $G_1$ and its fixed varieties}
Let \(R\in\Gfour\setminus \{\Rone\}\), and define elements of \(\Aut(\exx, R)\):
\begin{equation*}
    \begin{split}
        \psi_\lambda\colon ([t_0:t_1],[y_0:y_1:y_2]) & \mapsto ([t_0:t_1],[\lambda y_0:y_1: \lambda^{-1} y_2]) \\ \iota_1\colon ([t_0:t_1],[y_0:y_1:y_2]) & \mapsto ([t_1: t_0],[y_0:y_1: y_2]) \\ \iota_2\colon ([t_0:t_1],[y_0:y_1:y_2]) & \mapsto ([t_0: t_1],[y_2:y_1: y_0]).
    \end{split}
\end{equation*}
Let \(G_1\) be the subgroup of $\Aut(\exx, R)$ generated by \(\{\psi_\lambda\mid\lambda\in\bb C^*\}\), \(\iota_1\), and \(\iota_2\).

\begin{lemma}\label{lem:G_1-invariant-curves}
    The \(G_1\)-invariant curves are
    \[C_1 = (y_0=y_2=0), \quad C_{2,\pm}'' = (t_0\pm t_1 = y_1 = 0), \quad C^{a''}_{3,\pm} = (t_0 \pm t_1 = y_1^2 - a y_0 y_2=0)\] for \(a\in\bb C^*\),
    and the \(G_1\)-invariant points are \(([\pm 1:1],[0:1:0])\).
\end{lemma}

\begin{proof}    
    Suppose \(Z\) is a \(G_1\)-invariant subvariety of \(\bb P^1\times\bb P^2\). Then \(\piX_2(Z)\subset\bb P^2\) is irreducible and invariant under \([y_0:y_1:y_2] \mapsto [\lambda y_0 : y_1 : \lambda^{-1} y_2]\) and \([y_0:y_1:y_2]\mapsto [y_2:y_1:y_0]\), and so must be one of
    \(\bb P^2\), \((y_1=0)\), \((y_1^2-ay_0y_2 = 0)\) for some \(a \in \bb C^\ast\) or \([0:1:0]\).
    Similarly, $\piX_1(Z)$ is one of \(\bb P^1\), \([1:1]\), or \([1:-1]\).
    Let us pin down the possibilities for $Z$.

    We first consider the case when \(Z\) is a \(G_1\)-curve. If \(\piX_2(Z)=[0:1:0]\) then \(Z=C_1\).
    If \(\piX_2(Z)\) is \(y_1=0\), then invariance under \(\psi^*_\lambda\) and \(\iota_1\), together with irreducibility of \(Z\), imply that \(Z\) is \(C_{2,+}''\) or \(C_{2,-}''\).
\begin{detail}
    In more detail, $Z \subset \bb P^1_{[t_0:t_1]} \times \bb P^1_{[y_0:y_2]}$. Let $g = \sum_{0 \leq i \leq k,0 \leq j \leq d}{b_{i,j}t_0^{k-i}t_1^iy_0^{d-j}y_2^j}$ be its defining equation of bidegree $(k,d)$. We have $\psi_\lambda^*(b_{i,j}t_0^{k-i}t_1^iy_0^{d-j}y_2^j) = b_{i,j}\lambda^{d-2j}t_0^{k-i}t_1^iy_0^{d-j}y_2^j$ for every $i,j$. If $d>0$, then $(g=0)$ is invariant under $\psi_\lambda$ if and only if there exists $0 \leq j' \leq d$ such that $b_{i,j} = 0$ whenever $j \neq j'$, in which case $y_0$ or $y_2$ would divide $g$, i.e $g$ is reducible unless $g = by_0$ or $by_2$ for some $b \in \bb C^\ast$. But that would contradict, for example, the hypothesis that $\piX_2(Z) = (y_1=0)$. Thus, $d=0$ and we can write $g = \sum_{0 \leq i \leq k}{b_it_0^{k-i}t_1^i}$. Then $g$ is irreducible if and only if $k=1$, in which case $g = b_0t_0+b_1t_1$ with either $b_0$ or $b_1$ nonzero. By the invariance of $(g=0)$ under $\iota_1$, $b_0 = \pm b_1$. Thus, $Z = C_{2,\pm}''$.
\end{detail}
    If $\piX_2(Z) = (y_1^2-ay_0y_2=0)$, then similar arguments to those in the proof of Lemma~\ref{lem:GinvGamma2} show that \(Z\) is \(C^{a''}_{3,+}\) or \(C^{a''}_{3,-}\).
\begin{detail}
    Indeed, $Z$ is a curve inside $W = \bb P^1_{[t_0:t_1]} \times (y_1^2-ay_0y_2 = 0)$. We consider the image of $Z$ under the isomorphism \begin{align*}
    \qquad \tau \colon \bb P^1_{[t_0:t_1]} \times (y_1^2-ay_0y_2 = 0) &\xrightarrow{\cong} \bb P^1_{[t_0:t_1]} \times \bb P^1_{[w_0:w_1]} \\
            ([t_0:t_1],[w_0^2:\sqrt{a}w_0w_1:w_1^2]) &\mapsfrom ([t_0:t_1],[w_0:w_1]).
            \end{align*}
            Since $Z$ is $G_1$-invariant, $\tau(Z)$ is invariant under the automorphisms \begin{align*}
            \qquad  \widetilde{\phi}_\lambda \colon ([t_0:t_1],[w_0:w_1]) &\mapsto ([t_0:t_1],[\lambda w_0:\lambda^{-1} w_1]) \qquad \textrm{for all } \lambda \in \CC^\ast \\
            \widetilde{\iota}_1 \colon ([t_0:t_1],[w_0:w_1]) &\mapsto ([t_1:t_0],[w_0:w_1]) \\
            \widetilde{\iota}_2 \colon ([t_0:t_1],[w_0:w_1]) &\mapsto ([t_0:t_1],[w_1:w_0]).
            \end{align*}
            Similar to the preceding paragraph, one can deduce from the irreducibility of $\tau(Z)$ and invariance of $\tau(Z)$ under $\widetilde{\sigma}_\lambda$ and $\widetilde{\iota}_1$ that $\tau(Z) = (t_0 \pm t_1 = 0) \subset \bb P^1_{[t_0:t_1]} \times \bb P^1_{[w_0:w_1]}$ for some $c \in \CC^\ast$. Consequently, $Z = C^{a''}_{3,\pm}$.
\end{detail}

    Finally, if \(Z\in \bb P^1\times\bb P^2\) is a \(G\)-invariant point, then \(Z=(([t_0:t_1],[0:1:0])\). Invariance under \(\iota_1\) implies that \(Z\) is \(([1:1],[0:1:0])\) or \(([-1:1],[0:1:0])\).
\end{proof}

\begin{lemma}\label{lem:G4deltabounds}
Let \(R \in \Gfour\), and let \(Z\) be one of the \(G_1\)-invariant curves or points in Lemma~\ref{lem:G_1-invariant-curves}. Then \(\delta_{\exx,\dee}(Z) > 1\) for all \(c\in (0,\tfrac{3}{4})\). \end{lemma}

\begin{proof}The computations proceed as in the proof of Lemma~\ref{lem:G2deltabounds}.

\noindent\underline{Case $Z = C_1$:} The inequality \(\delta_{\exx,\dee}(C_1) > 1\) for \(c\in (0,1)\) was proven in Lemma~\ref{lem:G2deltabounds}.

\noindent\underline{Case $Z = C_{2,\pm}''$:}
Using the flag \(C''_2\subset Y_2\coloneqq(y_1=0)\subset \exx\), we compute that \(A_{Y_2, \dee_{Y_2}}(C''_2)=1\) and \(S_{Y_2, \dee_{Y_2}; W_{\bullet \bullet}}(C''_2) = 1-c\), so $\delta_{C''_2}(Y_2, \dee_{Y_2}; W_{\bullet \bullet}) >1$ for $c\in (0,1)$. By Abban--Zhuang theory and Lemma~\ref{lem:divis-stable}, this establishes the desired inequality.

\begin{detail}
In more detail, we have \(Y_2=\bb P^1 \times\bb P^1\), \(\cal O_\exx(a,b)|_{Y_2} = \cal O_{Y_2}(a,b)\), and \(C''_2\in|\cal O_{Y_2}(1,0)|\).
The restriction \(R|_{Y_2}\) is defined by \(t_0 t_1 y_0y_2\), so \(A_{Y_2, \dee_{Y_2}}(C''_2)=1\). We already computed in Lemma~\ref{lem:G2deltabounds} that \(\vol(W_{m,j}) = 3(2-2c)(3-2c)^2\). Next,
\[\cal F^{mt}_{C''_2} W_{m,j} = \{\lceil mt\rceil C_2 + H^0(Y_2, \cal O_{Y_2}(m(2-2c)-\lceil mt\rceil, m(3-2c)-j)\}\]
so
\[ \resizebox{1\textwidth}{!}{ $ \dim(\cal F^{mt}_{C''_2} W_{m,j}) = \begin{cases} (m(2-2c)-\lceil mt\rceil+1)(m(3-2c)-j+1) & \text{if }m(2-2c-t)\geq 0 \text{ and }m(3-2c)\geq j, \\ 0 & \text{otherwise} \end{cases} $ }\] and
\begin{equation*}\begin{split}
\vol(\cal F^{mt}_{C''_2} W_{m,j}) &= 6 \lim_{m\to\infty}\left(2-2c-t +\frac{1}{m}\right) \lim_{m\to\infty} \sum_{j=0}^{m(3-2c)} \left(3-2c-\frac{j-1}{m} \right)\frac{1}{m} \\
&= 6(2-2c-t) \int_{0}^{3-2c} 3-2c-x \; dx = 3(2-2c-t)(3-2c)^2.
\end{split}\end{equation*}
Thus,
\begin{equation*}
\begin{split}
S_{Y_2, \dee_{Y_2}; W_{\bullet \bullet}}(C''_2) &= \frac{1}{3(2-2c)(3-2c)^2} \int_0^{2-2c} 3(2-2c-t)(3-2c)^2 dt \\
&= 1 - c.
\end{split}
\end{equation*}
\end{detail}

\noindent\underline{Case $Z = C^{a''}_{3,\pm}$:}
Recall that the curve \(C^{a''}_{3,\pm}\) is defined by \(t_0 \pm t_1 = y_1^2 - a y_0 y_2=0\). We use the flag \(C^{a''}_{3,\pm}\subset Y''_{3,\pm} \coloneqq (t_0\pm t_1=0) = [1:\pm 1]\times \bb P^2 \subset\exx\). In what follows, write \(C = C^{a''}_{3,\pm}\) and \(Y = Y''_{3,\pm}\).

By Abban--Zhuang theory and Lemma~\ref{lem:divis-stable}, it suffices show that \(\delta_C(Y, \dee_Y; W_{\bullet \bullet}) >1\) for \(c\in (0,1)\). We compute that \(\vol(W_{\bullet \bullet}) = 3(2-2c)(3-2c)^2\), \(\vol(\cal F^{mt}_C W_{m,j}) = 3 (2-2c)(3-2c-2t)^2\), and \(S_{Y, \dee_Y; W_{\bullet \bullet}}(C) = (3 - 2 c)/6\). The log discrepancy \(A_{Y, \dee_Y}(C)\) is either \(1\) or \(1-c\), so \[\delta_C(Y, \dee_Y; W_{\bullet \bullet}) \geq \frac{6 (1 - c)}{3 - 2 c} > 1 \text{ for }c\in(0,\tfrac{3}{4}).\]

\begin{detail}
    In more detail, first note that \(\cal O_\exx (a,b)|_Y=\cal O_{\bb P^2}(b)\). The refinement by \(Y\) is given by
    \[W_{m,j} = \begin{cases}H^0(\bb P^2, \cal O(m(3-2c))) & \text{if }m(2-2c)-j\geq 0 \text{ and } m(3-2c)\geq 0, \\ 0 & \text{otherwise},\end{cases} \]
    so
    \[\dim W_{m,j} = \begin{cases} \displaystyle \frac{(m(3-2c) + 2)(m(3-2c) + 1)}{2} & \text{if }m(2-2c)-j\geq 0 \text{ and } m(3-2c)\geq 0, \\ 0 & \text{otherwise}\end{cases}\]
    and the volume of \(W_{\bullet \bullet}\) is
    \begin{equation*}\begin{split}
    \vol(W_{\bullet \bullet}) &= 3 \lim_{m\to\infty} \sum_{j=0}^{m(2-2c)} \left( (3-2c)^2 + 3(3-2c)\frac{1}{m} + 2\frac{1}{m^2} \right)\frac{1}{m} \\
    &= 3 \int_0^{2-2c} (3-2c)^2 \; dx = 3(2-2c)(3-2c)^2.
    \end{split}\end{equation*}

    Next,
    \[\cal F^{mt}_C W_{m,j} = \begin{cases}\lceil mt\rceil C + H^0(\bb P^2, \cal O(m(3-2c) - 2\lceil mt\rceil)) & \text{if } m(2-2c)\geq j \text{ and } m(3-2c)\geq 2\lceil mt \rceil \\
    0 & \text{otherwise}\end{cases}\]
    so
    \[ \resizebox{1\textwidth}{!}{ $ \dim(\cal F^{mt}_C W_{m,j}) = \begin{cases} \displaystyle\frac{(m(3-2c) - 2\lceil mt\rceil+2)(m(3-2c) - 2\lceil mt\rceil+1)}{2} & \text{if } m(2-2c)\geq j \text{ and } m(3-2c)\geq 2\lceil mt \rceil \\
    0 & \text{otherwise} \end{cases} $ }\]
    and the volume is
    \begin{equation*}\begin{split}
    \vol(\cal F^{mt}_C W_{m,j}) &= 3 \lim_{m\to\infty} \sum_{j=0}^{m(2-2c)} \left((3-2c-2t)^2 + 3\frac{(3-2c-2t)}{m} + \frac{2}{m^2}\right) \frac{1}{m} \\
    &= 3 \int_0^{2-2c} (3-2c-2t)^2 dx = 3 (2-2c)(3-2c-2t)^2. \end{split}\end{equation*}
    Thus, \[ S_{Y, \dee_Y; W_{\bullet \bullet}}(C) = \frac{1}{3(2-2c)(3-2c)^2} \int_0^{(3-2c)/2} 3 (2-2c)(3-2c-2t)^2 dt = \frac{3 - 2 c}{6}.\]
\end{detail}

\noindent\underline{Case $Z = ([\pm 1:1],[0:1:0])$:} Write \(\pt_{\pm} \coloneqq ([\pm 1 :1],[0:1:0])\). First, note that for any \(\pt\notin R\) the inequality \(\delta_{\exx, \dee}(\pt)>1\) holds for all \(c\in(0,1)\) by \cite[Lemma 2.2]{CFKP23}.

If \(\pt_{\pm}\in R\) then \(b_{11} = \mp 2\), so (up to exchanging \(+\) and \(-\)) we may assume \(R\) is the surface \(\Rfour\) defined by \((t_0 + t_1)^2 y_1^2 + t_0 t_1 y_0 y_2\). The associated discriminant curve is an ox, whose nodal point is \(\piR(\pt_-) = [0:1:0]\). The fiber of \(\piR_2\) over \([0:1:0]\) is finite, so \(\pt_- \in \Rfour\) is an \(A_1\) singularity. The fiber of \(\piR_1\) over \(\piR_1(\pt_-)\) is reduced, so \(\delta_{p_-}(\exx, c\Rfour) > 1\) for \(c \in (0,1)\) by Theorem~\ref{thm:local-delta-A_n-singularities}\eqref{part:local-delta-A_n-singularities:A_1-finite-reduced}.
\end{proof}

\section{Wall crossing and the K-moduli replacement at \texorpdfstring{\(\cnot\)}{c1}}\label{sec:wall-crossing-c_0}

\subsection{K-stability properties of \texorpdfstring{$\Rthree$}{R3}}
Let \(\Rthree\) be the non-normal surface in Theorem~\ref{thm:moduli-spaces-2.18}, defined by \[t_0t_1y_1^2 + (t_0y_2 + t_1y_0)^2 = 0.\]
The singular locus of \(\Rthree\) is the curve $\Cthree$ defined by $(y_1 = t_0y_2+t_1y_0=0)$. In this section, we establish K-stability properties of the pair \((\bP^1 \times \bP^2, c \Rthree)\). 
Let \(\cnot \approx 0.472\) be the irrational number defined as the smallest root of the polynomial \(10c^3-34c^2+35c-10\).

\begin{proposition}\label{prop:K-stability-Rthree}
    For \(c\in(0,1) \cap \bb Q\), the pair \((\mathbb{P}^1 \times \mathbb{P}^2,c\Rthree)\) is K-polystable if \(c <\cnot\) and is K-unstable if \(c > \cnot\).
\end{proposition}

\begin{remark}
    Recall from Lemma~\ref{lem:worse-than-A_n} that there are GIT semistable but not polystable surfaces \(R\) with type \(D\) singularities that are \(S\)-equivalent to \(\Rthree\). In fact, for these \(R\), the pair \((\bb P^1\times\bb P^2,cR)\) becomes K-unstable for \(c>\cnot\) (see Remark~\ref{rem:worse-than-A_n-after-wall} in the next section).
\end{remark}

The key step in establishing Proposition~\ref{prop:K-stability-Rthree} is recorded as Lemma \ref{lem:Rthree-delta-C}.
\begin{lemma}\label{lem:Rthree-delta-C}
    Let $Y$ be the exceptional divisor of the blow-up of \(X = \bb P^1\times\bb P^2\) along $\Cthree$.
    \begin{enumerate}
        \item\label{part:Rthree-unstable-Y} There is an inequality $\delta_{\exx, c\Rthree}(Y) > 1$ (resp. \(\geq 1\)) if and only if $c < \cnot$ (resp. \(c\leq\cnot\)).
        \item\label{part:Rthree-flag} Let \(G_3\) be the group defined in Section~\ref{sec:TheGroupG3}. If \(c<\cnot\) (resp. \(c \leq \cnot\)), then \(\delta_{\exx, c\Rthree}(Z) > 1\) (resp. \(\geq 1\)) for all \(G_3\)-invariant subvarieties \(Z\subset Y\).
    \end{enumerate}
\end{lemma}

\begin{proof}
    Let $\pi \colon \tX\rightarrow \exx \coloneqq \mathbb{P}^1 \times \mathbb{P}^2$ be the blow-up of $\exx$ along  $\Cthree$. Write \(\dee = c \Rthree\) and \(L = -K_\exx-\dee\).

    First we show part~\eqref{part:Rthree-unstable-Y}.
    Let $S = (y_1=0)$ be the \((0,1)\)-surface containing \(\Cthree\). The exceptional divisor $Y$ of $\pi$ is $\mathbb{P}(N_{\Cthree/X}) \cong \mathbb{F}_1$.
\begin{detail}
Indeed, we compute $N_{\Cthree/X} = \mathcal{O}(-1) \oplus \mathcal{O}(-2)$ from the normal bundle sequence.
\end{detail}
    Let \(\tS \cong \bb P^1\times\bb P^1\) be the strict transform of \(S\) in \(\tX\), \(l_1^\circ\)  a member of the ruling \(\cal O_{\tS}(0,1)\), \(l_2^\circ\) a member of the other ruling \(\cal O_{\tS}(1,0)\), and \(f\) a fiber of \(Y \to \Cthree\). The classes of \(l_1^\circ,l_2^\circ,f\) generate the Mori cone of \(\tX\), and from this one can verify that $\pi^*L-uY$ is nef for $u \leq 2-2c$ and has pseudo-effective threshold $ u = 3-2c$.
\begin{detail}
    Indeed, we compute that \((\pi^*L-uY) \cdot l_1^\circ = 2-2c-u\), \((\pi^*L-uY) \cdot l_2^\circ = 3-2c-u\), and \((\pi^*L-uY) \cdot f = u\). Furthermore, \(-K_\tX \cdot l_1^\circ = -1 < 0\).
\end{detail}

We next compute the Nakayama--Zariski decomposition of \(\pi^*L-uY = P(u) + N(u)\).
At \(u=2-2c\) we contract \(l_1^\circ\) in a morphism \(\phi\colon \tX \to Z\) to a smooth Fano threefold of Picard rank 2.  This contracts one ruling of $\tS$.
In fact, one can compute that $(-K_Z)^3 = 48$, so by the classification of Fano threefolds \cite{MoriMukai}, $Z$ is a $(1,1)$-divisor in $\bP^2 \times \bP^2$.
\begin{detail}
Since $\phi$ separates $R$ and $S$, \(- \phi^* K_Z = -K_\tX + \tS\), \(\phi^* R_Z = \widetilde{R}\), and \(\phi^* Y_Z = Y + \tS\), and we have
\[ \pi^*L - uY = -K_\tX - \dee_\tX - (u + 2c -1)Y = \phi^*(-K_Z-\dee_Z-(u+2c-1)Y_Z) + (u+2c-2) \tS.\]
\end{detail}
Denoting these two regions by the subscripts \(\regionA, \regionB\), we find that
\[ \resizebox{1\textwidth}{!}{ $ \begin{array}{rll}
    0 \le u \le 2-2c: & P_{\regionA}(u)|_Y = \pi^*L-uY & N_{\regionA}(u)|_Y = 0 \\
    2-2c \le u \le 3-2c: & P_{\regionB}(u)|_Y = \phi^*(-K_Z-\dee_Z-(u+2c-1)Y_Z) & N_{\regionB}(u)|_Y = (u+2c-2) \tS 
\end{array} $ } \]
To find the volumes on these two regions, first note that, if \(\sigma \sim l_1^\circ + l_2^\circ\) denotes the negative section of \(Y\to \Cthree\), then we have
\(\pi^*L|_Y = (5-4c)f\), \(Y|_Y = -\sigma + f\), \(P_{\regionA}(u)|_\tS \sim (3-2c-u)l_1^\circ + (2-2c-u)l_2^\circ\), and \(\tS|_\tS \sim -l_2^\circ\). Then one computes, by hand of using Lemma~\ref{lem:volume-lemma-3}, that
\begin{equation*}\begin{split}
    P_{\regionA}(u)^3 &= -3 (2 c + u - 3) (4 c^2 - 2 c u - 10 c - u^2 + 2 u + 6) \\
    P_{\regionB}(u)^3 &= P_{\regionA}(u)^3 - 3 (u +2 c - 3) (u + 2 c - 2)^2
\end{split} \end{equation*}
so
\[ S_{\exx,\dee}(Y) = \frac{1}{\vol L} \left( \int_0^{2-2c} P_{\regionA}(u)^3 du + \int_{2-2c}^{3-2c} P_{\regionB}(u)^3 du \right)  = \frac{-14 c^3 + 62 c^2 - 91 c + 44}{3 (3 - 2 c)^2}.\]
\begin{detail}
In more detail, for \(\regionA\),
\begin{equation*}\begin{split}
    \pi^*L-uY &= L^3 - 3u (\pi^*L|_Y)^2 + 3u^2 (\pi^*L|_Y)(Y|_Y) - u^3 (Y|_Y)^2 \\
    &= 3(2-2c)(3-2c)^2 + 3u^2(-5+4c) - u^3(-3) \\
    &= -3 (2 c + u - 3) (4 c^2 - 2 c u - 10 c - u^2 + 2 u + 6).
\end{split} \end{equation*}
For \(\regionB\), we have
\begin{equation*}\begin{split}
    P_{\regionB}(u)^3 &= (P_{\regionA}(u) - (u+2c-2) \tS)^3 \\
    &= P_{\regionA}(u)^3 - 3 (u+2c-2)(P_{\regionA}(u)|_\tS)^2 + 3 (u+2c-2)^2 (P_{\regionA}(u)|_\tS)(\tS|_\tS) - (u+2c-2)^3 (\tS|_\tS)^2 \\
    &= P_{\regionA}(u)^3 - 3 (u+2c-2) 2 (3-2c-u) (2-2c-u) - 3 (u+2c-2)^2 (3-2c-u).
\end{split} \end{equation*}
\end{detail}
Since the log discrepancy is $A_{\exx,\dee}(Y) = 2-2c$, we find
\[\frac{A_{\exx,\dee}(Y)}{S_{\exx,\dee}(Y)} = \frac{6 (3 - 2 c)^2 (c - 1)}{14 c^3 - 62 c^2 + 91 c - 44},\]
and this quantity is \(\geq 1\) if and only if \(c \leq \cnot\). This shows part~\eqref{part:Rthree-unstable-Y}.

For part~\eqref{part:Rthree-flag}, we first find the \(G_3\)-invariant subvarieties of \(Y\). In coordinates, the blow-up $\tX\rightarrow \exx$ of  $\Cthree$ is defined by
\[ ( u(t_0 y_2 + t_1 y_0) - w t_0 y_1 = v(t_0 y_2 + t_1 y_0) - wt_1 y_1 = v t_0 - u t_1 = 0 ) \subset \exx \times \bb P^2_{[u:v:w]} , \]
and the exceptional divisor \(Y \cong \bb F_1\) is \((vt_0 - u t_1 = 0) \subset \bb P^1_{[t_0:t_1]} \times \bb P^2_{[u:v:w]}\). The automorphisms of \((\exx,\Rthree)\) defined in Section~\ref{sec:TheGroupG3} lift to \(Y\) as
\begin{equation*}\begin{split}
        \phi_\lambda\colon ([t_0:t_1],[u:v:w]) & \mapsto ([\lambda t_0: \lambda^{-1} t_1],[\lambda u: \lambda^{-1}v:w]) \\ \iota\colon ([t_0:t_1],[u:v:w]) & \mapsto ([t_1: t_0],[v:u:w])
\end{split}\end{equation*}
so by Lemma~\ref{lem:G_3-invariant-curves}, the \(G_3\)-invariant subvarieties of \(Y\) are \(Y\) itself and the curves
\[\widetilde{C}_1 \coloneqq (u=v=0), \qquad \widetilde{C}_{2,\pm} \coloneqq (t_0v \pm t_1u = w = 0), \qquad \widetilde{C}^{a}_{3,\pm} \coloneqq (t_0v \pm t_1 u = w^2 - a u v = 0)\] for \(a\in\bb C^*\). By part~\eqref{part:Rthree-unstable-Y} it remains to bound the \(\delta\)-invariant for these curves. Recalling that \(\sigma\) and \(f\) denote the class of the negative section and a ruling of \(Y\cong\bb F_1\), from part~\eqref{part:Rthree-unstable-Y} we have
\[ \begin{array}{rll}
    0 \le u \le 2-2c: & P_{\regionA}(u)|_Y = (5-4c-u)f + u \sigma & N_{\regionA}(u)|_Y = 0 \\
    2-2c \le u \le 3-2c: & P_{\regionB}(u)|_Y = (5-4c-u)f + (2-2c) \sigma & N_{\regionB}(u)|_Y = (u+2c-2) \sigma 
\end{array} \]
and we next compute Zariski decompositions of \(P(u)|_Y - vC\), for the \(G_3\)-invariant curves \(C \subset Y\), and apply the formulas of \cite[Theorem 4.8]{Fujita-3.11} recalled in Section~\ref{sec:stability-of-A_n}.
By \cite[Theorem 4.1]{Zhuang-equivariant} it suffices to find a lower bound on $\frac{A_{Y,\dee_Y} (C)}{S_{Y,\dee_Y; V_{\bullet \bullet}}(C)}$ as these are the only $G_3$-invariant curves on $Y$.

\noindent\underline{Case \(\widetilde{C}_1 = \sigma\):} The Zariski decomposition of \(P(u)|_Y - v \sigma\) is
\begin{equation*}\begin{split}
P(u,v) &= \begin{cases} (5-4c-u)f + (u-v) \sigma & 0 \le u \le 2-2c, \; 0 \le v \le u \\ (5-4c-u)f + (2-2c-v) \sigma & 2-2c \le u \le 3-2c, \; 0 \le v \le 2-2c \end{cases}\end{split}\end{equation*}
and \(N(u,v)=0\) for \(0 \le u \le 3-2c\). We have \(\ord_\sigma(N_{\regionA}(u)|_Y) = 0\) and \(\ord_\sigma(N_{\regionB}(u)|_Y) = u+2c-2\), so since \(A_{Y,\dee_Y} (\sigma) = 1\), we compute that
\[ S_{Y,\dee_Y; V_{\bullet \bullet}}(\sigma) = \frac{3 - 2 c}{3}, \qquad \frac{A_{Y,\dee_Y} (\sigma)}{S_{Y,\dee_Y; V_{\bullet \bullet}}(\sigma)} = \frac{3}{3 - 2 c} > 1 \text{ for }c\in(0,1). \]

\noindent\underline{Case \(\widetilde{C}_{2,\pm}\):} Let \(C_2 = \widetilde{C}_{2,+}\) or \(\widetilde{C}_{2,-}\); note that \(C_2\sim \sigma + f\) and \(A_{Y,\dee_Y} (C_2) = 1\). The Zariski decomposition of \(P(u)|_Y - v C_2\) is
\begin{equation*}\begin{split}
P(u,v) &= \begin{cases} (5-4c-u-v)f + (u-v) \sigma & 0 \le u \le 2-2c, \; 0 \le v \le u \\ (5-4c-u-v)f + (2-2c-v) \sigma & 2-2c \le u \le 3-2c, \; 0 \le v \le 2-2c \end{cases}\end{split}\end{equation*}
and \(N(u,v)=0\), so we compute that
\[ S_{Y,\dee_Y; V_{\bullet \bullet}}(C_2) = \frac{(1 - c) (6 c^2 - 20 c + 17)}{3 (3 - 2 c)^2}, \qquad \frac{A_{Y,\dee_Y} (C_2)}{S_{Y,\dee_Y; V_{\bullet \bullet}}(C_2)} = \frac{3 (3 - 2 c)^2}{(1 - c) (6 c^2 - 20 c + 17)} > 1\] for \(c\in(0,1)\).

\noindent\underline{Case \(\widetilde{C}^{a}_{3,\pm}\):} Let \(C_3 = \widetilde{C}^{a}_{3,+}\) or \(\widetilde{C}^{a}_{3,-}\) for some \(a\in\bb C^*\); note that \(C_3\sim 2(\sigma + f)\). The log discrepancy is \(A_{Y,\dee_Y} (C_3) = 1\) or \(1-c\). The Zariski decomposition of \(P(u)|_Y - v C_3\) is
\begin{equation*}\begin{split}
P(u,v) &= \begin{cases} (5-4c-u-2v)f + (u-2v) \sigma & 0 \le u \le 2-2c, \; 0 \le v \le u/2 \\ (5-4c-u-2v)f + (2-2c-2v) \sigma & 2-2c \le u \le 3-2c, \; 0 \le v \le 1-c \end{cases}\end{split}\end{equation*}
and \(N(u,v)=0\), so we compute that
\[ S_{Y,\dee_Y; V_{\bullet \bullet}}(C_3) = \frac{(1 - c)(6 c^2 - 20 c + 17)}{6 (3 - 2 c)^2}, \qquad \frac{A_{Y,\dee_Y} (C_3)}{S_{Y,\dee_Y; V_{\bullet \bullet}}(C_3)} \geq \frac{6 (3 - 2 c)^2}{(6 c^2 - 20 c + 17)} > 1\] for \(c\in(0,1)\).
\end{proof}

\begin{proof}[Proof of Proposition~\ref{prop:K-stability-Rthree}]
    For the forward direction, \((\exx,\dee)\) is K-unstable for \(c > \cnot\) by Lemma~\ref{lem:Rthree-delta-C}\eqref{part:Rthree-unstable-Y}. For the reverse implication, Section~\ref{sec:G3-polystability} and Lemma~\ref{lem:Rthree-delta-C}\eqref{part:Rthree-flag} show that \(\delta_{G_3}(\exx,c\Rthree)>1\) for \(c < \cnot\), so K-polystability follows from Theorem~\ref{lem:K-ps-G-invariant-delta}.
\end{proof}
\subsection{The K-moduli replacement for \texorpdfstring{\(c > \cnot\)}{c > c1}}\label{sec:construction-of-X3}
Let $\Xthree$ be the unique proper toric threefold whose fan $\Sigma$ is characterized by the following properties:
\begin{enumerate}
    \item The rays of $\Sigma$ have primitive vectors (see Figure~\ref{fig:fan-of-X3}) \[(1,0,0),\; (0,1,0),\; (0,0,1),\; (-1,-1,0),\; (0,-1,0), \text{ and } (-1,-2,-1).\]
    \item The fan $\Sigma$ has a unique cone which is not simplicial. This cone has rays \[(0,0,1),\; (-1,-2,-1),\; (-1,-1,0), \text{ and } (0,-1,0).\]
\end{enumerate}

\begin{figure}
    \begin{tikzpicture}
        \draw[thick] (0,0)--(3,0);
        \node at (3.6,1.5) {$(1,0,0)$};
        \draw[thick] (0,0)--(0,3);
        \node at (0,3.2) {$(0,1,0)$};
        \draw[thick] (0,0) -- (3,1.5);
        \node at (3.6,0) {$(0,0,1)$};
        \draw[thick] (0,0) -- (0,-3);
        \node at (0.7,-3.25) {$(0,-1,0)$};
        \draw[thick] (0,0)--(-1,-3);
        \node at (-1.4,-3.25) {$(-1,-1,0)$};
        \draw[thick] (0,0)--(-3,-3);
        \node at (-4,-3.25) {$(-1,-2,-1)$};
    \end{tikzpicture}
    \caption{Rays in the fan of $\Xthree$.}\label{fig:fan-of-X3}
\end{figure}
\begin{remark}\label{rem:X3-description-from-P(1112)}
    Subdividing $\Sigma$ by introducing a two-dimensional cone with rays $(-1,-2,-1)$ and $(0,0,1)$ yields a toric variety $W^+$. Observe $W^+$ is obtained by blowing up $\mathbb{P}(1,1,1,2)$ in the singular point $\conept$ and one other (without loss of generality torus-fixed) point $\ptE$. Write $\piP\colon W^+ \rightarrow \mathbb{P}(1,1,1,2)$ for the blow-up map.
\end{remark}

\begin{proposition}\label{prop:properties-of-Xthree}
    \(\Xthree\) is a Gorenstein Fano threefold that is not \(\bb Q\)-factorial. The singular locus of \(\Xthree\) consists of one ordinary double point. The Picard group of \(\Xthree\) is \(\bb Z^2\), and the Weil divisor class group is \(\bb Z^3\).
\end{proposition}

\begin{proof}
    First, a computation (e.g. using \texttt{Magma} \cite{Magma} and adapting the code from \cite{ACG-code}) shows that the anticanonical polytope associated to the fan of $\Xthree$ is reflexive, and thus $\Xthree$ is a Gorenstein Fano variety. 
\begin{detail}
    Indeed, one can use the following \texttt{Magma} code, which is adapted from \cite{ACG-code}:

    \begin{verbatim}
    rays := [ [1,0,0],[0,1,0], [0,0,1], [-1,-1,0], [0,-1,0], [-1,-2,-1] ];
        
    cones := [ [1,2,3], [1,3,5], [2,3,4], [1,2,6], [2,6,4], [3,4,5,6],[1,5,6]];

    F := Fan(rays,cones);

    n:=\#Rays(F);

    P:=HalfspaceToPolyhedron(Rays(F)[1],-1);

    for i in [2..n]

    do P:=P meet HalfspaceToPolyhedron(Rays(F)[i],-1);

    end for;

    IsReflexive(P);
\end{verbatim}

\end{detail}
    For the singularities of \(\Xthree\), observe that the fan of $\Xthree$ has one singular cone, which is an ordinary double point singularity. Since the fan is not simplicial, \(\Xthree\) is not \(\bb Q\)-factorial. Next, \(\Xthree\) has the same class group as the small resolution \(W^+\). From the description as a blow-up \(\piP\colon W^+ \rightarrow \mathbb{P}(1,1,1,2)\), we find that this group is \(\bb Z^3\).

    Finally, we find the Picard group of \(\Xthree\), which is torsion free because \(\Xthree\) is a toric variety. Denote the rank of this group by $d$. The Picard group of $\Xthree$ is isomorphic to the group of piecewise linear functions on the fan of $\Xthree$ modulo the group of linear functions, see \cite{FultonToricVar}. (In some more recent work, these piecewise linear functions are called strict piecewise linear or conewise linear functions, and the phrase piecewise linear function has a different meaning.) A piecewise linear function on a fan is specified by its values on the generator of each ray. The fan of $\Xthree$ has six rays, giving a $\mathbb{Z}^6$ of choices for these values. All but one cone is simplicial, but this one cone has four rays, imposing a linear constraint on the $\mathbb{Z}^6$ of choices. Since two piecewise linear functions specify the same line bundle if and only if they differ by a linear function, we may adjust any piecewise linear function so that it adopts the value zero on the generators of the rays in the positive $x$,$y$ and $z$ directions, and there is a unique linear function which performs this adjustment. Thus, the first three entries of our original $\mathbb{Z}^6$ may be taken as zero, so we have one further linear constraint and are left with a $\mathbb{Z}^2$ of choice. Therefore $d=2$.
\end{proof}

\subsubsection{Degeneration from $\mathbb{P}^1 \times \mathbb{P}^2$}
\begin{proposition}\label{prop:degentoX3}
    There is a (necessarily not toric) degeneration $\overline{\calX}\rightarrow \mathbb{A}^1$ with generic fiber $\mathbb{P}^1 \times \mathbb{P}^2$ and fiber over zero given by $\Xthree$.
\end{proposition}

\begin{proof}
    Our degeneration will be constructed by applying a number of birational modifications to the family $\mathbb{P}^1 \times \mathbb{P}^2\times \mathbb{A}^1$.
    Consider first the blow-up $\tX \rightarrow \cal X\coloneqq \mathbb{P}^1 \times \mathbb{P}^2\times \mathbb{A}^1$ of the curve $\Cthree\times \{0\} \subset \Rthree\times \{0\}$ in the central fiber. The central fiber of the resulting space has two irreducible components $\tX_0 = \tX\cup \widetilde{W}$, where the exceptional divisor $\widetilde{W} = \bb P_{\Cthree}(\cal O\oplus\cal O(-1)\oplus\cal O(-2))$ is a $\mathbb{P}^2$-bundle over $\Cthree = \mathbb{P}^1$.

    Writing $S_{\tX}$ for the strict transform of $S$ in $\tX$, we contract the ruling $l_1^\circ$ of $S_{\tX}$ as in the proof of Lemma \ref{lem:Rthree-delta-C}, defining a flip \(\widetilde{\calX} \dashrightarrow \calX'\). The new central fiber is \(\calX'_0 = Z \cup W\), where \(Z\) is the threefold constructed in the proof of Lemma~\ref{lem:Rthree-delta-C}, and \(W\) is the blow-up of \(\widetilde{W}\) along the intersection \(S_\tX \cap \widetilde{W} \cong\Cthree\). This intersection is a torus invariant subvariety, and thus we identify \(W\) with the toric variety ${\bb F}_1$-bundle over $\Cthree = \mathbb{P}^1$ associated to the line bundles $\mathcal{O}(-1)\oplus \mathcal{O}(-2)$; the flip glues \(S_{\tX}\) in as the negative section of each \(\bb F_1\).

    Next, we may contract the strict transform of the ruling $l_2^\circ$ in $Z$. (Using the description of $Z$ as a $(1,1)$-divisor in $\bP^2 \times \bP^2$, this is just the projection onto one of the factors.) This contracts the divisor $Z$ in $\calX'$, giving a new fourfold $\calX' \to \overline{\calX}$.  On $W$, this contracts the negative section \(l_2\) of $Z \cap W {\cong\bb F_1}$.
    Finally, we check that $\Xthree$ coincides with the image of $W$ in \(\overline{\calX}\). Indeed, the fan of the $\bb F_1$-bundle over $\mathbb{P}^1$ is exactly the fan of $\Xthree$, except that the singular cone is subdivided into two three-dimensional cones. The small contraction $W \rightarrow \Xthree$ is precisely this blow-down map.
\end{proof}
    The diagrams below show these birational transformations on the families and central fibers.
    
    \[\begin{tikzcd}
    \widetilde{\cal X} \arrow[rr, dashed, "\text{contract }l_1^\circ"] \ar[d, "\text{blow up }\Cthree \times \{0\}" {swap}]& & \calX' \ar[d, "\text{contract }l_2"] \\
    \cal X \arrow[r] & \bb A^1 & \overline{\calX} \arrow[l] \end{tikzcd}
    \;
    \begin{tikzcd}
    \widetilde{\cal X}_0 = \tX \cup \widetilde{W} \arrow[rr, dashed, "\text{contract }l_1^\circ"] \ar[d, "\text{blow up }\Cthree \times \{0\}" {swap}]& & \calX'_0 = Z \cup W \ar[d, "\text{contract }l_2"] \\
    \cal X_0 = \bb P^1 \times \bb P^2 \arrow[r] & \{0\} & \overline{\calX}_0 = \Xthree \arrow[l] \end{tikzcd}\]

\subsubsection{Extremal contractions of \(\Xthree\)} Write $\mathfrak{t}\mathbb{P}^2$ for the fan of $\mathbb{P}^2$ and $\mathfrak{t}\Xthree$ for the fan of $X$. There is a map of fans $ \mathfrak{t}\Xthree \rightarrow \mathfrak{t}\mathbb{P}^2$ defined on underlying lattices as the map projecting $(a,b,c) \mapsto (a,c)$. Write $\piXthree$ for the associated morphism of toric varieties $\piXthree:\Xthree\rightarrow \mathbb{P}^2$.

\begin{lemma}
    Fibers of the morphism $\piXthree\colon \Xthree\to\bb P^2$ are $\mathbb{P}^1$, except over the torus-fixed point $\ptE$ of $\mathbb{P}^2$ corresponding to the cone $\sigma$ spanned by rays $(0,1)$ and $(-1,-1)$.
\end{lemma}
\begin{proof}
    Away from the preimage of $\sigma$, the map of fans is flat, and thus so is the associated morphism of toric varieties. Since the preimage of the interior of every cone in $\mathfrak{t}\mathbb{P}^2$ consists of a single cone and its star, fibers are irreducible toric varieties. The morphism $\piXthree$ is of relative dimension one, so away from $\ptE$ fibers are all the unique one-dimensional toric variety $\mathbb{P}^1$. The preimage of $\ptE$ is a toric variety with fan given by the star fan of the ray $(-1,-1,0)$. This star fan is the fan of $\mathbb{P}^2$.
\end{proof}

Next, we describe the \defi{contractions} of \(\Xthree\), i.e. the surjective projective morphisms with connected fibers from \(\Xthree\) to a normal variety.

\begin{theorem}\label{thm:extremalcontractionsofX3}
    The nodal Fano threefold $\Xthree$ admits exactly two contraction morphisms. The first is the morphism $\piXthree \colon \Xthree \to \bb P^2$, and the second is a birational contraction of the divisor $S \subset \Xthree$ to a non $\bb Q$-Gorenstein threefold $V_{\mathfrak{n}}$. 
\end{theorem}

\begin{proof}
    By Proposition~\ref{prop:properties-of-Xthree}, the Picard rank of $\Xthree$ is 2 and hence $\overline{NE(\Xthree)} \subset N_1(\Xthree)$ is 2-dimensional with extremal rays $R_1$ and $R_2$.  As contraction morphisms correspond to contractions of extremal faces in $\overline{NE(\Xthree)}$, we thus have at most two non-trivial contractions. 

    By the previous construction, $\piXthree \colon \Xthree \to \bb P^2$ is a contraction, with one fiber $E \subset \Xthree$ and all other fibers $\bP^1$.  We now verify the existence of the other contraction.  Let $R$ be a surface on $\Xthree$ of class $2D_1$.  From the Mori cone computation in Lemma~\ref{lem:Mori-cone-of-W+}, $R$ is big and nef and trivial only on curves contained in the surface $S \cong \bP^2$.  Because $mR$ is nef for all $m >0$, $|mR|$ is base-point-free and defines a morphism $\Xthree \to \bP(H^0(\Xthree, mR))$ which, for $m \gg 0$, is an isomorphism onto its image $V_{\mathfrak{n}}$ away from $S$ and contracts $S$ to a point.  This is the other extremal contraction.  Note that the image of this contraction could alternatively be constructed by blowing up $\ptE \subset \bP(1,1,1,2)$ and contracting $l_1$, the strict transform of the ruling containing $\ptE$.  The resulting threefold is not $\bb Q$-Gorenstein because the contracted curve $l_1$ is not $K$-trivial.
\end{proof}

\begin{corollary}
    Suppose $f: \Xthree \to Z$ is any projective morphism with connected fibers.  Then, the image of $\Xthree$ is isomorphic to either $\Xthree$, $V_{\mathfrak{n}}$, $\bP^2$, or a point. 
\end{corollary}

\subsection{Small resolutions of \texorpdfstring{$\Xthree$}{Xn}}\label{sec:small-resolutions-X_3}
We construct the following diagram of birational models of \(\Xthree\), which will be useful for us when computing lower bounds for \(\delta\)-invariant in Section~\ref{sec:K-stability-Xthree}.
\[\begin{tikzcd}
    & & \widehat{W} \arrow[ld, "g" {swap}] \arrow[rd, "g^+"] \\
    \widetilde{W} \arrow[r, dashed] & W \ar[rd, "\hsmall = \text{ contract }l_2 \text{ (small)}" {swap}] & & W^+ \cong \Bl_{\ptE} \bb P_{\bb P^2}(\cal O\oplus\cal O(2)) \arrow[ld, "\hsmallplus = \text{ contract }l_1 \text{ (small)}"] \arrow[rr, "\text{blow up }\ptE"] \arrow[rrd, "\piP" {swap}] & & \bb P_{\bb P^2}(\cal O\oplus\cal O(2)) \arrow[d, "\text{contract }S"] \\
    & & \Xthree & & & \bb P(1,1,1,2) \end{tikzcd}\]

The morphism $\piP$ was described in Remark \ref{rem:X3-description-from-P(1112)}. The left hand square consists of toric varieties whose fans differs from the fan of $\Xthree$ by changing only the singular cone in the following ways, see Figure~\ref{fig:subivisions}. 

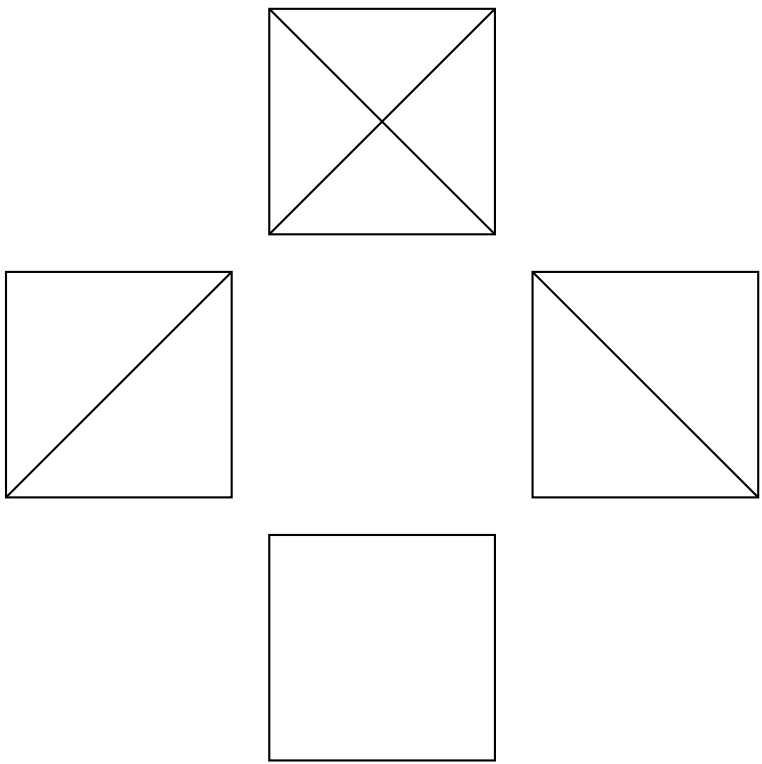
\begin{figure}
    \centering
\begin{tikzpicture}
    \coordinate (a) at (0,0);
    \coordinate (b) at (3,0);
    \coordinate (c) at (3,3);
    \coordinate (d) at (0,3);
    \draw[thick] (a) -- (b) -- (c) -- (d)-- cycle;
    \begin{scope}[every coordinate/.style={shift={(3.5,3.5)}}]
        \draw[thick] ([f]a) -- ([f]b) -- ([f]c) -- ([f]d)-- cycle;
        \draw[thick] ([f]b)-- ([f]d);
    \end{scope}
    \begin{scope}[every coordinate/.style={shift={(-3.5,3.5)}}]
        \draw[thick] ([f]a) -- ([f]b) -- ([f]c) -- ([f]d)-- cycle;
        \draw[thick] ([f]a)-- ([f]c);
    \end{scope}
    \begin{scope}[every coordinate/.style={shift={(0,7)}}]
        \draw[thick] ([f]a) -- ([f]b) -- ([f]c) -- ([f]d)-- cycle;
        \draw[thick] ([f]a)-- ([f]c);
        \draw[thick] ([f]b)-- ([f]d);
    \end{scope}
\end{tikzpicture}
\caption{The cone over the bottom square is the singular cone in the fan of $\Xthree$. In these diagrams, corners of each square correspond to rays of the singular cone in fan of $\Xthree$, and lines indicate where two rays are faces in a 2-dimensional cone. The middle left corresponds to the fan of $W$; the middle right $W^+$, and the top $\widehat{W}$. }\label{fig:subivisions}
\end{figure}

\begin{enumerate}
    \item To form $W$, subdivide the singular cone in the fan of $\Xthree$ by adding a two-dimensional cone with rays $(-1,-2,-1)$ and $(0,0,1)$. This corresponds to a small birational modification, whose exceptional locus is a torus-fixed curve that we denote $l_2$
    \item The fan of $W^+$ is obtained by instead inserting a two-dimensional cone with rays $(-1,-1,0)$ and $(0,-1,0)$. The exceptional curve is written $l_1$.
    \item Finally $\widehat{W}$ is the common refinement of the fans of $W$ and $W^+$, and thus corresponds to the blow-up of the singular point in $\Xthree$. The exceptional $F$ is a copy of $\mathbb{P}^1 \times \mathbb{P}^1$.
\end{enumerate}  

To understand the right hand triangle, observe the fan of $W^+$ is the barycentric subdivision of the fan of $\mathbb{P}_{\mathbb{P}^2}(\mathcal{O}\oplus \mathcal{O}(2))$ which introduces the ray $(-1,-1,0)$. 
Thus there is a blow-up map $$W^+ \rightarrow \mathbb{P}_{\mathbb{P}^2}(\mathcal{O}\oplus \mathcal{O}(2) )$$ whose exceptional $E^+$ is a copy of $\bb P^2$. The blow-up center is a torus-fixed point written $p_E$.

To complete the diagram, write $S^+$ for the copy of $\bb P^2$ in $\mathbb{P}_{\mathbb{P}^2}(\mathcal{O}\oplus \mathcal{O}(2))$ associated to the ray $(-1,-1,0)$. There is a unique toric modification to $\mathbb{P}(1,1,1,2)$ which contracts $S^+$.

\subsubsection{Mori cones of small resolutions of $\Xthree$}\label{sec:mori-cone-small-resol-Xthree}
We work out the extremal rays in $\overline{NE(W)}$ and $\overline{NE(W^+)}$, which will be useful in later computations. First, we fix some notation.

Let $E\cong \bP^2$ and $S \cong \bP^2$ denote the respective surfaces on $\Xthree$.  Let $E', S'$ (respectively $E^+, S^+$) denote their strict transforms on $W$ (respectively $W^+$).  In $W$, $E' \cong S' \cong \bF_1$, and these two surfaces meet along their negative sections, which is the exceptional locus $l_2$ of the small resolution $W \to \Xthree$. In $W^+$, $E^+ \cong S^+ \cong \bP^2$ and are disjoint and each meets the exceptional locus $l_1$ of the small resolution $W^+ \to \Xthree$ in a single point. 

We include a picture of the threefolds and the relevant curves in Figure~\ref{fig:Xthree-small-resolutions}.

\begin{figure}
    \centering
    \begin{tikzcd}
        \begin{tabular}{c}
\begin{tikzpicture}

\draw[thick] (1.5,0) -- (2,0) -- (2,1.5) -- (0,1.5) -- (0,1) --(0,0) -- (1.5,0); 
\draw[thick] (0,0) -- (-.7,-.7) -- (1.3,-.7) -- (2,0);
\draw[blue, thick] (0,0) -- (2,0); 
\draw[teal, thick] (0.5,-.7) -- (1.2,0); 
\draw[olive, thick] (1.2,0) -- (1.2,1.5);

\node[right, node font=\tiny] at (1.3,-0.7) {$S$};
\node[below right, node font=\tiny] at (2,1.5) {$E$};
\node[above, node font=\tiny] at (0.5,0) {$l_2$};
\node[below right, node font=\tiny] at (1.2,1.5) {$l_E$};
\node[above right, node font=\tiny] at (0.8,-0.7) {$l_S$};

\node[below, node font=\small] at (.6,-.8) {$W$};

\end{tikzpicture}
\end{tabular} \arrow[rd,swap, "h", start anchor={[xshift=-1ex, yshift=1ex]}] & & \begin{tabular}{c}
\begin{tikzpicture}

\draw[thick] (0,0) -- (-.7,-.7) -- (.5,-.7) -- (0,0);
\draw[thick] (0, 1.5) -- (-0.7,2) -- (0.5,2) -- (0,1.5);

\draw[blue, thick] (0,1.5) -- (0,0); 
\draw[teal, thick] (-0.2,-.7) -- (0,0); 
\draw[olive, thick] (-.2,2) -- (0,1.5);

\node[right, node font=\tiny] at (0.5,-0.7) {$S^+$};
\node[left, node font=\tiny] at (0,.7) {$l_1$};
\node[right, node font=\tiny] at (.5,2) {$E^+$};
\node[right, node font=\tiny] at (-.2,1.8) {$l_E$};
\node[right, node font=\tiny] at (-.2,-0.5) {$l_S$};

\node[below, node font=\small] at (-.1,-.8) {$W^+$};

\end{tikzpicture}
\end{tabular} \arrow[ld, "\hsmallplus"] \\
        & \begin{tabular}{c}
\begin{tikzpicture}

\draw[thick] (0,0) -- (-.7,-.7) -- (.5,-.7) -- (0,0);
\draw[thick] (0, 0) -- (-0.7,.7) -- (0.5,.7) -- (0,0);

\draw[teal, thick] (-0.2,-.7) -- (0,0); 
\draw[olive, thick] (-.2,.7) -- (0,0);
\draw[blue, fill=blue] (0,0) circle (0.05);

\node[right, node font=\tiny] at (0.5,-0.7) {$S$};
\node[right, node font=\tiny] at (.5,.7) {$E$};
\node[right, node font=\tiny] at (-.2,.5) {$l_E$};
\node[right, node font=\tiny] at (-.2,-0.5) {$l_S$};

\node[below, node font=\small] at (-.1,-.8) {$\Xthree$};

\end{tikzpicture}
\end{tabular} & \\
    \end{tikzcd}
    \vspace{-.4in}\caption{A visualization of $\Xthree$ and it small resolutions.}
    \label{fig:Xthree-small-resolutions}
\end{figure}
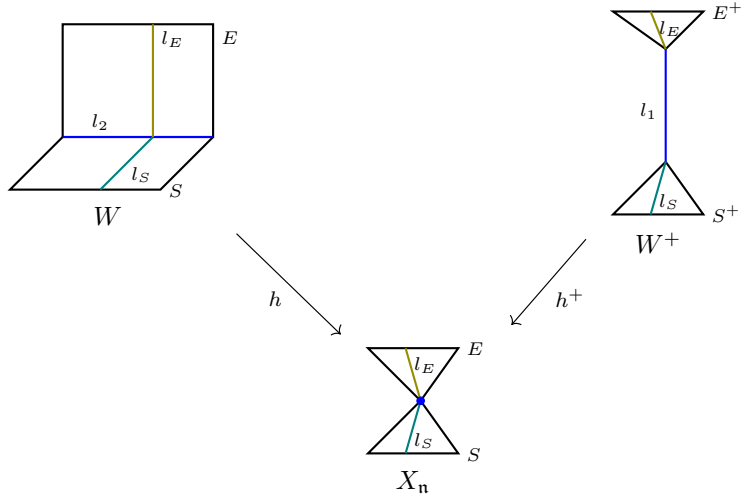

\begin{lemma}\label{lem:Mori-cone-of-W}
    The Mori cone of $W$ is generated by the exceptional curve $l_2$, a fiber $l_S$ of the ruled surface $S'$, and a fiber $l_E$ of the ruled surface $E'$.
\end{lemma}

\begin{proof}
    The threefold $W$ is a smooth $\bF_1$-bundle over $\bP^1$ of Picard rank 3.  Each of the stated curves corresponds to a primitive contraction: the contraction of $l_2$ produces $\Xthree$, the contraction of $l_S$ contracts the negative section of each fiber of the $\bF_1$-bundle $W$ (contracting $S'$ to its intersection with $E'$), and the contraction of $l_E$ contracts each fiber of the $\bF_1$-bundle to its negative section (contracting $W$ onto the surface $S'$).  Hence, the stated curves are extremal in $\overline{NE(W)}$ and linearly independent in $N_1(W) \cong \mathbb{R}^3$.  It suffices to show that every irreducible effective curve $C \subset W$ can be written as a non-negative combination of $l_2, l_S,$ and $l_E$.  This is clear for curves contained in fibers $F$ of the $\bF_1$-bundle structure: the negative section of $F$ is equivalent to $l_S$, and the fiber is equivalent to $l_E$.  This is also clear for curves contained in $S'$ and $E'$ because they are each $\cong \bF_1$ with negative section $l_2$ and fiber given by $l_S$ and $l_E$, respectively.  So, suppose $C$ is an irreducible curve not contained in fibers $F$, or $E'$ or $S'$, and write $C = a l_2 + b l_S + c l_E$.  We have the following intersection numbers: 
    
    \begin{gather*} F \cdot l_2 = 1, \quad F \cdot l_S = 0, \quad F \cdot l_E = 0, \quad S' \cdot l_2 = -1, \quad S' \cdot l_S = -1, \quad S' \cdot l_E = 1, \\ E' \cdot l_2 = -1, \quad E' \cdot l_S = 1, \quad E' \cdot l_E = 0. \end{gather*}
    Because $a = F \cdot C \ge 0$, $-a -b + c = S' \cdot C \ge 0$, and $-a + b = E' \cdot C \ge 0$, we obtain $a, b, c \ge 0$ and thus $C$ is contained in the non-negative span of $l_2, l_S$, and $l_S$.  This implies that these curves generate the Mori cone. 
\end{proof}

\begin{lemma}\label{lem:Mori-cone-of-W+}
    The Mori cone of $W^+$ is generated by the exceptional curve $l_1$, a line $l_S$ in $S^+ \cong \bP^2$, and a line $l_E$ in $E^+ \cong \bP^2$.
\end{lemma}

\begin{proof}
    The threefold $W^+$ is the blow-up of $\bP(1,1,1,2)$ at two points---the singular point $\conept$ and a smooth point $\ptE$---with exceptional divisors $S^+ \cong \bP^2$ and $E^+ \cong \bP^2$, respectively.  It has Picard rank $3$.  Each of the stated curves corresponds to a primitive contraction: the contraction of $l_1$ produces $\Xthree$; the contraction of $l_S$ contracts $S^+$; and the contraction of $l_E$ contracts $E^+$.  Hence, the stated curves are extremal in $\overline{NE(W^+)}$ and linearly independent in $N_1(W) \cong \mathbb{R}^3$.  It suffices to show that every irreducible effective curve $C \subset W^+$ can be written as a non-negative combination of $l_1, l_S,$ and $l_E$.  This is clear for curves contained in $S^+$ by $E^+$.  Let $\Gamma \subset W$ denote the strict transform of a section of $|\mathcal{O}_{\bP(1,1,1,2)}(1)|$.  If $\Gamma$ does not intersect $E^+$, then $\Gamma \cong \bF_2$ with negative section $l_S$ and fiber numerically equivalent to $l_1 + l_E$, so it is clear that all curves in $\Gamma$ are non-negative combinations of $l_1, l_S, l_E$.  Similarly, if $\Gamma$ intersects $E^+$, then $\Gamma$ is the blow-up of $\bF_2$ at a point away from the negative section, still with negative section $l_S$ are reducible fiber the union of $l_1$ and $l_E$, and it is straightforward to show these curves generate all curves in $\Gamma$.  Finally, suppose $C$ is an irreducible curve not contained in any $\Gamma$, or $E^+$ or $S^+$, and write $C = a l_1 + b l_S + c l_E$.  Choose $\Gamma$ intersecting $E$.  We have the following intersection numbers: 
    \begin{gather*} \Gamma \cdot l_1 = -1, \quad \Gamma \cdot l_S = 1, \quad \Gamma \cdot l_E = 1, \quad S^+ \cdot l_1 = 1, \quad S^+ \cdot l_S = -2, \quad S^+ \cdot l_E = 0, \\ E^+ \cdot l_1 = 1, \quad E^+ \cdot l_S = 0, \quad E^+ \cdot l_E = -1. \end{gather*}
    Because $-a+b+c = \Gamma \cdot C \ge 0$, $a -2b = S^+ \cdot C \ge 0$, and $a - c = E^+ \cdot C \ge 0$, we obtain $a, b, c \ge 0$ and thus $C$ is contained in the non-negative span of $l_1, l_S$, and $l_S$.  This implies that these curves generate the Mori cone. 
\end{proof}

\subsection{Deformation theory of \texorpdfstring{\(\Xthree\)}{Xn}}

Recall that $\Xthree$ is a Gorenstein Fano threefold with an isolated singularity that is an ordinary double point (Proposition~\ref{prop:properties-of-Xthree}).  The deformation theory of such varieties is well understood; Namikawa proves the following in \cite[Proposition 3]{Namikawa}.

\begin{theorem}\label{thm:namikawa}
    Suppose $X$ is a Gorenstein Fano threefold with ordinary double points.  Then, $X$ has unobstructed deformations. 
\end{theorem}

In the case of one ordinary double point, the versal deformation space of the singularity is one dimensional (see, e.g. \cite{Friedman86}) and there are no local-to-global obstructions by \cite[Proposition 4.1]{Friedman86}, so we conclude the following.  To obtain the last statement, we use the existence of the smoothing of $\Xthree$ to $\bP^1 \times \bP^2$ constructed in Proposition~\ref{prop:degentoX3} and the ridigity of $\bP^1 \times \bP^2$. 

\begin{theorem}\label{thm:unobstructeddeformations}
    The threefold $\Xthree$ has unobstructed deformations, and the base of the miniversal deformation space is a smooth curve.  In particular, any small deformation of $\Xthree$ is either $\Xthree$ or $\bP^1 \times \bP^2$.
\end{theorem}

Since any unirational curve is rational, Theorem~\ref{thm:unobstructeddeformations} shows that the miniversal deformation of \(\Xthree\) is \(\overline{\cal X}\to\bb A^1\) constructed in Proposition~\ref{prop:degentoX3}. This also follows from direct computation following \cite[Corollary 7.4]{Altmann}.

\subsubsection{Divisors on $\Xthree$ arising as limits of surfaces on $\bP^1 \times \bP^2$}

Next, we find additional restrictions on Cartier divisors in \(\Xthree\) that arise from a degeneration of a $(2,2)$-surface in $\bP^1 \times \mathbb{P}^2$. Volume of a divisor is preserved under degeneration, and thus
$$(-K_{\Xthree} - cR)^3=3(2-2c)(3-2c)^2.$$

\begin{lemma}\label{lem:strict-transform-of-X_3-surfaces-in-P}
     The pullback of $R$ to $W^+$ is in the class group the same as $2 D_1$, where $D_1$ is the toric Cartier divisor associated to the ray $(0,1,0)$. In particular,
     \begin{enumerate}
         \item\label{item:R-in-X_3-degree-4} If $R$ is the strict transform of a surface $R_{\bb P}$ in $\bP(1,1,1,2)$, then $R_\mathbb{P}$ is a surface of degree $4$.
         \item\label{item:R-in-X_3-multiplicity-2-pE} The surface $R_{\bb P}$ has multiplicity $2$ at $\ptE$.
         \item\label{item:R-in-X_3-contains-S-if-meets} If the surface $R$ meets $S$, then $R$ contains $S$. 
     \end{enumerate}
\end{lemma}
\begin{proof}
    Write $D_1$ for the toric boundary divisor associated to ray $(0,1,0)$; $D_2$ for the ray $(1,0,0)$ and $D_3$ for the ray $(0,0,1)$ in the fan of $W^+$. Together the $D_i$ generate the class group of $\Xthree$. We compute intersection products through the toric dictionary:
    \[D_1^3 = 3 \quad D_2^3 = 0,\quad D_1^2 D_2 = 2, \quad D_2^2D_1 = 1, \quad D_3^2 D_i = 0, \quad D_3D_2^2 = 0, \quad D_3 D_1^2 = 1, \quad D_1 D_2 D_3 =1.\]
    Now suppose that in the class group $R = -(uD_1 + vD_2 + w D_3)$ for some integers $u,v,w$. By comparing coefficients of powers of $c$, the condition that $(-K_{\Xthree} - cR)^3=3(2-2c)(3-2c)^2$ gives three equations that $u,v,w$ must satisfy:
\begin{detail}
        for example, by using the following \texttt{Magma} code:
        \begin{verbatim}
        B := RationalField();
    
        Q<u,v,w,D1,D2,D3,c,F>:=AffineAlgebra<B, u,v,w,D1,D2,D3,c,F | D1^3 - 3*F, D2^3, D3^3,D1*D2^2-F,D2^2*D1-F,D1^2*D2-2*F, D1*D3^2,D1*D2*D3-F,D2^2*D3,D1^2*D3-F>;
        
        ((2*D1+D2) + c*(u*D1+v*D2+w*D3))^3;
        \end{verbatim}
\end{detail}
    \begin{gather*}
        63 u + 36 v + 24w = -126, \quad 3(u^3 + 2vu^2 +  u v^2 + wu^2 +2uvw)= -24, \\ 3(8u^2 + 2v^2 + 4vw+10uv +6uw )= 96.
    \end{gather*}
    These equations have a unique integer solution $u=-2$ and $v = w=0$. Property~\eqref{item:R-in-X_3-degree-4} is immediate. For property~\eqref{item:R-in-X_3-multiplicity-2-pE}, observe that on \(W^+\), there is an equality $D_1^2 E^+ =1$ where $E^+$ is the exceptional over $p_E$.
    Finally, for property~\eqref{item:R-in-X_3-contains-S-if-meets}, suppose that $R$ is a divisor on \(W^+\) of class $\RClass$ which meets $S^+$. The image $R_\mathbb{P}$ of $R$ in $\mathbb{P}(1,1,1,2)$ is rationally equivalent to $\mathcal{O}(4)$ and meets $p_S$. Observe now that $$\RClass = [\piP^* \mathcal{O}_{\bb P(1,1,1,2)}(4)] - [E^+] = 2[D_1].$$ However, since $R_\mathbb{P}$ passes through $p_S$, the class of the strict transform of $R_\mathbb{P}$ is $[\pi^* \mathcal{O}(4)] - [E^+]-k[S^+]$ for some positive integer $k$. Since $[S^+]$ is not rationally equivalent to zero, $R$ is not equal to its strict transform, and thus contains $S^+$ as an irreducible component. 
\end{proof}

\begin{lemma}\label{lem:D_1-on-X_3-big-nef}
    The Cartier divisor \(D_1\) on \(\Xthree\) is big and nef but not ample. For any $c \in (0,1) \cap \bb Q$ and $R$ of class $2D_1$, the divisor $-K_{\Xthree} - cR$ is ample. 
\end{lemma}

\begin{proof}
    As $D_1$ is Cartier, it is big and nef if and only if its pullback $D_1^+$ to $W^+$ is big and nef.  By direct computation, we have $D_1^+ \cdot l_1 = 0$, $D_1^+ \cdot l_E = 1$ and $D_1^+ \cdot S = 0$, where $l_1, l_E, l_S$ are as in Lemma \ref{lem:Mori-cone-of-W+}.  These curves generate the Mori cone of $W^+$ by Lemma \ref{lem:Mori-cone-of-W+} and $D_1^+$ is nef.  The morphism defined by the linear system $D_1^+$ is birational, contracting the divisor $S^+$ and the curve $l_1$, so $D_1^+$ is big but not ample. Similarly, the pullback of $K_{\Xthree}$ to $W^+$ is $K_{W^+}$, and by direct computation, $K_{W^+} \cdot l_1 = 0$, $K_{W^+} \cdot l_E = -2$ and $K_{W^+} \cdot S = -1$.  Therefore, for $R^+$ the pullback of $R$, $-K_{W^+} - cR^+$ is positive on all curves not contracted by by the map to $W^+$ and hence $-K_{\Xthree} - cR$ is ample.
\end{proof}

\begin{lemma}\label{lem:X3-only-after-wall}
    Let \(\Xthree\) be the threefold constructed above. Let \(R \subset \Xthree\) be the strict transform of a divisor in \(|\cal O_{\bb P(1,1,1,2)}(4)|\) that has multiplicity 2 at \(\ptE\), and let \(c\in(0,1)\). Then \(\delta_{\Xthree, c R}(E') > 1\) (resp. \(\geq 1\)) if and only if \(c > \cnot\) (resp. \(\geq \cnot\)).
\end{lemma}

In particular, Lemma~\ref{lem:X3-only-after-wall} shows that the threefold \(\Xthree\) cannot appear in \(\Kmodulistack{c}\) for \(c < \cnot\).

\begin{proof}
Let $h: W \to \Xthree$ be the $\bb Q$-factorialization constructed above.  Write \(D\coloneqq cR\) and \( L \coloneqq h^*(-K_{\Xthree} - \dee)\). We compute \(S_{\Xthree, \dee}(E')\).  Because $W$ is an $\bF_1$ bundle over $\bP^1$, one may compute that the pseudoeffective threshold of \(L-uE'\) is \(u=3-2c\).  As $W$ has Mori cone generated by the curves $l_2, l_E, l_S$ by Lemma \ref{lem:Mori-cone-of-W}, we find that $L-uE'$ is ample for $0 \le u \le 1$, and the ample model contracts $S' \cong \bF_1$ to its negative section $l_2$ at $u = 1$.  At $u  = 3-2c$, the divisor is no longer big. 
\begin{detail}
Indeed, we compute that, so we flip \(l_1\) at \(u=0\). On \(\widehat{W}\), if \(l_S \subset \hS\cong\bb F_1\) and \(l_E\subset \widehat{E}\cong\bb F_1\) denote the rulings, then
\begin{equation*}\begin{split}
(L-uE') \cdot l_2 &= u \\
(L-uE') \cdot l_S = 1-u \\
(L-uE') \cdot l_E &= 2-2c
\end{split}\end{equation*}
so we contract \(l_S\) at \(u=1\).
\end{detail}

Using these birational transformations, we compute that the positive parts of the Nakayama--Zariski decomposition of \(L-uE' = P(u) + N(u)\) on \(W\) are
\[ \begin{array}{rl}
    0 \le u \le 1: & P_{\regionA}(u) = L-uE' \\
    1 \le u \le 3-2c: & P_{\regionB}(u) = L-uE' - (u-1)S'
\end{array} \]
and, on the respective regions, have volumes
\[ \begin{array}{rl}
    0 \le u \le 1: & \vol(P_{\regionA}(u)) = \frac{(5-4c)^3}{2} - \frac{4}{8} - (2-2c+u)^3 + 3u^3 - 2 u^3 \\
    1 \le u \le 3-2c: & \vol(P_{\regionB}(u)) = \frac{(5-4c)^3}{2} - 4\left(u-\frac{1}{2} \right)^3 - (2-2c+u)^3 + (2u-1)^3
\end{array} \]
so
\[ S_{\Xthree, \dee}(E') = \frac{-10 c^3 + 46 c^2 - 71 c + 37}{3 (3 - 2 c)^2}.\]
The assumptions imply that \(A_{\Xthree, \dee}(E') = 1\), we find that \(\delta_{\Xthree, \dee}(E') > 1\) if and only if \(c>\cnot\), and \(\delta_{\Xthree, \dee}(E') = 1\) if and only if \(c=\cnot\).
\begin{detail}
We give more detail for the volume computations.  
If $g: \widehat{W} \to W$ is the blow-up of $l_2$ with exceptional divisor $F$, we may compute volumes by comparing $P(u)$ on $W$ to its pullback on $\widehat{W}$.  Let \(\hsmallplus\colon W^+ \to \Xthree\) be the other \(\bb Q\)-factorialization. Write \(D\coloneqq cR\) and \( L^+ \coloneqq \hsmallplus^*(-K_{\Xthree} - \dee)\).  

Recalling from above that \(\piP\colon W^+ \to \bb P(1,1,1,2)\) is the composition of the blow-ups at \(\ptE\) and the cone point \(\conept\), we have \(L^+ = \piP^*(-K_{\bb P} - D_{\bb P}) - \frac{1}{2} S^+- (2-2c) E^+\) and \(\vol L^+ = 3(2-2c)(3-2c)^2\).

Indeed, the restrictions \(\piP^*(-K_{\bb P} - D_{\bb P})|_{S^+}\), \(\piP^*(-K_{\bb P} - D_{\bb P})|_{E^+}\), and \(S^+|_{E^+}\) are trivial; \(-K_{\bb P} - D_{\bb P} \sim \cal O_{\bb P(1,1,1,2)}(5 - 4c)\); \(S^+|_{S^+} \cong \cal O_{\bb P^2}(-2)\); and \(E^+|_{E^+} \cong\cal O_{\bb P^2}(-1)\). So the volume of \(L\) is
\begin{equation*}\begin{split}
\vol L &= (-K_{\bb P} - D_{\bb P})^3 - \tfrac{1}{8}(S^+)^3 - (2-2c)^3 (E^+)^3 \\
&= \frac{(5-4c)^3}{2} - \frac{4}{8} - (2-2c)^3 = 3(2-2c)(3-2c)^2.
\end{split}\end{equation*}

On the region \(0\leq u\leq 1\) we have \(g^*(P_{\regionA}(u)) = (g^+)^*\hsmallplus^*(L^+-uE^+) - u F \), and we compute the restrictions \((g^+)^*\hsmallplus^*(L^+-uE^+)|_{F} \in |\cal O_F(0,-u)|\) and \(F|_F \in |\cal O_F(-1,-1)|\). So
\begin{equation*}\begin{split}
P_{\regionA}(u)^3 &= (L^+-uE^+)^3 - 3u\cdot 0 + 3 u^2 \cdot u - u^3 (F|_F)^3 \\
&= (5-4c)^3 / 2 - 4/8 - (2-2c+u)^3 + 3u^3 - 2 u^3 .
\end{split}\end{equation*}
For region \(\regionB\) with \(1 \leq u\leq 3-2c\) we have \begin{equation*}\begin{split}P_{\regionB}(u) &= g^*(P_{\regionA}(u)-(u-1)\hS) \\&= (g^+)^*(\piP^*(-K_{\bb P} - D_{\bb P}) - \tfrac{1}{2} S^+ - (2-2c+u) E^+) - (u-1) \hS - (2u-1) F \\
&= (g^+)^*(\piP^*(-K_{\bb P} - D_{\bb P}) - (u-\tfrac{1}{2}) S^+ - (2-2c+u) E^+) - (2u-1) F
\end{split}\end{equation*}
so, denoting \(P' \coloneqq (g^+)^*(\piP^*(-K_{\bb P} - D_{\bb P}) - (u-\tfrac{1}{2}) S^+ - (2-2c+u) E^+) = (g^+)^*\hsmallplus^*(L-uE^+) - (u-1) S^+ \), we have \(P'|_F\in|\cal O_F(0,1-2u)|\). So we compute
\begin{equation*}\begin{split}
P_{\regionB}(u)^3 &= P'^3 - 3(2u-1)(P'_F)^2 + 3(2u-1)^3(P'_F \cdot F|_F) - (2u-1)^2(F|_F)^2 \\
&= (5-4c)^3/2 - 4(u-1/2)^3 - (2-2c+u)^3 + 0 + 3(2u-1)^2(-1+2u) -2(2u-1)^3.
\end{split}\end{equation*}
\end{detail}
\end{proof}

\begin{corollary}\label{cor:R-does-not-contain-E}
    If $R \subset \Xthree$ is any surface of class $2D_1$ containing $E$, then $(\Xthree, cR)$ is K-unstable for all $c \in (0,1) \cap \bb Q$. 
\end{corollary}

\begin{proof}
    Denote $cR$ by $\dee$. If $R$ contains $E$, then $A_{\Xthree, \dee}(E') \le 1 - c $. Using the computation of \(S_{\Xthree, \dee}(E')\) in Lemma~\ref{lem:X3-only-after-wall}, this implies $\delta_{\Xthree,\dee}(E') < 1$ for all $c \in (0,1)$. 
\end{proof}

Next, we compute $\delta_{\Xthree,\dee}(S)$ to show that, under the K-semistability assumption, $R$ cannot meet $S$.  Recall that, by Lemma \ref{lem:strict-transform-of-X_3-surfaces-in-P}, if $R$ meets $S$, then in fact $R$ contains $S$.

\begin{lemma}\label{lem:R-cannot-contain-S}
    If $R \subset \Xthree$ is a surface of class $2D_1$ containing $S$, then $(\Xthree, cR)$ is K-unstable for all $c \in (0,1) \cap \bb Q$.  
\end{lemma}

\begin{proof}
    Let $\dee = cR$. We work on the small modification $\hsmall\colon W \to \Xthree$ whose Mori cone is generated by the three curves $l_1$, the exceptional locus of the modification, a fiber $l_E$ of the ruled surface $E \cong \bF_1$, and a fiber $l_S$ of the ruled surface $S \cong \bF_1$.  Because $W$ is a $\bb F_1$-bundle over $\bP^1$ and $S$ is the surface swept out by the negative section, we compute that $L - uS$ is ample for $u \le 2-2c$, where $L = \hsmall^*(-K_{\Xthree} - \dee)$.  At $u = 2-2c$, $L-uS$ is trivial on $l_E$.  The contraction of $l_E$ contracts the threefold $W$ to the surface $S$, so the pseudoeffective threshold is $2-2c$.
\begin{detail}
    Therefore, 
    \[ S_{\Xthree,\dee}(S) = \frac{1}{3(2-2c)(3-2c)^2} \int_0^{2-2c} (L-uS)^3 \; du .\]
\end{detail}

    Because $S \cong \bF_1$ with negative section $l_1$ and fiber $l_S$ satisfying $S \cdot l_1 = -1$ and $S \cdot l_S = -1$, we have $S|_S = -l_1 - 2l_S$.  Similarly, $L\cdot l_1 = 0$ and $L \cdot l_S = 1$ so $L|_S = l_1+l_S$.  Therefore, we compute 
    \[ \resizebox{1\textwidth}{!}{ $ \displaystyle S_{\Xthree,\dee}(S) = \frac{1}{3(2-2c)(3-2c)^2} \int_0^{2-2c} ( 3(2-2c)(3-2c)^2 -3u - 6u^2 - 3u^3)\; du  = \frac{(1-c)(18c^2 - 52c +37)}{3(3-2c)^2}. $ }\]
    If $R$ contains $S$, then $A_{\Xthree,\dee}(S) \le 1 -c$. This implies that $\delta_{\Xthree,\dee}(S) < 1$ for all $c \in (0,1)$, and hence the pair is K-unstable. 
\end{proof}

If moreover the multiplicity 2 point at \(\ptE\) is an \(A_1\) singularity, then the strict transform of \(R\) in \(\bb P(1,1,1,2)\) can be written in a particular form:

\begin{lemma}\label{lem:equation-of-R_P}
Let \(R_{\bb P}\in|\cal O_{\bb P(1,1,1,2)}(4)|\) be a divisor that does not pass through the cone point and that has an \(A_1\) singularity at a smooth point \(\ptE\) of \(\bb P(1,1,1,2)\). Then one may choose coordinates \([x:y:z:w]\) on \(\bb P(1,1,1,2)\) so that \(R_{\bb P}\) is defined by an equation of the form \[w^2 = xyz^2 + z(ax^3 + by^3) + f_4(x,y)\] where \(f_4\in\bb C[x,y]\) is homogeneous of degree \(4\), and \(\ptE = [0:0:1:0]\).
\end{lemma}

\begin{proof}
Let \(x,y,z,w\) be coordinates on \(\bb P(1,1,1,2)\), where \(w\) has weight 2. Since \(R_{\bb P}\) does not contain the cone point \(\conept = [0:0:0:1]\), it is defined by an equation of the form \(w^2 + 2w h_2 + h_4\). After a change of variables \(w + h_2 \mapsto w\), the defining equation becomes \(w^2 = g\) where \(g\) is homogeneous of degree 4 in \(x,y,z\). After a further change of coordinates in \(x,y,z\), we may assume \(\ptE = [0:0:1:0]\). Then, since \(R_{\bb P}\) has multiplicity 2 at \(\ptE\), it is defined by \(w^2 = z^2 f_2(x,y) + z f_3(x,y) + f_4(x,y)\), and the assumption that \(\ptE\) is an \(A_1\) singularity implies that, up to a coordinate change, \(f_2=xy\).
\begin{detail}
    Indeed, if we write \(f_2(x,y)=ax^2+bxy+cy^2\), then the Hessian matrix at \(\ptE\) is \[\begin{pmatrix} 2a & b & 0 \\ b & 2c & 0 \\ 0 & 0 & -1\end{pmatrix} .\]
\end{detail}
We may then write $f_3(x,y) = (ax^3 + by^3 + cx^2y + dxy^2)$ and combine the terms divisible by $xy$ to obtain the equation $w^2 = xy(z^2 + cxz + dyz) + z(ax^3 + by^3) + f_4(x,y)$.  Now, completing the square in the first term and making the change of coordinates $z \mapsto z - \frac{1}{2}(cx + dy)$ yields a curve of the form $w^2 = xyz^2 + z(ax^3 + by^3) + f_4'(x,y)$, as desired. 
\begin{detail}
    Indeed, we may write the equation as 
    \begin{align*}
        w^2 &= xy(z^2 + cxz + dyz) + z(ax^3 + by^3) + f_4(x,y) \\
        &= xy(z+ \frac{1}{2}(cx + dy))^2 + z(ax^3 + by^3) + f_4(x,y) - \frac{xy}{4}(cx+dy)^2,  
    \end{align*}
    and after applying the coordinate change $z \mapsto z - \frac{1}{2}(cx + dy)$, we obtain
    \begin{align*}
        &w^2 = xyz^2 + (z-\frac{1}{2}(cx + dy))(ax^3 + by^3) + f_4(x,y) - \frac{xy}{4}(cx+dy)^2 \\
        &= xyz^2 + z(ax^3+by^3) + f_4'(x,y), 
    \end{align*}
    where $f_4'(x,y) = f_4(x,y) -\frac{xy}{4}(cx+dy)^2 - \frac{1}{2}(cx+dy)(ax^3+by^3) $. 
\end{detail}
\end{proof}

\begin{remark}\label{rem:nodal-plane-quartics-classification-A_n}
    In the setting of Lemma~\ref{lem:equation-of-R_P}, the plane quartic defined by \(xyz^2 + z f_3(x,y) + f_4(x,y)\) has an \(A_1\) singularity at \([0:0:1]\). If this quartic has another \(A_n\) singularity, then the classification of singularities of plane quartics shows that \(n\leq 5\), see e.g \cite{Wall-quartics} or \cite[Section 8.7.1]{Dolgachev-CAG}.
\end{remark}

\subsection{The discriminant curve}

In this section, we consider the discriminant curve \(\Delta\) in families. First we work over \(\bb A^1\). Let $\overline{\calX}\rightarrow \bb A^1$ be the degeneration of $\bb P^1 \times\bb P^2$ to $\Xthree$ described in Proposition~\ref{prop:degentoX3}.

\begin{lemma}\label{lem:morphism-from-family-to-P2}
    There is a morphism $f\colon\overline{\cal{X}}\rightarrow \mathbb{P}^2 \times \bb A^1$ such that
    \begin{enumerate}
        \item The restriction $f|_{f^{-1}(\bb P^2\times\{0\})}\colon \Xthree \to \bb P^2$ is the morphism $\piXthree$ defined in Section~\ref{sec:construction-of-X3}, and
        \item Away from the central fiber \(0\in\bb A^1\), \(f\) restricts on fibers to the projection $\bb P^1\times\bb P^2 \to \mathbb{P}^2$ (up to an automorphism of $\mathbb{P}^2$).
    \end{enumerate}
\end{lemma}
\begin{proof}
    Start with the projection map $\mathbb{P}^1 \times \mathbb{P}^2\times \mathbb{A}^1 \rightarrow \mathbb{P}^2 \times \mathbb{A}^1$,
    which is the data of the $(0,1)$ linear system on fibers of the map to $\mathbb{A}^1$.
    Pulling back \(\cal O_{\bb P^1\times\bb P^2}(0,1)\) to $\widetilde{\mathcal{X}}$ defines a line bundle $\mathcal{L}_0$ on $\widetilde{\mathcal{X}}$. The strict transform $\widetilde{X}$ of the special fiber and the exceptional $\widetilde{W}$ are both Cartier divisors. Consider the line bundle $\cal L'\coloneqq \mathcal{L}_0(-\widetilde{X})$. To build a line bundle on $\mathcal{X}'$, observe that the restriction of $\mathcal{L}_0(-\widetilde{X})$ to $\widetilde{X}$ is $\mathcal{L}_0|_{\widetilde{X}}(\widetilde{W}\cap \widetilde{X})$. The restriction of $\mathcal{L}_0(\widetilde{Z})$ to $S = \mathbb{P}^1\times \mathbb{P}^1$ is the $(0,1)$ class and thus $\mathcal{L}|_{\widetilde{Z}}(\widetilde{W}\cap \widetilde{X})$ descends to a line bundle on $\mathcal{X}'$. A similar argument shows the resulting line bundle further descends to a line bundle \(\overline{\cal L}\) on $\overline{\mathcal{X}}$. On points away from the preimage of $0 \in \mathbb{A}^1$, the line bundle \(\overline{\cal L}\) is degree $(0,1)$ and thus induces the required projection. 
    
    The restriction of \(\overline{\cal L}\) to the special fiber pulls back to the line bundle $\mathcal{L}_0|_{\widetilde{X}}(\widetilde{W}\cap \widetilde{X})$ on $\widetilde{W}$ and thus corresponds to the piecewise linear function with value 1 on the rays $(0,0,1)$ and $(-1,-1,0)$ but zero on other rays. This line bundle descends to a line bundle on $W$ because its piecewise linear function does not bend on the two-dimensional cone generated by $(0,-1,0)$ and $(-1,-1,0)$. Moreover, points in the polytope associated to \(\overline{\cal L}\) are (up to translation) $(0,0,0), (1,0,0)$ and $(0,0,1)$, thus inducing the map $\piXthree$.
\end{proof}

\begin{corollary}\label{cor:discriminant-family}
    Let \((\cal X,\cal R)\to T\) be a flat family of pairs, where \(T\) is \(\bb A^1\) or the germ of a smooth curve. Assume the following hold:
    \begin{enumerate}
        \item \(\cal X\) is the constant family \(\bb P^1 \times\bb P^2\times T\), the constant family \(\Xthree\times T\), or an isotrivial degeneration of \(\bb P^1\times\bb P^2\) to \(\Xthree\).
        \item\label{item:discriminant-family-assumptions} If \(\cal X_t=\bb P^1\times\bb P^2\) then \(\cal R_t\) is an integral \((2,2)\)-surface. If \(\cal X_t=\Xthree\) then \(\cal R_t\) is the strict transform of a degree 4 divisor in \(\bb P(1,1,1,2)\) that does not pass through the cone point and that has an \(A_1\) singularity at \(\ptE\) (as in Lemma~\ref{lem:equation-of-R_P}).
    \end{enumerate}
    Then there is a Cartier divisor \(\cal D\subset\bb P^2\times T\), flat over \(T\), such that each fiber \(\cal D_t\) is the branch locus \(\Delta_t\) of \(\cal R_t \to \bb P^2\). (Equivalently, if \(\cal Y\to\cal X\) is the double cover of \(\cal X\) branched along \(\cal R\), the fiber \(\cal D_t\) is the discriminant curve of the conic bundle \(\cal Y_t\to\bb P^2\).)
\end{corollary}

\begin{proof}
    It suffices to prove this for \(T = \bb A^1\).
    In each case, there is a morphism \(f\colon\cal X\to \bb P^2\times T\), given by \(\piX_2\times\id\) (for the constant family \(\bb P^1\times\bb P^2\times T\)), \(\piXthree\times\id\) (for the constant family \(\Xthree\times T\)), or the pullback of the map from the versal deformation $\overline{\cal X} \to \bb A^1$ from Lemma~\ref{lem:morphism-from-family-to-P2} to $\cal X$ for an isotrivial degeneration of \(\bb P^1\times\bb P^2\) to \(\Xthree\).
    Let \(\cal Z\subset\overline{\cal X}\) be defined by the \(0\)th Fitting ideal of the sheaf of differentials \(\Omega_{\cal R/\bb P^2\times T}\) \cite[\href{https://stacks.math.columbia.edu/tag/0C3H}{Tag 0C3H}]{stacks-project}. Then the fiber \(\cal Z_t\) is the closed subscheme of \(\cal X_t\) defined by the \(0\)th Fitting ideal of \(\Omega_{\cal R_t/\bb P^2}\). By \cite[Tags \href{https://stacks.math.columbia.edu/tag/0C3J}{0C3J}, \href{https://stacks.math.columbia.edu/tag/0C3K}{0C3K}, and \href{https://stacks.math.columbia.edu/tag/024L}{024L}]{stacks-project}, \(\cal Z_t\subset\cal R_t\) defines the locus where the fibers of \(\cal R_t\to\bb P^2\) are singular or infinite. Computing using the assumption~\eqref{item:discriminant-family-assumptions} and the equations for \(\cal R_t\) in \(\bb P^1\times\bb P^2\) or \(\Xthree\), we see that \(\cal Z_t\subset\cal R_t\) has pure codimension 1, so \(\cal Z\subset\cal R\) is a Weil divisor, and is flat over $T$. Define \(\cal D\coloneqq (f|_{\cal R})_*\cal Z\). This divisor is Cartier as \(\bb P^2\times T\) is smooth, and it is flat over $T$. By a direct computation, one sees that the fiber \(\cal D_t\) is equal to the discriminant of the conic bundle \(\cal Y_t\to\bb P^2\) as defined in Definition~\ref{defn:discriminant}.
\begin{detail}
    Indeed, if \(\cal X_t = \bb P^1\times\bb P^2\), then \(\cal R_t\) is defined by \(t_0^2 Q_1 + 2 t_0 t_1 Q_2 + t_1^2 Q_3\). On the affine chart \(U_{00} = (t_0,y_0\neq 0)\), we have the presentation of \(\Omega_{\cal R_t / \bb P^2}|_U\)
    \[\cal O_U \xrightarrow{\begin{pmatrix} Q_2(1,y_1,y_2) + t_1 Q_3(1,y_1,y_2) \end{pmatrix}} \cal O_U \xrightarrow{\begin{pmatrix} d t_1 \end{pmatrix}} \Omega_{\cal R_t / \bb P^2}|_U \to 0\]
    so \(\mathrm{Fit}_0(\Omega_{\cal R_t / \bb P^2})|_{U_{00}}\) is generated by \(Q_2(1,y_1,y_2) + t_1 Q_3(1,y_1,y_2)\). Computing similarly on the other charts, one finds that
    \[V(\mathrm{Fit}_0(\Omega_{\cal R_t / \bb P^2})) = (t_0 Q_1 + t_0 Q_2 = t_0 Q_2 + t_1 Q_3 = 0) \subset \bb P^1\times \bb P^2 , \]
    and that the pushforward of this is \((Q_1 Q_3 - Q_2^2 = 0) \subset\bb P^2\).

    If \(\cal X_t = \Xthree\), then by Lemma~\ref{lem:equation-of-R_P} \(\cal R_t\), is defined by the strict transform of \((\cal R_t)_{\bb P}\coloneqq (w^2 = xyz^2 + z(ax^3 + by^3) + f_4(x,y))\subset\bb P(1,1,1,2)\) where \(\ptE = [0:0:1:0]\). Then a computation shows that \(\mathrm{Fit}_0(\Omega_{(\cal R_t)_{\bb P} / \bb P^2}) = (w = xyz^2 + z(ax^3 + by^3) + f_4(x,y) = 0)\), and the pushforward of this to \(\bb P^2\) is \((xyz^2 + z(ax^3 + by^3) + f_4(x,y) = 0)\). Since \(\cal R_t \to\bb P^2\) and \((\cal R_t)_{\bb P} \to \bb P^2\) agree on \(\bb P^2\) away from the point \(\piXthree(E) \in \bb P^2\), we see that \(\cal D_t = (xyz^2 + z(ax^3 + by^3) + f_4(x,y) = 0)\).
\end{detail}
\end{proof}

\section{K-stability of surfaces in \texorpdfstring{\(\Xthree\)}{Xn}}\label{sec:K-stability-Xthree}
Let \(\Xthree\) be the threefold constructed in Section~\ref{sec:construction-of-X3}; recall that \(\Xthree\) is a non-\(\bb Q\)-factorial Fano threefold with a single ordinary double point. In this section, we will show K-stability properties of surfaces in \(\Xthree\) that are degenerations of \((2,2)\)-surfaces in \(\bb P^1\times\bb P^2\), by computing \(\delta\)-invariants. Throughout this section, we will use the description of \(\Xthree\) from Remark~\ref{rem:X3-description-from-P(1112)}, and the computations will be similar to those in Section~\ref{sec:stability-of-A_n}. We refer the reader to Section~\ref{sec:fujita-abban-zhuang} for an overview of the strategy and relevant formulas that we will use.

Recall that if \(R\subset\Xthree\) is attained as a limit of $(2,2)$-divisors degenerating to \(\Rthree\), then either \(R\) is the strict transform of a degree 4 surface in \(\bb P(1,1,1,2)\) that has multiplicity 2 at \(\ptE\) (the smooth point of \(\bb P(1,1,1,2)\) that is blown up) and does not contain the cone point \(\conept\) of \(\bb P(1,1,1,2)\) (Lemma~\ref{lem:strict-transform-of-X_3-surfaces-in-P}), or $R$ contains $E^+$ or $S^+$.  By Corollary~\ref{cor:R-does-not-contain-E} and Lemma~\ref{lem:R-cannot-contain-S}, the latter possibility is impossible under the assumption that $(\Xthree, cR)$ is K-semistable.  Therefore, we assume throughout this section that $R$ is the strict transform of a surface $R_{\bP} \subset \bP(1,1,1,2)$ of degree 4 that has multiplicity 2 at $\ptE$ and does not contain the cone point $\conept$. Recall from Section~\ref{sec:mori-cone-small-resol-Xthree} that \(E,S\subset\Xthree\) are the strict transforms of the exceptional divisors over \(\ptE,\conept\).

Let \([x:y:z:w]\) be coordinates on \(\bb P(1,1,1,2)\) such that \(\ptE = [0:0:1:0]\) and \(\conept = [0:0:0:1]\), and let \(\Rox, \Rwall\subset\Xthree\) be the strict transforms of the degree 4 surfaces in \(\bb P(1,1,1,2)\) defined by \[(\Rox)_{\bb P}: \quad w^2=xy(z^2-xy), \qquad (\Rwall)_{\bb P}: \quad w^2=xyz^2.\]

\begin{theorem}\label{thm:local-delta-X_3}
    Let \(R\subset\Xthree\) be a surface whose strict transform \(R_{\bb P}\) in \(\bb P(1,1,1,2)\) is a divisor in \(|\cal O_{\bb P(1,1,1,2)}(4)|\) that has an \(A_1\) singularity at \(\ptE\) and that does not contain the cone point \(\conept\). Let \(\dee\coloneqq cR\), and consider the log Fano pair \((\Xthree,\dee)\) for \(c\in(0,1)\). Let \(\cone\approx 0.3293\) (resp. \(\cnot\approx 0.472\)) be the smallest root of the cubic polynomial \(7 c^3 - 23 c^2 + 22 c - 5\) (resp. \(10c^3-34c^2+35c-10=0\)).
    \begin{enumerate}
        \item\label{item:smooth-pts-X_3} If \(\pt \in \Xthree \setminus R\) or if \(\pt\in R\) is a smooth point, then \(\delta_{\pt}(\Xthree,\dee)>1\) for all \(c\in(\cnot,\tfrac{1}{2}]\).
        \item\label{item:A_1-sings-X_3} If \(\pt\in R \setminus E\) is an \(A_1\) singularity, then \(\delta_{\pt}(\Xthree,\dee)>1\) for all \(c\in(\cnot,1)\).
        \item\label{item:A_2-sings-X_3} If \(\pt\in R \setminus E\) is an \(A_2\) singularity, then \(\delta_{\pt}(\Xthree,\dee)>1\) for all \(c\in(0,\tfrac{1}{2}]\).
    \end{enumerate}
    For \(\pt\notin E\), let \(\Gamma_{\pt} \in|\cal O_{\bb P(1,1,1,2)}(1)|\) be the unique divisor passing through the images of \(E\) and \(\pt\).
    \begin{enumerate}\setcounter{enumi}{3}
        \item\label{item:A_n-sings-X_3} If \(\pt\in R \setminus E\) is an \(A_n\) singularity for \(n\geq 3\) and \(R_{\bb P}|_{\Gamma_{\pt}}\) is reduced, then \(\delta_{\pt}(\Xthree,\dee)>1\) for all \(c\in(\cone, 0.7055]\).
        \item\label{item:A_n-sings-X_3-non-reduced-fiber} If \(\pt\in R \setminus E\) is an \(A_n\) singularity for \(n\geq 3\) and \(R_{\bb P}|_{\Gamma_{\pt}}\) is not reduced, then \((\Xthree, cR)\) is not K-stable for \(c\in(\cnot,\tfrac{1}{2}] \cap \bb Q\). Furthermore, \((\Xthree, cR)\) degenerates to \((\Xthree, c\Rox)\), where \(\Rox\) is the surface defined above.
    \end{enumerate}
\end{theorem}

\begin{theorem}\label{thm:X_3-unstable}
    Let \(R\subset\Xthree\) be a surface whose strict transform \(R_{\bb P}\) in \(\bb P(1,1,1,2)\) is a divisor in \(|\cal O_{\bb P(1,1,1,2)}(4)|\) that has multiplicity 2 at \(\ptE\) and that does not contain the cone point \(\conept\).
    \begin{enumerate}
        \item\label{item:unstable-worse-than-A_2-at-E} If the strict transform \(R_{\bb P}\) has a non-\(A_n\) singularity at \(\ptE\) (including the non-isolated case) or an $A_n$ singularity for $n \ge 3$, then \(\delta(\Xthree,cR)<1\) for \(c \in (0,1)\). In particular, \((\Xthree, cR)\) is K-unstable for all \(c\in(0, 1) \cap \bb Q\).
        \item\label{item:unstable-A_2-at-E} If the strict transform \(R_{\bb P}\) has an \(A_n\) singularity at \(\ptE\) for some \(n\geq 2\), then \((\Xthree, cR)\) is K-unstable for \(c\in(0, 0.6] \cap \bb Q\). 
        \item\label{item:unstable-non-A_n-isolated} If \(\pt\in R\) is an isolated singularity that is not an \(A_n\) singularity, then \((\Xthree, cR)\) is K-unstable for all \(c\in (\cnot,1)\cap \bb Q\). 
        \item\label{item:unstable-non-isolated} If \(R\) has non-isolated singularities, then \(\delta(\Xthree,cR)<1\) unless \(c=\cnot\) and \(R = \Rwall\).
    \end{enumerate}
\end{theorem}

\begin{corollary}\label{cor:Rox-polystable}
    \(\Rox\) is the unique surface in \(\Xthree\) that is a degeneration of a \((2,2)\)-surface in \(\bb P^1\times\bb P^2\) and such that \((\Xthree, c\Rox)\) is strictly K-polystable for \(c\in (\cnot, \tfrac{1}{2}]\cap\bb Q\).
\end{corollary}

\begin{proof}[Proof of Theorem~\ref{thm:local-delta-X_3}]
    Parts~\eqref{item:smooth-pts-X_3}, \eqref{item:A_1-sings-X_3}, \eqref{item:A_2-sings-X_3}, and~\eqref{item:A_n-sings-X_3} are shown in Sections~\ref{sec:smooth-pts-X_3}, \ref{sec:A_1-sings-X_3}, \ref{sec:A_2-sings-X_3}, and \ref{sec:higher-A_n-sings-X_3}, respectively. Part~\eqref{item:A_n-sings-X_3-non-reduced-fiber} is shown in Lemma~\ref{lem:non-reduced-Gamma-cases}\eqref{part:higher-A_n-non-reduced-Gamma} and Corollary~\ref{cor:X_3-non-reduced-A_n-not-stable}.
\end{proof}

\begin{proof}[Proof of Theorem~\ref{thm:X_3-unstable}]
    Parts~\eqref{item:unstable-worse-than-A_2-at-E} and~\eqref{item:unstable-A_2-at-E} are shown in Section~\ref{sec:unstable-X_3-sings-at-E}. Part~\eqref{item:unstable-non-A_n-isolated} is shown in Lemma~\ref{lem:D_n-on-X_3-unstable}. Part~\eqref{item:unstable-non-isolated} is shown in Section~\ref{sec:X_3-Rwall-ss}.
\end{proof}

\begin{remark}
    The conditions in Theorem~\ref{thm:local-delta-X_3}\eqref{item:A_n-sings-X_3} and~\eqref{item:A_n-sings-X_3-non-reduced-fiber} about the reducedness of the restrictions \(R_{\bb P}|_{\Gamma_{\pt}}\) can be viewed as analogues of the reducedness of the fibers \(\piR_1\colon R\to\bb P^1\) for \(R\subset\bb P^1\times\bb P^2\).
    More precisely, recall that \(W\) is the unique toric small resolution of \(\Xthree\) with a del Pezzo fibration, namely the \(\bb F_1\)-fibration \(\piW\colon W\to\Cthree\cong\bb P^1\) (Section~\ref{sec:small-resolutions-X_3}).
    The strict transform of \(\Gamma_{\pt}\) in \(W\) is precisely the fiber of \(\piW\) containing the strict transform of \(\pt\).
    Recall also that in Section~\ref{sec:wall-crossing-c_0}, we showed these surfaces on \(\Xthree\) arise after the wall and will replace \([\Rthree]\in\Gthree\). On \(\bb P^1\times\bb P^2\), \(\Gthree\) parametrizes \((2,2)\)-surfaces with \(A_3\)-singularities in non-reduced fibers of \(\piR_1\colon R\to\bb P^1\) (Remark~\ref{rem:frakG_i-surfaces-descriptions}), where \(\piR_1\) is the restriction of the \(\bb P^2\)-fibration given by projection \(\bb P^1\times\bb P^2 \to \bb P^1\); after the wall, Theorem~\ref{thm:local-delta-X_3}\ref{item:A_n-sings-X_3-non-reduced-fiber} shows that surfaces on \(\Xthree\) with non-reduced fibers of \(\piW\) are all strictly K-semistable and are S-equivalent to \((\Xthree, c\Rox)\). (Note that for surfaces on \(\bb P^1\times\bb P^2\) we had assumptions on the finiteness of fibers of \(\piR_2\colon R\to\bb P^2\) in Theorem~\ref{thm:stable-members}, whereas for surfaces on \(\Xthree\) this assumption is automatic by the condition that their strict transforms on \(\bb P(1,1,1,2)\) don't pass through the cone point.)
\end{remark}

\begin{notation}
    We retain much of the notation from Section~\ref{sec:wall-crossing-c_0}. Recall that $l_2$ (resp. \(l_1\)) is the exceptional curve of the small contraction $\hsmall\colon W \to \Xthree$ (resp. $\hsmallplus\colon W^+ \to \Xthree$), see Section~\ref{sec:small-resolutions-X_3}. We denote by \(\piP\colon W^+ \to \bb P(1,1,1,2)\) the composition of blow-ups described in Remark~\ref{rem:X3-description-from-P(1112)}.
\end{notation}

\subsection{K-stability for points away from the singular locus of \texorpdfstring{\(R\)}{R}}\label{sec:smooth-pts-X_3}

In this section, we prove Theorem~\ref{thm:local-delta-X_3}\eqref{item:smooth-pts-X_3}. First, we address points of \(\Xthree\) not contained in \(E\) or the singular locus of \(R\) in Lemma~\ref{lem:X3-smooth-notin-E}, and next we address points of \(E\) that are not the singular point \(E\cap S\) of \(\Xthree\) in Lemma~\ref{lem:X3-smooth-in-E}; both of these computations also work for \(c\in [\tfrac{1}{2},1)\). Finally, in Lemma~\ref{lem:X3-singular-pt-delta} we address the singular point of \(\Xthree\); this is where we use that \(c\leq\tfrac{1}{2}\). Recall that \(R\cap S = \emptyset\) by assumption.
\begin{lemma}\label{lem:X3-smooth-notin-E}
Let \(R\subset\Xthree\) be as in Theorem~\ref{thm:local-delta-X_3}. Let \(\pt\in\Xthree\setminus E\), and further assume that \(\pt\) is not contained in the singular locus of \(R\). Then \(\delta_\pt(\Xthree, cR) > 1\) for \(c\in (\cnot,1)\).
\end{lemma}

\begin{proof}
We work on the small resolution \(\hsmallplus\colon W^+\to\Xthree\).
Let \(Y_{\bb P}\in|\cal O_{\bb P(1,1,1,2)}(1)|\) be a divisor that contains the strict transform of \(\pt\) and does not contain \(l_1\), and let \(Y\subset W^+\) be its strict transform. We will first compute \(\delta_{\Xthree, \dee}(Y)\).

Write \(L = (\hsmallplus)^*(-K_{\Xthree}-\dee)\). First, \(A_{\Xthree, \dee}(Y) = 1\). To compute the Nakayama--Zariski decomposition, we find the associated birational transformations. The pseudoeffective threshold is \(u=3-2c\), and at \(u=1\) we contract \(S\) in a morphism \(W^+ \to X_1\).
\begin{detail}
Indeed, \(L-uY =  \piP^*\cal O_{\bb P(1,1,1,2)}(5-4c-u) - \frac{1}{2}(1-u) S^+- (2-2c) E^+\). By Lemma~\ref{lem:Mori-cone-of-W+}, the Mori cone of \(W^+\) is generated by \(l_1\) (the curve contracted by \(\hsmallplus\)), a line \(l_S\) in \(S\), and a line \(l_E\) in \(E\). We have
\[ (L-uY) \cdot l_1 = 0, \qquad (L-uY) \cdot l_S = 1-u , \qquad (L-uY) \cdot l_E = 2-2c \]
so at \(u=1\), we contract \(S\) in a morphism \(f_1\colon W^+ \to X_1\).
Since \(f_1^*(K_{X_1}+\dee_1 + u Y_1) = L-uY - (1-u)S^+\), by intersecting with \(l_E\) we find that the pseudoeffective threshold is \(u=3-2c\).
\end{detail}
So, using the description of \(L\) from \(\bb P(1,1,1,2)\) as in the proof of Lemma~\ref{lem:X3-only-after-wall}, we have
\[ \resizebox{1\textwidth}{!}{ $ \begin{array}{rll}
    0 \le u \le 1: & P_{\regionA}(u) = \piP^*\cal O_{\bb P(1,1,1,2)}(5-4c-u) - \frac{1}{2}(1-u) S^+- (2-2c) E^+ & N_{\regionA}(u)  =  0 \\
    1 \le u \le 3-2c: & P_{\regionB}(u) = \piP^*\cal O_{\bb P(1,1,1,2)}(5-4c-u) - \frac{1}{2}(u-1) S^+- (2-2c) E^+ & N_{\regionB}(u) = (u-1)S^+
\end{array} $} \]
where here, as in Section~\ref{sec:stability-of-A_n}, we use the subscripts \(\regionA, \regionB\) to denote the respective regions for \(u\). Using Lemma~\ref{lem:volume-lemma-2} and the computation of \(\vol L\) in Lemma~\ref{lem:X3-only-after-wall}, we find that
\(\vol(P_{\regionA}(u)) = \tfrac{(5-4c-u)^3}{2} - \tfrac{(1-u)^3}{2} - (2-2c)^3\) and \(\vol(P_{\regionB}(u)) = \tfrac{(5-4c-u)^3}{2} + \tfrac{(1 - u)^3}{2} - (2-2c)^3\). Therefore, we find
\begin{detail}
\[ \resizebox{1\textwidth}{!}{ $
   S_{\Xthree, \dee}(Y) = \displaystyle \frac{1}{\vol L} \left( \int_0^1 \left( \frac{(5-4c-u)^3}{2} - \frac{(1-u)^3}{2} - (2-2c)^3 \right) du + \int_1^{3-2c} \left( \frac{(5-4c-u)^3}{2} + \frac{(1-u)^3}{2} - (2-2c)^3 \right) du \right), $ } \]
\end{detail}
\[ S_{\Xthree, \dee}(Y) = \frac{-(6 c^3 - 30 c^2 + 48 c - 25)}{3 (3 - 2 c)^2}, \qquad \frac{A_{\Xthree, \dee}(Y)}{S_{\Xthree, \dee}(Y)} = \frac{-3 (3 - 2 c)^2}{6 c^3 - 30 c^2 + 48 c - 25} > 1\]
for \(c\in (0,1)\). 

Next, let \(f\subset Y\) be a ruling containing the strict transform of \(\pt\), and let \(\sigma\) denote the negative section of \(Y\cong\bb F_2\). We will use the Abban--Zhuang method. Let \(R|_Y\) denote the intersection of \(Y\) with the strict transform of \(Y\) in \(W^+\). We have two cases, depending on how \(f\) meets \(R|_Y\).

\noindent\underline{Case: \(q\notin R|_Y\), or \(R|_Y\) meets \(f\) transversally at \(q\):} We use the flag \(q\in f\subset Y\). From above, we have
\[ \begin{array}{rll}
    0 \le u \le 1: & P_{\regionA}(u)|_Y = (5-4c-u)f + (2-2c)\sigma & N_{\regionA}(u)|_Y  =  0 \\
    1 \le u \le 3-2c: & P_{\regionB}(u)|_Y = (5-4c-u)f + (3-2c-u)\sigma & N_{\regionB}(u)|_Y = (u-1)\sigma
\end{array} \]

On region \(\regionA\), the Zariski decomposition of \(P_{\regionA}(u)|_Y - vf\) is
\begin{equation*}\begin{split}P_{\regionA}(u,v) &= \begin{cases} (5-4c-u-v)f + (2-2c)\sigma & 0 \le v \le 1-u \\ (5-4c-u-v)(f + \frac{1}{2}\sigma) & 1-u \le v \le 5-4c-u , \end{cases} \\
N_{\regionA}(u,v) &= \begin{cases} 0 \hphantom{(5-4c-u-v)f + (2-c)\sigma} & 0 \le v \le 1-u   \\
\tfrac{1}{2}(u+v-1)\sigma & 1-u \le v \le 5-4c-u .
\end{cases}\end{split}\end{equation*}
\begin{detail}
So the relevant quantities for \(0 \leq u \leq 1\) are
\[\begin{array}{r|ll}
   & 0 \leq v \leq 1-u & 1-u \leq v \leq 5-4c-u \\
       \hline
   P_{\regionA}(u,v)^2   & 2(5-4c-u-v)(2-2c) - 2(2-2c)^2 & (5-4c-u-v)^2/2\\
   P_{\regionA}(u,v)\cdot f  & 2-2c  & (5-4c-u-v)/2 \\
   (P_{\regionA}(u,v)\cdot f)^2  & (2-2c)^2 & (5-4c-u-v)^2/4 \\
   \ord_q((N'_{\regionA}(u)|_Y + N_{\regionA}(u,v))|_f) & 0 &
   \left\{ \begin{array}{ll}
      0  & q \notin \sigma  \\
     (u+v-1)/2   & q\in\sigma
   \end{array} \right. \\
\end{array}\]
\end{detail}
On region \(\regionB\), the Zariski decomposition of \(P_{\regionB}(u)|_Y - vf\) is
\begin{equation*}\begin{split}P_{\regionB}(u,v) &= \begin{cases} (5-4c-u-v)f + (3-2c-u)\sigma & 0 \le v \le u-1 \\ (5-4c-u-v)(f + \frac{1}{2}\sigma) & u-1 \le v \le 5-4c-u , \end{cases} \\
N_{\regionB}(u,v) &= \begin{cases} 0 \hphantom{(5-4c-u-v)f + (3-c-u)\sigma} & 0 \le v \le u-1   \\
\tfrac{1}{2} (-u + v + 1)\sigma & u-1 \le v \le 5-4c-u .
\end{cases}\end{split}\end{equation*}
\begin{detail}
So the relevant quantities for \(1 \leq u \leq 3-2c\) are
\[\resizebox{1\textwidth}{!}{ $ \begin{array}{r|ll}
   & 0 \leq v \leq u-1 & u-1 \leq v \leq 5-4c-u \\
       \hline
   P_{\regionB}(u,v)^2   & 2(5-4c-u-v)(3-2c-u) - 2(3-2c-u)^2 & (5-4c-u-v)^2/2\\
   P_{\regionB}(u,v)\cdot f  & 3-2c-u  & (5-4c-u-v)/2 \\
   (P_{\regionB}(u,v)\cdot f)^2  & (3-2c-u)^2 & (5-4c-u-v)^2/4 \\
   \ord_q((N'_{\regionB}(u)|_Y + N_{\regionB}(u,v))|_f) & \left\{ \begin{array}{ll}
      0  & q \notin \sigma  \\
     u-1   & q \in \sigma
   \end{array} \right. &
   \left\{ \begin{array}{ll}
      0  & q \notin \sigma  \\
     (u + v - 1)/2   & q \in \sigma
   \end{array} \right. \\
\end{array} $ }\]
\end{detail}
We have \(A_{Y,\dee_Y}(f) = 1\), so using the above and \cite[Theorem 4.8]{Fujita-3.11}, we find that
\[S_{Y,\dee_Y; V_{\bullet \bullet}}(f) = \frac{-14 c^3 + 58 c^2 - 80 c + 37}{3 (3 - 2 c)^2}, \qquad \frac{A_{Y,\dee_Y}(f)}{S_{Y,\dee_Y; V_{\bullet \bullet}}(f)} = \frac{3 (3 - 2 c)^2}{-14 c^3 + 58 c^2 - 80 c + 37} > 1\] for \(c\in(\cone,1)\).

Next, we address the refinement \(W_{\bullet \bullet \bullet}\) by \(f\). First, note that \(A_{f,\dee_f}(q)\) is equal to \(1\) or \(1-c\) by assumption. Moreover, if \(q\in\sigma\) then \(p\notin R\) by assumption, so \(A_{f,\dee_f}(q) = 1\). Computing using the above Zariski decompositions and \cite[Theorem 4.17]{Fujita-3.11}, we find
\[\resizebox{1\textwidth}{!}{ $ S_{f,\dee_f; W_{\bullet \bullet \bullet}}(q) = \begin{cases}
    \displaystyle \frac{(1 - c) (6 c^2 - 20 c + 17)}{3 (3 - 2 c)^2} & q\notin \sigma \\
    \displaystyle \frac{(1 - c) (18 c^2 - 52 c + 37)}{3 (3 - 2 c)^2} & q\in\sigma ,
\end{cases} \quad
\displaystyle\frac{A_{C,\dee_C}(q)}{S_{C,\dee_C; W_{\bullet \bullet \bullet}}(q)} = \begin{cases}
    \displaystyle \geq \frac{3 (3 - 2 c)^2}{6 c^2 - 20 c + 17} & q\notin \sigma \\
    \displaystyle \frac{3 (3 - 2 c)^2}{(1 - c) (18 c^2 - 52 c + 37)} & q\in\sigma
\end{cases} $}\]
so \(A_{C,\dee_C}(q)/S_{C,\dee_C; W_{\bullet \bullet \bullet}}(q) > 1\) for \(c\in [0.2535, 1)\) (more precisely, the lower bound for \(c\) is the smallest root of the cubic polynomial \(18 c^3 - 58 c^2 + 53 c - 10\)).

\noindent\underline{Case: \(R|_Y\) is tangent to \(f\) at \(q\):} Let \(\hY \to Y\) be the \((2,1)\) weighted blow-up separating \(R|_Y\) and \(f\), with exceptional divisor \(e\). We will use the flag \(\qhatpt\in e\subset\hY\). By abuse of notation, we use \(f\) and \(\sigma\) to denote their strict transforms in \(\hY\).
\begin{detail}
We have
\[ \resizebox{1\textwidth}{!}{ $ \begin{array}{rll}
    0 \le u \le 1: & P_{\regionA}(u)|_\hY = (5-4c-u)f + (2-2c)\sigma + (10-8c-2u) e & N_{\regionA}(u)|_\hY  =  0 \\
    1 \le u \le 3-2c: & P_{\regionB}(u)|_\hY = (5-4c-u)f + (3-2c-u)\sigma + (10-8c-2u) e & N_{\regionB}(u)|_\hY = (u-1)\sigma
\end{array} $ }\]
and the intersection theory on \(\hY\) is given by \(f^2 = \sigma^2 = -2\), \(e^2 = -\tfrac{1}{2}\), \(e\cdot f = f\cdot \sigma = 1\), and \(e\cdot\sigma = 0\).
\end{detail}

On region \(\regionA\), the Zariski decomposition of \(P_{\regionA}(u)|_\hY - ve\) is
\begin{equation*}\begin{split}P_{\regionA}(u,v) &= \begin{cases} (5-4c-u)f + (2-2c)\sigma + (10-8c-2u-v) e & 0 \le v \le 2-2c \\ (2-2c)(\tfrac{1}{2}f + \sigma) + (10-8c-2u-v) (\tfrac{1}{2}f + e) & 2-2c \le v \le 4-2c-2u \\ (10-8c-2u-v)(\tfrac{2}{3} f + \tfrac{1}{3} \sigma + e) & 4-2c-2u \le v \le 10-8c-2u , \end{cases} \\
N_{\regionA}(u,v) &= \begin{cases} 0 \hphantom{(5-4c-u)f + (2-c)\sigma + (10-8c-2u-v) e} & 0 \le v \le 2-2c   \\
\tfrac{1}{2} (v - 2 + 2 c) f & 2-2c \le v \le 4-2c-2u \\ \tfrac{1}{3} (4 c + u + 2 v - 5) f + \frac{1}{3}(2 c + 2 u + v - 4)\sigma & 4-2c-2u \le v \le 10-8c-2u . \end{cases} \end{split}\end{equation*}
\begin{detail}
So the relevant quantities for \(0 \leq u \leq 1\) and \(\qhatpt\in e\) are
\[\resizebox{1\textwidth}{!}{ $ \begin{array}{r|lll}
   & 0 \le v \le 2-2c & 2-2c \le v \le 4-2c-2u & 4-2c-2u \le v \le 10-8c-2u \\
       \hline
   P_{\regionA}(u,v)^2   & -2(5-4c-u)^2 -2(2-2c)^2 -(10-8c-2u-v)^2/2 + 2(5-4c-u)(2-2c) + 2(5-4c-u)(10-8c-2u-v) & (-3/2)(2-2c)^2+(2-2c)(10-8c-2u-v) & (10-8c-2u-v)^2/6\\
   P_{\regionA}(u,v)\cdot e  & v/2  & 1-c & (10-8c-2u-v)/6 \\
   (P_{\regionA}(u,v)\cdot e)^2  & v^2/4 & (1-c)^2 & (10-8c-2u-v)^2/36 \\
   \ord_{\qhatpt}((N'_{\regionA}(u)|_\hY + N_{\regionA}(u,v))|_e) & 0 &
   \left\{ \begin{array}{ll}
      0  & \qhatpt \notin f  \\
     (v - 2 + 2 c)/2   & q\in f
   \end{array} \right. &
   \left\{ \begin{array}{ll}
      0  & \qhatpt \notin f  \\
     (4 c + u + 2 v - 5)/3   & q\in f
   \end{array} \right. \\
\end{array} $ }\]
\end{detail}
On region \(\regionB\), the Zariski decomposition of \(P_{\regionB}(u)|_\hY - ve\) is
\[ \resizebox{1\textwidth}{!}{ $ \begin{array}{rl}
   P_{\regionB}(u,v) &= \begin{cases} (5-4c-u)f + (3-2c-u)\sigma + (10-8c-2u-v) e & 0 \le v \le 3-2 c - u \\ (3-2c-u)(\tfrac{1}{2}f + \sigma) + (10-8c-2u-v) (\tfrac{1}{2}f + e)  & 3-2 c - u \le v \le 1 -2 c + u\\ (10-8c-2u-v)(\tfrac{2}{3} f + \tfrac{1}{3} \sigma + e) & 1 -2 c + u \le v \le 10-8c-2u , \end{cases} \\
N_{\regionB}(u,v) &= \begin{cases} 0 \hphantom{(5-4c-u)f + (3-2c-u)\sigma + (1-8c-2u-v) e} & 0 \le v \le 3-2 c - u   \\
\tfrac{1}{2} (u + v - 3 + 2 c) f & 3-2 c - u \le v \le 1 -2 c + u \\ \tfrac{1}{3} (4 c + u + 2 v - 5)f + \tfrac{1}{3}(2 c - u + v - 1) \sigma & 1 -2 c + u \le v \le 10-8c-2u .\end{cases}
\end{array} $ } \]
\begin{detail}
So the relevant quantities for \(1 \leq u \leq 3-2c\) and \(\qhatpt\in e\) are
\[\resizebox{1\textwidth}{!}{ $ \begin{array}{r|lll}
   & 0 \le v \le 3-2 c - u & 3-2 c - u \le v \le 1 -2 c + u & 1 -2 c + u \le v \le 10-8c-2u \\
       \hline
   P_{\regionB}(u,v)^2   & -2(5-4c-u)^2 -2(3-2c-u)^2 -(10-8c-2u-v)^2 /2 + 2(5-4c-u)(3-2c-u) + 2(5-4c-u)(10-8c-2u-v) & (-3/2)(3-2c-u)^2+(3-2c-u)(10-8c-2u-v) & (10-8c-2u-v)^2/6 \\
   P_{\regionB}(u,v)\cdot e  & v/2  & (3-2c-u)/2 & (10-8c-2u-v)/6 \\
   (P_{\regionB}(u,v)\cdot e)^2  & v^2/4 & (3-2c-u)^2/4 & (10-8c-2u-v)^2/36 \\
   \ord_{\qhatpt}((N'_{\regionB}(u)|_Y + N_{\regionB}(u,v))|_e) & 0 &
   \left\{ \begin{array}{ll}
      0  & \qhatpt \notin f  \\
     (u + v - 3 + 2 c)/2   & \qhatpt \in f
   \end{array} \right. &
   \left\{ \begin{array}{ll}
      0  & \qhatpt \notin f  \\
     (4 c + u + 2 v - 5)/3   & \qhatpt \in f
   \end{array} \right. \\
\end{array} $ }\]
\end{detail}
We have \(A_{\hY,\dee_\hY}(e) = 3-2c\).
Using \cite[Theorem 4.8]{Fujita-3.11}, we find that
\[ S_{\hY,\dee_\hY; V_{\bullet \bullet}}(e)= \frac{-(34 c^3 - 142 c^2 + 197 c - 91)}{3 (3 - 2 c)^2}, \quad \frac{A_{\hY,\dee_\hY}(e)}{S_{\hY,\dee_\hY; V_{\bullet \bullet}}(e)} = \frac{-3 (3 - 2 c)^3}{34 c^3 - 142 c^2 + 197 c - 91} > 1 \] for \(c\in (\cnot,1)\).

Finally, for the refinement \(W_{\bullet \bullet \bullet}\) by \(e\), we compute using \cite[Theorem 4.17]{Fujita-3.11} that
\[A_{e,\dee_e}(\qhatpt) = \begin{cases}
    1, 1-c, \text{ or }\frac{1}{2} & \qhatpt\notin f \\
    1 & \qhatpt\in f ,
\end{cases} \qquad
S_{e,\dee_e; W_{\bullet \bullet \bullet}}(\qhatpt) = \begin{cases}
    \displaystyle \frac{(1 - c) (6 c^2 - 20 c + 17)}{6 (3 - 2 c)^2} & \qhatpt \notin f \\
    \displaystyle \frac{-14 c^3 + 58 c^2 - 80 c + 37}{3 (3 - 2 c)^2}& \qhatpt\in f.
\end{cases}\]
Thus, we find that \(A_{e,\dee_e}(\qhatpt) / S_{e,\dee_e; W_{\bullet \bullet \bullet}}(\qhatpt) > 1\) for \(c\in (0,1)\) if \(\qhatpt \notin f\), and for \(c\in (\cone,1)\) if \(\qhatpt \in f\). Here \(\cone\approx 0.3293\) is the smallest root of the cubic polynomial \(7 c^3 - 23 c^2 + 22 c - 5\).
\end{proof}

Next, we consider \(\pt\in E\). Recall that \(R|_E\) is a smooth conic by the assumption in Theorem~\ref{thm:local-delta-X_3}.
\begin{lemma}\label{lem:X3-smooth-in-E}
Let \(R\) be as in Theorem~\ref{thm:local-delta-X_3}. Let \(\pt\in E\subset\Xthree\), and assume \(\pt\) is not the singular point of \(\Xthree\) (i.e. \(\pt\neq E\cap S\)). Then \(\delta_\pt(\Xthree, cR) > 1\) for \(c\in (\cnot,1)\).
\end{lemma}

\begin{proof} First, the assumption that \(R_{\bb P}\) has an \(A_1\) singularity at \(\ptE\) implies that \(R|_E \subset E\cong \bb P^2\) is a smooth conic. Furthermore, by assumption \(R\) does not contain the singular point of \(\Xthree\), so \(\pt\in\Xthree\) is either in \(\Xthree\setminus R\) or is a smooth point of \(R\). Let \(\widehat{E}\cong\bb F_1\) be the strict transform of \(E\) on \(\widehat{W}\) (the common resolution defined in Section~\ref{sec:small-resolutions-X_3}). Let \(f\subset\widehat{E}\) be the ruling whose image in \(E\) passes through \(\pt\) (note that \(f \sim l_E\) in the notation from the proof of Lemma~\ref{lem:X3-only-after-wall}), and let \(\sigma\) denote the negative section. Let \(q\) be the strict transform of \(\pt\) in \(\widehat{W}\).

By Lemma~\ref{lem:X3-only-after-wall}, we have \(A_{\Xthree, \dee}(E^+)/S_{\Xthree, \dee}(E^+) > 1\) for \(c\in(\cnot,1)\). Next, we apply the Abban--Zhuang method. As in the proof of Lemma~\ref{lem:X3-smooth-notin-E}, we will have two cases, and we will use the flag coming either from a ruling of \(\widehat{E}\) or from a \((2,1)\) weighted blow-up.

\noindent\underline{Case: \(q\notin R_{\widehat{E}}\), or \(R_{\widehat{E}}\) meets \(f\) transversally at \(q\):}
We will use the flag \(q\in f\subset\widehat{E}\).
From the proof of Lemma~\ref{lem:X3-only-after-wall}, we have
\[ \begin{array}{rll}
    0 \le u \le 1: & P_{\regionA}(u)|_{\widehat{E}} = (2-2c+u)f + (2-2c)\sigma & N_{\regionA}(u)|_{\widehat{E}}  =  0 \\
    1 \le u \le 3-2c: & P_{\regionB}(u)|_{\widehat{E}} = (2-2c+u)f + (3-2c-u)\sigma & N_{\regionB}(u)|_{\widehat{E}} = (u-1)\sigma
\end{array} \]
and \(d(u)=0\). On region \(\regionA\), the Zariski decomposition of \(P_{\regionA}(u)|_{\widehat{E}} - vf\) is
\begin{equation*}\begin{split}P_{\regionA}(u,v) &= \begin{cases} (2-2c+u-v)f + (2-2c)\sigma & 0 \le v \le u \\ (2-2c+u-v)(f + \sigma) & u \le v \le 2-2c+u , \end{cases} \\
N_{\regionA}(u,v) &= \begin{cases} 0 \hphantom{(5-4c-u-v)f + (2-c)\sigma} & 0 \le v \le u   \\
(v-u)\sigma & u \le v \le 2-2c+u .
\end{cases}\end{split}\end{equation*}
\begin{detail}
So for \(0\leq u\leq 1\) and \(q\in f\), we have
\[ \begin{array}{r|ll}
   & 0 \le v \le u & u \le v \le 2-2c+u \\
       \hline
   P_{\regionA}(u,v)^2   & 2 (2-2c+u-v)(2-2c) - (2-2c)^2 & (2-2c+u-v)^2 \\
   P_{\regionA}(u,v)\cdot f  & 2-2c  & 2-2c+u-v \\
   (P_{\regionA}(u,v)\cdot f)^2  & (2-2c)^2 & (2-2c+u-v)^2 \\
   \ord_q((N'_{\regionA}(u)|_Y + N_{\regionA}(u,v))|_f) & 0 &
   \left\{ \begin{array}{ll}
      0  & q \notin \sigma  \\
     v-u   & q \in \sigma
   \end{array} \right. \\
\end{array} \]
\end{detail}
On region \(\regionB\), the Zariski decomposition of \(P_{\regionB}(u)|_{\widehat{E}} - vf\) is
\begin{equation*}\begin{split}P_{\regionB}(u,v) &= \begin{cases} (2-2c+u-v)f + (3-2c-u)\sigma & 0 \le v \le 2u - 1 \\ (2-2c+u-v)(f + \sigma) & 2u - 1 \le v \le 2-2c+u , \end{cases} \\
N_{\regionB}(u,v) &= \begin{cases} 0 \hphantom{(2-c+u-v)f + (3-2c-u)\sigma} & 0 \le v \le 2u - 1   \\
(v + 1 -2 u)\sigma & 2u -1 \le v \le 2-2c+u .
\end{cases}\end{split}\end{equation*}
\begin{detail}
So for \(1\leq u\leq 3-2c\) and \(q\in f\), we have
\[ \resizebox{1\textwidth}{!}{ $ \begin{array}{r|ll}
   & 0 \le v \le 2u - 1 & 2u - 1 \le v \le 2-2c+u \\
       \hline
   P_{\regionB}(u,v)^2   & 2 (2-2c+u-v)(3-2c-u) - (3-2c-u)^2 & (2-2c+u-v)^2 \\
   P_{\regionB}(u,v)\cdot f  & 3-2c-u  & 2-2c+u-v \\
   (P_{\regionB}(u,v)\cdot f)^2  & (3-2c-u)^2 & (2-2c+u-v)^2 \\
   \ord_q((N'_{\regionB}(u)|_Y + N_{\regionB}(u,v))|_f) & 0 &
   \left\{ \begin{array}{ll}
      0  & q \notin \sigma  \\
     v + 1 -2 u   & q \in \sigma
   \end{array} \right. \\
\end{array} $ } \]
\end{detail}
We have \(A_{\widehat{E},\dee_{\widehat{E}}}(f) = 1\). We compute, using \cite[Theorem 4.8]{Fujita-3.11}, that
\[ S_{\widehat{E},\dee_{\widehat{E}}; V_{\bullet \bullet}}(f) = \frac{-(14 c^3 - 58 c^2 + 80 c - 37)}{3 (3 - 2 c)^2} \]
so \(A_{\widehat{E},\dee_{\widehat{E}}}(f)/S_{\widehat{E},\dee_{\widehat{E}}; V_{\bullet \bullet}}(f) > 1\) for \(c\in (\cone,1)\).
Next, using \cite[Theorem 4.17]{Fujita-3.11}, for \(q\in f\),
\[S_{f,\dee_f; W_{\bullet \bullet \bullet}}(q) = \begin{cases}
    \displaystyle \frac{(1 - c) (6 c^2 - 20 c + 17)}{3 (3 - 2 c)^2} & q\notin \sigma \\
    \displaystyle \frac{(1 - c) (8 c^2 - 28 c + 23)}{3 (3 - 2 c)^2} & q\in\sigma .
\end{cases}\]
Since \(A_{f,\dee_f}(q) = 1\) or \(1-c\), we have \(A_{f,\dee_f}(q) / S_{f,\dee_f; W_{\bullet \bullet \bullet}}(q) > 1 \) for \(c\in (0,1)\).

\noindent\underline{Case: \(R_{\widehat{E}}\) is tangent to \(f\) at \(q\):} Let \(\overline{E} \to \widehat{E}\) be the \((2,1)\) weighted blow-up separating \(R_{\widehat{E}}\) and \(f\), with exceptional divisor \(e\). We will use the flag \(\qbarpt\in e\subset\overline{E}\). By abuse of notation, we use \(f\) and \(\sigma\) to denote their strict transforms in \(\overline{E}\).
\begin{detail}
We have
\[ \resizebox{1\textwidth}{!}{ $ \begin{array}{rlll}
    0 \le u \le 1: & P_{\regionA}(u)|_{\overline{E}} = (2-2c+u)f + (2-2c)\sigma + (4-4c + 2u) e & N_{\regionA}(u)|_{\overline{E}}  =  0 \\
    1 \le u \le 3-2c: & P_{\regionB}(u)|_{\overline{E}} = (2-2c+u)f + (3-2c-u)\sigma + (4-4c + 2u) e & N_{\regionB}(u)|_{\overline{E}} = (u-1)\sigma
\end{array} $ }\]
and the intersection theory on \(\overline{E}\) is given by \(f^2 = -2\), \(\sigma^2 = -1\), \(e^2 = -\tfrac{1}{2}\), \(e\cdot f = f\cdot \sigma = 1\), and \(e\cdot\sigma = 0\).
\end{detail}
On region \(\regionA\), the Zariski decomposition of \(P_{\regionA}(u)|_{\overline{E}} - ve\) is
\begin{equation*}\begin{split}P_{\regionA}(u,v) &= \begin{cases} (2-2c+u)f + (2-2c)\sigma + (4-4c + 2u - v) e & 0 \le v \le 2-2c \\ (2-2c)(\sigma + \frac{1}{2}f) + (4-4c+2u-v)(e + \frac{1}{2}f) & 2-2c \le v \le 2-2c+2u \\ (4-4c+2u-v)(e+\sigma+f) & 2-2c+2u \le v \le 4-4c+2u , \end{cases} \\
N_{\regionA}(u,v) &=
\begin{cases} 0 \hphantom{(2-c)(\sigma + f/2) + (4-4c+2u-v)(e + f/2)} & 0 \le v \le 2-2c   \\
\frac{1}{2}(v-2+2c) f & 2-2c \le v \le 2-2c+2u \\ (v-2+2c-2u)\sigma + (v-2+2c-u)f & 2-2c+2u \le v \le 4-4c+2u . \end{cases}\end{split}\end{equation*}
\begin{detail}
So for \(0\leq u\leq 1\) and \(\qbarpt\in e\), we have
\[ \resizebox{1\textwidth}{!}{ $ \begin{array}{r|lll}
   & 0 \le v \le 2-2c & 2-2c \le v \le 2-2c+2u & 2-2c+2u \le v \le 4-4c+2u  \\
       \hline
   P_{\regionA}(u,v)^2   & -2(2-2c+u)^2 - (2-2c)^2 - (4-4c + 2u - v)^2/2 + 2(2-2c+u)(2-2c) + 2(2-2c+u)(4-4c + 2u - v) & -(2-2c)^2/2 + (2-2c)(4-4c+2u-v) & (4-4c+2u-v)^2 /2 \\
   P_{\regionA}(u,v)\cdot e  & (2-2c+u) - (4-4c + 2u - v)/2  & (2-2c)/2 & (4-4c+2u-v)/2 \\
   (P_{\regionA}(u,v)\cdot e)^2  & ((2-2c+u) - (4-4c + 2u - v)/2)^2 & (2-2c)^2/4 & (4-4c+2u-v)^2/4 \\
   \ord_{\qbarpt}((N'_{\regionA}(u)|_{\overline{E}} + N_{\regionA}(u,v))|_e) & 0 &
   \left\{ \begin{array}{ll}
      0  & \qbarpt \notin f  \\
     (v-2+2c)/2   & \qbarpt \in f
   \end{array} \right. & \left\{ \begin{array}{ll}
      0  & \qbarpt \notin f  \\
     v-2+2c-u   & \qbarpt \in f
   \end{array} \right. \\
\end{array} $ }\]
\end{detail}
On region \(\regionB\), the Zariski decomposition of \(P_{\regionB}(u)|_{\overline{E}} - ve\) is
\begin{equation*}\begin{split}P_{\regionB}(u,v) &= \begin{cases} (2-2c+u)f + (3-2c-u)\sigma + (4-4c + 2u-v) e & 0 \le v \le 3-2c-u \\ (3-2c-u)(\sigma + \frac{1}{2}f) + (4-4c+2u-v)(e+\frac{1}{2}f) & 3-2c-u \le v \le 1-2c+3u \\
(4-4c+2u-v)(e+\sigma+f) & 1-2c+3u \le v \le 4-4c+2u , \end{cases} \\
N_{\regionB}(u,v) &= \begin{cases} 0 \hphantom{(3-c-u)(\sigma + f/2) + (4-4c+2u-v)(e+f/2)} & 0 \le v \le 3-2c-u   \\
\frac{1}{2}(v-3+2c+u)f & 3-2c-u \le v \le 1-2c+3u \\
(v-1+2c-3u) \sigma + (v -2 + 2c - u)f & 1-2c+3u \le v \le 4-4c+2u .
\end{cases}\end{split}\end{equation*}
\begin{detail}
So for \(1\leq u\leq 3-2c\) and \(\qbarpt\in e\), we have
\[ \resizebox{1\textwidth}{!}{ $ \begin{array}{r|lll}
   & 0 \le v \le 3-2c-u & 3-2c-u \le v \le 1-2c+3u & 1-2c+3u \le v \le 4-4c+2u \\
       \hline
   P_{\regionB}(u,v)^2   & -2(2-2c+u)^2 -(3-2c-u)^2 - (4-4c + 2u-v)^2/2 + 2(2-2c+u)(3-2c-u) + 2(2-2c+u)(4-4c + 2u-v) & -(3-2c-u)^2/2 + (3-2c-u)(4-4c+2u-v) & (4-4c+2u-v)^2/2 \\
   P_{\regionB}(u,v)\cdot e  & (2-2c+u) - (4-4c + 2u-v)/2  & (3-2c-u)/2 & (4-4c+2u-v)/2 \\
   (P_{\regionB}(u,v)\cdot e)^2  & ((2-2c+u) - (4-4c + 2u-v)/2)^2 & (3-2c-u)^2/4 & (4-4c+2u-v)^2/4 \\
   \ord_{\qbarpt}((N'_{\regionB}(u)|_{\overline{E}} + N_{\regionB}(u,v))|_e) & 0 &
   \left\{ \begin{array}{ll}
      0  & \qbarpt \notin f  \\
     (v-3+2c+u)/2   & \qbarpt \in f
   \end{array} \right. & \left\{ \begin{array}{ll}
      0  & \qbarpt \notin f  \\
     v -2 + 2c - u   & \qbarpt \in f
   \end{array} \right. \\
\end{array} $ } \]
\end{detail}

Since \(A_{\overline{E},\dee_{\overline{E}}}(e) = 3-2c\), we compute using \cite[Theorem 4.8]{Fujita-3.11} that
\[ S_{\overline{E},\dee_{\overline{E}}; V_{\bullet \bullet}}(e) = -\frac{34 c^3 - 142 c^2 + 197 c - 91}{3 (3 - 2 c)^2} , \qquad \frac{A_{\overline{E},\dee_{\overline{E}}}(e)}{S_{\overline{E},\dee_{\overline{E}}; V_{\bullet \bullet}}(e)} = -\frac{3 (3 - 2 c)^3}{34 c^3 - 142 c^2 + 197 c - 91} > 1 \]
for \(c\in(\cnot,1)\). Finally, for \(\qbarpt\in e\), we use \cite[Theorem 4.17]{Fujita-3.11} and compute that
\[A_{e,\dee_e}(\qbarpt) = \begin{cases}
    1, 1-c, \text{ or }\frac{1}{2} & \qbarpt\notin f \\
    1 & \qbarpt\in f ,
\end{cases} \qquad
S_{e,\dee_e; W_{\bullet \bullet \bullet}}(\qbarpt) = \begin{cases}
    \displaystyle \frac{(6 c^2 - 20 c + 17) (1 - c)}{6 (3 - 2 c)^2} & \qbarpt \notin f \\
    \displaystyle -\frac{14 c^3 - 58 c^2 + 80 c - 37}{3 (3 - 2 c)^2} & \qbarpt\in f
\end{cases}\]
so \(A_{e,\dee_e}(\qbarpt)/S_{e,\dee_e; W_{\bullet \bullet \bullet}}(\qbarpt) > 1\) for \(c\in(\cone,1)\).
\end{proof}

\begin{lemma}\label{lem:X3-singular-pt-delta}
Let \(\pt = E\cap S\in\Xthree\) be the singular point, and let \(R\) be as in Theorem~\ref{thm:local-delta-X_3}. Then \(\delta_{\pt}(\Xthree,cR) > 1\) for \(c\in(\cnot,\tfrac{1}{2}]\).
\end{lemma}

\begin{proof}
Let \(F\cong\bb P^1\times\bb P^1\) be the exceptional divisor of the blow-up \(\widehat{W}\to\Xthree\) of \(\pt\). Write \(L = (g^+)^* (\hsmallplus)^*(-K_{\Xthree}-\dee)\). Recall from Section~\ref{sec:small-resolutions-X_3} that \(l_1\) is the ruling of \(F\) contracted by \(\widehat{W}\to W\) and that \(l_2\) is contracted by \(\widehat{W}\to W^+\).
First, we find the birational transformations associated to \(L-uF\).
At \(u=1\), we contract \(\widehat{S}\cong\bb F_1\) onto its negative section.
At \(u=2-2c\), we contract \(\widehat{E}\cong\bb F_1\) onto its negative section.
Finally, at \(u=5-4c\), the contraction of \(l_1\) is not birational.
Thus, we find that the Nakayama--Zariski decomposition of \(L-uF\) is
\[ \resizebox{1\textwidth}{!}{ $ \begin{array}{rll}
    0 \le u \le 1: & P_{\regionA}(u) = L-uF & N_{\regionA}(u) = 0 \\
    1 \le u \le 2-2c: & P_{\regionB}(u) = L-uF-\frac{u-1}{2}\widehat{S} & N_{\regionB}(u) = \frac{u-1}{2}\widehat{S} \\
    2-2c \le u \le 5-4c: & P_{\regionC}(u) = L-uF- \frac{u-1}{2}\widehat{S} - (u-2+2c)\widehat{E} & N_{\regionC}(u) = \frac{u-1}{2}\widehat{S} + (u-2+2c)\widehat{E}.
\end{array} $ } \]
Recall that
\(L =  (g^+)^*\piP^*\cal O_{\bb P(1,1,1,2)}(5-4c) - \frac{1}{2} \widehat{S} - (2-2c) \widehat{E}\).
Using that
\(\widehat{S}|_F = \widehat{E}|_F = l_2\), \((g^+)^*\piP^*\cal O_{\bb P(1,1,1,2)}(5-4c)|_F = \frac{5-4c}{2} l_2\), and \(F|_F = -l_1-l_2\), we compute that the volumes on the regions for \(u\) are
\[ \begin{array}{rl}
    0 \le u \le 1: & \vol(P_{\regionA}(u)) = 3(2-2c)(3-2c)^2 - 2 u^3 \\
    1 \le u \le 2-2c: & \vol(P_{\regionB}(u)) = \tfrac{(5-4c)^3}{2} - \tfrac{u^3}{2} - (2-2c)^3 + 3u^2(\tfrac{u-1}{2}) - 2 u^3 \\
    2-2c \le u \le 5-4c: & \vol(P_{\regionC}(u)) = \tfrac{(5-4c)^3}{2} - \tfrac{u^3}{2} - u^3 - 3u^2(\tfrac{-4 c - 3 u + 5}{2}) - 2 u^3
\end{array} \]
and therefore
\[S_{\Xthree,\dee}(Y) = -\frac{2 (14 c^3 - 58 c^2 + 80 c - 37)}{3 (3 - 2 c)^2}.\]
Since \(A_{\Xthree,\dee}(Y) = 2\), we find that \(A_{\Xthree,\dee}(Y)/S_{\Xthree,\dee}(Y)>1\) for \(c\in(\cone, \tfrac{1}{2}]\).

Next, we apply the Abban--Zhuang method to the flag \(q\in l_1\subset F\). From above, we have
\[ \begin{array}{rll}
    0 \le u \le 1: & P_{\regionA}(u)|_{F} = u l_1 + u l_2 & N_{\regionA}(u)|_{F} = 0 \\
    1 \le u \le 2-2c: & P_{\regionB}(u)|_{F} = u l_1 + \frac{u+1}{2} l_2 & N_{\regionB}(u)|_{F} = \frac{u-1}{2}\widehat{S}|_{F} \\
    2-2c \le u \le 5-4c: & P_{\regionC}(u)|_{F} = u l_1 + \frac{5-4c-u}{2} l_2 & N_{\regionC}(u)|_{F} = \frac{u-1}{2}\widehat{S}|_{F} + (u-2+2c)\widehat{E}|_{F}
\end{array} \]
so on each region, the pseudo-effective threshold of \(P(u)|_{F} - v l_1\) is \(v=u\).
\begin{detail}
For \(0 \le v \le u\) and \(q\in F\), the relevant quantities are
\[\resizebox{1\textwidth}{!}{ $ \begin{array}{r|lll}
   & 0 \le u \le 1 & 1 \le u \le 2-2c & 2-2c \le u \le 5-4c \\
       \hline
   P(u,v)^2   & 2u(u-v) & (u+1)(u-v) & (5-4c-u)(u-v) \\
   P(u,v)\cdot l_1  & u & (u+1)/2 & (5-4c-u)/2 \\
   (P(u,v)\cdot l_1)^2  & u^2 & (u+1)^2/4 & (5-4c-u)^2/4 \\
   \ord_{q}((N'(u)|_F + N(u,v))|_{l_1}) & 0 &
   \left\{ \begin{array}{ll}
      0  & q \notin \widehat{S}|_{F}  \\
     (u-1)/2   & q\in \widehat{S}|_{F}
   \end{array} \right. &
   \left\{ \begin{array}{ll}
      0  & q \notin \widehat{S}|_{F}, \widehat{E}|_{F}  \\
     (u-1)/2   & q\in \widehat{S}|_{F} \\
     u-2+2c   & q\in \widehat{E}|_{F}
   \end{array} \right. \\
\end{array} $ }\]
\end{detail}

The log discrepancy is \(A_{F,\dee_{F}}(l_1) = 1\), and using \cite[Theorem 4.8]{Fujita-3.11} we compute that
\[ S_{F,\dee_{F}; V_{\bullet \bullet}}(l_1) = \frac{-14 c^3 + 58 c^2 - 80 c + 37}{3 (3 - 2 c)^2}, \]
so \(A_{F,\dee_{F}}(l_1)/S_{F,\dee_{F}; V_{\bullet \bullet}}(l_1) > 1\) for \(c\in(\cone,\tfrac{1}{2}]\). Finally, using \cite[Theorem 4.17]{Fujita-3.11}, we compute
\[ S_{l_1,\dee_{l_1}; W_{\bullet \bullet \bullet}}(q) = \begin{cases}
    \displaystyle \frac{32 c^4 - 176 c^3 + 360 c^2 - 324 c + 107}{12 (2 - 2c) (3 - 2 c)^2} & q \notin \widehat{S}|_{F}, \widehat{E}|_{F} \\
    \displaystyle \frac{(1 - c) (18 c^2 - 52 c + 37)}{3 (3 - 2 c)^2} & q \in \widehat{S}|_{F} \\
    \displaystyle \frac{-10 c^3 + 46 c^2 - 71 c + 37}{3 (3 - 2 c)^2} & q \in \widehat{E}|_{F} .
\end{cases}\]
Since \(A_{l_1,\dee_{l_1}}(q) = 1\), we have \(A_{l_1,\dee_{l_1}}(q)/S_{l_1,\dee_{l_1}; W_{\bullet \bullet \bullet}}(q) > 1\) for \(c\in(\cnot,\tfrac{1}{2}]\). 
\end{proof}

\subsection{\texorpdfstring{\(A_1\)}{A1} singularities on \texorpdfstring{\(R \subset \Xthree\)}{R in Xn}}\label{sec:A_1-sings-X_3}
In this section, we prove Theorem~\ref{thm:local-delta-X_3}\eqref{item:A_1-sings-X_3}.

Let \(\pt\in R\setminus E\) be an \(A_1\) singularity. To bound \(\delta_{\pt}(\Xthree,\dee)\), we will use the exceptional divisor of the blow-up of \(\pt\).
We will work the small resolution \(\hsmallplus\colon W^+ \to \Xthree\) that contracts \(l_1\). Let \(X_1 \to W^+\) be the blow-up of the preimage of \(\pt\) in \(W^+\), and let \(Y\subset X_1\) be the exceptional divisor. Let \(\phi_1\colon X_1 \to \Xthree\) denote the composition of these two morphisms, and denote \(L = \phi_1^*(-K_{\Xthree} - c R)\).

\subsubsection{Computation of \(A_{\Xthree,\dee}(Y)/S_{\Xthree,\dee}(Y)\)}\label{sec:X_3-local-delta-A_n:A_1-finite-part-1}
To compute \(S_{\Xthree,\dee}(Y)\), we first find the birational transformations associated to \(L-u Y\).
Let \(l_{\pt}\) be the ruling of \(\bb P(1,1,1,2)\) containing the image of \(\pt\), and let \(\Gamma\subset\bb P(1,1,1,2)\) be the unique divisor in \(|\cal O_{\bb P(1,1,1,2)}(1)|\) containing the images of \(\ptE\) and \(\pt\).
Using \(\Gamma\), we compute that the pseudoeffective threshold of \(L-u Y\) is \(u=5-4c\).
\begin{detail}
Indeed, denoting \(\bb P \coloneqq \bb P(1,1,1,2)\), we have \(-K_{\bb P}-\dee_{\bb P} \sim (5-4c) \Gamma\), and the pullback of \(\Gamma\) to \(X_1\) contains \(Y\) with multiplicity 1.
\end{detail}

By Lemma~\ref{lem:moriconeforblowupofX3} and the geometry of the surfaces $E^+, S^+,$ and $\Gamma$, we find that the Mori cone of $X_1$ is generated by $l_p$, $l_1$, $l_E$, $l_S$, and $L_Y$, where $l_Y$ is a line in the exceptional divisor $Y$.  We then compute the birational transformations giving the Nakayama--Zariski decomposition.  At \(u = 2-2c\), we have the Atiyah flop \(X_1 \dashrightarrow X_1^+\) of \(l_{\pt}\), which factors as the blow-up \(f_1\colon \hX_1 \to X_1\) of \(l_{\pt}\) with exceptional divisor \(F \cong \bb P^1\times\bb P^1\), followed by the contraction \(f_1^+\colon \hX_1 \to X_1^+\) of the other ruling of \(F\). Let \(S^+\cong\bb P^2\) the strict transform on \(W^+\) of the exceptional divisor over \(\conept\), as defined in Section~\ref{sec:mori-cone-small-resol-Xthree}. The strict transforms \(Y_1^+, S_1^+\) of \(Y, S^+\) in \(X_1^+\) are each isomorphic to \(\bb F_1\), and \(l_{\pt}\) is glued into each as the negative section.
\begin{detail}
This is because \((L-uY)\cdot l_{\pt} = 2-2c - u\) and \(S^+ \cdot l_{\pt} = Y\cdot l_{\pt} = 1\).
\end{detail}
At \(u=3-2c\), we contract the strict transforms of \(S^+\) and \(\Gamma\) in a morphism \(f_4\colon X_1^+ \to X_4\).
We will use subscripts \(\regionA,\regionB,\regionC\) to denote values of \(u\) in each of these regions.

Thus, the Nakayama--Zariski decomposition of \(f_1^*(L-uY)\) is on \(\hX_1\) is
\[ \resizebox{1\textwidth}{!}{ $ \begin{array}{rll}
    0 \le u \le 2-2c: & P_{\regionA}(u) = -K_{\hX_1} - \widehat{\dee}_1 + (2-2c-u)\hY_1 + F & N_{\regionA}(u)  =  0 \\
    2-2c \le u \le 3-2c: & P_{\regionB}(u) = -K_{\hX_1} - \widehat{\dee}_1 + (2-2c-u)\hY_1 + (3-2c-u) F & N_{\regionB}(u) = (u-2+2c)F \\
    3-2c \le u \le 5-4c: & P_{\regionC}(u) = -K_{\hX_1} - \widehat{\dee}_1 + (2-2c-u) \hY_1 + (3-2c-u) (\hS_1 + \widehat{\Gamma}_1) + (5-4c-2u) F & N_{\regionC}(u)  = (u-2+2c)F + (u-3+2c)(\hS_1 + \widehat{\Gamma}_1 + F)
\end{array} $} \]
where \(\widehat{\dee}_1, \hY_1, \hS_1, \widehat{\Gamma}_1\) (resp. \(\dee_1^+, Y_1^+\)) are the strict transforms of the corresponding divisors on \(\hX_1\) (resp. \(X_1^+\)).
\begin{detail}
In detail, for \(0\leq u \leq 2-2c\), we have
\[P_{\regionA}(u) = f_1^*(-K_{X_1} + \dee_1 + (2-2c-u) Y ) = -K_{\hX_1} - \widehat{\dee}_1 + (2-2c-u)\hY_1 + F\]
and \(N_{\regionA}(u) = 0\).
For \(2-2c \leq u \leq 3-2c\),
\[ \resizebox{1\textwidth}{!}{ $
P_{\regionB}(u) = (f_1^+)^*(-K_{X_1^+} - \dee_1^+ + (2-2c-u)Y_1^+) = -K_{\hX_1} + F - \widehat{\dee}_1 + (2-2c-u)\hY_1 + (2-2c-u) F $} \]
and \(N_{\regionB}(u) = (2-2c-u) F\).
Finally, for \(3-2c \leq u \leq 5-4c\),
\begin{equation*}\begin{split}
P_{\regionC}(u) &= (f_1^+)^*(f_4^*(-K_{X_4} - D_4 + (2-2c-u)Y_4)) \\
&= (f_1^+)^*(-K_{X_1^+} - \dee_1^+ + (2-2c-u)Y_1^+ + (3-2c-u) (S_1^+ + \Gamma_1^+)) \\
&= -K_{\hX_1} + F - \widehat{\dee}_1 + (2-2c-u)(\hY_1 + F) + (3-2c-u) (\hS_1 + \widehat{\Gamma}_1 + F)
\end{split}\end{equation*}
and \(N_{\regionC}(u) = (u-2+2c)F + (u-3+2c)(\hS_1 + \widehat{\Gamma}_1 + F)\).
\end{detail}
\[\begin{tikzcd}
& & F \arrow[hook, r] & \hX_1 \arrow[ld, "f_1" {swap}] \arrow[rd, "f_1^+"] & & & & & & \\
\Xthree & W^+ \arrow[l, "\hsmallplus" {swap}] & X_1 \arrow[l, "g_0" {swap}] \arrow[rr, dashed, "\text{Atiyah flop of }l_{\pt}"] & & X_1^+ \arrow[rrr, "\text{contract }S_1^+ \cong \bb F_1"] & & & X_4' \arrow[rrr, "\text{contract }\Gamma_4' \cong \bb F_1"] & & & X_4
\end{tikzcd}\]
We next compute the volume of \(L-uY\) on each of these regions. On region \(\regionA\), this is straightforward using the description of \(L = (\hsmallplus)^*(-K_{\Xthree}-\dee)\) from \(\bb P(1,1,1,2)\): 
\[ L = \piP^*(-K_{\bb P} - D_{\bb P}) - \tfrac{1}{2} S^+- (2-2c) E^+. \] For \(\regionB\), we apply Lemma~\ref{lem:volume-lemma-4}. On \(\regionC\), we compute the volume in two parts. Write \(f_4\colon X_1^+\to X_4\) as the composition of the contraction \(X_1^+\to X_4'\) of \(S_1^+ \cong \bb F_1\), followed by the contraction of the strict transform \(\Gamma_4'\cong\bb F_1\) of \(\Gamma\). For each of these contractions, we apply Lemma~\ref{lem:volume-lemma-3} with \(n=1\).
\begin{detail}
We give more details for the computation. First, on \(S_1^+\cong\bb F_1\), let \(f_S\) denote the class of a fiber and \(\sigma_S\) the negative section. Applying Lemma~\ref{lem:volume-lemma-3} to \(P_1 \coloneqq -K_{X_1^+}-D_{X_1^+} + (2-2c-u) Y_1^+\) and computing that \(P_1|_{S_1^+} = (3-2c-u) \sigma_S + f_S\) and \(S_1^+|_{S_1^+} = -\sigma_S - 2f_S\), we find that \[\vol(P_1 + (3-2c-u) S_1^+) = P_{\regionB}(u)^3 + 3 (-2 c - u + 3)^2 (2 c + u - 2).\] Next, the strict transform of \(\Gamma\) in \(X_4'\) is isomorphic to \(\bb F_1\), with fiber \(l_1\) and negative section \(e_4 \coloneqq E_4'|_{\Gamma_4'}\). Write \(P_1' = -K_{X_4'} - D_4' + (2-2c-u)Y_4'\); this pulls back to \(P_1 + (3-2c-u) S_1^+\) on \(X_4'\), so \(\vol(P_1')=\vol(P_1 + (3-2c-u) S_1^+)\). By computation, \(P_1'|_{\Gamma_4'} = (3-2c-u)e + (5-4c-u)l_1\) and \(\Gamma_4'|_{\Gamma_4'} = -e_4\), so applying Lemma~\ref{lem:volume-lemma-3} again yields
\[\vol(P_{\regionC}(u)) = \vol(P_1'+ (3-2c-u)\Gamma_4')^3) = \vol(P_1')+ 2 (-2 c - u + 3)^2 (-5 c - u + 6).\]
\end{detail}
We find that:
\[ \resizebox{1\textwidth}{!}{ $ \begin{array}{rl}
    0 \le u \le 2-2c: & \vol(P_{\regionA}(u)) = 3(2-2c)(3-2c)^2 - u^3 \\
    2-2c \le u \le 3-2c: & \vol(P_{\regionB}(u)) = 3(2-2c)(3-2c)^2 - u^3 + (u-2+2c)^3 \\
    3-2c \le u \le 5-4c: & \vol(P_{\regionC}(u)) = 3(2-2c)(3-2c)^2 - u^3 + (u-2+2c)^3 + (-2 c - u + 3)^2 (-4 c + u + 6).
\end{array} $} \]
These volumes are the same as those in Section~\ref{sec:local-delta-A_n:A_1-finite-part-1}, and \(A_{\Xthree,\dee}(Y) = 3-2c\), so we have
\[ S_{\Xthree,\dee}(Y) = \frac{9 - 7 c}{3}, \qquad \frac{A_{\Xthree,\dee}(Y)}{S_{\Xthree,\dee}(Y)} = \frac{9 - 6 c}{9 - 7 c} > 1 \text{ for }c\in (0,1). \]

\subsubsection{Computation of \(\delta_q(Y, \dee_Y; V_{\bullet \bullet})\) for \(q\in Y\)}
Letting \(R|_Y\) denote the intersection of \(Y\) with the strict transform of \(R\) in \(X_1\), our assumptions imply that \(R|_Y\) is a smooth conic and does not contain the point \(l_{\pt} \cap Y\).
We have several cases, mirroring those in Section~\ref{sec:local-delta-A_n:A_1-finite-part-2} for \(A_1\) singularities of \((2,2)\)-surfaces in \(\bb P^1\times\bb P^2\). These cases use the following flags \(q\in C\subset Y\):
\begin{enumerate}
    \item\label{item:X3-A_1-case-1} Case 1: If $q \notin R|_Y$, then \(C\) will be the line connecting \(q\) and \(l_{\pt} \cap Y\).
    \item\label{item:X3-A_1-case-2} Case 2: $q \in R|_Y$ and the line \(C\) connecting $q$ to $l_{\pt} \cap Y$ is not tangent to $R|_Y$ at $q$.
    \item\label{item:X3-A_1-case-3} Case 3: $q \in R|_Y$ and the line connecting $q$ to $l_{\pt} \cap Y$ is tangent to $R|_Y$ at $q$.

    We use the flag coming from the \((2,1)\) weighted blow-up of \(q\), which separates \(R|_Y\) and its tangent line at \(q\), and we take \(C\) to be the exceptional divisor.
\end{enumerate}

We will work on the strict transform \(\hY_1\) of \(Y\) in \(\hX_1\), which is the blow-up of \(Y\cong\bb P^2\) at the point \(l_{\pt}\cap Y\). Then \(\hY_1\cong\bb F_1\) with negative section \(e\coloneqq F|_{\hY_1}\), \(f_\Gamma\coloneqq \widehat{\Gamma}_1|_{\hY_1}\) is a fiber, and \(\hS_1|_{\hY_1} = 0\). Thus, using the formulas for \(P(u)\) on each region \(\regionA,\regionB,\regionC\) computed above in Section~\ref{sec:X_3-local-delta-A_n:A_1-finite-part-1}, we find
\[ \resizebox{1\textwidth}{!}{ $ \begin{array}{rll}
    0 \le u \le 2-2c: & P_{\regionA}(u)|_{\hY_1} = u e + u f & N_{\regionA}(u)|_{\hY_1}  =  0 \\
    2-2c \le u \le 3-2c: & P_{\regionB}(u)|_{\hY_1} = (2-2c)e + uf & N_{\regionB}(u)|_{\hY_1} = (u-2+2c)e \\
    3-2c \le u \le 5-4c: & P_{\regionC}(u)|_{\hY_1} = (5-4c-u)e + (3-2c)f & N_{\regionC}(u)|_{\hY_1}  = (u-2+2c)e + (u-3+2c)(e+f_\Gamma) .
\end{array} $} \]
\begin{detail}
Indeed, \(( -K_{\hX_1} - \widehat{\dee}_1 + (2-2c-u)\hY_1 + F)|_{\hY_1}\) is the pullback of \((\hsmallplus)^*(-K_{\Xthree}-\dee)-uY\) from \(X_1\). On \(Y\cong\bb P^2\), recall from the computation of \(\vol(P_{\regionA}(u))\) that we have \((\hsmallplus)^*(-K_{\Xthree}-\dee)|_Y \sim 0\) and \(Y|_Y \in |\cal O(-1)|\).
\end{detail}
These are the same as the formulas for \(P(u)|_\hY, N_{\regionA}(u)|_\hY, N_{\regionB}(u)|_\hY\) from Section~\ref{sec:local-delta-A_n:A_1-finite-part-2}.

\noindent\underline{Cases~\ref{item:X3-A_1-case-1} and~\ref{item:X3-A_1-case-2}:}
Let \(f\subset \hY_1\) be the strict transform of the line \(C\subset Y\). The computations on regions \(\regionA\) and \(\regionB\) are unchanged from those in Section~\ref{sec:local-delta-A_n:A_1-finite-part-2}. We have two cases, depending on whether \(f\) is the distinguished fiber in the support of \(N_{\regionC}(u)\).

\underline{Subcase \(f\neq f_\Gamma\):}
If \(f \neq f_\Gamma\), then the computation of \(A_{Y,\dee_Y} (C)/ S_{Y,\dee_Y; W_{\bullet \bullet}}(C)\) is identical to that of \(A_{Y,\dee_Y} (C)/ S_{Y,\dee_Y; V_{\bullet \bullet}}(C)\) in Cases~\ref{item:A_1-case-1} and~\ref{item:A_1-case-2-notin-lF} of Section~\ref{sec:local-delta-A_n:A_1-finite-part-2}, so this quantity is \(>1\) for \(c\in (0,1)\). For \(W_{\bullet \bullet \bullet}\), the relevant values on regions \(\regionA,\regionB\) are the same as in Section~\ref{sec:local-delta-A_n:A_1-finite-part-2}. On region \(\regionC\),
\[ \begin{array}{r|ll}
   & 0 \le v \le u-2+2c & u-2+2c \le v \le 3-2c \\
       \hline
   \ord_{\qhatpt}((N_{\regionC}'(u)|_{\hY_1} + N_{\regionC}(u,v) - ve)|_\hC) & 
   \left\{ \begin{array}{ll}
      0  & \qhatpt \notin e  \\
      2u-5+4c-v   & \qhatpt \in e
   \end{array} \right. & \left\{ \begin{array}{ll}
      0  & \qhatpt \notin e  \\
      u+2c-3   & \qhatpt \in e
   \end{array} \right. 
\end{array} \]
for \(\qhatpt\in f\). The case when \(\qhatpt\notin e\) is therefore the same as the \(\qhatpt \notin e \cup l_\tF\) case in Section~\ref{sec:local-delta-A_n:A_1-finite-part-2}. If \(\qhatpt\in e\), then we are in Case~\ref{item:X3-A_1-case-1} by assumption, and we compute
\[ S_{C,\dee_C; W_{\bullet \bullet \bullet}}(q) = \begin{cases}
    \displaystyle \frac{(1-c)(6c^2-20c+17)}{3(3-2c)^2} & \qhatpt \notin e \text{ and } C \neq f_1(f_\Gamma) \\
    \displaystyle \frac{-14 c^3 + 58 c^2 - 80 c + 37}{3 (3 - 2 c)^2} & q = l_{\pt} \cap Y \notin R|_Y.
\end{cases}\]
For \(q \neq l_{\pt} \cap Y\) and \(C \neq f_1(f_\Gamma)\), since we are in Cases~\ref{item:X3-A_1-case-1} and~\ref{item:X3-A_1-case-2} we have \(A_{C, \dee_C}(q) = 1\) or \(1-c\).
If \(q = l_{\pt} \cap Y \notin R|_Y\), then \(A_{C, \dee_C}(q) = 1\). So in both of these cases, for \(f \neq f_\Gamma\) we have \(A_{C, \dee_C}(q) / S_{C,\dee_C; W_{\bullet \bullet \bullet}}(q) > 1 \) for \(c \in (\cone,1)\), where \(\cone\approx 0.3293\) is the smallest root of the cubic polynomial \(7 c^3 - 23 c^2 + 22 c - 5\).

\underline{Subcase \(f = f_\Gamma\):} If \(f = f_\Gamma\) then, recalling that \(d(u) = \ord_f (N(u)|_{\hY_1})\), we have
\[ \begin{array}{rll}
    0 \le u \le 2-2c: & N_{\regionA}'(u)|_{\hY_1} =  0  & \quad d(u)= 0\\
    2-2c \le u \le 3-2c: & N_{\regionB}'(u)|_{\hY_1} = (u-2+2c)e & \quad d(u) = 0 \\
    3-2c \le u \le 5-4c: & N_{\regionC}'(u)|_{\hY_1}  = (2u-5+4c)e & \quad d(u) = u-3+2c .
\end{array} \]
The computations on regions \(\regionA,\regionB\) are unchanged from Cases~\ref{item:A_1-case-1} and~\ref{item:A_1-case-2-notin-lF} of Section~\ref{sec:local-delta-A_n:A_1-finite-part-2}. For \(\regionC\),
\[ \resizebox{1\textwidth}{!}{ $\begin{array}{r|ll}
   & 0 \le v \le u-2+2c & u-2+2c \le v \le 3-2c \\
       \hline
   \ord_{\qhatpt}((N_{\regionC}'(u)|_{\hY_1} + N_{\regionC}(u,v) - (d(u)+v)e)|_\hC) & 
   \left\{ \begin{array}{ll}
      0  & \qhatpt \notin e  \\
      u+2c-2-v   & \qhatpt \in e
   \end{array} \right. & 0
\end{array}$ } \]
which is the same as Cases~\ref{item:A_1-case-1} and~\ref{item:A_1-case-2-notin-lF} in Section~\ref{sec:local-delta-A_n:A_1-finite-part-2}. Therefore, the computation of \(A_{C, \dee_C}(\qhatpt)/S_{C,\dee_C; W_{\bullet \bullet \bullet}}(\qhatpt)\) is unchanged from Section~\ref{sec:local-delta-A_n:A_1-finite-part-2}, and this quantity is \(>1\) for \(c\in (0,1)\).

For \(W_{\bullet \bullet}\), we have an additional term on region \(\regionC\) in the computation of \(S_{Y,\dee_Y; W_{\bullet \bullet}}(C)\). Since \(A_{Y, \dee_Y}(C) = 1\), we compute that
\[S_{Y,\dee_Y; W_{\bullet \bullet}}(C) = \frac{-(14 c^3 - 58 c^2 + 80 c - 37)}{3 (3 - 2 c)^2}, \qquad \frac{A_{Y,\dee_Y} (C)}{S_{Y,\dee_Y; W_{\bullet \bullet}}(C)} > 1 \text{ for }c\in (\cone,1).\]
\begin{detail}
In more detail, \(S_{Y,\dee_Y; W_{\bullet \bullet}}(C)\) is computed by
\[ \resizebox{1\textwidth}{!}{ $\begin{array}{rl}
S_{Y,\dee_Y; W_{\bullet \bullet}}(C) &= \displaystyle \frac{3}{\vol L} \int_0^{\tau^+} \left( (P(u)|_{\hY_1})^2\cdot d(u) + \int_0^\infty \vol( P(u)|_{\hY_1} - vf_\Gamma) \; dv \right) du \\
&= \displaystyle \frac{3}{3(2-2c)(3-2c)^2} \left(\int_{3-2c}^{5-4c} (-(5-4c-u)^2 + 2(5-4c-u)(3-2c))(u-3+2c) \; du\right) + \frac{3-2c}{3} \\
&= \displaystyle \frac{-(2 (3 c - 5) (c - 1)^2)}{3 (3 - 2 c)^2} + \frac{3-2c}{3} = \frac{-(14 c^3 - 58 c^2 + 80 c - 37)}{3 (3 - 2 c)^2}.
\end{array}$ } \]
Note that \((1-c)/S_{Y,\dee_Y; W_{\bullet \bullet}}(C) < 1\).
\end{detail}

\noindent\underline{Case~\ref{item:X3-A_1-case-3}:} If the tangent line to \(R|_Y\) at \(q\) is not the image \(f_1(f_\Gamma)\) of the distinguished fiber \(f_\Gamma\), then the computation is identical to Case~\ref{item:A_1-case-3} in Section~\ref{sec:local-delta-A_n:A_1-finite-part-2}.

Suppose therefore that \(f_1(f_\Gamma)\) is the tangent line to \(R|_Y\) at \(q\), and let \(\beta\colon \oY \to \hY_1\) be the \((2,1)\) weighted blow-up of \(f_1^{-1}(q)\) separating \(R_{\hY_1}\) and its tangent line \(f_\Gamma\) at \(f_1^{-1}(q)\). Let \(z\) be the exceptional divisor of \(\beta\). By abuse of notation, we let \(e, f, f_\Gamma\) denote their strict transforms in \(\oY\). Using the flag \(z\subset \oY\), so that \(d(u) = \ord_z (N(u)|_{\oY})\), we have
\[ \begin{array}{rll}
    0 \le u \le 2-2c: & N_{\regionA}'(u)|_{\oY} =  0  & \quad d(u)= 0\\
    2-2c \le u \le 3-2c: & N_{\regionB}'(u)|_{\oY} = (u-2+2c)e & \quad d(u) = 0 \\
    3-2c \le u \le 5-4c: & N_{\regionC}'(u)|_{\oY}  = (2u-5+4c)e + (u-3+2c) f_\Gamma  & \quad d(u) = 2(u-3+2c)
\end{array} \]
The computations on regions \(\regionA,\regionB\) are unchanged from Section~\ref{sec:local-delta-A_n:A_1-finite-part-2}. For \(\regionC\), for \(\qbarpt\in z\) (which is disjoint from \(e\)) we have
\[ \resizebox{1\textwidth}{!}{ $\begin{array}{r|lll}
   & 0 \le v \le 5-4c-u  & 5-4c-u \le v \le u+1  & u+1 \le v \le 6-4c \\
       \hline
   \ord_{\qbarpt}((N_{\regionC}'(u)|_{\oY} + N_{\regionC}(u,v))|_z) & 
   \left\{ \begin{array}{ll}
      0  & \qbarpt \notin f_\Gamma  \\
      u-3+2c   & \qbarpt = f_\Gamma \cap z
   \end{array} \right. &
   \left\{ \begin{array}{ll}
      0  & \qbarpt \notin f_\Gamma  \\
      (8 c + 3 u + v - 11)/2  & \qbarpt = f_\Gamma \cap z
   \end{array} \right. &
   \left\{ \begin{array}{ll}
      0  & \qbarpt \notin f_\Gamma  \\
      u + v + 4 c - 6  & \qbarpt = f_\Gamma \cap z
   \end{array} \right.
\end{array}$ } \]

For \(W_{\bullet \bullet}\), we have an additional term on region \(\regionC\) in the computation of \(S_{\oY,\dee_{\oY}; W_{\bullet \bullet}}(z)\). Using the \(q\notin l_\tF\) case from Case~\ref{item:A_1-case-3} of Section~\ref{sec:local-delta-A_n:A_1-finite-part-2}, since \(A_{\oY,\dee_{\oY}} (z) = 3-2c\) we find that
\[S_{\oY,\dee_{\oY}; W_{\bullet \bullet}}(z) = \frac{-34 c^3 + 142 c^2 - 197 c + 91}{3 (3 - 2 c)^2} , \qquad \frac{A_{\oY,\dee_{\oY}} (z)}{S_{\oY,\dee_{\oY}; W_{\bullet \bullet}}(z)} > 1 \text{ for }c\in (\cnot,1). \]
\begin{detail}
In more detail, \(S_{Y,\dee_Y; W_{\bullet \bullet}}(C)\) is computed by
\[ \resizebox{1\textwidth}{!}{ $\begin{array}{rl}
S_{\oY,\dee_{\oY}; W_{\bullet \bullet}}(z) &= \displaystyle \frac{3}{\vol L} \int_0^{\tau^+} \left( (P(u)|_{\oY})^2\cdot d(u) + \int_0^\infty \vol( P(u)|_{\oY} - v z) \; dv \right) du \\
&= \displaystyle \frac{3}{3(2-2c)(3-2c)^2} \left(\int_{3-2c}^{5-4c} ((5-4c-u)e + (3-2c)f + (6-4c)z)^2\cdot 2(u-3+2c) \; du\right) -\frac{22 c^3 - 98 c^2 + 145 c - 71}{3 (3 - 2 c)^2} \\
&= \displaystyle -\frac{4 (3 c - 5) (c - 1)^2}{3 (3 - 2 c)^2} - \frac{22 c^3 - 98 c^2 + 145 c - 71}{3 (3 - 2 c)^2} \\
&= \displaystyle \frac{-34 c^3 + 142 c^2 - 197 c + 91}{3 (3 - 2 c)^2}
\end{array}$ } \]
so
\[\frac{A_{\oY,\dee_{\oY}} (z)}{S_{\oY,\dee_{\oY}; W_{\bullet \bullet}}(z)} = \frac{3 (2 c - 3)^3}{c (2 c (17 c - 71) + 197) - 91}.\]
\end{detail}

For \(W_{\bullet \bullet \bullet}\), the computation is unchanged from Case~\ref{item:A_1-case-3} of Section~\ref{sec:local-delta-A_n:A_1-finite-part-2} if \(\qbarpt \in z \setminus f_\Gamma \cap z\).
For \(\qbarpt = f_\Gamma \cap z\), we compute using \cite[Theorem 4.17]{Fujita-3.11} and the above (together with the formulas from Section~\ref{sec:local-delta-A_n:A_1-finite-part-2} for \(\regionA\), \(\regionB\), and \(P(u,v)\cdot z\) on \(\regionC\)) that \(F_{\qbarpt}(W_{\bullet \bullet \bullet}) = (-22 c^3 + 90 c^2 - 123 c + 57)/(6 (3 - 2 c)^2)\). Since \(A_{z, \dee_z}(\qbarpt) = 1\), we have
\[ S_{z,\dee_z; W_{\bullet \bullet \bullet}}(\qbarpt) = \frac{-14 c^3 + 58 c^2 - 80 c + 37}{3 (3 - 2 c)^2} , \qquad
\frac{A_{z, \dee_z}(\qbarpt)}{S_{z,\dee_z; W_{\bullet \bullet \bullet}}(\qbarpt)} > 1 \text{ for }c\in (\cone, 1).\]
\begin{detail}
In more detail,
\begin{equation*}\begin{split}
F_{\qbarpt}(W_{\bullet \bullet \bullet}) &= \frac{6}{\vol L} 
\int_0^{5-4c} \left( \int_{0}^{t(u)} (P(u,v)\cdot z)\ord_{\qbarpt}((N'(u)|_{\oY} + N(u,v))|_z dv) \right) du \\
&= \frac{6}{3(2-2c)(3-2c)^2} \left( \frac{(-1 + c)^4}{3} + \frac{9 - 20 c + 15 c^2 - 4 c^3}{6} + \frac{(-1 + c)^2 (23 - 30 c + 10 c^2)}{3} \right) \\
&= \frac{-22 c^3 + 90 c^2 - 123 c + 57}{6 (3 - 2 c)^2}
\end{split}\end{equation*}
so, using the \(S_{z,\dee_z; W_{\bullet \bullet \bullet}}(\qbarpt)\) computation from the \(\qbarpt \neq f \cap z, l_\tF \cap z\) case in Section~\ref{sec:local-delta-A_n:A_1-finite-part-2},
\begin{equation*}\begin{split}
S_{z,\dee_z; W_{\bullet \bullet \bullet}}(\qbarpt) &= \frac{(1-c)(6c^2-20c+17)}{6(3-2c)^2} + \frac{-22 c^3 + 90 c^2 - 123 c + 57}{6 (3 - 2 c)^2} \\
&= \frac{-14 c^3 + 58 c^2 - 80 c + 37}{3 (3 - 2 c)^2}.
\end{split}\end{equation*}
\end{detail}
Therefore, by the Abban--Zhuang method, we have shown that \(\delta_\pt(\Xthree,\dee) > 1\) for \(c\in(\cnot,1)\).

\subsection{\texorpdfstring{\(A_2\)}{A2} singularities on \texorpdfstring{\(R \subset \Xthree\)}{R in Xn}}\label{sec:A_2-sings-X_3}
In this section, we prove Theorem~\ref{thm:local-delta-X_3}\eqref{item:A_2-sings-X_3}.

As in Section~\ref{sec:A_1-sings-X_3}, let \(X_1 \to W^+\) be the blow-up of the strict transform of \(\pt\), with exceptional divisor \(Y\cong\bb P^2\). The intersection \(R|_Y\) is a rank \(\leq 2\) conic. Let \(\tX_2 \to X_1\) be the blow-up of a (reduced) irreducible component of \(R|_Y\), with exceptional divisor \(\widetilde{Z}_2\cong\bb F_2\), and let \(\tX_2 \to \tX_1\) be the contraction of the strict transform of \(Y\); note that \(\tX_1\) is singular at the image of \(Y\). Let \(Z\cong\bb P(1,1,2)\) be the image of \(\widetilde{Z}_2\), and let \(g_0\colon \tX_1 \to W^+\) be the induced morphism. We will use the plt-type divisor \(Z\).

Let \(L=g_0^*(\hsmallplus)^*(-K_{\Xthree}-\dee)\).
Before computing the volume of \(L-uZ\), we define several curves.
Let \(l_{\pt}\) be the ruling of \(\bb P(1,1,1,2)\) containing the image of \(\pt\), and let \(\Gamma\subset\bb P(1,1,1,2)\) be the unique divisor in \(|\cal O_{\bb P(1,1,1,2)}(1)|\) containing the images of \(\ptE\) and \(\pt\). (Note that if \(\pt\in R\) is an \(A_2\) singularity, then \(R_{\bb P}|_{\Gamma}\) is necessarily reduced.)
\begin{detail}
Next, the assumptions imply that \(R_{\bb P}|_\Gamma \in |\cal O_{\bb P(1,1,2)}(4)|\) has multiplicity 2 at \(\ptE\) and (the strict transform of) \(\pt\), and that it does not pass through the cone point of \(\bb P(1,1,2)\). The strict transform of \(R_{\bb P}|_\Gamma\) on \(\bb F_2 = \Bl_{\conept}\Gamma\) is a degree 2 multisection with multiplicity 2 at (the strict transforms of) \(\ptE\) and \(\pt\). So it is either a doubled section (i.e. \(R_{\bb P}|_\Gamma\) is non-reduced), or it is reducible with two irreducible components meeting transversally at these two singular points. (By blowing up \(\ptE\), and then contracting the strict transform of the fiber of \(\bb F_2\) containing \(\ptE\) and the strict transform of the negative section, we see that \(R_{\bb P}|_\Gamma\) cannot be reduced and irreducible.) If \(R_{\bb P}|_\Gamma\) is reducible, then since it doesn't contain the cone point \(\conept\), each of its irreducible components is in \(|\cal O_{\bb P(1,1,2)}(2)|\).
\end{detail}
Define the curve \(r_1 \subset \Gamma \cong \bb P(1,1,2)\) as follows:
\begin{enumerate}
\item If \(R|_\Gamma\) is reduced, then it is the union of two curves in \(|\cal O_{\bb P(1,1,2)}(2)|\). In \(X_1\), the strict transform of each of these components meets exactly one of the two irreducible components of \(R|_Y\). Let \(r_1\) be the component of \(R|_\Gamma\) that meets the curve blown up in \(\tX_2 \to X_1\). (In \(\tX_1\), \(r_1\) does not pass through the singular point.)
\item If \(R|_\Gamma\) is not reduced, then \(r_1\in|\cal O_{\bb P(1,1,2)}(2)|\) will be a curve passing through \(\ptE\) and the strict transform of the singular point \(\pt\). We choose \(r_1\) to be the unique such curve whose strict transform meets \(Y|_\Gamma\) at the point that gets blown up meeting \(Z\).
\end{enumerate}
In the strict transform of \(\Gamma\) in \(\tX_1\), the strict transform $r$ of \(r_1\) has self-intersection \(-1\).  Thus, by Lemma~\ref{lem:moriconeforblowupofX3}, we find that the Mori cone of $\tX_1$ is generated by $r$, $l_p$, $l_1$, $l_E$, $l_S$, and $L_Z$, where $l_Z$ is a line in the exceptional divisor $Z$.

The pseudoeffective threshold of \(L-u Z\) is \(u=5-4c\).
We compute the Nakayama--Zariski decomposition of \(L-uZ\) by finding the associated birational transformations. At \(u=3-2c\) we flip \(r\) in \(\tX_1 \dashrightarrow \tX_1^+\), and at \(u=4-4c\) we flip \(l_{\pt}\) in \(\tX_1^+ \dashrightarrow X_2\).
Let \(g_1\colon \hX_1 \to \tX_1\) be the blow-up of \(r\) with exceptional divisor \(F_1\cong\bb P^1\times\bb P^1\), and let \(g_1^+\colon \hX_1 \to \tX_1^+\) the contraction of the other ruling.
Let \(g_2\colon \hX_2\to\tX_1^+\) be the blow-up of the singular point (re-extracting \(Y \cong \bb P^2\)) followed by the blow-up of the strict transform of \(l_{\pt}\) with exceptional divisor \(F_2\), and let \(g_2^+\colon \hX_2 \to X_2\) be the contraction of the other ruling of \(F_2\) followed by the contraction of \(Y\). Finally, let \(\oX\) be the common resolution of \(\hX_1\) and \(\hX_2\) constructed by blowing up the strict transform of \(r\) in \(\hX_2\). Denote the induced morphism \(\psi\colon \oX\to\tX_1\). We have the following diagram:
\[\begin{tikzcd}
& & & \hX_1 \arrow[ld, "g_1"] \arrow[rd, "g_1^+" {swap}] & \oX \ar[l, "h_1" {swap}] \ar[r, "h_2"] & \hX_2 \arrow[ld, "g_2"] \arrow[rd, "g_2^+" {swap}] & & & & \\
\Xthree & W^+ \arrow[l, "\hsmallplus"] & \tX_1 \arrow[l, "g_0"] \arrow[rr, dashed, "\text{Atiyah flop of }r" {swap}] & & \tX_1^+ \arrow[rr, dashed, "\text{flip }l_{\pt}" {swap}] & & X_2
\end{tikzcd}\]

We find that the Nakayama--Zariski decomposition of \(\psi^*(L-uZ)\) on \(\oX\) is
\[ \resizebox{1\textwidth}{!}{ $ \begin{array}{rll}
    0 \le u \le 3-2c: & P_{\regionA}(u) = \psi^*(L-uZ) & N_{\regionA}(u)  =  0 \\
    3-2c \le u \le 4-4c: & P_{\regionB}(u) = h_1^*(g_1^*(L-uZ)-(u-3+2c)F_1) & N_{\regionB}(u) = h_1^* (u-3+2c)F_1 \\
    4-4c \le u \le 5-4c: & P_{\regionC}(u) = P_{\regionB}(u) - h_2^* \tfrac{(u-4+4c)}{2}(\widehat{Y}_2 + 2 F_2) & N_{\regionC}(u)  = h_2^* \tfrac{(u-4+4c)}{2}(\widehat{Y}_2 + 2 F_2)
\end{array} $} \]
and, using Lemmas~\ref{lem:volume-lemma-4} (on \(\hX_1\)) and~\ref{lem:volume-lemma-5} (on \(\hX_2\)), we compute the volumes on these regions:
\[ \begin{array}{rl}
    0 \le u \le 3-2c: & P_{\regionA}(u)^3 = 3(2-2c)(3-2c)^2 - \tfrac{u^3}{2} \\
    3-2c \le u \le 4-4c: & P_{\regionB}(u)^3 = 3(2-2c)(3-2c)^2 - \tfrac{u^3}{2} + (u-3+2c)^3 \\
    4-4c \le u \le 5-4c: & P_{\regionC}(u)^3 = 3(2-2c)(3-2c)^2 - \tfrac{u^3}{2} + (u-3+2c)^3 + \tfrac{1}{2}(u-4+4c)^3 . \end{array}  \]
\begin{detail}
In more detail, on region \(\regionA\), \(P_{\regionA}(u)\) is the pullback of \(L-uZ = -K_{\tX_1} - \widetilde{\dee}_1 + (3-3c-u) Z\).
We have \(L|_Z = 0\) and \(Z|_Z = -l_Z\), where \(l_Z \subset Z \cong \bb P(1,1,2)\) is a ruling. So
\[P_{\regionA}(u)^3 = L^3 - u^3 (Z|_Z)^2 = 3(2-2c)(3-2c)^2 - u^3/2. \]
On \(\hX_1\), we have \(g_1^*(L-uZ) = -K_{\hX_1} + F_1 - \widehat{\dee}_1 + (3-3c-u) \widehat{Z}_1 - u F_1\).

On region \(\regionB\), \(P_{\regionB}(u)\) is the pullback from \(\hX_1\) of
\[(g_1^+)^*(-K_{\tX_1^+} - \widetilde{\dee}_1^+ + (3-3c-u)\widetilde{Z}_1^+ ) = -K_{\hX_1} + F_1 - \widehat{\dee}_1 + (3-3c-u)\widehat{Z}_1 + (3-3c-u) F_1 , \]
showing that \(N_{\regionB}(u) = h_1^* (u-3+2c)F_1\). The volume of \(P_{\regionB}(u)\) follows by applying Lemma~\ref{lem:volume-lemma-4} with \(p=-(u-3+2c)\).

On region \(\regionC\), \(P_{\regionC}(u)\) is the pullback of
\[ \resizebox{1\textwidth}{!}{ $ (g_2^+)^*(-K_{X_2} - \dee_2 + (3-3c-u) Z_2 ) = -K_{\hX_1} - \widehat{\dee}_2 + (3-3c-u)\widehat{Z}_2 + (4-4c-u) \widehat{Y}_2 + (5-4c-u) F_2 , $ } \]
showing that \(N_{\regionC}(u) = h_2^*((u/2 - 2 + 2c) \widehat{Y}_2 + (u-4+4c) F_2)\). To find the volume of \(P_{\regionC}(u)\), we apply Lemma~\ref{lem:volume-lemma-5} with \(p=-(u-4+4c)/2\).
\end{detail}
These volumes are the same as those for \(\bb P^1\times\bb P^2\) in Section~\ref{sec:local-delta-A_n:A_2-part-1}, so by the same computation we obtain that \[S_{\Xthree, cR}(Z) = -\frac{2 (17 c^3 - 73 c^2 + 104 c - 49)}{3 (3 - 2 c)^2} .\]
The log discrepancy is \(A_{\Xthree, cR}(Z) = 4-3c\), so \(A_{\Xthree, cR}(Z)/S_{\Xthree, cR}(Z) > 1\) for \(c\in (0,1)\).

Next, we apply the Abban--Zhuang method to a flag on \(Z\). We work on the common resolution \(\oX\).

The strict transform \(\oZ\) of \(Z\) in \(\oX\) is isomorphic to the blow-up of \(\bb F_2\) at a point (the intersection point with \(r_1\) in \(\hX_2\)). Let \(e\) denote the exceptional divisor of this blow-up, and let \(f\) and \(\sigma\) denote the strict transforms of the ruling and negative section of \(\bb F_2\). From the Nakayama--Zariski decompositions computed above, we find that
\[ \resizebox{1\textwidth}{!}{ $ \begin{array}{rll}
    0 \le u \le 3-2c: & P_{\regionA}(u)|_\oZ = u(\sigma/2 + f + e) & N_{\regionA}(u)|_\oZ  =  0 \\
    3-2c \le u \le 4-4c: & P_{\regionB}(u)|_\oZ = u(\sigma/2 + f) + (3-2c)e & N_{\regionB}(u)|_\oZ = (u-3+2c) e \\
    4-4c \le u \le 5-4c: & P_{\regionC}(u)|_\oZ = (2-2c)\sigma + uf+ (3-2c)e & N_{\regionC}(u)|_\oZ  =(u/2-2+2c)\sigma + (u-3+2c) e.
\end{array} $} \]
These are the same as the formulas in Section~\ref{sec:local-delta-A_n:A_2-part-2}, so the computations in that section show that \(\delta_{\pt}(\Xthree,cR)>1\) for \(c\in (0,\tfrac{1}{2}]\).

\subsection{\texorpdfstring{\(A_n\)}{An} singularities on \texorpdfstring{\(R \subset \Xthree\)}{R in Xn} with reduced restrictions to \texorpdfstring{\(\Gamma_{\pt}\)}{Gammap}}\label{sec:higher-A_n-sings-X_3}
In this section, we prove Theorem~\ref{thm:local-delta-X_3}\eqref{item:A_n-sings-X_3}. Let \(\pt\in R \subset \Xthree\) be an \(A_n\) singularity for \(n\geq 3\).
First, we construct a suitable plt blow-up. (This will resemble the construction from Section~\ref{sec:local-delta-A_n:higher-A_n-part-1}.)
We work on the small resolution \(\hsmallplus\colon W^+ \to \Xthree\). Let \(\tX_1 \to W^+\) be the blow-up of the strict transform of \(\pt\), with exceptional divisor \(Y\). The strict transform \(R|_Y\) of \(R\) in \(Y\) is a rank 2 conic; let \(\tX_2 \to \tX_1\) be the blow-up of the singular point of \(R|_Y\) with exceptional divisor \(\widehat{Z}_1\). Finally, let \(\tX_2 \to \tX\) be the contraction of the strict transform \(\hY_1\) of \(Y\). Note that \(\tX\) is singular along the image of \(\hY_1\). Let \(\phi\colon \tX \to W^+\) be the induced morphism, and let \(Z \cong\bb P^2\) be the image of \(\widehat{Z}_1\) in \(\tX\). This is the desired plt blow-up.

Let \(l_{\pt}\) be the ruling of \(\bb P(1,1,1,2)\) containing the image of \(\pt\), and let \(\Gamma\in |\cal O_{\bb P(1,1,1,2)}(1)|\) be the unique divisor containing the images of \(\ptE\) and \(\pt\). Let \(e = E \cap \Gamma\). By abuse of notation, we often use the same notation for the strict transforms of the various curves and divisors on different birational models.  By Lemma~\ref{lem:moriconeforblowupofX3}, we find that the Mori cone of $\tX$ is generated by $l_p$, $l_1$, $e=l_E$, $l_S$, and $L_Z$, where $l_Z$ is a line in the exceptional divisor $Z$.

If \(Y'\) is the exceptional divisor of the blow-up of \(\bb P(1,1,1,2)\) at the image of \(\pt\), then the divisors in \(|\cal O_{\bb P(1,1,1,2)}(1)|\) containing \(l_{\pt}\) sweep out a pencil of lines in \(Y'\cong\bb P^2\); let \(\Gamma'\in |\cal O_{\bb P(1,1,1,2)}(1)|\) be the unique divisor whose strict transform in \(Y'\) passes through the singular point of \(R|_Y\). The assumption that \(R_{\bb P}|_\Gamma\) is reduced is equivalent to \(\Gamma \neq \Gamma'\).
The strict transform of \(\Gamma'\) in \(W^+\) is the blow-up of \(\bb F_2\) at \((\hsmallplus)^{-1}(\pt)\); let \(\sigma'\) denote the strict transform of the negative section of \(\bb F_2\).

Let \(L=\phi^*(\hsmallplus)^*(-K_{\Xthree}-\dee)\). The birational transformations associated to \(L-uZ\) are:
\begin{enumerate}
\item At \(u=4-4c\), we perform the flop of \(l_{\pt}\) as constructed in Section~\ref{construction:complicatedflop} in a rational map \(\tX \dashrightarrow \tX^+\).  This flops of \(l_{\pt}\) from \(\Gamma'\) into \(Y\).
\item At \(u=5-4c\), we have the contraction \(g_2 \colon\tX^+\to\tX^{(2)}\) of \(\Gamma'\) to the curve \(S \cap Z\). (On \(\tX^+\) we have \(\Gamma'\cong\bb P^1\times\bb P^1\), and we contract the ruling given by the strict transform of \(\sigma'\).)
\item At \(u=6-4c\), we contract \(S\cong\bb P(1,1,2)\) and then contract \(\Gamma\cong\bb F_1\), yielding a morphism \(g_3 \colon\tX^{(2)} \to \bb P^3\). This final birational model is \(\bb P^3\) because one can check that it is smooth, has Picard rank 1, and (using \(e\)) has Fano index 4.
\end{enumerate}
Finally, we stop at the pseudoeffective threshold \(u=8-6c\).

\[\begin{tikzcd}
& \tX_1 \arrow[d, "\text{blow up }\pt" {swap}] & & \tX_2 \arrow[ll, "\text{blow up }\Sing(R|_Y)" {swap}] \arrow[d, "\text{contract }\hY_1"] \\
\Xthree & W^+ \arrow[l, "\hsmallplus"] & & \tX \arrow[ll, "\phi" {swap}, "\text{extract }Z"] \arrow[r, dashed, "\text{flop }l_{\pt}" {swap}] & \tX^+ \arrow[rr, "\text{contract }\Gamma'" {swap}] & & \tX^{(2)} \arrow[rr, "\text{contract } S" {swap}] & & \tX^{(3)} \arrow[rr, "\text{contract } \Gamma" {swap}] & & \bb P^3
\end{tikzcd}\]

We first compute \(A_{\Xthree,\dee}(Z)/S_{\Xthree,\dee}(Z)\). The log discrepancy is \(A_{\Xthree,\dee}(Z) = 5-4c\). Using the above birational transformations to describe the Nakayama--Zariski decomposition of \(\psi^*(L-uZ)\) on \(\oX\), we compute the volume of \(L-uZ\) using Lemmas \ref{lem:volume-lemma-2}, \ref{lem:volume-lemma-3}, and \ref{lem:volume-lemma-6}:
\[ \resizebox{1\textwidth}{!}{ $ \begin{array}{rll}
    0 \le u \le 4-4c: & P_{\regionA}(u)^3 = 3(2-2c)(3-2c)^2 - \tfrac{1}{4} u^3 \\
    4-4c \le u \le 5-4c: & P_{\regionB}(u)^3 = 3(2-2c)(3-2c)^2 - \tfrac{1}{4} u^3 - 2(2-2c-\tfrac{u}{2})^3 \\
    5-4c \le u \le 6-4c: & P_{\regionC}(u)^3 = 3(2-2c)(3-2c)^2 - \tfrac{1}{4} u^3 - 2(2-2c-\tfrac{u}{2})^3 +3(2-2c)(5-4c-u)^2 \\    6-4c \le u \le 8-6c: & P_{\regionD}(u)^3 = (8-6c-u)^3.
\end{array} $ }\]
In particular, these volumes are the same as those for \(\bb P^1\times\bb P^2\) in Section~\ref{sec:local-delta-A_n:higher-A_n-part-1}. Therefore, from that computation, we have \(A_{\Xthree,\dee}(Z)/S_{\Xthree,\dee}(Z) >1\) for all \(c \in (0,1)\).

\begin{detail}
In more detail, using the subscripts \(\regionA, \regionB, \regionC\), and \(\regionD\) to denote values of \(u\) in the four regions \([0,4-4c], [4-4c, 5-4c], [5-4c, 6-4c]\), and \([6-4c, 8-6c]\), respectively, we find that the Nakayama--Zariski decomposition of \(\psi^*(L-uZ)\) is
\[ \resizebox{1\textwidth}{!}{ $ \begin{array}{rll}
    0 \le u \le 4-4c: & P_{\regionA}(u) = \psi^*(L-uZ) & N_{\regionA}(u)  =  0 \\
    4-4c \le u \le 5-4c: & P_{\regionB}(u) = (\psi^+)^*(-K_{\tX^+}-c\widetilde{\dee}^+-uZ^+) & N_{\regionB}(u) = (u-4+4c)(\tfrac{1}{2}E_{\pt}+\tfrac{1}{2}E_3+E_4) \\
    5-4c \le u \le 6-4c: & P_{\regionC}(u) = (\psi^+)^*g_2^*(-K_{\tX^{(2)}}-c\widetilde{\dee}^{(2)}-u\widetilde{Z}^{(2)}) & N_{\regionC}(u)  = (u-4+4c)(\tfrac{1}{2}E_{\pt}+\tfrac{1}{2}E_3+E_4) + (u-5+4c) \Gamma' \\
    6-4c \le u \le 8-6c: & P_{\regionD}(u) = (\psi^+)^* g_2^* g_3^* \cal O_{\bb P^3}(8-6c-u) & N_{\regionD}(u)  = N_{\regionC}(u) + (u-6+4c)(\Gamma + \tfrac{1}{2} Y + S + E_{\pt} + E_3 + 2 E_4 )
\end{array} $} \]
Indeed, for region \(\regionB\), the flip is locally constructed in the same way as in Section~\ref{sec:local-delta-A_n:higher-A_n-part-1}, so we obtain \(N_{\regionB}(u)\) by the same computation. For region \(\regionC\), \(N_{\regionC}(u) = N_{\regionB}(u) + (u-5+4c)(\psi^+)^*\Gamma'\). For region \(\regionD\), write \(g_3\colon \tX^{(2)} \to \bb P^3\) as the composition of \(g'\colon \tX^{(2)} \to \tX^{(3)}\) (contraction of \(S\)) and \(g''\colon \tX^{(3)} \to \bb P^3\) (contraction of \(\Gamma\)). Then
\begin{equation*}\begin{split}
N_{\regionD}(u) - N_{\regionC}(u) &= (u-6+4c)(\psi^+)^* g_2^*(\tfrac{1}{2} S + g''^*\Gamma) = (u-6+4c)(\psi^+)^* g_2^*(S + \Gamma) \\
&= (u-6+4c)(\Gamma + \tfrac{1}{2}(Y + E_{\pt} + E_4) + S + E_{\pt} + E_3 + 2 E_4 + \Gamma')
\end{split}\end{equation*}
using that \(g''^*\Gamma = \Gamma + \tfrac{1}{2} S\).

To compute the volumes, on \(\regionA\) we use that \(\vol(P_{\regionA}(u)) = \vol(L-uZ)\). Since \(L|_Z=0\) and \(2Z|_Z \in |\cal O_Z(-1)|\), we find that \(\vol(P_{\regionA}(u)) = \vol L - u^3/4\). For the remaining regions \(\regionB\), \(\regionC\), and \(\regionD\), we use the descriptions of the birational transformations from \(\bb P^3\). Region \(\regionD\) is straightforward.

For region \(\regionC\), we have that \(\vol(P_{\regionC}(u)) = \vol(g_3^*P_{\regionD}(u) + (u-6+4c)(\Gamma+S))\). We compute this volume in two stages. First, consider \(g''\colon \tX^{(3)} \to \bb P^3\), which is the weighted blow-up of a line with exceptional divisor \(\Gamma\cong\bb F_1\). Let \(\sigma, f\) denote the negative section and fiber of \(\Gamma\). Then \(\Gamma|_\Gamma = -\sigma/2 + f/2\) and \(g''^*\cal O_{\bb P^3}(8-6c-u)|_{\Gamma} = (8-6c-u)f\), so
\[g''^*\cal O_{\bb P^3}(8-6c-u) + (u-6+4c)\Gamma)^3 = (8-6c-u)^3 + \tfrac{3}{2}(u-6+4c)^2 (u-8+6c) - \tfrac{3}{4} (u-6+4c)^3. \]
Next, write \(P' = g''^*\cal O_{\bb P^3}(8-6c-u) + (u-6+4c)\Gamma\). Then \(P_{\regionC}(u) = g'^*P' + (u-6+4c)\tfrac{1}{2} S\). On \(S\cong\bb P(1,1,2)\), let \(l_S\) denote a ruling. Then \(g'^* P'|_S = 0\) and \(S|_S = -2 l_S\), and \(l_S^2 = \tfrac{1}{2}\), so
\begin{equation*}\begin{split}(g'^*P' + (u-6+4c)\tfrac{1}{2}S)^3 &= P'^3 + \tfrac{2}{8}(u-6+4c)^3
\end{split}\end{equation*}
and we find that
\[\vol (P_{\regionC}(u)) = (8-6c-u)^3 - \tfrac{3}{2}(u-6+4c)^2 (8-6c-u) - \tfrac{1}{2} (u-6+4c)^3 . \]

Finally, for \(\regionB\), let \(\sigma'\) and \(z'\) denote the two rulings of \(\Gamma'\cong\bb P^1\times\bb P^1\) (recall that \(\sigma'\) is the ruling contracted by \(g_2\)). Then \(\Gamma'|_{\Gamma'} = -z'\) and \(g_2^*P_{\regionC}(u)|_{\Gamma'} = (2-2c)\sigma'\), so we find that
\begin{equation*}\begin{split}\vol(P_{\regionB}(u)) &= \vol(g_2^*P_{\regionC}(u) + (u-5+4c)(\Gamma')) \\
&= P_{\regionC}(u)^3 - 3(u-5+4c)^2(2-2c).
\end{split}\end{equation*}
\end{detail}

Next, we apply the Abban--Zhuang method. Let \(\oY\), \(\oZ\), and \(\overline{\Gamma'}'\) denote the strict transforms on \(\oX\), and write \(l_{\Gamma'} = \overline{\Gamma}'|_{\oZ}\) and \(l_Y = \oY|_Z\). We have that \(\oZ\cong\bb F_1\); let \(\sigma\) denote the negative section and \(f\) the class of a fiber. Note that \(l_{\Gamma'}\sim l_Y\sim f\) and \(E_3|_{\oZ}=\sigma\). From the descriptions above, we find that the restriction of the Nakayama--Zariski decomposition of \(\psi^*(L-uZ)\) to \(\oZ\) is
\[ \resizebox{1\textwidth}{!}{ $ \begin{array}{rll}
    0 \le u \le 4-4c: & P_{\regionA}(u)|_\oZ = (u/2)(\sigma+f) & N_{\regionA}(u)|_\oZ  =  0 \\
    4-4c \le u \le 5-4c: & P_{\regionB}(u)|_\oZ = (2-2c)\sigma + (u/2) f & N_{\regionB}(u)|_\oZ = (u/2+2c-2)\sigma \\
    5-4c \le u \le 6-4c: & P_{\regionC}(u)|_\oZ = (2-2c)\sigma + (5-4c-u/2)f & N_{\regionC}(u)|_\oZ  =(u/2+2c-2)\sigma + (u-5+4c)l_{\Gamma'} \\
    6-4c \le u \le 8-6c: & P_{\regionD}(u)|_\oZ = (8-6c-u)(\sigma+f) & N_{\regionD}(u)|_\oZ  =(3u/2-8+6c)\sigma + (u-5+4c)l_{\Gamma'} + (u/2-3+2c)l_Y.
\end{array} $} \]
\begin{detail}
In more detail, we compute \(P_{\regionA}(u)|_\oZ\) by restricting \(L-uZ\) to \(Z\) and then pulling back. Since \(L|_Z = 0\), \(2Z|_Z \in |\cal O_{Z}(-1)|\), and \(\psi|_{\oZ}\colon \oZ \to Z\) is the blow-up of \(Z\) at \(l_{\pt}\cap Z\), we have \(Z\cong\bb F_1\) with negative section \(\sigma\) and fiber \(f\). So \(P_{\regionA}(u)|_\oZ = (u/2)(\sigma+f)\).
We compute \(P_{\regionB}(u)|_\oZ\) and \(N_{\regionB}(u)|_\oZ\) using that \((\tfrac{1}{2}E_{\pt}+\tfrac{1}{2}E_3+E_4)|_{\oZ}=(\tfrac{1}{2}E_3)|_{\oZ} = \sigma/2\). Regions \(\regionC\) and \(\regionD\) are straightforward, 
\end{detail}
This is the same as in Section~\ref{sec:local-delta-A_n:higher-A_n-part-2}, and so the computation from that section shows that \(\delta_{\pt}(\Xthree,\dee)>1\) for all \(c\in(\cone, 0.7055]\).

\subsection{Non-reduced restrictions to \texorpdfstring{\(\Gamma_{\pt}\)}{Gammap}}

\begin{lemma}\label{lem:non-reduced-Gamma-cases}
Let \(\pt \in R \setminus E^+\) be a singular point. Let \(\Gamma_{\pt}\in|\cal O_{\bb P(1,1,1,2)}(1)|\) be the section containing \(\ptE\) and the strict transform of \(\pt\).
\begin{enumerate}
\item\label{part:non-reduced-restriction-Gamma}
If \(R_{\bb P}|_{\Gamma_{\pt}}\) is non-reduced, then the quartic \(\Delta \subset \bb P^2\) is reducible and one of its irreducible components is a line.
\item\label{part:worse-than-A_n-non-reduced-Gamma}
Assume the Hessian matrix of \(R\) has rank \(\leq 1\) at \(\pt\). Then \(R_{\bb P}|_{\Gamma_{\pt}}\) is non-reduced, and \((\Xthree, cR)\) degenerates to \((\Xthree, c \Rwall)\).
\item\label{part:higher-A_n-non-reduced-Gamma}
Assume \(\pt\) is an \(A_n\) singularity for some \(n\geq 2\) and \(R_{\bb P}|_{\Gamma_{\pt}}\) is non-reduced. Then \((\Xthree, cR)\) degenerates to \((\Xthree, c \Rox)\).
\end{enumerate}
\end{lemma}

\begin{proof}
By abuse of notation, we also use \(\pt\) to denote its strict transform in \(\bb P(1,1,1,2)\). By Lemma~\ref{lem:equation-of-R_P} we may choose coordinates \([x:y:z:w]\) on \(\bb P(1,1,1,2)\) so that \(R_{\bb P}\) is defined by \(w^2 = xyz^2 + z f_3(x,y) + f_4(x,y)\) for \(f_i\in \bb C[x,y]\), and \(\ptE = [0:0:1:0]\). After a further coordinate change, we may assume \(\pt = [1:0:0:0]\), \(\ptE = [0:0:1:0]\), \(\conept = [0:0:0:1]\), and \(R_{\bb P}\) has defining equation \[w^2 = (ax+by)(cx+dy)z^2 + zg_3(x,y) + g_4(x,y)\] for some \(ad \neq bc\) and \(g_i(x,y)\in\bb C[x,y]\) of degree \(i\).
The assumption that \(\pt\in R_{\bb P}\) implies \(g_4(x,0) = 0\).

We have \(\Gamma_{\pt} = (y=0)\), so \(R_{\bb P}|_\Gamma\) has defining equation
\(w^2 = ac x^2 z^2 + zg_3(x,0) + g_4(x,0)\). Hence \(R_{\bb P}|_\Gamma\) is non-reduced if and only if \(ac = g_3(x,0) = 0\). Thus, if \(R_{\bb P}|_\Gamma\) is non-reduced, then \(\Delta\) is reducible with \((y=0)\) as an irreducible component, showing the first part of~\eqref{part:non-reduced-restriction-Gamma}.

For part~\eqref{part:worse-than-A_n-non-reduced-Gamma}, a direct computation shows that if \(\pt\in R_{\bb P}\) is a singular point where the Hessian matrix has rank 1, then \(R_{\bb P}|_{\Gamma_{\pt}}\) is non-reduced.
\begin{detail}
Indeed, if \(\pt\in R_{\bb P}\) is a singular point, then \(\left.\frac{\partial g_4}{\partial y}\right|_{y=0} = \left.g_3\right|_{y=0} = 0\). If furthermore \(\pt\in R_{\bb P}\) is a worse than \(A_n\) singularity, then the Hessian matrix at \(\pt\) has rank 1, so \(\frac{\partial^2 g_4}{\partial y^2}|_{y=0} = \frac{\partial g_3}{\partial y}|_{y=0} = ac = 0\). Therefore \(R_{\bb P}|_{\Gamma}\) is necessarily non-reduced.
\end{detail}
The degeneration claim will be shown in the next paragraph.

Now assume \(R_{\bb P}|_{\Gamma_{\pt}}\) is non-reduced and \(\pt \in R_{\bb P}\) is a worse than \(A_1\) singularity. After a possible further coordinate change, \(R_{\bb P}\) has defining equation \[w^2 = xyz^2 + z h_3(x,y) + h_4(x,y) \] and contains the points \(\pt = [1:0:0:0]\) and \(\ptE = [0:0:1:0]\) (and does not contain the cone point \(\conept = [0:0:0:1]\)).
The assumptions imply
\[h_3(x,y) = a_2 xy^2 + a_3 y^3, \qquad h_4(x,y) = b_2 x^2 y^2 + b_3 x y^3 + b_4 y^4\]
for some \(a_2, a_3, b_2, b_3, b_4 \in \bb C\).
\begin{detail}
Indeed, the assumption \(\pt\in R_{\bb P}\) implies \(h_4(x,y)\) does not have an \(x^4\) term.
On the chart \((x\neq 0)\), the Jacobian matrix of \(R_{\bb P}\) at \(\pt\) is
\[\begin{pmatrix}
\frac{\partial h_4}{\partial y}|_{y=0} & h_3|_{y=0} & 0 \end{pmatrix}\]
and the Hessian matrix is
\[\begin{pmatrix}
\frac{\partial^2 h_4}{\partial y^2}|_{y=0} & \frac{\partial h_3}{\partial y}|_{y=0} & 0 \\
\frac{\partial h_3}{\partial y}|_{y=0} & 0 & 0 \\ 0 & 0 & -2
\end{pmatrix}\]
so the assumption implies that \(\frac{\partial h_4}{\partial y}|_{y=0} = h_3|_{y=0} = \frac{\partial h_3}{\partial y}|_{y=0}\).
\end{detail}
Performing the coordinate change \(x \mapsto \lambda x, y \mapsto \tfrac{1}{\lambda} y\) and taking \(y\to\infty\) shows that the equation of \(R_{\bb P}\) degenerates to \[w^2 = xy(z^2 + b_2 x y).\]
Furthermore, \(b_2\neq 0\) if and only if the Hessian of \(R_{\bb P}\) has rank 2 at \(\pt\). In particular, this shows parts~\eqref{part:worse-than-A_n-non-reduced-Gamma} and~\eqref{part:higher-A_n-non-reduced-Gamma} (since \(b_2 \neq 0\) if \(\pt\) is an \(A_n\) singularity).
\end{proof}

\begin{lemma}\label{lem:D_n-on-X_3-unstable}
    Under the assumptions of Theorem~\ref{thm:local-delta-X_3},
    let \(\pt \in R \setminus E^+\) be a non-\(A_n\) isolated singular point of $R$.  Then, \(\delta_p(\Xthree, cR) = 1\) for $c = \cnot$ and \(\delta_p(\Xthree, cR) < 1\) for all $c > \cnot$.  In particular, $(\Xthree, cR)$ is K-unstable for all $c > \cnot$.
\end{lemma}

\begin{proof}
    First, note that \((\Xthree, cR)\) as in the statement of the Lemma has \(\delta_p(\Xthree, cR) = 1\) for $c = \cnot$ because it admits an isotrivial degeneration to \((\Xthree, c\Rwall)\) by Lemma~\ref{lem:non-reduced-Gamma-cases}\eqref{part:worse-than-A_n-non-reduced-Gamma}.  The delta-invariant of \((\Xthree, c\Rwall)\) with semicontinuity of these invariants implies that \((\Xthree, cR)\) has \(\delta_p(\Xthree, cR) = 1\). 
    
    Now, assume that $c > \cnot$ and \((\Xthree, cR)\) is K-semistable and $R$ has a non-$A_1$ isolated singularity at some point $p$ such that $R_{\bb P}|_{\Gamma_p}$ is non-reduced.  By Lemma \ref{lem:R-cannot-contain-S}, Lemma \ref{lem:p_E-not-A_2}, Lemma \ref{lem:equation-of-R_P}, and the computation in the proof of Lemma \ref{lem:non-reduced-Gamma-cases} (parts~\eqref{part:worse-than-A_n-non-reduced-Gamma} and ~\eqref{part:higher-A_n-non-reduced-Gamma}), we may assume that $R_{\bb P}$ has equation \[w^2 = xyz^2 + z h_3(x,y) + h_4(x,y) \] where
    \[h_3(x,y) = a_2 xy^2 + a_3 y^3, \qquad h_4(x,y) =  b_2 x^2y^2 + b_3 x y^3 + b_4 y^4.\]
    In particular, $R_{\bb P}$ has equation \[w^2 = y(xz^2 + a_2 xyz + a_3y^2 z + b_2 x^2y+ b_3xy^2 + b_4 y^2). \]  By construction, the surfaces have $A_1$ singularities at $\ptE$ and rational double points elsewhere.

    By projecting away from \(\conept\), these surfaces are double covers of $\bP^2$ branched over the reducible quartic \[y(xz^2 + a_2 xyz + a_3y^2 z + b_2 x^2y+ b_3xy^2 + b_4 y^2)\]
    with a node at $[0:0:1]$ and a singular point at $[1:0:0]$.  Furthermore, the singular point $[1:0:0]$ is not unibranch and is a multiplicity two intersection point of a cubic curve and a line.  By the classification of ADE singularities, this implies that $R_{\bb P}$ has either an $A_3$ or $D_n$, $n = 4,5,6$ singularity.  This implies the discriminant curve has one of the forms in Figure \ref{fig:worse-than-A_n-discriminants}.\footnote{The cubic curve pictured in the first $A_3$ case may be singular away from the marked points, and could even be reducible.  It has at worst $A_2$ singularities.}

    \begin{figure}[h]
    \centering
    $A_3$: \begin{tabular}{c}
\begin{tikzpicture}
\draw [-] (0,2) to (2,2);
\draw [-] plot [smooth, tension=1] coordinates { (0.4,2.4)  (0.7,1.6) (1.4,2) (1.9,1.4) };

\draw (.46,2) circle (0.02);
\draw (1.4,2) circle (0.02);
\node[below left, node font=\tiny] at (.53,2) {$A_1$};
\node[above right, node font=\tiny] at (1.4,2) {$A_3$};

\end{tikzpicture}
\end{tabular} or
    \begin{tabular}{c}
\begin{tikzpicture}

\draw (0.7,0.7) circle (0.5);
\draw [-] (0.15,1.2) to (1.7,1.2);
\draw [-] (0.94,0) to (1.56,1.4);

\draw (.7,1.2) circle (0.02);
\draw (1.16,0.5) circle (0.02);
\node[above left, node font=\tiny] at (.7,1.2) {$A_3$};
\node[below right, node font=\tiny] at (1.16,0.5) {$A_3$};

\end{tikzpicture}
\end{tabular}

    $D_4$: \begin{tabular}{c}
\begin{tikzpicture}
\draw [-] (0,2) to (2,2);
\draw [-] plot [smooth, tension=1] coordinates { (0.4,2.4)  (0.7,1.7) (1.4,2) (1.65,2.4) (1.4, 2.6) (1.22,2.4) (1.4,2) (1.9,1.4) };

\draw (.46,2) circle (0.02);
\draw (1.4,2) circle (0.02);
\node[below left, node font=\tiny] at (.53,2) {$A_1$};
\node[below, node font=\tiny] at (1.4,1.9) {$D_4$};

\end{tikzpicture}
\end{tabular} or 
    \begin{tabular}{c}
\begin{tikzpicture}

\draw [-] (0,2) to (2,2);
\draw [-] (0.2,1.4) to (1.8, 2.2);
\draw (0.95,2) circle (0.45);

\draw (.5,2) circle (0.02);
\draw (1.4,2) circle (0.02);
\node[below left, node font=\tiny] at (.53,2) {$A_1$};
\node[below right, node font=\tiny] at (1.4,2) {$D_4$};

\end{tikzpicture}
\end{tabular} or 
    \begin{tabular}{c}
\begin{tikzpicture}
\draw [-] (0,2) to (2,2);
\draw [-] (1.4, 2.4) to (1.4, 0.7);
\draw [-] (0.2,1.4) to (1.8, 2.2);
\draw [-] (0.3,2.2) to (1.6, 0.9);

\draw (.5,2) circle (0.02);
\draw (1.4,2) circle (0.02);
\node[below left, node font=\tiny] at (.53,2) {$A_1$};
\node[below right, node font=\tiny] at (1.4,2) {$D_4$};

\end{tikzpicture}
\end{tabular}

    $D_5$: \begin{tabular}{c}
\begin{tikzpicture}
\draw [-] (0,2) to (2,2);
\draw [-] plot [smooth, tension=1] coordinates { (0.4,2.4)  (0.8,1.5)   (1.4,2)  };
\draw [-] plot [smooth, tension=1] coordinates {(1.4,2) (1.6,1.6) (1.9,1.4) };

\draw (.49,2) circle (0.02);
\draw (1.4,2) circle (0.02);
\node[below left, node font=\tiny] at (.53,2) {$A_1$};
\node[above, node font=\tiny] at (1.4,2) {$D_5$};

\end{tikzpicture}
\end{tabular}
    $D_6$: \begin{tabular}{c}
\begin{tikzpicture}

\draw [-] (0,2) to (2,2);
\draw [-] (1.4,2.6) to (1.4, 1.4);
\draw (0.95,2) circle (0.45);

\draw (.5,2) circle (0.02);
\draw (1.4,2) circle (0.02);
\node[below left, node font=\tiny] at (.53,2) {$A_1$};
\node[below right, node font=\tiny] at (1.4,2) {$D_6$};

\end{tikzpicture}
\end{tabular}
    \caption{Possible discriminant curves for surfaces with non-reduced $R_{\bP}|_{\Gamma_p}$.}
    \label{fig:worse-than-A_n-discriminants}
    \end{figure}
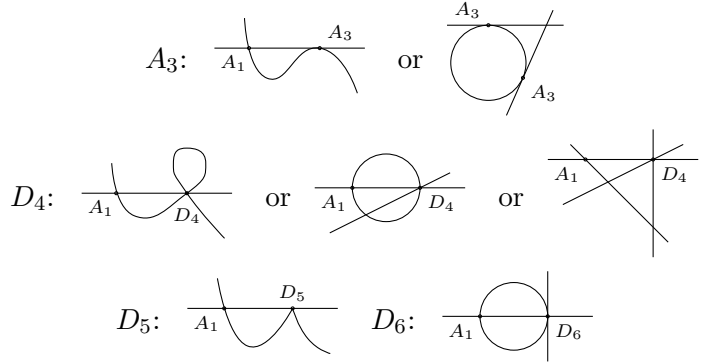

    Let $R_{D_4}$ be a surface such that the associated surface $R_{\bb P}$ has discriminant of the first $D_4$ type in Figure \ref{fig:worse-than-A_n-discriminants}, a nodal cubic curve plus a line meeting the node to order 2 and at one smooth point.  Denote this discriminant by $\Delta_{D_4}$.  Still assuming that $(\Xthree, cR_{D_4})$ is K-semistable, we will deduce a contradiction.  Up to a coordinate change, we may assume without loss of generality that the surface $R_{D_4}$ is given by the equation \[w^2 = y(xz^2 + y^2(y-x)).  \] This surface is the central fiber of the the family $\mathcal{R}$ over $\bA^1_t$ given by 
    \[w^2 = y(xz^2 + y(y-tx)(y-x)).  \]
    Near $t = 0$, the generic surface $\mathcal{R}_t$ for $t \ne 0$ has an $A_3$ singularity at $[1:0:0:0]$ and is K-semistable by Lemma~\ref{lem:non-reduced-Gamma-cases}\eqref{part:higher-A_n-non-reduced-Gamma}.  At $t = 0$, we have $\mathcal{R}_0 = R_{D_4}$.
    Now, consider the family $\mathcal{R'}$ over $\bA^2_{t,s}$ given by 
    \[w^2 = y(xz^2 + s^2y^3 - s(t+1)xy^2 + tx^2y).  \]
    For a the general fiber $\mathcal{R'}_{s,t} \cong \mathcal{R}_t $ by the isomorphism $y \mapsto s^{-1/2}y$ and $x \mapsto s^{1/2} x$.  Taking the limit along the pointed curve $C \subset \bA^2$ defined by $s = t+1$ approaching $s =0$, we have $\mathcal{R'}_{0,-1} \cong \Rox$ as it is defined by the equation 
    \[w^2 = xy(z^2 -xy).  \]
    Therefore, we have two families of K-semistable objects over pointed curves, isomorphic away from the special fiber, so \cite[Theorem 1.1]{BlumXu19} implies that their central fibers are S-equivalent, i.e. $\Rox$ is S-equivalent to $R_{D_4}$.  However, for any $c > \cnot$, the pair $(\Xthree, c\Rox)$ is K-polystable, so this implies that $R_{D_4}$ admits an isotrivial degeneration to $\Rox$.  Because such a degeneration of the surfaces induces a degeneration of the quartic discriminant curve by Corollary~\ref{cor:discriminant-family}, this forces the quartic curve $\Delta_{D_4}$ to admit a degeneration to the ox.  This is a contradiction as $\Delta_{D_4}$ has a triple point and the ox has at most double points.  Therefore, $(\Xthree, cR_{D_4})$ is K-unstable for all $c > \cnot$. 

    Finally, we may conclude that any surface with $D_n$ singularities has $(\Xthree, cR)$ K-unstable for $c > \cnot$.  Indeed, this follows from openness of K-semistability as the remaining discriminants of surfaces with $D_n$ singularities are further degenerations of the surface with discriminant $\Delta_{D_4}$.  
\end{proof}

\begin{remark}\label{rem:worse-than-A_n-after-wall}
    By a similar argument, using Corollary~\ref{cor:discriminant-family} we may conclude that any surface $R$ with worse than $A_n$ isolated singularities on $\bP^1 \times \bP^2$ has $(\bP^1 \times \bP^2, cR)$ K-unstable for $c > \cnot$.  Indeed, the only pairs $(\bP^1 \times \bP^2, cR)$ for which $R$ could be K semistable (and hence GIT semistable) with worse than $A_n$ isolated singularities could be $(2,2)$ surfaces with $D_n$ singularities. Using the explicit equation of $R$ in Lemma~\ref{lem:GITsemistable-D-n}, we can find a 1-parameter family $\mathcal{R} \to \mathbb{A}^1_{a_{22}}$ (in coordinates, varying the coefficient $a_{22}$) of surfaces with $A_3$ singularities degenerating to $R = \mathcal{R}_{1}$ (for generic $a_{22} \ne 1$, the surface has an $A_3$ singularity, and for $a_{22} = 1$, we have $R$).  However, for $a_{22} \ne 1$, there exists an isotrivial 1-parameter family $\mathcal{R}_{a_{22}} \to \mathbb{A}^1_t$ degenerating the surface to a surface on $\Xthree$ with $A_3$ singularities via the degeneration to $\Rthree$ and this family varies with $a_{22}$.  In particular, taking the two-parameter family $\mathcal{R} \to \mathbb{A}^2_{a_{22},t} \setminus \{(t = 0) \cup (a_{22} \ne 1)\}$, the pullback to the curve $a_{22} = t+1$ has generic fiber a surface with $A_3$ singularities and limit over $t = 0$ a surface on $\Xthree$ with $A_3$ singularities.  The pullback to the curve $t = 1$, approaching $a_{22} = 0$, has generic fiber the same surface with $A_3$ singularities but special fiber the surface with $D_n$ singularities.  If this special fiber were K-semistable, but the same argument in the Lemma~\ref{lem:D_n-on-X_3-unstable}, we obtain a contradiction.
\end{remark}

\begin{corollary}\label{cor:X_3-non-reduced-A_n-not-stable}
    Suppose that \((\Xthree, cR)\) is a pair such that $R$ has a non-$A_1$ isolated singularity at some point $p$ such that $R_{\bb P}|_{\Gamma_p}$ is non-reduced.  Then, \((\Xthree, cR)\) is not K-stable for any $c$.  
\end{corollary}

\begin{proof}
    For $c < \cnot$, any pair $(\Xthree, cR)$ is K-unstable. 
 Now suppose $R$ is a surface on $\Xthree$ with non-$A_1$ isolated singularities such that $R_{\bb P}|_{\Gamma_p}$ is non-reduced.  If $c = \cnot$, by Lemma~\ref{lem:non-reduced-Gamma-cases} parts~\eqref{part:worse-than-A_n-non-reduced-Gamma} and~\eqref{part:higher-A_n-non-reduced-Gamma}, every such pair $(\Xthree, cR)$ admits a degeneration to $(\Xthree, c\Rwall)$.  This is clear for surfaces with non-$A_n$ singularities and follows for surfaces with $A_n$ singularities as the pair $(\Xthree,c\Rox)$ degenerates to $(\Xthree, c\Rwall)$.  Therefore, none of these pairs can be K-stable.  Finally, if $c > \cnot$, by Lemma \ref{lem:D_n-on-X_3-unstable}, if $R$ has worse than $A_n$ singularities, then \((\Xthree, cR)\) is K-unstable.  Therefore, if $(\Xthree, cR)$ is K-semistable, it has only $A_n$ singularities, but by Lemma~\ref{lem:non-reduced-Gamma-cases}\eqref{part:higher-A_n-non-reduced-Gamma}, it admits an isotrivial degeneration to the polystable (non-stable) pair $(\Xthree, c\Rox)$ so is not K-stable. 
\end{proof}

\subsection{K-unstable surfaces with worse than \texorpdfstring{\(A_1\)}{A1} singularities at \texorpdfstring{\(\ptE\)}{ptE}}\label{sec:unstable-X_3-sings-at-E}
In this section, we prove Theorem~\ref{thm:X_3-unstable}\eqref{item:unstable-worse-than-A_2-at-E} and~\eqref{item:unstable-A_2-at-E}. That is, we show surfaces \(R\subset \Xthree\) whose strict transforms in \(\bb P(1,1,1,2)\) have worse than \(A_1\) singularities at \(\ptE\) are K-unstable. For \(A_2\) singularities, we show:
\begin{lemma}\label{lem:p_E-not-A_2}
In the setting of Theorem~\ref{thm:X_3-unstable}, if \(\ptE\in R_{\bb P}\) is an \(A_2\) singularity, then \(\delta(\Xthree, cR)<1\) for \(c\in (0, \tfrac{1}{2}]\).
\end{lemma}
By openness of K-stability \cite{BLX22}, since \(D_n\) singularities and higher \(A_n\) singularities are degenerations of \(A_2\) singularities, Lemma~\ref{lem:p_E-not-A_2} implies:
\begin{corollary}\label{cor:X_3-higher-A_n-at-ptE-unstable}
    In the setting of Theorem~\ref{thm:X_3-unstable}, if \(\ptE\in R_{\bb P}\) is an \(A_n\) singularity for \(n\geq 2\), then \(\delta(\Xthree, cR)<1\) for \(c\in (0, 0.6]\).
\end{corollary}

\begin{remark}
    If $n \ge 3$ in Corollary~\ref{cor:X_3-higher-A_n-at-ptE-unstable}, one can in fact show that $\delta(\Xthree,cR) < 1$ for all $c \in (0, 1)$.  More precisely, if $n \ge 3$, then after blowing the singular point of $R$, there exists a unique singular point of the strict transform.  If $E$ is the exceptional divisor of the blow-up of this point, then by computation we have $A_{\Xthree, cR}(E) = 3-2c$ and $S_{\Xthree,cR} = -\frac{30c^3 - 130c^2+188c -91}{3(3-2c)^2}$ and therefore $\delta(\Xthree, cR) \le \delta(\Xthree, cR)(E) < 1$ for all $c \in (0,1)$.  In particular, if $R_{\bb P }$ has worse than an $A_2$ singularity at $p_E$, then $(\Xthree,cR)$ is K-unstable for every $c$.  
\end{remark}

Before proving Lemma~\ref{lem:p_E-not-A_2}, we first prove Theorem~\ref{thm:X_3-unstable}\eqref{item:unstable-worse-than-A_2-at-E} because the computation in this case is more straightforward:
\begin{lemma}
    In the setting of Theorem~\ref{thm:X_3-unstable}, if \(\ptE \in R_{\bb P}\) is a non-isolated singularity, then \(\delta(\Xthree, cR)<1\) for \(c\in (0,1)\).
\end{lemma}
\begin{proof}
On the small resolution \(\hsmall\colon W \to \Xthree\), let \(R|_{E'}\) be the intersection of \(E'\) with the strict transform of \(R\) on \(W\). The assumption implies that \(R|_{E'}\) is either a doubled section of $E' \cong \bb F_1$; write \(r\) for the support of \(R|_{E'}\). Let \(\phi\colon \tX \to W\) be the blow-up of \(r\) with exceptional divisor \(Y\cong\bb F_2\). We will show that \(A_{\Xthree, \dee}(Y)/S_{\Xthree, \dee}(Y) < 1\).

Let \(l_E\subset E'\cong\bb F_1\) and \(l_S\subset S'\cong\bb F_1\) denote rulings, and recall that \(E'\) and \(S'\) meet along their negative sections \(l_2\). By abuse of notation, we use the same notation for their strict transforms on \(\tX\). Write \(L\coloneqq\phi^*\hsmall^*(-K_{\Xthree}-\dee)\). We compute that the pseudoeffective threshold of \(L-uY\) is \(u=3-2c\), and the Nakayama--Zariski decomposition is
\[ \begin{array}{rll}
    0 \le u \le 2-2c: & P_{\regionA}(u) = L-uY & N_{\regionA}(u)  =  0 \\
    2-2c \le u \le 3-2c: & P_{\regionB}(u) = L-uY-(u-2+2c)\widetilde{E} & N_{\regionB}(u) = (u-2+2c)\widetilde{E}
\end{array} \]
where \(\widetilde{E}\) is the strict transform of \(E'\) on \(\tX\), and the volume of \(L-uY\) is
\[ \vol(L-uY) = 3(2-2c)(3-2c)^2 - 3(2-2c)u^2 \]
for \(0\leq u\leq 3-2c\).
\begin{detail}
In more detail, we compute that \((L-uY)\cdot l_E = 2-2c-u\), so at \(u=2-2c\) we contract \(\widetilde{E}\) in a morphism \(\psi\colon\tX \to \tX_2\). At \(u=3-2c\), the contraction of \(l_S\) on \(\tX_2\) is not birational.

For the volumes, let \(l_Y\) denote a ruling and \(\sigma_Y\) denote the negative section of \(Y\cong\bb F_2\). Then \(L|_Y = (2-2c) l_Y\) and we may compute \(Y|_Y = -\sigma_Y-l_Y\): because $Y$ is an exceptional divisor of the blow-up of a smooth curve in $W$, $Y \cdot l_Y = -1$, and $\sigma_Y$ is the (transverse) intersection of $Y$ and $\widetilde{E}$, so $Y \cdot \sigma_Y$ is equal to the self-intersection of $\sigma_Y$ in $\widetilde{E}$, which is equal to $1$.   Therefore, 
\[ P_{\regionA}(u)^3 = L^3 + 3u^2(L|_Y \cdot Y|_Y)
= 3(2-2c)(3-2c)^2 - 3(2-2c)u^2 . \]
On \(\widetilde{E}\cong E'\cong\bb F_1\), we have \((L-uY)|_{\widetilde{E}} = (2-2c-u)(\sigma_E + f_E)\) and \(\widetilde{E}|_{\widetilde{E}} = -\sigma_E - 2l_E \)
so by Lemma \ref{lem:volume-lemma-2} with $p = 2-2c-u$, $q = 1$, $n = 1$, $a = 0$, and $b = -1$, we find 
\[ P_{\regionB}(u)^3 = P_{\regionA}(u)^3 = 3(2-2c)(3-2c)^2 - 3(2-2c)u^2. \]
\end{detail}
Since \(A_{\Xthree, \dee}(Y) \leq 2-2c\), we thus compute that, \[S_{\Xthree, \dee}(Y) = \frac{2 (3 - 2 c)}{3}, \qquad \frac{A_{\Xthree, \dee}(Y)}{S_{\Xthree, \dee}(Y)} \leq \frac{3 - 3 c}{3 - 2 c}<1 \text{ for }c\in(0,1). \]
\end{proof}

\begin{proof}[Proof of Lemma~\ref{lem:p_E-not-A_2}]
By assumption \(R|_{E^+}\) is a rank 2 conic. Let \(q_1\in R|_{E^+}\) be the singular point.

We first construct a suitable plt blow-up. First, blow up \(q_1\in W^+\), creating the exceptional divisor \(Y\cong\bb P^2\). Then, blow up the line \(R|_Y \subset Y\) (which, by the assumption that $R$ has an $A_2$ singularity at $\ptE$, is equal to the intersection of \(Y\) with the strict transform of \(E^+\)), and let \(Z_0 \cong \bb F_2\) be the exceptional divisor. Finally, contract the strict transform of \(Y\) to obtain a threefold \(X_1\) with a morphism \(\phi\colon X_1\to W^+\). Let \(Z\cong\bb P(1,1,2)\) denote the image of \(Z_0\) in \(X_1\); then \(\phi\) is the plt blow-up extracting \(Z\). We will show that \(A_{\Xthree,\dee}(Z)/S_{\Xthree,\dee}(Z) < 1\). First, the log discrepancy is \(A_{\Xthree,\dee}(Z) = 4-2c\).

Let \(L=\phi^*(\hsmallplus)^*(-K_{\Xthree}-\dee)\). To compute the volume of \(L-uZ\), we first define the following curves. Let \(H_1\subset X_1\) be the strict transform of the unique divisor in \(|\cal O_{\bb P(1,1,1,2)}(1)|\) whose restriction to \(E^+\) passes through \(q_1\). Let \(e_0\) be the intersection of \(E_1 \cong \bb F_1\) and \(H_1\) (where the subscript \(1\) is used to denote the strict transform on \(X_1\)), let \(l_S\) be a line in \(S_1\cong\bb P^2\), and let \(z\) be a ruling of \(Z\cong\bb P(1,1,2)\).  By Lemma \ref{cor:generatorsofmoriconeonblowupofX3}, the curves \(e_0, z, l_1, l_S\), together with any other extremal curves in $H_1$, generate the Mori cone of \(X_1\). 

From the steps of the contractions associated to $L-uZ$, we will find that the pseudoeffective threshold is \(u=11-8c\).
By computing intersection numbers of \(L-uZ\), we find the birational transformations associated to \(L-uZ\): $(L-uZ) \cdot e_0 = 2-2c - u$, $(L-uZ) \cdot z= u/2$, $(L-uZ) \cdot l_1 = 0$, and $(L-uZ) \cdot l_S = 1$.  We find the following list of birational contractions, which will be explained in more detail as the volumes are computed. 
\begin{enumerate}
\item At \(u=2-2c\), we perform the contraction \(f_1\colon X_1 \to X_2\) of \(E_1\cong\bb F_1\) onto its negative section; \(X_2\) has a \(\tfrac{1}{3}(1,1)\) singularity along the image of \(E_1\). Then we flip \(l_1\) in \(X_2 \dashrightarrow X_3\).
\item At \(u=4-2c\), we flip \(l_S\) in \(X_3 \dashrightarrow X_4\).
\item At \(u=5-2c\), we contract \(l_1^+\) (contracting the strict transform of \(S\)).
\item Finally, at \(u=6-4c\), we contract \(z\).
\end{enumerate}
Let \(P(u)\) denote the positive part of the Nakayama--Zariski decomposition. Using the subscripts \(\regionA, \regionB, \regionC, \regionD, \regionE\) to denote values of \(u\) on each of these regions, we will use the birational transformations described above to compute the volume of \(L-uZ\):
\[ \resizebox{1\textwidth}{!}{ $ \begin{array}{rl}
    0 \le u \le 2-2c: & P_{\regionA}(u)^3 = 3(2-2c)(3-2c)^2 - \tfrac{u^3}{2} \\
    2-2c \le u \le 4-2c: & P_{\regionB}(u)^3 = 3(2-2c)(3-2c)^2 - \tfrac{u^3}{2} + \tfrac{(2-2c-u)^2 (10-10c+4u)}{9} - \tfrac{(2-2c-u)^3}{18} \\
    4-2c \le u \le 5-2c: & P_{\regionC}(u)^3 = 3(2-2c)(3-2c)^2 - \tfrac{u^3}{2} + \tfrac{(2-2c-u)^2 (10-10c+4u)}{9} - \tfrac{(2-2c-u)^3}{18} - \tfrac{(4-2c-u)^3}{2} \\
    5-2c \le u \le 6-4c: & P_{\regionD}(u)^3 = 3(2-2c)(3-2c)^2 - \tfrac{u^3}{2} + \tfrac{(2-2c-u)^2 (10-10c+4u)}{9} - \tfrac{(2-2c-u)^3}{18} - \tfrac{(4-2c-u)^3}{2} + \tfrac{(5-2c-u)^3}{9} \\
    6-4c \le u \le 11-8c: & P_{\regionE}(u)^3 =  3(2-2c)(3-2c)^2 - \tfrac{u^3}{2} + \tfrac{(2-2c-u)^2 (10-10c+4u)}{9} - \tfrac{(2-2c-u)^3}{18} - \tfrac{(4-2c-u)^3}{2} + \tfrac{(5-2c-u)^3}{9} + \tfrac{(6-4c-u)^3}{2} \\
\end{array} $} \]
and from these volumes, we find that
\[ S_{\Xthree, \dee}(Z) = -\frac{40 c^3 - 176 c^2 + 259 c - 128}{3 (3 - 2 c)^2}, \qquad \frac{A_{\Xthree, \dee}(Z)}{S_{\Xthree, \dee}(Z)} = -\frac{6 (2 - c) (3 - 2 c)^2}{40 c^3 - 176 c^2 + 259 c - 128} < 1\]
for \(c < (9 - \sqrt{17})/8 \approx 0.6096\).

It remains to show the claimed values for the volumes. We make repeated use of the Lemmas in Appendix~\ref{appendix:volume-lemmas}.  First, at $u = 2-2c$, we contract the divisor $E_1 \cong \bb F_1$ onto its negative section.  Let $P_1 = P_{\regionA} = L-uZ$.  By direct computation, the negative section $\sigma$ of $E_1$ (equivalent to $2z$) has $P_1 \cdot \sigma = u$, $E_1 \cdot \sigma = 2$.  The fiber $e \subset E_1$ contracted has $P_1 \cdot e = 2-2-c-u$ and $E_1 \cdot e = -3$. Applying Lemma~\ref{lem:volume-lemma-2} with $p = 2-2c-u$, $q = 3$, $a = u$ and $b = 2$, we find 
\begin{equation}\label{eqn:AtoBpart1}
    P_2 = f_1^* P_1 + \frac{2-2c-u}{3} E_1
\end{equation}
where $P_2$ is the pushforward of $P_1$ and 
\[ P_2^3 = P_1^3 + \frac{(2-2c-u)^2}{9}(10-10c+4u). \]
We use equation (\ref{eqn:AtoBpart1}) to compute the new intersections of $P_2$ with the images of $l_1, l_S,$ and $z$, which we continue to denote $l_1, l_S,$ and $z$ by abuse of notation.  Here, $P_2 \cdot z = \frac{4-4c+u}{6}$; $P_2 \cdot l_1 = \frac{2-2c-u}{3}$; and $P_2 \cdot l_S = 1$.  So, still at $u = 2-2c$, we must flip the curve $l_1$.  This flip is computed as (i) the blow-up of the $\frac{1}{3}(1,1)$ singularity of $X_2$ (re-extracting $E_1$); (ii) the Atiyah flop of curve $l_1$; (iii) the reverse of the flip constructed in Lemma \ref{lem:volume-lemma-5} applied to the curve $e_0 \subset E_1$ that is the strict transform of the fiber of $E_1$ meeting $l_1$; and (iv) the contraction of the strict transform of $E_1$ back to a $\frac{1}{3}(1,1)$ singularity.  

Using Lemmas~\ref{lem:volume-lemma-4} and~\ref{lem:volume-lemma-5}, we compute that after the flip, $P_{\regionB} \cdot z = 1-c$; $P_{\regionB}\cdot l_1^+ - \frac{u-2+2c}{6}$; and $P_{\regionB} \cdot l_S = \frac{4-2c-u}{2}$ and that 
\[ P_{\regionB}^3 = P_2^3 - \frac{(2-2c-u)^3}{18}. \]

Next, at $u = 4-2c$, we must flip $l_S$, which is precisely of type as outlined in Lemma \ref{lem:volume-lemma-5} where $p = \frac{4-2c-t}{2}$, and by that result, we may compute the intersections $P_{\regionC}$ after the flip: $P_{\regionC} \cdot z = \frac{6-4c-u}{2}$; $P_{\regionC} \cdot l_1^+ = \frac{5-2c-u}{3}$; and $P_{\regionC} \cdot l_S^+ = u - 4 + 2c$.  The remaining two steps are to contract $l_1^+$ (which is the class of a ruling of the strict transform of $S$ now $S' \cong \bP(1,1,3)$) and $z$ (which is the class of a ruling of the strict transform of $H_1$ now $H_1' \cong \bP(1,1,2)$).  We may compute the volume change of both using that $S' \cdot l_1^+ = -1$ and $H_1' \cdot z = -1/2$ and applying Lemma \ref{lem:volume-lemma-2} to conclude they are as claimed above. 
\end{proof}

\subsection{The strictly K-polystable surface \texorpdfstring{\(\Rox \subset \Xthree\)}{Rox in Xn} for \texorpdfstring{\(c>\cnot\)}{c > c1}}\label{sec:X3-surface-ox}

\begin{proof}[Proof of Corollary~\ref{cor:Rox-polystable}]
    Recall that the strict transform \((\Rox)_{\bb P}\) of \(R\) in \(\bb P(1,1,1,2)\) is defined by \(w^2=xy(z^2-xy)\), where the coordinates are chosen as in Lemma~\ref{lem:equation-of-R_P}. Consider the automorphisms \(\phi_\lambda\colon [x:y:z:w] \mapsto [\lambda x: \lambda^{-1} y:z:w]\) and \(\iota\colon [x:y:z:w] \mapsto [y:x:z:w]\) of \((\bb P(1,1,1,2),(\Rox)_{\bb P})\). Then \(\ptE = [0:0:1:0]\) is the only singular point of \((\Rox)_{\bb P}\) fixed by \(\phi_\lambda, \iota\). Since \(\phi_\lambda\) and \(\iota\) fix \(\ptE, \conept\), and \(l_1 = (x=y=0)\), these induce automorphisms of \((\Xthree, \Rox)\); let \(G_3'\subset\Aut(\Xthree, \Rox)\) denote the subgroup they generate. By Lemmas~\ref{lem:X3-smooth-in-E} and~\ref{lem:X3-singular-pt-delta} we have \(\delta_{\pt}(\Xthree, c\Rox)>1\) for all \(\pt\in E\) and \(c\in(\cnot,\tfrac{1}{2}]\). Then, by Lemma~\ref{sec:smooth-pts-X_3} and our above observation about the singular points of \((\Rox)_{\bb P}\) fixed by the group action, this shows that \(\delta_G(\Xthree, c\Rox)>1\) for \(c\in(\cnot,\tfrac{1}{2}]\). By Theorem~\ref{lem:K-ps-G-invariant-delta}, this shows the K-polystability of \((\Xthree, c\Rox)\) for \(c\in (\cnot,\tfrac{1}{2}]\cap\bb Q\).

    It remains to that this is the unique K-polystable non-stable surface on \(\Xthree\) for \(c>\cnot\). This will follow from collecting previous results in this section. Indeed, by Theorems~\ref{thm:local-delta-X_3} and~\ref{thm:X_3-unstable}, if \((\Xthree, cR)\) is K-semistable but not K-stable then \(R|_E\) is smooth, \(R\) only has (isolated) \(A_n\) singularities, and \(R\) has an \(A_n\) singularity \(\pt\) with \(n\geq 3\) such that \(R_{\bb P}|_{\Gamma_{\pt}}\) is non-reduced. Then Theorem~\ref{thm:local-delta-X_3}\eqref{item:A_n-sings-X_3-non-reduced-fiber} implies that \((\Xthree, cR)\) degenerates to \((\Xthree, c\Rox)\).
\end{proof}

\subsection{The non-normal surface \texorpdfstring{\(\Rwall \subset \Xthree\)}{Rwall in Xn} at the wall \texorpdfstring{\(c=\cnot\)}{c=c1}}\label{sec:X_3-Rwall-ss}

We conclude this section by analyzing the stability of the pair $(\Xthree, c\Rwall)$ and proving Theorem~\ref{thm:X_3-unstable}\eqref{item:unstable-non-isolated}.  Recall that $\Rwall$ is the surface whose strict transform on $\bP(1,1,1,2)$ has equation $w^2 = xyz^2$.

\begin{lemma}
    Let \(R\subset\Xthree\) be as in Theorem~\ref{thm:X_3-unstable}, and assume \(R\) is non-normal. If \(\delta(\Xthree, c\Rthree) \geq 1\) for some \(c\in (0,\tfrac{1}{2}]\), then, up to a coordinate change, we have \(R = \Rwall\).
\end{lemma}

\begin{proof}
    Theorem~\ref{thm:X_3-unstable}\eqref{item:unstable-worse-than-A_2-at-E} and \eqref{item:unstable-A_2-at-E} imply that \(R_{\bb P}\) has an \(A_1\) singularity at \(\ptE\). Furthermore, by assumption, the fibers of \(R\to\bb P^2\) are all finite. So we may write the equation of \(R_{\bb P}\) as in Lemma~\ref{lem:equation-of-R_P}. Since \(R\) is non-normal by assumption, the quartic curve \(xyz^2 + z f_3(x,y) + f_4(x,y)\) must contain a double line, a triple line, or a double conic. The only possibility with an \(A_1\) singularity is when the quartic is the union of a double line and two reduced lines, and \(\ptE\) maps to the intersection of the two reduced components. That is, up to a coordinate change, \(R = \Rwall\).
\end{proof}

\begin{lemma}\label{lem:delta_R3wall_on_X3}
    For points $p$ in the singular locus of $\Rwall$, \(\delta_p(\Xthree, c\Rwall) \geq 1\) if and only if $c \le \cnot$.
\end{lemma}

\begin{proof}
    Let $C$ be the singular locus of $\Rwall$, which is the preimage of the curve $(w = z = 0)$ in $\bP(1,1,1,2)$.  Let $\phi\colon X_1 \to W^+$ be the blow-up of this curve in $W^+$, and let $Y \cong \bb F_1$ be the exceptional divisor.  We first compute \(A_{\Xthree,c\Rwall}(Y)/S_{\Xthree,c\Rwall}(Y)\). The log discrepancy is \(A_{\Xthree,c\Rwall}(Y) = 2-2c\).

    Let \(L=\phi^*(\hsmallplus)^*(-K_{\Xthree}-\dee)\). To compute the volume of \(L-uY\), let $f$ denote a fiber of the exceptional divisor $Y \cong \bb F_1$.  Let $E_1 \cong \bP^2$ and $S_1 \cong \bP^2$ denote the (isomorphic) strict transforms of $E^+$ and $S^+$ in $X_1$.  Denote by $l_E$, $l_S$ the class of a line in each surface.  Denote by $H_w \cong \bb F_1 \subset W^+$ the strict transform of the section of the cone $\bP(1,1,1,2)$ given by $(w  =0)$ and by $H_z \cong \bb F_2 \subset W^+$ the strict transform of $(z = 0)$.  The negative section of $H_w \cong \bb F_1$ is the class of $l_E$; denote a fiber by $f_w$.  Similarly, the negative section of $H_z \cong \bb F_2$ is the class of $l_S$; denote a fiber by $f_z$.  The Mori cone is generated by a subset of these curve classes.
    Computing the associated birational contractions to $L- uY$, we find that:

    \begin{enumerate}
        \item At $u = 2-2c$, we contract the fiber $f_z$, contracting $H_z$ to a smooth curve as $H_z \cdot f_z = -1$.
        \item At $u = 3-2c$, we contract the fiber $f_w$; this is not a birational contraction, as $H_w \cdot f_w = 0$, and we find the pseudoeffective threshold is $u = 3-2c$.
    \end{enumerate}

    To find the volume of $L-uY$, we consider two regions $\regionA$ for $0 \le u \le 2-2c$ and $\regionB$ for $2 -2c \le u \le 3-2c$, and we let $P(u)$ denote the positive part of the Nakayama--Zariski decomposition of $L-uY$, adding a subscript to indicate the region. We have: 
    \[ \begin{array}{rll}
    0 \le u \le 2-2c: & P_\regionA(u) = L - uY & N_\regionA = 0 \\
    2 -2c \le u \le 3-2c: & P_\regionB(u) = L - uY + (2 - 2c - u)H_z & N_\regionB(u) = (u - 2 + 2c) H_z.
\end{array} \]
    
    Using that $Y \cong \bb F_1$ with fiber class $f$ and negative section $\sigma$ equivalent to $2f_z + l_S$,
\begin{detail}
    we have the following intersection numbers: 
    \[ L \cdot  f = 0, \quad L\cdot \sigma = 5-4c, \quad Y \cdot f = -1, \quad Y \cdot l_S = 2,\]
    and then
\end{detail}
    expanding the intersection formula yields
    \[ \vol (P_\regionA) = (P_\regionA)^3 = (L- uY)^3 = 3(2-2c)(3-2c)^2 - 3u^2(5-4c) + 3u^3. \]
    To find the volume of $P_\regionB$, we apply Lemma \ref{lem:volume-lemma-3} for the contraction of $H_z$, and
\begin{detail}
    plugging in $p = 2-2c-u$, $q = 1$, $n = 2$, $a = 1$, and $b = 1$,
\end{detail}
we find
    \[ \vol (P_\regionB) = (P_\regionB)^3 = 3(2-2c)(3-2c)^2 - 3u^2(5-4c) + 3u^3 +3(2-2c-u)^2(3-2c-u). \]

    We then compute 
    \[ S_{\Xthree,c\Rwall}(Y) = -\frac{(14c^3 - 62c^2 + 91c - 44)}{3(3-2c)^2} , \qquad \frac{A_{\Xthree,c\Rwall}(Y)}{S_{\Xthree,c\Rwall}(Y)} \ge 1 \iff c \le \cnot .\]
    This coincides exactly with the computation in Lemma \ref{lem:Rthree-delta-C}.  Furthermore, the restrictions of $P(u)$ and $N(u)$ to $Y$ yield exactly the same terms as in Lemma \ref{lem:Rthree-delta-C}.  Finally, we have a group action $G_3'$ on $\Xthree$ fixing $\Rwall$. Namely, in coordinates on $\bP(1,1,1,2)$, the elements $\phi_\lambda\colon [x:y:z:w] \mapsto [x:y: \lambda z: \lambda w]$ and $\iota\colon [x:y:z:w] \mapsto [y:x:z:w]$ fix $\Rwall$ and the points $\ptE$ and $\conept$, so they extend to automorphisms of the pair $(\Xthree, \Rwall)$.  In $Y$, these have the same $G_3'$-invariant curves, and thus applying the Abban--Zhuang method and the argument in Lemma~\ref{lem:Rthree-delta-C} implies that 
    \(\delta_p(\Xthree, c\Rwall) \geq 1 \) for \(p\in C\) and \(c\leq\cnot\).
\end{proof}

Using the previous lemma, we conclude the following. 

\begin{corollary}
        The log Fano pair \((\Xthree, c\Rwall)\) has \(\delta(\Xthree, c\Rwall) \geq 1\) if and only if \(c = \cnot\).
\end{corollary}

\begin{proof}
    For any points $p$ of $\Rwall$ contained in $E$, we have \(\delta_p(\Xthree, c\Rwall) \geq 1\) if and only if $c \ge \cnot$ by Lemma \ref{lem:X3-smooth-in-E} which gives one direction of the statement.  For any point $p \in \Xthree \setminus E$ away from the singular locus of $\Rwall$, we also have \(\delta_p(\Xthree, c\Rwall) \geq 1\) for $c \ge \cnot$ by Lemma~\ref{lem:X3-smooth-notin-E}.
    Conversely, by Lemma \ref{lem:delta_R3wall_on_X3}, we conclude that for $p$ a singular point of $\Rwall$, \(\delta_p(\Xthree, c\Rwall) \geq 1\) if and only if $c \le \cnot$.  Putting these together, we have \(\delta(\Xthree, c\Rwall) \geq 1\) if and only if \(c = \cnot\).
\end{proof}

\section{Singular plane quartics and GIT}\label{sec:GIT-higher-An-quartics}

The GIT polystability of plane quartics is well known (see Section~\ref{sec:GIT-quartics}). In this section, we recall results of Wall on quartics that have higher \(A_n\) singularities and are not GIT polystable.
First, Section~\ref{sec:A_1+A_n-Wall} addresses quartics with an \(A_1\) singularity and a higher \(A_n\) singularity in Section~\ref{sec:A_1+A_n-Wall}. Section~\ref{sec:unstable-An-quartics} contains results on the GIT unstable quartics with \(A_n\) singularities.
Finally, Section~\ref{sec:nodalGIT} defines an alternative GIT for curves arising as discriminants of K-semistable pairs $(\Xthree, cR)$. 
For consistency with \cite{Wall-quartics,Wall91}, in this section we use \([x:y:z]\) for coordinates on \(\bb P^2\), instead of \([y_0:y_1:y_2]\).
Recall that a quartic with only \(A_1\) or \(A_2\) singularities is GIT stable.

If a quartic curve \(\Delta\) has an \(A_n\) singularity with \(n\geq 3\), then \cite[Section 3]{Wall-quartics} shows that the defining equation may be written in a particular form: after changing coordinates so this \(A_n\) singularity is at \([0:0:1]\) and the tangent is \((x=0)\), \(\Delta\) is defined by an equation of the form \((xz-\lambda y^2)^2 - f(x,y)\). Wall then classifies the singularities of \(\Delta\) depending on \(\lambda\) and the coefficient of \(y^4\) in \(f(x,y)\), giving the following cases \cite[pages 421--422]{Wall-quartics}:
\begin{itemize}
    \item \(\alpha,\beta\): \(\coeff_{f(x,y)} (y^4)\neq 0\) and \(\coeff_{f(x,y)} (y^4)\neq \lambda^2\). Degenerating along the 1-PS \(\diag(1,0,-1)\) shows \(\Delta\) is S-equivalent to a cat-eye.
    \item \(\gamma\): \(\lambda\neq 0\) and \(\coeff_{f(x,y)} (y^4) = 0\). Then the 1-PS \(\diag(1,0,-1)\) shows \(\Delta\) is S-equivalent to a double conic.
    \item \(\delta\): \(\coeff_{f(x,y)} (y^4) = \lambda^2 \neq 0\). Then the 1-PS \(\diag(1,0,-1)\) shows \(\Delta\) is S-equivalent to an ox.
    \item \(\epsilon\): \(\coeff_{f(x,y)} (y^4) = \lambda = 0\). In this case, the defining equation of \(\Delta\) is \(x(xz^2-y(x-y)(x-ty))\), \(x(xz^2-y^2(x-y))\), or \(x(xz^2-y^3)\). The singular point \([0:0:1]\) is an \(A_5\) singularity, and the 1-PS \(\diag(4,-1,-3)\) shows \(\Delta\) is GIT unstable.
\end{itemize}
The numbers in the cases in \cite{Wall-quartics} list the multiplicities of the factors of \(f(x,y)\), and the underlined number is the multiplicity of the factor \(x\). In cases \(\alpha,\beta,\gamma,\delta\), the quartic \(\Delta\) is GIT semistable but not polystable. \(\Delta\) is GIT unstable in case \(\epsilon\).

There are five cases in \cite[page 422]{Wall-quartics} with \(A_5\) singularities: the \(\epsilon\) cases, which are the union of a cubic curve and an inflectional tangent line; \(\gamma\) (\underline{2}11), which have an oscnode; and \(\gamma\) (\underline{2}2), which is the union of two conics meeting at two points with multiplicities 3 and 1.

\subsection{Quartics with a higher \texorpdfstring{\(A_n\)}{An} + \texorpdfstring{\(A_1\)}{A1} singularity}\label{sec:A_1+A_n-Wall}

In this section, we focus on the cases from \cite{Wall-quartics} that have an \(A_1\) singularity and an \(A_n\) singularity for some \(n\geq 3\).
We will use this analysis in the proof of Theorem~\ref{thm:only-have-X_n-or-P1P2} in the following section.

From \cite[page 422]{Wall-quartics} (note that there is a typographical error in $\delta$ (22), which should read $A_3 + 3A_1$), these quartics are as listed in the following table. The unique GIT unstable curve in this list is case \(\epsilon\) (\underline{1}21), which is the union of a nodal cubic and a line meeting the cubic at a smooth point to order 3.

\[ \begin{array}{c|l|l}
    \text{Case in \cite{Wall-quartics}} & \text{Equation} & \text{Singularities} \\
    \hline
    \alpha (211) & (xz-\lambda y^2)^2 - y^2(y-x)(y-tx) \text{ with }\lambda^2 \neq 0, 1 \text{ and }t\neq 1 & A_3 + A_1 \\
    \alpha (22) & (xz-\lambda y^2)^2 - y^2(y-x)^2 \text{ with }\lambda^2 \neq 0, 1 & A_3 + 2 A_1 \\
    \beta (211) & x^2z^2 - y^2(y-x)(y-tx) \text{ with }t\neq 1 & A_3 + A_1 \\
    \beta (22) & x^2z^2 - y^2(y-x)^2 & A_3 + 2 A_1 \\
    \gamma (2\underline{1}1) & (xz- y^2)^2 - A xy^2 (x-y) & A_4 + A_1 \\
    \gamma (\underline{2}2) & (xz- y^2)^2 - x^2 y^2 & A_5 + A_1 \\
    \delta(1111) & (xz- y^2)^2 - y(y-x)(y-t_1x)(y-t_2x) \text{ with }t_i\neq 0, 1\text{; }t_1\neq t_2 & A_3 + A_1 \\
    \delta(211) & (xz- y^2)^2 - y^2(y-x)(y-tx) \text{ with }t\neq 0, 1 & A_3 + 2 A_1 \\
    \delta(31) & (xz- y^2)^2 - y^3(y-x) & A_3 + A_2 + A_1 \\
    \delta (22) & (xz- y^2)^2 - y^2(y-x)^2 & A_3 + 3 A_1 \\
    \delta(4) & (xz- y^2)^2 - y^4 & 2 A_3 + A_1 \\
    \epsilon (\underline{1}21) & x(xz^2 - xy^2 + y^3) & A_5 + A_1 .
\end{array} \]

Next, we show that, away from the cases \(\delta\) (1111), (31), and (4), the higher \(A_n\) + \(A_1\) quartics can be written in a certain form.
\begin{lemma}\label{lem:nodal-quartics-211-22}
    Every quartic curve of type (211) or (22) in \cite[Section 3]{Wall-quartics} can be written in the form $x^2z^2 - 2s xz y^2 + r y^4 + t xy^3 - x^2 y^2$ for some \([s:t:r]\in\bb P(1,1,2)\). For any \([s:t:r]\), the corresponding quartic has an \(A_1\) singularity at \([1:0:0]\).
\end{lemma}

\begin{proof}
    First, a direct computation shows that for any \(s,t,r\), the quartic has an \(A_1\) singularity at \([1:0:0]\).
\begin{detail}
    Indeed, on the chart \(x\neq 0\), the equation is \(z^2 - 2s z y^2 + r y^4 + t y^3 - y^2\), so the Jacobian and Hessian matrices on this chart are
    \[\begin{pmatrix}
        - 4 s y z + 4 r y^3 + 3 t y^2 - 2 y & -2 s y^2 + 2 z
    \end{pmatrix}, \qquad
    \begin{pmatrix}
        -2 + 6 t y + 12 r y^2 - 4 s z & -4 s y \\
        -4 s y & 2
    \end{pmatrix},\]
    respectively. At \([1:0:0]\) the Jacobian has rank 0 and the Hessian is \(\diag(-2,2)\), so we see that \([1:0:0]\) is an \(A_1\) singularity.
\end{detail}    
    Next, we check case by case, using Wall's equations as listed above. First, the curve $\epsilon$ (\underline{1}21) is parametrized by the point $[0:1:0]$.
    Away from this point, the curves parametrized by $(r = 0)$ correspond to $\delta$ (211) and (22); those parametrized by $(s^2 = 0)$ correspond to $\beta$ (211) and (22); those parametrized by $(r = s^2)$ correspond to $\gamma$ (211) and (22); and those parametrized by $r = as^2$ for some $a \in \bb C$, $a \ne 0, 1$ correspond to $\alpha$ (211) and (22).
\end{proof}

Notice that in each of the \(\delta\) cases, the quartic \(\Delta\) defined by \((xz- y^2)^2 - y(y-t_1x)(y-t_2x)(y-t_3x)\) is the union of the line \((x=0)\) and a cubic meeting at the points \([0:0:1]\) and \([0:2:t_1+t_2+t_3]\), and this latter point is an \(A_1\) singularity of \(\Delta\).
\begin{detail}
Indeed, the equation factors as
\begin{gather*} (xz- y^2)^2 - y(y-t_1x)(y-t_2x)(y-t_3x) \\= x (t_1 t_2 t_3 x^2 y - t_1 t_2 x y^2 - t_1 t_3 x y^2 - t_2 t_3 x y^2 + t_1 y^3 + t_2 y^3 + t_3 y^3 + x z^2 - 2 y^2 z) \end{gather*}
and the restriction of the cubic to \((x=0)\) is \(y^2 ((t_1 + t_2 + t_3) y - 2 z)\).
\end{detail}
In cases \(\delta\) (1111), (31), and (4), there is exactly one \(A_1\) singularity on \(\Delta\), given by the intersection of the line \((x=0)\) and the conic. In cases \(\delta\) (211) and (22), \(\Delta\) contains at least two \(A_1\) singularities, only one of which lies on the line \((x=0)\).

Finally, we list the quartics with more than one \(A_1\) singularity. These are cases \(\alpha\) (22), \(\beta\) (22), \(\delta\) (211), and \(\delta\) (22). Using the same coordinates as above, the \(A_3\) singularity is at \([0:0:1]\), and:
\begin{itemize}
\item
\(\alpha\) (22): The two \(A_1\) singularities are \([1 : 1 : \lambda]\) and \([1:0:0]\).
\item
\(\beta\) (22): The two \(A_1\) singularities are \([1:1:0]\) and \([1:0:0]\).
\item
\(\delta\) (211): The defining equation is \(x (-t x y^2 + t y^3 + x z^2 + y^3 - 2 y^2 z)\) and the two \(A_1\) singularities are \([0:2:t+1]\) and \([1:0:0]\).
\item
\(\delta\) (22): The defining equation is \(x (z - y) (x y + x z - 2 y^2)\) and the three \(A_1\) singularities are \([0:1:1]\), \([1:1:1]\), and \([1:0:0]\).
\end{itemize}

\begin{remark}\label{rem:nodes-are-equivalent}
We may summarize the previous list as follows: in each (22) case, there are exactly two $A_1$ singularities away from the tangent line \((x=0)\) to the curve at the $A_3$ singularity. These points are $[1:0:0]$ and $[1: 1: \lambda]$, and it is straightforward to show that these points can be interchanged by an isomorphism of the curve.  
\begin{detail}
    Let $\Delta$ be given by the equation \[(xz-\lambda y^2)^2 - y^2(y-x)^2.\]  If $\lambda = 0$, this is case $\beta$ (22); if $\lambda = 1$, this is case $\delta$ (22); if $\lambda \ne 0,1$ this is case $\alpha$.  Apply the isomorphism switching $y$ and $y-x$, i.e. $y \mapsto y-x$ and $x \mapsto -x$ to obtain the equation \[(-xz-\lambda (y-x)^2)^2 - (y-x)^2y^2.\]
    Multiplying out the first term, it can be written as $x(-z+2\lambda y - \lambda x) - \lambda y^2$, and applying the isomorphism $-z+2\lambda y - \lambda x \mapsto z$, we find the same equation of $\Delta$ and hence the composition is an automorphism of $\Delta$.  By construction, this automorphism interchanges the points $[1:0:0]$ and $[1:1: \lambda]$. 
\end{detail}
Therefore, we may assume without loss of generality that a node away from the tangent line to the curve at the $A_3$ singularity is the point $[1:0:0]$.
\end{remark}

\subsection{GIT unstable quartics with \texorpdfstring{\(A_n\)}{An} singularities}\label{sec:unstable-An-quartics}

We now turn to the GIT unstable quartics with \(A_n\) singularities. These are cases \(\epsilon\) (\underline{1}111), (\underline{1}21), and (\underline{1}3) in the classification of \cite[page 422]{Wall-quartics}, and these are the union of a cubic and inflectional tangent line.
In case \(\epsilon\) (\underline{1}111) the cubic is smooth, in case \(\epsilon\) (\underline{1}21) the cubic is nodal, and in case \(\epsilon\) (\underline{1}3) the cubic is cuspidal.

To match the notation in \cite{Wall91}, we make a coordinate change so that:
\[ \begin{array}{c|l|l}
    \text{Case in \cite{Wall-quartics}} & \text{Equation} & \text{Singularities} \\
    \hline
    \epsilon (\underline{1}111) & z(y^2 z - x(z-x)(\frac{1}{t} z - x)) \text{ with }t\neq 0, 1 & A_5 \\
    \epsilon (\underline{1}21) & z(y^2 z - x^3 - x^2 z) & A_5 + A_1\\
    \epsilon (\underline{1}3) & z(y^2 z - x^3) & A_5 + A_2 .
\end{array} \]
\begin{detail}
    In all cases, we first perform the coordinate change \(\begin{pmatrix} 0 & 0 & 1 \\ 1 & 0 & 0 \\ 0 & 1 & 0 \end{pmatrix}\). In case \(\epsilon\) (\underline{1}3) we're done.
    In case \(\epsilon\) (\underline{1}111), perform the further coordinate change \(y\mapsto \sqrt{t} y\) and factor out \(t\). Note that in case \(\epsilon\) (\underline{1}111) we have \(t\neq 0\) (otherwise \(\Delta\) contains the double line \((x^2=0)\)) and \(t\neq 1\) (otherwise we are in the \(\epsilon\) (\underline{1}21) case).
    In case \(\epsilon\) (\underline{1}21), instead perform the further coordinate change \(z\mapsto -z\).
\end{detail}
In each case \(\Delta\) has an \(A_5\) singularity at \([0:1:0]\). In cases \(\epsilon\) (\underline{1}21) and (\underline{1}3) there is an additional singular point at \([0:0:1]\).

A tedious but straightforward computation shows the following:
\begin{lemma}\label{lem:epsilon-quartics-coordinate-changes}
    Let \(\Delta\) be a quartic of type \(\epsilon\) (\underline{1}111) or (\underline{1}21).
    \begin{enumerate}
        \item Assume \(\Delta\) has type \(\epsilon\) (\underline{1}111) and has defining equation \(f_t\coloneqq z(y^2 z - x(z-x)(\frac{1}{t} z - x))\). If \(t \neq -1\), then the only coordinate changes on \(\bb P^2\) that preserve \(f_t\) are the identity and \(\diag(1,-1,1)\). If \(t = -1\), then there are two additional coordinate changes \(\diag(1, i , -1)\) and \(\diag(1, - i , -1)\).

        The quartics defined by \(f_t\) and \(f_{1/t}\) are equivalent by \(\diag(1, \pm 1/\sqrt{t}, t)\).
        Furthermore, for \(t' \notin \{t, \tfrac{1}{t}\}\) the quartics defined by \(f_t\) and \(f_{t'}\) are not \(\Aut(\bb P^2)\)-equivalent.
        \item Assume \(\Delta\) has type \(\epsilon\) (\underline{1}21). Then the only coordinate changes on \(\bb P^2\) that preserve \(\Delta\) are the identity and \(\diag(1,-1,1)\).
    \end{enumerate}
\end{lemma}

We will study the GIT semistable \((2,2)\)-surfaces in \(\bb P^1\times\bb P^2\) whose associated quartics are in case \(\epsilon\).
\cite{Wall91} shows that every plane quartic can be written in the form \(Q_1 Q_3-Q_2^2\), unless the quartic is the union \(y(y^2 z- x^3)\) of a cuspidal cubic with its cuspidal tangent. For all other cases of the union of a line and an irreducible cubic meeting at a point \(\pt\), \cite[Pages 421--423]{Wall91} gives formulas for all possible \(Q_i\) under the assumption that the \(Q_i\) do not all simultaneously vanish at \(\pt\). We recall the relevant formulas in the cases of the \(\epsilon\) quartics below.

\begin{proposition}[{\cite{Wall91}}]\label{prop:Q_i-for-epsilon-curves}
    Let \(\Delta\) be a quartic of type \(\epsilon\). Let \(R=(t_0^2 Q_1 + 2t_0 t_1 Q_2 + t_1^2 Q_3 = 0)\) be a \((2,2)\)-surface with \(\Delta=(Q_1Q_3-Q_2^2=0)\) such that the morphism \(\piR_2\colon R\to\bb P^2\) has a finite fiber over the \(A_5\) singularity of \(\Delta\).
    \begin{enumerate}
        \item Then \(\Delta\) cannot be the \(\epsilon\) (\underline{1}3) curve. That is, \(\Delta\) is necessarily in case \(\epsilon\) (\underline{1}111) or (\underline{1}21).
        \item Assume \(\Delta\) is in case \(\epsilon\) (\underline{1}111), and write its defining equation \(z(y^2 z - x(z-x)(\frac{1}{t} z - x))\) for some \(t\neq 0,1\). Then the \(Q_i\) are given as follows, for some \(b\in \bb C\):
        
    \begin{equation*}
        \begin{split}
            Q_1 =& -x^2b^2 + 2xyb + \tfrac{(t + 1)}{t}xzb^2 - \tfrac{(t - 1)^2}{4 t^2}xz - y^2 - \tfrac{t + 1}{t}yzb - \tfrac{1}{t}z^2b^2 \\
            Q_2 =& -x^2b(b + \tfrac{2t}{t-1}) + xyb + xy(b + \tfrac{2t}{t-1}) + \tfrac{t + 1}{t}xzb(b + \tfrac{2t}{t-1}) - \tfrac{(t-1)^2}{4 t^2}xz - y^2 - \tfrac{t + 1}{2t}yzb \\ \hphantom{=}& - \tfrac{t + 1}{2t}yz(b + \tfrac{2t}{t-1}) - \tfrac{1}{t}z^2b(b + \tfrac{2t}{t-1}) \\
            Q_3 =& -x^2(b + \tfrac{2t}{t-1})^2 + 2xy(b + \tfrac{2t}{t-1}) + \tfrac{t + 1}{t}xz(b + \tfrac{2t}{t-1})^2 - \tfrac{(t-1)^2}{4 t^2}xz - y^2 - \tfrac{t + 1}{t}yz(b + \tfrac{2t}{t-1}) \\ \hphantom{=}& - \tfrac{1}{t}z^2(b + \tfrac{2t}{t-1})^2 .
        \end{split}
    \end{equation*}
        \item Assume \(\Delta\) is in case \(\epsilon\) (\underline{1}21), and write its defining equation \(z(y^2 z - x^3 - x^2 z)\). Then the \(Q_i\) are given as follows, for some \(b\in \bb C\):
    \begin{equation*}
        \begin{split}
            Q_1 &= -x^2b^2 + 2xyb - xz(b^2 + \tfrac{1}{4}) - y^2 + yzb \\
            Q_2 &= -x^2b(b+2) + 2xyb + 2xy - xzb^2 - 2xzb - \tfrac{1}{4}xz - y^2 + yzb + yz \\
            Q_3 &= -x^2(b+2)^2 + 2xy(b+2) - xz((b+2)^2 + \tfrac{1}{4}) - y^2 + yz(b + 2) .
        \end{split}
    \end{equation*}
    \end{enumerate}
    Any choice of \(b\in\bb C\) satisfies \(\Delta = (Q_1 Q_3 - Q_2^2 = 0)\).
    
    Furthermore, unless \(b'=-b - \tfrac{2t}{t-1}\) (case \(\epsilon\) (\underline{1}111)) or \(b'=b-2\) (case \(\epsilon\) (\underline{1}21)), distinct choices of \(b\neq b'\in\bb C\) give non-\(\Aut(\bb P^1\times\bb P^2)\)-equivalent \((2,2)\)-surfaces \(R\).
\end{proposition}

\begin{proof}
    Let \(\Gamma\) be a reducible quartic defined by \(z(y^2z - x^3 - 2ux^2z - vxz^2)\); \(\Gamma\) is the union of a line and a cubic meeting at the point \([0:1:0]\). Then \cite[Pages 421--423]{Wall91} shows that the \((2,2)\)-surfaces \(R = (t_0^2 Q_1 + 2 t_0 t_1 Q_2 + t_1^2 Q_3 = 0)\) such that \(\Gamma = (Q_1 Q_3 - Q_2^2=0)\) and the fiber of \(\piR_2\colon R \to \bb P^2\) over \([0:1:0]\) is finite are precisely those of the form
    \begin{equation*}
        \begin{split}
            Q_1 =& - vb^2z^2 + (-4u^2 + 2u(2a-b^2)-a^2+v)xz - y^2 + (4bu-2ab)yz + (2a-2u-b^2)x^2 + 2bxy \\
            Q_2 =& - vbdz^2 + (-4u^2 + 2u(a+c-bd)-ac+v)xz - y^2 + (2u(b+d)-ad-bc)yz \\ \hphantom{=} & + (a+c-2u-bd)x^2 + (b+d)xy \\
            Q_3 =& - vd^2z^2 + (-4u^2 + 2u(2c-d^2)-c^2+v)xz - y^2 + (4du-2cd)yz + (2c-2u-d^2)x^2 + 2dxy 
        \end{split}
    \end{equation*}
    where \(\alpha \coloneqq a-c\), \(\beta \coloneqq b-d\), \(\gamma \coloneqq ad-bc\) satisfy \(-\alpha^2 + 2u\beta^2 + 2\beta\gamma = 0\), \(2\alpha\beta=0\), and \((\gamma+2u\beta)^2-v\beta^2=1\) \cite[Page 422, Equation (1)]{Wall91}.
    This system has no solutions if \(u^2-v=0\).
\begin{detail}
    In more detail, setting \(\beta=0\) leads to a contradiction, so we must have \(\alpha = 0\) and \(\beta \neq 0\). Then substituting \(2\beta(u\beta+\gamma)=0\) into the third equation gives \((u^2-v)\beta = 1\), which is impossible if \(u^2-v=0\).
\end{detail}
    If \(u^2-v \neq 0\), then solving this system, we find that \(a=c=u\) and \(b = d \pm \frac{1}{\sqrt{u^2-v}}\).
\begin{detail}
    This yields
    \begin{equation*}\begin{split}
    Q_1 &= -vb^2z^2+(-2 b^2 u - u^2+v)xy-y^2 + 2 b u yz -b^2 x^2 + 2bxy \\
    Q_2 &= -vbdz^2+(-2 b d u - u^2 + v)xy-y^2 + u (b + d)yz -bdx^2 + (b+d)xy\\
    Q_3 &= -vd^2z^2+(-2 d^2 u - u^2+v)xy-y^2 + 2duyz -d^2 x^2 + 2dxy
    \end{split}\end{equation*}
    with \(b = d \pm \frac{1}{\sqrt{u^2-v}}\).
\end{detail}

    \noindent\underline{Case \(\epsilon\) (\underline{1}111):} From the equations we have \(u = - \frac{1+\frac{1}{t}}{2}\) and \(v = \frac{1}{t}\). Taking \(d = b + \tfrac{2t}{t-1}\) yields the claimed formulas for \(Q_i\). The \(d = b - \tfrac{2t}{t-1}\) case is obtained by the coordinate change \(t_0 \leftrightarrow t_1\) on \(\bb P^1\), and a straightforward computation shows that if \(t_0 Q_1 + t_1 Q_3 = -x^2b'^2 + 2xyb' + \tfrac{(t + 1)}{t}xzb'^2 - \tfrac{(t - 1)^2}{4 t^2}xz - y^2 - \tfrac{t + 1}{t}yzb' - \tfrac{1}{t}z^2b'^2\) for some \(b'\) and \([t_0:t_1]\in\bb P^1\), then \([t_0:t_1] = [1:0]\) or \([0:1]\). For the final claim, by construction, any \(b'\) defines \(Q_i'\) with \(Q_2'^2-Q_1'Q_3'\) defining the same quartic \(\Delta\). Lemma~\ref{lem:epsilon-quartics-coordinate-changes} describes \(\Aut(\bb P^2,\Delta)\), and from this, one can verify that \(b'\neq b\) gives a non-\(\Aut(\bb P^1\times\bb P^2)\)-equivalent \((2,2)\)-surface unless \(b'=-b - \tfrac{2t}{t-1}\).
\begin{detail}
    Indeed, for \(t\neq -1\) it suffices to check the coordinate change \(C = \diag(1,-1,1)\). This sends
    \begin{equation*}
        \begin{split}
            \phi_C^* Q_1 =& -x^2b^2 + 2xy(-b) + \tfrac{(t + 1)}{t}xzb^2 - \tfrac{(t - 1)^2}{4 t^2}xz - y^2 - \tfrac{t + 1}{t}yzb - \tfrac{1}{t}z^2b^2 \\
            \phi_C^* Q_3 =& -x^2(-b - \tfrac{2t}{t-1})^2 + 2xy(-b - \tfrac{2t}{t-1}) + \tfrac{t + 1}{t}xz(-b - \tfrac{2t}{t-1})^2 - \tfrac{(t-1)^2}{4 t^2}xz - y^2 - \tfrac{t + 1}{t}yz(-b - \tfrac{2t}{t-1}) \\ \hphantom{=}& - \tfrac{1}{t}z^2(-b - \tfrac{2t}{t-1})^2 ,
        \end{split}
    \end{equation*}
    and the further coordinate change \(t_0\leftrightarrow t_1\) shows that \(-b - \tfrac{2t}{t-1}\) defines the same \((2,2)\)-surface as \(b\). For \(t=-1\), an explicit computation shows that the pullback of \[Q_3 = -x^2(b + 1)^2 + 2xy(b + 1) - xz - y^2 + z^2(b + 1)^2\] by the coordinate change \(\diag(1,i,-1)\) (resp. \(\diag(1,-i,-1)\)) does not preserve this presentation of \(Q_3\).
\end{detail}

    \noindent\underline{Case \(\epsilon\) (\underline{1}21):} From the equations we have \(u = \frac{1}{2}\) and \(v=0\). Taking \(d=b+2\) yields the claimed formulas for \(Q_i\). The \(d=b-2\) case is obtained by the coordinate change \(t_0 \leftrightarrow t_1\) on \(\bb P^1\), and a straightforward computation shows that if \(t_0 Q_1 + t_1 Q_3 = -x^2b'^2 + 2xyb' - xzb'^2 - \tfrac{1}{4}xz - y^2 + yzb'\) for some \(b'\) and \([t_0:t_1]\in\bb P^1\), then \([t_0:t_1] = [1:0]\) or \([0:1]\). The uniqueness claim follows using the description of \(\Aut(\bb P^2,\Delta)\) in Lemma~\ref{lem:epsilon-quartics-coordinate-changes}, as above.
\begin{detail}
    Indeed, it suffices to check the coordinate change \(C = \diag(1,-1,1)\). This sends
    \begin{equation*}
        \begin{split}
            \phi_C^* Q_1 &= -x^2(-b)^2 + 2xy(-b) - xz((-b)^2 + \tfrac{1}{4}) - y^2 + yz(-b) \\
            \phi_C^* Q_3 &= -x^2(-b-2)^2 + 2xy(-b-2) - xz((-b-2)^2 + \tfrac{1}{4}) - y^2 + yz(-b - 2) ,
        \end{split}
    \end{equation*}
    and the further coordinate change \(t_0\leftrightarrow t_1\) shows that \(-b - 2\) defines the same \((2,2)\)-surface as \(b\).
\end{detail}

    \noindent\underline{Case \(\epsilon\) (\underline{1}3):} In this case \(u=v=0\). Then \(u^2-v=0\), so as observed above, it is impossible to solve for values of \(a,b,c,d\) in the equations of the \(Q_i\).
\end{proof}

\begin{corollary}\label{cor:surfaces-R-for-epsilon-curves}
    Let \(\Delta\) be a quartic of type \(\epsilon\) (\underline{1}111), (\underline{1}21), or (\underline{1}3), and let \(R\subset\bb P^1\times\bb P^2\) be a \((2,2)\)-surface with discriminant curve \(\Delta\).
    \begin{enumerate}
        \item If \(\Delta\) has type \(\epsilon\) (\underline{1}111) or (\underline{1}21), then \(R\) is GIT stable if and only if the fiber of \(\piR_2\colon R \to \bb P^2\) over the \(A_5\) singularity of \(\Delta\) is finite. Furthermore, the locus in \(\GITR\) with discriminant curve \(\Delta\) is 1-dimensional.
        \item If \(\Delta\) has type \(\epsilon\) (\underline{1}3), then \(R\) is GIT unstable.
    \end{enumerate}
\end{corollary}

\begin{proof}
    \underline{Case \(\epsilon\) (\underline{1}111):} In this case, \(\Delta\) has exactly one singular point at \([0:1:0]\). First, assume by contradiction that the fiber of \(\piR_2\) over this point is not finite and that \(R\) is GIT semistable. Then Proposition~\ref{prop:A_n-sings-not-stable}\ref{item:higher-A_n-not-stable-non-finite}, Remark~\ref{rem:frakG_i-surfaces-descriptions}, and Corollary~\ref{cor:discriminant-family} imply that \(\Delta\) admits a 1-parameter degeneration to a GIT polystable quartic, which is impossible since \(\Delta\) is GIT unstable.
    On the other hand, if the fiber of \(\piR_2\) over \([0:1:0]\) is finite, then \(R\) is defined by \(t_0^2 Q_1 + 2 t_0 t_1 Q_2 + t_1^2 Q_3\) where \(Q_i\) have the form in Proposition~\ref{prop:Q_i-for-epsilon-curves}. Then a direct computation shows that \(([-1:1],[0:1:0])\) is the only singular point of \(R\) and that the fiber of \(\piR_1\colon R\to\bb P^1\) over \([-1:1]\) is reduced. Then \((\bb P^1\times\bb P^2, cR)\) is K-stable for \(c\in (0, 0.7]\) by Theorem~\ref{thm:local-delta-A_n-singularities} and \cite{CFKP23}, and in particular \(R\) is GIT stable by Theorem~\ref{thm:K-implies-GIT}.
    Furthermore, Proposition~\ref{prop:Q_i-for-epsilon-curves} shows that there is a 1-dimensional family in \(\GITR\) of GIT stable \((2,2)\)-surfaces with discriminant \(\Delta\).

    \underline{Case \(\epsilon\) (\underline{1}21):} \(\Delta\) has an \(A_5\) singularity at \([0:1:0]\) and an \(A_1\) singularity at \([0:0:1]\). If the fiber of \(\piR_2\) over the \(A_5\) singularity is not finite, then \(R\) is GIT unstable by the same argument as in \(\epsilon\) (\underline{1}111). If instead the fiber of \(\piR_2\) over the \(A_5\) singularity is finite, then an explicit computation using the \(Q_i\) from Proposition~\ref{prop:Q_i-for-epsilon-curves} shows that \(R\) has exactly one singular point at \(([-1:1],[0:1:0])\) and that the fiber of \(\piR_1\colon R\to\bb P^1\) over \([-1:1]\) is reduced. (The fiber of \(\piR_2\) over the \(A_1\) singularity of \(\Delta\) is non-finite.) Then \((\bb P^1\times\bb P^2, cR)\) is K-stable for \(c\in (0, 0.7]\) by Theorem~\ref{thm:local-delta-A_n-singularities} and \cite{CFKP23}, and in particular \(R\) is GIT stable by Theorem~\ref{thm:K-implies-GIT}.
    Furthermore, Proposition~\ref{prop:Q_i-for-epsilon-curves} shows that there is a 1-dimensional family in \(\GITR\) of GIT stable \((2,2)\)-surfaces with discriminant \(\Delta\).
    
    \underline{Case \(\epsilon\) (\underline{1}3):} Let \(R\) be a \((2,2)\)-surface with discriminant \(\Delta\). By Proposition~\ref{prop:Q_i-for-epsilon-curves} the fiber of \(\piR_2\) over the \(A_5\) singularity is necessarily non-finite, so $R$ is GIT unstable by the above argument.
\end{proof}

\subsection{Degenerations of nodal quartics}\label{sec:nodalGIT}

To study the locus of quartic curves that appear as discriminants on K-semistable pairs $(\Xthree, cR)$, we define an alternate GIT for nodal quartic curves.  

Recall from Lemma \ref{lem:equation-of-R_P}, every such $R$ appearing in a K-semistable pair $(\Xthree, cR)$ for $c \in (\cnot, \frac{1}{2})$ can be written as the strict transform of the surface $w^2 = xyz^2 + z(ax^3 + by^3) + f_4(x,y)$ on $\bP(1,1,1,2)$.  Up to isomorphism, there is thus a one-to-one correspondence of $R$ with the nodal quartic $xyz^2 + z(ax^3 + by^3) + f_4(x,y)$, where the node is at the point $[0:0:1]$.  From Theorem~\ref{thm:local-delta-X_3}, all $A_n$ singularities $\pt\in R$ have $\delta_{\pt}(\Xthree, cR) \ge 1$, i.e. surfaces with these singularities yield K-semistable pairs.  Therefore, we cannot expect a morphism to $\GITquartic$ in general. Indeed, choosing the nodal quartic to be the union of a nodal cubic and an inflectional line as in $\epsilon$ (\underline{1}21) yields $R$ with an $A_5$ singularity, and we find that $(\Xthree, cR)$ is K-stable but the quartic discriminant is GIT unstable.  In fact, Theorem~\ref{thm:local-delta-X_3} and the classification of GIT semistable quartics (Section~\ref{sec:A_1+A_n-Wall}) imply that is the unique unstable discriminant:

\begin{corollary}\label{cor:one-unstable-delta-on-Xthree}
    For $c > \cnot$, there is a unique (up to isomorphism) K-semistable pair $(\Xthree, cR)$ with GIT unstable discriminant curve.  The associated discriminant is the union of a nodal cubic curve and a line meeting the cubic at a smooth point to order 3. This pair is K-stable for $c \in (\cnot, \frac{1}{2}]$. 
\end{corollary}

We define an alternative GIT to capture all quartic discriminants of K-semistable surfaces on $\Xthree$.  We are thankful to Yuchen Liu for suggesting this approach.  Given a quartic curve in $\bP^2$ with equation $xyz^2 + z(ax^3 + by^3) + f_4(x,y) = 0$, we identify this with a point in the vector space $\mathbb{A} \coloneqq \mathbb{C} x^3 \oplus \mathbb{C} y^3 \oplus H^0(\bP^1, \mathcal{O}(4))$.  Consider the quotient by the action of $\mathbb{G}_m$, acting with weight 1 on $\mathbb{C} x^3 \oplus \mathbb{C} y^3$ and with weight 2 on $H^0(\bP^1, \mathcal{O}(4))$.  Acting by $\lambda \in \mathbb{G}_m$ on the coefficients of the equation defines an isomorphic curve (seen easily by scaling $z$ by $\lambda$), so we may identify the quartic with a point in the weighted projective space $\bb P(1^2,2^5)$ which is the coarse space of the quotient stack $[(\mathbb{A} \setminus \{0\})/\mathbb{G}_m]$. (See also \cite[Definition 5.7]{ADL19} for a similar construction.)

To define the GIT of nodal quartics, let $\sigma$ be the action of $t \in \mathbb{G}_m$ on $\mathbb{A}$ acting by \[\sigma_t \cdot (xyz^2 + z(ax^3 + by^3) + f_4(x,y)) = xyz^2 + z(at^3 x^3 + bt^{-3} y^3) + f_4(tx, t^{-1}y)\]
for \(t\in\bb G_m\). This commutes with the quotient, and therefore we may consider $\sigma$ as an action on the weighted projective space $(\bb P(1^2,2^5), \mathcal{O}_{\bP (1^2,2^5)}(1))$.  Define the GIT (poly/semi)stability of the nodal quartic $\Delta$ with equation $xyz^2 + z(ax^3 + by^3) + f_4(x,y)$ to be the GIT (poly/semi)stability with respect to the action $\sigma$ on $(\bb P(1^2,2^5), \mathcal{O}_{\bP (1^2,2^5)}(1))$.

\begin{definition}\label{defn:GIT-nodal-quartics}
    Let $\GITnstack \coloneqq [\bb P(1^2,2^5)/\mathbb{G}_m]$ denote the \defi{GIT moduli stack of (pointed) nodal plane quartics}, and let $\GITn$ denote the associated good moduli space.
\end{definition}

\begin{lemma}\label{lem:nodal-quartics-description-of-GIT}
    A nodal quartic curve $\Delta$ given by $xyz^2 + z(ax^3 + by^3) + f_4(x,y)$ is GIT semistable with respect to $\sigma$ if and only if $\Delta$ does not have a triple point and is not \(xyz^2\).  It is GIT stable with respect to $\sigma$ if and only if, up to isomorphism, it does not contain the line $x = 0$ and have an $A_3$ singularity along this line.  It is GIT strictly polystable with respect to $\sigma$ if and only if, up to isomorphism, it is the ox given by $xyz^2 + x^2y^2 = 0$.
\end{lemma}

\begin{proof}
    For brevity, write $xyz^2 + g(x,y,z)$ for the equation of $\Delta$.  Working on the affine space $\mathbb{A}$, recall that $\Delta$ is GIT stable if $\lim_{t \to 0} \sigma_t \cdot g(x,y,z) = g(tx, t^{-1}y, z)$ and \(\lim_{t \to 0} \sigma_{t^{-1}} \cdot g(x,y,z)\) do not exist. It is GIT strictly semistable if $\lim_{t \to 0} \sigma_t \cdot g(x,y,z)$ and $\lim_{t \to 0} \sigma_{t^{-1}} \cdot g(x,y,z)$ exist and are not zero. Otherwise, it is GIT unstable.
    
    First suppose $\Delta$ is GIT semistable with respect to the classical GIT of plane quartics (see Section~\ref{sec:GIT-quartics}).  We are considering the action $\sigma$ which is $\diag(1,-1,0)$ (or $\diag(-1,1,0)$) in the linearization from $\mathcal{O}_{\bP^2}(1)$, which is a 1-parameter subgroup in the classical GIT. By assumption, the limit in classical GIT either does not exist (implying that $\Delta$ is GIT stable with respect to $\sigma$), or it exists and is not zero.  In the latter case, there are two possibilities for this limit.  If it is not equal to the quartic $xyz^2$, in which case the limit of $\sigma_t \cdot g(x,y,z)$ is not zero in the GIT defined with respect to $\sigma$, then $\Delta$ is GIT semistable with respect to $\sigma$.  If instead the limit if equal to $xyz^2$, which implies the limit of $\sigma_t \cdot g(x,y,z)$ is zero, we have a contradiction: up to interchanging $x$ and $y$, the only such curve $\Delta$ is $xyz^2 + z(\alpha x^3) + \beta x^4 + \gamma x^3 y$, $\alpha, \beta, \gamma \in \bb C$, but this has a triple point at $[0:1:0]$, which is unstable in the classical GIT.  Therefore, if $\Delta$ is a quartic with a node that is GIT semistable in the classical sense, then $\Delta$ is GIT semistable with respect to $\sigma$.  Using Wall's classification (Section~\ref{sec:A_1+A_n-Wall}), this proves that any quartic with a node without triple points that is not the union of a cubic curve and an inflectional line is GIT semistable with respect to $\sigma$.  

    Now, suppose instead that $\Delta$ has a triple point.  As $\Delta$ has a node at $[0:0:1]$, by the classification in \cite[Page 420, Case \(\delta\)]{Wall-quartics} (see also \cite[Section 8.7.1]{Dolgachev-CAG}), $\Delta$ must be one of the following:  
    \begin{enumerate}
        \item\label{item:triple-pt-cubic-line} A line meeting an irreducible singular cubic curve to order 2 at the singular point of the cubic (\(A_1 + D_4\) or \(A_1 + D_5\)); 
        \item\label{item:triple-pt-conic-2-lines} Two lines and a smooth conic through their intersection point (\(2A_1 + D_4\));
        \item\label{item:triple-pt-conic-tangent-lines} A smooth conic with a tangent line, and another line meeting the $A_3$ singularity transversally (\(A_1 + D_6\)); or
        \item\label{item:triple-pt-4-lines} Four lines such that three pass through a common point (\(3A_1 + D_4\)).
    \end{enumerate}
    We will show that, in each case, the equation of \(\Delta\) can be written as $xyz^2 + z(ax^3 + by^3) + f_4(x,y)$, uniquely up to isomorphism preserving this form of the equation.
    Cases~\eqref{item:triple-pt-cubic-line} and~\eqref{item:triple-pt-conic-tangent-lines} have a unique node and, up to isomorphism fixing the branches $x = y = 0$ of the node, we can write the equation as below.  In case~\eqref{item:triple-pt-conic-2-lines}, choosing the node to lie on the line $x = 0$, up to isomorphism we may write the curve as $x(yz-x^2)(z+ax)$ for $a \ne 0$, which after applying the isomorphism $z \mapsto z - \frac{a}{2}x$, yields the curve $xyz^2 - zx^3 - \frac{a}{2}x^4 - \frac{a^2}{4} x^3 y$. In case~\eqref{item:triple-pt-4-lines}, we may choose the line $y = 0$ to contain the 3 $A_1$ singularities which are equivalent up to isomorphism so the curve has equation $xyz(z+x)$.  Applying the isomorphism $z \mapsto z- \frac{x}{2}$, this is the curve $xyz^2 - \frac{1}{4}x^3 y$.  From the equation of $\Delta$, we see explicitly that these curves are GIT unstable with respect to $\sigma$. We find:

    \begin{enumerate}
        \item $x(yz^2 - x^2(x-ay))$, $a \in \bb C$, is unstable as $\lim_{t \to 0} \sigma_t \cdot(- x^3(x-ay)) = 0$;
        \item $xyz^2 - zx^3 - \frac{a}{2}x^4 - \frac{a^2}{4} x^3 y$ is unstable as $\lim_{t \to 0} \sigma_t \cdot(- zx^3 - \frac{a}{2} x^4 - \frac{a^2}{4} x^3 y) = 0$;
        \item $xz (yz - x^2)$ is unstable as $\lim_{t \to 0} \sigma_t \cdot(- x^3z) = 0$; and
        \item $xyz^2 - \frac{1}{4}x^3 y$ is unstable as  $\lim_{t \to 0} \sigma_t \cdot(- \frac{1}{4}x^3y) = 0$.
    \end{enumerate}

    Finally, suppose $\Delta$ is the union of a cubic curve and an inflectional line.  By assumption that $\Delta$ is nodal, the cubic curve must be nodal.  This is case \(\epsilon\) (\underline{1}21) from Wall's classification \cite{Wall-quartics}, and we may assume $\Delta$ has equation $z(y^2z-x^2(x+z))$.  After a suitable change of coordinates preserving the node at $[0:0:1]$, this has equation $xyz^2 + \frac{1}{4} z(x^3 + y^3)+ f_4(x,y) $.  Because $x^3$ and $y^3$ both have nonzero coefficients in the middle term, $\lim_{t\to 0} \sigma_t \cdot (\frac{1}{4} z(x^3 - y^3)+ f_4(x,y)) $ does not exist and hence the curve is GIT stable with respect to $\sigma$. 

    Therefore, we have shown that a nodal quartic curve $\Delta$ given by $xyz^2 + z(ax^3 + by^3) + f_4(x,y)$ is GIT semistable with respect to $\sigma$ if and only if $\Delta$ does not have a triple point.

    Finally, the results on (poly)stability are straightforward by comparison with the classical GIT of quartics, using the equation in Lemma \ref{lem:nodal-quartics-211-22} and that the strictly polystable limit of a strictly semistable curves has a node (and is thus in our GIT problem) if and only if it is the ox. 
\end{proof}

By the classification in Theorem~\ref{thm:local-delta-X_3} and Lemma~\ref{lem:non-reduced-Gamma-cases}, we have the following:

\begin{corollary}\label{cor:GIT-of-nodal-curves-and-K-on-Xn}
    For $c \in (\cnot, \frac{1}{2}]$, a pair $(\Xthree, cR)$ is K-(poly/semi)stable if and only if the discriminant curve $\Delta \subset \bP^2$ is GIT (poly/semi)stable with respect to the action $\sigma$. 
\end{corollary}

\section{Proofs of main theorems}\label{sec:proofs-of-main-thms}
In this section, we prove the main theorems of the paper. In order, we prove Theorems~\ref{thm:stable-members}, \ref{thm:GIT-polystable}, and~\ref{thm:stability-on-X3}, Corollary~\ref{cor:double-covers-of-X3}, and Theorems~\ref{thm:moduli-spaces-2.18} and~\ref{thm:wall-crossing-resolves-delta}. We also prove Theorem~\ref{thm:X_3-higher-A_n-description}, which gives an explicit description of a certain locus in \(\Kmodulispace{1/2}\)

First, we prove Theorem~\ref{thm:stable-members}, characterizing the stable \((2,2)\)-surfaces in \(\bb P^1\times\bb P^2\).
\begin{proof}[Proof of Theorem~\ref{thm:stable-members}]
    First, if~\eqref{item:stable-members-thm-A_n} holds, then Theorem~\ref{thm:local-delta-A_n-singularities}, the comment at the beginning of Section~\ref{sec:local-delta-A_n:A_1-infinite-onesingularity}, and \cite[Lemmas 2.2, 2.3, and 2.4]{CFKP23} imply that \(\delta_{\pt}(\bb P^1\times\bb P^2, cR)>1\) for all \(\pt\in\bb P^1\times\bb P^2\) and all \(c\in(0,\tfrac{1}{2}]\). So, by Theorem~\ref{thm:valuative-criterion-K-stability}, \eqref{item:stable-members-thm-A_n}\(\Rightarrow\)\eqref{item:stable-members-thm-K}. Next, \eqref{item:stable-members-thm-K}\(\Rightarrow\)\eqref{item:stable-members-thm-GIT} by Theorem~\ref{thm:K-implies-GIT}. Finally, \eqref{item:stable-members-thm-GIT}\(\Rightarrow\)\eqref{item:stable-members-thm-A_n} by Lemmas~\ref{lem:non-A_n-not-GIT-stable}, \ref{lem:worse-than-A_n}, and~\ref{lem:reducible-not-GIT-stable}, and Proposition~\ref{prop:A_n-sings-not-stable}.
\end{proof}

Next, we prove Theorem~\ref{thm:GIT-polystable}, which describes the strictly polystable \((2,2)\)-surfaces.
\begin{proof}[Proof of Theorem~\ref{thm:GIT-polystable}]
    Proposition~\ref{prop:representatives-degen-Gamma_i}, Lemma~\ref{lem:G_i-equivalent-in-GIT}, and Theorem~\ref{thm:Kpolystable2,2div} show that \(R\) is GIT polystable if and only if \(R\in\Gfour\cup\Gtwo\cup\Gthree\), and that this is a union of three rational curves meeting at \([\Rone]\). Theorem~\ref{thm:Kpolystable2,2div} and Proposition~\ref{prop:K-stability-Rthree} additionally show parts~\eqref{item:thm-GIT-polystable-not-R3} and~\eqref{item:thm-GIT-polystable-R3}. The descriptions of the singularities in parts \eqref{item:thm-GIT-polystable-R1}--\eqref{item:thm-GIT-polystable-Gthree} are contained in Remark~\ref{rem:frakG_i-surfaces-descriptions}. Finally, the last claim---that any normal \((2,2)\)-surface with the singularities described in parts~\eqref{item:thm-GIT-polystable-Gfour}--\eqref{item:thm-GIT-polystable-Gthree} is GIT semistable---is shown by Proposition~\ref{prop:A_n-sings-not-stable}.
\end{proof}

Now we prove Theorem~\ref{thm:stability-on-X3}, which describes the K-polystable surfaces that appear on \(\Xthree\). After this, we prove Corollary~\ref{cor:double-covers-of-X3}, which describes the double covers of \(\Xthree\) branched along these surfaces.
\begin{proof}[Proof of Theorem~\ref{thm:stability-on-X3}]
    Lemma~\ref{lem:strict-transform-of-X_3-surfaces-in-P} shows that if \(R\subset\Xthree\) is a degeneration of a \((2,2)\)-surface on \(\bb P^1\times\bb P^2\), then either (i) its strict transform \(R_{\bb P}\) on \(\bb P(1,1,1,2)\) is degree 4 and has multiplicity \(2\) at \(\ptE\), or (ii) \(R\) contains \(E\) or \(S\). By Lemma~\ref{lem:X3-only-after-wall}, if \((\Xthree, cR)\) is K-semistable, then \(c \geq \cnot\). Moreover, in this situation, \(R\) is a strict transform of a degree 4 surfaces on \(\bb P(1,1,1,2)\) and does not meet \(S\) by Lemma~\ref{lem:strict-transform-of-X_3-surfaces-in-P}\eqref{item:R-in-X_3-contains-S-if-meets}, Corollary~\ref{cor:R-does-not-contain-E}, and Lemma~\ref{lem:R-cannot-contain-S}. The bound on the possible \(A_n\) singularities is given by Remark~\ref{rem:nodal-plane-quartics-classification-A_n}.

    Part~\eqref{item:thm-stability-on-X3-stable} of Theorem~\ref{thm:stability-on-X3} follows from Theorems~\ref{thm:local-delta-X_3} and~\ref{thm:X_3-unstable}. Indeed, if \((\Xthree, cR)\) is K-stable for some \(c\in(\cnot,\tfrac{1}{2}]\cap\bb Q\), then Theorem~\ref{thm:X_3-unstable}\eqref{item:unstable-worse-than-A_2-at-E} and \eqref{item:unstable-A_2-at-E} show that \(R\cap E\) must be smooth. Furthermore, Theorem~\ref{thm:X_3-unstable}\eqref{item:unstable-non-A_n-isolated} and \eqref{item:unstable-non-isolated} imply that \(R\) only has (isolated) \(A_n\) singularities, and Theorem~\ref{thm:local-delta-X_3} then finishes the proof. Finally, for strictly polystable \(R\), Corollary~\ref{cor:Rox-polystable} shows part~\eqref{item:thm-stability-on-X3-unique-polystable-ox} of Theorem~\ref{thm:stability-on-X3}.
\end{proof}

\begin{proof}[Proof of Corollary~\ref{cor:double-covers-of-X3}]
    First, \(K_{\Xthree}+\frac{1}{2} R\) is anti-ample by Lemma~\ref{lem:D_1-on-X_3-big-nef}. Thus, the double cover \(Y\) of \(\Xthree\) branched along \(R\) is Fano. By Theorem~\ref{thm:stability-on-X3}, the surface \(R\) has only \(A_n\) singularities, so the double cover \(Y\) has isolated compound Du Val (cDV) singularities, i.e. Gorenstein terminal singularities. Thus, by \cite{JR11}, \(Y\) has the same Picard rank as smooth threefolds in family \textnumero2.18, which is 2. Hence, the Mori cone \(\overline{NE(Y)}\) has two extremal rays. As $Y \to \Xthree$ is a double cover, the description of the extremal contractions of \(\Xthree\) in Theorem~\ref{thm:extremalcontractionsofX3} now completes the proof.
\end{proof}

To prove Theorem~\ref{thm:moduli-spaces-2.18}, we first show that the only pairs parametrized by $\Kmodulispace{c}$ are surfaces on $\bP^1 \times \bP^2$ and $\Xthree$.

\begin{theorem}\label{thm:only-have-X_n-or-P1P2}
    If $[(X,cD)] \in \Kmodulispace{c}$ is a K-semistable pair and $c \in (0, \frac{1}{2}] \cap \bb Q$, then $X \cong \bP^1 \times \bP^2$ or $\Xthree$.
\end{theorem}

\begin{proof}
    For \(0<\epsilon\ll 1\), we have \(\Kmodulispace{\epsilon}\cong\GITR\) by Theorem~\ref{thm:K-GIT-isomorphic-small-c}.
    Theorems~\ref{thm:stable-members} and~\ref{thm:GIT-polystable} show that the K-stability conditions do not change for \(c\in(0,\cnot)\cap\bb Q\), and hence \(\Kmodulispace{\epsilon}\cong\Kmodulispace{c}\) for \(c\in(0,\cnot)\cap\bb Q\). 
\begin{detail}
    From the definition of the K-moduli functor \cite[Definition 7.23]{XuKmoduliBook}, as the K-stability conditions do not change for $c \in (0,\cnot) \cap \bb Q$, the moduli stacks $\Kmodulistack{\epsilon}$ and $\Kmodulistack{c}$ are isomorphic.  Therefore, their good moduli spaces are isomorphic and hence \(\Kmodulispace{\epsilon}\cong\Kmodulispace{c}\).  
\end{detail}
    Therefore, $X \cong \bP^1 \times \bP^2$ for $c < \cnot$.
    
    Now assume $\cnot < c \le \frac{1}{2}$.  Let $\mathcal{U} \subset \Kmodulistack{c}$ and $U \subset \Kmodulispace{c}$ denote the loci of pairs $[(X,cR)]$ parametrizing surfaces on $\bP^1 \times \bP^2$ or $\Xthree$.  First, observe that $\mathcal{U} \subset \Kmodulistack{c}$ is saturated and hence $U \subset \Kmodulispace{c}$ is open (c.f. \cite[Remark 6.2]{Alper13}) by Theorem \ref{thm:unobstructeddeformations} and Theorem \ref{thm:local-delta-X_3}: if $[(X, cR)] \in \Kmodulistack{c}$ is any point that has image in the good moduli space morphism $[(X, cR)] \in U$, then $[(X, cR)] \in \mathcal{U}$.  
    
    Let $\mathcal{E}_{\mathfrak{n}}$ (respectively, $E_{\mathfrak{n}}$) denote the locus in the moduli stack \(\Kmodulistack{c}\) (respectively, moduli space $\Kmodulispace{c}$) parameterizing pairs $(\Xthree, \dee)$.  Because $\mathcal{E}_{\mathfrak{n}}$ is saturated by Theorem \ref{thm:unobstructeddeformations} and Theorem \ref{thm:local-delta-X_3}, $\mathcal{E}_{\mathfrak{n}} \to E_{\mathfrak{n}}$ is a good moduli space \cite[Remark 6.2]{Alper13}.  We first show that $E_\mathfrak{n}$ is closed in $\Kmodulispace{c}$.

    By Corollary \ref{cor:GIT-of-nodal-curves-and-K-on-Xn}, there is an open immersion $\GITnstack \to \mathcal{E}_{\mathfrak{n}}$, where \(\GITnstack\) is the GIT moduli stack defined in Definition~\ref{defn:GIT-nodal-quartics}.  Indeed, we may consider the universal family over $\GITnstack$ and consider the corresponding universal family of surfaces in $\bP(1,1,1,2)$ by associating a quartic $\Delta$ with equation $xyz^2 + z(ax^3 + by^3) + f_4(x,y) = 0$ to the surface $w^2 = xyz^2 + z(ax^3 + by^3) + f_4(x,y) $.  Taking strict transforms in $\Xthree$ yields a family of K-semistable pairs $(\Xthree, cR)$ by Corollary \ref{cor:GIT-of-nodal-curves-and-K-on-Xn}.  By universality of the K-moduli space, we have an open immersion $\GITnstack \to \mathcal{E}_{\mathfrak{n}}$.  Taking good moduli spaces, we have a morphism $\GITn \to E_{\mathfrak{n}}$.  By construction and Corollary \ref{cor:GIT-of-nodal-curves-and-K-on-Xn}, this is a bijection on closed points.  Because $E_\mathfrak{n}$ is normal (c.f. Theorem \ref{thm:unobstructeddeformations} and the proof of Corollary \ref{cor:Kmoduli-is-normal}), this is in fact an isomorphism by Zariski's Main Theorem. As $\GITn$ is a proper variety, this implies that $E_{\mathfrak{n}}$ is proper and hence closed in $\Kmodulispace{c}$.  

    Finally, let $Z \subset \Kmodulispace{c}$ be the locus of pairs $[(X', cD')] \in \Kmodulispace{c}$ where $X'$ is not isomorphic to $\bP^1 \times \bP^2$ or $\Xthree$.   As demonstrated above, $Z$ is closed.  Because $E_{\mathfrak{n}}$ is closed, we have $\overline{E_{\mathfrak{n}}} \cap \overline{Z} = E_{\mathfrak{n}} \cap Z \emptyset$.  Since the K-stability conditions for \((2,2)\)-surfaces \(R\not\cong\Rthree\) do not change for \(c\in(0,\frac{1}{2}]\cap\bb Q\) by Theorems~\ref{thm:stable-members} and~\ref{thm:GIT-polystable}, there is a rational map $\Kmodulispace{c} \dashrightarrow \GITR$ of normal projective varieties (\cite{XZ20positivity}, see e.g. \cite{Mumford-Fogarty} or Corollary \ref{cor:Kmoduli-is-normal} for normality) that is an isomorphism over \(\GITR\setminus[\Rthree]\). By Zariski's Main Theorem, a resolution \(f\colon \widehat{M}\to\GITR\) of this rational map must have connected fibers.  The preimages of \(Z\) and \(E_{\mathfrak{n}}\) in \(\widehat{M}\) are disjoint and cover the fiber of \(f\) over \([\Rthree]\), so the connectedness of this fiber implies that $Z = \emptyset$. Therefore, again using Theorem \ref{thm:unobstructeddeformations}, we conclude that if $[(X,cD)] \in \Kmodulispace{c}$ is a K-semistable pair and $c \in (0, \frac{1}{2}] \cap \bb Q$, then $X \cong \bP^1 \times \bP^2$ or $\Xthree$.
\end{proof}

Using the previous theorem and the deformation theory of $\Xthree$, we have the following corollary. 

\begin{corollary}\label{cor:Kmoduli-is-normal}
    The K-moduli space \(\Kmodulispace{c}\) is normal for \(c\in(0,\tfrac{1}{2}]\cap\bb Q\).
\end{corollary}

\begin{proof}
    By Theorem~\ref{thm:only-have-X_n-or-P1P2}, the only threefolds that may appear in the moduli stack $\Kmodulistack{c}$ are $X = \bP^1 \times \bP^2$ or $X = \Xthree$.  By calculation for $X = \bP^1 \times \bP^2$ or by Theorem~\ref{thm:unobstructeddeformations}, $X$ has unobstructed deformations.  Let $R$ be a divisor on $X$ such that $R \subset X$ and $(X, cR) \in \Kmodulispace{c}$.  By construction, all divisors $R$ considered are Cartier, big, and nef (using Lemmas~\ref{lem:strict-transform-of-X_3-surfaces-in-P}, \ref{lem:D_1-on-X_3-big-nef}, and~\ref{lem:R-cannot-contain-S} and Corollary~\ref{cor:R-does-not-contain-E} for \(\Xthree\)), so $H^1(X,\mathcal{O}_X(R)) = 0$ by the Kawamata--Viehweg vanishing theorem.  Therefore, there is no obstruction to extending a deformation of $X$ to a deformation of $R$, and thus every K-semistable pair $(X,cR)$ has unobstructed deformations.  This implies that the stack $\Kmodulistack{c}$ is smooth and hence, by \cite[Theorem 4.6(viii)]{Alper13}, $\Kmodulispace{c}$ is normal. 
\end{proof}

Before the proof of Theorem~\ref{thm:moduli-spaces-2.18}, we first construct the bijective morphism \(\iota\colon \Kmodulispace{1/2} \to \Kmodulitwoeighteen\).
\begin{proposition}\label{prop:K-moduli-2.18-vs-double-covers}
    There is a morphism \(\iota\colon \Kmodulispace{1/2} \to \Kmodulitwoeighteen\), which is a bijection.
\end{proposition}
\begin{proof}
    The proof follows that of \cite[Proposition 6.12]{ADL-quartic-K3}. First, Theorem~\ref{thm:unobstructeddeformations} and Theorem~\ref{thm:only-have-X_n-or-P1P2} implies that if \((X, \tfrac{1}{2}R)\) is a K-polystable pair parametrized by \(\Kmodulispace{1/2}\), then \(X\) is \(\bb P^1\times\bb P^2\) or \(\Xthree\). Next, we note that the \(\bb Q\)-Gorenstein deformations of \(\bb P^1\times\bb P^2\) and of \(\Xthree\) are unobstructed \cite{Namikawa}, and there is no torsion in the class group of these threefolds by Proposition~\ref{prop:properties-of-Xthree}.
    The argument of \cite[Proposition 6.12]{ADL-quartic-K3} then shows the existence of the morphism \(\iota\colon \Kmodulispace{1/2} \to \Kmodulitwoeighteen\) of good moduli spaces, which, on closed points, assigns a K-polystable pair \((X, \tfrac{1}{2} R)\) to the double cover of \(X\) branched along \(R\).
    Next, we show that \(\iota\) is a bijection. First, \(\iota\) is dominant because, by \cite{CFKP23}, its image contains the double covers of \(\bb P^1\times\bb P^2\) branched along smooth \((2,2)\)-divisors. So properness of \(\Kmodulispace{1/2}\) implies \(\iota\) is surjective.
    
    It remains to show that if \([(X,\tfrac{1}{2} R)], [(X', \tfrac{1}{2} R')] \in \Kmodulispace{1/2}\) are two points with isomorphic associated double covers \(Y\cong Y'\), then \((X, \tfrac{1}{2} R)\cong(X', \tfrac{1}{2} R')\). Since the singular loci of \(Y, Y'\) have the same dimension, the surfaces \(R,R'\) must either both be normal or both be non-normal.

    \noindent\underline{Case 1: \(R, R'\) are normal.}
    Theorems~\ref{thm:GIT-polystable} and~\ref{thm:stability-on-X3} imply that \(R,R'\) only have \(A_n\) singularities. Therefore, the double covers \(Y,Y'\) have Gorenstein terminal singularities, so arguing using \cite{JR11} as in the proof of Corollary~\ref{cor:double-covers-of-X3}, we have \(\rho(Y)=\rho(Y')=2\).
    
    Thus, the Mori cones \(\overline{NE(Y)}\) (resp. \(\overline{NE(Y')}\)) each have 2 extremal rays, which we denote \(R_1,R_2\) (resp. \(R_1',R_2'\)). If \(X=\bb P^1\times\bb P^2\), then the two extremal contractions are (1) the conic bundle \(Y\to\bb P^1\times\bb P^2\xrightarrow{\piX_2}\bb P^2\) and (2) the quadric surface fibration \(Y\to\bb P^1\times\bb P^2\xrightarrow{\piX_1}\bb P^1\). If instead \(X=\Xthree\), then by Theorem~\ref{thm:extremalcontractionsofX3} the two extremal contractions are (1) the non-flat conic bundle \(Y\to\Xthree\xrightarrow{\piXthree}\bb P^2\) and (2) the birational morphism that contracts the preimage of \(S\). This shows that \(X\cong X'\).
        
    Thus, we have a diagram of (possibly non-flat) conic bundles
    
    \begin{equation}\label{diagram:conic-bundles-Y-Y'}\begin{tikzcd}
    Y \arrow[r, "\cong"] \arrow[d, "2:1" {swap}] & Y' \arrow[d, "2:1"] \\
    X \arrow[d, "\psi" {swap}] & X' \arrow[d, "\psi"] \\ \bb P^2 & \bb P^2 \end{tikzcd}\end{equation}
    where the bottom maps \(\psi\) are both either \(\piX_2\colon\bb P^1\times\bb P^2\to\bb P^2\) or \(\piXthree\colon\Xthree\to\bb P^2\). The composition \(\phi\colon Y\to Y'\to X'\to\bb P^2\) is a conic bundle structure on \(Y\), so \(\phi\) contracts an extremal face in \(\overline{NE(Y)}\).
\begin{detail}
    (Indeed, let \(f\colon W \to Z\) be a morphism of normal projective varieties with connected fibers, and let \(F\subset\overline{NE(W)}\) be the (closure of the) set of curves \(C\) whose image \(f(C)\) is a point in \(Z\). Let \(D\) be an ample Cartier divisor on \(Z\). Then \(D_W\coloneqq f^*D\) has positive intersection with the curves not contracted by \(f\) and has trivial intersection with the curves contracted by \(f\), so \(D_W^\perp\subset\overline{NE(W)}\) must be an extremal face.)
\end{detail}
    Since \(\rho(Y)=2\), the only extremal faces are \(R_1,R_2\), and all of \(\overline{NE(Y)}\); therefore, \(\phi\) must be the contraction of the extremal ray \(R_1\).
    Therefore, \(\phi\) factors through \(X\) and we may fill in horizontal arrows in diagram~\eqref{diagram:conic-bundles-Y-Y'} making the bottom square commute. But this implies that the top square also commutes, i.e. the double covers \(Y\to X\) and \(Y'\to X'\) are isomorphic.

    \noindent\underline{Case 2: \(R, R'\) are not normal.}
    In this case, by Theorems~\ref{thm:GIT-polystable} and~\ref{thm:stability-on-X3} we have \(X \cong X' \cong \bb P^1\times\bb P^2\) and \(R,R'\in\{\Rone,\Rtwo,\Rthree\}\). The singular locus of \(\Rone\) has two connected components, whereas the singular loci of \(\Rtwo\) and \(\Rthree\) each have a single irreducible component \(\cong\bb P^1\), so the double cover branched along \(\Rone\) cannot be isomorphic to the double cover branched along \(\Rtwo\) or \(\Rthree\). It remains to consider the case \(R\cong\Rtwo\) and \(R'\cong\Rthree\). Since the surface \(\Rthree\) has two pinch points whereas \(\Rtwo\) has none, in this case we obtain \(Y\not\cong Y'\).

    This completes the proof that \(\iota\) is injective; hence, it is an isomorphism.
\end{proof}

We are now ready to prove Theorem~\ref{thm:moduli-spaces-2.18}.

\begin{proof}[Proof of Theorem~\ref{thm:moduli-spaces-2.18}]
    As noted in the first paragraph of the proof of Theorem \ref{thm:only-have-X_n-or-P1P2}, we have \(\Kmodulispace{\epsilon}\cong\GITR\cong\Kmodulispace{c}\) for \(c\in(0,\cnot)\cap\bb Q\) by Theorems~\ref{thm:K-GIT-isomorphic-small-c}, \ref{thm:stable-members} and~\ref{thm:GIT-polystable}. This proves part~\eqref{item:moduli-spaces-2.18-before-wall}.

    For part~\eqref{item:moduli-spaces-2.18-after-wall}, Theorems~\ref{thm:stable-members}, \ref{thm:GIT-polystable}, and~\ref{thm:stability-on-X3} show that for \(c\in(\cnot,\tfrac{1}{2}]\cap\bb Q\), the K-stability conditions for \((X,cR)\) do not change for \(X=\bb P^1\times\bb P^2\) (resp. \(\Xthree\)) and \((2,2)\)-surfaces \(R\) (resp. surfaces on \(\Xthree\) that are degenerations of \((2,2)\)-surfaces). Since \(\Xthree\) only admits a nontrivial deformation to $\bP^1 \times \bP^2$  (Theorem~\ref{thm:unobstructeddeformations}) and by Theorem \ref{thm:only-have-X_n-or-P1P2}, \(\bb P^1\times\bb P^2,\Xthree\) are the only threefolds \(X\) that appear on the K-moduli stacks \(\Kmodulistack{c}\) for \(c\in(0,\tfrac{1}{2}]\cap\bb Q\). Therefore, \(\Kmodulispace{c}\cong\Kmodulispace{1/2}\) for \(c\in(\cnot,\tfrac{1}{2}]\cap\bb Q\). The bijective morphism \(\Kmodulispace{1/2}\to\Kmodulitwoeighteen\) is described in Proposition~\ref{prop:K-moduli-2.18-vs-double-covers}.

    It remains to show part~\eqref{item:moduli-spaces-2.18-wall-crossing} and describe the wall-crossing morphism. By Theorem~\ref{thm:GIT-polystable} we know that \(\Rthree\) is the only GIT polystable \((2,2)\)-surface \(R\) for which \((\bb P^1\times\bb P^2, c R)\) is K-unstable for $c \in (\cnot,\tfrac{1}{2}]\cap\bb Q$.  Recall from the proof of Theorem~\ref{thm:only-have-X_n-or-P1P2} that $E_{\mathfrak{n}}$ denotes the locus in \(\Kmodulispace{1/2}\) parameterizing pairs $(\Xthree, \dee)$. We first show that \(E_{\mathfrak{n}}\) is a divisor. The proof of Theorem~\ref{thm:only-have-X_n-or-P1P2} shows that \(E_{\mathfrak{n}} \cong \GITn\) is closed in \(\Kmodulispace{1/2}\). Since \(\GITn\) is the GIT quotient of \(\bb P(1^2,5^2)\) by \(\mathbb G_m\) (see Section~\ref{sec:nodalGIT}), it is irreducible and 5-dimensional. This shows \(E_{\mathfrak{n}}\) is an irreducible (Weil) divisor in the 6-dimensional variety \(\Kmodulispace{1/2}\).
    
    By \cite{XZ20positivity} and Corollary~\ref{cor:Kmoduli-is-normal}, the K-moduli space $\Kmodulispace{c}$ is a normal projective variety for any $c \in (0,\tfrac{1}{2}] \cap \bb Q$. Therefore, there exists a birational map of normal projective varieties $\Kmodulispace{1/2} \dashrightarrow \GITR$ defined away from $E_{\mathfrak{n}}$ such that $\Kmodulispace{1/2} \setminus E_{\mathfrak{n}}$ and $\GITR \setminus [\Rthree] $ are isomorphic.  By the following general construction $(\star)$ applied to \(V = \Kmodulispace{1/2}\), we obtain a morphism $\Kmodulispace{1/2}  \to \GITR$.  Note that the morphism $f\colon \hX \to \GITR$ in $(\star)$ must contract the preimage of $E_{\mathfrak{n}}$ to $[\Rthree]$: because $\GITR$ is normal, by Zariski's Main Theorem, $f$ has connected fibers.  For any $p \in \GITR \setminus [\Rthree]$, there is a unique preimage $q$ of $p$ in $\Kmodulispace{1/2} \setminus E_{\mathfrak{n}}$ (and hence a unique preimage $\hat{q}$ in $\hX \setminus g^{-1}(E_{\mathfrak{n}})$), so by connectedness of the fibers of $f$, $\hat{q}$ must be the the only preimage of $p$ in $\hX$.  Therefore, the equivalence relation condition is satisfied and $\Kmodulispace{1/2}  \to \GITR$ is a morphism, which by construction is is an isomorphism on $\Kmodulispace{1/2} \setminus E_{\mathfrak{n}}$ and contracts $E_{\mathfrak{n}}$ to the point $[\Rthree]$. This proves Theorem \ref{thm:moduli-spaces-2.18}. 
    
    \noindent $(\star)$ \textit{Construction.} Suppose $X$ and $Y$ are normal proper varieties and $E \subset X$ a closed subscheme such that there exists a rational map $X \dashrightarrow Y$ defined away from $E$.  Let $\hX$ be a resolution of the rational map:

    
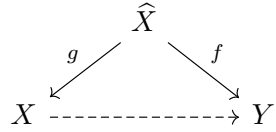
\begin{center}
    \begin{tikzcd}
        & \hX \arrow[rd,"f"] \arrow[ld,swap,"g"] & \\
        X \arrow[rr,dashed] & & Y 
    \end{tikzcd}
    \end{center}
    and assume that $g$ is an isomorphism over $U \coloneqq X \setminus E$.  Define equivalence relations $\sim_g$, $\sim_f$ on $\hX$ by $x_1 \sim_g x_2 \iff g(x_1) = g(x_2)$ and similarly $x_1 \sim_f x_2 \iff f(x_1) = f(x_2)$.  Let \(V\subset X\) be an open subset containing \(U\), and write \(\widehat{V} = g^{-1}(V)\). If $x_1 \sim_g x_2 \implies x_1 \sim_f x_2$ for all $x_1, x_2 \in \widehat{V}$, then the rational map $X \dashrightarrow Y$ extends to a morphism $h \colon V \to Y$ that agrees with $X \dashrightarrow Y$ on its locus of definition, as follows:
    
    First, because $V = \widehat{V} / \sim_g$ and $Y = \hX/ \sim_f$, the assumptions and the universal property of quotient spaces gives a map of topological spaces $h \colon V \to Y$.  To extend this to a morphism of varieties, we observe that this gives a compatible map on the structure sheaves.
    First, $\mathcal{O}_X \cong g_* \mathcal{O}_\hX$ because $g$ is birational and $X$ is normal, so we have \(\cal O_V\cong (g|_{\widehat{V}})_* \cal O_{\widehat{V}}\). Then, the morphism of sheaves $\mathcal{O}_Y \to (f|_{\widehat{V}})_* \cal O_{\widehat{V}} = h_* (g|_{\widehat{V}})_* \cal O_{\widehat{V}}$ gives the required map $\mathcal{O}_Y \to h_* \mathcal{O}_V$.  
\end{proof}

\begin{remark}
    We expect that the morphism $\Kmodulispace{1/2}  \to \GITR$ is a weighted blow-up of the point $[\Rthree]$ with exceptional divisor $E_{\mathfrak{n}}$ (see, e.g. \cite[Theorem 5.6]{ADL19} for a similar argument).
\end{remark}

Finally, we prove Theorem~\ref{thm:wall-crossing-resolves-delta}, which describes the discriminant map.

\begin{proof}[Proof of Theorem~\ref{thm:wall-crossing-resolves-delta}]
    First, there is an \(\SL_3\)-equivariant rational map \(|\cal O_{\bb P^1\times\bb P^2}(2,2)| \dashrightarrow |\cal O_{\bb P^2}(4)|\) defined by assigning \(V(t_0^2 Q_1 + 2t_0t_1 Q_2 + t_1^2 Q_3)\) to \(V(Q_1Q_3 - Q_2^2)\). Furthermore, it respects the \(\SL_2\times\SL_3\)-action \cite[Theorem 4.5(i)]{FJSVV}, so it induces a rational map \(\GITR \dashrightarrow \GITquartic\) on the GIT moduli spaces. Using the descriptions of the singularities on the GIT polystable \((2,2)\)-surfaces in Theorems~\ref{thm:stable-members} and~\ref{thm:GIT-polystable}, together with the description of \(\GITquartic\) (Section~\ref{sec:GIT-quartics}), we see that this latter rational map is defined away from \(\Rthree\) and the 2-dimensional locus of surfaces with discriminant a cubic and inflectional line (Section~\ref{sec:unstable-An-quartics}).  Composing this with the wall-crossing morphism \(\Kmodulispace{1/2} \to \Kmodulispace{\cnot-\epsilon}\cong \GITR\) from Theorem~\ref{thm:moduli-spaces-2.18}, we have a rational map \(\Disc\colon\Kmodulispace{1/2} \dashrightarrow \GITquartic\). It remains to show \(\Disc\) extends to pairs $(\Xthree, \frac{1}{2}R)$ for surfaces $R$ other than the unique surface with discriminant curve of type $\epsilon$ (\underline{1}21) (Corollary \ref{cor:one-unstable-delta-on-Xthree}), for which $(\Xthree, \frac{1}{2}R)$ is K-stable.  We again use the construction $(\star)$ above.  Let $V \subset \Kmodulispace{1/2}$ be the locus of pairs with GIT semistable discriminant curve.  We may resolve the rational map 
    \begin{center}
    \begin{tikzcd}
        & \widehat{M} \arrow[rd,"f"] \arrow[ld,swap,"g"] & \\
        \Kmodulispace{1/2} \arrow[rr,dashed] & & \GITquartic 
    \end{tikzcd}
    \end{center}
    so that \(g\) is an isomorphism over $V \setminus (E_{\mathfrak{n}} \cap V)$. To verify that we may apply $(\star)$, suppose $x_1, x_2 \in \widehat{V}$ are points such that $g(x_1) = g(x_2)$.  We then wish to show that $f(x_1) = f(x_2)$.  This is obvious if $g(x_1),g(x_2) \notin E_{\mathfrak{n}} \cap V$, so assume both are contained in $E_{\mathfrak{n}} \cap V$.  Let $C_1 \subset \widehat{V}$ (resp. $C_2 \subset \widehat{V}$) be a germ of a pointed curve with generic point $\eta_1$ not contained in $g^{-1}(E_{\mathfrak{n}}\cap V)$ and closed point equal to $x_1$ (resp. $x_2$).  Then, $g(C_1)$ and $g(C_2)$ intersect at a unique point $[(\Xthree,\frac{1}{2}R_x)]\coloneqq g(x_1) = g(x_2)$.  By construction, $g(\eta_1) \in \Kmodulispace{1/2} \setminus E_{\mathfrak{n}}$ where the rational map $\Kmodulispace{1/2} \dashrightarrow \GITquartic$ is defined, so $f(\eta_1)$ agrees with the image of $g(\eta_1)$ in this rational map.  By properness and separatedness of $\GITquartic$, there is a unique GIT polystable quartic curve $\Delta_1$ such that $[\Delta_1]$ is in the closure of $f(\eta_1)$.  Because $f$ is morphism of varieties, this implies that $f(x_1) = [\Delta_1]$.  Similarly, there is a unique GIT polystable quartic curve $\Delta_2$ such that $f(x_2) = [\Delta_2]$.  Furthermore, there exists a GIT semistable quartic in the closure of both $f(\eta_1)$ and $f(\eta_2)$: by Corollary~\ref{cor:discriminant-family}, because $g(x_1) = g(x_2)$, one such curve is the discriminant $\Delta_x$ of $[(\Xthree,\frac{1}{2}R_x)]$.  Therefore, by separatedness of $\GITquartic$, both $\Delta_1$ and $\Delta_2$ must be polystable curves S-equivalent to $\Delta_x$ and hence isomorphic.  Therefore, $f(x_1) = f(x_2)$.  So we may apply $(\star)$ and obtain a morphism $\Disc\colon V \to \GITquartic$, as desired.  By separatedness of $\GITquartic$, this map must coincide with the (S-equivalence class of the) discriminant in families over pointed curves $C$ with generic point not contained in $E_{\mathfrak{n}}$. By Corollary~\ref{cor:discriminant-family}, we see that \(\Disc\) assigns the (S-equivalence class of the)  discriminant curve to any closed point $[(X,\frac{1}{2}R)] \in V \subset \Kmodulispace{1/2}$.

    Finally, we describe induced map on $E_\mathfrak{n}$. By Theorem~\ref{thm:moduli-spaces-2.18}\eqref{item:moduli-spaces-2.18-wall-crossing}, $E_{\mathfrak{n}}$ is isomorphic to the GIT moduli space \(\GITn\) of (pointed) nodal quartic plane curves. A general nodal quartic has exactly one node and no other singularities. Such a quartic is GIT stable in the classical setting and hence the unique member of its S-equivalence class. Thus, by Corollary~\ref{cor:one-unstable-delta-on-Xthree} and Lemma~\ref{lem:nodal-quartics-description-of-GIT}, the map $E_\mathfrak{n} \cong \GITn \dashrightarrow \GITquartic$ is defined away from a single point and is birational to its image, which is the nodal locus in $\GITquartic$.
\end{proof}

\begin{remark}\label{rem:discriminant-morphism-not-finite}
    The rational map \(\Disc\) in Theorem~\ref{thm:wall-crossing-resolves-delta} is generically finite of degree 63, see \cite[Theorem 6.2.3]{Dolgachev-CAG}. However, this map is \emph{not} finite where it is defined. For example, surfaces \(R\) in \(\bb P^1\times\bb P^2\) or \(\Xthree\) with \(A_n\) (\(n\geq 3\)) singularities satisfying certain conditions on fibers of the projections are K-stable by Theorems~\ref{thm:stable-members} and~\ref{thm:stability-on-X3}. Their associated discriminant curves are plane quartics with \(A_n\) singularities, and if this plane quartic \(\Delta\) has type \(\gamma\) in \cite[Section 3]{Wall-quartics}, then \(\Delta\) is GIT semistable but not polystable and degenerates to the double conic (see Section~\ref{sec:GIT-higher-An-quartics}). In particular, by considering the locus parametrizing pairs \((\Xthree, \tfrac{1}{2} R)\) where \(R_{\bb P}\) is the double cover branched along a type \(\gamma\) (2\underline{1}1) quartic, we see that the fiber of the discriminant map over the double conic is infinite.
\end{remark}

In future work, we plan to study wall crossings in the region $c \in (\frac{1}{2}, 1)$.  To that benefit, we also describe the image of $E_\mathfrak{n}$ in the rational map $\Disc$.  Denote by $E_{\mathfrak{n}}^{ns}$ the locus of pairs $(\Xthree, cR)$ whose associated discriminant curve is not GIT stable.

\begin{theorem}\label{thm:X_3-higher-A_n-description}
    The locus $E_{\mathfrak{n}}^{ns}$ is the union $[(\Xthree, c\Rox)] \cup \bP(1,1,2)$, where $\bP(1,1,2)$ is the parameter space for curves in Lemma \ref{lem:nodal-quartics-211-22}.  Projecting away from the point $[0:1:0] \in \bP(1,1,2)$ gives a surjective rational map $\bP(1,1,2)$ to $\bP^1$ whose target is the strictly polystable locus in $\GITquartic$, and this map identifies the GIT strictly semistable discriminants with their polystable representatives.
\end{theorem}

\begin{proof}
We use the coordinates of \cite{Wall-quartics} used in Section~\ref{sec:GIT-higher-An-quartics}, \emph{not} the coordinates used in Lemma \ref{lem:equation-of-R_P}.  Because all discriminant curves of surfaces on $\Xthree$ have at least one node, these non-stable curves correspond to quartics with at least one $A_1$ singularity and at least one $A_n$ singularity \(\pt\) with $n \ge 3$.  The equations for the possible quartics are given by \cite{Wall-quartics}, as described in Section~\ref{sec:A_1+A_n-Wall}. Recall from that section that all such quartics are GIT semistable except    
case $\epsilon$ (\underline{1}21), which is the union of a nodal cubic and a line meeting the cubic at a smooth point to order 3.
We will parametrize $E_{\mathfrak{n}}^{ns}$ explicitly, and its two components will correspond to the two cases depending on the reducedness of $R_{\bb P}|_{\Gamma_p}$.

First, if $R_{\bb P}|_{\Gamma_p}$ is non-reduced, then by Lemma~\ref{lem:non-reduced-Gamma-cases} and Corollary~\ref{cor:Rox-polystable}, \((\Xthree, cR)\) is K-semistable and has polystable representative $(\Xthree, c\Rox)$. Because $(\Xthree, c\Rox)$ is K-polystable, this is a closed point in $\Kmodulispace{c}$. Using the coordinates in \cite{Wall-quartics}, the surfaces in this locus are given by:
\begin{enumerate}
    \item \(R_{\bb P}\) is branched over a quartic \(\Delta\) in 
    cases $\delta$ (1111), (31), (4); and
    \item \(R_{\bb P}\) is branched over a quartic \(\Delta\) in  cases $\delta$ (211) and (22) where \(\ptE\in\bb P(1,1,1,2)\) is chosen to map to a specific \(A_1\) singularity of \(\Delta\).
    In case \(\delta\) (211), \(\Delta\) is the union of a line \(\ell\) and a cubic meeting in two points; choose \(\ptE\) to map to the \(A_1\) singularity where these two components meet with multiplicity one.
    In case \(\delta\) (22), \(\Delta\) is the union of a conic and two lines. Exactly one line \(\ell\) is tangent to the conic, and one of the \(A_1\) singularities of quartic \(\Delta\) lies on \(\ell\); choose \(\ptE\) to  map to this \(A_1\) singularity.
    (See Section~\ref{sec:A_1+A_n-Wall} for descriptions of these cases; in the notation of that section \(\ell = (x=0)\).)
    \end{enumerate}
Furthermore, because $(\Xthree, c\Rox)$ is K-polystable, this is a closed point in $\Kmodulispace{c}$.

The remaining cases, in which $R_{\bb P}|_{\Gamma_p}$ is reduced, can be parametrized by a surface as follows. Let $\bP(1,1,2)_{[s:t:r]}$ denote the locus of (isomorphism classes of) surfaces written as strict transforms of $w^2 = x^2z^2 - 2s xz y^2 + r y^4 + t xy^3 - x^2 y^2$.

    By Lemma~\ref{lem:nodal-quartics-211-22} and Remark~\ref{rem:nodes-are-equivalent}, every nodal quartic of type (211) or (22) can be written in this form $x^2z^2 - 2s xz y^2 + r y^4 + t xy^3 - x^2 y^2$ with a node at \([1:0:0]\).
    The surfaces $R$ where $R_{\bb P}|_{\Gamma_p}$ is reduced but the discriminant is not stable therefore correspond to $R_{\bb P}$ where $\ptE$ is the point $[1:0:0:0]$. By Theorem~\ref{thm:local-delta-X_3}, the corresponding pairs are all K-stable.

    To study the map to $\GITquartic$, we find the polystable limits of GIT semistable quartics parametrized by $\bP(1,1,2)$, using the description in the proof of Lemma~\ref{lem:nodal-quartics-211-22}.  The curve $\epsilon$ (\underline{1}21) is parametrized by the point $[0:1:0]$, and away from this point, the curves parametrized by $(r = 0)$ correspond to $\delta$ (211) and (22) (with GIT polystable limit the ox); those by $(s^2 = 0)$ correspond to $\beta$ (211) and (22) (with GIT polystable limit a cat-eye); those by $(r = s^2)$ correspond to $\gamma$ (211) and (22) (with GIT polystable limit the double conic); and those by $r = as^2$ for some $a \in \bb C$, $a \ne 0, 1$ correspond to $\alpha$ (211) and (22) (with GIT polystable limit a cat-eye).  In particular, the point $[0:1:0]$ is the base point of this pencil of sections of $\mathcal{O}_{\bP(1,1,2)}(2)$ that determines the rational map to the strictly polystable locus in $\GITquartic$.
\end{proof}

\appendix
\section{Volumes under certain birational transformations of threefolds}\label{appendix:volume-lemmas}
In this appendix, we collect formulas relating the volumes of divisors under certain birational transformations of threefolds, which we use in Sections~\ref{sec:stability-of-A_n} and~\ref{sec:K-stability-Xthree}.
We assume all threefolds are normal and $\mathbb{Q}$-Gorenstein, and that any $P_2 \in \Pic(X_2)\otimes\bb R$ is nef.
We repeatedly exploit the fact that if \(P_2\) satisfies the above hypotheses, then \(\vol (P_2) = P_2^3\).

\subsection{Contracting a weighted projective space}

Let \(\phi\colon X_1 \to X_2\) be a birational morphism of threefolds contracting an exceptional divisor \(Z = \bb P(1,1,n)\) for some \(n\in\bb Z_{\geq 1}\). Let \(z\) denote the class of a ruling of \(Z\).

\begin{lemma}\label{lem:volume-lemma-2}
 Let \(P_1 \in\Pic(X_1)\otimes\bb R\) and \(P_2 \in \Pic(X_2) \otimes \bb R\). Assume that \(P_2\) is nef, \(\phi^* P_2 = P_1 + \frac{p}{q} Z\), and \(p,q\in \bb R\) satisfy \(P_1\cdot z = p\) and \(Z\cdot z = -q\). Then \[\vol (P_2) = P_1^3 + n\frac{p^3}{q}.\]
 In particular, if $n = 1$ and $\phi$ contracts $Z = \bP^2$ to a smooth point of $X_2$, then $q = 1$, so we have $\vol(P_2) = P_1^3 + p^3$.  If $\phi$ contracts $Z=\bb P^2$ to a singular point of type $\frac{1}{2}(1,1,1)$, then $q = 2$ and $\vol(P_2) = P_1^3 + \frac{p^3}{2}$.
\end{lemma}

\begin{proof}
The assumptions imply that \(P_1|_Z = np z\) and \(Z|_Z = -n q z\). Since \(z^2 = \tfrac{1}{n}\) on \(Z\), we compute
\begin{align*}
    \vol (P_2) &= \vol(\phi^*(P_2)) = (P_1 + \tfrac{p}{q} Z)^3 = P_1^3 + 3\tfrac{p}{q} P_1^2 Z + 3 \tfrac{p^2}{q^2} P_1 Z^2 + \tfrac{p^3}{q^3} Z^3 \\
    &=  P_1^3 + 3\tfrac{p}{q}(npz)^2 - 3\tfrac{p^2}{q^2}(npz \cdot nqz) + \tfrac{p^3}{q^3}(-nqz)^2 = P_1^3 + \tfrac{p^3}{q}(3\tfrac{n^2}{n} - 3\tfrac{n^2}{n}+ \tfrac{n^2}{n}) = P_1^3 + n\tfrac{p^3}{q}.
\end{align*}
\end{proof}

\subsection{Contracting Hirzebruch surfaces}
Let \(n\in\bb Z_{\geq 0}\), and let \(\phi\colon X_1 \to X_2\) be a birational morphism of threefolds contracting the exceptional divisor \(Y = \bb F_n\) onto its negative section \(\sigma\). Let \(f\) denote the ruling of \(Y\).

\begin{lemma}\label{lem:volume-lemma-3}
 Suppose \(P_2 \in \Pic(X_2)\otimes\bb R\) is nef and such that \(\phi^* P_2 = P_1 + \frac{p}{q} Y\), where \(p,q\in \bb R\) satisfy \(P_1\cdot f = p\) and \(Y\cdot f = -q\). Write \(P_1 \cdot \sigma = a\) and \(Y \cdot \sigma = b\).
Then \[\vol (P_2) = P_1^3 + \frac{p^2}{q^2} (3aq + bp + npq).\]
In particular, if $n = 0$, then $Y = \bP^1 \times \bP^1$ and 
\(\vol (P_2) = P_1^3 + \frac{p^2}{q^2} (3aq + bp)\).
\end{lemma}

\begin{proof}
    First observe $P_1|_Y = p \sigma + (a+pn)f$ and $Y|_Y = -q\sigma + (b-qn)f$.
    The statement follows by using these observations to expand the expression $\vol(P_2) = \vol(\phi^* P_2) = (P_1 + \frac{p}{q}Y)^3$.
\begin{detail}
    Indeed, using that $\sigma^2 = -n$, $f\cdot \sigma = 1$, and $f^2 = 0$, we obtain 
    \begin{equation*}
        \begin{split}
            \vol(P_2) &= \left(P_1 + \tfrac{p}{q}Y\right)^3 = P_1^3 + 3\tfrac{p}{q}P_1^2Y + 3\tfrac{p^2}{q^2}P_1Y^2 + \tfrac{p^3}{q^3}Y^3 \\
            &= \resizebox{.85\hsize}{!}{$ P_1^3 + 3\tfrac{p}{q} (p\sigma + (a+pn) f)^2 + 3 \tfrac{p^2}{q^2}(p \sigma+ (a+pn)f)(-q \sigma + (b-qn)f) + \tfrac{p^3}{q^3}(-q \sigma + (b-qn)f)^2 $ } \\
            &= P_1^3 + \tfrac{p^2}{q^2}(3aq+bp+npq) .
        \end{split}
    \end{equation*}
\end{detail}
\end{proof}

\subsection{The Atiyah flop}
Suppose there is a smooth rational curve $C$ contained in the smooth locus of a threefold $X$ such that $\mathcal{N}_{C/X} = \calO(-1) \oplus \calO(-1)$.  Blowing up this curve in $X$ yields a morphism $g\colon \hX \to X$ with exceptional divisor $Y \subset \hX$.  The normal bundle of $C$ in $X$ shows $Y \cong \bP^1 \times \bP^1$.  Denote by $\ell$ a ruling of $Y$ contracted by $g$, which by construction satisfies $Y \cdot \ell = -1$.  Denote the other ruling of $Y$ by $\ell^+$.  From the normal bundle of $C$, we also have $Y \cdot \ell^+ = -1$ and conclude that $\ell^+$ can be contracted in a morphism $g^+\colon \hX \to X^+$.  This composition is known as the \defi{Atiyah flop}.
\[\begin{tikzcd}
\mathbb{P}^1 \times \mathbb{P}^1 \cong Y \arrow[r, hook] & \hX \arrow[ld, "g"] \arrow[rd, "g^+" {swap}] &     \\
X \arrow[rr,dashed]                                       &                                           & X^+
\end{tikzcd}\]
The diagram above describes the flop.
In this diagram, the morphisms $g,g^+$ contract $Y$ to $\mathbb{P}^1$ in the two coordinate directions. As above, $\ell^+$ (resp. \(\ell\)) is a ruling contracted by $g^+$ (resp. $g$).

\begin{lemma}\label{lem:volume-lemma-4}
    Assume that there exists $P\in\Pic(X)\otimes\bb R$ such that $P \cdot C = p$ and a nef $P^+\in\Pic(X^+)\otimes\bb R$ such that $g^* P + p Y = (g^{+})^*P^+$.  Then $\vol(P^+) = P^3 - p^3$.
\end{lemma}

\begin{proof}
    Apply the $n=0$ case of Lemma \ref{lem:volume-lemma-3}, to $P_1 \coloneqq g^*P$ and $P_2 \coloneqq P^+$.  Here, $p= (g^*P) \cdot \ell^+ = P \cdot C$ by assumption, $q=-Y\cdot \ell^+=1$ by construction of the flop, $a=(g^*P)\cdot \ell  = 0$ because $\ell$ is contracted by $g$, and $b=Y\cdot \ell=-1$ again by construction of the flop.  Lemma \ref{lem:volume-lemma-3} then implies 
    \(\vol(P^+) = (g^*P)^3 + \frac{p^2}{1^2}(-p) \), and as $P^3 = (g^*P)^3$, we conclude 
    \( \vol(P^+) = P^3 - p^3\).
\end{proof}

\subsection{Flipping a curve through a \texorpdfstring{$\frac{1}{2}(1,1,1)$}{1/2(1,1,1)} singularity.}\label{construction:flip}
Suppose that $X$ is a threefold with a $\frac{1}{2}(1,1,1)$ singularity at a point $y \in X$, and $C$ is a smooth rational curve containing $y$ but no other singularities of $X$.  Let $g_1\colon X_1 \to X$ be the blow-up of $y \in X$ with exceptional divisor $E_1 \cong \bP^2$ such that $\calO_{E_1}(-E_1) = \calO(2)$.  Let $C_1$ denote the strict transform of $C$ in $X_1$ and assume that $\mathcal{N}_{C_1/X_1} = \calO(-1) \oplus \calO(-1)$.  As in the previous section, we may perform the Atiyah flop of $C_1$ by first blowing up $C_1$ in a morphism $g_2\colon \hX \to X_1$, extracting an exceptional divisor $Y \cong \bP^1 \times \bP^1$ such that $\calO_Y(-Y) = \calO(1,1)$.  The flop is obtained by contracting the opposite ruling of $Y$ in a morphism $g_2^+\colon \hX \to X_1^+$.

Consider the strict transform $E_1^+$ of $E_1$ in $X_1^+$.  From construction of the flop, we have $E_1^+ \cong \mathbb{F}_1$ and, denoting a fiber of the Hirzebruch surface $E_1^+$ by $f$, we compute $E_1^+ \cdot f = -1$.  Indeed, let $\widehat{E}_1 \cong \bb F_1$ denote the strict transform of $E_1$ on $\hX$.  As $g_2$ was the blow-up of a curve not contained in $E_1$, we have $g_2^*E_1 = \widehat{E}_1$.  The morphism $g_2^+$ then contracts $Y$ onto the negative section of $E_1^+$, so $(g_2^+)^*E_1^+ = \widehat{E}_1+Y$.  Therefore \(g_2^*E_1 = (g_2^+)^*E_1^+ - Y\).
Next, by taking intersections with a fiber $f$ of the Hirzebruch surface $E_1^+ \cong \widehat{E}_1$, we find $-2 = E_1^+ \cdot f - 1$, i.e. $E_1^+ \cdot f = -1$.  From this, one may conclude that $f$ is contractible in $X_1^+$.  Let $g_1^+\colon X_1^+ \to X^+$ be the contraction of $E_1^+$ to its negative section.
We denote the compositions $g \coloneqq g_1 \circ g_2$ and $g^+\coloneqq g_1^+ \circ g_2^+$. We have a diagram: 
$$\begin{tikzcd}
                  &                        &  \bb P^1\times\bb P^1 \cong Y \arrow[r, hook] & \hX \arrow[ld, "g_2"] \arrow[rd, "g^+_2" {swap}] &                           &     \\
                  & \bb P^2\cong E_1 \arrow[r, hook]    & X_1 \arrow[ld, "g_1"]   &                                               & X^+_1 \arrow[rd, "g_1^+" {swap}] \arrow[hookleftarrow, r] &   E_1^+ \cong \bb F_1 \\
 & X \arrow[rrrr, dashed] &                         &                                               &                           & X^+ .
\end{tikzcd}$$

\begin{remark}
    From the assumption on the normal bundle of $C_1$ in $X_1$, we have $K_{X_1} \cdot C_1 = 0$.  As $K_{X_1} = g_1^*K_X + \frac{1}{2}E_1$, we conclude $K_X \cdot C = -\frac{1}{2}$.  Similarly, denoting by $C_1^+$ the image of $Y$ in $X_1^+$, we have $K_{X_1}^+ \cdot C_1^+ = 0$.  Let $C^+$ denote the image of $C_1^+$ in $X^+$.  By the preceding paragraph, we have $g_2^*E_1 = (g_2^+)^*E_1^+ - Y $, and intersecting this equation with any preimage $\ell$ of $C_1^+$ in $Y$, we find that $0 = E_1^+ \cdot C_1^+ +1$, so  $E_1^+ \cdot C_1^+ = -1$.  As $K_{X_1^+} = (g_1^+)^*K_{X^+} + E_1^+$, we conclude $K_{X^+} \cdot C^+ = 1$.  Because $X, X^+$ are isomorphic away from $C, C^+$, and $K_{X} \cdot C < 0$ while $K_{X^+} \cdot C^+ > 0$, the threefold $X^+$ is \textit{the} flip of the curve $C \subset X$ (see e.g. \cite[Corollary 6.4]{KollarMori}). 
\end{remark}

\begin{lemma}\label{lem:volume-lemma-5}
    Suppose $P\in\Pic(X)\otimes\bb R$ and $P^+\in\Pic(X^+)\otimes\bb R$ are identified where the birational map describe above is well defined.  If $P \cdot C = p$, then $g^*P + p\widehat{E}_1 + 2p Y = (g^+)^* P^+$. Furthermore, if $P^+$ is nef, then $\vol(P^+) = P^3 - 4p^3$.
\end{lemma}

\begin{proof}
    First, we find the real numbers $a$ and $b$ such that 
    \begin{equation}\label{eqn:defab}
        g^* P+a\widehat{E}_1 + b Y = (g^+)^* P^+.
    \end{equation} 
    By construction, the strict transform of $E_1$ on $\hX$ is $\widehat{E}_1 \cong \mathbb{F}_1$.  Denote its negative section by $\sigma$ and fiber class by $f$. Write $\ell$ for a curve contracted by $g$ and $\ell^+$ for a curve contracted by $g^+$.
    Observe \[\widehat{E}_1 \cdot f = -2,\quad \widehat{E}_1 \cdot \sigma = 0, \quad\widehat{E}_1|_{\widehat{E}_1} = - 2(\sigma + f), \quad Y|_Y = - \ell - \ell^+, \quad \widehat{E}_1|_Y = \ell, \quad Y|_{\widehat{E}_1} = \sigma.\]
\begin{detail}
    The first one is \(-2\) because it contracts to a \(\frac{1}{2}(1,1,1)\) singularity.
\end{detail}
    As $p = P \cdot C = (g^* P) \cdot \ell^+$ and $(g^* P) \cdot \ell = 0$, it follows that $(g^*P)|_Y = p \ell$.   We can therefore deduce the values of $a$ and $b$ by intersecting equation (\ref{eqn:defab}) with $f$ and $\ell^+$.  Indeed, intersecting with $f$ (a curve contracted by both $g$ and $g^+$), we find 
    \(0+  a(-2) + b(1) = 0\),
    and intersecting with $\ell^+$ (contracted by $g^+$), we find 
    \( p + a(1) + b(-1) = 0\).
    Solving this linear system yields $a = p$ and $b = 2p$. 

    Finally, assuming $P^+$ is nef, we compute
    \( \vol(P^+) = \vol((g^+)^*P^+) = (g^* P+p\widehat{E}_1 + 2p Y)^3 \)
    using the previously determined intersection products.  Note that $g^*P$ is trivial on $\widehat{E}_1$ and is supported on a fiber of the ruled surface $Y \to C_1$, so $(g^*P)|_{\widehat{E}_1} = 0$ and $((g^*P)|_{Y})^2 = 0. $
\begin{detail}
    Therefore,
    \begin{align*}
        \vol(P^+) &= (g^* P+p\widehat{E}_1 + 2p Y)^3 \\
        &= (g^*P)^3 + 3p(g^*P)^2(\widehat{E}_1+2Y) + 3p^2(g^*P)(\widehat{E}_1+2Y)^2 + p^3(\widehat{E}_1+2Y)^3 \\
        &= P^3 + 3p^2(g^*P)(4Y^2) + p^3(\widehat{E}_1^3 + 6 \widehat{E}_1^2 Y + 12 \widehat{E}_1 Y^2 + 8Y^3) \\
        &= P^3 +3p^2(-4p) + p^3(4+0+12(-1)+8(2)) = P^3 - 4p^3. 
    \end{align*}
\end{detail}
\end{proof}

\subsection{Flopping a curve through a \texorpdfstring{$\frac{1}{2}(1,1,0)$}{1/2(1,1,0)} singularity}\label{construction:complicatedflop}

We compute one more example of a flop that occurs several times in the paper.  Suppose that $X$ is a $\mathbb{Q}$-Gorenstein threefold singular along a curve $T$, locally of type $\frac{1}{2}(1,1,0)$.  Let $g_1\colon X_1 \to X$ be the blow-up of $T$ with exceptional divisor $E_1$ a ruled surface over $T$, satisfying $E_1 \cdot f = -2$ for fibers $f$ of the ruled surface $E_1$.  Suppose now that $C \subset X$ is a smooth rational curve meeting $T$ transversally with strict transform $C_1 \subset X_1$ and that $\mathcal{N}_{C_1/X_1} = \calO(-1) \oplus \calO(-1)$.  The assumption on the normal bundle implies that $K_{X_1} \cdot C_1 = 0$ and, as $K_{X_1} = g_1^*K_{X}$ (due to the $\frac{1}{2}(1,1,0)$ singularity along \(T\)), we have $K_X \cdot C = 0$.  We construct the flop of $C$ and the corresponding change in volume.  

The transversality assumption on $C$ and $T$ is that $C_1$ meets the ruled surface $E_1$ transversally at some point $p_1 \in E_1$.  Let $f_1$ be the fiber of the ruled surface $E_T$ through $p_1$.  With this notation, the flop of $C$ in $X$ can be constructed in four steps:

\begin{enumerate}
    \item Blow up $T$ in the morphism $g \colon X_1 \to X$.
    \item Perform the Atiyah flop $X_1 \dashrightarrow X_2$ of the curve $C_1 \subset X_1$.  Denote $C_1^+$ by $C_2$.
    \item Let $f_2$ denote the strict transform of $f_1$ in $X_2$.  By direct computation, $\mathcal{N}_{f_2/X_2} = \calO(-1) \oplus \calO(-1)$.  Perform the Atiyah flop $X_2 \dashrightarrow X_3$ of the curve $f_2 \subset X_2$.  Denote $f_2^+$ by $f_3$.
    \item Let $E_3$ be the strict transform of $E_1$ in $X_3$.  $E_3$ is a smooth ruled surface with rulings equivalent to $C_3$ and $E_3 \cdot C_3 = -2$. Contract the rulings to a curve $T^+$ in a morphism $g^+: X_3 \to X^+$ creating a $\frac{1}{2}(1,1,0)$ singularity along $T^+$. Let $C^+$ be the image of $f_3$.
\end{enumerate}

The resulting threefold $X^+$ is the flop of $C$ in $X$.  Indeed, they are isomorphic along $X \setminus C$ and $X^+ \setminus C^+$, and $K_X \cdot C = K_{X^+} \cdot C^+ = 0$.

This is summarized diagrammatically below.

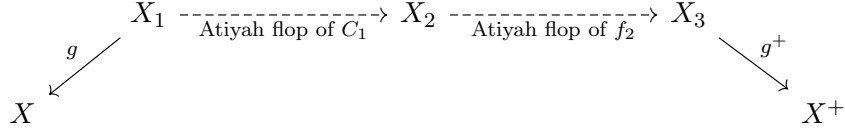
\[\begin{tikzcd}
& X_1 \arrow[ld, "g", swap] \arrow[rrr, dashed, "\text{Atiyah flop of }C_1" {swap}] && &   X_2 \arrow[rrr, dashed, "\text{Atiyah flop of }f_2" {swap}] && &  X_3 \arrow[rd, "g^+"] & \\
X & & & & & & & & X^+ \\
\end{tikzcd}\]

We indicate the configuration of curves in these birational transformations in Figure \ref{fig:complicatedflop}.  We label the self intersections of relevant curves in the exceptional divisors.

    \begin{figure}[h]
    \centering
    \adjustbox{width=\textwidth}{
    \begin{tikzcd}
    & \begin{tabular}{c}
\begin{tikzpicture}

\draw [-] (0,0) to (1,0) to (1,1) to (0,1) to (0,0);
\draw [-] (-.75,-0.75) to (0,0);
\draw [very thick, teal] (-.75,-.75) to (0,0); 
\draw [very thick, blue] (1,0) to (0,0); 

\node[above left, node font=\tiny] at (-.4,-.4) {$C_1$};
\node[node font=\tiny] at (0.5,0.7) {$E_1$};
\node[above, node font=\tiny] at (0.5,0) {$0$};
\node[below, node font=\tiny] at (0.5,0) {$f_1$};

\end{tikzpicture}
\end{tabular} \arrow[ld, "g", swap] \arrow[rrr, dashed, "\text{Atiyah flop of }C_1" {swap}] && &   \begin{tabular}{c}
\begin{tikzpicture}

\draw [-] (-0.87,1.5) to (0.87, 1.5) to (0.87, 0.5) to (0,0) to (-0.87, 0.5) to (-0.87, 1.5);
\draw [very thick, teal] (0,0) to (-0.87,0.5);
\draw [very thick, blue] (0,0) to (0.87,0.5);

\node[below, node font=\tiny] at (0,1.5) {$E_2$};
\node[node font=\tiny] at (0.4,.45) {$-1$};
\node[below right, node font=\tiny] at (0.4,.35) {$f_2$};
\node[node font=\tiny] at (-0.4,.45) {$-1$};
\node[below left, node font=\tiny] at (-0.4,.35) {$C_2$};

\end{tikzpicture}
\end{tabular} \arrow[rrr, dashed, "\text{Atiyah flop of }f_2" {swap}] && &  \begin{tabular}{c}
\begin{tikzpicture}

\draw [-] (0,0) to (-1,0) to (-1,1) to (0,1) to (0,0);
\draw [-] (.75,-0.75) to (0,0);
\draw [very thick, blue] (.75,-.75) to (0,0); 
\draw [very thick, teal] (-1,0) to (0,0); 

\node[above right, node font=\tiny] at (.4,-.4) {$f_3$};
\node[node font=\tiny] at (-0.5,0.7) {$E_3$};
\node[above, node font=\tiny] at (-.5,0) {$0$};
\node[below, node font=\tiny] at (-.5,0) {$C_3$};

\end{tikzpicture}
\end{tabular} \arrow[rd, "g^+"] & \\
    \begin{tabular}{c}
\begin{tikzpicture}

\draw [-] (0,0) to (0,1);
\draw [very thick, black!60] (0,0) to (0,1);
\draw [-] (0,0) to (0,1);
\draw [-] (-.75,-0.75) to (0,0);

\node[above left, node font=\tiny] at (-.4,-.4) {$C$};
\node[right, node font=\tiny] at (0,0.5) {$T$};

\end{tikzpicture}
\end{tabular} & & & & & & & & \begin{tabular}{c}
\begin{tikzpicture}

\draw [-] (0,0) to (0,1);
\draw [very thick, black!60] (0,0) to (0,1);
\draw [-] (0,0) to (0,1);
\draw [-] (.75,-0.75) to (0,0);

\node[above right, node font=\tiny] at (.4,-.4) {$C^+$};
\node[right, node font=\tiny] at (0,0.5) {$T^+$};

\end{tikzpicture}
\end{tabular} \\
    \end{tikzcd}
    }
    \caption{The flop of $C \subset X$.}
    \label{fig:complicatedflop}
    \end{figure}
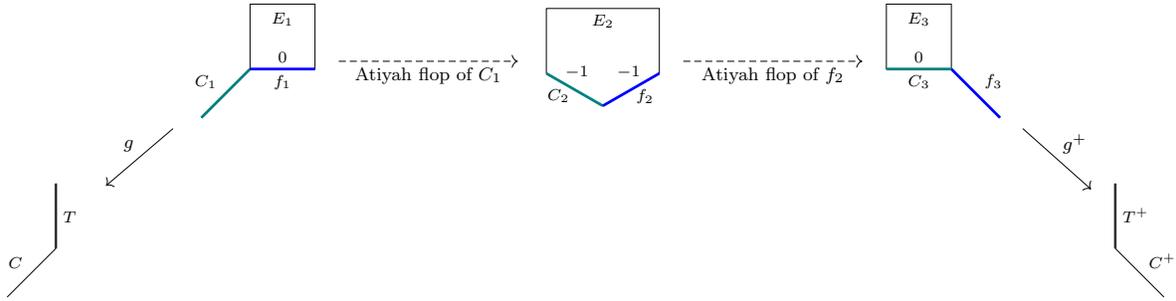

We may construct a common resolution of the Atiyah flops.  The map $X_1 \dashrightarrow X_2$ is resolved by the blow-up $\hX_1$ of $C_1$.  Denote the induced maps by $h_1\colon \hX_1 \to X_1$ and $h_1^+\colon \hX_1 \to X_2$.  Denote this exceptional divisor by $\Delta_1$.  The map $X_2 \dashrightarrow X_3$ is resolved by the blow-up $\hX_2$ of $f_2$.  Denote the induced maps by $h_2\colon \hX_2 \to X_2$ and $h_2^+\colon \hX_2 \to X_3$.  Denote this exceptional divisor by $\Delta_2$.  There is a rational map $\hX_1 \dashrightarrow \hX_2$ resolved by the blow-up of the strict transform of $f_1$ (extracting the strict transform of $\Delta_2$) and one fiber of $\Delta_1$, creating an exceptional divisor $\Delta_3$.  Let $\hX$ be this common resolution and (by abuse of notation) denote by $\Delta_1, \Delta_2, \Delta_3$ the strict transforms of these divisors on $\hX$.  Let $h\colon \hX \to X$ and $h^+\colon \hX \to X^+$.  With this notation, we have the following lemma.

\begin{lemma}\label{lem:volume-lemma-6}
    Suppose $P\in\Pic(X)\otimes\bb R$ and $P^+\in \Pic(X^+)\otimes\bb R$ are identified where \(X\dashrightarrow X^+\) is defined.  Assume that $P^+$ is nef.  If $p = P \cdot C$, there is an equality on $\hX$: $$h^* P + p \Delta_1 + p \Delta_2 + 2p \Delta_3 = (h^+)^*P^+ , \quad\text{ and moreover }\quad \vol(P^+) = \vol(P)-2p^3.$$
\end{lemma}
\begin{proof}
    Let $P_1 = g^*P$ and $P_3 = g^*P^+$. Let $P_2$ be defined by $h_1^*(P_1) + p \Delta_1 = (h_1^+)^*P_2$.  By computation, $h_2^* P_2 + p \Delta_2 = (h_2^+)^*P_3$.  By construction of $\hX$, we have $h^* P + p \Delta_1 + p \Delta_2 + 2p \Delta_3 = (h^+)^*P^+$.  Furthermore, we may compute $\vol(P^+)$ using Lemma \ref{lem:volume-lemma-4}: 
    \begin{align*}
        \vol(P^+) &= \vol(P_3) = P_3^3 \\
        &= P_2^3 - p^3 \quad \text{ by Lemma~\ref{lem:volume-lemma-4} applied to } P_2, P_3 \\
        &= (P_1^3 - p^3) - p^3 \quad \text{ by Lemma~\ref{lem:volume-lemma-4} applied to } P_1, P_2 \\
        &= P^3 - 2p^3. 
    \end{align*}
\end{proof}

\subsection{Mori cones of blow-ups of \texorpdfstring{$\bP^1 \times \bP^2$}{P1xP2} and \texorpdfstring{$\Xthree$}{Xn}}

\begin{lemma}\label{lem:moricone}
    Let $X = \bP^1 \times \bP^2$.  Suppose $\phi\colon X' \to X$ is a plt blow-up extracting a divisor $E \subset X'$ over a point $p \in X$.  Let $F$ be the unique fiber of the first projection map $X \to \bP^1$ at $p$, and let $S$ denote a $(0,1)$-surface containing the fiber of the second projection $X \to \bP^2$ at $p$.  Write $F'$ and $S'$ for their strict transforms on $X'$, and let $n, m_S$ be defined by 
    \[ \phi^* F = F' + nE, \qquad \phi^* S = S' + m_{S}E .\]
    Let $l_1 \subset X'$ be the strict transform of the fiber of the second projection at $p$. Assume the following:
    \begin{enumerate}
        \item Assume $\rho(E) = 1$, and let $e \subset E$ be a generator of $\Pic(E) \otimes \bb Q$.
        \item Assume $\rho(F') = 2$. This implies that $\Pic(E) \otimes \bb Q$ is generated by $e$ (a rational multiple of the intersection of $F'$ and $E$) and $l_2$ (the strict transform of some line $l \subset F$).
        \item Assume that, for any choice of $S$, $\Pic(S') \otimes \bb Q$ can be generated by $l_1$, convex combinations of $l_2$ and $e$, and curves of equivalence class $a l_1 + bl_2 - ce$ for some real numbers $a,b,c \ge 0$.  
    \end{enumerate}
    Suppose there exists an additional extremal ray $R$ in the cone of curves $\overline{NE(X')}$ that is either the class of an effective curve $C$ or the limit of irreducible effective curves $C_t$ as $ t \to \infty$.  If $C$ (or $C_t$ for $t \gg 0$) is not contained in $E$, $F'$, or $S'$ for some choice of $S$, then we must have $n E \cdot l_1 < 1$ and $m_{S} E\cdot l_2 < 1$ for this choice of \(S\).
\end{lemma}

\begin{proof}
    Assume there exists such an extremal ray as outlined in the statement of the lemma.  Let $C$ denote a curve with class $R$ or $C_t$ as above for $t \gg 0$.  Note that the assumptions on $E, F'$, and $S'$ force there to be no other extremal rays with support contained in $E, F'$ or $S$, and therefore 
    \[ C = a l_1 + b l_2 - ce \]
    for real numbers $a,b,c > 0$.  The positivity of $a,b$ follows as $a$ and $b$ are determined by the numerical equivalence $ \phi(C) = a \phi(l_1) + b \phi(l_2)$, as the images of $l_1$ and $l_2$ generate the cone of curves of $X$. Neither $a$ nor $b$ could be $0$, as that would force $C \subset F'$ or $C \subset S'$.  The negativity of the coefficient of $e$ then follows by assumption that $C$ is (sufficiently near) the extremal ray $R$. 

    $E \cdot C > 0$ as $C$ is extremal. Because $C$ is not contained in $F'$ or $S'$, we have $F' \cdot C \ge 0$ and $S' \cdot C \ge 0$.  Using that 
    \( \phi^* F = F' + nE \) and 
    \( \phi^* S = S' + m_{S}E \),
    we find that 
    \begin{align*}
    F' \cdot e &= -n E \cdot e, &  F' \cdot l_1 &= 1 - nE \cdot l_1, & F' \cdot l_2 &= -n E \cdot l_2, \\ S' \cdot e &= -m_S E \cdot e, &  S' \cdot l_1 &= - m_{S} E \cdot l_1, & S' \cdot l_2 &= 1-m_S E \cdot l_2.
    \end{align*}
    Computing the intersections $E \cdot C$, $F' \cdot C$, and $S' \cdot C$, we have the inequalities: 
    \begin{equation*}
        \begin{split}
            E \cdot C &= a E \cdot l_1 + b E \cdot l_2 - c E \cdot e > 0, \\
            F' \cdot C &=  a(1 - n E \cdot l_1)  - b n E \cdot l_2 + c n E \cdot e \ge 0, \\
            S' \cdot C &= - a m_S E \cdot l_1 + b(1 - m_S E \cdot l_2) + c m_S E \cdot e \ge 0.
        \end{split}
    \end{equation*}
    The latter two inequalities imply that $a \ge n E \cdot C$ and $b \ge m_S  E \cdot C$. Then, using that $-c E \cdot e > 0$, $E\cdot l_1 \geq 0, E\cdot l_2 \geq 0$, the first inequality implies
    \( a > a n E \cdot l_1 \)
    and 
    \( b > b m_S E \cdot l_2\).
    This gives the desired conclusion that $1 > n E \cdot l_1$ and $1 > m_S E \cdot l_2$.  
\end{proof}

The main application of the previous result will be the following.  In words, we conclude that under certain numeric conditions, we may find all possible extremal rays in $\overline{NE(X')}$ by finding the generators for the Mori cones of the surfaces $E$, $F'$, and $S'$. 

\begin{corollary}\label{cor:generatorsofmoricone}
    Let $\phi\colon X' \to X$ be a divisorial contraction as in Lemma~\ref{lem:moricone}, and suppose that $1 \le n E \cdot l_1$ or $1 \le m_S E \cdot l_2$. (E.g. these assumptions hold if $X'$ is smooth.)  Then all extremal rays in $\overline{NE(X')}$ are spanned by the class of a rational curve $C$ such that $C$ is contained in $E, F'$, or $S'$.  
\end{corollary}

\begin{proof}
    By Lemma~\ref{lem:moricone} and its proof, extremal rays must be approximated by curves $C$ whose support is contained in one of the rational surfaces $E, F'$, or $S'$ satisfying either $E \cdot C < 0$, $F' \cdot C < 0$, or $S' \cdot C < 0$, which implies that the extremal ray is equal to such a curve. 
\end{proof}

Next, we derive similar results for the Mori cone of blow-ups of the small modification $W^+$ of $\Xthree$ (as defined in Section~\ref{sec:wall-crossing-c_0}).  Recall that $W^+$ can alternately be constructed as the blow-up of $\bP(1,1,1,2)$ in the singular point $\conept$ and a smooth point $\ptE$, with exceptional divisors $S^+ \cong \bP^2$ and $E^+ \cong \bP^2$, respectively.  By Lemma~\ref{lem:Mori-cone-of-W+}, the Mori cone of $W^+$ is generated by the exceptional curve $l_1$ of the small modification, a line $l_S$ in $S^+$, and a line $l_E$ in $E^+$.

\begin{lemma}\label{lem:moriconeforblowupofX3}
    Let $X = \Xthree$ and let $W^+$ be the small modification of $X$ admitting a map to $\bP(1,1,1,2)$.  Suppose $\phi\colon W' \to W^+$ is a plt blow-up extracting a divisor $Z \subset W'$ over a point $p \in W^+$ such that $p\notin S^+ \cup l_1$.  Let $\Gamma$ be the strict transform of the unique surface in $|\calO_{\bP(1,1,1,2)}(1)|$ containing $p$ and $l_1$.  Write $E'$ and $S'$ for the strict transforms of $E^+$ and $S^+$ on $W'$ and $\Gamma'$ for the strict transform of $\Gamma$.  Let $n, m$ be defined by 
    \[ \phi^*(E^+)  = E' + nZ \qquad \phi^* \Gamma = \Gamma' + m Z.\]
    Assume the following:
    \begin{enumerate}
        \item Assume $\rho(Z) = 1$, and let $z \subset Z$ be a generator of $\Pic(Z) \otimes \bb Q$.
        \item Assume $\rho(E') \le 2$. This implies that $\Pic(E) \otimes \bb Q$ is generated by $l_E$ (the strict transform of some line $l_E \subset E^+$) and $z$ if $p \in E^+$ (a rational multiple of the intersection of $E'$ and $Z$, ).
        \item Assume that $\Pic(\Gamma') \otimes \bb Q$ can be generated by nonnegative combinations of $l_1$, $l_E$, $l_S$, $z$, and curves of equivalence class $a l_1 + bl_E + cl_S - dz$ for some real numbers $a,b,c,d \ge 0$. If additionally $p \notin E^+$, denote by $l_\Gamma$ the strict transform of the unique ruling of $\bP(1,1,1,2)$ containing the image of $p$ that is contained in the image of $\Gamma$ (which, by construction, satisfies $l_\Gamma = l_1 + l_E - dz$ for some $d \ge 0$). 
  
    \end{enumerate}
    Suppose there exists an additional extremal ray $R$ in the cone of curves $\overline{NE(W')}$ that is either the class of an effective curve $C$ or the limit of irreducible effective curves $C_t$ as $ t \to \infty$.  If $C$ (or $C_t$ for $t \gg 0$) is not contained in $Y$, $E'$, $S'$, or $\Gamma'$ for some choice of $\Gamma$, then we must have $(n+m) Z \cdot l_E < 1$.  If additionally $p \notin E^+$, then we must have $2m Z \cdot l_\Gamma < 1$.
\end{lemma}

\begin{proof}
    Assume there exists such an extremal ray as outlined in the statement of the lemma.  Let $C$ denote a curve with class $R$ or $C_t$ as above for $t \gg 0$.  Note that the assumptions on $Z, E'$, $S'$, and $\Gamma'$ force there to be no other extremal rays with support contained in $Y, E'$, $S'$, or $\Gamma'$, and therefore 
    \[ C = a l_1 + b l_E + cl_S - dz \]
    for real numbers $a,b,c \ge 0$ and $d > 0$.  The nonnegativity of $a,b,c$ follows as $a$, $b$, and $c$ are determined by the numerical equivalence $ \phi(C) = a \phi(l_1) + b \phi(l_E) + c\phi(l_S)$, as the images of $l_1$, $l_E$, $l_S$ generate the cone of curves of $W^+$.  The negativity of the coefficient of $z$ then follows by assumption that $C$ is (sufficiently near) the extremal ray $R$. 

    $Z \cdot C > 0$ as $C$ is extremal in $W'$ but was not extremal in $W$. Because $C$ is not contained in $E'$, $S'$, or $\Gamma'$, we have $E' \cdot C \ge 0$, $S' \cdot C \ge 0$ and $\Gamma' \cdot C \ge 0$.  Using that 
    \( \phi^*(E^+)  = E' + nZ \) and 
    \( \phi^* \Gamma = \Gamma' + m_\Gamma Z \) and 
    \( \phi^* (S^+) = S' \) because $p$ was not contained in $S^+$,
    we find that 
    \begin{align*}
    E' \cdot l_1 &= 1, & E' \cdot l_E &= -1 - nZ \cdot l_E, & E' \cdot l_S &= 0, & E' \cdot z &= -nZ \cdot z \\
    S' \cdot l_1 &= 1, & S' \cdot l_E &= 0, & S' \cdot l_S &= -2, & S' \cdot z &= 0 \\
    \Gamma' \cdot l_1 &= -1, & \Gamma' \cdot l_E &= 1 - m_\Gamma Z \cdot l_E, & \Gamma' \cdot l_S &= 1, & \Gamma' \cdot z &= -m_\Gamma Z \cdot z . \\
    \end{align*}
    Computing the intersections $Z \cdot C$, $E' \cdot C$, $S' \cdot C$, and $\Gamma'\cdot C$, we have the inequalities: 
    \begin{equation*}
        \begin{split}
            Z \cdot C &=  b Z \cdot l_E - dZ \cdot z > 0, \\
            E' \cdot C &=  a + b(-1-nZ\cdot l_E) + dnZ \cdot z \ge 0, \\
            S' \cdot C &= a - 2c \ge 0, \\
            \Gamma' \cdot C &= -a + b(1-mZ \cdot l_E) + c+ dmZ\cdot z \ge 0.
        \end{split}
    \end{equation*}
    The first and fourth inequalities imply that $-a + b + c \ge 0$, or $b \ge a - c$.  The last inequality implies that $a - c \ge c$, so we conclude $b \ge c$.  Adding the second and fourth gives $c \ge (n+m)(bZ \cdot l_E -d Z \cdot z)$.  Combining these and using that $-dZ \cdot z > 0$, we have $b > (n+m)bZ \cdot l_E$.  Note now that $b \ge a -c \implies b \ge a$ and $b \ge c$ implies that, if $ b =0$, also $a = c = 0$, which is impossible, so $b \ne 0$ and therefore $1 > (n+m) Z \cdot l_E$. 

    Now, assume $p \notin E^+$ which implies that $n = 0$.  Write $l_\Gamma = l_1 + l_E - d'z$ for some $d' > 0$.  From the first two equations above, we have $a - b \ge 0$, so we may write $C = a''l_1 + bl_\Gamma + c l_S - d'' z$ where $a'' = a-b$.  Note that $d'' > 0$ by assumption that $C$ is extremal and not contained in $\Gamma'$.  From the intersections $Z \cdot l_\Gamma$, $E' \cdot l_\Gamma = 0$, $S' \cdot l_\Gamma = 1$, and $\Gamma' \cdot l_\Gamma = -mZ \cdot l_\Gamma$, we may derive the equations
        \begin{equation*}
        \begin{split}
            Z \cdot C &=  b Z \cdot l_\Gamma - d''Z \cdot z > 0, \\
            E' \cdot C &=  a'' \ge 0, \\
            S' \cdot C &= a'' + b - 2c \ge 0, \\
            \Gamma' \cdot C &= -a'' - bmZ \cdot l_\Gamma + c+ dmZ\cdot z \ge 0.
        \end{split}
    \end{equation*}
    From the third equation, we have $b \ge 2c - a'' \ge 2c - 2a''$, and from the last equation, we have $c - a'' \ge bmZ \cdot l_\Gamma - dmZ \cdot z$, and as $-dm Z \cdot z >0$, we have $b > 2bmZ \cdot l_\Gamma$, which implies $1 > 2m Z \cdot l_\Gamma$.
\end{proof}

As above, the main application will be to verify claims made about Mori cones of blow-ups of $W^+$. For example,

\begin{corollary}\label{cor:generatorsofmoriconeonblowupofX3}
    Let $\phi\colon W' \to W$ be a divisorial contraction as in Lemma~\ref{lem:moriconeforblowupofX3}, and suppose that $1 \le (n+m) Z \cdot l_E$ or $1 \le 2m Z \cdot l_\Gamma$. (E.g. these assumptions hold if $X'$ is smooth at the intersection of $Z$ and $l_E$ or $l_\Gamma$.)  Then all extremal rays in $\overline{NE(W')}$ are spanned by the class of a rational curve $C$ such that $C$ is contained in $Z, E', S'$, or $\Gamma'$.  
\end{corollary}

In other words, it suffices to determine generators for the Mori cones of the surfaces $Z, E', S'$, and $\Gamma'$ to determine all possible generators for the Mori cone of $W'$.

\bibliographystyle{alpha}
\bibliography{references.bib}

\end{document}